\pdfoutput=1 
\documentclass{amsart}

\usepackage[T1]{fontenc}
\usepackage{lmodern}
\usepackage{ifthen}
\usepackage{amsfonts}
\usepackage{amsxtra}
\usepackage{amssymb}
\usepackage{array}
\usepackage{eucal}
\usepackage[a4paper]{geometry}
\usepackage{xcolor}
\definecolor{cite}{rgb}{0.50,0.00,1.00}
\definecolor{url}{rgb}{0.00,0.50,0.75}
\definecolor{link}{rgb}{0.00,0.00,0.50}
\usepackage[colorlinks,linkcolor=link,urlcolor=url,citecolor=cite,pagebackref,breaklinks]{hyperref}
\usepackage{colonequals}
\usepackage{mathtools}
\usepackage{mathrsfs}
\usepackage[all]{xy}
\usepackage[lite,abbrev,msc-links]{amsrefs}
\usepackage{enumerate}
\usepackage[shortlabels]{enumitem}

\numberwithin{equation}{section}

\theoremstyle{plain}
\newtheorem{proposition}{Proposition}[subsection]

\newtheorem{corollary}[proposition]{Corollary}
\newtheorem{lem}[proposition]{Lemma}
\newtheorem{theorem}[proposition]{Theorem}

\theoremstyle{definition}
\newtheorem{definition}[proposition]{Definition}
\newtheorem{notation}[proposition]{Notation}
\newtheorem{assumption}[proposition]{Assumption}
\newtheorem{variant}[proposition]{Variant}
\newtheorem{construction}[proposition]{Construction}

\theoremstyle{remark}
\newtheorem{example}[proposition]{Example}
\newtheorem{remark}[proposition]{Remark}

\renewcommand{\b}[1]{\mathbf{#1}}
\renewcommand{\c}[1]{\mathcal{#1}}
\renewcommand{\d}[1]{\mathbb{#1}}
\newcommand{\f}[1]{\mathfrak{#1}}
\renewcommand{\r}[1]{\mathrm{#1}}
\newcommand{\s}[1]{\mathscr{#1}}
\renewcommand{\sf}[1]{\mathsf{#1}}
\renewcommand{\(}{\left(}
\renewcommand{\)}{\right)}
\newcommand{\res}{\mathbin{|}}
\newcommand{\Sec}{\S}

\newcommand{\bA}{\b A}
\newcommand{\bB}{\b B}

\newcommand{\bD}{\b D}

\newcommand{\bL}{\b L}

\newcommand{\bT}{\b T}

\newcommand{\cA}{\c A}
\newcommand{\cB}{\c B}
\newcommand{\cC}{\c C}
\newcommand{\cD}{\c D}
\newcommand{\cE}{\c E}
\newcommand{\cF}{\c F}
\newcommand{\cG}{\c G}
\newcommand{\cH}{\c H}
\newcommand{\cI}{\c I}

\newcommand{\cK}{\c K}
\newcommand{\cL}{\c L}
\newcommand{\cM}{\c M}
\newcommand{\cN}{\c N}
\newcommand{\cO}{\c O}

\newcommand{\cQ}{\c Q}
\newcommand{\cR}{\c R}
\newcommand{\cS}{\c S}
\newcommand{\cT}{\c T}
\newcommand{\cU}{\c U}
\newcommand{\cV}{\c V}
\newcommand{\cW}{\c W}
\newcommand{\cX}{\c X}
\newcommand{\cY}{\c Y}
\newcommand{\cZ}{\c Z}

\newcommand{\dG}{\d G}

\newcommand{\dN}{\d N}

\newcommand{\dQ}{\d Q}
\newcommand{\dR}{\d R}
\newcommand{\dS}{\d S}

\newcommand{\dZ}{\d Z}

\newcommand{\fC}{\f C}

\newcommand{\fF}{\f F}

\newcommand{\fR}{\f R}

\newcommand{\fm}{\f m}

\newcommand{\rB}{\r B}

\newcommand{\rD}{\r D}
\newcommand{\rE}{\r E}

\newcommand{\rG}{\r G}
\newcommand{\rH}{\r H}
\newcommand{\rI}{\r I}

\newcommand{\rL}{\r L}

\newcommand{\rN}{\r N}

\newcommand{\rR}{\r R}

\newcommand{\rT}{\r T}

\newcommand{\ra}{\r a}
\newcommand{\rb}{\r b}
\newcommand{\rc}{\r c}

\newcommand{\rh}{\r h}

\newcommand{\sF}{\s F}
\newcommand{\sG}{\s G}

\newcommand{\sK}{\s K}

\newcommand{\sfC}{\sf C}

\newcommand{\sfK}{\sf K}
\newcommand{\sfL}{\sf L}

\newcommand{\sfU}{\sf U}
\newcommand{\sfV}{\sf V}
\newcommand{\sfW}{\sf W}

\newcommand{\sfp}{\sf p}
\newcommand{\sfq}{\sf q}

\newcommand{\all}{\r{all}}
\newcommand{\atimes}{\overset{\ra}\otimes}

\newcommand{\bis}{$^{\text{bis}}$}
\newcommand{\cart}{\r{cart}}
\newcommand{\Cat}{\c{C}\r{at}_\infty}
\newcommand{\cat}{\c{C}\r{at}_1}
\newcommand{\cEs}{\cE_{\r{s}}}
\newcommand{\cEt}{\cE_{\r{t}}}
\newcommand{\Cart}{\cC\r{art}}
\newcommand{\CCpt}{\cC\r{pt}}
\newcommand{\Chp}{\c{C}\r{hp}}
\newcommand{\Chpars}{\Chp^{\r{Ar}}_{\r{lft}/\dS}}
\newcommand{\Chpfppf}{\Chp^{\r{fppf}}}
\newcommand{\Chppre}{\Chp^{\r{pre}}}
\newcommand{\colim}{\varinjlim}
\newcommand{\compl}{\r{compl}}
\newcommand{\cons}{\r{cons}}
\newcommand{\Corr}{\r{Corr}}
\newcommand{\corr}{\r{corr}}
\newcommand{\cosk}{\r{cosk}}
\newcommand{\Cov}{\r{Cov}}
\newcommand{\Cpt}{\b{Cpt}}
\newcommand{\Crt}{\r{Cart}}
\renewcommand{\deg}{\r{deg}}
\newcommand{\del}{\b{\Delta}}
\newcommand{\dgflat}{\r{dg}\text{-}\r{flat}}
\newcommand{\diag}{\r{diag}}

\newcommand{\Esp}{\c{E}\r{sp}}
\newcommand{\espet}{\r{esp}.\acute{\r{e}}\r{t}}
\newcommand{\Espqcs}{\Esp^{\r{qc.sep}}}
\newcommand{\Et}{\acute{\r{E}}\r{t}}
\newcommand{\et}{\acute{\r{e}}\r{t}}
\newcommand{\Ext}{\r{Ext}}

\newcommand{\Fibr}{\r{Fibr}}
\newcommand{\Fin}{\c{F}\r{in}_*}
\newcommand{\FunL}{\r{Fun^L}}
\newcommand{\FunR}{\r{Fun^R}}
\newcommand{\HOM}{\c{H}\r{om}}
\newcommand{\Hom}{\r{Hom}}
\newcommand{\id}{\r{id}}

\newcommand{\Kart}{\c{K}\r{art}}
\newcommand{\Komp}{\c{K}\r{pt}}
\newcommand{\Kov}{\r{Kov}}
\newcommand{\Kpt}{\mathrm{Kpt}}
\renewcommand{\lim}{\varprojlim}
\newcommand{\Liset}{\r{Lis}\text{-}\acute{\r{e}}\r{t}}
\newcommand{\liset}{{\r{lis}\text{-}\acute{\r{e}}\r{t}}}
\newcommand{\ltor}{\Box\text{-}\tor}
\newcommand{\Mset}{\c{S}\r{et}_{\Delta}^+}
\newcommand{\op}{\r{op}}
\newcommand{\pr}{\r{pr}}
\newcommand{\PR}{\c{P}\r{r}}
\newcommand{\PRing}{\c{PR}\r{ing}}
\newcommand{\PRL}{\PR^{\rL}}
\newcommand{\PRR}{\PR^{\rR}}
\newcommand{\PS}{\PR_{\r{st}}}
\newcommand{\PSL}{\PS^{\rL}}
\newcommand{\PSR}{\PS^{\rR}}
\newcommand{\PSLM}{\CAlg(\Cat)_{\r{pr,st,cl}}^\rL}
\newcommand{\PSLt}[1][]{\PR^{\r{L}}_{\r{st},\r{t}#1}}
\newcommand{\PSRt}[1][]{\PR^{\r{R}}_{\r{st},\r{t}#1}}
\newcommand{\PTopos}{\c{PT}\r{opos}}
\newcommand{\qset}{\r{qs}.\acute{\r{e}}\r{t}}
\newcommand{\RCpt}{\r{Cpt}}
\newcommand{\Ret}{\r{Ret}}
\newcommand{\Rhom}{\rR\HOM}
\newcommand{\Rind}{\c{R}\r{ind}}

\newcommand{\Ring}{\c{R}\r{ing}}
\newcommand{\RingedPTopos}{\Ring\r{ed}\PTopos}
\newcommand{\RKE}{\r{RKE}}
\newcommand{\rres}{\r{res}}
\newcommand{\Sch}{\c{S}\r{ch}}
\newcommand{\Schaff}{\Sch^{\r{aff}}}
\newcommand{\Schqcs}{\Sch^{\r{qc.sep}}}
\newcommand{\Schqs}{\Sch^{\r{qs}}}
\newcommand{\Set}{\c{S}\r{et}}
\newcommand{\St}{\mathnormal{St}}
\newcommand{\tcEs}{\tilde\cE_{\r{s}}}
\newcommand{\tcEt}{\tilde\cE_{\r{t}}}

\newcommand{\tor}{\r{tor}}
\newcommand{\trdeg}{\r{tr.deg}}
\newcommand{\Un}{\r{Un}}
\newcommand{\univ}{\r{univ}}
\newcommand{\Vosm}{\r{Vo}^{\r{sm}}}
\newcommand{\Voet}{\r{Vo}^{\et}}
\renewcommand{\leq}{\leqslant}
\renewcommand{\geq}{\geqslant}
\renewcommand{\le}{\leqslant}
\renewcommand{\ge}{\geqslant}

\newcommand{\Chpar}[1]{\Chp^{\ifthenelse{\equal{#1}{}}{}{#1\text{-}}\r{Ar}}}
\newcommand{\Chpdm}[1]{\Chp^{\ifthenelse{\equal{#1}{}}{}{#1\text{-}}\r{DM}}}
\newcommand{\Chplft}[1]{\Chp^{\r{Ar}}_{\r{lft}/{#1}}}
\newcommand{\Chplmb}[1]{\Chp^{\r{LMB}}_{\r{lft}/{#1}}}
\newcommand{\EO}[4]{\prescript{#1}{#2}{\r{EO}}^{#3}_{#4}}
\newcommand{\Sset}[1][]{\Set_{#1\Delta}}
\newcommand{\TS}[3]{\prescript{#1}{}{#2}^{#3}}

\DeclareMathOperator{\Alg}{Alg}
\DeclareMathOperator{\Ar}{Ar}
\DeclareMathOperator{\CAlg}{CAlg}
\DeclareMathOperator{\card}{card}
\DeclareMathOperator{\Ch}{Ch}
\DeclareMathOperator{\codim}{codim}
\DeclareMathOperator{\Fun}{Fun}
\DeclareMathOperator{\Map}{Map}
\DeclareMathOperator{\Mod}{Mod}
\DeclareMathOperator{\Ob}{Ob}
\DeclareMathOperator{\Spec}{Spec}
\DeclareMathOperator{\Tr}{Tr}

\begin{document}

\title[Enhanced six operations and base change theorem]
{Enhanced six operations and base change theorem for higher Artin stacks} 

\author{Yifeng Liu}
\address{Institute for Advanced Study in Mathematics, Zhejiang University, Hangzhou 310058, China}
\email{liuyf0719@zju.edu.cn}

\author{Weizhe Zheng}
\address{Morningside Center of Mathematics, Academy of Mathematics and Systems Science, Chinese Academy of Sciences, Beijing 100190, China; University of the Chinese Academy of Sciences, Beijing 100049, China}
\email{wzheng@math.ac.cn}

\date{\today}
\subjclass[2020]{14F08 (primary), 14A20, 14F20, 18N50, 18N60 (secondary)}

\begin{abstract}
In this article, we develop a theory of Grothendieck's six operations for 
derived categories in \'etale cohomology of Artin stacks, for both torsion 
and adic coefficients. We prove several desired properties of the  
operations, including the base change theorem in derived categories. This 
extends many previous theories on this subject, including the one developed 
by Laszlo and Olsson, in which the operations are subject to more 
assumptions and the base change isomorphism is only constructed on the 
level of sheaves. Moreover, our theory works for higher Artin stacks as 
well. In addition, we define perverse t-structures on higher Artin stacks 
for general perversity, extending Gabber's work on schemes. 

Our method differs from previous approaches, as we exploit the theory of 
stable $\infty$-categories developed by Lurie. We enhance derived 
categories, functors, and natural isomorphisms to the level of 
$\infty$-categories and introduce $\infty$-categorical (co)homological 
descent. To handle the issue of ``homotopy coherence'', we develop a 
general technique for gluing subcategories of $\infty$-categories and 
several other $\infty$-categorical techniques. We obtain categorical 
equivalences between simplicial sets associated to certain multisimplicial 
sets. Such equivalences can be used to construct functors in different 
contexts. One of our category-theoretical results generalizes Deligne's 
gluing theory developed in the construction of the extraordinary 
pushforward operation in \'etale cohomology of schemes. 
\end{abstract}

\maketitle

\tableofcontents

\section*{Introduction}
\label{0}

This article is an amalgamation, with minor improvements, of the following three preprints we previously posted on the arXiv:
\begin{itemize}
  \item \emph{Gluing restricted nerves of $\infty$-categories}, \href{http://arxiv.org/abs/1211.5294}{arXiv:1211.5294},

  \item \emph{Enhanced six operations and base change theorem for higher Artin stacks}, \href{http://arxiv.org/abs/1211.5948}{arXiv:1211.5948},

  \item \emph{Enhanced adic formalism and perverse t-structures for higher Artin stacks}, \href{http://arxiv.org/abs/1404.1128}{arXiv:1404.1128}.
\end{itemize}

Derived categories in \'etale cohomology on Artin stacks and Grothendieck's 
six operations (also known as six-functors) between such categories have been 
developed by many authors including \cite{Zh2} (for Deligne--Mumford stacks), 
\cite{LMB}, \cite{Behrend}, \cite{Olsson} and \cite{LO1}. These theories all 
have some restrictions. In the most recent and general one \cite{LO1} by 
Laszlo and Olsson on Artin stacks, a technical condition was imposed on the 
base scheme which excludes, for example, the spectra of certain 
fields.\footnote{For example, the field $k(x_1,x_2,\dots)$ obtained by 
adjoining countably infinitely many variables to an algebraically closed 
field $k$ in which $\ell$ is invertible.} More importantly, the base change 
isomorphism was constructed only on the level of (usual) cohomology sheaves 
\cite{LO1}*{\Sec5}. The Base Change theorem is fundamental in many 
applications. In the Geometric Langlands Program for example, the theorem has 
already been used on the level of perverse cohomology. It is thus necessary 
to construct the Base Change isomorphism not just on the level of cohomology, 
but also in the derived category. Another limitation of most previous works 
is that they dealt only with constructible sheaves. When working with 
morphisms \emph{locally} of finite type, it is desirable to have the six 
operations for more general sheaves. 

In this article, we develop a theory that provides the desired extensions of 
previous works. Instead of the usual unbounded derived category, we work with 
its enhancement, which is a stable $\infty$-category in the sense of Lurie 
\cite{HA}*{Definition 1.1.1.9}. This makes our approach different from 
previous ones. We construct functors and produce relations in the world of 
$\infty$-categories, which themselves form an $\infty$-category. We start by 
upgrading the known theory of six operations for (coproducts of) 
quasi-compact and separated schemes to $\infty$-categories. The coherence of 
the construction is carefully recorded. This enables us to apply 
$\infty$-categorical descent to carry over the theory of six operations, 
including the Base Change theorem, to algebraic spaces, higher 
Deligne--Mumford stacks and higher Artin stacks.

\subsection{Results for torsion coefficients}
\label{0ss:results_torsion}

In this section, we will state our results only in the classical setting of 
Artin stacks on the level of usual derived categories (which are homotopy 
categories of the derived $\infty$-categories), among other simplifications. 
We refer the reader to Chapter~\ref{b6ss} for a list of complete results for 
higher Deligne--Mumford stacks and higher Artin stacks, stated on the level 
of stable $\infty$-categories. 

By an \emph{algebraic space}, we mean a sheaf in the big fppf site satisfying 
the usual axioms \cite{SP}*{025Y}: its diagonal is representable (by 
schemes); and it admits an \'{e}tale and surjective map from a scheme (in 
$\Sch_\cU$; see \Sec\ref{0ss:conventions}). 

By an \emph{Artin stack}, we mean an algebraic stack in the sense of 
\cite{SP}*{026O}: it is a stack in (1-)groupoids over 
$(\Sch_\cU)_{\r{fppf}}$; its diagonal is representable by algebraic spaces; 
and it admits a smooth and surjective map from a scheme. In particular, we do 
not assume that an Artin stack is quasi-separated. Our main results are the 
construction of the six operations on the derived categories of sheaves in 
the \'etale cohomology of Artin stacks and the expected relations among them. 
In what follows, $\Lambda$ denotes a (unital commutative) ring, or more 
generally, a ringed diagram in Definition \ref{b2de:ringed_diagram}. 

To an Artin stack $\cX$, we associate a triangulated category 
$\rD(\cX,\Lambda)$. If $\cX$ is Deligne--Mumford, then this is simply the 
unbounded derived category $\rD(\cX_{\et},\Lambda)$ of 
$\Mod(\cX_{\et},\Lambda)$, the Abelian category of  
$(\cX_{\et},\Lambda)$-modules, where $\cX_{\et}$ is the \'etale topos  
associated to $\cX$. In general, although our construction does not make use 
of the lisse-\'etale topos, $\rD(\cX,\Lambda)$ turns out to be equivalent to 
a full subcategory of $\rD(\cX_{\liset},\Lambda)$, the unbounded derived 
category of $(\cX_{\liset},\Lambda)$-modules, where $\cX_{\liset}$ is the 
lisse-\'etale topos $\cX_{\liset}$ associated to $\cX$. Recall that an 
$(\cX_{\liset},\Lambda)$-module $\sF$ is equivalent to an assignment to each 
smooth morphism $v\colon Y\to\cX$ with $Y$ an algebraic space a 
$(Y_{\et},\Lambda)$-module $\sF_v$ and to each 2-commutative triangle 
\begin{align*}
\xymatrix{
  Y' \ar[rr]^-{f} \ar[dr]_-{v'}
  &  &    Y \ar[dl]^{v}    \\
  & \cX \ar@{}[u] |\sigma
}
\end{align*}
with $v$, $v'$ smooth and $Y$, $Y'$ being algebraic spaces, a morphism 
$\tau_\sigma\colon f^*\sF_v\to\sF_{v'}$ that is an isomorphism if $f$ is 
\'{e}tale, such that the collection $\{\tau_\sigma\}$ satisfies a natural 
cocycle condition \cite{LMB}*{Lemme 12.2.1}. An 
$(\cX_{\liset},\Lambda)$-module $\sF$ is \emph{Cartesian} if in the above 
description, \emph{all} morphisms $\tau_\sigma$ are isomorphisms 
\cite{LMB}*{D\'{e}finition 12.3}. Let $\rD_\cart(\cX_{\liset},\Lambda)$ be 
the full subcategory of $\rD(\cX_{\liset},\Lambda)$ spanned by complexes 
whose cohomology sheaves are all Cartesian. We have an equivalence of 
categories $\rD(\cX,\Lambda)\simeq \rD_\cart(\cX_{\liset},\Lambda)$. 

Let $f\colon\cY\to\cX$ be a morphism of Artin stacks. We define the following four operations in \Sec\ref{b6ss:operations}:
\begin{align*}
f^*&\colon
\rD(\cX,\Lambda)\to\rD(\cY,\Lambda),\\
f_*&\colon \rD(\cY,\Lambda)\to\rD(\cX,\Lambda),\\
-\otimes_\cX-&\colon
\rD(\cX,\Lambda)\times\rD(\cX,\Lambda)\to\rD(\cX,\Lambda),\\
\HOM_\cX&\colon
\rD(\cX,\Lambda)^{op}\times\rD(\cX,\Lambda)\to\rD(\cX,\Lambda).
\end{align*}
The pairs $(f^*,f_*)$ and $(-\otimes_\cX\sfK,\HOM_\cX(\sfK,-))$ for every $\sfK\in\rD_\cart(\cX_{\liset},\Lambda)$ are pairs of adjoint functors.

To state the other two operations, we fix a nonempty set $\Box$ of rational primes. A ring is \emph{$\Box$-torsion} \cite{SGA4}*{Expos\'{e} ix, D\'{e}finition 1.1} if each element of it is killed by an integer that is a product of primes in $\Box$. An Artin stack $\cX$ is {\em $\Box$-coprime} if there exists a morphism $\cX\to \Spec\dZ[\Box^{-1}]$. If $\cX$ and $\cY$ are $\Box$-coprime (resp.\ Deligne--Mumford), $f\colon \cY\to \cX$ is locally of finite type, and $\Lambda$ is $\Box$-torsion (resp.\ torsion), then there is another pair of adjoint functors:
\begin{align*}
f_!&\colon \rD(\cY,\Lambda)\to\rD(\cX,\Lambda),\\
f^!&\colon \rD(\cX,\Lambda)\to\rD(\cY,\Lambda).
\end{align*}

Next we list some properties of the six operations. We refer the reader to \Sec\ref{b6ss:operations} for a more complete list.

\begin{theorem}[K\"{u}nneth Formula, Theorem \ref{b6th:kunneth}]
Let $\Lambda$ be a $\Box$-torsion (resp.\ torsion) ring, and
\[
\xymatrix{\cY_1\ar[d]_-{f_1} & \cY\ar[l]_-{q_1}\ar[d]^-{f}\ar[r]^-{q_2} & \cY_2\ar[d]^{f_2}\\
\cX_1 & \cX\ar[l]_{p_1}\ar[r]^-{p_2} & \cX_2}
\]
a diagram of $\Box$-coprime Artin stacks (resp.\ of arbitrary 
Deligne--Mumford stacks) that exhibits $\cY$ as the limit 
$\cY_1\times_{\cX_1} \cX\times_{\cX_2} \cY_2$, where $f_1$ and $f_2$ are 
locally of finite type. Then there is a natural isomorphism of functors: 
\[
f_!(q_1^*-\otimes_\cY q_2^*-)\simeq (p_1^*f_{1!}-)\otimes_\cX
(p_2^*f_{2!}-)\colon
\rD(\cY_{1},\Lambda)\times\rD(\cY_{2},\Lambda)
\to\rD(\cX,\Lambda).
\]
\end{theorem}

\begin{corollary}[Base Change]\label{0co:base_change}
Let $\Lambda$ be a $\Box$-torsion (resp.\ a torsion) ring, and
\begin{align*}
\xymatrix{
\cW \ar[d]_-{q} \ar[r]^-{g} & \cZ \ar[d]^-{p} \\
\cY \ar[r]^-{f} & \cX   }
\end{align*}
a \emph{Cartesian} diagram of $\Box$-coprime Artin stacks (resp.\ of arbitrary Deligne--Mumford stacks) where $p$ is locally of finite type. Then there is a natural isomorphism of functors:
\[
 f^*\circ p_!\simeq q_!\circ g^*\colon \rD(\cZ,\Lambda)\to\rD(\cY,\Lambda).
\]
\end{corollary}

\begin{corollary}[Projection Formula]
Let $\Lambda$ be a $\Box$-torsion (resp.\ torsion) ring, and $f\colon\cY\to\cX$ a morphism locally of finite type of $\Box$-coprime Artin stacks (resp.\ of arbitrary Deligne--Mumford stacks). Then there is a natural isomorphism of functors:
\[
f_!(-\otimes_\cY f^*-)\simeq (f_!-)\otimes_\cX-\colon\rD(\cY,\Lambda)\times\rD(\cX,\Lambda)\to\rD(\cX,\Lambda).
\]
\end{corollary}

\begin{theorem}[Trace map and Poincar\'e duality, Theorem \ref{b6th:poincare}]
Let $\Lambda$ be a $\Box$-torsion ring, and $f\colon \cY\to\cX$ a flat morphism locally of finite presentation of $\Box$-coprime Artin stacks. Then
\begin{enumerate}
  \item There is a functorial trace map
      \[
      \Tr_f\colon\tau^{\geq0}f_!\Lambda_\cY\langle d\rangle=\tau^{\geq0}f_!(f^*\Lambda_\cX)\langle d\rangle\to\Lambda_\cX,
      \]
      where $d$ is an integer larger than or equal to the dimension of 
       every geometric fiber of $f$; $\Lambda_\cX$ and $\Lambda_\cY$ denote 
       the constant sheaves placed in degree $0$; and $\langle 
       d\rangle=[2d](d)$ is the composition of the shift by $2d$ and the 
       $d$-th power of Tate's twist. 

  \item If $f$ is moreover smooth, then the induced natural transformation
      \[
      u_f\colon f_!\circ f^*\langle\dim f\rangle\to\id_\cX
      \]
      is a counit transformation, where $\id_\cX$ is the identity functor 
      of $\rD(\cX,\Lambda)$. In other words, there is a natural isomorphism 
      of functors: 
      \[
      f^*\langle\dim f\rangle\simeq f^!\colon \rD(\cX,\Lambda)\to\rD(\cY,\Lambda).
      \]
\end{enumerate}
\end{theorem}

\begin{corollary}[Smooth Base Change, Corollary \ref{b6co:smooth_base_change}]
Let $\Lambda$ of a $\Box$-torsion ring, and
\begin{align*}
\xymatrix{
\cW \ar[d]_-{q} \ar[r]^-{g} & \cZ \ar[d]^-{p} \\
\cY \ar[r]^-{f} & \cX   }
\end{align*}
a Cartesian diagram of $\Box$-coprime Artin stacks where $p$ is smooth. Then the natural transformation of functors
\[
p^*f_*\to g_*q^*\colon \rD(\cY,\Lambda)\to \rD(\cZ,\Lambda)
\]
is a natural isomorphism.
\end{corollary}

\begin{theorem}[Descent, Corollary \ref{b6co:descent}]\label{0th:descent}
Let $\Lambda$ be a ring, $f\colon \cY\to \cX$ a morphism of Artin stacks, and $y\colon \cY^+_0\to\cY$ a smooth surjective morphism. Let $\cY^+_\bullet$ be the \v{C}ech nerve of $y$ with the morphism $y_n\colon \cY^+_n\to\cY^+_{-1}=\cY$. Put $f_n=f\circ y_n\colon \cY^+_n\to \cX$.
\begin{enumerate}
  \item For every complex $\sfK\in\rD^{\geq0}(\cY,\Lambda)$, there is a convergent spectral sequence
      \[
      \rE_1^{p,q}=\rH^q(f_{p*}y_p^*\sfK)\Rightarrow \rH^{p+q} f_* \sfK.
      \]

  \item If $\cX$ is $\Box$-coprime; $\Lambda$ is $\Box$-torsion; and $f$ is locally of finite type, then for every complex
      $\sfK\in\rD^{\leq0}(\cY,\Lambda)$, there is a convergent spectral sequence
      \[
      \tilde\rE_1^{p,q}=\rH^q(f_{-p!}y_{-p}^!\sfK)\Rightarrow\rH^{p+q}f_!\sfK.
      \]
\end{enumerate}
\end{theorem}

\begin{remark}
Note that even in the case of schemes, Theorem \ref{0th:descent}(2) seems to be a new result.
\end{remark}

To state our results for constructible sheaves, we work over a $\Box$-coprime 
base scheme $\dS$ that is either quasi-excellent finite-dimensional or 
regular of dimension $\leq1$. We consider only Artin stacks $\cX$ that are 
locally of finite type over $\dS$. Let $\Lambda$ be a Noetherian 
$\Box$-torsion ring. We let $\rD_\cons(\cX,\Lambda)\subseteq 
\rD(\cX,\Lambda)$ denote the full subcategories spanned by those objects 
whose pullback to every scheme $X$, of finite type over $\dS$, has 
constructible cohomology sheaves in the usual sense. Let 
$\rD_\cons^{(+)}(\cX,\Lambda)$ (resp.\ $\rD_\cons^{(-)}(\cX,\Lambda)$) be the 
full subcategory of $\rD_\cons(\cX,\Lambda)$ spanned by complexes whose 
cohomology sheaves are locally bounded from below (resp.\ from above). We 
show in \Sec\ref{b6ss:constructible} that the six operations mentioned 
previously restrict to the following ones (see Proposition 
\ref{b6pr:constructible} and Proposition \ref{b6pr:constructible2} for 
precise statements): 
\begin{align*}
f^*&\colon
\rD_\cons(\cX,\Lambda)\to\rD_\cons(\cY,\Lambda),\\
f^!&\colon \rD_\cons(\cX,\Lambda)\to\rD_\cons(\cY,\Lambda),\\
-\otimes_\cX-&\colon
\rD_\cons^{(-)}(\cX,\Lambda)\times\rD_\cons^{(-)}(\cX,\Lambda)
\to\rD_\cons^{(-)}(\cX,\Lambda),\\
\HOM_\cX&\colon
\rD_\cons^{(-)}(\cX,\Lambda)^{op}\times\rD_\cons^{(+)}(\cX,\Lambda)
\to\rD_\cons^{(+)}(\cX,\Lambda).
\end{align*}
If $f$ is \emph{quasi-compact and quasi-separated}, then there are two more:
\begin{align*}
f_*&\colon \rD_\cons^{(+)}(\cY,\Lambda)\to\rD_\cons^{(+)}(\cX,\Lambda),\\
f_!&\colon \rD_\cons^{(-)}(\cY,\Lambda)\to\rD_\cons^{(-)}(\cX,\Lambda).
\end{align*}

We will also show that when the base scheme, the coefficient ring, and the 
morphism $f$ are all in the range of \cite{LO1}, our operations for 
constructible complexes are compatible with those constructed by Laszlo and 
Olsson on the level of usual derived categories. In particular, Corollary 
\ref{0co:base_change} implies that their operations satisfy Base Change in 
derived categories, which was left open in \cite{LO1}.


\subsection{Why $\infty$-categories?}
\label{0ss:why}

The $\infty$-categories in this article refer to the ones studied by A.~Joyal 
in \cite{Joyal1} and \cite{Joyal2} (where they are called 
\emph{quasi-categories}), J.~Lurie \cite{HTT}, et al. Namely, an 
$\infty$-category is a simplicial set satisfying the right lifting properties 
with respect to inner horn inclusions \cite{HTT}*{Definition 1.1.2.4}. In 
particular, they are models for $(\infty,1)$-categories, that is, higher 
categories whose $n$-morphisms are invertible for $n\geq2$. There are also 
other models for $(\infty,1)$-categories such as topological categories, 
simplicial categories, complete Segal spaces, Segal categories, model 
categories, and, in a looser sense, differential graded (DG) categories and 
$A_\infty$-categories. We address two questions in this section. First, why 
do we need $(\infty,1)$-categories instead of (usual) derived categories? 
Second, why do we choose this particular model of $(\infty,1)$-categories? 

To answer these questions, let us fix an Artin stack $\cX$ and an atlas 
$u\colon X\to\cX$, that is, a smooth and surjective morphism with $X$ an 
algebraic space. We denote by $\Mod_{\cart}(\cX_{\liset},\Lambda)$ the 
Abelian category of Cartesian $(\cX_{\liset},\Lambda)$-modules. Let 
$p_\alpha\colon X\times_\cX X\to X$ ($\alpha=1,2$) be the two projections. We 
know that for $\sF\in\Mod_\cart (\cX_{\liset},\Lambda)$, there is a natural 
isomorphism $\sigma\colon p_1^*u^*\sF\xrightarrow{\sim}p_2^*u^*\sF$ 
satisfying a cocycle condition. Conversely, an object 
$\sG\in\Mod(X_{\et},\Lambda)$ such that there exists an isomorphism 
$\sigma\colon p_1^*\sG\xrightarrow{\sim}p_2^*\sG$ satisfying the same cocycle 
condition is isomorphic to $u^*\sF$ for some 
$\sF\in\Mod_\cart(\cX_{\liset},\Lambda)$. More formally, 
$\Mod_\cart(\cX_{\liset},\Lambda)$ is the (2-)limit of the following diagram 
\[
\xymatrix{
   \Mod(X_{\et},\Lambda) \ar@<.5ex>[r]^-{p_1^*}\ar@<-.5ex>[r]_-{p_2^*}
   & \Mod((X\times_\cX X)_{\et},\Lambda) \ar[r]\ar@<1ex>[r]\ar@<-1ex>[r]
   & \Mod((X\times_\cX X\times_\cX X)_{\et},\Lambda)   }
\]
in the $2$-category of Abelian categories. Therefore, to study 
$\Mod_\cart(\cX_{\liset},\Lambda)$, we only need to study 
$\Mod(X_{\et},\Lambda)$ for (all) algebraic spaces $X$ in a ``2-coherent 
way'', that is, we need to track down all the information of natural 
isomorphisms (2-cells). Such 2-coherence is not more complicated than the one 
in Grothendieck's theory of descent \cite{TDTEI}. 

One may want to apply the same idea to derived categories. The problem is 
that the descent property mentioned previously, in its na\"{\i}ve sense, does 
not hold anymore, since otherwise the classifying stack $\rB\dG_{\r{m}}$ over 
an algebraically closed field would have finite cohomological dimension, 
which is false. In fact, when forming derived categories, we throw away too 
much information on the coherence of homotopy equivalences or 
quasi-isomorphisms, which causes the failure of such descent. A descent 
theory in a weaker sense, known as cohomological descent 
\cite{SGA4}*{Expos\'{e} vbis} and due to Deligne, does exist partially on the 
level of objects. It is one of the main techniques used in Olsson 
\cite{Olsson} and Laszlo--Olsson \cite{LO1} for the definition of the six 
operations on Artin stacks in certain cases. However, it has the following 
restrictions. First, Deligne's cohomological descent is valid only for 
complexes bounded from below. Although a theory of cohomological descent for 
unbounded complexes was developed in \cite{LO1}, it comes at the price of 
imposing further finiteness conditions and restricting to constructible 
complexes when defining the remaining operators. Second, relevant spectral 
sequences suggest that cohomological descent cannot be used directly to 
define $!$-pushforward. 

A more natural solution can be reached once the derived categories are 
``enhanced''. Roughly speaking (see Proposition \ref{b5pr:descent_lisse} for 
the precise statement), writing
\[
X_n=X\times_\cX\dots\times_\cX X \text{ ($(n+1)$-fold),}
\]
we define $\cD(\cX,\Lambda)$ to be the limit of following cosimplicial 
diagram 
\[
\xymatrix{
   \cD(X_{0,\et},\Lambda) \ar@<.5ex>[r]^-{p_1^*}\ar@<-.5ex>[r]_-{p_2^*}
   & \cD(X_{1,\et},\Lambda) \ar[r]\ar@<1ex>[r]\ar@<-1ex>[r]
   & \cD(X_{2,\et},\Lambda) \ar@<1.5ex>[r]\ar@<0.5ex>[r]\ar@<-0.5ex>[r]\ar@<-1.5ex>[r] &\dotsb }
\]
in a suitable $\infty$-category of \emph{presentable stable 
$\infty$-categories}. This is completely parallel to the descent property for 
module categories. Here $\cD(X_{n,\et},\Lambda)$ is the derived 
$\infty$-category of the Grothendieck Abelian category 
$\Mod(X_{n,\et},\Lambda)$. It is a presentable stable $\infty$-category that 
enhances $\rD(X_{n,\et},\Lambda)$. We then define $\rD(\cX,\Lambda)$ to be 
the homotopy category of $\cD(\cX,\Lambda)$. Strictly speaking, the previous 
diagram is incomplete in the sense that we do not mark all the higher cells 
in the diagram, that is, all natural equivalences of functors, ``equivalences 
between natural equivalences'', etc. In fact, there is an infinite hierarchy 
of (homotopy) equivalences hidden behind the limit of the previous diagram, 
not just the 2-level hierarchy in the classical case. To deal with such kind 
of ``homotopy coherence'' is the major difficulty of the work, that is, we 
need to find a way to encode all such hierarchy \emph{simultaneously} in 
order to make the idea of descent work. In other words, we need to work in 
the {\em totality} of all $\infty$-categories of concern. 

It is possible that such a descent theory (and other relevant 
higher-categorical techniques introduced below) can be realized by using 
other models for higher categories. We have chosen the theory developed by 
Lurie in \cite{HTT}, \cite{HA} for its elegance and availability. Precisely, 
we will use the techniques of the (marked) straightening/unstraightening 
construction, Adjoint Functor Theorem, and the $\infty$-categorical 
Barr--Beck Theorem. Based on Lurie's theory, we develop further 
$\infty$-categorical techniques to treat the homotopy-coherence problem 
mentioned as above. These techniques would enable us to, for example, 
\begin{itemize}
  \item find a coherent way to decompose morphisms (\Sec\ref{a4ss}); 

  \item gluing data from Cartesian diagrams to general ones 
      (\Sec\ref{a5ss}); 

  \item take partial adjoints along given directions 
      (\Sec\ref{b1ss:partial_adjoints}); 

  \item make a coherent choice of descent data 
      (\Sec\ref{b4ss:construction}). 
\end{itemize}
In \Sec\ref{0ss:enhance}, we will have a chance to explain some of them. 

We would also like to remark that Lurie's theory has already been used, for 
example, in \cite{BFN} to study quasi-coherent sheaves on certain (derived) 
stacks with many applications. This work, which studies lisse-\'{e}tale 
sheaves, is another manifestation of the power of Lurie's theory.

\subsection{Results for adic coefficients}
\label{0ss:results_adic}

In this section, we discuss the adic formalism and adic analogues of results 
in \Sec\ref{0ss:results_torsion}. This extends many previous theories on the 
subject, including SGA~5 \cite{SGA5}, Deligne \cite{WeilII}, Ekedahl 
\cite{Ekedahl} (for schemes), Behrend \cite{Behrend} and Laszlo--Olsson 
\cite{LO2}. We prove, among other things, the base change theorem in derived 
categories, which was previous known only on the level of sheaves \cite{LO2} 
(and under other restrictions). Another limitation of the existing theories, 
including those for schemes, is the constructibility assumption. This 
assumption is not often met, for example, when considering morphisms between 
Artin stacks that are only locally of finite type. By contrast, the adic 
formalism developed in this article applies to unrestricted derived 
categories. 

As in \Sec\ref{0ss:results_torsion}, we will state our constructions and 
results only in the classical setting of Artin stacks on the level of usual 
derived categories (which are homotopy categories of the derived 
$\infty$-categories), among other simplifications. See Chapter \ref{c1} for 
the complete results for higher Artin stack higher (and higher 
Deligne--Mumford stacks), stated on the level of stable $\infty$-categories. 

Let $\cX$ be an Artin stack and let $\lambda=(\Xi,\Lambda)$ be a ringed 
diagram, that is, a functor $\Lambda$ from the opposite of a partially 
ordered set $\Xi$ to the category of unital commutative rings. A typical 
example is the projective system 
\[
\dots\to\dZ/\ell^{n+1}\dZ\to\dZ/\ell^n\dZ\to\dots\to\dZ/\ell\dZ,
\]
where $\ell$ is a fixed prime number and the transition maps are natural 
projections. Recall that for every $\xi\in\Xi$, $\rD(\cX,\Lambda(\xi))$ has a 
natural $\infty$-categorical enhancement $\cD(\cX,\Lambda(\xi))$. In fact, 
there is a functor $\rN(\Xi)^{op}\to\Cat$ from the nerve of $\Xi^{op}$ to the 
$\infty$-category of $\infty$-categories sending $\xi$ to 
$\cD(\cX,\Lambda(\xi))$, with the transition functors being (derived) 
extension of scalars. We define 
\[
\cD(\cX,\lambda)_\ra\coloneqq\varprojlim_{\rN(\Xi)^{op}}\cD(\cX,\Lambda(\xi))
\]
and let $\rD(\cX,\lambda)_\ra$ be its homotopy category. It is crucial that 
the limit be taken on the level of $\infty$-categories. 

Let $f\colon\cY\to\cX$ be a morphism of Artin stacks. We then define operations:
\begin{align*}
f^{*\ra}&\colon\rD(\cX,\lambda)_\ra\to\rD(\cY,\lambda)_\ra,\\
f_{*\ra}&\colon\rD(\cY,\lambda)_\ra\to\rD(\cX,\lambda)_\ra,\\
-\atimes_\cX-&\colon\rD(\cX,\lambda)_\ra\times\rD(\cX,\lambda)_\ra\to\rD(\cX,\lambda)_\ra,\\
\HOM^\ra_\cX&\colon\rD(\cX,\lambda)_\ra^{op}\times\rD(\cX,\lambda)_\ra\to\rD(\cX,\lambda)_\ra.
\end{align*}
The pairs $(f^{*\ra},f_{*\ra})$ and $(-\atimes_\cX\sfK,\HOM^\ra_\cX(\sfK,-))$ for every $\sfK\in\rD(\cX,\lambda)_\ra$ are pairs of adjoint functors.

To state the other two operations, we fix a nonempty set $\Box$ of rational 
primes. If $\cX$ and $\cY$ are $\Box$-coprime, $f\colon \cY\to \cX$ is 
locally of finite type, and $\lambda$ is a $\Box$-torsion ringed diagram, 
then there is another pair of adjoint functors: 
\begin{align*}
f_{!\ra}&\colon\rD(\cY,\lambda)_\ra\to\rD(\cX,\lambda)_\ra,\\
f^{!\ra}&\colon\rD(\cX,\lambda)_\ra\to\rD(\cY,\lambda)_\ra.
\end{align*}
Among these functors, $f^{*\ra}$, $f_{!\ra}$ and $-\atimes_\cX-$ are naturally defined from the limit construction of $\rD(-,\lambda)_\ra$. These six operations satisfy the similar properties as in the non-adic version as stated in \Sec\ref{0ss:results_torsion}.

We show that $\rD(\cX,\lambda)_\ra$ is canonically equivalent to the full 
subcategory of $\rD(\cX,\lambda)$ spanned by so-called \emph{adic complexes}, 
which admits a \emph{colocalization functor} 
$\fR_\cX\colon\rD(\cX,\lambda)\to\rD(\cX,\lambda)_\ra$. Moreover, $f^{*\ra}$, 
$f_{!*}$ and $-\atimes_\cX-$ are simply restrictions of $f^*$, $f_!$ and 
$-\otimes_\cX-$, respectively, as they preserve adic complexes. For the other 
three, we have $f_{*\ra}=\fR_\cX\circ f_*$, $f^{!\ra}=\fR_\cY\circ f^!$ and 
$\HOM^\ra_\cX=\fR_\cX\circ\HOM_\cX$.  We refer the reader to 
\Sec\ref{c1ss:adic_properties} and \Sec\ref{c1ss:relation} for more details. 

The adic formalism introduced above does \emph{not} assume the 
constructibility at the first place. In other words, we are free to talk 
about adic complexes for any sheaves. In particular, in terms of 
Grothendieck's fonctions-faisceaux dictionary, we make sense of divergent 
integrals on stacks over finite fields. Those appear for example in 
\cite{FN}. 

In \Sec\ref{c2ss:madic}, we study a special setup, the $\fm$-adic formalism. 
Let $\Lambda$ be a ring and $\fm\subseteq \Lambda$ a principal ideal 
generated by a nonzerodivisor. The pair $(\Lambda,\fm)$ gives rise to a 
ringed diagram $\Lambda_{\bullet}$ with the underlying category 
$\dN=\{0\to1\to2\to\cdots\}$ and $\Lambda_n=\Lambda/\fm^{n+1}$. This setup is 
sufficient for most applications. The $\fm$-adic formalism enjoys very nice 
properties. For example, the adic complexes in this case are stable under the 
six operations.  In \Sec\ref{c2ss:compatibility}, we show that our theory of 
constructible adic formalism coincides with Laszlo--Olsson \cite{LO3} under 
their assumptions.

\subsection{What do we need to enhance?}\label{0ss:enhance}

In Section \ref{0ss:why}, we mentioned the enhancement $\cD(\cX,\Lambda)$ of 
a single triangulated category $\rD(X,\Lambda)$, namely, a stable 
$\infty$-category whose homotopy category (which is an ordinary category) is 
naturally equivalent to $\rD(\cX,\Lambda)$. The enhancement of operations is 
understood in the similar way. For example, the enhancement of $*$-pullback 
for $f\colon\cY\to\cX$ is an exact functor 
\begin{align}\label{0eq:upperstar}
f^*\colon  \cD(\cX,\Lambda)\to\cD(\cY,\Lambda)
\end{align}
such that the induced functor
\[
\rh f^*\colon\rD(\cX,\Lambda)\to\rD(\cY,\Lambda)
\]
is the $*$-pullback functor of usual derived categories.

However, such enhancement is not enough for us to do descent. The reason is that we need to put all schemes and then algebraic spaces together. Let us denote by $\Schqcs$ the category of coproducts of quasi-compact and separated schemes. The enhancement of $*$-pullback for schemes in the strong sense is a functor:
\begin{align}\label{0eq:upperstar_operation}
\EO{\Lambda}{\Schqcs}{*}{}\colon \rN(\Schqcs)^{op}\to\PSL
\end{align}
where $\rN$ denotes the nerve functor (see the definition following \cite{HTT}*{Definition 1.1.2.1}) and $\PSL$ is a certain $\infty$-category of presentable stable $\infty$-categories, which will be specified later. Then \eqref{0eq:upperstar} is just the image of the edge $f\colon \cY\to\cX$ if $f$ belongs to $\Schqcs$. The construction of \eqref{0eq:upperstar_operation} (and its right adjoint which is the enhancement of $*$-pushforward) is not hard, with the help of the general construction in \cite{HA}. The difficulty arises in the enhancement of $!$-pushforward. Namely, we need to construct a functor:
\[
\EO{\Lambda}{\Schqcs}{}{!}\colon \rN(\Schqcs)_F\to\PSL,
\]
where $\rN(\Schqcs)_F$ is the subcategory of $\rN(\Schqcs)$ only allowing morphisms that are locally of finite type. The basic idea is similar to the classical approach: using Nagata compactification theorem. The problem is the following: for a morphism $f\colon Y\to X$ in $\Schqcs$, locally of finite type, we need to choose (non-canonically!) a relative compactification
\begin{align*}
\xymatrix{
Y \ar[d]_-{f} \ar[r]^-{i}
& \overline{Y} \ar[d]^-{\overline{f}}  \\
X & \coprod_I X \ar[l]_-{p},
}
\end{align*}
where $i$ is an open immersion and $\overline{f}$ is proper, and define 
$f_!=p_!\circ \overline{f}_*\circ i_!$ (in the derived sense). It turns out 
that the resulting functor of usual derived categories is independent of the 
choice, up to natural isomorphism. First, we need to upgrade such natural 
isomorphisms to natural equivalences between $\infty$-categories. Second and 
more importantly, we need to ``remember'' such natural equivalences for all 
different compactifications, and even ``equivalences among natural 
equivalences''. We immediately find ourselves in the same scenario of an 
infinite hierarchy of homotopy equivalences again. To handle this kind of 
homotopy coherence, we develop a technique called \emph{multisimplicial 
descent} in \Sec\ref{a4ss}, which can be viewed as an $\infty$-categorical 
generalization of \cite{SGA4}*{Expos\'{e} xvii {\Sec}3.3}. 

This is not the end of the story since our goal is to prove all expected relations among six operations. To use the same idea of descent, we need to ``enhance'' not just operations, but also relations as well. To simplify the discussion, let us temporarily ignore the two binary operations ($\otimes$ and $\HOM$) and consider how to enhance the ``Base Change theorem'' which essentially involves $*$-pullback and $!$-pushforward. We define a simplicial set $\delta^*_{2,\{2\}}\rN(\Schqcs)^\cart_{F,\all}$ in the following way:
\begin{itemize}
  \item The vertices are objects $X$ of $\Schqcs$.

  \item The edges are \emph{Cartesian} diagrams
     \begin{align}\label{0eq:one_cell}
     \xymatrix{
       X_{01} \ar[d]_-{q} \ar[r]^-{g} & X_{00} \ar[d]^-{p} \\
       X_{11} \ar[r]^-{f} & X_{10}   }
     \end{align}
     with $p$ locally of finite type, whose source is $X_{00}$ and target is $X_{11}$.

  \item Simplices of higher dimensions are defined in a similar way.
\end{itemize}
Note that this is \emph{not} an $\infty$-category. Assuming that $\Lambda$ is 
torsion, the enhancement of the Base Change theorem (for $\Schqcs$) is a 
functor 
\begin{equation}\label{0eq:EO*!}
\EO{\Lambda}{\Schqcs}{*}{!}\colon \delta^*_{2,\{2\}}\rN(\Schqcs)^\cart_{F,\all}\to\PSL
\end{equation}
sending the edge
\begin{align*}
\xymatrix{
X_{00} \ar[d] \ar[r]^-{\id} & X_{00} \ar[d]^-{p} \\
X_{11} \ar[r]^-{\id} & X_{11}
}\qquad
\text{resp.\ }
\xymatrix{
X_{11} \ar[d]_-{\id} \ar[r] & X_{00} \ar[d]^-{\id} \\
X_{11} \ar[r]^-{f} & X_{00}
}
\end{align*}
to $p_!\colon \cD(X_{00,\et},\Lambda)\to\cD(X_{11,\et},\Lambda)$ (resp.\ 
$f^*\colon \cD(X_{00,\et},\Lambda)\to\cD(X_{11,\et},\Lambda)$). The upshot is 
that the image of the edge \eqref{0eq:one_cell} is a functor 
$\cD(X_{00,\et},\Lambda)\to\cD(X_{11,\et},\Lambda)$ which is naturally 
equivalent to both $f^*\circ p_!$ and $q_!\circ g^*$. In other words, this 
functor has already encoded the Base Change theorem (for $\Schqcs$) in a 
homotopy coherent way. This allows us to apply the descent method to 
construct the enhancement of the Base Change theorem for Artin stacks, which 
itself includes the enhancement of the four operations $f^*$, $f_*$, $f_!$ 
and $f^!$ by restriction and adjunction. To deal with the homotopy coherence 
involved in the construction of $\EO{\Lambda}{\Schqcs}{*}{!}$, we develop 
another technique called \emph{Cartesian gluing} in \Sec\ref{a5ss}, which can 
be viewed as an $\infty$-categorical variant of \cite{Zh1}*{\Sec6, \Sec7}. 

In fact, the source $\delta^*_{2,\{2\}}\rN(\Schqcs)^\cart_{F,\all}$ of the 
map $\EO{\Lambda}{\Schqcs}{*}{!}$ is categorically equivalent to the 
\emph{$(2,1)$-category of correspondences} $\rN(\Schqcs)_{\corr\colon 
F,\all}$.\footnote{See Example \ref{a4ex:corr} for a precise definition.}  An 
object of $\rN(\Schqcs)_{\corr\colon F,\all}$ is an object of $\Schqcs$. A 
morphism of $\rN(\Schqcs)_{\corr\colon F,\all}$ from $X$ to $Y$ is a 
correspondence 
\[\xymatrix{Y'\ar[r]^g\ar[d]_q & X\\
Y}
\] 
where $g$ and $q$ are morphisms in $\Schqcs$, with $q$ locally of finite 
type. The map $\EO{\Lambda}{\Schqcs}{*}{!}$ \eqref{0eq:EO*!} encoding the 
four operations and the base change theorem can be equivalently formulated as 
a functor 
\[
\EO{\Lambda}{\Schqcs}{}{\corr}\colon\rN(\Schqcs)_{\corr\colon F,\all}\to\PSL
\]
between $\infty$-categories.

We hope that the discussion so far explains the meaning of enhancement to 
some degree. The actual enhancement \eqref{b3eq:operation} constructed in the 
article is more complicated than the ones mentioned previously, since we need 
to include also the information of binary operations, the projection formula 
and extension of scalars.

\subsection{About this work}

As we mentioned at the beginning of the introduction, this article 
amalgamates and improves three preprints we initially posted on the arXiv in 
the years 2012 and 2014. 

During the preparation of this article, Gaitsgory \cite{Gai1} and 
Gaitsgory--Rozenblyum studied operations for ind-coherent sheaves on DG 
schemes and derived stacks in the framework of $\infty$-categories, which was 
later published as the book \cite{GR}. Their work bears some similarity to 
ours, but is in a different setup. In particular, their approach uses 
$(\infty,2)$-categories (see \cite{GR}*{Chapter V}), while we stay in the 
world of $(\infty,1)$-categories. We would like to point out that the 
alternative formulation of our results using the category of correspondences 
(see Example \ref{a4ex:corr} and \Sec\ref{b6ss:correspondences}) was added 
after we learned this concept, due to Lurie, from \cite{Gai1}.

More recently, Mann \cite{Man} improved and simplified our formulation of the 
six operations, while working in the context of rigid-analytic 
geometry.\footnote{His work relies on results from Chapter~\ref{ass} of our 
work.} The readers may consult Lecture II in Scholze's notes \cite{Scho} for 
a comparison of our work, \cite{GR}, and \cite{Man}. 

Since the posting of our work, the $\infty$-categorical techniques developed in this (series of) work have been extensively used to construct (enhanced) six operations in many other contexts. Here is an incomplete list of such examples:
\begin{itemize}
  \item \cite{Man} in the context of rigid-analytic geometry, which has been mentioned above,

  \item \cite{KR21} and \cite{Cho} in the context of (stable) motivic homotopy category for algebraic stacks,

  \item \cite{Park} in the context of Nisnevich sheaves for divided log spaces,

  \item \cite{GHW} in the context of \'{e}tale sheaves on diamonds and v-stacks,

  \item \cite{HP} in the context of Dirac geometry,

  \item \cite{HM} in the context of representation theory.
\end{itemize}

On the other hand, the main outcome of this work -- the enhanced six 
operations for \'{e}tale sheaves on (higher) Artin stacks -- has also been 
found necessary in many works, for example, \cite{BKV}, \cite{ALRR}, 
\cite{HL}, \cite{RS20}, etc. It is worth mentioning that the recent work 
\cite{FYZ} on derived special cycles on the moduli of shtukas uses our result 
for genuinely \emph{higher} Artin stacks.

\subsection{Structure of the article}

The article has three parts. The first part consists of Chapters \ref{ass} and \ref{b1ss}, where we focus on the categorical preparation. The second part consists of Chapters \ref{b2ss}, \ref{b4ss}, \ref{b5ss}, and \ref{b6ss}, where we develop the theory of enhanced six operations for torsion coefficients. The third part consists of Chapters \ref{c1}, \ref{c3}, and \ref{c4}, where we develop the theory of enhanced six operations for adic coefficients, introduce perverse t-structures, and prove some hyperdescent properties.

In Chapter~\ref{ass}, we develop a general technique for gluing subcategories 
of $\infty$-categories. 

In Chapter~\ref{b1ss}, we collect further preliminaries on $\infty$-categories, including the technique of taking partial adjoints (\Sec\ref{b1ss:partial_adjoints}).

In Chapter~\ref{b2ss}, we construct enhanced operation maps for ringed topoi and certain schemes. The enhanced operation maps encode even more information than the enhancement of the Base Change theorem we mentioned in \Sec\ref{0ss:enhance}. We also prove several properties of the maps that are crucial for later constructions.

In Chapter~\ref{b4ss}, we develop an abstract program which we name DESCENT. The program allows us to extend the existing theory to a larger category. It will be run recursively from schemes to algebraic spaces, then to Artin stacks, and eventually to higher Artin or Deligne--Mumford stacks.

In Chapter~\ref{b5ss}, we run the program DESCENT, and prove certain compatibility between our theory and existing ones.

In Chapter~\ref{b6ss}, we write down the resulting six operations for the most general situations and summarize their properties. We also develop a theory of constructible complexes, based on finiteness results of Deligne \cite{SGA4d}*{Th.\ finitude} and Gabber \cite{TG}*{Expos\'{e} XIII}. Finally, we show that our theory is compatible with the work of Laszlo and Olsson \cite{LO1}.

In Chapter~\ref{c1}, we develop the adic formalism for Grothendieck's six 
operations, which includes the most common application, namely, the 
$\ell$-adic one. 

In Chapter~\ref{c3}, we study perverse t-structures for stacks for both torsion and adic coefficients.

In Chapter~\ref{c4}, we study hyperdescent properties for certain operations on stacks for both torsion and adic coefficients.

For more detailed descriptions of the individual chapters, we refer to the beginning of these chapters.

We assume that the reader has some knowledge of Lurie's theory of $\infty$-categories, especially Chapters 1 through 5 of \cite{HTT}, and
Chapters 1 through 4 of \cite{HA}. In particular, we assume that the reader is familiar with basic concepts of simplicial sets \cite{HTT}*{{\Sec}A.2.7}. However, an effort has been made to provide precise references for notation, concepts, constructions, and results used in this article, (at least) at their first appearance.

\subsection{Conventions and notation}
\label{0ss:conventions}

\begin{itemize}
  \item All rings are assumed to be commutative with unity; and ring homomorphisms are assumed to preserve unity.
\end{itemize}

For \emph{set-theoretical issues}:

\begin{itemize}
  \item We fix two (Grothendieck) universes $\cU$ and $\cV$ such that $\cU$ belongs to $\cV$. The adjective \emph{small} means $\cU$-small. In particular, Grothendieck Abelian categories and presentable $\infty$-categories are relative to $\cU$. A topos means a $\cU$-topos.

  \item All rings are assumed to be $\cU$-small. We denote by $\Ring$ the category of rings in $\cU$. By the usual abuse of language, we call
      $\Ring$ the category of $\cU$-small rings.

  \item All schemes are assumed to be $\cU$-small. We denote by $\Sch$ the category of schemes belonging to $\cU$ and by $\Schaff$ the
      full subcategory consisting of affine schemes belonging to $\cU$. There is an equivalence of categories $\Spec \colon (\Ring)^{op}\to
      \Schaff$. The big fppf site on $\Schaff$ is not a $\cU$-site, so that we need to consider prestacks with values in $\cV$. More
      precisely, for $\cW=\cU$ or $\cV$, let $\cS_\cW$ \cite{HTT}*{Definition 1.2.16.1} be the $\infty$-category of spaces in $\cW$. We define the $\infty$-category of prestacks to be $\Fun(\rN(\Schaff)^{op},\cS_\cV)$ \cite{HTT}*{Notation 1.2.7.2}. However, a (higher) Artin stack is assumed to be contained in the essential image of the full subcategory $\Fun(\rN(\Schaff)^{op},\cS_\cU)$. See \Sec\ref{b5ss:artin} for more details.

      The (small) \'etale site of an algebraic scheme and the lisse-\'etale site of an Artin stack are $\cU$-sites.

  \item For every $\cV$-small set $I$, we denote by $\Sset[I]$ the category of $I$-simplicial sets in $\cV$. See also variants in \Sec\ref{a3ss}. We denote by $\Cat$ the (non $\cV$-small) $\infty$-category of $\infty$-categories in $\cV$
      \cite{HTT}*{Definition 3.0.0.1}.\footnote{In \cite{HTT}, $\Cat$ denotes the category of small $\infty$-categories. Thus, our $\Cat$ corresponds more closely to the notation $\widehat{\c{C}\r{at}}_\infty$ in \cite{HTT}*{Remark 3.0.0.5}, where the extension of universes is tacit.} (Multi)simplicial sets and $\infty$-categories are usually tacitly assumed to be $\cV$-small.
\end{itemize}

For \emph{lower categories}:

\begin{itemize}
  \item Unless otherwise specified, a category will be understood as an ordinary category. A $(2,1)$-category $\cC$ is a (strict) $2$-category in which all $2$-cells are invertible, or, equivalently, a category enriched in the category of groupoids. We regard $\cC$ as a simplicial category by taking $\rN(\Map_\cC(X,Y))$ for all objects $X$ and $Y$ of $\cC$.

  \item Let $\cC,\cD$ be two categories. We denote by $\Fun(\cC,\cD)$ the \emph{category of functors} from $\cC$ to $\cD$, whose objects are
      functors and morphisms are natural transformations.

  \item Let $\cA$ be an additive category. We denote by $\Ch(\cA)$ the category of cochain complexes of $\cA$.

  \item Recall that a \emph{partially ordered set} $P$ is an (ordinary) category such that there is at most one arrow (usual denoted as $\leq$) between each pair of objects. For every element $p\in P$, we identify the overcategory $P_{/p}$ (resp.\ undercategory $P_{p/}$) with the full partially ordered subset of $P$ consisting of elements $\leq p$ (resp.\ $\geq p$). For $p,p'\in P$, we identify $P_{p//p'}$ with the full partially ordered subset of $P$ consisting of elements both $\geq p$ and $\leq p'$, which is empty unless $p\leq p'$.

  \item Let $[n]$ be the ordered set $\{0,\dots,n\}$ for $n\geq0$, and put $[-1]=\emptyset$. Let us recall the \emph{category of combinatorial simplices} $\del$ (resp.\ $\del^{\leq n}$, $\del_+$, $\del_+^{\leq n}$). Its objects are the linearly ordered sets $[i]$ for $i\geq0$ (resp.\ $0\leq i\leq n$, $i\geq-1$, $-1\leq i\leq n$) and its morphisms are given by (nonstrictly) order-preserving maps. In particular, for every $n\geq 0$ and $0\leq k\leq n$, there is the face map $d^n_k\colon [n-1]\to[n]$ that is the unique injective map with $k$ not in the image; and the degeneration map $s^n_k\colon[n+1]\to [n]$ that is the unique surjective map such that $s^n_k(k+1)=s^n_k(k)$.
\end{itemize}

For \emph{higher categories}:

\begin{itemize}
  \item As we have mentioned, the word \emph{$\infty$-category} refers to the one defined in \cite{HTT}*{Definition 1.1.2.4}. Throughout the article, an effort has been made to keep our notation consistent with those in \cite{HTT} and \cite{HA}.

  \item For $\cC$ a category, a $(2,1)$-category, a simplicial category, or an $\infty$-category, we denote by $\id_\cC$ the identity functor of $\cC$. We denote by $\rN(\cC)$  the (simplicial) nerve of a (simplicial) category $\cC$ \cite{HTT}*{Definition 1.1.5.5}. We identify $\Ar(\cC)$ (the set of arrows of $\cC$) with $\rN(\cC)_1$ (the set of edges of $\rN(\cC)$) if $\cC$ is a category. Usually, we will not distinguish between $\rN(\cC^{op})$ and $\rN(\cC)^{op}$ for $\cC$ a category, a $(2,1)$-category or a simplicial category.

  \item We denote the homotopy category \cite{HTT}*{Definition 1.1.3.2, Proposition 1.2.3.1} of an $\infty$-category $\cC$ by $\rh\cC$ and we view it as an ordinary category. In other words, we ignore the $\cH$-enrichment of $\rh\cC$.

  \item Let $\cC$ be an $\infty$-category, and $c^\bullet\colon\rN(\del)\to\cC$ (resp.\ $c_\bullet\colon\rN(\del)^{op}\to\cC$) a cosimplicial (resp. simplicial) object of $\cC$. Then the limit \cite{HTT}*{Definition 1.2.13.4} $\lim(c^\bullet)$ (resp.\ colimit or geometric realization $\colim(c_\bullet)$), if it exists, is denoted by $\lim_{n\in\del}c^n$ (resp.\ $\colim_{n\in\del^{op}}c_n$). It is viewed as an object (up to equivalences parameterized by a contractible Kan complex) of $\cC$.

  \item Let $\cC$ be an ($\infty$-)category, and $\cC'\subseteq\cC$ a full subcategory. We say that a morphism $f\colon y\to x$ in $\cC$ is {\em representable in $\cC'$} if for every Cartesian diagram \cite{HTT}*{{\Sec}4.4.2}
      \[
      \xymatrix{
      w \ar[d] \ar[r] & z \ar[d] \\
      y \ar[r]^-{f} & x   }
      \]
      such that $z$ is an object of $\cC'$, $w$ is equivalent to an object of $\cC'$.

  \item We refer the reader to the beginning of \cite{HTT}*{{\Sec}2.3.3} for the terminology \emph{homotopic relative to $A$ over $S$}. We say that $f$ and $f'$ are \emph{homotopic over $S$} (resp.\ \emph{homotopic relative to $A$}) if $A=\emptyset$ (resp.\ $S=*$).

  \item Recall that $\Cat$ is the $\infty$-category of $\cV$-small $\infty$-categories. In \cite{HTT}*{Definition 5.5.3.1}, the subcategories $\PRL,\PRR\subseteq\Cat$ are defined.\footnote{Under our convention, the objects of $\PRL$ and $\PRR$ are the $\cU$-presentable
      $\infty$-categories in $\cV$.} We define subcategories $\PSL,\PSR\subseteq\Cat$ as follows:
      \begin{itemize}
        \item The objects of both $\PSL$ and $\PSR$ are the $\cU$-presentable stable $\infty$-categories in $\cV$ \cite{HTT}*{Definition 5.5.0.1}, \cite{HA}*{Definition 1.1.1.9}.

        \item A functor $F\colon \cC\to\cD$ of presentable stable $\infty$-categories is a morphism of $\PSL$ if and only if $F$ preserves small colimits, or, equivalently, $F$ is a left adjoint functor \cite{HTT}*{Definition 5.2.2.1, Corollary 5.5.2.9(1)}.

        \item A functor $G\colon \cC\to\cD$ of presentable stable $\infty$-categories is a morphism of $\PSR$ if and only if $G$ is accessible and preserves small limits, or, equivalently, $G$ is a right adjoint functor \cite{HTT}*{Corollary 5.5.2.9(2)}.
      \end{itemize}
      We adopt the notation of \cite{HTT}*{Definition 5.2.6.1}: for $\infty$-categories $\cC$ and $\cD$, we denote by $\FunL(\cC,\cD)$ (resp.\ $\FunR(\cC,\cD)$) the full subcategory of $\Fun(\cC,\cD)$ \cite{HTT}*{Notation 1.2.7.2} spanned by left (resp.\ right) adjoint functors. Small limits exist in $\Cat$, $\PRL$, $\PRR$, $\PSL$ and $\PSR$. Such limits are preserved by the natural inclusions $\PSL\subseteq\PRL\subseteq\Cat$ and $\PSR\subseteq\PRR\subseteq\Cat$ by \cite{HTT}*{Proposition 5.5.3.13, Theorem 5.5.3.18} and \cite{HA}*{Theorem 1.1.4.4}.

  \item For a simplicial model category $\bA$, we denote by $\bA^\circ$ the subcategory spanned by fibrant-cofibrant objects.

  \item For the simplicial model category $\Mset$ of marked simplicial sets in $\cV$ \cite{HTT}*{Notation 3.1.0.2} with respect to the Cartesian model structure \cite{HTT}*{Proposition 3.1.3.7, Corollary 3.1.4.4}, we fix a \emph{fibrant replacement simplicial functor}
      \[
      \Fibr\colon\Mset\to(\Mset)^\circ
      \]
      via the Small Object Argument \cite{HTT}*{Proposition A.1.2.5, Remark A.1.2.6}. By construction, it commutes with finite products. If $\cC$ is a $\cV$-small simplicial category \cite{HTT}*{Definition 1.1.4.1}, we let $\Fibr^\cC\colon(\Mset)^\cC\to((\Mset)^\circ)^\cC\subseteq(\Mset)^\cC$ be the induced fibrant replacement simplicial functor with respect to the projective model structure \cite{HTT}*{Remark A.3.3.1}.
\end{itemize}

\subsubsection*{Acknowledgments}

We thank Ofer~Gabber, Rein~Groesbeek, David Hansen, Luc~Illusie, 
Aise~Johan~de~Jong, Jo\"el~Riou, Will~Sawin, Shenghao~Sun, Xiangdong Wu, and 
Xinwen~Zhu for helpful conversations and useful comments. Part of this work 
was done during visits of the first author to the Morningside Center of 
Mathematics, Chinese Academy of Sciences, in Beijing for several times. He 
thanks the Center for its hospitality. Y.~L. was partially supported by (US) 
NSF grants DMS--1302000, DMS--1702019, and a Sloan Research Fellowship. W. Z. 
was partially supported by National Natural Science Foundation of China 
Grants 12125107, 11321101, 11621061, 11688101; China's Recruitment Program of 
Global Experts; Chinese Academy of Sciences Project for Young Scientists in 
Basic Research Grant YSBR-033; National Center for Mathematics and 
Interdisciplinary Sciences and Hua Loo-Keng Key Laboratory of Mathematics, 
Chinese Academy of Sciences.

\section{Gluing restricted nerves of $\infty$-categories}
\label{ass}

The extraordinary pushforward, one of Grothendieck's six operations, in 
\'etale cohomology of schemes was constructed in \cite{SGA4}*{Expos\'{e} 
xvii}. Let $\Sch'$ be the category of quasi-compact and quasi-separated 
schemes, with morphisms being separated of finite type, and let $\Lambda$ be 
a fixed torsion ring. For a morphism $f\colon Y\to X$ in $\Sch'$, the 
extraordinary pushforward by $f$ is a functor 
\begin{align*}
f_!\colon \rD(Y,\Lambda)\to\rD(X,\Lambda),
\end{align*}
between unbounded derived categories of $\Lambda$-modules in the \'{e}tale 
topoi. The functoriality of this operation is encoded by a pseudofunctor 
\begin{align*}
F\colon \Sch'\to\cat
\end{align*}
sending a scheme $X$ in $\Sch'$ to $\rD(X,\Lambda)$ and a morphism $f\colon 
Y\to X$ in $\Sch'$ to the functor $f_!$. Here $\cat$ denotes the 
$(2,1)$-category of categories.\footnote{A $(2,1)$-category is a $2$-category 
in which all $2$-cells are invertible.} There are obvious candidates for the 
restrictions $F_P$ and $F_J$ of $F$ to the subcategories $\Sch'_P$ and 
$\Sch'_J$ of $\Sch'$ spanned respectively by proper morphisms and open 
immersions. The construction of $F$ thus amounts to gluing the two 
pseudofunctors. For this, Deligne developed a general theory for gluing two 
pseudofunctors of target $\cat$ \cite{SGA4}*{Expos\'{e} xvii, {\Sec}3}. 
Deligne's gluing theory, together with its variants 
(\cite{Ayoub}*{{\Sec}1.3}, \cite{Zh1}), have found several other applications 
(\cite{Ayoub}, \cite{CD} and \cite{Zh2}). 

In this chapter, we study the problem of gluing in higher categories. The 
technique developed here can be used to construct Grothendieck's six 
operations in different contexts (see, for example, \cite{Robalo}). In later 
chapters, we use the gluing technique to construct higher categorical six 
operations in \'etale cohomology of higher Artin stacks and prove the base 
change theorem. Even for $1$-Artin stacks and ordinary six operations, this 
theorem was previously only established on the level of sheaves (and subject 
to other restrictions) (\cite{LO1} and \cite{LO2}). Our construction of the 
six operations makes essential use of higher categorical descent, so that 
even if one is only interested in the six operations and base change in 
ordinary derived categories, the enhanced version is still an indispensable 
step of the construction. As a starting point for the descent procedure, we 
need an enhancement of the pseudofunctor $F$ above. In the language of 
$\infty$-categories developed in \cite{HTT}, such an enhancement is a functor 
\begin{align*}
F^{\infty}\colon \rN(\Sch')\to\Cat
\end{align*}
between $\infty$-categories, where $\rN(\Sch')$ is the nerve of $\Sch'$ and 
$\Cat$ denotes the $\infty$-category of $\infty$-categories. For every scheme 
$X$ in $\Sch'$, $F^\infty(X)$ is an $\infty$-category $\cD(X,\Lambda)$, whose 
homotopy category is equivalent to $\rD(X,\Lambda)$. For every morphism 
$f\colon Y\to X$ in $\Sch'$, the image $F^{\infty}(f)$ is a functor 
\[
f_!^{\infty}\colon \cD(Y,\Lambda)\to\cD(X,\Lambda)
\]
such that the induced functor $\rh f_!^{\infty}$ between homotopy categories 
is equivalent to the classical $f_!$. 

One major difficulty of the construction of $F^\infty$ is the need to keep 
track of coherence of all levels. By Nagata compactification \cite{Conrad}, 
every morphism $f$ in $\Sch'$ can be factorized as $p\circ j$, where $j$ is 
an open immersion and $p$ is proper. One can then define $F(f)$ as 
$F_P(p)\circ F_J(j)$. The issue is that such a factorization is not 
canonical, so that one needs to include coherence with composition as part of 
the data. Since the target of $F$ is a $(2,1)$-category, in Deligne's theory 
coherence up to the level of $2$-cells suffices. The target of $F^{\infty}$ 
being an $\infty$-category, we need to consider coherence of \emph{all} 
levels. 

Another complication is the need to deal with more than two subcategories. 
This need is already apparent in \cite{Zh2}. We will give another 
illustration in the proof of Corollary \ref{0co:scheme2} below. 

To handle these complications, we propose the following general framework. 
Let $\cC$ be an (ordinary) category and let $k\geq 2$ be an integer. Let 
$\cE_1,\dots,\cE_k\subseteq\Ar(\cC)$ be $k$ sets of arrows of $\cC$, each 
containing every identity morphism in $\cC$. In addition to the nerve 
$\rN(\cC)$ of $\cC$, we define another simplicial set, which we denote by 
$\delta^*_k\rN(\cC)^\cart_{\cE_1,\dots,\cE_k}$. Its $n$-simplices are 
functors $[n]^k\to\cC$ such that the image of a morphism in the $i$-th 
direction is in $\cE_i$ for $1\leq i\leq k$, and the image of every square in 
direction $(i,j)$ is a Cartesian square (also called pullback square) for 
$1\le i<j\le k$. For example, when $k=2$, the $n$-simplices of 
$\delta^*_2\rN(\cC)^\cart_{\cE_1,\cE_2}$ correspond to diagrams 
\begin{align}\label{0eq:simplex}
\xymatrix{
c_{00} \ar[r]\ar[d] & c_{01} \ar[r]\ar[d] & \dotsb \ar[r] & c_{0n}\ar[d] \\
c_{10} \ar[r]\ar[d] & c_{11} \ar[r]\ar[d] & \dotsb \ar[r] & c_{1n}\ar[d] \\
\vdots \ar[d]       & \vdots \ar[d]       &               & \vdots\ar[d] \\
c_{n0} \ar[r]       & c_{n1} \ar[r]       & \dotsb \ar[r] & c_{nn}
}
\end{align}
where vertical (resp.\ horizontal) arrows are in $\cE_1$ (resp.\ $\cE_2$) and 
all squares are Cartesian. The face and degeneracy maps are defined in the 
obvious way. Note that $\delta^*_k\rN(\cC)^\cart_{\cE_1,\dots,\cE_k}$ is 
\emph{seldom} an $\infty$-category. It is the simplicial set associated to a 
$k$-simplicial set $\rN(\cC)^\cart_{\cE_1,\dots,\cE_k}$. The latter is a 
special case of what we call the \emph{restricted multisimplicial nerve} of 
an ($\infty$-)category with extra data (Definition 
\ref{a3de:restricted_nerve}). 

Let $\cE_0\subseteq \Ar(\cC)$ be a set of arrows stable under composition and 
containing $\cE_1$ and $\cE_2$. Then there is a natural map 
\begin{equation}\label{0eq:equivalence}
g\colon \delta^*_k\rN(\cC)^\cart_{\cE_1,\cE_2,\cE_3,\dots,\cE_k}\to\delta^*_{k-1}\rN(\cC)^\cart_{\cE_0,\cE_3,\dots,\cE_k}
\end{equation}
of simplicial sets, sending an $n$-simplex of the source corresponding to a 
functor $[n]^k\to\cC$, to its partial diagonal 
\[
[n]^{k-1}=[n]\times[n]^{k-2}\xrightarrow{\r{diag}\times\r{id}_{[n]^{k-2}}}[n]^k=[n]^2\times[n]^{k-2}\to\cC,
\]
which is an $n$-simplex of the target. 

We say that a subset $\cE\subseteq\Ar(\cC)$ is \emph{admissible} (Definition 
\ref{a3de:admissible_edge}) if $\cE$ contains every identity morphism, $\cE$ 
is stable under pullback, and for every pair of composable morphisms 
$p\in\cE$ and $q$ in $\cC$, $p\circ q$ is in $\cE$ if and only if $q\in\cE$. 
One main result of this chapter is the following. 

\begin{theorem}[Special case of Theorem \ref{a5th:cartesian_descent}]\label{0th:main}
Let $\cC$ be a category admitting pullbacks and let 
$\cE_0,\cE_1,\dots,\cE_k\subseteq \Ar(\cC)$, $k\ge 2$, be sets of morphisms 
containing every identity morphism and satisfying the following conditions: 
\begin{enumerate}
  \item $\cE_1,\cE_2\subseteq \cE_0$; $\cE_0$ is stable under composition 
      and $\cE_1,\cE_2$ are admissible. 

  \item For every morphism $f$ in $\cE_0$, there exist $p\in\cE_1$ and 
      $q\in\cE_2$ such that $f=p\circ q$. 

  \item For every $3\le i \le k$, $\cE_i$ is stable under pullback by 
      $\cE_1$. 
\end{enumerate}
Then the natural map \eqref{0eq:equivalence} 
\[
g\colon\delta^*_k\rN(\cC)^\cart_{\cE_1,\cE_2,\cE_3,\dots,\cE_k}\to\delta^*_{k-1}\rN(\cC)^\cart_{\cE_0,\cE_3,\dots,\cE_k}
\]
is a \emph{categorical equivalence} (Definition 
\ref{a1de:categorical_equivalence}). 
\end{theorem}

Taking $k=2$ and $\cE_0=\Ar(\cC)$ we obtain the following. 

\begin{corollary}\label{0co:main}
Let $\cC$ be a category admitting pullbacks. Let 
$\cE_1,\cE_2\subseteq\Ar(\cC)$ be admissible subsets. Assume that for every 
morphism $f$ of $\cC$, there exist $p\in\cE_1$ and $q\in\cE_2$ such that 
$f=p\circ q$. Then the natural map 
\[
g\colon \delta^*_2\rN(\cC)^\cart_{\cE_1,\cE_2}\to \rN(\cC)
\]
is a categorical equivalence. 
\end{corollary}

In the situation of Corollary \ref{0co:main}, for every $\infty$-category 
$\cD$, the functor 
\[
\Fun(\rN(\cC),\cD)\to \Fun(\delta^*_2\rN(\cC)^\cart_{\cE_1,\cE_2},\cD)
\]
is an equivalence of $\infty$-categories. We remark that such equivalences 
can be used to construct functors in many different contexts. For instance, 
we can take $\cD$ to be $\rN(\cat)$,\footnote{Here $\rN(\cat)$ denotes the 
simplicial nerve \cite{HTT}*{Definition 1.1.5.5} of $\cat$, the latter 
regarded as a simplicial category.} $\Cat$, or the $\infty$-category of 
differential graded categories. 

In the above discussion, we may replace $\rN(\cC)$ by an $\infty$-category 
$\cC$ (not necessarily the nerve of an ordinary category), and define the 
simplicial set $\delta_k^*\cC_{\cE_1,\dots,\cE_k}^\cart$. Moreover, in later 
application, we need to encode information such as the Base Change 
isomorphism, which involves both pullback and (extraordinary) pushforward. To 
this end, we will define in \Sec\ref{a3ss}, for every subset $L\subseteq 
\{1,\dots,k\}$, a variant $\delta_{k,L}^*\cC^\cart_{\cE_1,\dots,\cE_k}$ of 
$\delta_k^*\cC_{\cE_1,\dots,\cE_k}^\cart$ by ``taking the opposite'' in the 
directions in $L$. For $L\subseteq \{3,\dots, k\}$, the theorem remains valid 
modulo slight modifications. We refer the reader to Theorem 
\ref{a5th:cartesian_descent} for a precise statement. Let us mention in 
passing that there exists a canonical categorical equivalence from the 
simplicial set $\delta^*_{2,\{2\}}\cC^\cart_{\cE_1,\cE_2}$ to the 
$\infty$-category of correspondences introduced in \cite{Gai1}; see Example 
\ref{a4ex:corr}. 

Next we turn to applications to categories of schemes. 

\begin{corollary}\label{0co:scheme1}
Let $P\subseteq \Ar(\Sch')$ be the subset of proper morphisms and let 
$J\subseteq \Ar(\Sch')$ be the subset of open immersions. Then the natural 
map 
\[
\delta_2^*\rN(\Sch')^\cart_{P,J}\to\rN(\Sch')
\]
is a categorical equivalence. 
\end{corollary}

\begin{proof}
This follows immediately from Corollary \ref{0co:main} applied to 
$\cC=\Sch'$, $\cE_1=P$, $\cE_2=J$. 
\end{proof}

As many important moduli stacks are not quasi-compact, later we will work 
with Artin stacks that are not necessarily quasi-compact. Accordingly, we 
need the following variant of Corollary \ref{0co:scheme1}. 

\begin{corollary}\label{0co:scheme2}
Let $\Sch''$ be the category of disjoint unions of quasi-compact and 
quasi-separated schemes, with morphisms being separated and \emph{locally} of 
finite type. Let $F=\Ar(\Sch'')$ be the set of morphisms of $\Sch''$. Let 
$P\subseteq F$ be the subset of proper morphisms, and let $I\subseteq F$ be 
the subset of local isomorphisms \cite{EGAIn}*{D\'efinition 4.4.2}. Then the 
natural map 
\[
\delta_2^*\rN(\Sch'')^\cart_{P,I}\to\rN(\Sch'')
\]
is a categorical equivalence. 
\end{corollary}

Corollary \ref{0co:scheme2} still holds if one replaces $I$ by the subset 
$E\subseteq F$ of \'{e}tale morphisms. 

One might be tempted to apply Corollary \ref{0co:main} by taking $\cE_1=P$, 
$\cE_2=I$. However, the assumption of Corollary \ref{0co:main} does not hold. 
For example, we may take $f$ to be the structural morphism of the disjoint 
union of varieties of unbounded dimensions over a field. 

\begin{proof}[Proof of Corollary \ref{0co:scheme2}]
Put $\cC=\Sch''$. We introduce the following auxiliary sets of morphisms. Let 
$F_{\r{ft}}\subseteq F$ be the set of separated morphisms of finite type, and 
let $I_{\r{ft}}=I\cap F_{\r{ft}}$. Consider the following commutative diagram 
\[
\xymatrix{
\delta_3^*\rN(\cC)^\cart_{P,I_{\r{ft}},I} \ar[r] \ar[d] &  \delta_2^*\rN(\cC)^\cart_{F_{\r{ft}},I} \ar[d] \\
\delta_2^*\rN(\cC)^\cart_{P,I} \ar[r]  & \rN(\cC),
}
\]
where the upper arrow is induced by ``composing morphisms in $P$ and 
$I_{\r{ft}}$'', while the left arrow is induced by ``composing morphisms in 
$I_{\r{ft}}$ and $I$''. We will apply Theorem \ref{0th:main} to all arrows in 
the diagram, except the lower one, to show that they are categorical 
equivalences. It then follows that the lower arrow is also a categorical 
equivalence. 

For the upper arrow, we apply Theorem \ref{0th:main} to $k=3$, 
$\cE_0=F_{\r{ft}}$, $\cE_1=P$, $\cE_2=I_{\r{ft}}$, $\cE_3=I$. Conditions (1) 
and (3) are obviously satisfied. For Condition (2), note that every morphism 
$f$ in $F_{\r{ft}}$ can be written as a disjoint union $\coprod f_i$ of 
morphisms $f_i$ of $\Sch'$. It then suffices to apply Nagata compactification 
to each $f_i$. 

For the left arrow, we apply Theorem \ref{0th:main} to $k=3$, 
$\cE_0=\cE_1=I$, $\cE_2=I_{\r{ft}}$, $\cE_3=P$. All the conditions are 
obviously satisfied. 

For the right arrow, note that the map 
$\delta_2^*\rN(\cC)^\cart_{F_{\r{ft}},I}\to\delta_2^*\rN(\cC)^\cart_{I,F_{\r{ft}}}$ 
given by ``flipping the squares in \eqref{0eq:simplex} along the diagonal'' 
is an isomorphism, which is compatible with the maps to $\rN(\cC)$. Thus, it 
suffices to show that the map 
$\delta_2^*\rN(\cC)^\cart_{I,F_{\r{ft}}}\to\rN(\cC)$ is a categorial 
equivalence. For this, we apply Corollary \ref{0co:main} to ($k=2$, 
$\cE_0=F$,) $\cE_1=I$, $\cE_2=F_{\r{ft}}$. To verify the assumption of 
Corollary \ref{0co:main}, let $f$ be a morphism of $\Sch''$. Then $f$ has the 
form $\coprod_{i,j} X_{ij}\to \coprod_i Y_i$ and is induced by morphisms 
$X_{ij}\to Y_i$, where $X_{ij}$ and $Y_i$ are quasi-compact and 
quasi-separated schemes. Then $f$ is the composition $\coprod_{i,j} 
X_{ij}\xrightarrow{q}\coprod_{i,j}Y_i\xrightarrow{p}\coprod_i Y_i$ with $p\in 
I$ and $q\in F_{\r{ft}}$. 
\end{proof}

The proof of Theorem \ref{0th:main} consists of two steps. Let us illustrate 
them in the case of Corollary \ref{0co:main}. The map $g$ can be decomposed 
as 
\[
\delta_2^*\rN(\cC)_{\cE_1,\cE_2}^\cart\xrightarrow{g'}\delta_2^*\rN(\cC)_{\cE_1,\cE_2}\xrightarrow{g''}\rN(\cC),
\]
where $\delta_2^*\rN(\cC)_{\cE_1,\cE_2}$ is the simplicial set whose 
$n$-simplices are diagrams \eqref{0eq:simplex} without the requirement that 
every square is Cartesian, $g'$ is the natural inclusion and $g''$ is the map 
remembering the diagonal. We prove that both $g'$ and $g''$ are categorical 
equivalences. The fact that $g''$ is a categorical equivalence is an 
$\infty$-categorical generalization of Deligne's result 
\cite{SGA4}*{Expos\'{e} xvii, Proposition 3.3.2}. 

This chapter is organized as follows. In \Sec\ref{a1ss}, we collect some 
basic definitions and facts in the theory of $\infty$-categories \cite{HTT} 
for the reader's convenience. In \Sec\ref{a2ss}, we develop a general 
technique for constructing functors to $\infty$-categories. In 
\Sec\ref{a3ss}, we introduce several notions related to multisimplicial sets 
used in the statements of our main results. In particular, we define the 
restricted multisimplicial nerve of an $\infty$-category with extra data. In 
\Sec\ref{a4ss}, we prove a multisimplicial descent theorem, which implies 
that the map $g''$ is a categorical equivalence.  In \Sec\ref{a5ss}, we prove 
a Cartesian gluing theorem, which implies that the inclusion $g'$ is a 
categorical equivalence. A Cartesian gluing formalism for pseudofunctors 
between $2$-categories was developed in \cite{Zh1}. Our treatment here is 
quite different and more adapted to the higher categorical context. In 
\Sec\ref{a6ss}, we prove some facts about inclusions of simplicial sets used 
in the previous sections.

\subsection{Simplicial sets and $\infty$-categories}
\label{a1ss}

In this section, we collect some basic definitions and facts in the theory of $\infty$-categories developed by Joyal in \cite{Joyal1} and \cite{Joyal2} (who calls them ``quasi-categories'') and Lurie \cite{HTT}.

For $n\ge 0$, we let $[n]$ denote the totally ordered set $\{0,\dots,n\}$ and we put $[-1]\coloneqq\emptyset$. We let $\del$ denote the \emph{category of combinatorial simplices}, whose objects are the totally ordered sets $[n]$ for $n\geq0$ and whose morphisms are given by (non-strictly) order-preserving maps. For $n\geq 0$ and $0\leq k\leq n$, the face map $d^n_k\colon [n-1]\to[n]$ is the unique injective order-preserving map such that $k$ is not in the image; and the degeneracy map $s^n_k\colon [n+1]\to[n]$ is the unique surjective order-preserving map such that $(s^n_k)^{-1}(k)$ has two elements.

\begin{definition}[Simplicial set and $\infty$-category]
We let $\Set$ denote the category of sets.\footnote{More rigorously, $\Set$ is the category of sets in a universe that we fix once and for all.}
\begin{itemize}
  \item We define the category of \emph{simplicial sets}, denoted by $\Sset$, to be the functor category $\Fun(\del^{op},\Set)$. For a simplicial set $S$, we denote by $S_n=S([n])$ its set of $n$-simplices.

  \item For $n\geq0$, we denote by $\Delta^n=\Fun(-,[n])$ the simplicial set represented by $[n]$. We let $\partial\Delta^n\subseteq\Delta^n$ denote the simplicial subset obtained by removing the interior, namely the $n$-simplex defined by $\id_{[n]}\colon [n]\to[n]$. In particular, $\partial\Delta^0=\emptyset$. For each $0\leq k\leq n$, we define the \emph{$k$-th horn} $\Lambda^n_k\subseteq\partial\Delta^n$ to be the simplicial subset obtained by removing the face opposite to the $k$-th vertex, namely
      the $(n-1)$-simplex defined by $d^n_k\colon [n-1]\to [n]$.

  \item An \emph{$\infty$-category} (resp.\ Kan complex) is a simplicial set $\cC$ such that $\cC\to \Delta^0$ has the right lifting property
      with respect to all inclusions $\Lambda^n_k\subseteq\Delta^n$ with $0<k<n$ (resp.\ $0\le k\le n$). In other words, a simplicial set $\cC$ is an $\infty$-category (resp.\ Kan complex) if and only if every map $\Lambda^n_k\to\cC$ with $0<k<n$ (resp.\ $0\le k\le n$) can be extended to a map $\Delta^n\to\cC$.
\end{itemize}
\end{definition}

Note that a Kan complex is an $\infty$-category. The lifting property in the definition of $\infty$-category was first introduced (under the name of ``restricted Kan condition'') by Boardman and Vogt \cite{BV}*{Definition IV.4.8}.

The lifting property defining $\infty$-category (resp.\ Kan complex) can be adapted to the relative case. More precisely, a map $f\colon T\to S$ of simplicial sets is called an \emph{inner fibration} (resp.\ \emph{Kan fibration}) if it has the right lifting property with respect to all
inclusions $\Lambda^n_k\subseteq\Delta^n$ with $0<k<n$ (resp.\ $0\leq k\leq n$). A map $i\colon A\to B$ of simplicial sets is said to be \emph{inner anodyne} (resp.\ \emph{anodyne}) if it has the left lifting property with respect to all inner fibrations (resp.\ Kan fibrations).

\begin{example}[Nerve of an ordinary category]
Let $\cC$ be an ordinary category. The \emph{nerve} $\rN(\cC)$ of $\cC$ is the simplicial set given by $\rN(\cC)_n=\Fun([n],\cC)$. It is easy to see that $\rN(\cC)$ is an $\infty$-category and we can identify $\rN(\cC)_0$ and $\rN(\cC)_1$ with the set of objects $\Ob(\cC)$ and the set of arrows $\Ar(\cC)$, respectively.
\end{example}

Conversely, given a simplicial set $S$, one constructs an ordinary category $\rh S$, the \emph{homotopy category of} $S$ (\cite{HTT}*{Definition 1.1.5.14}, ignoring the enrichment) such that $\Ob(\rh S)=S_0$. For an $\infty$-category $\cC$, $\Hom_{\rh\cC}(x,y)$ consists of homotopy classes of edges $x\to y$ in $\cC_1$ \cite{HTT}*{Proposition 1.2.3.9}. By \cite{HTT}*{Proposition 1.2.3.1}, $\rh$ is left adjoint to the nerve functor $\rN$.

\begin{definition}[Object, morphism, equivalence]
Let $\cC$ be an $\infty$-category. Vertices of $\cC$ are called \emph{objects} of $\cC$ and edges of $\cC$ are called \emph{morphisms} of $\cC$. A morphism of $\cC$ is called an \emph{equivalence} if it defines an isomorphism in the homotopy category $\rh \cC$.
\end{definition}

The category $\Sset$ is Cartesian-closed. For objects $S$ and $T$ of $\Sset$, we let $\Map(S,T)$ denote the internal mapping object defined by
\[
\Hom_{\Sset}(K,\Map(S,T))\simeq \Hom_{\Sset}(K\times S,T).
\]
If $\cC$ is an $\infty$-category, we write $\Fun(S,\cC)$ instead of $\Map(S,\cC)$. One can show that $\Fun(S,\cC)$ is an $\infty$-category \cite{HTT}*{Proposition 1.2.7.3(1)} (see also \cite{HTT}*{Corollary 2.3.2.5}).

\begin{definition}[Functor, natural transformation, natural equivalence]
Objects of $\Fun(S,\cC)$ are called \emph{functors} $S\to \cC$, morphisms of $\Fun(S,\cC)$ are called \emph{natural transformations}, and equivalences in $\Fun(S,\cC)$ are called \emph{natural equivalences}.
\end{definition}

\begin{remark}
Let $f,g\colon S\to\cC$ be functors and $\phi\colon f\to g$ a natural transformation. Then $\phi$ is a natural equivalence if and only if for every vertex $s$ of $S$, the morphism $\phi(s)\colon f(s)\to g(s)$ is an equivalence in $\cC$. We refer the reader to \cite{HTT}*{Proposition 3.1.2.1} for a generalization (see \cite{HTT}*{Remark 2.4.1.4}).
\end{remark}

\begin{remark}\label{a1re:homotopy}
Let $\cC$ be an $\infty$-category and let $f,g\colon x\to y$ be morphisms of $\cC$. Then $f$ and $g$ are homotopic (namely, having the same image in $\rh\cC$) if and only if they are equivalent when viewed as objects of the $\infty$-category defined by the fiber of the map $\Fun(\Delta^1,\cC)\to\Fun(\partial \Delta^1,\cC)$. Indeed, the latter condition means that there exist a morphism $h\colon x\to y$ and two $2$-simplices of $\cC$ as shown in the diagram
\[
\xymatrix{
x\ar[d]_{\id_x}\ar[r]^f\ar[rd]^h & y\ar[d]^{\id_y}\\x\ar[r]^g&y.
}
\]
By definition (resp.\ \cite{HTT}*{Remark 1.2.3.6}), the existence of the $2$-simplex in the upper right (resp.\ lower left) corner means that $f$ (resp.\ $g$) and $h$ are homotopic. This proves the ``if'' part. For the ``only if'' part, it suffices to take $h=g$ and to take the $2$-simplex in the lower left corner to be degenerate.
\end{remark}

We now recall the notion of categorical equivalence of simplicial sets, which is essential to our article. There are several equivalent definitions of categorical equivalence. The one given below (equivalent to \cite{HTT}*{Definition 1.1.5.14} in view of \cite{HTT}*{Proposition
2.2.5.8}), due to Joyal \cite{Joyal2}, will be used in the proofs of our theorems.

\begin{definition}[Categorical equivalence]\label{a1de:categorical_equivalence}
A map $f\colon T\to S$ of simplicial sets is a \emph{categorical equivalence} if for every $\infty$-category $\cC$, the induced functor
\[
\rh\Fun(S,\cC)\to\rh\Fun(T,\cC)
\]
is an equivalence of ordinary categories.
\end{definition}

If $f\colon T\to S$ is a categorical equivalence, then the induced functor $\rh T\to \rh S$ is an equivalence of ordinary categories. An inner anodyne map is a categorical equivalence \cite{HTT}*{Lemma 2.2.5.2}. The category $\Sset$ admits the Joyal model structure \cite{HTT}*{Theorem 2.2.5.1}, for which weak equivalences are precisely categorical equivalences.

\begin{remark}
Let $\cC$ and $\cD$ be $\infty$-categories. A functor $f\colon \cC\to \cD$ is a categorical equivalence if and only if there exist a functor $g\colon\cD\to \cC$ and natural equivalences between $f\circ g$ and $\id_\cD$ and between $g\circ f$ and $\id_{\cC}$. Indeed, the ``only if'' direction follows from \cite{HTT}*{Proposition 1.2.7.3} and the other direction is clear.
\end{remark}

The following criterion of categorical equivalence will be used in the proofs of our theorems. Given maps of simplicial sets $v,v'\colon Y\to X$ and an inner fibration $p\colon X\to S$ such that $p\circ v=p\circ v'$, we say that $v$ and $v'$ are \emph{homotopic over $S$} if they are equivalent when viewed as objects of the $\infty$-category defined by the fiber of the inner fibration $\Map(Z,X)\to \Map(Z,S)$ induced by $p$.

\begin{lem}\label{a1le:categorical_equivalence}
A map of simplicial sets $f\colon Y\to Z$ is a categorical equivalence if and only if the following conditions are satisfied for every
$\infty$-category $\cD$:
\begin{enumerate}
  \item For every $l=0,1$ and every commutative diagram
     \[
     \xymatrix{
     Y\ar[r]^-v\ar[d]_f & \Fun(\Delta^l,\cD)\ar[d]^p\\
     Z\ar[r]^-w & \Fun(\partial \Delta^l,\cD)
     }
     \]
     where $p$ is induced by the inclusion $\partial \Delta^l\subseteq\Delta^l$, there exists a map $u\colon Z\to \Fun(\Delta^l,\cD)$ satisfying $p\circ u=w$ such that $u\circ f$ and $v$ are homotopic over $\Fun(\partial \Delta^l,\cD)$.

\item For $l=2$ and every commutative diagram as above, there exists a map $u\colon Z\to \Fun(\Delta^l,\cD)$ satisfying $p\circ u=w$.
\end{enumerate}
\end{lem}

\begin{proof}
By definition, that $f$ is a categorical equivalence means that for every $\infty$-category $\cD$, the functor
\[
F\colon \rh \Fun(Z,\cD)\to \rh\Fun(Y,\cD)
\]
induced by $f$ is an equivalence of categories. We show that the conditions for $l=0,1,2$ mean that $F$ is essentially surjective, full, and faithful, respectively. For $l=0$, this is clear. For $l=1$, this follows from Remark \ref{a1re:homotopy}. For $l=2$, the condition means that for functors $g_0,g_1,g_2\colon Z\to \cD$, and natural transformations $\phi\colon g_0\to g_1$, $\psi\colon g_1\to g_2$, $\chi\colon g_0\to g_2$ such that $F([\psi]\circ [\phi])=F([\chi])$, we have $[\psi]\circ [\phi]=[\chi]$. Here $[\phi]$, $[\psi]$, $[\chi]$ denote the homotopy classes of $\phi$, $\psi$, $\chi$, respectively. The condition is clearly satisfied if $F$ is faithful. Conversely, if $F$ is faithful, it suffices to take $g_1=g_2$ and $\psi=\id$.
\end{proof}

In \Sec\ref{a3ss}, we will introduce the notion of multi-marked simplicial sets, which generalizes the notion of marked simplicial sets in
\cite{HTT}*{Definition 3.1.0.1}. Since marked simplicial sets play an important role in many arguments for $\infty$-categories, we recall its
definition.

\begin{definition}[Marked simplicial set]
A \emph{marked simplicial set} is a pair $(X,\cE)$ where $X$ is a simplicial set and $\cE\subseteq X_1$ is a subset containing all degenerate edges. A morphism $f\colon (X,\cE)\to (X',\cE')$ of marked simplicial sets is a map $f\colon X\to X'$ of simplicial sets satisfying $f(\cE)\subseteq\cE'$. We let $\Mset$ denote the category of marked simplicial sets.
\end{definition}

The forgetful functor $F \colon \Mset\to \Sset$ carrying $(X,\cE)$ to $X$ admits a right adjoint carrying a simplicial set $S$ to $S^\sharp=(S,S_1)$ and a left adjoint carrying $S$ to $S^\flat=(S,\cE)$, where $\cE$ is the set of all degenerate edges. For an $\infty$-category $\cC$, we let $\cC^\natural$ denote the marked simplicial set $(\cC,\cE)$, where $\cE$ is the set of all edges of $\cC$ that are equivalences. The category $\Mset$ is equipped with the \emph{Cartesian model structure} \cite{HTT}*{Proposition 3.1.3.7}. The adjoint pair $((-)^\flat, F)$ is a Quillen equivalence between the Joyal model structure on $\Sset$ and the Cartesian model structure on $\Mset$ \cite{HTT}*{Theorem 3.1.5.1}.

The category $\Mset$ is Cartesian-closed. For objects $X$ and $Y$ of $\Mset$, we let $\Map^\flat(X,Y)$ denote the underlying simplicial set of
the internal mapping object $Y^X$. We let $\Map^\sharp(X,Y)\subseteq\Map^\flat(X,Y)$ denote the largest simplicial subset such that
$\Map^\sharp(X,Y)^\sharp\subseteq Y^X$. If $\cC$ is an $\infty$-category, then $\Map^\flat(X,\cC^\natural)$ is an $\infty$-category and
$\Map^\sharp(X,\cC^\natural)$ is the largest Kan complex \cite{HTT}*{Proposition 1.2.5.3} contained in $\Map^\flat(X,\cC^\natural)$
\cite{HTT}*{Remark 3.1.3.1} (see also \cite{HTT}*{Lemma 3.1.3.6}), so that $(\cC^\natural)^X=\Map^\flat(X,\cC^\natural)^\natural$.

\subsection{Constructing functors via the category of simplices}
\label{a2ss}

In this section, we develop a general technique for constructing functors to $\infty$-categories, which is the key to several constructions in this article and its sequels. For a functor $F\colon K\to \cC$ from a simplicial set $K$ to an $\infty$-category $\cC$, the image $F(\sigma)$ of a simplex $\sigma$ of $K$ is a simplex of $\cC$, functorial in $\sigma$. Here we address the problem of constructing $F$ when, instead of having a canonical choice for $F(\sigma)$, one has a weakly contractible simplicial set $\cN(\sigma)$ of candidates for $F(\sigma)$.

We start with some generalities on diagrams of simplicial sets. Let $\cI$ be a (small) ordinary category. We consider the injective model structure on the functor category $(\Sset)^\cI\coloneqq\Fun(\cI,\Sset)$. We say that a morphism $i\colon \cN\to \cM$ in $(\Sset)^\cI$ is \emph{anodyne} if $i(\sigma)\colon \cN(\sigma)\to \cM(\sigma)$ is anodyne for every object $\sigma$ of $\cI$. We say that a morphism $\cR\to \cR'$ in $(\Sset)^\cI$ is an \emph{injective fibration} if it has the right lifting property with respect to every anodyne morphism $\cN\to \cM$ in $(\Sset)^\cI$. We say that an object $\cR$ of $(\Sset)^\cI$ is \emph{injectively fibrant} if the morphism from $\cR$ to the final object $\Delta^0_\cI$ is an injective fibration. The right adjoint of the diagonal functor $\Sset\to (\Sset)^\cI$ is the global section functor
\[
\Gamma\colon (\Sset)^\cI\to \Sset, \quad
\Gamma(\cN)_q=\r{Hom}_{(\Sset)^\cI}(\Delta^q_\cI,\cN),
\]
where $\Delta^q_\cI\colon \cI \to \Sset$ is the constant functor of value $\Delta^q$.

\begin{notation}
Let $\Phi\colon \cN\to \cR$ be a morphism of $(\Sset)^\cI$. We let $\Gamma_\Phi(\cR)\subseteq \Gamma(\cR)$ denote the simplicial subset, union
of the images of $\Gamma(\Psi)\colon \Gamma(\cM)\to \Gamma(\cR)$ for all factorizations
\[
\cN\xrightarrow{i}\cM\xrightarrow{\Psi}\cR
\]
of $\Phi$ such that $i$ is anodyne.
\end{notation}

\begin{remark}\label{a2re:components}
As the referee pointed out, $\Gamma_\Phi(\cR)$ can be computed using one single factorization
\[
\cN\xrightarrow{i'}\cM'\xrightarrow{\Psi'}\cR
\]
of $\Phi$, where $i'$ is anodyne and $\Psi'$ is an injective fibration. For every factorization $\cN\xrightarrow{i}\cM\xrightarrow{\Psi} \cR$ of $\Phi$ such that $i$ is anodyne, there exists a dotted arrow rendering the diagram
\[
\xymatrix{
\cN\ar[r]^{i'}\ar[d]_i & \cM'\ar[d]^{\Psi'}\\
\cM\ar[r]^{\Psi}\ar@{..>}[ru] & \cR
}
\]
commutative. Thus $\Gamma_\Phi(\cR)$ is simply the image of $\Gamma(\Psi')$. Since $\Gamma(\Psi')$ is a Kan fibration, $\Gamma_\Phi(\cR)$ is a union of connected components of $\Gamma(\cR)$. Indeed, the inclusion $\Gamma_\Phi(\cR)\subseteq \Gamma(\cR)$ satisfies the right lifting property with respect to the inclusion $\Delta^{\{j\}}\subseteq\Delta^n$ for all $0\le j\le n$.
\end{remark}

\begin{remark}
By definition, the map $\Gamma(\Phi)\colon \Gamma(\cN)\to\Gamma(\cR)$ factorizes through $\Gamma_\Phi(\cR)$. The construction of $\Gamma_\Phi(\cR)$ enjoys the following functoriality. For a commutative diagram
\[
\xymatrix{
\cN\ar[r]^\Phi\ar[d]_G & \cR\ar[d]^F \\ \cN' \ar[r]^{\Phi'} & \cR'
}
\]
in $(\Sset)^\cI$, the map $\Gamma(F)\colon \Gamma(\cR)\to \Gamma(\cR')$ carries $\Gamma_\Phi(\cR)$ into $\Gamma_{\Phi'}(\cR')$. Indeed, for every factorization $\cN\xrightarrow{i}\cM\xrightarrow{\Psi} \cR$ of $\Phi$ such that $i$ is anodyne, we have a commutative diagram
\[
\xymatrix{
\cN\ar[r]^i\ar[d]_G & \cM\ar[r]^\Psi\ar[d] & \cR\ar[d]^F\\
\cN'\ar[r]^{i'}&\cM'\ar[r]^{\Psi'} & \cR',
}
\]
where $i'$ is the pushout of $i$ by $G$, hence is anodyne.

For a functor $g\colon \cI'\to \cI$, composition with $g$ induces a functor $g^*\colon (\Sset)^\cI\to (\Sset)^{\cI'}$. By a slight abuse of notation, we still denote by $g^*\colon \Gamma(\cR)\to \Gamma(g^*\cR)$ the pullback map induced by the functor $g^*$. Then the pullback map $g^*$ carries $\Gamma_\Phi(\cR)$ into $\Gamma_{g^*\Phi}(g^*\cR)$. Indeed, for every factorization $\cN\xrightarrow{i}\cM\xrightarrow{\Psi} \cR$ of $\Phi$ such that $i$ is anodyne, $g^*\cN\xrightarrow{g^*i}g^*\cM\xrightarrow{g^*\Psi}g^*\cR$ is a factorization of $g^*\Phi$ such that $g^* i$ is anodyne, and we have the following commutative diagram
\[
\xymatrix{
\Gamma(\cM)\ar[d]_{g^*}\ar[rr]^{\Gamma(\Psi)}&&\Gamma(\cR)\ar[d]^{g^*}\\
\Gamma(g^*\cM)\ar[rr]^{\Gamma(g^*\Psi)}&& \Gamma(g^*\cR).
}
\]
\end{remark}

Our construction technique relies on the following property of $\Gamma_\Phi(\cR)$. In a previous draft of this article, the statement of
part (1) in the following lemma was incorrect. We thank the referee for suggesting the following correction.

\begin{lem}\label{a1le:component}
Let $\cI$ be a category. Let $\cN$, $\cR$ be objects of $(\Sset)^\cI$ such that $\cN(\sigma)$ is weakly contractible for all objects $\sigma$ of $\cI$ and $\cR$ is injectively fibrant.
\begin{enumerate}
  \item For every morphism $\Phi\colon \cN\to \cR$, the simplicial set $\Gamma_\Phi(\cR)$ is nonempty and connected, hence a connected
    component of $\Gamma(\cR)$.

  \item For homotopic morphisms $\Phi,\Phi'\colon \cN\to \cR$, we have $\Gamma_\Phi(\cR)=\Gamma_{\Phi'}(\cR)$.
\end{enumerate}
\end{lem}

The condition in (2) means that there exists a morphism $H\colon\Delta^1_\cI\times \cN\to \cR$ such that $H\res \Delta^{\{0\}}_\cI\times\cN=\Phi$ and $H\res \Delta^{\{1\}}_\cI\times \cN=\Phi'$. Note that $\Gamma(\cR)$ is a Kan complex.

\begin{proof}
(1) We apply Remark \ref{a2re:components}. Since the morphism $\cM'\to\Delta^0_\cI$ is a trivial fibration, $\Gamma(\cM')$ is a contractible Kan
complex. Therefore its image $\Gamma_\Phi(\cR)$ is nonempty and connected, hence a connected component of $\Gamma(\cR)$.

(2) We define an object $\cN^\triangleright$ by $\cN^\triangleright(\sigma)=\cN(\sigma)^\triangleright$ \cite{HTT}*{Notation 1.2.8.4}. Since the inclusion $\Delta^1_\cI\times \cN\hookrightarrow\Delta^1_\cI\times \cN^\triangleright$ is anodyne and $\cR$ is injectively fibrant, we can find a morphism $H'$ as shown in the diagram
\[
\xymatrix{
\Delta_\cI^1\times\cN\ar[r]^-{H}\ar@{^(->}[d]& \cR\\
\Delta_\cI^1\times\cN^{\triangleright}\ar@{..>}[ur]_{H'}
}
\]
rendering the diagram commutative. We denote by $h\colon \Delta_\cI^1\to\cR$ the restriction of $H'$ to the cone point of $\cN^{\triangleright}$, corresponding to an edge of $\Gamma(\cR)$. Then $h(0)$ belongs to $\Gamma_\Phi(\cR)$ and $h(1)$ belongs to $\Gamma_{\Phi'}(\cR)$. Since $\Gamma_\Phi(\cR)$ and $\Gamma_{\Phi'}(\cR)$ are connected components of $\Gamma(\cR)$ by (1), we have $\Gamma_\Phi(\cR)=\Gamma_{\Phi'}(\cR)$.
\end{proof}

Let $K$ be a simplicial set. The \emph{category of simplices of} $K$, which we denote by $\del_{/K}$ following \cite{HTT}*{Notation 6.1.2.5}, plays a key role in our construction technique. Recall that $\del_{/K}$ is the strict fiber product $\del\times_{\Sset}(\Sset)_{/K}$. An object of $\del_{/K}$ is a pair $(n,\sigma)$, where $n\geq0$ is some integer and $\sigma\in\Hom_{\Sset}(\Delta^n,K)$. A morphism $(n,\sigma)\to (n',\sigma')$ is a map $d\colon \Delta^n\to \Delta^{n'}$ such that $\sigma=\sigma'\circ d$. Note that $d$ is a monomorphism (resp.\ epimorphism) if and only if the underlying map $[n]\to [n']$ is injective (resp.\ surjective). Every epimorphism of $\del_{/K}$ is split. Moreover, $\del_{/K}$ admits pushouts of epimorphisms by epimorphisms. In what follows, we sometimes simply write $\sigma$ for an object of $\del_{/K}$ if $n$ is insensitive.

The usefulness of $\del_{/K}$ is demonstrated by the following lemma.

\begin{lem}[\cite{Hovey}*{Lemma 3.1.3}]\label{a2le:Hovey}
The maps $\sigma\colon \Delta^n\to K$ exhibit $K$ as the colimit of the functor $\del_{/K}\to\Sset$ carrying $(n,\sigma)$ to $\Delta^n$.
\end{lem}

\begin{proof}
We include a proof for completeness. Let $X$ be the colimit. Given $m\geq 0$, the set $X_m$ is the colimit of the functor $F_m\colon\del_{/K}\to\Set$ carrying $(n,\sigma)$ to $(\Delta^n)_m\simeq\Hom_{\Sset}(\Delta^m,\Delta^n)$. We denote by $\del_{/K}^{[m]}$ the
category of elements of $F_m$. Objects of $F_m$ are triples
\[
(n,\sigma,\tau)\colon\Delta^m\xrightarrow{\tau}\Delta^n\xrightarrow{\sigma}K,
\]
and morphisms $(n,\sigma,\tau)\to(n',\sigma',\tau')$ are commutative diagrams
\[
\xymatrix{
\Delta^m \ar[r]^-{\tau}\ar@{=}[d] & \Delta^n\ar[r]^{\sigma}\ar[d]^-{d} & K \ar@{=}[d] \\
\Delta^m \ar[r]^-{\tau'} & \Delta^{n'}\ar[r]^{\sigma'}  & K.
}
\]
Note that $\del_{/K}^{[m]}$ is a disjoint union of categories indexed by $\rho=\sigma\tau\in K_m$, each admitting an initial object
\[
(m,\rho,\id_{\Delta^m})\colon \Delta^m\xrightarrow{\id} \Delta^m\xrightarrow{\rho} K.
\]
The lemma then follows from the fact that the colimit of any functor $F\colon\cC\to \Set$ from a category $\cC$ to $\Set$ can be identified with the set of connected component of the category of elements of $F$.
\end{proof}

\begin{notation}
We define a functor $\Map[K,-]\colon \Mset\to (\Sset)^{(\del_{/K})^{op}}$ as follows. For a marked simplicial set $M$, we define $\Map[K,M]$  by
\[
\Map[K,M](n,\sigma)=\Map^\sharp((\Delta^n)^\flat,M),
\]
for every object $(n,\sigma)$ of $\del_{/K}$. A morphism $d\colon(n,\sigma)\to (n',\sigma')$ in $\del_{/K}$ goes to the natural restriction
map $\r{Res}^d\colon \Map^\sharp((\Delta^{n'})^\flat,M)\to\Map^\sharp((\Delta^n)^\flat,M)$. For an $\infty$-category $\cC$, we set
$\Map[K,\cC]=\Map[K,\cC^\natural]$.
\end{notation}

The following remark shows how $\Map[K,-]$ is related with the problem of constructing functors.

\begin{remark}\label{a2re:functors}
The map
\[
\Map^\sharp(K^\flat,M)\to \Gamma(\Map[K,M])
\]
induced by the restriction maps $\Map^\sharp(K^\flat,M)\to\Map^\sharp((\Delta^n)^\flat,M)$ is an isomorphism of simplicial sets. Indeed, the set of $m$-simplices of $\Map^\sharp(K^\flat,M)$ can be identified with $\Hom_{\Mset}((\Delta^m)^\sharp\times K^\flat,M)$, while the set of $m$-simplices of $\Gamma(\Map[K,M])$ is a limit of the functor $\del_{/K}\to \Set$ carrying $(n,\sigma)$ to
$\Hom_{\Mset}((\Delta^m)^\sharp\times (\Delta^n)^\flat,M)$. We are thus reduced to showing that the maps $\sigma\colon \Delta^n\to K$ exhibit
$(\Delta^m)^\sharp \times K^\flat$ as the colimit of the functor $\del_{/K}\to \Sset$ carrying $(n,\sigma)$ to $(\Delta^m)^\sharp\times(\Delta^n)^\flat$. Note that the functor $(\Delta^m)^\sharp\times(-)^\flat\colon \Sset\to \Mset$ admits a right adjoint $\Map^\flat((\Delta^m)^\sharp,-)$, hence preserves colimits. The assertion then follows from Lemma \ref{a2le:Hovey}.

Note that $\Map^\sharp(K^\flat,\cC^\natural)$ is the largest Kan complex contained in $\Fun(K,\cC)$.
\end{remark}

If $g\colon K'\to K$ is a map, then composition with the functor $\del_{/K'}\to \del_{/K}$ induced by $g$ defines a functor $g^*\colon(\Sset)^{(\del_{/K})^{op}}\to (\Sset)^{(\del_{/K'})^{op}}$. We have $g^*\Map[K,M]=\Map[K',M]$.

\begin{proposition}\label{a1le:injective_fibration}
Let $f\colon Z\to T$ be a fibration in $\Mset$ with respect to the Cartesian model structure, and let $K$ be a simplicial set. Then the morphism $\Map[K,f]\colon \Map[K,Z]\to \Map[K,T]$ is an injective fibration in $(\Sset)^{(\del_{/K})^{op}}$. In other words, for every commutative square in $(\Sset)^{(\del_{/K})^{op}}$ of the form
\[
\xymatrix{
\cN\ar[r]^-\Phi\ar@{^{(}->}[d] & \Map[K,Z]\ar[d]^-{\Map[K,f]}\\
\cM\ar[r]_-\Psi\ar@{..>}[ur]^{\Omega} & \Map[K,T]
}
\]
such that $\cN\hookrightarrow \cM$ is anodyne, there exists a dotted arrow as indicated, rendering the diagram commutative.
\end{proposition}

The proof of this proposition will be given after Remark \ref{a1re:simplices}.

\begin{corollary}\label{a2co:key}
Let $f\colon Z\to T$ be a fibration in $\Mset$ with respect to the Cartesian model structure, $K$ a simplicial set, $a\colon K^\flat\to T$ a map, and $\cN\in (\Sset)^{(\del_{/K})^{op}}$ such that $\cN(\sigma)$ is weakly contractible for all $\sigma\in \del_{/K}$. We let $\Map[K,f]_a$ denote the fiber of $\Map[K,f]\colon \Map[K,Z]\to \Map[K,T]$ at the section $\Delta_K^0\to \Map[K,T]$ corresponding to $a$.
\begin{enumerate}
  \item For every morphism $\Phi\colon \cN\to \Map[K,f]_a$, the simplicial set $\Gamma_\Phi(\Map[K,f]_a)$ is a (nonempty) connected component of $\Gamma(\Map[K,f]_a)$.

  \item For homotopic $\Phi,\Phi'\colon \cN\to \Map[K,f]_a$, we have
      \[
      \Gamma_\Phi(\Map[K,f]_a)=\Gamma_{\Phi'}(\Map[K,f]_a).
      \]
\end{enumerate}
\end{corollary}

The condition in (2) means that there exists a morphism $H\colon\Delta_K^1\times \cN\to \Map[K,f]_a$ in $(\Sset)^{(\del_{/K})^{op}}$ such
that $H\res \Delta_K^{\{0\}}\times \cN=\Phi$, $H\res \Delta_K^{\{1\}}\times\cN=\Phi'$.

\begin{proof}
By Proposition \ref{a1le:injective_fibration}, if $\cC$ is an $\infty$-category, then $\Map[K,\cC]$ is an injectively fibrant object of $(\Sset)^{(\del_{/K})^{op}}$ since $\cC^\natural$ is a fibrant object of $\Mset$ \cite{HTT}*{Proposition 3.1.4.1}. Then the corollary follows from Lemma \ref{a1le:component} applied to $\cR=\Map[K,f]_a$.
\end{proof}

\begin{remark}\label{a2re:Quillen}
The functor $\Map[K,-]$ admits a left adjoint $F\colon(\Sset)^{(\del_{/K})^{op}}\to \Mset$ carrying $\cR$ to the coend of the
diagram
\[
(\del_{/K})^{op}\times \del_{/K}\to \Mset,\quad
((n,\sigma),(m,\tau))\mapsto\cR(n,\sigma)^\sharp\times (\Delta^m)^\flat.
\]
The functor $F$ can be described more explicitly as follows. Note that for functors $G\colon \cC^{op}\to \Set$, $H\colon \cC\to \Set$, where $\cC$ is a category, the coend of the diagram
\[
\cC^{op}\times \cC\to \Set,\quad (A,B)\mapsto G(A)\times H(B)
\]
can be identified with the colimit of the functor $\cD^{op}\to\Set$ carrying $(A,h)$ to $G(A)$, where $\cD$ is the category of elements of $H$.
Thus if we write $F\cR=(X,\cE)$, then $X_m$ is the colimit of the functor $(\del_{/K}^{[m]})^{op}\to \Set$ carrying $(n,\sigma,\tau)$ to
$\cR(n,\sigma)_m$, where $\del_{/K}^{[m]}$ is the category defined in the proof of Lemma \ref{a2le:Hovey}. Therefore, $X_m$ is the disjoint union of $\cR(m,\sigma)_m$ for all $m$-simplices $\sigma\colon \Delta^m\to K$. Moreover, $\cE\subseteq X_1$ is the union of $\cR(1,\sigma)_1$ for all
degenerate edges $\sigma\colon \Delta^1\to K$.

It follows from the above description that $F$ preserves monomorphisms. Thus Proposition \ref{a1le:injective_fibration} shows that the pair
$(F,\Map[K,-])$ is a Quillen adjunction between $\Mset$ endowed with the Cartesian model structure and $(\Sset)^{(\del_{/K})^{op}}$ endowed with the injective model structure.
\end{remark}

\begin{remark}\label{a1re:simplices}
If we replace $\del_{/K}$ by the full subcategory $\del_{/K}^{\r{nd}}$ spanned by nondegenerate simplices, then Proposition
\ref{a1le:injective_fibration} still holds and the proof becomes simpler. However, $\del_{/K}^{\r{nd}}$ is only functorial with respect to
monomorphisms of simplicial sets, which is insufficient for our applications.
\end{remark}

\begin{proof}[Proof of Proposition \ref{a1le:injective_fibration}]
For $n\ge 0$, we let $\cI_n$ denote the full subcategory of $\del_{/K}$ spanned by $(m,\sigma)$ for $m\le n$. We construct $\Omega\res \cI_n^{op}$ by induction on $n$. It suffices to construct, for every $\sigma\colon\Delta^n\to K$, a map $\Omega(n,\sigma)$ as the dotted arrow rendering the following diagram commutative
\[
\xymatrix{
\cN(n,\sigma) \ar@{^(->}[d] \ar[rr]^-{\Phi(n,\sigma)}
&&    \Map^\sharp((\Delta^n)^\flat, Z) \ar[d]^-{\Map^\sharp((\Delta^n)^\flat,f)}   \\
\cM(n,\sigma) \ar@{..>}[urr]^-{\Omega(n,\sigma)}\ar[rr]_-{\Psi(n,\sigma)}
&& \Map^\sharp((\Delta^n)^\flat,T),
}
\]
such that for every monomorphism $d\colon (n-1,\rho)\to(n,\sigma)$ and every epimorphism $s\colon (n,\sigma)\to (n-1,\tau)$, the following two diagrams commute
\[
\xymatrix{
\cM(n,\sigma) \ar[d]_-{\cM(d)} \ar[rr]^-{\Omega(n,\sigma)}
&& \Map^\sharp((\Delta^{n})^\flat,Z) \ar[d]^-{\r{Res}^d} \\
\cM(n-1,\rho) \ar[rr]^-{\Omega(n-1,\rho)}
&& \Map^\sharp((\Delta^{n-1})^\flat,Z),
}
\]
\[
\xymatrix{
\cM(n-1,\tau) \ar[d]_-{\cM(s)} \ar[rr]^-{\Omega(n-1,\tau)}
&& \Map^\sharp((\Delta^{n-1})^\flat,Z) \ar[d]^-{\r{Res}^s} \\
\cM(n,\sigma) \ar[rr]^-{\Omega(n,\sigma)} && \Map^\sharp((\Delta^{n})^\flat,Z).
}
\]

By the induction hypothesis, the maps $\Omega(n-1,\rho)$ amalgamate into a map $\cM(n,\sigma)\to\Map^\sharp((\partial\Delta^{n})^\flat,Z)$, and
the maps $\Omega(n-1,\tau)$ amalgamate into a map $\cM(n,\sigma)^\deg\to\Map^\sharp((\Delta^n)^\flat,Z)$, where $\cM(n,\sigma)^\deg\subseteq\cM(n,\sigma)$ is the union of the images of $\cM(s)\colon \cM(n-1,\tau)\to\cM(n,\sigma)$. These maps amalgamate with $\Phi(n,\sigma)\colon\cN(n,\sigma)\to\Map^\sharp((\Delta^{n})^\flat,Z)$ into a map $\Omega'\colon A\to Z$, where
\[
A=(\cN(n,\sigma)\cup \cM(n,\sigma)^\deg)^\sharp\times (\Delta^{n})^\flat
\coprod_{(\cN(n,\sigma)\cup \cM(n,\sigma)^\deg)^\sharp\times(\partial\Delta^n)^\flat}
\cM(n,\sigma)^\sharp \times (\partial\Delta^{n})^\flat,
\]
fitting into the commutative square
\[
\xymatrix{
   A \ar@{^(->}[d]_i \ar[rr]^-{\Omega'}   && Z \ar[d]^-{f} \\
   \cM(n,\sigma)^\sharp\times (\Delta^{n})^\flat \ar[rr]_-{\Psi(n,\sigma)}
   \ar@{..>}[rru]^-{\Omega(n,\sigma)}    && T.
}
\]
It suffices to show that $i$ is a trivial cofibration in $\Mset$ with respect to the Cartesian model structure, so that there exists a dotted arrow rendering the above diagram commutative.

Let us first remark that for every epimorphism $s\colon (n',\sigma')\to(n'',\sigma'')$ of $\del_{/K}$, the left square of the commutative diagram
\[
\xymatrix{
\cN(n'',\tau'')\ar[d]\ar[r]^{\cN(s)} & \cN(n',\tau')\ar[d]\ar[r]^{\cN(d)} &\cN(n'',\tau'')\ar[d]\\
\cM(n'',\tau'')\ar[r]^{\cM(s)} & \cM(n',\tau')\ar[r]^{\cM(d)} & \cM(n'',\tau''),
}
\]
is a pullback by Lemma \ref{a2le:retract} below. Here $d$ is a section of $s$.

Next we prove that the map $\cN(n,\sigma)^\deg\to \cM(n,\sigma)^\deg$ is anodyne, where $\cN(n,\sigma)^\deg\subseteq \cN(n,\sigma)$ is the union of the images of $\cN(s)$. More generally, we claim that, for pairwise distinct epimorphisms $s_1,\dots ,s_m$, where $s_j\colon(n,\sigma)\to (n-1,\tau_j)$, the inclusion $\cN(\tau_1)\cup\dots\cup\cN(\tau_m)\hookrightarrow\cM(\tau_1)\cup\dots \cup \cM(\tau_m)$ is anodyne. Here $\cN(\tau_j)\subseteq \cN(n,\sigma)$ denotes the image of the split monomorphism $\cN(s_j)$ and similarly for $\cM(\tau_j)$. We proceed by induction on $m$ (simultaneously for all $n$). The case $m=0$ is trivial and we assume $m\ge 1$. For $1\le j\le m-1$, form the pushout
\[
\xymatrix{
(n,\sigma)\ar[r]^{s_j}\ar[d]_{s_m} & (n-1,\tau_j)\ar[d]\\
(n-1,\tau_m)\ar[r]^{s'_j} & (n-2,\tau'_j).
}
\]
By Lemma \ref{a2le:absolute} below, we have $\cN(\tau_j)\cap\cN(\tau_m)=\cN(\tau'_j)$, where $\cN(\tau'_j)$ denotes the image of $\cN(s'_js_m)$. The same holds for $\cM$. It follows that we have the following pushout square
\[
\xymatrix{
f_0\ar[r]\ar[d] & f_1\ar[d] \\ f_2 \ar[r]& f_3
}
\]
in the category $(\Sset)^{[1]}$, where
\begin{alignat*}{2}
f_0\colon&&\cN(\tau'_1)\cup \dots\cup\cN(\tau'_{m-1})&\to\cM(\tau'_1)\cup \dots\cup\cM(\tau'_{m-1}),\\
f_1\colon&&\cN(\tau_m)&\to\cM(\tau_m),\\
f_2\colon&&\cN(\tau_1)\cup \dots\cup\cN(\tau_{m-1})&\to\cM(\tau_1)\cup \dots\cup\cM(\tau_{m-1}),\\
f_3\colon&&\cN(\tau_1)\cup \dots\cup\cN(\tau_m)&\to\cM(\tau_1)\cup \dots\cup\cM(\tau_m)
\end{alignat*}
are natural arrows. By assumption, $f_1$ is anodyne. By induction hypothesis, $f_0$ and $f_2$ are anodyne. Since $\cN(\tau_m)\cap\cM(\tau'_j)=\cN(\tau'_j)$ by the remark of the preceding paragraph, Lemma \ref{a2le:pushout_anodyne} implies that $f_3$ is anodyne.

By the remark again, we have $\cN(n,\sigma)\cap\cM(n,\sigma)^\deg=\cN(n,\sigma)^\deg$. Thus the inclusion
$\cN(n,\sigma)\subseteq\cN(n,\sigma)\cup\cM(n,\sigma)^\deg$ is a pushout of $\cN(n,\sigma)^\deg\subseteq \cM(n,\sigma)^\deg$, hence is anodyne. By assumption, the inclusion $\cN(n,\sigma)\subseteq\cM(n,\sigma)$ is anodyne. By the two-out-of-three property for weak equivalences, it follows that the inclusion $\cN(n,\sigma)\cup \cM(n,\sigma)^\deg\subseteq\cM(n,\sigma)$ is anodyne, and consequently the inclusion
$(\cN(n,\sigma)\cup\cM(n,\sigma)^\deg)^\sharp\subseteq\cM(n,\sigma)^\sharp$ is a trivial cofibration in $\Mset$ (see Remark \ref{a3re:Quillen} below). The lemma then follows from the fact that trivial cofibrations in $\Mset$ are stable under smash products with cofibrations \cite{HTT}*{Corollary 3.1.4.3}.
\end{proof}

We say that a square in a category $\cC$ is an \emph{absolute pullback} (resp.\ \emph{absolute pushout}) if every functor $F\colon \cC\to\cD$
carries the square to a pullback (resp.\  pushout) square in $\cD$.

\begin{lem}\label{a2le:retract}
Let $\cC$ be a category. Given a commutative diagram in $\cC$
\[
\xymatrix{
X\ar[d]_f\ar[r]^{s} & Y\ar[r]^r\ar[d]^{g} & X\ar[d]^f\\
X'\ar[r]^{s'} & Y'\ar[r]^{r'} & X'
}
\]
in which both horizontal compositions are identities and $g$ is a monomorphism, then the square on the left is a pullback square. In particular, if $g$ is a split monomorphism, then the square on the left is an absolute pullback.
\end{lem}

\begin{proof}
The second assertion follows immediately from the first one. To show the first assertion, let $a\colon W\to X'$ and $b\colon W\to Y$ be morphisms satisfying $s'a=gb$. If $c\colon W\to X$ is a morphism satisfying $fc=a$ and $sc=b$, then we have $c=rsc=rb$. Conversely, we have $f(rb)=r'gb=r's'a=a$ and $s(rb)=b$. The last equality follows from $gsrb=s'frb=s'a=gb$, since $g$ is a monomorphism.
\end{proof}

\begin{lem}\label{a2le:absolute}
In $\del_{/K}$, pushouts of epimorphisms by epimorphisms are absolute pushouts.
\end{lem}

In the case of $\del\simeq \del_{/\Delta^0}$ the lemma is \cite{JT}*{Theorem 1.2.1} (see also \cite{GZ}*{{\Sec}II.3.2}). The proof in the general case is similar. We include a proof for completeness.

\begin{proof}
Factorizing epimorphisms into compositions of $s^n_i$'s (for the notation see the beginning of \Sec\ref{a1ss}), we are reduced to the case of the pushout of $s^n_i$ by $s^n_j$, where $i\le j$. This case follows from Lemma \ref{a2le:retract} applied to the diagram
\[
\xymatrix{
(n,\tau)\ar[d]_{s^{n-1}_{j-1}}\ar[r]^{d^{n+1}_i} & (n+1,\sigma)\ar[r]^{s^n_i}\ar[d]^{s^n_j} & (n,\tau)\ar[d]^{s^{n-1}_{j-1}}\\
(n-1,\tau')\ar[r]^{d^n_i} &(n,\sigma')\ar[r]^{s^{n-1}_i} & (n-1,\tau')
}
\]
for $i<j$, and to the diagram
\[
\xymatrix{
(n,\tau)\ar[d]_{\id} \ar[r]^{d^{n+1}_i} & (n+1,\sigma)\ar[d]^{s^n_i}\ar[r]^{s^n_i} & (n,\tau)\ar[d]^{\id}\\
(n,\tau)\ar[r]^{\id} & (n,\tau)\ar[r]^{\id} & (n,\tau)
}
\]
for $i=j$.
\end{proof}

\begin{lem}\label{a2le:pushout_anodyne}
Consider a pushout square
\[
\xymatrix{
f_0\ar[r]^u\ar[d] & f_1\ar[d] \\ f_2\ar[r]&f_3
}
\]
in $(\Sset)^{[1]}$, where $f_i\colon Y_i\to X_i$. Assume that $f_0$, $f_1$, $f_2$ are anodyne (resp.\ right anodyne) and the map $X_0\coprod_{Y_0}Y_1\to X_1$ induced by $u$ is a monomorphism. Then $f_3$ is anodyne (resp.\ right anodyne).
\end{lem}

\begin{proof}
The square corresponds to a cube in $\Sset$, which can be decomposed into a commutative diagram
\[
\xymatrix{
Y_0\ar[rr]\ar[rd]^{f_0}\ar[dd]&&Y_1\ar@{=}[rr]\ar[rd]^{g_0}\ar@{-->}[dd] && Y_1\ar[rd]^{f_1}\ar@{-->}[dd]\\
&X_0\ar[dd]\ar[rr] && Z_0\ar[rr]^(.3){a_0}\ar[dd] && X_1\ar[dd]\\
Y_2\ar[rd]^{f_2}\ar@{-->}[rr] && Y_3\ar[rd]^{g_2}\ar@{==}[rr] && Y_3\ar@{-->}[rd]^{f_3}\\
&Y_2\ar[rr] && Z_2\ar[rr]^{a_2} && X_3,
}
\]
where the top and bottom squares on the left are pushout squares, and the front and back squares are pushout squares. For $i=0,2$, the map $g_i$ is a pushout of $f_i$, hence is anodyne (resp.\ right anodyne). Since $f_1$ is anodyne (resp.\ right anodyne) and $a_0$ is a monomorphism by assumption, $a_0$ is anodyne by the two-out-of-three property of weak equivalences (resp.\ right anodyne by \cite{HTT}*{Proposition 4.1.1.3}). Thus the pushout $a_2$ of $a_0$ is anodyne (resp.\  right anodyne). Therefore, $f_3=a_2g_2$ is anodyne (resp.\ right anodyne).
\end{proof}

We now give the form of the construction technique as used in Sections \ref{a4ss} and \ref{a5ss}.

\begin{proposition}\label{a1pr:extension}
Let $K$ be a simplicial set, $\cC$ an $\infty$-category, and $i\colon A\hookrightarrow B$ a monomorphism of simplicial sets. Denote by $f\colon\Fun(B,\cC)\to\Fun(A,\cC)$ the map induced by $i$. Let $\cN$ be an object of $(\Sset)^{(\del_{/K})^{op}}$ such that $\cN(\sigma)$ is weakly contractible for all $\sigma\in\del_{/K}$, and let $\Phi\colon \cN\to \Map[K,\Fun(B,\cC)]$ be a morphism such that $\Map[K,f]\circ\Phi\colon \cN\to \Map[K,\Fun(A,\cC)]$ factorizes through $\Delta^0_{(\del_{/K})^{op}}$ to give a functor $a\colon K\to\Fun(A,\cC)$. Then there exists $b\colon K\to\Fun(B,\cC)$ such that $b\circ p=a$ and for every map $g\colon K'\to K$ and every global section $\nu\in \Gamma(g^*\cN)_0$, the maps $b\circ g$ and $g^*\Phi\circ\nu\colon K'\to\Fun(B,\cC)$ are homotopic over $\Fun(A,\cC)$. Here $g^*\Phi\colon g^*\cN\to g^*\Map[K,\Fun(B,\cC)]=\Map[K',\Fun(B,\cC)]$.
\end{proposition}

In the statement we have implicitly used isomorphisms provided by Remark\ref{a2re:functors} such as $\Map^\sharp(K,\Fun(A,\cC)^\natural)\simeq\Gamma(\Map[K,\Fun(A,\cC)])$.

\begin{proof}
Since $\Fun(-,\cC)^\natural=(\cC^\natural)^{(-)^\flat}$, the map $f^\natural\colon \Fun(B,\cC)^\natural\to \Fun(A,\cC)^\natural$ is a
fibration in $\Mset$ for the Cartesian model structure by Lemma \ref{a1le:marked_fibration} below. Thus by Proposition
\ref{a1le:injective_fibration}, $\Map[K,f^\natural]$ is an injective fibration. We let $\Map[K,f^{\natural}]_a$ denote the fiber of
$\Map[K,f^{\natural}]$ at $a$, which is injectively fibrant. By Lemma \ref{a1le:component} (1), $\Gamma_\Phi(\Map[K,f^{\natural}]_a)$ is a
(nonempty) connected component of $\Gamma(\Map[K,f^{\natural}]_a)$. Note that $\Gamma(\Map[K,f^\natural]_a)$ is the fiber of
$\Gamma(\Map[K,\Fun(B,\cC)])\to \Gamma(\Map[K,\Fun(A,\cC)])$ at $a$. Any vertex of $\Gamma_\Phi(\Map[K,f^{\natural}]_a)$ then provides the desired $b$. Indeed, for given $g$ and $\nu$, both $b\circ g$ and $g^*\Phi\circ\nu$ are given by vertices of the connected Kan complex
$\Gamma_{g^*\Phi}(\Map[K',f^\natural]_{g^*a})$, which are necessarily equivalent.
\end{proof}

\begin{lem}\label{a1le:marked_fibration}
Let $X\to Y$ be a fibration and $i\colon A\to B$ be a cofibration in $\Mset$ with respect to the Cartesian model structure. Then the induced map
\[
X^B\to X^A\times_{Y^A} Y^B
\]
is a fibration in $\Mset$ with respect to the Cartesian model structure.
\end{lem}

\begin{proof}
This follows immediately from the fact that trivial cofibrations in $\Mset$ are stable under smash product with cofibrations \cite{HTT}*{Corollary 3.1.4.3}.
\end{proof}

\subsection{Restricted multisimplicial nerves}
\label{a3ss}

In this section, we introduce several notions related to multisimplicial sets. The restricted multisimplicial nerve (Definition \ref{a3de:restricted_nerve}) of a multi-tiled simplicial set (Definition \ref{a3de:multitiled_simplicial}) will play an essential role in the statements of our theorems.

\begin{definition}[Multisimplicial set]\label{a3de:multisimplicial}
Let $I$ be a set. We define the category of \emph{$I$-simplicial sets} to be $\Sset[I]\coloneqq\Fun((\del^I)^{op},\Set)$, where $\del^I\coloneqq\Fun(I,\del)$. For an integer $k\ge 0$, we define the category of $k$-simplicial sets to be $\Sset[k]\coloneqq\Sset[I]$, where $I=\{1,\dots,k\}$. We identify $\Sset[1]$ with $\Sset$.
\end{definition}

We denote by $\Delta^{n_i\mid i\in I}$ the $I$-simplicial set represented by the object $([n_i])_{i\in I}$ of $\del^I$. For an $I$-simplicial set $S$, we denote by $S_{n_i\mid i\in I}$ the value of $S$ at the object $([n_i])_{i\in I}$ of $\del^I$. An $(n_i)_{i\in I}$-\emph{simplex} of an $I$-simplicial set $S$ is an element of $S_{n_i\mid i\in I}$. By the Yoneda lemma, there is a canonical bijection between the set $S_{n_i\mid i\in I}$ and the set of maps from $\Delta^{n_i\mid i\in I}$ to $S$.

For $J\subseteq I$, composition with the partial opposite functor $\del^I\to\del^I$ sending $(\dots,P_{j'},\dots,P_j,\dots)$ to $(\dots,P_{j'},\dots,P_j^{op},\dots)$ (taking $op$ for $P_j$ when $j\in J$) defines a functor $\op^I_J\colon \Sset[I]\to \Sset[I]$. We put $\Delta^{n_i\res i\in I}_J\coloneqq\op^I_J\Delta^{n_i\res i\in I}$. Although $\Delta^{n_i\res i\in I}_J$ is isomorphic to $\Delta^{n_i\res i\in I}$, it will be useful in specifying the variance of many constructions. When $J=\emptyset$, $\op^I_\emptyset$ is the identity functor so that $\Delta^{n_i\res i\in I}_\emptyset=\Delta^{n_i\res i\in I}$.

\begin{remark}
The category $\Sset[I]$ is Cartesian-closed. In fact, for two $I$-simplicial sets $X$ and $Y$, the internal mapping object $\Map(Y,X)$ is an $I$-simplicial set such that $\Hom_{\Sset[I]}(Z,\Map(Y,X))\simeq \Hom_{\Sset[I]}(Z\times Y,X)$ for every $Z\in\Sset[I]$. We have $\op^I_J \Map(Y,X)\simeq \Map(\op^I_J Y,\op^I_J X)$.
\end{remark}

\begin{definition}
Let $I,J$ be two sets.
\begin{enumerate}
  \item Let $f\colon J\to I$ be a map of sets. Composition with $f$ defines a functor $\del^f\colon \del^I\to \del^J$. Composition with $(\del^f)^{op}$ induces a functor $(\del^f)^*\colon \Sset[J]\to\Sset[I]$, which has a right adjoint $(\del^f)_*\colon \Sset[I]\to\Sset[J]$. We will now look at two special cases.

  \item Let $f\colon J\to I$ be an injective map. The functor $\del^f$ has a right adjoint $c_f\colon \del^J\to \del^I$ given by $c_f(F)_i=F_j$ if $f(j)=i$ and $c_f(F)_i=[0]$ if $i$ is not in the image of $f$. The functor $(\del^f)_*$ can be identified with the functor $\epsilon^f$ induced by composition with $(c_f)^{op}$. If $J=\{1,\dots,k'\}$, we write $\epsilon^I_{f(1)\dotsm f(k')}$ for $\epsilon^f$.

  \item Consider the map $f\colon I\to \{1\}$. Then $\delta_I\coloneqq\del^f\colon \del\to\del^I$ is the diagonal functor, and composition with $(\delta_I)^{op}$ induces the \emph{diagonal functor} $\delta^*_I=(\del^f)^*\colon \Sset[I]\to\Sset$. We define
      \[
      \Delta^{[n_i]_{i\in I}}\coloneqq\delta_I^*\Delta^{n_i\res i\in I} =\prod_{i\in I}\Delta^{n_i}.
      \]
      We define the \emph{multisimplicial nerve} functor to be the right adjoint $\delta^I_*\colon \Sset\to\Sset[I]$ of $\delta^*_I$. An $(n_i)_{i\in I}$-simplex of $\delta^I_*X$ is given by a map $\Delta^{[n_i]_{i\in I}}\to X$.

  \item For $J\subseteq I$, we define the \emph{twisted diagonal functor} $\delta^*_{I,J}$ as $\delta^*_I\circ\op^I_J\colon\Sset[I]\to\Sset$. We define
      \[
      \Delta^{[n_i]_{i\in I}}_J\coloneqq \delta_{I,J}^*\Delta^{n_i\res i\in I}=\delta_I^*\Delta^{n_i\res i\in I}_J
      =\(\prod_{i\in I-J}\Delta^{n_i}\)\times\(\prod_{j\in J}(\Delta^{n_j})^{op}\).
      \]
      When $J=\emptyset$, we have $\delta^*_{I,\emptyset}=\delta^*_I$ and $\Delta^{[n_i]_{i\in I}}_{\emptyset}=\Delta^{[n_i]_{i\in I}}$.
\end{enumerate}
\end{definition}

When $I=\{1,\dots,k\}$, we write $k$ instead of $I$ in the previous notation. For example, in (2) we have $(\epsilon^k_jK)_n=K_{0,\dots,n,\dots,0}$, where $n$ is at the $j$-th position and all other indices are $0$. In (3) we have $\delta^*_k\colon \Sset[k]\to\Sset$ defined by $(\delta^*_kX)_n=X_{n,\dots,n}$.

\begin{remark}\label{a3re:eps}
For any map $f\colon J\to I$, we have $\del^f\circ \delta_I=\delta_J$, so that $(\del^f)_*\circ \delta^I_*\simeq \delta^J_*$. In particular, for $f$ injective, we have $\epsilon^f\circ \delta^I_*\simeq \delta^J_*$. For $\alpha\in I$, we have $\epsilon^I_\alpha\circ \delta^I_*\simeq \id_{\Sset}$.
\end{remark}

\begin{remark}
For $f\colon J\to I$ injective, we have $\del^f\circ c_f=\id_{\del^J}$, so that $\epsilon^f\circ (\del^f)^*=\id_{\Sset[J]}$. The counit transformation $(\del^f)^* \circ \epsilon^f\to \id_{\Sset[I]}$ is a monomorphism. Indeed, for each object $P$ of $\del^I$, the unit morphism $P\to (c_f\circ \del^f)(P)$ admits a section. Applying the functor $\delta_I^*$, we obtain a monomorphism $\delta_J^*\circ \epsilon^f\to \delta_I^*$.
\end{remark}

\begin{remark}\label{a3re:adjunction}
For every map $f\colon J\to I$, the adjunction formula for presheaves provides a canonical isomorphism
\[
\Map(Y,(\del^f)_*X)\simeq (\del^f)_*\Map((\del^f)^*Y,X)
\]
for every $I$-simplicial set $X$ and every $J$-simplicial set $Y$. This map is the composite map
\begin{align*}
\Map(Y,(\del^f)_*X)&\xrightarrow{(\del^f)^*}(\del^f)_*\Map((\del^f)^*Y,(\del^f)^*(\del^f)_*X) \\
& \to (\del^f)_*\Map((\del^f)^*Y,X),
\end{align*}
where the second map is induced by the counit map $(\del^f)^*(\del^f)_*X\to X$.

Specializing to the case of $\delta_I$ and applying the functor $\epsilon^I_\alpha$, where $\alpha\in I$, we get an isomorphism
\[
\epsilon^I_\alpha\Map(X,\delta^I_*S)\simeq \Map(\delta_I^*X,S)
\]
for every $I$-simplicial set $X$ and every simplicial set $S$, which is the composite map
\[
\epsilon^I_\alpha\Map(X,\delta^I_*S)\xrightarrow{\delta_I^*} \Map(\delta_I^*X,\delta_I^*\delta^I_*S)\to\Map(\delta_I^*X,S).
\]
\end{remark}

\begin{definition}[Exterior product]
Let $I=\coprod_{j\in J}I_j$ be a partition. We define a functor
\[
\boxtimes_{j\in J}\colon \prod_{j\in J}\Sset[I_j]\to\Sset[I]
\]
by the formula $\boxtimes_{j\in J}S^j=\prod_{j\in J}(\del^{\iota_j})^* S^j$, where $\iota_j\colon I_j\hookrightarrow I$ is the inclusion. For
$J=\{1,\dots,m\}$, $I_j=\{1,\dots,k_j\}$, we define
\[
-\boxtimes\dots \boxtimes -\colon \Sset[k_1]\times\dots \times\Sset[k_m]\to\Sset[k].
\]
by $(S^1\boxtimes\dots\boxtimes S^m)_{n^1_1,\dots,n^1_{k_1},\dots,n^m_{1},\dots,n^m_{k_m}}=S^1_{n^1_1,\dots,n^1_{k_1}}\times\dots\times S^m_{n^m_{1},\dots,n^m_{k_m}}$.
\end{definition}

We have the isomorphisms $\boxtimes_{i\in I}\Delta^{n_i}\simeq\Delta^{n_i\mid i\in I}$ and $\delta_I^*\boxtimes_{j\in J}S^j\simeq\prod_{j\in J}\delta_{I_j}^*S^j$.

\begin{remark}
For a map $f\colon J\to I$, we have $(\del^f)^*\Delta^{n_i\mid i\in I}\simeq\boxtimes_{i\in I}\Delta^{[n_j]_{j\in f^{-1}(i)}}$, so that an
$(n_j)_{j\in J}$-simplex of $(\del^f)_* X$ is given by a map $\boxtimes_{i\in I} \Delta^{[n_j]_{j\in f^{-1}(i)}}\to X$.
\end{remark}

We next turn to restricted variants of the multisimplicial nerve functor $\delta^I_*$. We start with restrictions on edges.

\begin{definition}[Multi-marked simplicial set]\label{a3de:multimarked_simplicial}
An \emph{$I$-marked simplicial set} (resp.\ \emph{$I$-marked $\infty$-category}) is the data $(X,\cE=\{\cE_i\}_{i\in I})$, where $X$ is a simplicial set (resp. an $\infty$-category) and, for all $i\in I$, $\cE_i$ is a set of edges of $X$ which contains every degenerate edge. The data $\cE$ is sometimes called an \emph{$I$-marking} on $X$. A morphism $f\colon(X,\{\cE_i\}_{i\in I})\to (X',\{\cE'_i\}_{i\in I})$ of $I$-marked simplicial sets is a map $f\colon X\to X'$ having the property that $f(\cE_i)\subseteq\cE'_i$ for all $i\in I$. We denote the category of $I$-marked simplicial sets by $\Sset^{I+}$. It is the strict fiber product of $I$ copies of $\Mset$ over $\Sset$.

For a simplicial set $X$ and a subset $J\subseteq I$, we define an $I$-marked simplicial set $X^{\sharp^I_J} =(X,\cE)$ by $(X,\cE_j)=X^\sharp$ for $j\in J$ and $(X,\cE_i)=X^\flat$ for $i\in I-J$. We write $X^{\sharp^I}=X^{\sharp^I_I}$ and $X^{\flat^I}=X^{\sharp^I_\emptyset}$. The functor $\Sset\to \Sset^{I+}$ carrying $X$ to $X^{\sharp^I}$ (resp.\ $X^{\flat^I}$) is a right (resp.\ left) adjoint of the forgetful functor $\Sset^{I+}\to \Sset$.
\end{definition}

Consider the functor $\delta_{I+}^*\colon \Sset[I]\to \Sset^{I+}$ sending $S$ to $(\delta^*_I S,\{\cE_i\}_{i\in I})$, where $\cE_i$ is the set of edges of $\epsilon^I_i S\subseteq\delta_I^*S$. This functor admits a right adjoint $\delta^{I+}_*\colon \Sset^{I+}\to \Sset[I]$. Since $\delta_{I+}^*\Delta^{n_i\mid i\in I}=\prod_{i\in I}(\Delta^{n_i})^{\sharp^I_{\{i\}}}$, the functor $\delta^{I+}_*$ carries $(X,\{\cE_i\}_{i\in I})$ to the $I$-simplicial subset of $\delta^I_* X$ whose $(n_i)_{i\in I}$-simplices are maps $\Delta^{[n_i]_{i\in I}}\to X$ such that for every $j\in I$ and every map $\Delta^1\to \epsilon^I_j\Delta^{n_i\res i\in I}$, the composition
\[
\Delta^1\to\epsilon^I_j\Delta^{n_i\res i\in I} \to\Delta^{[n_i]_{i\in I}} \to X
\]
is in $\cE_j$. We have $\delta^I_*(X)=\delta^{I+}_*(X^{\sharp I})$. When 
$I=\{1,\dots,k\}$, we use the notation $\Sset^{k+}$, $\delta_{k+}^*$ and 
$\delta^{k+}_*$.\footnote{In particular, $\Sset^{2+}$ in our notation is 
$\Sset^{++}$ in \cite{HA}*{Definition 4.7.4.2}.} 

\begin{definition}[Restricted multisimplicial nerve]
We define the \emph{restricted $I$-simplicial nerve} of an $I$-marked simplicial set $(X,\cE=\{\cE_i\}_{i\in I})$ to be the $I$-simplicial set
\[
X_{\cE}=X_{\{\cE_i\}_{i\in I}}\coloneqq\delta^{I+}_*(X,\{\cE_i\}_{i\in I}).
\]
In particular, for any marked simplicial set $(X,\cE)$, the simplicial set $X_\cE$ is the simplicial subset of $X$ spanned by the edges in $\cE$.
\end{definition}

\begin{remark}\label{a3re:Quillen}
The functor $\delta^*_{1+}\colon \Sset\to \Mset$ carries $S$ to $S^\sharp$. The functor $\delta^{1+}_*\colon \Mset\to \Sset$ carries $(X,\cE)$ to the simplicial subset of $X$ consisting of all simplices whose edges are all marked edges. In other words, $X_\cE=\delta^{1+}_*(X,\cE)$ is the largest simplicial subset $S\subseteq X$ such that $S^\sharp \subseteq (X,\cE)$. We have $\delta^{1+}_*\simeq\Map^\sharp((\Delta^0)^\flat,-)$. For objects $X$ and $Y$ of $\Mset$, we have $\Map^\sharp(X,Y)=\delta^{1+}_*(Y^X)$.

The pair $(\delta^*_{1+},\delta^{1+}_*)$ is a Quillen adjunction for the Kan model structure on $\Sset$ and the Cartesian model structure on $\Mset$. This is a special case of Remark \ref{a2re:Quillen} but we can also check this easily as follows. Clearly $\delta^*_{1+}$ preserves cofibrations. To see that it also preserves trivial cofibrations, note that for any anodyne map of simplicial sets $T\to S$ and any $\infty$-category $\cC$, the induced map $\Map^\sharp(S^\sharp,\cC^\natural)\to\Map^\sharp(T^\sharp,\cC^\natural)$ is a trivial Kan fibration.
\end{remark}

Next we consider restrictions on squares. By a \emph{square} of a simplicial set $X$, we mean a map $\Delta^1\times \Delta^1 \to X$. The transpose of a square is obtained by swapping the two $\Delta^1$'s. Composition with the maps $\id\times d^1_0,\id\times d^1_1\colon\Delta^1\simeq \Delta^1\times\Delta^0 \to \Delta^1\times \Delta^1$ induce maps $\Hom(\Delta^1\times\Delta^1,X)\to X_1$ and composition with the map $\id\times s^0_0 \colon\Delta^1\times \Delta^1\to \Delta^1\times \Delta^0\simeq \Delta^1$ induces a
map $X_1\to \Hom(\Delta^1\times \Delta^1,X)$.

\begin{definition}[Multi-tiled simplicial set]\label{a3de:multitiled_simplicial}
An \emph{$I$-tiled simplicial set} (resp.\ \emph{$I$-tiled $\infty$-category}) is the data $(X,\cE=\{\cE_i\}_{i\in I},\cQ=\{\cQ_{ij}\}_{i,j\in I, i\neq j})$, where $(X,\cE)$ is an $I$-marked simplicial set (resp.\ $\infty$-category) and, for all $i,j\in I$, $i\neq j$, $\cQ_{ij}$ is a set of squares of $X$ such that $\cQ_{ij}$ and $\cQ_{ji}$ are obtained from each other by transposition of squares, and $\id\times d^1_0$, $\id\times d^1_1$ induce maps $\cQ_{ij}\to \cE_i$, and $\id\times s^0_0$ induces $\cE_i\to \cQ_{ij}$. A morphism $f\colon(X,\cE,\cQ)\to(X',\cE',\cQ')$ of $I$-tiled simplicial sets is a map $f\colon X\to X'$ having the property that $f(\cE_i)\subseteq f(\cE'_i)$ and $f(\cQ_{ij})\subseteq\cQ'_{ij}$ for all $i,j$. We denote the category of $I$-tiled simplicial sets by $\Sset^{I\square}$. The data $\cT=(\cE,\cQ)$ is sometimes called an \emph{$I$-tiling} on $X$. For brevity, we adopt the conventions $\cT_i=\cE_i$ and $\cT_{ij}=\cQ_{ij}$.
\end{definition}

\begin{remark}\label{a3re:multitile}
Note that $\cE_i$ is the image of $\cQ_{ij}$ under either of the maps $\cQ_{ij}\to \cE_i$ given by $\id\times d^1_0$ and $\id\times d^1_1$. Moreover, $f(\cQ_{ij})\subseteq \cQ'_{ij}$ implies $f(\cE_i)\subseteq\cE'_i$ and $f(\cE_j)\subseteq \cE'_j$.
\end{remark}

Consider the functor $\delta_{I\square}^*\colon \Sset[I]\to\Sset^{I\square}$ carrying $S$ to $(\delta^*_{I+} S,\cQ)$, where $\cQ_{ij}$ is the image of the injection
\begin{align*}
(\epsilon^I_{ij}S)_{11}&=\Hom_{\Sset[2]}(\Delta^{1,1},\epsilon^I_{ij}S) \\
&\xrightarrow{\delta^*_2}
\Hom_{\Sset}(\Delta^1\times \Delta^1,\delta^*_2\epsilon^I_{ij}S)\subseteq \Hom_{\Sset}(\Delta^1\times \Delta^1,\delta^*_I S).
\end{align*}
This functor admits a right adjoint $\delta^{I\square}_*\colon \Sset^{I\square}\to \Sset[I]$ carrying $(X,\cE,\cQ)$ to the $I$-simplicial subset of $\delta^{I+}_*(X,\cE)\subseteq \delta^I_* X$ whose $(n_i)_{i\in I}$-simplices are maps $\Delta^{[n_i]_{i\in I}}\to X$ satisfying the additional condition that for every pair of elements $j,k\in I$, $j\neq k$, and every map $\Delta^1\boxtimes\Delta^1\to\epsilon^I_{jk}\Delta^{n_i\res i\in I}$, the composition
\[
\Delta^1\times \Delta^1\to\delta_2^*\epsilon^I_{jk}\Delta^{n_i\res i\in I}\to\Delta^{[n_i]_{i\in I}} \to X
\]
is in $\cQ_{jk}$. When $I=\{1,\dots,k\}$, we use the notation $\Sset^{k\square}$, $\delta^*_{k\square}$, $\delta_*^{k\square}$.

\begin{definition}[Restricted multisimplicial nerve]\label{a3de:restricted_nerve}
We define the \emph{restricted $I$-simplicial nerve} of an $I$-tiled simplicial set $(X,\cT)$ to be the $I$-simplicial set
$\delta_*^{I\square}(X,\cT)$.
\end{definition}

\begin{notation}\label{a3no:cartesian}
The underlying functor $\sfU\colon \Sset^{I\square}\to \Sset^{I+}$ carrying $(X,\cE,\cQ)$ to $(X,\cE)$ admits a left adjoint $\sfV\colon\Sset^{I+}\to\Sset^{I\square}$ and a right adjoint $\sfW\colon \Sset^{I+}\to\Sset^{I\square}$, which can be described as follows.
\begin{itemize}
  \item We have $\sfV(X,\cE)=(X,\cE,\cQ)$, where $\cQ_{ij}$ is the union of the image of $\cE_i$ under $-\circ(\id \times s_0^0)$ and the
      image of $\cE_j$ under $-\circ (s_0^0\times \id)$.

  \item For sets of edges $\cE_1$ and $\cE_2$ of $X$, we denote by $\cE_1*_X\cE_2$ the set of squares $f\colon\Delta^1\times\Delta^1\to X$
      \[
      \xymatrix{f(0,0)\ar[r] \ar[d] & f(0,1)\ar[d]\\f(1,0)\ar[r] & f(1,1)
      }
      \]
      such that the vertical edges $f\circ (\id \times d^1_\alpha)$, $\alpha=0,1$ belong to $\cE_1$ and the horizontal edges $f\circ (d^1_\alpha\times \id)$, $\alpha=0,1$ belong to $\cE_2$. We have $\sfW(X,\cE)=(X,\cE,\cQ)$, where $\cQ_{ij}=\cE_i*_X\cE_j$.
\end{itemize}
We have $\delta^*_{I+}\simeq\sfU\circ\delta^*_{I\square}$ and $\delta_*^{I+}\simeq \delta_*^{I\square}\circ\sfW$.
\end{notation}

\begin{definition}[Cartesian multisimplicial nerve]\label{a3de:cartesian_nerve}
If $\cC$ is an $\infty$-category and $\cE_1$, $\cE_2$ are sets of edges of $\cC$, we denote by $\cE_1*^\cart_{\cC}\cE_2$ the subset of $\cE_1*_{\cC}\cE_2$ consisting of Cartesian squares. For an $I$-marked $\infty$-category $(\cC,\cE)$, we denote by $(\cC,\cE^\cart)$ the $I$-tiled $\infty$-category such that $\cE^\cart_i=\cE_i$ for $i\in I$ and $\cE^\cart_{ij}=\cE_i*_{\cC}^\cart\cE_j$ for $i,j\in I$ and $i\neq j$. We define the \emph{Cartesian $I$-simplicial nerve} of an $I$-marked $\infty$-category $(\cC,\cE)$ to be
\[
\cC^\cart_{\cE}\coloneqq\delta^{I\square}_*(\cC,\cE^\cart).
\]
\end{definition}

For reference in later sections, we define a few properties of sets of edges and squares. As in the definition of marked simplicial sets, we are mainly interested in those sets of edges that contain all degenerate edges. However, many sets of squares of interest, when regarded as sets of edges in suitable simplicial sets, do not contain all degenerate edges. For this reason, we allow sets of edges not containing all degenerate edges in the definitions below.

\begin{definition}\label{a3de:edges}
Let $X$ be a simplicial set, and let $\cE$ be a set of edges of $X$. We say that $\cE$ is
\begin{enumerate}
    \item \emph{composable} if every map $\Lambda^2_1\to X$ whose restrictions to $\Delta^{\{0,1\}}$ and to $\Delta^{\{1,2\}}$ are in $\cE$ extends to a $2$-simplex $\Delta^2\to X$ whose restriction to $\Delta^{\{0,2\}}$ is in $\cE$.

    \item \emph{stable under composition} for every $2$-simplex $\sigma$ of $X$ such that $\sigma\circ d^2_0,\sigma\circ d^2_2\in\cE$, we have $\sigma\circ d^2_1\in \cE$.
\end{enumerate}
\end{definition}

If $\cE$ contains every degenerate edge, then (1) above is equivalent to every one of the following conditions
\begin{itemize}
   \item $(X,\cE)$ has the extension property with respect to the inclusion $(\Lambda^2_1)^\sharp\subseteq (\Delta^2)^\sharp$;

   \item $X_\cE$ has the extension property with respect to the inclusion $\Lambda^2_1\subseteq \Delta^2$;
\end{itemize}
and (2) above is equivalent to every one of the following conditions
\begin{itemize}
   \item $(X,\cE)$ has the extension property with respect to the inclusion
       \[
       (\Lambda^2_1)^\sharp\coprod_{(\Lambda^2_1)^\flat}
       (\Delta^2)^\flat\subseteq (\Delta^2)^\sharp;
       \]

   \item $X_\cE\to X$ has the right lifting property with respect to the inclusion $\Lambda^2_1\subseteq \Delta^2$;

   \item $X_\cE\to X$ is an inner fibration.
\end{itemize}
If $X$ has the extension property with respect to $\Lambda^2_1\subseteq\Delta^2$, then (2) implies (1).

\begin{definition}\label{a3de:admissible_edge}
Let $\cC$ be an $\infty$-category and let $\cE$, $\cF$ be two sets of edges of $\cC$. We say that $\cE$ is
\begin{enumerate}
   \item \emph{stable under homotopy} if for $e\in\cE$ and $f\in\cC_1$ that have the same image in $\rh\cC$, we have $f\in\cE$;

   \item \emph{stable under equivalence} if for $e\in\cE$ and $f\in\cC_1$ that are equivalent as objects of $\Fun(\Delta^1,\cC)$, we have $f\in\cE$;

   \item \emph{stable under pullback by $\cF$} if for every Cartesian square in $\cC$ of the form
       \[
       \xymatrix{
       y' \ar[d]_{e'} \ar[r] & y \ar[d]^e \\
       x' \ar[r]^f & x
       }
       \]
       with $e\in\cE$ and $f\in\cF$, we have $e'\in\cE$;

   \item \emph{stable under pullback} (see \cite{HTT}*{Notation 6.1.3.4}) if it is stable under pullback by $\cC_1$;

   \item \emph{admissible} if $\cE$ contains every degenerate edge of $\cC$, is stable under pullback, and for every $2$-simplex of $\cC$ of the form
       \begin{align}\label{a3eq:2cell}
       \xymatrix{
       &y\ar[rd]^p \\ z\ar[ru]^q\ar[rr]^r && x
       }
       \end{align}
       with $p\in\cE$, we have $q\in\cE$ if and only if $r\in\cE$.
\end{enumerate}
\end{definition}

In the above definition, (5) implies (4); (4) implies (3); (2) implies (1). Moreover, if $\cF$ contains every edge of $\cC$ that is an equivalence (resp.\ degenerate), then (3) implies (2) (resp.\ (1)). If $\cE$ satisfies (3) with $\cE$ and $\cF$ each containing all degenerate edges of $\cC$, then $\cE$ contains all equivalences of $\cC$. The last condition in (5) is equivalent to saying that $X_\cE\to X$ is a right fibration \cite{HTT}*{Definition 2.0.0.3}.

\begin{remark}\label{a3re:admissible}
If $\cC$ admits pullbacks, then $\cE$ is admissible if and only if it contains every degenerate edge of $\cC$ and is stable under composition,
pullback, and taking diagonal in $\cC$. The ``only if'' part is clear. For the ``if'' part, note that in the $2$-simplex \eqref{a3eq:2cell}, $q$ is a composition of
\[
z\to z\times_x y \to y,
\]
where the first morphism is a pullback of the diagonal $y\to y\times_x y$ of $p$ and the second morphism is a pullback of $r$ by $p$. Indeed, we have a diagram with pullback squares
\[
\xymatrix{
z\ar[r]^q\ar[d] & y\ar[d]\\
z\times_x y\ar[r]\ar[d] & y\times_x y \ar[r]\ar[d] & y\ar[d]^p\\
z\ar[r]^q & y\ar[r]^p & x.
}
\]
\end{remark}

In an $\infty$-category $\cC$, a set of edges $\cE$ is composable if and only if its image in $\rh\cC$ is stable under composition. Thus if $\cE$ is composable and stable under homotopy, then $\cE$ is stable under composition. The converse holds if $\cE$ contains every degenerate edge. In the next section, we will need the following extension property of composable sets of edges.

\begin{lem}\label{a3le:comp}
Let $I$ be a set. Let $(B,\cF)$ be an $I$-marked simplicial set and $(\cC,\cE)$ an $I$-marked $\infty$-category.  Let $A\subseteq B$ be a
categorical equivalence such that for each $i\in I$, $\cF_i$ is contained in the smallest set of edges of $B$ containing $\cG_i=A_1\cap \cF_i$ and stable under composition. Assume $\cE_i$ composable for all $i\in I$ and $\cF_i\cap\cF_j\subseteq A_1$ for all $i,j\in I$, $i\neq j$. Then $(\cC,\cE)$ has the extension property with respect to $(A,\cG)\subseteq (B,\cF)$.
\end{lem}

\begin{proof}
Let $f\colon (A,\cG)\to (\cC,\cE)$ be a map of $I$-marked simplicial sets. Choose an extension $g\colon B\to \cC$ of $f$. For each $i\in I$, let $\cE'_i$ denote the set of edges of $\cC$ that are homotopic to some edge of $\cE_i$. Then $\cE'_i$ is stable under composition, and hence so is its inverse image under $g$. Thus $g$ induces $(B,\cF)\to(\cC,\cE')$. Let $D\subseteq B$ be the union of $A$ and the edges in $\cF$. We construct a map $h_0\colon (D,\cF)\to (\cC,\cE)$ extending $f$ and a natural equivalence $g\res D\to h_0$ extending $\id_f$, by choosing for each edge $e$ in $\cF_i$ but not in $A_1$, a homotopy from $g(e)$ to an edge $h_0(e)$ in $\cE_i$. By \cite{HTT}*{Lemma 2.4.6.3}, $h_0$ extends to a map $(B,\cF)\to (\cC,\cE)$, as desired.
\end{proof}

\begin{definition}\label{a3de:squares}
For a simplicial set $X$, the map
\[
\Hom(\Delta^1\times\Delta^1,X)\to\Hom(\Delta^1, \Map(\Delta^1,X))
\]
carrying $f$ to $a\mapsto(b\mapsto f(a,b))$ (resp.\ $a\mapsto (b\mapsto f(b,a))$) is an isomorphism.
\begin{enumerate}
  \item We say that a set of squares $\cQ$ of $X$ is \emph{stable under composition in the first (resp.\ second) direction} if the resulting
      set of edges of $\Map(\Delta^1,X)$ is stable under composition.
\end{enumerate}
Now let $\cQ$ and $\cQ'$ be sets of squares of an $\infty$-category $\cC$.
\begin{enumerate}\setcounter{enumi}{1}
  \item We say that $\cQ$ is \emph{stable under equivalence} if $\cQ$, when viewed as a set of edges of $\Map(\Delta^1,\cC)$ via the above
      isomorphism, is stable under equivalence.

  \item We say that $\cQ$ is \emph{stable under pullback by $\cQ'$ in the first (resp.\ second) direction}, if $\cQ$ is stable under pullback
      by $\cQ'$ in $\Map(\Delta^1,\cC)$, where $\cQ$ and $\cQ'$ are viewed as sets of edges via the above isomorphism.

  \item We say that $\cQ$ is \emph{stable under pullback in the first (resp.\ second) direction} if (3) holds for $\cQ'=\Fun(\Delta^1\times\Delta^1,\cC)$, the set of all squares of $\cC$.
\end{enumerate}
\end{definition}

By \cite{HTT}*{Corollary 5.1.2.3}, condition (3) means that for any cube in $\cC$ of the form
\begin{equation}\label{a3eq:cube}
\xymatrix{
y'(0)\ar[rr]\ar[rd]\ar[dd]&&y(0)\ar[rd]\ar@{-->}[dd]\\
&y'(1)\ar[dd]\ar[rr] && y(1)\ar[dd]\\
x'(0)\ar[rd]\ar@{-->}[rr] && x(0)\ar@{-->}[rd]\\
&x'(1)\ar[rr] && x(1),
}
\end{equation}
such that the front and back squares are pullback, such that the right square is in $\cQ$, and such that the bottom square is in $\cQ'$, the left square is in $\cQ$. Here we interpret the horizontal and vertical arrows as in the first (resp.\ second) direction and the oblique arrows as in the other direction.

\begin{lem}\label{a3le:cart}
Let $\cC$ be an $\infty$-category. Let $\cQ^\cart$ be the set of all pullback squares of $\cC$. Then the image of $\cQ^\cart$ under each of the
two isomorphisms in Definition \ref{a3de:squares} is an admissible set of edges. In particular, $\cQ^\cart$ is stable under equivalence, stable under composition in both directions, and stable under pullback in both directions.
\end{lem}

\begin{proof}
The last condition in the definition of admissibility is \cite{HTT}*{Lemma 4.4.2.1}. It remains to show the stability under pullback. Consider a cube of the form \eqref{a3eq:cube} in which the front, back, and right squares are pullback. By the ``if'' part of (1), the square with vertices $y'(0)$, $y(1)$, $x'(0)$, $x(1)$ is a pullback square. We conclude by the ``only if'' part of (1).
\end{proof}

\begin{remark}\label{a3re:cart_square}
Let $\cC$ be an $\infty$-category and let $\cE_1$, $\cE_2$, $\cE_3$ be sets of edges of $\cC$. Lemma \ref{a3le:cart} has the following consequences.
\begin{enumerate}
  \item If $\cE_1$ is stable under composition, then $\cE_1*^\cart_\cC\cE_2$ is stable under composition in the first direction.

  \item If $\cE_2$ and $\cE_3$ are stable under pullback by $\cE_1$, then $\cE_2*_\cC\cE_3$ and $\cE_2*_\cC^\cart\cE_3$ are stable under
      pullback by $\cE_1*^\cart_\cC \cE_3$ in the first direction.

  \item If $\cE_3$ is stable under pullback by $\cE_2$, and $\cE_2$ is stable under pullback by $\cE_1$, then $\cE_2*_\cC^\cart\cE_3$ is
      stable under pullback by $\cE_1*_\cC \cE_3$ (and, in particular, by $\cE_1*_\cC^\cart \cE_3$) in the first direction.
\end{enumerate}
\end{remark}

\begin{remark}
Let $\cC$ be an ordinary category, and let $\cE_1,\dots,\cE_k$ be sets of morphisms of $\cC$ stable under composition and containing identity
morphisms. Then $\rN(\cC)_{\cE_1,\dots,\cE_k}$ and $\rN(\cC)^\cart_{\cE_1,\dots,\cE_k}$ can be interpreted as the $k$-fold nerves in the sense of Fiore and Paoli \cite{FP}*{Definition 2.14} of suitable $k$-fold categories. More generally, if $\cQ_{ij}$ are sets of squares stable under composition in both directions such that $(\rN(\cC),\cE,\cQ)$ is a $k$-tiled $\infty$-category, then $\delta^{k\square}_*(\rN(\cC),\cE,\cQ)$ is
the $k$-fold nerve of a suitable $k$-fold category.
\end{remark}

\subsection{Multisimplicial descent}
\label{a4ss}

In this section, we study the map of simplicial sets obtained by composing two directions in a multisimplicial nerve. The main result is Theorem \ref{a4th:multisimplicial_descent}, which is a general criterion for the map to be a categorical equivalence. We then give more specific sufficient conditions in two important special cases: Theorem \ref{a4co:multisimplicial_descent} and Theorem \ref{a4pr:descent}. The latter can be regarded as a generalization of Deligne's result \cite{SGA4}*{Expos\'e xvii, Proposition 3.3.2} (see Remark \ref{a4re:deligne}).

In Deligne's theory, a fundamental role is played by the category of compactifications of a morphism $f$, whose objects are factorizations of $f$ as $p\circ q$, where $p$, $q$ belong respectively to the two classes of morphisms in question. To properly formulate compactifications of simplices of higher dimensions, we introduce a bit of notation.

We identify \emph{partially ordered sets} with ordinary categories in which there is at most one arrow between each pair of objects, by the convention $p\le q$ if and only if there exists an arrow $p\to q$. For every element $p\in P$, we identify the overcategory $P_{/p}$ (resp.\ undercategory $P_{p/}$) with the full partially ordered subset of $P$ consisting of elements $\leq p$ (resp.\ $\geq p$). For $p,p'\in P$, we identify $P_{p//p'}$ with the full partially ordered subset of $P$ consisting of elements both $\geq p$ and $\leq p'$, which is empty unless $p\leq p'$. For a subset $Q$ of $P$, we write $Q_{p/}=Q\cap P_{p/}$, etc.

\begin{notation}
Let $n\geq0$ be an integer. We consider the bisimplicial set $\Delta^{n,n}$ and the partially ordered set $[n]\times [n]$, related by the natural isomorphisms of simplicial sets $\delta_2^*\Delta^{n,n}\simeq\Delta^n\times\Delta^n\simeq\rN([n]\times[n])$. We enumerate their vertices by coordinates $(i,j)$ for $0\leq i,j\leq n$. We define $\Cpt^n\subseteq\Delta^{n,n}$ to be the bisimplicial subset obtained by the vertices $(i,j)$ with $0\leq i\leq j\leq n$. We define $\RCpt^n\subseteq[n]\times [n]$ to be the full partially ordered subset spanned by $(i,j)$ with $0\leq i\leq j\leq n$. We have
\begin{align*}
\delta_2^*\Cpt^n\simeq \Box^n\subseteq\CCpt^n\coloneqq\rN(\RCpt^n),
\end{align*}
where we have put $\Box^n\coloneqq\bigcup_{k=0}^n\Box_k^n$ and $\Box_k^n\coloneqq\rN(\RCpt^n_{(0,k)//(k,n)})$ is the nerve of the full
partially ordered subset of $[n]\times[n]$ spanned by $(i,j)$ with $0\leq i\leq k\leq j\leq n$.
\end{notation}

Below is the Hasse diagram of $\RCpt^3$, rotated so that the initial object is shown in the upper-left corner. The dashed box represents $\Box^3_1$, while bullets represent elements in the image of the diagonal embedding $[3]\to \RCpt^3$.
\begin{equation}\label{a4eq:Cpt}
\begin{xy}
(0,5)="00"; (5,5)*\cir<1.8pt>{}="01"**\dir{-};
(10,5)*\cir<1.8pt>{}="02"**\dir{-}; (15,5)*\cir<1.8pt>{}="03"**\dir{-};
(5,0)="11"; (10,0)*\cir<1.8pt>{}="12"**\dir{-};
(15,0)*\cir<1.8pt>{}="13"**\dir{-};
(10,-5)="22";
(15,-5)*\cir<1.8pt>{}="23"**\dir{-};
(15,-10)="33"**\dir{-};
"01"; "11"**\dir{-};
"02"; "12"**\dir{-}; "22"**\dir{-};
"03"; "13"**\dir{-}; "23"**\dir{-};
"00"*{\bullet}; "11"*{\bullet}; "22"*{\bullet}; "33"*{\bullet};
(3,-2)="b"; (3,7)**\dir{--}; (17,7)**\dir{--}; (17,-2)**\dir{--}; "b"**\dir{--};
\end{xy}
\end{equation}
Note that the first coordinate is represented vertically and the second one is represented horizontally.

We now review compactifications in ordinary categories.

\begin{definition}\label{a4de:comp}
Let $\cC$ be an ordinary category and let $\cE_1$, $\cE_2$ be two sets of morphisms of $\cC$ containing all identity morphisms. Let $\tau\colon [n]\to\cC$ be a functor, corresponding to a sequence of morphisms
\[
c_0\to c_1\to \dots \to c_n.
\]

We define a \emph{compactification of $\tau$} to be a functor $\sigma\colon\RCpt^n\to \cC$ satisfying the following conditions:
\begin{enumerate}
  \item The functor $\sigma$ carries ``vertical'' morphisms $(i,j)\to(i',j)$ of $\RCpt^n$ into $\cE_1$ and ``horizontal'' morphisms
      $(i,j)\to (i,j')$ into $\cE_2$.

  \item The composition $[n]\to \RCpt^n \xrightarrow{\sigma} \cC$ is $\tau$. Here $[n]\to\RCpt^n$ is the diagonal functor carrying $i$ to $(i,i)$.
\end{enumerate}

Assume that $\cE_\alpha$ is stable under composition for $\alpha=1$ or $\alpha=2$. The compactifications of $\tau$ can be organized into a category $\Kpt^\alpha(\tau)$ as follows. Given two compactifications $\sigma,\sigma'\colon \RCpt^n\to \cC$ of $\tau$, a morphism in
$\Kpt^\alpha(\tau)$ is a natural transformation $\gamma\colon \sigma\to\sigma'$ satisfying the following conditions:
\begin{enumerate}
  \item For every $(i,j)\in \RCpt^n$, the morphism $\gamma(i,j)\colon\sigma(i,j)\to \sigma'(i,j)$ is in $\cE_\alpha$.

  \item The restriction of $\gamma$ to $[n]$ via the diagonal functor is $\id_\tau$.
\end{enumerate}
\end{definition}

For $\alpha=1$ (and $n\le 3$), $\Kpt^1(\tau)$ is the category of compactifications considered by Deligne \cite{SGA4}*{Expos\'e xvii, D\'efinition 3.2.5}.

In the language of $2$-marked simplicial sets, we can reformulate the two conditions (1) in Definition \ref{a4de:comp} as follows. Condition (1) in the definition of compactifications means that the restriction of $\rN(\sigma)\colon \CCpt^n\to \rN(\cC)$ to $\Box^n$ induces a map of $2$-marked simplicial sets $\delta^*_{2+}\Cpt^n\to (\rN(\cC),\cE_1,\cE_2)$. Condition (1) in the definition of morphisms means that the restriction of $\rN(\gamma)\colon \Delta^1\times \RCpt^n\to \rN(\cC)$ to $\Delta^1\times\Box^n$, where $\gamma$ is regarded as a functor $[1]\times\RCpt^n\to \cC$, induces a map of $2$-marked simplicial sets $(\Delta^1)^{\sharp^{2}_{\{\alpha\}}}\times \delta^*_{2+}\Cpt^n\to(\rN(\cC),\cE_1,\cE_2)$. See Definition \ref{a3de:multimarked_simplicial} for the notation $(-)^{\sharp^{2}_{\{\alpha\}}}$.

We now define compactifications in $\infty$-categories, and more generally in simplicial sets. Besides the need to deal with simplices of higher dimensions, the definition is more complicated in two other ways: we consider an extra set $K$ of ``directions'' and we consider restrictions not only on edges, but also on squares, which leads to the use of multi-tiled simplicial sets.

\begin{definition}\label{a4de:compactification}
Let $K$ be a set and let $(X,\cT)$ be a $(\{1,2\}\coprod K)$-tiled simplicial set. For $L\subseteq K$, integers $n,n_k\geq 0$ ($k\in K$), a map $\tau\colon\Delta^{n,n_k\res k\in K}_L\to \delta_*^{\{0\}\amalg K}X$, and $\alpha\in\{1,2\}\coprod K$, we define $\Komp^\alpha(\tau)=\Komp^\alpha_{(X,\cT)}(\tau)$, the \emph{$\alpha$-th simplicial set of compactifications of $\tau$}, to be the limit of the
diagram
\begin{equation}\label{a4eq:limit}
\resizebox{\hsize}{!}{
\xymatrix{
&&\epsilon^{\{1,2\}\amalg K}_\alpha
\Map(\Cpt^n\boxtimes\Delta^{n_k\res {k\in K}}_L,\delta^{(\{1,2\}\amalg K)\square}_*(X,\cT))\ar@{^{(}->}[d]^g\\
&\Map(\CCpt^n\times\Delta^{[n_k]_{k\in K}}_L,X)
\ar[r]^-{\rres_1}\ar[d]^-{\rres_2}
& \Map(\Box^n\times\Delta^{[n_k]_{k\in K}}_L,X)\\
\{\tau\}\ar@{^{(}->}[r]& \Map(\Delta^{[n,n_k]_{k\in K}}_L,X)
}
}
\end{equation}
in the category $\Sset$ of simplicial sets, where
\begin{itemize}
  \item we regard $\tau$ as a map $\Delta^{[n,n_k]_{k\in K}}_L\to X$, hence a vertex of $\Map(\Delta^{[n,n_k]_{k\in K}}_L,X)$;

  \item $\rres_1$ is induced by the inclusion $\Box^n\subseteq \CCpt^n$;

  \item $\rres_2$ is induced by the diagonal map $\Delta^n\to \CCpt^n$; and

  \item $g$ is the composition of maps
      \begin{align*}
      &\quad\epsilon^{\{1,2\}\amalg K}_\alpha\Map(\Cpt^n\boxtimes\Delta^{n_k\res {k\in K}}_L,\delta^{(\{1,2\}\amalg K)\square}_*(X,\cT)) \\
      &\hookrightarrow\epsilon^{\{1,2\}\amalg K}_\alpha
      \Map(\Cpt^n\boxtimes\Delta^{n_k\res {k\in K}}_L,\delta^{\{1,2\}\amalg K}_*X)\\
      &\simeq\Map(\Box^n\times\Delta^{[n_k]_{k\in K}}_L, X),
      \end{align*}
      where the isomorphism is the adjunction formula of Remark \ref{a3re:adjunction}.
\end{itemize}

For a $(\{1,2\}\coprod K)$-marked simplicial set $(X,\cE)$, we put $\Komp^\alpha_{X,\cE}(\tau)\coloneqq\Komp^\alpha_{\sfW(X,\cE)}(\tau)$, where $\sfW$ is the functor in Notation \ref{a3no:cartesian}. We put $\Komp^\alpha(\tau)_L\coloneqq\Komp^\alpha(\tau)$ if $\alpha\not\in L$, and $\Komp^\alpha(\tau)_L\coloneqq\Komp^\alpha(\tau)^{op}$ if $\alpha\in L$.
\end{definition}

For brevity, we sometimes write $I$ for $\{1,2\}\coprod K$.

\begin{remark}\label{a4re:g}
Let us give a more explicit description of $g$ in \eqref{a4eq:limit}. To simplify notation, we let $Y$ denote the source of $g$. We let
$\iota_\alpha\colon \{1\}\to I$ denote the map with image $\alpha$. For any simplicial set $S$, we have isomorphisms
\begin{align*}
\Hom_{\Sset}(S,Y)&\simeq \Hom_{\Sset[I]}((\del^{\iota_\alpha})^*S,\Map(\Cpt^n\boxtimes\Delta^{n_k\res {k\in K}}_L,\delta^{I\square}_*(X,\cT)))\\
&\simeq \Hom_{\Sset[I]}((\del^{\iota_\alpha})^*S\times (\Cpt^n\boxtimes\Delta^{n_k\res {k\in K}}_L),\delta^{I\square}_*(X,\cT))\\
&\simeq \Hom_{\Sset^{I\square}}(\delta_{I\square}^*(\del^{\iota_\alpha})^*S\times
\delta_{I\square}^*(\Cpt^n\boxtimes\Delta^{n_k\res {k\in K}}_L),(X,\cT))\\
&\simeq \Hom_{\Sset^{I\square}}(\sfV(S^{\sharp^I_{\{\alpha\}}})\times \sfW\delta_{I+}^*(\Cpt^n\boxtimes\Delta^{n_k\res {k\in K}}_L),(X,\cT)),
\end{align*}
where $\sfV$ and $\sfW$ are the functor in Notation \ref{a3no:cartesian}. Here in the last step we have used the isomorphisms
\[
\delta_{I\square}^*(\del^{\iota_\alpha})^*S\simeq \sfV(S^{\sharp^I_{\{\alpha\}}}),
\quad \delta_{I\square}^*(\Cpt^n\boxtimes\Delta^{n_k\res {k\in K}}_L)\simeq \sfW\delta_{I+}^*(\Cpt^n\boxtimes\Delta^{n_k\res {k\in K}}_L).
\]
We define $\{\cF_\beta\}_{\beta\in I}$ by the isomorphism
\[
\delta_{I+}^*(\Cpt^n\boxtimes\Delta^{n_k\res {k\in K}}_L)\simeq(\square^n\times \Delta^{[n_k]_{k\in K}}_L,\{\cF_\beta\}_{\beta \in I}).
\]
In other words, $\cF_\beta$ is the set of edges of $\epsilon^I_\beta(\Cpt^n\boxtimes\Delta^{n_k\res {k\in K}}_L)$, for all $\beta\in I$. Then,
\begin{itemize}
  \item A vertex of $Y$ is precisely a map of $I$-marked simplicial sets $\delta_{I\square}^*(\Cpt^n\boxtimes\Delta^{n_k\res {k\in K}}_L)\to
      (X,\cT)$. In other words, a map $\sigma\colon \square^n\times\Delta^{[n_k]_{k\in K}}_L\to X$ is a vertex of $Y$ if and only if it carries $\cF_\beta*\cF_{\beta'}$ into $\cT_{\beta\beta'}$ for all $\beta,\beta'\in I$ with $\beta\neq \beta'$. As we observed in Remark \ref{a3re:multitile}, the condition implies that $\sigma$ carries $\cF_\beta$ into $\cT_\beta$ for all $\beta \in I$.

  \item Given vertices $\sigma$, $\sigma'$ of $Y$, an edge $\gamma\colon\sigma\to \sigma'$ of $\Map(\square^n\times\Delta^{[n_k]_{k\in K}}_L,X)$ is an edge of $Y$ if and only if, for every $\beta\in I$, $\beta\neq \alpha$ and for every square
      \begin{equation}\label{a4eq:Komp}
      \xymatrix{y'\ar[d]\ar[r] & y\ar[d]\\x'\ar[r] & x}
      \end{equation}
      in $\cF_\alpha*\cF_\beta$, with vertical arrows in $\cF_\alpha$ and horizontal arrows in $\cF_\beta$, $\gamma$ carries the square
      \begin{equation}\label{a4eq:Komp2}
      \xymatrix{(0,y')\ar[d]\ar[r] & (0,y)\ar[d]\\(1,x')\ar[r] & (1,x)}
      \end{equation}
      to a square in $\cT_{\alpha\beta}$. Here we have regarded $\gamma$ as a map $\Delta^1\times (\square^n\times \Delta^{[n_k]_{k\in
      K}}_L)\to\cC$. We note two special cases of the condition:
      \begin{enumerate}
        \item For every edge $y\to x$ in $\cF_\alpha$, $\gamma$ carries $(0,y)\to (1,x)$ to an edge in $\cT_\alpha$.

        \item For every $\beta\in I$, $\beta\neq \alpha$ and for every edge $x'\to x$ in $\cF_\beta$, $\gamma$ carries the
            square
            \[
            \xymatrix{
            (0,x')\ar[d]\ar[r] & (0,x)\ar[d]\\(1,x')\ar[r] & (1,x)
            }
            \]
            to a square in $\cT_{\alpha\beta}$.
      \end{enumerate}
      If $\cT_{\alpha\beta}$ is stable under composition in the first direction for every $\beta$, then Condition (2) is also a sufficient condition for $\gamma$ to be an edge of $Y$.

  \item For $m\ge 2$, an $m$-simplex $\gamma$ of $\Map(\square^n\times\Delta^{[n_k]_{k\in K}}_L,X)$ is an $m$-simplex of $Y$ if and only if each edge of $\gamma$ is an edge of $Y$.
\end{itemize}
In particular, $g$ satisfies the (unique) right lifting property with respect to $\partial \Delta^m\subseteq \Delta^m$ for $m\ge 2$.
\end{remark}

\begin{remark}
In the situation of Definition \ref{a4de:comp}, we have a canonical isomorphism $\Komp^\alpha_{\rN(\cC),\cE_1,\cE_2}(\tau)\simeq\rN(\Kpt^\alpha(\tau))$. We will see in Lemma \ref{a4le:infinity_category} that the simplicial set $\Komp^\alpha_{(X,\cT)}$ is an $\infty$-category under mild hypotheses.
\end{remark}

\begin{remark}\label{a4re:fibration}
We let $D^n$ denote the intersection of $\square^n$ and the diagonal embedding $\Delta^n\to \CCpt^n$. Then $D^n$ is the disjoint union of $n+1$
points. Note that the diagram \eqref{a4eq:limit} can be completed into a commutative diagram
\begin{align*}
\resizebox{\hsize}{!}{
\xymatrix{
\Map(\CCpt^n\times\Delta^{[n_k]_{k\in K}}_L,X)\ar[rd]^-{\rres_4}\ar@/^1.5pc/[rrd]^-{\rres_1}\ar@/_1.5pc/[ddr]_-{\rres_2}
&&\epsilon^I_\alpha\Map(\Cpt^n\boxtimes \Delta^{n_k\res k\in K}_L, \delta^{I\square}_*(X,\cT))\ar[d]^g\\
&\Map((\Box^n\coprod_{D^n}
\Delta^n)\times\Delta^{[n_k]_{k\in K}}_L,X)\ar[r]\ar[d]
& \Map(\Box^n\times\Delta^{[n_k]_{k\in K}}_L,X)\ar[d]^-{\rres_3}\\
\{\tau\}\ar[r]& \Map(\Delta^{n}\times\Delta^{[n_k]_{k\in K}}_L,X)\ar[r] &
\Map(D^n\times\Delta^{[n_k]_{k\in K}}_L,X),
}
}
\end{align*}
where the lower right square is a pullback. Here the maps in the lower right square (including $\rres_3$) and $\rres_4$ are obvious restrictions. If $X$ is an $\infty$-category, then $\rres_i$, $2\le i\le 4$ are Cartesian fibrations (and coCartesian fibrations) by \cite{HTT}*{Proposition 3.1.2.1} and $\rres_1$ is a trivial Kan fibration by Lemma \ref{a6le:cpt_inner} and \cite{HTT}*{Corollaries 2.3.2.4, 2.3.2.5}. Moreover, $\rres_1$ is an isomorphism if $X$ is isomorphic to the nerve of an ordinary category.
\end{remark}

\begin{remark}
We have introduced $K$ in the definition mainly for convenience. In the case where $\alpha\in\{1,2\}$, which is our main case of interest, we could reach the same generality without $K$. In fact, we can define a $\{1,2\}$-tiled simplicial set $(X',\cT')$, where $X'$ is the full simplicial subset of $\Map(\Delta^{[n_k]_{k\in K}}_L,X)$ spanned by maps corresponding to maps $\Delta^{n_k\res {k\in K}}_L\to\delta_*^{K\Box}(X,\cT_K)\subseteq\delta_*^K X$ (where $\cT_K$ denotes the $K$-tiling induced by $\cT$), with the following property: If $\tau$ defines an $n$-simplex $\tau'$ of $X'$, then we have an isomorphism $\Komp^\alpha_{(X,\cT)}(\tau)\simeq \Komp^\alpha_{(X',\cT')}(\tau')$; otherwise $\Komp^\alpha_{(X,\cT)}(\tau)$ is empty.
\end{remark}

Note that by Remark \ref{a3re:adjunction}, the map $g$ is also equal to the composition
\begin{align*}
&\quad\epsilon^{\{1,2\}\amalg K}_\alpha\Map(\Cpt^n\boxtimes\Delta^{n_k\res {k\in K}}_L,\delta^{(\{1,2\}\amalg K)\square}_*(X,\cT))\\
&\xrightarrow{\delta^*_{\{1,2\}\amalg K}}\Map(\Box^n\times\Delta^{[n_k]_{k\in K}}_L,\delta_{\{1,2\}\amalg
K}^*\delta^{(\{1,2\}\amalg K)\square}_*(X,\cT))\\
&\hookrightarrow\Map(\Box^n\times\Delta^{[n_k]_{k\in K}}_L,\delta_{\{1,2\}\amalg
K}^*\delta^{\{1,2\}\amalg K}_*X)\\
&\to\Map(\Box^n\times\Delta^{[n_k]_{k\in K}}_L,X),
\end{align*}
where the last map is induced by the counit map. We consider the composition
\begin{align}\label{a4eq:compactification_phi}
\phi(\tau)\colon \Komp^\alpha(\tau)_L &\to \epsilon^{I}_\alpha\op^{I}_L\Map(\Cpt^n\boxtimes\Delta^{n_k\res {k\in K}}_L,\delta^{I\square}_*(X,\cT)) \\
&\xrightarrow{\delta^*_{I}} \Map(\square^n\times\Delta^{[n_k]_{k\in K}},\delta_{I,L}^*\delta^{I\square}_*(X,\cT)), \notag
\end{align}
which will be used in the proof of Theorem \ref{a4th:multisimplicial_descent} below.

\begin{remark}\label{a4re:phi}
By construction, the composition
\begin{align*}
\Komp^\alpha(\tau)_L & \xrightarrow{\phi(\tau)}
\Map(\Box^n\times\Delta^{[n_k]_{k\in K}},\delta_{I,L}^*\delta^{I\square}_*(X,\cT))\\
& \to \Map(D^n\times
\Delta^{[n_k]_{k\in K}},\delta_{I,L}^*\delta^{I\square}_*(X,\cT)),
\end{align*}
where the second map is induced by the inclusion $D^n\subseteq \square^n$ (see Remark \ref{a4re:fibration} for the notation), is constant of value $\delta_{I,L}^*\tau_0$, where
\[
\tau_0\colon \delta^2_*(D^n)\boxtimes\Delta^{n_k\res{k\in K}}_L\to \delta^{I}_*X
\]
is the restriction of $\tau$. If $\Komp^\alpha(\tau)$ is nonempty, then $\tau_0$ factorizes through $\delta^{I\square}_*(X,\cT)$.
\end{remark}

Next we consider $(\{0\}\coprod K)$-tilings. Let $(X,\cT')$ be a $(\{0\}\coprod K)$-tiled simplicial set. For brevity we sometimes write $J$
for $\{0\}\coprod K$. For $L\subseteq K$ and $\alpha'\in J$, we have the commutative diagram
\begin{equation}\label{a4eq:zeroK}
\resizebox{\hsize}{!}{
\xymatrix{
\epsilon^J_{\alpha'}\Map(\CCpt^n\boxtimes\Delta^{n_k\res {k\in K}}_L,\delta^{J\square}_*(X,\cT')) \ar@{^{(}->}[rr] \ar[d]_{\delta^*_J}
&&\epsilon^J_{\alpha'}
\Map(\CCpt^n\boxtimes\Delta^{n_k\res {k\in K}}_L,\delta^J_*X)\ar[d]^-{\simeq}\\
\Map(\CCpt^n\times\Delta^{[n_k]_{k\in K}}_L,\delta_J^*\delta^{J\square}_*(X,\cT')) \ar@{^{(}->}[r] &
\Map(\CCpt^n\times\Delta^{[n_k]_{k\in K}}_L,\delta^*_J\delta_*^J X)\ar[r]&
\Map(\CCpt^n\times\Delta^{[n_k]_{k\in K}}_L, X)
}
}
\end{equation}
by Remark \ref{a3re:adjunction}. This is similar to the situation of the map $g$ in Definition \ref{a4de:compactification}.

To compare the restricted multisimplicial nerves of $(X,\cT)$ and of $(X,\cT')$, we make some assumptions.

\begin{assumption}\label{a4as:composition}
Let $(X,\cT)$ be a $(\{1,2\}\coprod K)$-tiled simplicial set and let $(X,\cT')$ be a $(\{0\}\coprod K)$-tiled simplicial set. Consider the
following assumptions:
\begin{enumerate}
  \item For $\sigma\colon \Delta^2\to X$ with $\sigma\circ d^2_0\in\cT_1$, $\sigma\circ d^2_2\in \cT_2$, we have $\sigma\circ d^2_1\in\cT'_0$.

  \item For $k\in K$ and $\sigma \colon \Delta^2\times \Delta^1\to X$ satisfying $\sigma\circ (d^2_0\times \id)\in \cT_{1k}$,
      $\sigma\circ (d^2_2\times \id)\in \cT_{2k}$, we have $\sigma\circ(d^2_1\times \id)\in \cT'_{0k}$.

  \item For $k\in K$, we have $\cT_k\subseteq \cT'_k$. For distinct elements $k,k'\in K$, we have $\cT_{kk'}\subseteq \cT'_{kk'}$.
\end{enumerate}
\end{assumption}

Note that (2) implies (1) if $K$ is nonempty.

\begin{remark}\label{a4re:assump}
Assumption (1) implies $\cT_1,\cT_2\subseteq \cT'_0$. Conversely, if we have $\cT_1,\cT_2\subseteq \cT'_0$ and $\cT'_0$ is stable under composition, then Assumption (1) holds. Similarly, Assumption (2) implies $\cT_{1k},\cT_{2k}\subseteq \cT'_{0k}$. Conversely, if we have
$\cT_{1k},\cT_{2k}\subseteq \cT'_{0k}$ and $\cT'_{0k}$ is stable under composition in the first direction, then Assumption (2) holds.
\end{remark}

We consider the maps $\mu_0\colon \{1,2\}\to \{0\}$ and $\mu=\mu_0\coprod\id_K \colon \{1,2\}\coprod K\to \{0\}\coprod K$. For brevity, we sometimes write $I$ for $\{1,2\}\coprod K$ and $J$ for $\{0\}\coprod K$.

\begin{lem}\label{a4le:multisimp}
Suppose that Assumption \ref{a4as:composition} is satisfied. Then
\begin{enumerate}
  \item The isomorphism $\delta_*^{\{1,2\}\amalg K}X\simeq(\del^\mu)_*\delta_*^{\{0\}\amalg K} X$ induces an inclusion
      \[
      \delta_*^{(\{1,2\}\amalg K)\square}(X,\cT)\subseteq(\del^\mu)_*\delta_*^{(\{0\}\amalg K)\square}(X,\cT').
      \]

  \item The pullback of $g$ by $\rres_1$ in Definition \ref{a4de:compactification} factorizes through the upper-left corner of the diagram \eqref{a4eq:zeroK} with $\alpha'=\mu(\alpha)$.
\end{enumerate}
\end{lem}

\begin{proof}
(1) We have
\[
\delta^*_{I\square}\Delta^{n_k\mid k\in I}=\sfW\delta^*_{I+}\Delta^{n_k\mid k\in I},\quad
\delta^*_{J\square}(\del^\mu)^*\Delta^{n_k\mid k\in I}\simeq \sfW\delta^*_{J+}(\Delta^{[n_1,n_2]}\boxtimes\Delta^{n_k\mid k\in K}).
\]
We let $\cG_\beta$ denote the set of edges of $\epsilon^I_\beta\Delta^{n_k\mid k\in I}$ for $\beta\in I$, and let $\cG_0$ denote the set of edges of $\epsilon^J_0(\Delta^{[n_1,n_2]}\boxtimes\Delta^{n_k\mid k\in K})$. Then we have
\begin{align*}
\delta^*_{I+}\Delta^{n_k\mid k\in I}&\simeq (\Delta^{[n_k]_{k\in I}},\{\cG_\beta\}_{\beta\in I}),\\ \delta^*_{J+}(\Delta^{[n_1,n_2]}\boxtimes\Delta^{n_k\mid k\in K})&\simeq (\Delta^{[n_k]_{k\in I}},\{\cG_\beta\}_{\beta\in J}).
\end{align*}
Thus an $(n_k)_{k\in I}$-simplex of $\delta_*^{I\square}(X,\cT)$ is given by a map $\sigma\colon \Delta^{[n_k]_{k\in I}}\to X$ carrying $\cG_\beta$ into $\cT_\beta$ for all $\beta\in I$ and carrying $\cG_\beta*\cG_{\beta'}$ into $\cT_{\beta\beta'}$ for all $\beta,\beta'\in I$, $\beta\neq \beta'$.

Let us show first that $\sigma$ carries $\cG_\beta$ into $\cT'_\beta$ for all $\beta\in J$. For $\beta\in K$, this follows from the assumption
$\cT_\beta\subseteq \cT'_\beta$. An edge $e$ in $\cG_0$ has the form $(i,j,a)\to (i',j',a)$, where $a=(a_k)_{k\in K}$. Consider the $2$-simplex
\[
\xymatrix{
(i,j,a)\ar[r]^-{e''}\ar[rd]_-e&(i,j',a)\ar[d]^-{e'}\\&(i',j',a)
}
\]
of $\Delta^{[n_k]_{k\in I}}$, where $e'$ is in $\cG_1$ and $e''$ is in $\cG_2$. By Assumption \ref{a4as:composition} (1), $\sigma(e)$ is in $\cT'_0$.

Next we show that $\sigma$ carries $\cG_{\beta}*\cG_{\beta'}$ into $\cT'_{\beta\beta'}$ for all $\beta,\beta'\in J$, $\beta\neq \beta'$. For
$\beta,\beta'\in K$, this follows from the assumption $\cT_{\beta\beta'}\subseteq \cT'_{\beta\beta'}$. It remains to show that $\sigma$ carries $\cG_0*\cG_\beta$ into $\cT'_{0\beta}$ for $\beta\in K$. Every square $\varsigma$ in $\cG_0*\cG_\beta$ can be extended to a map $\Delta^2\times\Delta^1\to\Delta^{[n_k]_{k\in I}}$ as shown by the diagram
\[
\xymatrix{
(i,j,b)\ar[r]\ar[d] & (i,j,a)\ar[d]\\
(i,j',b)\ar[r]\ar[d] & (i,j',a)\ar[d]\\
(i',j',b)\ar[r] & (i',j',a)
}
\]
with $\varsigma$ as the outer square. The upper square is in $\cG_2*\cG_\beta$ and the lower square is in $\cG_1*\cG_\beta$. Thus, by
Assumption \ref{a4as:composition} (2), $\sigma(\varsigma)$ is in $\cT'_{0\beta}$.

(2) We let $Y'$ denote the pullback of $g$ by $\rres_1$ and let $Z$ denote the upper-left corner of the diagram \eqref{a4eq:zeroK}. We adopt the
notation of Remark \ref{a4re:g}. Note that for $\beta\in K$, $\cF_\beta$ can be identified with the set of edges of $\epsilon^J_\beta(\CCpt^n\boxtimes\Delta^{n_k\mid k\in K}_L)$. We let $\cF_0$ denote the set of edges of $\epsilon^J_0(\CCpt^n\boxtimes \Delta^{n_k\mid k\in K}_L)$. Then
\[
\delta^*_{J+}(\CCpt^n\boxtimes \Delta^{n_k\mid k\in K}_L)\simeq(\CCpt^n\times \Delta^{[n_k]_{k\in K}}_L,\{\cF_\beta\}_{\beta \in J}).
\]
Note that $Z$ admits a description similar to the description of $Y$ in Remark \ref{a4re:g}. In particular, for $m\ge 2$, $Z\hookrightarrow\Map(\CCpt^n\times\Delta^{[n_k]_{k\in K}}_L,X)$ has the right lifting property with respect to $\partial\Delta^m\subseteq\Delta^m$. Thus it suffices to check $Y'\subseteq Z$ on the level of vertices and edges.

Let $\sigma\colon \CCpt^n\times \Delta^{[n_k]_{k\in K}}_L\to X$ be a vertex of $Y'$. To show that $\sigma$ is a vertex of $Z$, we need to check that $\sigma$ carries $\cF_\beta$ to $\cT'_\beta$ for all $\beta\in J$ and carries $\cF_\beta*\cF_{\beta'}$ to $\cT'_{\beta\beta'}$ for all
$\beta,\beta'\in J$, $\beta\neq \beta'$. The proof is similar to that of (1). Note that for every edge $(i,j)\to (i',j')$ of $\CCpt^n$, $(i,j')$ is a vertex of $\CCpt^n$.

Let $\gamma$ be an edge of $Y'$, regarded as a map $\Delta^1\times(\CCpt^n\times \Delta^{[n_k]_{k\in K}}_L)\to X$. To show that $\gamma$ is an
edge of $Z$, we first check that for every edge $y\to x$ of $\cF_{\mu(\alpha)}$, $\gamma$ carries the edge $e\colon (0,y)\to (1,x)$ to an edge in $\cT'_{\mu(\alpha)}$. If $\alpha\in K$, then this follows from the assumption $\cT_\alpha\subseteq \cT'_\alpha$. If $\alpha\in \{1,2\}$,
then $e$ can be completed into a $2$-simplex of the form
\[
\xymatrix{
(0,i,j,a)\ar[r]^-{e''}\ar[rd]_-e & (\alpha-1,i,j',a)\ar[d]^-{e'}\\&(1,i',j',a).
}\]
Since $\gamma(e')$ is in $\cT_1$ and $\gamma(e'')$ is in $\cT_2$, we have $\gamma(e)\in \cT'_0$ by Assumption \ref{a4as:composition} (1). Finally we check that for every $\beta\in J$, $\beta\neq \mu(\alpha)$ and every square of the form \eqref{a4eq:Komp} in $\cF_{\mu(\alpha)}*\cF_\beta$ with vertical arrows in $\cF_{\mu(\alpha)}$ and horizontal arrows in $\cF_\beta$, $\gamma$ carries the square \eqref{a4eq:Komp2} to a square in
$\cT'_{\mu(\alpha)\beta}$. If $\alpha,\beta\in K$, then this follows from the assumption $\cT_{\alpha\beta}\subseteq \cT'_{\alpha\beta}$. In the remaining cases we apply Assumption \ref{a4as:composition} (2). If $\beta=0$, we factorize the square horizontally. If $\alpha\in \{1,2\}$, we factorize the square vertically, with the first component of the middle row given by $\alpha-1$.
\end{proof}

\begin{construction}
The main result of this section, Theorem \ref{a4th:multisimplicial_descent} below, is about the composition
\begin{align}\label{a4eq:composition}
\delta_{\{1,2\}\amalg K,L}^*\delta^{(\{1,2\}\amalg K)\square}_*(X,\cT)
&\simeq \delta_{\{0\}\amalg K,L}^*(\del^\mu)^*\delta^{(\{1,2\}\cup K)\square}_*(X,\cT) \\
& \hookrightarrow\delta_{\{0\}\amalg K,L}^*(\del^\mu)^*(\del^\mu)_*\delta^{(\{0\}\amalg K)\square}_*(X,\cT') \notag\\
&\to\delta_{\{0\}\amalg K,L}^*\delta^{(\{0\}\amalg K)\square}_*(X,\cT'), \notag
\end{align}
where the inclusion in the middle is given by Lemma \ref{a4le:multisimp} (1) and the last map is the counit map. An $n$-simplex of the left hand side of \eqref{a4eq:composition} corresponds to a map $\Delta^n\times\Delta^n\times(\Delta^n)^K\to X$. The map \eqref{a4eq:composition} carries it to the $n$-simplex corresponding to the composition
\[
\Delta^n\times(\Delta^n)^K\xrightarrow{\diag\times\id_{(\Delta^n)^K}}\Delta^n\times\Delta^n\times(\Delta^n)^K\to X,
\]
where $\diag\colon\Delta^n\to\Delta^n\times\Delta^n$ is the diagonal map.

For any map $\tau\colon \Delta^{n,n_k\res k\in K}_L\to \delta_*^{\{0\}\amalg K}X$, we consider the composition
\begin{align}\label{a4eq:compactification_psi}
\psi(\tau)\colon \Komp^\alpha(\tau)_L &\to \epsilon^{J}_{\mu(\alpha)}\op^{J}_L
\Map(\CCpt^n\boxtimes\Delta^{n_k\res {k\in K}}_L,\delta^{J\square}_*(X,\cT')) \\
& \xrightarrow{\delta^*_{J}}
\Map(\CCpt^n\times\Delta^{[n_k]_{k\in K}},\delta_{J,L}^*\delta^{J\square}_*(X,\cT')), \notag
\end{align}
where the first map is given by Lemma \ref{a4le:multisimp} (2). We have a commutative diagram
\[
\xymatrix{
\Komp^\alpha(\tau)_L \ar[r]^-{\psi(\tau)}\ar[d]_{\phi(\tau)}&
\Map(\CCpt^n\times\Delta^{[n_k]_{k\in K}},\delta_{J,L}^*\delta^{J\square}_*(X,\cT'))\ar[d]\\
\Map(\square^n\times\Delta^{[n_k]_{k\in K}},\delta_{I,L}^*\delta^{I\square}_*(X,\cT)\ar[r] &
\Map(\square^n\times\Delta^{[n_k]_{k\in K}},\delta_{J,L}^*\delta^{J\square}_*(X,\cT')),
}
\]
where $\phi(\tau)$ is defined in \eqref{a4eq:compactification_phi}, the lower horizontal arrow is induced by \eqref{a4eq:composition}, and the right vertical arrow is the obvious restriction.
\end{construction}

\begin{remark}\label{a4re:psi}
By construction, the composition
\begin{align*}
\Komp^\alpha(\tau)_L &\xrightarrow{\psi(\tau)}\Map(\CCpt^n\times\Delta^{[n_k]_{k\in K}},\delta_{J,L}^*\delta^{J\square}_*(X,\cT')) \\
&\to \Map(\Delta^{[n,n_k]_{k\in K}},\delta_{J,L}^*\delta^{J\square}_*(X,\cT')),
\end{align*}
where the second map is induced by the diagonal embedding $\Delta^n\to\CCpt^n$, is constant of value $\delta_{J,L}^*\tau$. If $\Komp^\alpha(\tau)$ is nonempty, then $\tau$ factorizes through $\delta_*^{J\square}(X,\cT')$.
\end{remark}

\begin{theorem}[Multisimplicial descent]\label{a4th:multisimplicial_descent}
Let $K$ be a set and let $\alpha\in\{1,2\}\coprod K$ be an element. Let $(X,\cT)$ be a $(\{1,2\}\coprod K)$-tiled simplicial set and let $(X,\cT')$ be a $(\{0\}\coprod K)$-tiled simplicial set, satisfying Assumption \ref{a4as:composition}. We assume that $\Komp^\alpha_{(X,\cT)}(\tau)$ is weakly contractible for every $n\ge 0$ and every $(n,n_k)_{k\in K}$-simplex $\tau$ of $\delta^{(\{0\}\amalg K)\square}_*(X,\cT')$ with $n_k=n$. Then, for every subset $L\subseteq K$, the map
\[
f\colon \delta_{\{1,2\}\amalg K,L}^*\delta^{(\{1,2\}\amalg K)\square}_*(X,\cT)\to \delta_{\{0\}\amalg K,L}^*\delta^{(\{0\}\amalg K)\square}_*(X,\cT'),
\]
composition of \eqref{a4eq:composition}, is a categorical equivalence.
\end{theorem}

Note that the assumption that the simplicial sets $\Komp^\alpha(\tau)$ for those $\tau$ are nonempty implies that $\cT_k=\cT'_k$ for all $k\in K$, and $\cT_{kk'}=\cT'_{kk'}$ for all $k,k'\in K$ with $k\neq k'$.

In the case $K=\emptyset$ and $\cT'_0=X_1$, Assumption \ref{a4as:composition} is clearly satisfied and the theorem takes the following form.

\begin{corollary}\label{a4co:multi2}
Let $\alpha$ be $1$ or $2$. Let $(X,\cT)$ be a $2$-tiled simplicial set such that $\Komp^\alpha_{(X,\cT)}(\tau)$ is weakly contractible for every simplex $\tau$ of $X$. Then the map
\[
f\colon \delta^*_2\delta_*^{2\square}(X,\cT)\to X
\]
induced by the counit map $\delta^*_2\delta_*^2 X\to X$ is a categorical equivalence.
\end{corollary}

\begin{proof}[Proof of Theorem \ref{a4th:multisimplicial_descent}]
We let $Y$ and $Z$ denote the source and target, respectively, of the map $f$ in the statement of the theorem. Consider a commutative diagram
\[
\xymatrix{
Y\ar[r]^-v\ar[d]_f & \Fun(\Delta^l,\cD)\ar[d]^p\\
Z\ar[r]^-w & \Fun(\partial \Delta^l,\cD)
}
\]
as in Lemma \ref{a1le:categorical_equivalence}. Let $\sigma$ be an $n$-simplex of $Z$, corresponding to a map $\tau\colon \Delta^{n,n_k\res
k\in K}_L\to \delta^{J\square}_*(X,\cT')$, where $n_k=n$. Consider the commutative diagram
\begin{align}\label{a4eq:descent}
\resizebox{\hsize}{!}{
\xymatrix{
\cN(\sigma)\ar[d]\ar[r] & \Fun(\Delta^l\times\CCpt^n\times\Delta^{[n_k]_{k\in K}},\cD)
\ar[d]^-{\rres_1}\ar[rr]^-{\rres_2} && \Fun(\Delta^l\times \Delta^n,\cD)\ar[d]\ar@/^4.5pc/[dd]^{\rres_4}\\
\Komp^\alpha(\tau)_L\ar[r]^-{h}\ar[rd]_{v_*\phi(\tau)} &  \Fun(H\times\Delta^{[n_k]_{k\in
K}},\cD)\ar[r]\ar[d]&\Fun(\partial \Delta^l\times \CCpt^n\times\Delta^{[n_k]_{k\in
K}},\cD)\ar[r]^-{\rres_2} & \Fun(\partial\Delta^l\times\Delta^n,\cD)\\
&\Fun(\Delta^l\times \square^n\times\Delta^{[n_k]_{k\in
K}},\cD)\ar[rr]^-{\rres_3}&&\Fun(\Delta^l\times D^n,\cD).
}
}
\end{align}
In the above diagram,
\begin{itemize}
  \item $\rres_1$ is induced by
      \[
      j\colon H=\Delta^l\times\Box^n\coprod_{\partial \Delta^l\times\Box^n}
      \partial\Delta^l\times\CCpt^n \hookrightarrow \Delta^l \times \CCpt^n;
      \]

  \item $h$ is the amalgamation of $v_*\phi(\tau)$ and $w_*\psi(\tau)$, where
      \[
      v_*\phi(\tau)\colon\Komp^\alpha(\tau)_L\to\Fun(\Delta^l\times\Box^n\times\Delta^{[n_k]_{k\in K}},\cD)
      \]
      is the composition of \eqref{a4eq:compactification_phi} and the map induced by $v$, and
      \[
      w_*\psi(\tau)\colon\Komp^\alpha(\tau)_L\to\Fun(\partial\Delta^l\times\CCpt^n\times\Delta^{[n_k]_{k\in K}},\cD)
      \]
      is the composition of \eqref{a4eq:compactification_psi} and the map induced by $w$;

  \item $\cN(\sigma)$ is defined so that the upper left square is a pullback square;

  \item the two maps $\rres_2$ are both induced by the diagonal embedding $\Delta^n\subseteq \CCpt^n\times \Delta^{[n_k]_{k\in K}}$;

  \item $D^n$ is as in Remark \ref{a4re:fibration};

  \item $\rres_3$ is induced by the diagonal embedding $D^n\subseteq\square^n\times \Delta^{[n_k]_{k\in K}}$; and

  \item $\rres_4$ is induced by the inclusion $D^n\subseteq \Delta^n$;

  \item the unnamed arrows in the middle column and in the upper right square are obvious restrictions.
\end{itemize}

By Lemma \ref{a6le:cpt_inner} and \cite{HTT}*{Corollaries 2.3.2.4, 2.3.2.5}, the map $j\times\id_{\Delta^{[n_k]_{k\in K}}}$ is inner anodyne, and consequently $\rres_1$ is a trivial Kan fibration. Thus $\cN(\sigma)$ is weakly contractible.

We let $\Phi(\sigma)\colon\cN(\sigma)\to\Fun(\Delta^l\times \Delta^n,\cD)$ denote the composition of the upper horizontal arrows in \eqref{a4eq:descent}. We let $\sigma_0\colon D^n\to Z$ denote the restriction of $\sigma$. Since $f$ induces a bijection on vertices, $\sigma_0$ factorizes uniquely through a map $D^n\to Y$, which we still denote by $\sigma_0$. By Remark \ref{a4re:phi}, $\rres_3\circ v_*\phi(\tau)$ is constant of value $v(\sigma_0)$. It follows that $\rres_4\circ \Phi(\sigma)$ is constant of value $v(\sigma_0)$. In particular, $\Phi(\sigma)$ induces a map
\[
\cN(\sigma)^\sharp\times(\Delta^n)^\flat\to\Fun(\Delta^l,\cD)^\flat\subseteq\Fun(\Delta^l,\cD)^{\natural}.
\]
Thus $\Phi(\sigma)$ induces a map $\cN(\sigma)\to\Map^\sharp((\Delta^n)^\flat, \Fun(\Delta^l,\cD)^{\natural})$, which we still denote by $\Phi(\sigma)$. This construction is functorial in $\sigma$, giving rise to a morphism $\Phi\colon \cN\to \Map[Z,\Fun(\Delta^l,\cD)]$ in
the category $(\Sset)^{(\del_{/Z})^{op}}$.

By Remark \ref{a4re:psi}, the composition of the middle row of \eqref{a4eq:descent} is constant of value $w(\sigma)$. Thus $\Map[Z,p]\circ\Phi\colon\cN\to \Map[Z,\Fun(\partial\Delta^l,\cD)]$ factorizes through the morphism $\Delta^0_{(\del_{/Z})^{op}}\to \Map[Z,\Fun(\partial\Delta^l,\cD)]$ corresponding to $w$ via Remark \ref{a2re:functors}.

Now let $\sigma'$ be an $n$-simplex of $Y$ corresponding to a map $\tau'\colon \Delta_L^{n,n,n_k\mid k\in K}\to \delta^{I\square}_*(X,\cT)$.
By restricting to $\CCpt^n\subseteq \Delta^{[n,n]}$ we obtain a vertex of $\Komp^\alpha(\tau)$. By restricting the composition
\[
\Delta^{n,n,n_k\mid k\in K}\xrightarrow{\tau'} \op^I_L\delta_*^{I\square}(X,\cT)\xrightarrow{v} \delta_*^I\Fun(\Delta^l,\cD),
\]
we obtain a vertex of $\Fun(\Delta^l\times\RCpt^n\times\Delta^{[n_k]_{k\in K}},\cD)$. The two vertices have the same image in $\Fun(H\times\Delta^{[n_k]_{k\in K}},\cD)$ and hence provide a vertex $\nu(\sigma')$ of $\cN(f(\sigma'))$, whose image under $\Phi(f(\sigma'))$ is $v(\sigma')$. This construction is functorial in $\sigma'$, giving rise to $\nu\in\Gamma(f^*\cN)_0$ such that $f^*\Phi\circ \nu=v$. Applying Proposition \ref{a1pr:extension} to $\Phi$, the map $f\colon Y\to Z$ and the global section $\nu$ of $f^*\cN$, we obtain a map $u\colon Z\to\Fun(\Delta^l,\cD)$ satisfying $p\circ u=w$ such that $u\circ f$ and $v$ are homotopic over $\Fun(\partial \Delta^l,\cD)$, as desired.
\end{proof}

Next we show that in a favorable case, the weak contractibility condition in the theorem can be reduced to a weak contractibility condition on a $2$-marked simplicial set.

\begin{theorem}\label{a4co:multisimplicial_descent}
Let $\cC$ be an $\infty$-category and $K$ a \emph{finite} set. Consider a $(\{0,1,2\}\coprod K)$-marked $\infty$-category
$(\cC,\cE_0,\cE_1,\cE_2,\{\cE_k\}_{k\in K})$ such that
\begin{enumerate}
  \item $\cE_1,\cE_2\subseteq \cE_0$;

  \item $\cE_0$ is stable under composition;

  \item $\cE_1$, $\cE_2$ are stable under pullback by $\cE_k$ for all $k\in K$;

  \item $\cE_k$ is stable under pullback by $\cE_1$ for all $k\in K$; and

  \item edges in $\cE_k$ admit pullbacks in $\cC$ by edges in $\cE_1$ for all $k\in K$.
\end{enumerate}
Then for every $(n,n_k)_{k\in K}$-simplex $\tau$ of the $(\{0\}\coprod K)$-tiled $\infty$-category $\cC_{\cE_0,\{\cE_k\}_{k\in K}}^\cart$, the
restriction map $\Komp^\alpha_{(\cC,\cT)}(\tau)\to\Komp^\alpha_{\cC,\cE_1,\cE_2}(\gamma)$, where $\gamma$ is the restriction of $\tau$ to $\Delta^n\times \{(n_k)_{k\in K}\}$, is a trivial Kan fibration for every $\alpha\in \{1,2\}$. Here $(\cC,\cT)=(\cC,\cE_1,\cE_2,\{\cE_k\}_{k\in K},\cQ)$ is the $(\{1,2\}\coprod K)$-tiled $\infty$-category in which $\cQ$ is determined by the conditions
\[
\cQ_{12}=\cE_1*_{\cC}\cE_2,\quad \cQ_{ij}=\cE_i*_\cC^\cart\cE_j,\ (i,j)\neq (1,2),(2,1).
\]
Moreover, if for some $\alpha\in \{1,2\}$ and for every simplex $\gamma$ of $\cC_{\cE_0}\subseteq \cC$, the simplicial set
$\Komp^\alpha_{\cC,\cE_1,\cE_2}(\gamma)$ is weakly contractible, then, for every subset $L\subseteq K$, the map
\[
f\colon \delta^*_{\{1,2\}\amalg K,L}\delta_*^{(\{1,2\}\amalg K)\square}(\cC,\cT)\to
\delta^*_{\{0\}\amalg K,L} \cC_{\cE_0,\{\cE_k\}_{k\in K}}^\cart
\]
is a categorical equivalence.
\end{theorem}

\begin{proof}
By Lemma \ref{a3le:cart} and Condition (2), $\cE_0*^\cart_\cC \cE_k$, $k\in K$ are stable under composition in the first direction. Thus by Remark \ref{a4re:assump} and Conditions (1) and (2), Assumption \ref{a4as:composition} is satisfied for $(\cC,\cT)$ and $\cC_{\cE_0,\{\cE_k\}_{k\in K}}^\cart$. By Theorem \ref{a4th:multisimplicial_descent}, it suffices to show the first assertion. Indeed, the assumption that $\Komp^\alpha_{\cC,\cE_1,\cE_2}(\gamma)$ is weakly contractible then implies that $\Komp^\alpha_{(\cC,\cT)}(\tau)$ is weakly contractible.

We let $\infty=(n_k)_{k\in K}$ denote the final object of $[n_k]_{k\in K}\coloneqq \prod_{k\in K} [n_k]$. We have the following commutative
diagram
\begin{equation}\label{a4eq:contractible}
\xymatrix{
\Komp^\alpha_{(\cC,\cT)}(\tau)\ar[rr]\ar[d]&& \Komp^\alpha_{\cC,\cE_1,\cE_2}(\gamma)\ar[d]\\
\Fun(\CCpt^n\times\Delta^{[n_k]_{k\in K}},\cC)_\RKE\ar[r]^-{\rres_1}&
\cK'\ar[d]
\ar[r]&\cK\ar[d]\\
&\Fun(\rN(Q\cup R),\cC)\ar[r]^-{\rres_2}& \Fun(\rN(Q)\cup\rN(R),\cC),
}
\end{equation}
where
\begin{itemize}
  \item $Q=[n]\times[n_k]_{k\in K}\subseteq \RCpt^n\times[n_k]_{k\in K}$ is induced by the diagonal inclusion $[n]\subseteq \RCpt^n$.

  \item $R=\RCpt^n\times\{\infty\}\subseteq \RCpt^n\times[n_k]_{k\in K}$.

  \item $\Fun(\CCpt^n\times\Delta^{[n_k]_{k\in K}},\cC)_\RKE\subseteq\Fun(\CCpt^n\times\Delta^{[n_k]_{k\in K}},\cC)$ is the full
      subcategory spanned by functors $F\colon\CCpt^n\times\Delta^{[n_k]_{k\in K}}\to\cC$ which are right Kan extensions of $F\res\rN(Q\cup R)$.

  \item $\cK'\subseteq \Fun(\rN(Q\cup R),\cC)$ is the full subcategory spanned by functors $F$ such that the composition $\rN(Q\cup
      R)_{(i,j,p)/}\to \rN(Q\cup R)\xrightarrow{F} \cC$ admits a limit for every vertex $(i,j,p)$ of $\CCpt^n\times \Delta^{[n_k]_{k\in K}}$.

  \item $\cK\subseteq \Fun(\rN(Q)\cup \rN(R),\cC)$ is the full subcategory spanned by functors $F$ such that the diagram
      \begin{equation}\label{a4eq:pullback}
      \xymatrix{& F(i,j,\infty)\ar[d] \\F(j,j,p)\ar[r]&F(j,j,\infty).}
      \end{equation}
      admits a limit in $\cC$ for every vertex $(i,j,p)$ of $\CCpt^n\times\Delta^{[n_k]_{k\in K}}$.

  \item The horizontal arrows are restrictions. We will check that $\rres_2$ carries $\cK'$ into $\cK$ below.

  \item The lower vertical arrows are inclusions.

  \item The upper right vertical arrow is the amalgamation of the inclusion $\Komp^\alpha_{\cC,\cE_1,\cE_2}(\gamma)\subseteq\Fun(\rN(R),\cC)$ with $\tau$, viewed as a vertex of $\Fun(\rN(Q),\cC)$. The fact that the image is in $\cK$ follows from Conditions (3) and (5) (Condition (3) is needed if $\# K\ge 2$).

  \item The left vertical arrow is induced by the inclusion
      \[
      \Komp^\alpha_{(\cC,\cT)}(\tau)\subseteq\Fun(\CCpt^n\times\Delta^{[n_k]_{k\in K}},\cC).
      \]
      We will check that the image is contained in $\Fun(\CCpt^n\times\Delta^{[n_k]_{k\in K}},\cC)_\RKE$ below.
\end{itemize}
For any vertex $(i,j,p)$ of $\CCpt^n\times \Delta^{[n_k]_{k\in K}}$, we let $g\colon \Delta^1\times \Delta^1\to \CCpt^n\times\Delta^{[n_k]_{k\in K}}$ denote the square
\[
\xymatrix{
(i,j,p)\ar[r]\ar[d]& (i,j,\infty)\ar[d] \\(j,j,p)\ar[r]&(j,j,\infty).
}
\]
We have $\Delta^1\times \Delta^1\simeq ((\Lambda^2_0)^{op})^\triangleleft$. The induced map $\Lambda^2_0\to (\rN(Q\cup R)_{(i,j,p)/})^{op}$ is cofinal by Lemma \ref{a3le:cofinal} below. Thus a functor $G\colon\CCpt^n\times\Delta^{[n_k]_{k\in K}}\to\cC$ is a right Kan extension of
$G\res\rN(Q\cup R)$ if and only if $G\circ g$ is a pullback square, for all $(i,j,p)$. For any vertex $G$ of $\Komp^\alpha_{(\cC,\cT)}(\tau)$, regarded as a functor $G\colon \CCpt^n\times\Delta^{[n_k]_{k\in K}}\to\cC$, the square $G\circ g$ is obtained by a finite sequence of compositions from squares in $\cT_{1k}=\cE_{1}*^\cart_\cC\cE_k$, $k\in K$. Therefore, the image of $\Komp^\alpha_{(\cC,\cT)}(\tau)\subseteq\Fun(\CCpt^n\times\Delta^{[n_k]_{k\in K}},\cC)$ is contained in
$\Fun(\CCpt^n\times\Delta^{[n_k]_{k\in K}},\cC)_\RKE$. Moreover, if $F\colon\rN(Q\cup R)\to \cC$ is a functor, then the composition $\rN(Q\cup
R)_{(i,j,p)/}\to \rN(Q\cup R)\xrightarrow{F} \cC$ admits a limit if and only if the diagram \eqref{a4eq:pullback} admits a limit. Thus $\rres_2$ carries $\cK'$ into $\cK$ and the lower right square in a pullback.

By \cite{HTT}*{Proposition 4.3.2.15}, $\rres_1$ is a trivial Kan fibration. We apply Lemma \ref{a6le:union} to show that the inclusion $\rN(Q)\cup\rN(R)\subseteq \rN(Q\cup R)$ is inner anodyne. For this we need to check that $Q\cup R=Q\coprod_{Q\cap R} R$ is a pushout in the category of partially ordered sets (see Remark \ref{a6re:order}). Let $(i,i,p)$ be in $Q$ and $(i',j',\infty)$ in $R$. If we have $(i',j',\infty)\le (i,i,p)$, then $p=\infty$ so that $(i,i,p)$ is in $Q\cap R$. On the other hand, if we have $(i,i,p)\le (i',j',\infty)$, then we have $(i,i,p)\le (i',i',\infty)\le(i',j',\infty)$. It follows that $\rres_2$ is a trivial Kan fibration.

To show that the upper horizontal arrow is a trivial Kan fibration, it remains to show that, ignoring the middle term in the second row, the upper square of \eqref{a4eq:contractible} is also a pullback square. This amounts to saying that for every $m$-simplex $\sigma$ of
$\Fun(\CCpt^n\times\Delta^{[n_k]_{k\in K}},\cC)_\RKE$, if the restriction of $\sigma$ to $\rN(Q)$ is $\tau$ and the restriction of $\sigma$ to $\rN(R)$ is in $\Komp^\alpha_{\cC,\cE_1,\cE_2}(\gamma)$, then $\sigma$ is a simplex of $\Komp^\alpha_{(\cC,\cT)}(\tau)$. By Remark \ref{a4re:g}, it suffices to treat the cases $m=0$ and $m=1$.

Case $m=0$. Consider integers $0\le i\le i'\le j\le j'\le n$ and a morphism $p\le q$ of $[n_k]_{k\in K}$. Since $\sigma\colon\CCpt^n\times\Delta^{[n_k]_{k\in K}}\to\cC$ is a right Kan extension of $\sigma \res \rN(Q\cup R)$, it carries the outer and right squares of the diagram
\[
\xymatrix{
(i,j,p)\ar[r]\ar[d]&(i,j,q)\ar[r]\ar[d] & (i,j,\infty)\ar[d]\\
(j,j,p)\ar[r] &(j,j,q)\ar[r] & (j,j,\infty)
}
\]
to pullback squares. It follows that $\sigma$ carries the left square to a pullback square. Thus, since the restriction of $\sigma$ to $\rN(Q)$ is $\tau$, $\sigma$ carries $\cF_k$ to $\cE_k$ for all $k\in K$ by Condition (4), where $\cF_k$ is defined in Remark \ref{a4re:g}. Moreover, since $\sigma$ carries the outer and lower squares of the diagram
\[
\xymatrix{
(i,j,p)\ar[r]\ar[d] & (i,j,q)\ar[d]\\
(i',j,p)\ar[r]\ar[d] & (i',j,q)\ar[d]\\
(j,j,p)\ar[r] & (j,j,q)
}
\]
to pullbacks, it carries the upper square to a pullback. Taking $q=\infty$, Condition (3) then implies that $\sigma$ carries $(i,j,p)\to (i',j,p)$ to a morphism in $\cE_1$. It follows that $\sigma$ carries $\cF_1*\cF_k$ into $\cE_1*^\cart\cE_k$ for every $k\in K$. Consider the cube
\[
\xymatrix{
(i,j,p)\ar[rr]\ar[rd]\ar[dd] && (i,j,q)\ar[rd]\ar@{-->}[dd]\\
&(i,j',p)\ar[rr]\ar[dd] &&(i,j',q)\ar[dd]\\
(j,j,p)\ar@{-->}[rr]\ar[rd] && (j,j,q)\ar@{-->}[rd]\\
&(j',j',p)\ar[rr] &&(j',j',q).
}
\]
The image of the bottom square under $\sigma$ can be obtained by a finite sequence of compositions from squares in $\cE_0*_\cC^\cart \cE_k$, $k\in K$. Since $\sigma$ carries the front and back squares to pullbacks as well, $\sigma$ carries the top square to pullback. Taking $q=\infty$, Condition (3) then implies that $\sigma$ carries $(i,j,p)\to (i,j',p)$ to a morphism in $\cE_2$. It follows that $\sigma$ carries $\cF_2*\cF_k$ into $\cE_2*^\cart \cE_k$ for every $k\in K$. Finally, given a square $S$ in $\cF_k*\cF_{l}$ for distinct $k,l\in K$, let $(i,j)$ be its projection in $\CCpt^n$ and $T$ its projection in $\Delta^{[n_k]_{k\in K}}$. Then $S$ can be identified with the top face of a cube, product of the edge $(i,j)\to(j,j)$ and the square $T$. Since $\sigma$ carries the other five faces of the cube to pullback squares, it carries $S$ to a pullback as well.

Case $m=1$. We check Condition (2) in Remark \ref{a4re:g}. For $0\le i\le j\le n$ and $p\le q$ in $[n_k]_{k\in K}$, consider the following cube in $\Delta^1\times\CCpt^n\times\Delta^{[n_k]_{k\in K}}$:
\[
\xymatrix{
(0,i,j,p)\ar[rd]\ar[dd]\ar[rr]&&(0,i,j,q)\ar[rd]\ar@{-->}[dd]\\
&(0,j,j,p)\ar[rr]\ar[dd]&&(0,j,j,q)\ar[dd]\\
(1,i,j,p)\ar@{-->}[rr]\ar[rd]&& (1,i,j,q)\ar@{-->}[rd]\\
&(1,j,j,p)\ar[rr]&&(1,j,j,q).
}
\]
Since $\sigma\colon \Delta^1\times\CCpt^n\times\Delta^{[n_k]_{k\in K}}\to\cC$ carries the top and bottom squares to pullbacks and carries the
front square to the identity on $\tau(j,p)\to \tau(j,q)$, it carries the back square to a pullback. Taking $q=\infty$, Condition (3) then implies that $\sigma$ carries $(0,i,j,p)\to (1,i,j,p)$ to a morphism in $\cE_\alpha$.
\end{proof}

\begin{lem}\label{a3le:cofinal}
Let $P$ be a partially ordered set and $f\colon \Lambda^2_0\to \rN(P)$ a map. Assume that $f(0)$ is the product (namely, supremum) of $f(1)$ and $f(2)$ in $P$, and $P_{/f(1)}\cup P_{/f(2)}=P$. Then $f$ is cofinal \cite{HTT}*{Definition 4.1.1.1}.
\end{lem}

\begin{proof}
By \cite{HTT}*{Theorem 4.1.3.1}, it suffices to show that for every $p\in P$, the simplicial set $S=\Lambda^2_0\times_{\rN(P)}\rN(P_{p/})$ is weakly contractible. By the second assumption, either $p\le f(1)$ or $p\le f(2)$. If exactly one of the two inequalities holds, then $S$ is a point. If both inequalities hold, then $p\le f(0)$ by the first assumption, and hence $S=\Lambda^2_0$.
\end{proof}

\begin{remark}
In Theorem \ref{a4co:multisimplicial_descent}, the Cartesian restriction on $\cE_k*_\cC\cE_l$ for $k,l\in K$ is not essential. To be more precise, under the assumptions of the theorem, consider an $I$-tiling $\cT=((\cE_{i})_{i\in I},(\cQ_{ij})_{i,j\in I,i\neq j})$ and a $J$-tiling $\cT'=((\cE_i)_{i\in J},(\cQ_{ij})_{i,j\in J,i\neq j})$ such that $\cQ_{12}=\cE_1*_\cC \cE_2$, $\cQ_{ik}=\cE_i*^\cart_\cC \cE_k$ for $i=0,1,2$ and $k\in K$, and $\cQ_{kl}\subseteq \cE_k*_\cC \cE_l$ is stable under pullback by $\cQ_{1l}=\cE_1*_\cC^\cart \cE_l$ in the first direction or stable under pullback by $\cQ_{k1}=\cE_k*_\cC^\cart\cE_1$ in the second direction for $k,l\in K$, $k\neq l$. Then the proof shows that the restriction map $\Komp^\alpha_{(\cC,\cT)}(\tau)\to \Komp^\alpha_{\cC,\cE_1,\cE_2}(\gamma)$ is a trivial Kan fibration for every $(n,n_k)_{k\in K}$-simplex $\tau$ of $\delta^{(\{0\}\amalg K)\square}_*(X,\cT')$, and if $\Komp^\alpha_{\cC,\cE_1,\cE_2}(\gamma)$ is weakly contractible for every $\gamma$, then
\[
f\colon \delta_{\{1,2\}\amalg K,L}^*\delta^{(\{1,2\}\amalg K)\square}_*(X,\cT)
\to \delta_{\{0\}\amalg K,L}^*\delta^{(\{0\}\amalg K)\square}_*(X,\cT')
\]
is a categorical equivalence.
\end{remark}

As promised, we give sufficient conditions for the simplicial set $\Komp^\alpha(\tau)$ to be an $\infty$-category.

\begin{lem}\label{a4le:infinity_category}
In the situation of Definition \ref{a4de:compactification},
\begin{enumerate}
  \item Assume that $\cT_{\alpha \beta}$ is stable under composition in the first direction (Definition \ref{a3de:squares}) for all
      $\beta\in \{1,2\}\coprod K$, $\beta\neq \alpha$. Then the map $g$ is an inner fibration. Moreover, if $X$ is an $\infty$-category,
      then $\Komp^\alpha_{(X,\cT)}(\tau)$ is an $\infty$-category.

  \item If we have $(X,\cT)=\sfW(X,\cE)$ and $\cE_\alpha$ is composable (Definition \ref{a3de:edges}) and $X$ is an $\infty$-category, then
      $\Komp^\alpha_{X,\cE}(\tau)=\Komp^\alpha_{(X,\cT)}(\tau)$ is an $\infty$-category.
\end{enumerate}
\end{lem}

The assumption in (1) implies that $\cT_\alpha$ is stable under composition. The assumption in (1) is satisfied if we have $(X,\cT)=\sfW(X,\cE)$ and $\cE_\alpha$ is stable under composition.

\begin{proof}
By Remark \ref{a4re:g}, $g$ satisfies the right lifting property with respect to every horn inclusion $\Lambda^m_i\subseteq \Lambda^m$ for $m\ge 3$. Thus, for the first assertion of (1), it suffices to show that $g$ satisfies the right lifting property with respect to $\Lambda^2_1\subseteq \Delta^2$. We use the notation of Remark \ref{a4re:g}. Let $\gamma$ be a $2$-simplex of $\Map(\square^n\times\Delta^{[n_k]_{k\in K}},X)$ such that the restriction of $\gamma$ to $\Lambda^2_1$ factorizes through $Y$. We regard $\gamma$ as a map $\Delta^2\times (\square^n\times \Delta^{[n_k]_{k\in K}})\to X$. For any square in $\cF_\alpha*\cF_\beta$ of the form \eqref{a4eq:Komp}, consider the map $\Delta^2\times \Delta^1\to \Delta^2\times (\square^n\times\Delta^{[n_k]_{k\in K}})$ as shown by the diagram
\[
\xymatrix{
(0,y')\ar[r]\ar[d] & (0,y)\ar[d]\\
(1,y')\ar[r]\ar[d] & (1,y)\ar[d]\\
(2,x')\ar[r] & (2,x).}
\]
By assumption, $\gamma$ carries the upper and lower squares to squares in $\cT_{\alpha\beta}$. (We could replace the second row by $(1,x')\to(1,x)$ without affecting the validity of the argument.) Since $\cT_{\alpha\beta}$ is stable under composition in the first direction, $\gamma$ carries the outer square to a square in $\cT_{\alpha\beta}$. Therefore, the restriction of $\gamma$ to $\Delta^{\{0,2\}}$ is an edge of $Y$.

The second assertion of (1) follows from the first assertion of (1) and the fact that $\rres_2$ is an inner fibration if $X$ is an $\infty$-category (Remark \ref{a4re:fibration}).

For (2), note that by Remark \ref{a4re:fibration}, we have a diagram with pullback square
\[
\xymatrix{
\Komp^\alpha(\tau)\ar[r]
&Z\ar[r]\ar[d] & Y\ar[d]^-{\rres_3\circ g}\\
&\{\tau\}\ar[r] & \Map(D^n\times\Delta^{[n_k]_{k\in K}}_L,X),}
\]
where $Z$ denotes the fiber of the map $\rres_3\circ g$ at $\tau$, and the map $\Komp^\alpha(\tau)\to Z$ is a pullback of the map $\rres_4$ in Remark \ref{a4re:fibration}, hence an inner fibration. Thus it suffices to show that $Z$ is an $\infty$-category. Since $\rres_3$ is an inner fibration and $g$ satisfies the right lifting property with respect to every horn inclusion $\Lambda^m_i\subseteq \Lambda^m$ for $m\ge 3$, it suffices to check that $Z$ satisfies the extension property with respect to $\Lambda^2_1\subseteq\Delta^2$. Let $f\colon \Lambda^2_1\to Z$ be a map. Unwinding the definition, to show that $f$ extends to a map $\Delta^2\to Z$, we are reduced to showing the extension property
\[
\xymatrix{
(A,A_1\cap \cG)\ar@{^{(}->}[d]\ar[r]^-{f'} & (X,\cE_\alpha),\\
(B,\cG)\ar@{..>}[ru]
}
\]
where we have $(B,\cG)=(\Delta^2)^\sharp\times (\square^n\times\Delta^{[n_k]_{k\in K}}_L,\cF_\alpha)$ and
\[
A=\Lambda^2_1\times(\square^n\times \Delta^{[n_k]_{k\in K}}_L)\coprod_{\Lambda^2_1\times
(D^n\times \Delta^{[n_k]_{k\in K}}_L)} \Delta^2\times (D^n\times\Delta^{[n_k]_{k\in K}}_L),
\]
and $f'$ is the amalgamation of $f$ and $\tau$. Every edge in $\cG$ that is not in $A$ has the form $(0,y)\to (2,x)$ with $y\to x$ in $\cF_\alpha$, and can be extended to a $2$-simplex of $B$
\[
\xymatrix{
&(1,y)\ar[rd] \\ (0,y)\ar[ru]\ar[rr] &&(2,x),
}
\]
where the oblique edges are in $A_1\cap \cG$. (Again we could replace $(1,y)$ by $(1,x)$.) Therefore, it suffices to apply Lemma \ref{a3le:comp}.
\end{proof}

We now give a criterion for the weak contractibility of certain $\infty$-categories of compactifications.

\begin{theorem}\label{a4pr:descent}
Let $(\cC,\cE_1,\cE_2)$ be a $2$-marked $\infty$-category. Suppose that the following conditions are satisfied:
\begin{enumerate}
  \item $\cE_1$ and $\cE_2$ are composable (Definition \ref{a3de:edges}).

  \item The $\infty$-category $\cC_{\cE_1}$ admits pullbacks and pullbacks are preserved by the inclusion $\cC_{\cE_1}\subseteq\cC$.

  \item For every morphism $f$ of $\cC$, there exists a $2$-simplex of $\cC$ of the form
      \begin{equation}\label{a4eq:2cell}
      \xymatrix{&y\ar[rd]^p\\z\ar[ru]^q\ar[rr]^f && x}
      \end{equation}
      with $p\in\cE_1$ and $q\in\cE_2$.
\end{enumerate}
Then, for every $n$-simplex $\tau$ of $\cC$, the simplicial set $\Komp^1_{\cC,\cE_1,\cE_2}(\tau)^{op}$ is a filtered $\infty$-category and
is weakly contractible. Moreover, the natural map
\[
\delta_2^*\cC_{\cE_1,\cE_2} \to \cC
\]
is a categorical equivalence.
\end{theorem}

Recall that an $\infty$-category is said to be \emph{filtered} \cite{HTT}*{Definition 5.3.1.7} if it satisfies the extension property with respect to the inclusion $A\subseteq A^\triangleright$ for every finite simplicial set $A$. Recall also that an ordinary category is filtered if and only if its nerve is a filtered $\infty$-category \cite{HTT}*{Proposition 5.3.1.13}. Thus in the case where $\cC$ is the nerve of an ordinary category, the first assertion of Theorem \ref{a4pr:descent} generalizes \cite{SGA4}*{Expos\'{e} xvii, Proposition 3.2.6}.

\begin{remark}
Condition (2) of Theorem \ref{a4pr:descent} is satisfied if the following conditions are satisfied:
\begin{enumerate}[(a)]
  \item morphisms in $\cE_1$ admit pullbacks in $\cC$ by morphisms in $\cE_1$;

  \item $\cE_1$ is stable under pullback by $\cE_1$;

  \item for every $2$-simplex of $\cC$ of the form \eqref{a4eq:2cell} such that $f$ and $p$ are in $\cE_1$, $q$ is in $\cE_1$.
\end{enumerate}
Indeed, Condition (c) implies that for every diagram $a\colon A\to \cC$, where $A$ is a nonempty simplicial set, the overcategory $(\cC_{\cE_1})_{/a}$ is a full subcategory of $\cC_{/a}$, so that a diagram $\bar a\colon A^\triangleleft\to \cC_{\cE_1}$ is a limit diagram if the composition $A^\triangleleft \xrightarrow{\bar a}\cC_{\cE_1}\to \cC$ is a limit diagram. Note that Conditions (b) and (c) hold if $\cE_1$ is admissible (Definition \ref{a3de:admissible_edge}).
\end{remark}

\begin{proof}[Proof of Theorem \ref{a4pr:descent}]
For brevity we write $\Komp^1(\tau)$ for $\Komp^1_{\cC,\cE_1,\cE_2}(\tau)$. Since $\cE_1$ is composable, $\Komp^1(\tau)$ is an $\infty$-category by Lemma \ref{a4le:infinity_category}. It suffices to show that $\Komp^1(\tau)^{op}$ is filtered. In fact, every filtered $\infty$-category is weakly contractible \cite{HTT}*{Lemma 5.3.1.18}. The last assertion of the proposition then follows from Corollary \ref{a4co:multi2}.

By \cite{HTT}*{Remark 5.3.1.10}, $\Komp^1(\tau)^{op}$ is filtered if and only if $\Komp^1(\tau)$ has the extension property with respect to the
inclusion $A\subseteq A^\triangleleft$ whenever $A$ is the nerve of a finite partially ordered set. We fix such an $A$ and proceed by induction on $n$. For $n=0$, $\Komp^1(\tau)$ is a point and the assertion holds trivially.

For $n\ge 1$, by the induction hypothesis, the composite map $f_{-1}\colon A\xrightarrow{f}\Komp^1(\tau)\to \Komp^1(\tau\circ d^n_n)$ extends to $g_{-1}\colon A^\triangleleft\to \Komp^1(\tau\circ d^n_n)$. We identify $\CCpt^{n-1}$ with its image under $d^n_n$, hence with the full subcategory of $\CCpt^n$ spanned by the objects $(i,j)$, $1\le i\le j\le n-1$. For $0\leq k\leq n$, consider the full subcategory $\CCpt^n_k$ of $\CCpt^n$ spanned by $\CCpt^{n-1}$ and the objects $(i,n)$ with $n-k\leq i\leq n$. We have $\CCpt^{n-1}\subseteq\CCpt^n_0\subseteq\dots\subseteq\CCpt^n_n=\CCpt^n$. Similarly we define $\Cpt^n_k\subseteq \Cpt^n$. Define $\Komp^1_k(\tau)$ similarly to $\Komp^1(\tau)$ but with $\CCpt^n$, $\Cpt^n$ and $\Box^n=\delta_2^*\Cpt^n$ replaced by $\CCpt^n_k$, $\Cpt^n_k$ and $\delta_2^*\Cpt^n_k$, respectively. We show by induction on $k$ that there exists a map $g_k\colon A^\triangleleft \to \Komp^1_k(\tau)$ compatible with $f_k$ and $g_{k-1}$, where $f_k$ is the composition of $f$ and the natural map $\Komp^1(\tau)\to\Komp^1_k(\tau)$, rendering the following diagram commutative:
\[
\resizebox{\hsize}{!}{
\xymatrix{
A \ar[dd] \ar[rd]^(.6){f=f_n} \ar@/^0.6pc/[rrrd]_-{f_k}\ar@/^0.6pc/[rrrrd]^(.7){f_{k-1}}\ar@/^1pc/[drrrrrr]^(.7){f_{-1}} \\
& \Komp^1(\tau) \ar[r] & \cdots \ar[r] &\Komp^1_k(\tau) \ar[r] & \Komp^1_{k-1}(\tau)\ar[r] &\cdots \ar[r] & \Komp^1(\tau\circ d^n_n). \\
A^\triangleleft \ar@/_1pc/[urrrrrr]_(.7){g_{-1}}
\ar@{..>}@/_0.6pc/[urrr]^-{g_k}\ar@{..>}@/_0.6pc/[urrrr]_(.7){g_{k-1}}  \ar@{..>}[ur]_(.6){g_n}
}
}
\]
The map $g_n$ will allow us to conclude the proof of the proposition.

Below are the Hasse diagrams of (the homotopy categories of) $\CCpt^3_0$ and $\CCpt^3_2$, respectively. Bullets in the first diagram represent vertices in the image of the diagonal embedding $\Delta^3\subseteq \CCpt^3_0$. Bullets in the second diagram represent vertices in the image of the embedding $\Delta^2\to \CCpt^3_2$ defined later in the proof.
\[
\begin{xy}
(0,5)="00"; (5,5)*\cir<1.8pt>{}="01"**\dir{-};
(10,5)*\cir<1.8pt>{}="02"**\dir{-};
(5,0)="11"; (10,0)*\cir<1.8pt>{}="12"**\dir{-};
(10,-5)="22";
(15,-10)="33"**\dir{-};
"01"; "11"**\dir{-};
"02"; "12"**\dir{-}; "22"**\dir{-};
"00"*{\bullet}; "11"*{\bullet}; "22"*{\bullet}; "33"*{\bullet};
\end{xy}\qquad\qquad\begin{xy}
(0,5)*\cir<1.8pt>{}="00"; (5,5)*\cir<1.8pt>{}="01"**\dir{-};
(10,5)*\cir<1.8pt>{}="02"**\dir{-};
(5,0)*\cir<1.8pt>{}="11"; (10,0)="12"**\dir{-};
(15,0)="13"**\dir{-};
(10,-5)*\cir<1.8pt>{}="22";
(15,-5)="23"**\dir{-};
(15,-10)*\cir<1.8pt>{}="33"**\dir{-};
"01"; "11"**\dir{-};
"02"; "12"**\dir{-}; "22"**\dir{-};
"13"; "23"**\dir{-};
"12"*{\bullet}; "13"*{\bullet}; "23"*{\bullet};
\end{xy}
\]

We first consider the case $k=0$. The map $f_0$ (resp.\ $g_{-1}$) corresponds to a map $\tilde{f}_0\colon A\times\CCpt^n_0\to\cC$ (resp.\
$\tilde{g}_{-1}\colon A^\triangleleft\times\CCpt^{n-1}\to\cC$). To find the desired map $g_0$, it suffices to construct a map $\tilde{g}_0\colon A^\triangleleft\times\CCpt^n_0\to\cC$, extending $\tilde{f}_0$ and $\tilde{g}_{-1}$ and the composition $A^\triangleleft\times\Delta^n\to\Delta^n\xrightarrow{\tau}\cC$, where the first map is the projection, via the diagonal embedding $\Delta^n\subseteq\CCpt^n_0$. This follows if $\cC$ has the extension property with respect to the smash product of $A\subseteq A^\triangleleft$ and $\CCpt^{n-1}\coprod_{\Delta^{n-1}}\Delta^n\subseteq\CCpt^n_0$. However, the latter inclusion is inner anodyne by Lemma \ref{a6le:union} applied to $Q=[n]^{op}$ and $R=(\RCpt^{n-1})^{op}$. Thus we may find the map $\tilde{g}_0$ by \cite{HTT}*{Corollary 2.3.2.4} as $\cC$ is an
$\infty$-category.

For $1\le k\le n$, consider the full subcategory $\Delta^2\subseteq\CCpt^n_k$ spanned by $\{(n-k,n-1),(n-k,n),(n-k+1,n)\}$. We identify
$\Delta^{\{0,2\}}$ with the subcategory of $\CCpt^n_{k-1}$ spanned by $\{(n-k,n-1),(n-k+1,n)\}$. The inclusion
$\CCpt^n_{k-1}\coprod_{\Delta^{\{0,2\}}}\Delta^2\subseteq \CCpt^n_k$ is inner anodyne by Lemma \ref{a6le:cut}, and so is its smash product
\[
S\coloneqq \(A^\triangleleft\times\(\CCpt^n_{k-1}\coprod_{\Delta^{\{0,2\}}}\Delta^2\)\)\cup(A\times \CCpt^n_k)\subseteq A^\triangleleft \times \CCpt^n_k
\]
with $A\subseteq A^\triangleleft$. We define $\cG_1$ and $\cG_2$ by
\[
(A^\triangleleft\times \delta_2^*\Cpt^n_k,\cG_1,\cG_2)\simeq(A^\triangleleft)^{\sharp^2_{\{1\}}} \times \delta_{2+}^{*}\Cpt^n_k.
\]
We let $-\infty$ denote the cone point of $A^\triangleleft$. Any edge in $\cG_1$ but not in $S$ has the form $(-\infty,n-k,n)\to (l,i,n)$ with $l$ in $A^\triangleleft$ and $i>n-k+1$, and can be extended to a $2$-simplex
\[
\xymatrix{
&(l,n-k+1,n)\ar[rd] \\ (-\infty,n-k,n)\ar[ru]\ar[rr] &&(l,i,n)
}
\]
with oblique edges in $S_1\cap \cG_1$. Any edge in $\cG_2$ but not in $S$ has the form $(-\infty,n-k,j)\to (-\infty,n-k,n)$ with $j<n-1$ and can be extended to a $2$-simplex
\[
\xymatrix{
&(-\infty,n-k,n-1)\ar[rd] \\ (-\infty,n-k,j)\ar[ru]\ar[rr] &&(-\infty,n-k,n)
}
\]
with oblique edges in $S_1\cap \cG_2$. Thus, by Condition (1) and Lemma \ref{a3le:comp}, it suffices to construct a map $(S,S_1\cap\cG_1,S_1\cap\cG_2)\to (\cC,\cE_1,\cE_2)$ extending the amalgamation $v\colon V\coloneqq A\times\CCpt^n_k\coprod_{A\times\CCpt^n_{k-1}}A^\triangleleft\times \CCpt^n_{k-1}\to \cC$ of $\tilde f_k$ and $\tilde g_k$, where $\tilde{f}_k\colon A\times\CCpt^n_k\to\cC$ (resp.\ $\tilde{g}_{k-1}\colon A^\triangleleft\times\CCpt^{n}_{k-1}\to\cC$) is the map given by $f_k$ (resp.\ $g_{k-1}$). For this, it suffices to construct a map $(A^\triangleleft)^{\sharp^2_{\{1\}}}\times T\to(\cC,\cE_1,\cE_2)$ extending the amalgamation of $\tilde{f}_k\res A\times\Delta^2$ and $\tilde{g}_{k-1}\res A^\triangleleft\times\Delta^{\{0,2\}}$. Here $T=(\Delta^2,\cF_1,\cF_2)$ is the $2$-marked simplicial set with $\cF_1$ (resp.\ $\cF_2$) consisting of the degenerate edges and the edge $1\to 2$ (resp.\ $0\to 1$).

We now lift $v$ to a map $V\to \cC_{/\tau(n)}$, corresponding to a map $(V^\triangleright,\cG'_1,\cG'_2)\to (\cC,\cE_1,\cE_2)$, where $\cG'_1$ is the union of $(V_1\cap \cG_1)\cup \{\id_{+\infty}\}$ and all edges $(l,i,n)\to +\infty$ in $V^\triangleright$ for $l\in A^\triangleleft$, and $\cG'_2\coloneqq (V_1\cap \cG_2)\cup \{\id_{+\infty}\}$. Here $+\infty$ denotes the cone point of $V^\triangleright$. Consider the inclusion
$\iota\colon A^\triangleleft \to V$ induced by the inclusion $\{(n,n)\}\subseteq \CCpt^n_{k-1}$. Since the restriction of $v$ to $A^\triangleleft$ is constant of value $\tau(n)$, the amalgamation of $v$ and the constant map $A^{\triangleleft\triangleright}\to \cC$ of value $\tau(n)$ provides a map $v'\colon(C^\triangleright(\iota),\cG''_1,\cG'_2)\to (\cC,\cE_1,\cE_2)$, where we have $C^\triangleright(\iota)\coloneqq V\coprod_{A^\triangleleft}A^{\triangleleft\triangleright}$, and $\cG''_1$ is the intersection of $\cG'_1$ and the set of edges of $C^\triangleright(\iota)$. Since the inclusions $\{(n,n)\}\subseteq \CCpt^n_{k-1}$ and $\{(n,n)\}\subseteq\CCpt^n_{k}$ are right anodyne by \cite{HTT}*{Lemma 4.2.3.6}, and so are their products with identity maps \cite{HTT}*{Corollary 2.1.2.7}, the inclusion $A^\triangleleft=A\coprod_A A^\triangleleft\subseteq V$ is right anodyne by Lemma \ref{a2le:pushout_anodyne}. By \cite{HTT}*{Lemma 2.1.2.3},
it follows that the inclusion $C^\triangleright(\iota)\subseteq V^\triangleright$ is inner anodyne. Every edge in $\cG'_1$ that is not in
$\cG''_1$ has the form $(l,i,n)\to +\infty$ and can be extended to a $2$-simplex
\[
\xymatrix{
&(l,n,n)\ar[rd]\\(l,i,n)\ar[rr]\ar[ru]&&+\infty
}
\]
with oblique edges in $\cG''_1$. Lemma \ref{a3le:comp} then provides the desired extension of $v'$ and hence $v$.

We are therefore reduced to showing that every map
\[
a\colon A^{\sharp^2_{\{1\}}}\times T \coprod_{A^{\sharp^2_{\{1\}}}\times (\Delta^{\{0,2\}})^{\flat^2}}
(A^\triangleleft)^{\sharp^2_{\{1\}}}\times (\Delta^{\{0,2\}})^{\flat^2}\to (\cC_{/x},\cE'_1,\cE'_2)
\]
whose restriction to $A\times\Delta^{\{1,2\}}\coprod_{A\times\Delta^{\{2\}}} A^\triangleleft \times \Delta^{\{2\}}$ factorizes through
$(\cC_{\cE_1})_{/x}$ extends to a map $(A^\triangleleft)^{\sharp^2_{\{1\}}}\times T\to (\cC_{/x},\cE'_1,\cE'_2)$. Here $x$ is an object of $\cC$ and $\cE'_i$ denotes the inverse image of $\cE_i$ via the map $\cC_{/x}\to \cC$ for $i=1,2$. Recall that $A$ is the nerve of a partially ordered set. We let $B\subseteq A^\triangleleft\times\Delta^2$ denote the full subcategory spanned by all vertices except $(-\infty,1)$. Consider the commutative diagram of inclusions
\[
\xymatrix{
A\times\Delta^{\{0,2\}}\ar[r]\ar[d] & A\times \Delta^2\ar[d]\ar[rd]\\
(A\times \Delta^{\{0,2\}})^\triangleleft \ar[r]\ar[d] &
A\times \Delta^2\coprod_{A\times\Delta^{\{0,2\}}}(A\times \Delta^{\{0,2\}})^\triangleleft\ar[r]^-{h}\ar[d]
&(A\times \Delta^2)^\triangleleft\ar[d]\\
A^\triangleleft\times \Delta^{\{0,2\}}\ar[r] & A\times \Delta^2\coprod_{A\times\Delta^{\{0,2\}}}A^\triangleleft\times \Delta^{\{0,2\}}\ar[r]^-{h'} & B
}
\]
where the lower left (resp.\ right) vertical arrow carries the cone point of $(A\times\Delta^{\{0,2\}})^\triangleleft$ (resp.\
$(A\times\Delta^2)^\triangleleft$) to $(-\infty,0)$, and the squares on the left are clearly pushouts. For any simplex $\sigma$ of $B$, if $\sigma$ is not a simplex of $(A\times \Delta^2)^\triangleleft$, then $(-\infty,2)$ is a vertex of $\sigma$, so that $\sigma$ is a simplex of $A^\triangleleft\times\Delta^{\{0,2\}}$. Thus $h'$ is a pushout of $h$, which is inner anodyne by \cite{HTT}*{Lemma 2.1.2.3}, since the inclusion $A\times \Delta^{\{0,2\}}\subseteq A\times \Delta^2$ is left anodyne by \cite{HTT}*{Corollary 2.1.2.7}. Thus $a$ extends to a map $a'\colon B\to\cC_{/x}$. We would like to apply \cite{HTT}*{Lemma 4.3.2.13} to conclude that there exists a right Kan extension $b\colon A^\triangleleft\times \Delta^2\to \cC_{/x}$ of $a'$. The only condition we need to check for this is that the induced diagram $B_{(-\infty,1)/}\to B\xrightarrow{a'}\cC_{/x}$ has a limit. However, the composite map factorizes through $a_0\colon
B_{(-\infty,1)/}\to(\cC_{\cE_1})_{/x}$. By Condition (2) and Lemma \ref{a4le:over_pull} below, the $\infty$-category $(\cC_{\cE_1})_{/x}$ admits
finite limits and such limits are preserved by the inclusion $(\cC_{\cE_1})_{/x}\subseteq\cC_{/x}$. We therefore obtain a limit diagram $b_0\colon A^\triangleleft \times \Delta^{\{1,2\}}\to (\cC_{\cE_1})_{/x}$ extending $a_0$ and a right Kan extension $b$ of $a'$. The restriction of $b$ to $(A^\triangleleft \times \Delta^{\{1,2\}})\cup B$ is equivalent to the amalgamation $b_1$ of $b_0$ and $a'$. Thus, by \cite{HTT}*{Lemma 2.4.6.3}, up to replacing $b$ by an extension of $b_1$, we may assume that $b\res A^\triangleleft \times \Delta^{\{1,2\}}$ factorizes through $(\cC_{\cE_1})_{/x}$.

Note that $b$ does not necessarily carry the edge $(-\infty,0)\to(-\infty,1)$ into $\cE'_2$, which is the last requirement to conclude that $b$ gives rise to the desired extension $(A^\triangleleft)^{\sharp^2_{\{1\}}}\times T\to (\cC_{/x},\cE'_1,\cE'_2)$. To overcome this problem, we apply Condition (3) to the arrow $b((-\infty,0)\to (-\infty,1))$ to get a $2$-simplex $\gamma$ of $\cC_{/x}$. Consider the totally ordered set $I=\{0<1^-<1<2\}$, which contains $[2]=\{0<1<2\}$. The amalgamation of $\gamma$ and $b$ is a map $c\colon K\to\cC_{/x}$, where
\[
K\coloneqq A^\triangleleft \times\Delta^2\coprod_{\{-\infty\}\times \Delta^1}
\{-\infty\}\times \Delta^{\{0,1^-,1\}}\subseteq A^\triangleleft \times \Delta^I,
\]
with $c((-\infty,0)\to (-\infty,1^-))\in\cE'_2$ and $c((-\infty,1^-)\to(-\infty,1))\in\cE'_1$. We let $\cF'_1$ (resp.\ $\cF'_2$) denote the set of all degenerate edges of $\Delta^I$ and all edges of $\Delta^{\{1^-,1,2\}}$ (resp.\ $\Delta^{\{0,1^-\}}$). Consider the pushout
\[
(L,\cH_1,\cH_2)=(A^\triangleleft)^{\sharp^2_{\{1\}}} \times
(\Delta^I,\cF'_1,\cF'_2)\coprod_{(A\times \Delta^I)^{\flat^2}} (A\times \Delta^2)^{\flat^2}
\]
given by the degeneracy map $I\to [2]$ identifying $1^-$ and $1$. The inclusion $K\subseteq L$ induced by the inclusion $K\subseteq A^\triangleleft \times \Delta^I$ is a pushout of the inclusion
\begin{multline*}
r\colon (\{-\infty\}\times \Delta^0)\star (A^\triangleleft \times \Delta^{\{1,2\}})
\coprod_{(\{-\infty\}\times \Delta^0)\star (\{-\infty\}\times
\Delta^{\{1\}})} (\{-\infty\}\times \Delta^{\{0,1^-\}})\star
(\{-\infty\}\times \Delta^{\{1\}}) \\
\to (\{-\infty\}\times\Delta^{\{0,1^-\}})\star (A^\triangleleft \times
\Delta^{\{1,2\}}).
\end{multline*}
Indeed, for any simplex $\sigma$ of $L$, if $(-\infty,1^-)$ is a vertex of $\sigma$, then $\sigma$ is a simplex of the target of $r$; otherwise
$\sigma$ is a simplex of $A^\triangleleft\times \Delta^2$. Moreover, $r$ is inner anodyne by \cite{HTT}*{Lemma 2.1.2.3}, since the inclusion
$\{-\infty\}\times \Delta^{\{1\}}\subseteq A^\triangleleft\times\Delta^{\{1,2\}}$ is left anodyne by \cite{HTT}*{Lemma 4.2.3.6}. Note that we have $\cH_2\subseteq K_1$ and $c$ induces a map $(K,K_1\cap\cH_1,\cH_2)\to (\cC_{/x},\cE'_1,\cE'_2)$. Moreover, any edge in $\cH_1$ that is not in $K$ has the form $(-\infty,1^-)\to (l,m)$ with $m\ge 1$ and can be extended to a $2$-simplex
\[
\xymatrix{
&(-\infty,1)\ar[rd] \\ (-\infty,1^-)\ar[rr]\ar[ru]&&(l,m)
}
\]
with oblique arrows in $K_1\cap \cH_1$. Thus, by Condition (1) and Lemma \ref{a3le:comp}, $c$ extends to a map $c'\colon (L,\cH_1,\cH_2)\to(\cC_{/x},\cE'_1,\cE'_2)$. The restriction of $c'$ to $A^\triangleleft\times\Delta^{\{0,1^-,2\}}\simeq A^\triangleleft\times\Delta^2$ provides the desired extension.
\end{proof}

\begin{lem}\label{a4le:over_pull}
Let $\cC$ and $\cD$ be $\infty$-categories and $f\colon \cC\to \cD$ a functor. Assume that $\cC$ admits pullbacks and pullbacks are preserved by $f$. Then, for any object $x$ of $\cC$, the overcategory $\cC_{/x}$ admits finite limits and such limits are preserved by the functor $f'\colon\cC_{/x}\to\cD_{/f(x)}$.
\end{lem}

\begin{proof}
The morphism $\id_x$ is a final object of $\cC_{/x}$ and $f(\id_x)=\id_{f(x)}$ is a final object of $\cD_{/f(x)}$. By Lemma \ref{a4le:over} below, $\cC_{/x}$ admits pullbacks and the functors $\cC_{/x}\to \cC$ and $\cD_{/f(x)}\to \cD$ preserve pullbacks. Since the latter is conservative, the functor $f'$ preserves pullbacks. We conclude by \cite{HTT}*{Corollaries 4.4.2.4, 4.4.2.5}.
\end{proof}

\begin{lem}\label{a4le:over}
Let $A$ and $B$ be simplicial sets. Assume that $B$ is weakly contractible. Let $\cC$ be an $\infty$-category and $p\colon A\to \cC$ a diagram. Then a diagram $f\colon B\to \cC_{/p}$ admits a limit if and only if the composition $B\xrightarrow{f} \cC_{/p}\to \cC$ admits a limit. Moreover, $\bar f\colon B^\triangleleft\to \cC_{/p}$ is a limit diagram if and only if the composition $B^\triangleleft \xrightarrow{\bar f} \cC_{/p}\to\cC$ is a limit diagram.
\end{lem}

This applies in particular to the case where $B=\Lambda^2_2$. In this case we have $B^\triangleleft\simeq \Delta^1\times \Delta^1$.

\begin{proof}
We let $q\colon B\star A\to \cC$ denote the diagram corresponding to $f$. We let $q_0$ denote the restriction of $q$ to $B$. Since the inclusion $B\subseteq B\star A$ is left anodyne by \cite{HTT}*{Lemma 4.2.3.6}, the map $\cC_{/q}\to \cC_{/q_0}$ is a trivial Kan fibration by \cite{HTT}*{Proposition 2.1.2.5}. Therefore, $\cC_{/q}$ admits a final object if and only if $\cC_{/q_0}$ admits a final object, and an object of $\cC_{/q}$ is a final object if and only if its image in $\cC_{/q_0}$ is a final object.
\end{proof}

\begin{remark}\label{a4re:deligne}
In the situation of Theorem \ref{a4pr:descent}, for every $\infty$-category $\cD$, the functor
\begin{equation}\label{a4eq:fun}
\Fun(\cC,\cD)\to \Fun(\delta_2^*\cC_{\cE_1,\cE_2},\cD)
\end{equation}
is an equivalence of $\infty$-categories. This generalizes Deligne's gluing result \cite{SGA4}*{Expos\'{e} xvii, Proposition 3.3.2}, which can be interpreted as saying that \eqref{a4eq:fun} induces a bijection between the sets of equivalence classes of objects when $\cC$ is the nerve of an ordinary category and $\cD=\rN(\cat)$.
\end{remark}

In the remaining part of this section, we will study a variant of the diagonal functor $\delta_2^*\colon \Sset[2]\to \Sset$, which will allow,
among other things, to express the $\infty$-category of correspondences in \cite{Gai1} in terms of our multisimplicial nerves. This will not be used in the later sections of this article. Therefore, the uninterested reader may safely skip the remaining part of this section and proceed to Section
\ref{a5ss}.

\begin{definition}\label{a4de:correspondence}
Let $X$ be a bisimplicial set. We let $\delta_{2\nabla}^*X$ denote the simplicial set defined by $(\delta_{2\nabla}^*X)_n=\Hom_{\Sset[2]}(\Cpt^n,X)$. This defines a functor $\delta_{2\nabla}^*\colon \Sset[2]\to\Sset$.
\end{definition}

Recall that we have $(\delta_{2}^*X)_n\simeq \Hom_{\Sset[2]}(\Delta^{n,n},X)$.

\begin{theorem}\label{a4th:correspondence}
The map
\[
f\colon \delta_2^* X\to \delta_{2\nabla}^* X
\]
induced by the inclusions $\Cpt^n\subseteq\Delta^{n,n}$ is a categorical equivalence.
\end{theorem}

Under our convention of representing the first direction vertically and second direction horizontally as in \eqref{a4eq:Cpt}, the map can be
described as ``forgetting the lower-left corner''. Before proving the theorem, let us look at a few examples.

\begin{example}
For $X=\Cpt^n$, we have a canonical isomorphism $\CCpt^n\simeq\delta_{2\nabla}^*\Cpt^n$. An $m$-simplex $\alpha$ of $\CCpt^n$ is given by a sequence $(i_0,j_0)\le \dots \le (i_m,j_m)$ in $\RCpt^n$. The isomorphism carries $\alpha$ to the $m$-simplex of $\delta_{2\nabla}^*\Cpt^n$ given by the map of bisimplicial sets $\Cpt^m\to \Cpt^n$ carrying $(a,b)$ to $(i_a,j_b)$. The map $f$ can be identified with the inclusion
$\square^n\subseteq \CCpt^n$, which is inner anodyne (Lemma \ref{a6le:cpt_inner}), and in particular a categorical equivalence.
\end{example}

\begin{example}
In the situation of Theorem \ref{a4pr:descent}, there exists a non-canonical categorical equivalence $\delta_{2\nabla}^*\cC_{\cE_1,\cE_2}\to \cC$ by Theorem \ref{a4th:correspondence} applied to the bisimplicial set $\cC_{\cE_1,\cE_2}$.
\end{example}

\begin{example}\label{a4ex:corr}
Given a $2$-marked $\infty$-category $(\cC,\cE_1,\cE_2)$ satisfying certain conditions, Gaitsgory defined an $\infty$-category of correspondences $\cC_{\corr:\cE_1,\cE_2}$ \cite{Gai1}*{{\Sec}5.1.2} ($\cE_1=vert$, $\cE_2=horiz$ in his notation) following an idea of Lurie. More generally, given an \emph{arbitrary} $2$-marked $\infty$-category $(\cC,\cE_1,\cE_2)$, using the above functor $\delta_{2\nabla}^*$, one can define the \emph{simplicial set of correspondences} to be
\[
\cC_{\corr:\cE_1,\cE_2}\coloneqq\delta_{2\nabla}^*(\op^2_{\{2\}}\cC^\cart_{\cE_1,\cE_2}).
\]
In other words, we have
\[
(\cC_{\corr:\cE_1,\cE_2})_n=\Hom_{\Sset[2]}(\Cpt^n,\op^2_{\{2\}}\cC^\cart_{\cE_1,\cE_2}).
\]
Applying Theorem \ref{a4th:correspondence} to the bisimplicial set $\op^2_{\{2\}}\cC^\cart_{\cE_1,\cE_2}$, we know that the natural map
\[
\delta^*_{2,\{2\}}\cC^\cart_{\cE_1,\cE_2}\to\cC_{\corr:\cE_1,\cE_2},
\]
given by ``forgetting the lower-right corner'', is a categorical equivalence.
\end{example}

\begin{proof}[Proof of Theorem \ref{a4th:correspondence}]
The proof is very similar to that of Theorem \ref{a4th:multisimplicial_descent}. Consider a commutative diagram
\[
\xymatrix{
\delta_2^* X\ar[r]^-v\ar[d]_f & \Fun(\Delta^l,\cD)\ar[d]^p\\
\delta_{2\nabla}^* X\ar[r]^-w & \Fun(\partial \Delta^l,\cD)
}
\]
as in Lemma \ref{a1le:categorical_equivalence}. Let $\sigma$ be an $n$-simplex of $\delta_{2\nabla}^* X$, corresponding to a map $\tau\colon\Cpt^n\to X$. Consider the commutative diagram
\begin{align}\label{a7eq:descent}
\resizebox{\hsize}{!}{
\xymatrix{
\cN(\sigma)\ar[d]\ar[r] & \Fun(\Delta^l\times\CCpt^n,\cD)
\ar[d]^-{\rres_1}\ar[rr]^-{\rres_2} && \Fun(\Delta^l\times \Delta^n,\cD)\ar[d]\ar@/^4.5pc/[dd]^{\rres_4}\\
\Delta^0\ar[r]^-{h}\ar[rd]_{v\circ \delta_2^* \tau} &  \Fun(H,\cD)\ar[r]\ar[d]&
\Fun(\partial \Delta^l\times \CCpt^n,\cD)\ar[r]^-{\rres_2} & \Fun(\partial\Delta^l\times\Delta^n,\cD)\\
&\Fun(\Delta^l\times \square^n,\cD)\ar[rr]^-{\rres_3}&&\Fun(\Delta^l\times D^n,\cD).
}
}
\end{align}
In the above diagram,
\begin{itemize}
  \item $H$ and the maps $\rres_i$, $1\le i\le 4$ are defined as in the proof of Theorem \ref{a4th:multisimplicial_descent};

  \item $h$ is the amalgamation of $v\circ \delta_2^*\tau\colon\square^n\to\Fun(\Delta^l,\cD)$ and $w\circ\delta_{2\nabla}^*\tau\colon\CCpt^n\to\Fun(\partial\Delta^l,\cD)$;

  \item $\cN(\sigma)$ is defined so that the upper left square is a pullback square;

  \item the unnamed arrows in the middle column and in the upper right square are obvious restrictions.
\end{itemize}

By \cite{HTT}*{Corollaries 2.3.2.4, 2.3.2.5}, the map $j\colon H\hookrightarrow \Delta^l\times \CCpt^n$ is inner anodyne, and consequently
$\rres_1$ is a trivial Kan fibration. It follows that $\cN(\sigma)$ is a contractible Kan complex.

We let $\Phi(\sigma)\colon\cN(\sigma)\to\Fun(\Delta^l\times \Delta^n,\cD)$ denote the composition of the upper horizontal arrows in \eqref{a7eq:descent}. Then $\Phi(\sigma)$ induces a map
\[
\cN(\sigma)^\sharp\times(\Delta^n)^\flat\to\Fun(\Delta^l,\cD)^\flat\subseteq\Fun(\Delta^l,\cD)^{\natural}.
\]
Thus $\Phi(\sigma)$ induces a map $\cN(\sigma)\to\Map^\sharp((\Delta^n)^\flat, \Fun(\Delta^l,\cD)^{\natural})$, which we still denote by $\Phi(\sigma)$. This construction is functorial in $\sigma$, giving rise to a morphism $\Phi\colon\cN\to\Map[\delta_{2\nabla}^*X,\Fun(\Delta^l,\cD)]$ in the category $(\Sset)^{(\del_{/ \delta_{2\nabla}^*X})^{op}}$.

The composition $\Delta^n\hookrightarrow \CCpt^n\xrightarrow{\delta^*_{2\nabla}\tau}X$, where the first map is the diagonal embedding, is $\sigma$. Thus the composition of the middle row of \eqref{a4eq:descent} is given by $w(\sigma)$. Thus $\Map[\delta_{2\nabla}^*X,p]\circ\Phi\colon\cN\to\Map[\delta_{2\nabla}^*X,\Fun(\partial\Delta^l,\cD)]$ factorizes through the morphism $\Delta^0_{(\del_{/\delta_{2\nabla}^* X})^{op}}\to \Map[\delta_{2\nabla}^* X,\Fun(\partial \Delta^l,\cD)]$ corresponding to $w$ via Remark \ref{a2re:functors}.

Now let $\sigma'$ be an $n$-simplex of $\delta_2^* X$ corresponding to a map $\tau'\colon \Delta^{n,n}\to X$. The restriction of $v\circ\delta^2\tau'\colon\Delta^{[n,n]}\to\Fun(\Delta^l,\cD)$ to $\CCpt^n\subseteq \Delta^{[n,n]}$ provides a vertex of $\nu(\sigma')$ of $\cN(f(\sigma'))$, whose image under $\Phi(f(\sigma'))$ is $v(\sigma')$. This construction is functorial in $\sigma'$, giving rise to $\nu\in\Gamma(f^*\cN)_0$ such that $f^*\Phi\circ \nu=v$. Applying Proposition \ref{a1pr:extension} to $\Phi$, the map $f\colon \delta_2^*X\to
\delta_{2\nabla}^* X$ and the global section $\nu$ of $f^*\cN$, we obtain a map $u\colon \delta_{2\nabla}^* X\to \Fun(\Delta^l,\cD)$ satisfying $p\circ u=w$ such that $u\circ f$ and $v$ are homotopic over $\Fun(\partial\Delta^l,\cD)$, as desired.
\end{proof}

\subsection{Cartesian gluing}
\label{a5ss}

In \Sec\ref{a4ss}, we gave a general criterion for multisimplicial descent (Theorem \ref{a4th:multisimplicial_descent}). It is often impossible to apply the theorem directly to Cartesian multisimplicial nerves, as the simplicial set of compactifications for Cartesian tilings is often empty for $n\ge 2$. However, we have seen that certain bigger multisimplicial nerves do satisfy multisimplicial descent (Theorem
\ref{a4co:multisimplicial_descent} and Theorem \ref{a4pr:descent}). In this section, we complete the picture by comparing Cartesian multisimplicial nerves with bigger multisimplicial nerves. The basic idea is to decompose a square $\sigma$ in an $\infty$-category
\begin{equation}\label{a5eq:square}
\xymatrix{w \ar[r]\ar[d]& y\ar[d]\\ z\ar[r]&x}
\end{equation}
into a diagram $\sigma'$
\begin{equation}\label{a5eq:decompose}
\xymatrix{w\ar[rd]\\&w' \ar[r]\ar[d]&y\ar[d]\\&z\ar[r]&x,}
\end{equation}
where the inner square is Cartesian. More precisely, $\sigma'$ is a right Kan extension of $\sigma$ along the full embedding $\Delta^1\times\Delta^1\to(\Delta^1\times \Delta^1)^\triangleleft$ carrying $(0,0)$ to the cone point $-\infty$ and carrying every other vertex $(i,j)$ to $(i,j)$. To deal with the oblique arrow $f\colon w'\to w$, we consider the square
\[
\xymatrix{
w\ar[r]^{\id_w}\ar[d]_{\id_w} & w\ar[d]^f \\ w\ar[r]^f & w'.
}
\]
If this square is a pullback square (which happens exactly when $f$ is a monomorphism), we stop. Otherwise, we apply the above procedure recursively, which leads to the diagonal map $\delta\colon w\to w\times_{w'}w$ of $f$, and the diagonal of $\delta$, and so on.

To state our result, we introduce a bit of notation. For sets of edges $\cE_1$, $\cE_2$, $\cE$ of an $\infty$-category $\cC$, we let
$\cE_1*_\cC^\cE\cE_2\subseteq \cE_1*_\cC\cE_2$ denote the set of squares that admit a decomposition as above with $w\to w'$ in $\cE$. We have
$\cE_1*^\cart_\cC\cE_2=\cE_1*_\cC^\cE\cE_2$, where $\cE$ is the set of equivalences of $\cC$ (or the set of degenerate edges of $\cC$).

The main result of this section is the following.

\begin{theorem}[Cartesian gluing]\label{a5th:cartesian_gluing}
Let $\cC$ be an $\infty$-category and $K$ a finite set. Let $(\cC,\cT)\subseteq(\cC,\cT')$ be two $(\{1,2\}\coprod K)$-tiled $\infty$-categories such that $\cT_j=\cT'_j$ for all $j\in \{1,2\}\coprod K$, and $\cT_{jj'}=\cT'_{jj'}$ for all $j,j'\in \{1,2\}\coprod K$ with $j\neq j'$, except when $(j,j')=(1,2)$ or $(2,1)$, we have $\cT_{12}=\cT_1*_\cC^\cart \cT_2$ and $\cT'_{12}=\cT_1*_\cC^{\cE}\cT_2$, where $\cE\subseteq \cT_1\cap \cT_2$ is a set of edges of $\cC$. Suppose that the following conditions are satisfied:
\begin{enumerate}
  \item $\cT_1*_{\cC}\cT_2=\cT_1*^{\cC_1}_{\cC}\cT_2$; $\cT_1$ (resp.\ $\cT_2$) is stable under composition and pullback by $\cT_2$
      (resp.\ $\cT_1$).

  \item Every morphism $f$ in $\cE$ is $n$-truncated for some integer $n\ge -2$ (which may depend on $f$) \cite{HTT}*{Definition 5.5.6.8}.
      Moreover, $\cE$ is stable under composition, pullback by $\cT_1\cup\cT_2$, and taking diagonals: for every edge $y\to x$ in $\cE$, its
      diagonal $y\to y\times_x y$ is in $\cE$ (the pullback $y\times_x y$ exists in $\cC$ by the first part of Condition (1)).

  \item For every $k\in K$, the set $\cT_{1k}$ (resp.\ $\cT_{2k}$) is stable under composition and pullback by $\cT_{2k}$ (resp.\ $\cT_{1k}$) in the first direction, and $\cT_{1k}\cap \cT_{2k}$ is stable under pullback by $\cT_{1k}\cup \cT_{2k}$ in the first direction. Moreover, we have
      \begin{equation}\label{a5eq:cond3}
      \cT_{1k}*_{\Fun(\Delta^1,\cC)}^{\cE*_\cC\cT_k}\cT_{2k}
      =\cT_{1k}*_{\Fun(\Delta^1,\cC)}^{(\cE*_\cC\cT_k)\cap\cT_{1k}\cap \cT_{2k}}\cT_{2k}.
      \end{equation}
      See Remark \ref{a5re:explicit} (3) below for an explicit description of the meaning of \eqref{a5eq:cond3}.

  \item For every pair $k,k'\in K$ with $k\neq k'$, and every \emph{Cartesian} square of the form \eqref{a5eq:square} of the $\infty$-category $\Fun(\Delta^1\times \Delta^1,\cC)$ (whose vertices are regarded as squares of $\cC$ in directions $k,k'$), with $y\to x$ given by a $(1,1,1)$-simplex of $\delta_*^{\{1,k,k'\}\square}(\cC,\cT)$ and $z\to x$ given by a $(1,1,1)$-simplex of $\delta_*^{\{2,k,k'\}\square}(\cC,\cT)$ (where the obvious restrictions of $\cT$ are still denoted by $\cT$), we have $w\in\cT_{kk'}$.
\end{enumerate}
Then, for any subset $L\subseteq K$, the inclusion map
\[
\iota\colon \delta^*_{\{1,2\}\amalg K,L}\delta_*^{(\{1,2\}\amalg K)\square}(\cC,\cT)\hookrightarrow\delta^*_{\{1,2\}\amalg K,L}\delta_*^{(\{1,2\}\amalg K)\square}(\cC,\cT')
\]
is a categorical equivalence.
\end{theorem}

We note that unlike the theorems in the last section, Theorem \ref{a5th:cartesian_gluing} is symmetric in $\cE_1$ and $\cE_2$.

\begin{remark}
Let us recall some facts about $n$-truncated morphisms, $n\ge -2$, in an $\infty$-category $\cC$.
\begin{itemize}
  \item A morphism $f$ of $\cC$ is $(-2)$-truncated (resp.\ $(-1)$-truncated) if and only if $f$ is an equivalence (resp.\ a monomorphism).

  \item The set of $n$-truncated morphisms of $\cC$ is admissible. Indeed, the set is stable under pullback by \cite{HTT}*{Remark 5.5.6.12}. It follows from the long exact sequence of homotopy groups that the set is stable under composition. Moreover, given a $2$-simplex $\sigma$ of $\cC$ of the form \eqref{a3eq:2cell}, if $r=\sigma\circ d^2_1$ is $n$-truncated and $p=\sigma\circ d^2_0$ is $(n+1)$-truncated, then $q=\sigma\circ d^2_2$ is $n$-truncated.

  \item Given a morphism $f\colon y\to x$ of $\cC$ such that the fiber product $y\times_x y$ exists, $f$ is $(n+1)$-truncated if and only if its diagonal $y\to y\times_x y$ is $n$-truncated (\cite{HTT}*{Lemma 5.5.6.15} assumes that $\cC$ admits finite limits, but the proof only uses the existence of $y\times_x y$).

  \item In an $(n+1)$-truncated category \cite{HTT}*{Definition 2.3.4.1}, every morphism is $n$-truncated by \cite{HTT}*{Proposition 2.3.4.18}.
\end{itemize}
\end{remark}

\begin{remark}\label{a5re:explicit}
We have the following remarks concerning the conditions in the above theorem.
\begin{enumerate}
  \item The conditions of the theorem imply that the sets $\cT_j$, $\cT_{ij}$, $\cT'_{ij}$ and $\cE$ are all stable under equivalence. Indeed, the second part of Condition (1) implies that $\cT_1$ and $\cT_2$ are stable under equivalence. The second part of Condition (2) implies that $\cE$ is stable under equivalence. It follows that $\cT_{12}$ and $\cT'_{12}$ are stable under equivalence. The first part of Condition (3) implies that $\cT_{1k}$ and $\cT_{2k}$ are stable under equivalence. It follows that $\cT_k$ is stable under equivalence. Finally, Condition (4) implies that $\cT_{kk'}$ is stable under equivalence.

  \item The first part of Condition (1) is satisfied if morphisms in $\cT_1$ admits pullback in $\cC$ by morphisms in $\cT_2$.

  \item The left hand side of \eqref{a5eq:cond3} clearly contains the right hand side. Since $\cT_{1k}$ and $\cT_{2k}$ are stable under equivalence, the meaning of the equality is as follows. Consider a square of the form \eqref{a5eq:square} in the $\infty$-category
      $\Fun(\Delta^1,\cC)$ (whose vertices are regarded as edges of $\cC$ in direction $k$), such that $y\to x,w\to z\in\cT_{1k}$ and $z\to x,w\to y\in\cT_{2k}$. If it has a decomposition of the form \eqref{a5eq:decompose} with $w\to w'$ in $\cE*_\cC \cT_k$, then $w\to w'$ is in $\cT_{1k}\cap\cT_{2k}$.

  \item Suppose that we have $\cT_{jj'}=\cT_j*^\cart_{\cC}\cT_{j'}$ for all $j,j'\in \{1,2\}\coprod K$ with $j\neq j'$. Then the identity
      \eqref{a5eq:cond3} holds automatically, by (the dual of) \cite{HTT}*{Lemma 4.4.2.1}. Moreover, the first part of Condition (3) implies Condition (4). To see this, consider a square $\sigma$ as in Condition (4). Applying Lemma \ref{a3le:cart} to the corresponding cube (whose vertices are edges of $\cC$ in direction $k'$, say), we get $w\in \cC_1*^\cart_\cC \cC_1$. Applying the first part of Condition (3) to the images of $\sigma$ under the maps $\Fun(\Delta^1\times \Delta^1,\cC)\to \Fun(\Delta^1,\cC)$ induced by $d^1_0\times \id$, $d^1_1\times \id$, we get $w\in\cC_1*_\cC\cT_{k'}$. Similarly, we have $w\in \cT_k*_\cC\cC_1$.

  \item Suppose that we have $\cT_{jj'}=\cT_j*^\cart_{\cC}\cT_{j'}$ for all $j,j'\in \{1,2\}\coprod K$ with $j\neq j'$, and moreover that $\cT_k$ is stable under pullback by either $\cT_1$ or $\cT_2$ for each $k\in K$. Then, by Remark \ref{a3re:cart_square}, Conditions (1)
      and (2) imply Condition (3), which in turn implies Condition (4).
\end{enumerate}
\end{remark}

Combining Theorem \ref{a5th:cartesian_gluing} with Theorem \ref{a4co:multisimplicial_descent} and Theorem \ref{a4pr:descent}, we obtain the following.

\begin{theorem}\label{a5th:cartesian_descent}
Let $\cC$ be an $\infty$-category and let $K$ be a finite set. We are given a $(\{0,1,2\}\coprod K)$-marked $\infty$-category $(\cC,\cE_0,\cE_1,\cE_2,\{\cE_k\}_{k\in K})$ such that
\begin{enumerate}
  \item $\cE_1,\cE_2\subseteq \cE_0$; $\cE_0$ is stable under composition. Moreover, for every morphism $f$ in $\cE_0$, there exists a $2$-simplex of $\cC$ of the form
      \[
      \xymatrix{&y\ar[rd]^p\\z\ar[ru]^q\ar[rr]^f && x}
      \]
      with $p\in\cE_1$ and $q\in\cE_2$.

  \item Every morphism $f$ in $\cE_1\cap \cE_2$ is $n$-truncated for some integer $n\ge -2$ (which may depend on $f$).

  \item $\cE_k$ is stable under pullback by $\cE_1$ for every $k\in K$.

  \item Edges in $\cE_1$ admit pullbacks in $\cC$ by edges in $\cE_k$ for all $k\in K$.

  \item $\cE_1*_\cC \cE_2=\cE_1 *_\cC^{\cE_1\cap \cE_2} \cE_2$. Moreover, $\cE_1$ (resp.\ $\cE_2$) is stable under composition and pullback b by $\cE_k$ for all $k\in K$ and by $\cE_2$ (resp.\ $\cE_1$); $\cE_1\cap\cE_2$ is stable under pullback by $\cE_1\cup \cE_2$.

  \item $\cC_{\cE_1}$ admits pullbacks and pullbacks are preserved by the functor $\cC_{\cE_1}\to \cC_{\cE_0}$.
\end{enumerate}
Then, for every subset $L\subseteq K$, the natural map
\[
g\colon \delta^*_{\{1,2\}\amalg K,L}\cC_{\cE_1,\cE_2,\{\cE_k\}_{k\in K}}^\cart
\to \delta^*_{\{0\}\amalg K,L}\cC_{\cE_0,\{\cE_k\}_{k\in K}}^\cart
\]
is a categorical equivalence (see Definition \ref{a3de:cartesian_nerve} for the notation).
\end{theorem}

\begin{remark}
If $\cC$ admits pullbacks and $\cE_1$, $\cE_2$ are admissible, then Conditions (4), (5), and (6) of Theorem \ref{a5th:cartesian_descent} hold. Moreover, in this case, Condition (1) of Theorem \ref{a5th:cartesian_descent} implies that $\cE_0$ is admissible by Remark \ref{a3re:admissible}. Indeed, $\cE_0$ is clearly stable under pullback, and given a $2$-simplex as in Condition (1), we have a diagram
\[
\xymatrix{
z\ar[r]^-{d_q}\ar[rd]_-{d_f} & z\times_y z\ar[r]\ar[d]&y\ar[d]^-{d_p} \\
& z\times_x z\ar[r] & y\times_x y,
}
\]
where the square is a pullback by Lemma \ref{a5le:pullback} below, so that the diagonal $d_f$ of $f$ belongs to $\cE_0$.
\end{remark}

\begin{lem}\label{a5le:pullback}
Let $\cC$ be an $\infty$-category admitting pullbacks. Consider two $2$-simplices of $\cC$ sharing an edge as depicted by the diagram
\[
\xymatrix{
z\ar[r]\ar[rd]&x'\ar[d]&y\ar[l]\ar[ld]\\ &x.
}
\]
Then we have a pullback square
\[
\xymatrix{
y\times_{x'} z\ar[d]\ar[r]& x'\ar[d]\\ y\times_{x} z\ar[r] & x'\times_x x',
}
\]
where the right vertical arrow is the diagonal of $x'\to x$.
\end{lem}

\begin{proof}
Indeed, we have a diagram
\[
\xymatrix{
y\times_{x'}z\ar[rr]\ar[dd]\ar@{-->}[rd]&& y\ar[dd]\ar[rd]\\
&y\times_{x} z\ar@{-->}[rr]\ar@{-->}[dd] && y\times_x x'\ar[r]\ar[dd] & y\ar[dd]\\
z\ar[rr]\ar[rd] && x'\ar[rd]\\
&x'\times_x z\ar[rr]\ar[d] && x'\times_x x'\ar[r]\ar[d] & x'\ar[d]\\
&z\ar[rr] && x'\ar[r] & x
}
\]
where the front face of the cube and the squares on the back page are pullbacks. It follows that the other two faces of the cube containing $x'$ are pullbacks. Therefore, all the faces of the cube are pullbacks.
\end{proof}

\begin{proof}[Proof of Theorem \ref{a5th:cartesian_descent}]
Denote by $(\cC,\cT)$ the $(\{1,2\}\coprod K)$-tiled simplicial set as in Theorem \ref{a4co:multisimplicial_descent}. Then the map $g$ factorizes as
\[
\delta^*_{\{1,2\}\amalg K,L}\cC_{\cE_1,\cE_2,\{\cE_k\}_{k\in K}}^\cart
\xrightarrow{\iota}\delta^*_{\{1,2\}\amalg K,L}\delta_*^{(\{1,2\}\amalg K)\square}(\cC,\cT) \xrightarrow{f}\delta^*_{\{0\}\amalg
K,L}\cC_{\cE_0,\{\cE_k\}_{k\in K}}^\cart.
\]
By Theorem \ref{a5th:cartesian_gluing} applied to the inclusion $(\cC,(\cE_1,\cE_2,\{\cE_k\}_{k\in K})^\cart)\subseteq(\cC,\cT)$ (see Definition \ref{a3de:cartesian_nerve} for the notation) and $\cE=\cE_1\cap\cE_2$, the inclusion $\iota$ is a categorical equivalence. Indeed, by Condition (3) of Theorem \ref{a5th:cartesian_descent} and Remark \ref{a5re:explicit} (5), it suffices to check Conditions (1) and (2) of Theorem \ref{a5th:cartesian_gluing}. The first part of Condition (2) of Theorem \ref{a5th:cartesian_gluing} is Condition (2) of Theorem \ref{a5th:cartesian_descent}. Condition (1) and the second part of Condition (2) of Theorem \ref{a5th:cartesian_gluing} follow from Condition (5) of Theorem \ref{a5th:cartesian_descent}. To show that $f$ is a categorical equivalence as well, we use Theorem \ref{a4co:multisimplicial_descent} (with $\alpha=1$). Conditions (1) and (2) of Theorem \ref{a4co:multisimplicial_descent} follow from Condition (1) of Theorem \ref{a5th:cartesian_descent}. Condition (3) of Theorem \ref{a4co:multisimplicial_descent} follows from Condition (5) of Theorem \ref{a5th:cartesian_descent}. Conditions (4) and (5) of Theorem \ref{a4co:multisimplicial_descent} are Conditions (3) and (4) of Theorem \ref{a5th:cartesian_descent}, respectively. It remains to check that $\Komp^1_{\cC,\cE_1,\cE_2}(\tau)$ is weakly contractible for every simplex
$\tau$ of $\cC_{\cE_0}$, which follows from Theorem \ref{a4pr:descent} applied to $(\cC_{\cE_0},\cE_1,\cE_2)$. Conditions (1), (2), (3) of Theorem \ref{a4pr:descent} follow from Conditions (5), (6), (1) of Theorem \ref{a5th:cartesian_descent}, respectively.
\end{proof}

The rest of this section is devoted to the proof of Theorem \ref{a5th:cartesian_gluing}. A key ingredient in the proof is an analogue of the diagram \eqref{a5eq:decompose} for decompositions of simplices of higher dimensions. Such decompositions are naturally encoded by certain lattices. Let us review some basic terminology.

\begin{definition}[Lattice]
By a \emph{lattice} we mean a nonempty partially ordered set admitting products (namely, infima) and coproducts (namely, suprema) of pairs of
elements, or equivalently, admitting finite nonempty products and coproducts. In a lattice, we denote products by $\wedge$ and coproducts by $\vee$. A lattice $P$ is said to be \emph{distributive} if $p\wedge (q\vee r)=(p\wedge q)\vee (p\wedge r)$ for all $p,q,r\in P$, or equivalently, $p\vee (q\wedge r)=(p\vee q)\wedge (p\vee r)$  for all $p,q,r\in P$ \cite{DP}*{Lemma 4.3}.

A map between lattices preserving finite nonempty products and coproducts is called a \emph{morphism} of lattices. A morphism of lattices necessarily preserves order.
\end{definition}

Note that a finite lattice admits arbitrary products and coproducts.

\begin{definition}[Sublattice]\label{a5de:interval}
A nonempty subset of a lattice is called a \emph{sublattice} if it is stable under finite nonempty products and coproducts. We endow the subset with the induced lattice structure.
\end{definition}

Subsets of a lattice $P$ of the forms $P_{p/}$, $P_{/q}$, $P_{p//q}$ for $p\le q$ in $P$ are necessarily sublattices of $P$.

\begin{definition}[Up-set lattice]
Let $P$ be a partially ordered set. A subset $Q$ of $P$ is called an \emph{up-set} if $q\in Q$ and $p\ge q$ with $p\in P$ imply $p\in Q$. We
order the set $\cU(P)$ of up-sets of $P$ by \emph{inverse inclusion}: $Q\le Q'$ if and only if $Q\supseteq Q'$. Then $\cU(P)$ becomes a distributive lattice admitting arbitrary products and coproducts. In fact, we have $Q\vee Q'=Q\cap Q'$ and $Q\wedge Q'=Q\cup Q'$. We call $\cU(P)$ the \emph{up-set lattice} of $P$.

We let $\varsigma^P\colon P\to \cU(P)$ denote the map carrying $p$ to $P_{p/}$, which is a fully faithful functor (namely, an order embedding)
since we have chosen the inverse inclusion order on $\cU(P)$. Note that $\varsigma^P$ preserves coproducts whenever they exist in $P$. On the other hand, $\varsigma^P$ does not preserve the product of any family of elements, unless the family admits a minimum.
\end{definition}

\begin{remark}
Although we do not need it in the sequel, let us recall the correspondence between finite partially ordered sets and finite distributive lattices \cite{DP}*{Chapter 5} via up-set lattices. An element $p$ of a lattice $L$ is said to be \emph{product-irreducible} if $p$ is not a final object (namely, maximum) of $L$ and $p=a\wedge b$ implies $p=a$ or $p=b$ for all $a,b\in L$. We let $\cI(L)\subseteq L$ denote the subset of product-irreducible elements of $L$. The map $\varsigma^P$ factorizes to give an embedding $P\to \cI(\cU(P))$, which is an isomorphism if $P$ is finite. The map $\eta_L\colon L\to \cU(\cI(L))$ carrying $x$ to $\cI(L)_{x/}$ is a morphism of lattices preserving initial and final objects. Birkhoff's representation theorem states that $\eta_L$ is an isomorphism for any finite distributive lattice $L$.
\end{remark}

We will need the following properties of up-set lattices.

\begin{remark}\label{a5re:U1}
We have an isomorphism $\cU(P^\triangleright)\simeq \cU(P)^\triangleright$ carrying $Q\neq \emptyset$ to $Q\cap P$ and carrying $\emptyset$ to the cone point of $\cU(P)^\triangleright$. In particular, $\cU(P)$ can be identified with the sublattice of $\cU(P^\triangleright)$ spanned by nonempty up-sets of $P^\triangleright$, or equivalently, up-sets of $P^\triangleright$ that contain the cone point.
\end{remark}

\begin{remark}\label{a5re:U2}
For $Q\in \cU(P)$, we have $Q\subseteq P$ and $\varsigma^P(Q)\subseteq\varsigma^P(P)\subseteq \cU(P)$. Moreover, we have $\varsigma^P(P)_{Q/}=\varsigma^P(Q)$. Thus a diagram $F\colon \rN(\cU(P))\to\cC$ in an $\infty$-category $\cC$ is a right Kan extension along $\rN(\varsigma^P)$ if and only if for every $Q\in \cU(P)$, the restriction of $F$ to $\rN(\varsigma^P(Q))^\triangleleft$ exhibits $F(Q)$ as the limit of $F\res \rN(\varsigma^P(Q))$. Note that when $Q\in\varsigma^P(P)$, the last condition is automatic. To alleviate notation, we will write $\varsigma^P$ for $\rN(\varsigma^P)$.
\end{remark}

\begin{definition}
Let $P$ and $P'$ be partially ordered sets and let $f\colon P'\to P$ be an order-preserving map. The map $\cU^f\colon \cU(P)\to \cU(P')$ carrying $Q$ to $f^{-1}(Q)$ is a morphism of lattices preserving products and coproducts. The functor $\cU^f$ admits a right adjoint $\cU_f\colon \cU(P')\to \cU(P)$ carrying an up-set $Q'$ of $P'$ to the up-set of $P$ generated by $f(Q')$. In other words, $\cU_f(Q')=\bigcup_{q\in Q'} P_{f(q)/}$. The functor $\cU_f$ preserves products.
\end{definition}

We will need the following properties of the functor $\cU_f$.

\begin{remark}\label{a5re:U3}
The following diagram commutes:
\[
\xymatrix{
P'\ar[r]^-{\varsigma^{P'}}\ar[d]_f & \cU(P')\ar[d]^{\cU_f}\\ P\ar[r]^-{\varsigma^P}& \cU(P).
}
\]
\end{remark}

\begin{remark}\label{a5re:U5}
Suppose that $P'$ admits nonempty coproducts and $f$ preserves such coproducts. For $Q'\in \cU(P')$, the map $f$ restricts to a map $Q'\to\cU_f(Q')$. We claim that the induced map $\rN(Q')^{op}\to \rN(\cU_f(Q'))^{op}$ is cofinal. Indeed, for every $Q\in \cU_f(Q')$, the partially ordered set $Q'\times_{\cU_f(Q')}\cU_f(Q')_{/Q}$ is nonempty and admits nonempty coproducts, hence admits a final object. Thus $\rN(Q')\times_{\rN(\cU_f(Q'))}\rN(\cU_f(Q'))_{/Q}$ is weakly contractible and we apply the criterion of cofinality \cite{HTT}*{Theorem 4.1.3.1}.

In this case, if $F\colon \rN(\cU(P))\to \cC$ is a right Kan extension along $\varsigma^P$, then $F\circ \rN(\cU_f)\colon \rN(\cU(P'))\to \cC$ is a right Kan extension along $\varsigma^{P'}$. Indeed, by Remark \ref{a5re:U2}, it suffices to check that for every $Q'\in \cU(P')$ and every limit diagram $\rN(\cU_f(Q'))^\triangleleft\to \cC$, the induced map $\rN(Q')^\triangleleft\to \cC$ is a limit diagram, which follows from the
above cofinality by \cite{HTT}*{Proposition 4.1.1.8}.
\end{remark}

\begin{lem}\label{a5le:U4}
If $P'$ admits coproducts indexed by a set $I$ and $f\colon P'\to P$ preserves such coproducts, then $\cU_f$ preserves coproducts indexed by $I$. In particular, if $P$ admits coproducts of pairs of elements and $f$ preserves such coproducts, then $\cU_f$ is a morphism of lattices.
\end{lem}

\begin{proof}
Let $Q'_i$, $i\in I$ be up-sets of $P'$. We have $\bigcap_{i\in I}\cU_f(Q'_i)\supseteq \cU_f(\bigcap_{i\in I}Q'_i)$. To show the inclusion in the other direction, let $y\in \bigcap_{i\in I}\cU_f(Q'_i)$. For each $i\in I$, there exists $x_i\in Q'_i$ such that $f(x_i)\le y$. Thus
$f(\bigvee_{i\in I} x_i)= \bigvee_{i\in I}f(x_i)\le y$. This implies $y\in\cU_f(\bigcap_{i\in I} Q'_i)$ since we have $\bigvee_{i\in I}x_i\in\bigcap_{i\in I}Q'_i$.
\end{proof}

\begin{definition}[Exact square]\label{a5de:exact}
By an \emph{exact square} in a lattice, we mean a square that is both a pushout square and a pullback square, or, equivalently, a square of the form
\[
\xymatrix{
x\wedge y\ar[r]\ar[d] & x\ar[d]\\y \ar[r] & x\vee y.
}
\]
The left vertical arrow is called an \emph{exact pullback} of the right vertical arrow.
\end{definition}

Exact squares in $\cU(P)$ correspond to pushout squares of sets. The relevance of such squares is shown by the following lemmas.

\begin{lem}\label{a5le:U7}
Every right Kan extension $F\colon \rN(\cU(P))\to \cC$ along $\varsigma^P$ carries exact squares to pullback squares. More generally, for every full subcategory $R\subseteq \cU(P)$ containing $\varsigma^P(P)$, every functor $F\colon\rN(R)\to\cC$ that is a right Kan extension of $F\res\rN(\varsigma^P(P))$ carries exact squares to pullback squares.
\end{lem}

\begin{proof}
Let
\begin{equation}\label{a5eq:exact}
\xymatrix{Q\cup Q'\ar[r]\ar[d] & Q\ar[d]\\Q'\ar[r] & Q\cap Q'}
\end{equation}
be an exact square in $R$. We consider $S=\varsigma^P(P)\cup \{Q,Q',Q\cap Q'\}$, satisfying $\varsigma^P(P)\subseteq S\subseteq R$. By \cite{HTT}*{Proposition 4.3.2.8}, $F$ is a right Kan extension of $F\res\rN(S)$. In particular, the restriction of $F$ exhibits $F(Q\cup Q')$ as a limit of $F\res \rN(S_{Q\cup Q'/})$. By Lemma \ref{a3le:cofinal}, the map $\Lambda^{2}_0\to \rN(S_{Q\cup Q'/})^{op}$ induced by the square
\eqref{a5eq:exact} is cofinal. Thus by \cite{HTT}*{Proposition 4.1.1.8}, $F$ carries the square to a pullback square in $\cC$.
\end{proof}

\begin{lem}\label{a5le:exact}
Let $P$ be a finite partially ordered set. Every morphism $Q\to Q'$ in $\cU(P)$ is the composition of a finite sequence of exact pullbacks of the morphisms $\omega^P(x)\colon \varsigma^P(x)\to \varsigma^P(x)-\{x\}$ for $x\in Q-Q'$.
\end{lem}

\begin{proof}
We may choose a (finite) sequence of morphisms $Q=Q_0\to\cdots\to Q_m=Q'$ such that for $1\leq i\leq m$, $Q_{i-1}=Q_i\cup\{x_i\}$, where $x_i\in Q-Q_i$ is a maximal element. For each $i$, the following diagram
\[
\xymatrix{
Q_{i-1} \ar[r]\ar[d] & \varsigma^P(x_i) \ar[d]^-{\omega^P(x_i)} \\
Q_i \ar[r]& \varsigma^P(x_i)-\{x_i\}
}
\]
is an exact square. Thus the lemma follows.
\end{proof}

The following lattices encode generalizations of the diagram \eqref{a5eq:decompose}.

\begin{notation}\label{a5no:cart}
For $n\ge 0$, we let $\Crt^n$ denote the sublattice of $\cU([n]\times [n])$ spanned by nonempty up-sets of $[n]\times [n]$ and we let $\varsigma^n\colon[n]\times [n]\to \Crt^n$ denote the map induced by $\varsigma^{[n]\times[n]}$ carrying $(p,q)$ to $([n]\times [n])_{(p,q)/}$. For an order-preserving map $d\colon [m]\to [n]$, we let $\Crt(d)\colon\Crt^m\to\Crt^n$ denote the map induced by $\cU_{d\times d}$. Put
$\Cart^n\coloneqq\rN(\Crt^n)$ and $\Cart(d)\coloneqq\rN(\Crt(d))$. We still write $\varsigma^n$ for $\rN(\varsigma^n)$.
\end{notation}

By Remark \ref{a5re:U1}, we have $\Crt^n\simeq \cU([n]\times [n]-\{(n,n)\})$. The definition of $\Crt^n$ given above has the advantage of being functorial with respect to $[n]$. Every up-set of $[n]\times [n]$ has the form $\{(p,q)\in [n]\times [n]\mid q\ge a_p\}$ for a sequence of integers $-1\le a_0\le\dots\le a_n\le n$. Thus the cardinality of $\Crt^n$ is $\binom{2n+2}{n+1}-1$.

Below are the Hasse diagrams of $\Crt^1$ and $\Crt^2$, rotated so that the initial objects are shown in the upper-left corners. Bullets represent elements in the images of $\varsigma^1$ and $\varsigma^2$. The dashed boxes represent $\Crt^1_{0,1}$ and $\Crt^2_{1,2}$ (see Construction \ref{a5cs:boxplus} (1) below).
\begin{equation}\label{a5eq:Hasse}
\begin{xy}
(0,5)="0000";
(5,0)*\cir<1.8pt>{}="00"**\dir{-}; (10,0)="01"**\dir{-};
(10,-5)="11"**\dir{-};
(5,-5)="10"**\dir{-};
"00"**\dir{-};
"0000"*{\bullet}; "01"*{\bullet}; "11"*{\bullet}; "10"*{\bullet};
(-2,-2)="b"; (-2,7)**\dir{--}; (12,7)**\dir{--}; (12,-2)**\dir{--}; "b"**\dir{--};
\end{xy}\qquad\qquad
\begin{xy}
(0,5)="000000"; (5,0)*\cir<1.8pt>{}="0000"**\dir{-};
(10,0)*\cir<1.8pt>{}="0001"**\dir{-}; (15,0)="0101"**\dir{-};
(5,-5)*\cir<1.8pt>{}="0010"; (10,-5)*\cir<1.8pt>{}="0011"**\dir{-};
(15,-5)*\cir<1.8pt>{}="0111"**\dir{-};
(5,-10)="1010";
(10,-10)*\cir<1.8pt>{}="1011"**\dir{-};
(15,-10)="1111"**\dir{-};
(12.5,-7.5)*\cir<1.8pt>{}="00";
(17.5,-7.5)*\cir<1.8pt>{}="01"**\dir{-};
(22.5,-7.5)="02"**\dir{-};
(12.5,-12.5)*\cir<1.8pt>{}="10";
(17.5,-12.5)*\cir<1.8pt>{}="11"**\dir{-};
(22.5,-12.5)="12"**\dir{-};
(12.5,-17.5)="20";
(17.5,-17.5)="21"**\dir{-};
(22.5,-17.5)="22"**\dir{-};
"00"; "10"**\dir{-}; "20"**\dir{-};
"01"; "11"**\dir{-}; "21"**\dir{-};
"02"; "12"**\dir{-}; "22"**\dir{-};
"0000"; "0010"**\dir{-}; "1010"**\dir{-};
"0001"; "0011"**\dir{-}; "1011"**\dir{-}; "10"**\dir{-};
"0101"; "0111"**\dir{-}; "1111"**\dir{-}; "11"**\dir{-};
"0011"; "00"**\dir{-};
"0111"; "01"**\dir{-};
"000000"*{\bullet}; "0101"*{\bullet}; "1010"*{\bullet}; "1111"*{\bullet};
"02"*{\bullet}; "12"*{\bullet}; "22"*{\bullet}; "21"*{\bullet}; "20"*{\bullet};
(3,-3)="b"; (3,-14.5)**\dir{--}; (24.5,-14.5)**\dir{--}; (24.5,-3)**\dir{--}; "b"**\dir{--};
\end{xy}
\end{equation}

The map $\Crt(d)$ is a morphism of lattices by Lemma \ref{a5le:U4}. Moreover, $\varsigma^n$ preserves coproducts and final objects. In particular, $\varsigma^n(p,q)=\varsigma^n(p,0)\vee \varsigma^n(0,q)$. By Remark \ref{a5re:U3}, the maps $\varsigma^n$ for different $n$ are compatible with $d$ in the sense that we have $\Crt(d)(\varsigma^m(p,q))=\varsigma^n(d(p),d(q))$ for all $(p,q)\in[m]\times[m]$.

By Remark \ref{a5re:U2}, a diagram $F\colon \Cart^n\to \cC$ in an $\infty$-category $\cC$ is a right Kan extension along $\varsigma^n$ if and
only if for every $Q\in \Crt^n$, the restriction of $F$ to $\rN(\varsigma^n(Q))^\triangleleft$ exhibits $F(Q)$ as the limit of $F\res\rN(\varsigma^n(Q))$. By Remark \ref{a5re:U5}, if $F\colon \Cart^n\to\cC$ is a right Kan extension along $\varsigma^n$, then $F\circ\Cart(d)\colon\Cart^m\to \cC$ is a right Kan extension along $\varsigma^m$.

\begin{definition}\label{a5de:cartesian}
Let $\cC$, $\cD$ be $\infty$-categories and let $\tau\colon \Delta^n\times\Delta^n\times \cD\to \cC$ be a functor.  We define $\Kart(\tau)$, the simplicial set of \emph{Cartesianizations of $\tau$}, to be the fiber of the restriction map
\begin{align*}
\Fun(\Cart^n\times\cD,\cC)_\RKE \xrightarrow{\rres}\Fun(\Delta^n\times\Delta^n\times\cD,\cC)
\end{align*}
at $\tau$. Here $\Fun(\Cart^n\times\cD,\cC)_\RKE\subseteq\Fun(\Cart^n\times\cD,\cC)$ is the full subcategory spanned by functors $F\colon\Cart^n\times\cD\to\cC$ that are right Kan extensions of $F\res\Delta^n\times\Delta^n\times\cD$ along $\varsigma^n\times \id_\cD$.
\end{definition}

\begin{remark}
By \cite{HTT}*{Proposition 4.3.2.9}, $\rres$ is the composition
\[
\Fun(\Cart^n\times\cD,\cC)_\RKE\to \cK \hookrightarrow \Fun(\Delta^n\times\Delta^n\times\cD,\cC)
\]
of a trivial Kan fibration with the inclusion of the full subcategory $\cK$ spanned by functors $\tau$ that admit right Kan extensions along
$\varsigma^n\times \id_\cD$. In particular, $\Kart(\tau)$ is a contractible Kan complex if $\tau$ admits a right Kan extension along $\varsigma^n\times\id_\cD$ and $\Kart(\tau)$ is empty otherwise.

If $\cC$ admits pullbacks, then $\rres$ is a trivial Kan fibration. Indeed, in this case, every diagram $\rN(Q)\to \cC$, where $Q\in \Crt^n$, admits a limit by Lemma \ref{a4le:over_pull}.
\end{remark}

The following projection map will play an important role.

\begin{notation}\label{a5no:pi}
Let $n\ge 0$ be an integer. We define a morphism of lattices
\[
\pi^n=(\pi^n_1,\pi^n_2)\colon \Crt^n\to [n]\times[n]
\]
to be the composite of the morphism of lattices $-\vee \xi^n(n,n)\colon\Crt^n\to \Crt^n_{n,n}$, where $\xi^n(n,n)=\varsigma^n(n,0)\wedge\varsigma^n(0,n)$ and $\Crt^n_{n,n}=\Crt^n_{\xi^n(n,n)/}$, and the isomorphism $\Crt^n_{n,n}\simeq[n]\times[n]$ carrying $\xi^n(n,n)_{(p,q)/}=\varsigma^n(p,n)\wedge\varsigma^n(n,q)$ to $(p,q)$. We still write $\pi^n$ for $\rN(\pi^n)$.
\end{notation}

Note that $\varsigma^n$ is a left adjoint of $\pi^n$, hence a section of $\pi^n$. We have the following characterizations of $\pi^n$: for $Q\in\Crt^n$, we have
\begin{align*}
\varsigma^n(\pi^n_1(Q),n)&=Q\vee \varsigma^n(0,n),\\
\varsigma^n(n,\pi^n_2(Q))&=Q\vee\varsigma^n(n,0),\\
\pi^n(Q)&=\(\min_{(p,q)\in Q}p,\min_{(p,q)\in Q} q\).
\end{align*}
The last equation implies that for every order-preserving map $d\colon[m]\to [n]$, we have $\pi^n\circ \Crt(d)=(d\times d)\circ \pi^m$. Indeed, for $Q\in \Crt^m$, we have
\[
\pi^n(\Crt(d)(Q))=\(\min_{(p,q)\in Q} d(p),\min_{(p,q)\in Q} d(q)\)=(d\times d)(\pi^n(Q)).
\]

\begin{lem}\label{a5le:kan}
Let $\cC$ be an $\infty$-category and $F\colon\Delta^n\times\Delta^n\to\cC$ a diagram. The following conditions are equivalent:
\begin{enumerate}
  \item $F$ is obtained from a map of bisimplicial sets $\Delta^{n,n}\to\cC^\cart_{\cC_1,\cC_1}$.

  \item $F$ is a right Kan extension of $F\res \rN(\xi^n(n,n))$.

  \item $F\circ \pi^n\colon \Cart^n\to \cC$ is a right Kan extension along $\varsigma^n$.
\end{enumerate}
\end{lem}

\begin{proof}
By Lemma \ref{a3le:cofinal}, the map $\Lambda^2_0\to\rN(\xi^n(n,n)_{(p,q)/})^{op}$ induced by the square
\[
\xymatrix{
(p,q)\ar[r]\ar[d] & (p,n)\ar[d]\\(n,q)\ar[r] & (n,n)
}
\]
is cofinal. Thus, by \cite{HTT}*{Proposition 4.1.1.8}, (2) is equivalent to the condition that $F$ carries the above square to a pullback. This
condition is a special case of (1), and is equivalent to (1) by \cite{HTT}*{Lemma 4.4.2.1}.

Next we show that (2) implies (3). Assume that $F\circ \pi^n\colon\Cart^n\to \cC$ is a right Kan extension along $\varsigma^n$. Then
$F(p,q)=F(\pi^n(\varsigma^n(p,n)\wedge \varsigma^n(n,q)))$ is a limit of $F\res \rN(\varsigma^n(p,n)\wedge \varsigma^n(n,q))$ by Remark \ref{a5re:U2}. This implies that $F$ is a right Kan extension of $F\res \rN(\xi^n(n,n))$.

Finally we show that (3) implies (2). Assume that $F$ is a right Kan extension of $F\res \rN(\xi^n(n,n))$. Then, for every $Q\in \Crt^n$, the
restriction $F\res \rN(Q)$ is a right Kan extension of $F\res \rN(Q\vee\xi^n(n,n))$. Indeed, for any $(p,q)\in Q$, we have $(Q\vee\xi^n(n,n))_{(p,q)/}=\xi^n(n,n)_{(p,q)/}$. Moreover, the restriction of $F$ exhibits $F(\pi^n(Q))$ as the limit of $F\res Q\vee \xi^n(n,n)$ since $\xi^n(n,n)_{\pi^n(Q)/}=Q\vee \xi^n(n,n)$. It follows that the restriction of $F\circ\pi^n$ exhibits $(F\circ \pi^n)(Q)$ as a limit of $F\circ\pi^n\res \rN(\varsigma^n(Q))$. Therefore, $F\circ \pi^n\colon \Cart^n\to\cC$ is a right Kan extension along $\varsigma^n$ by Remark \ref{a5re:U2}.
\end{proof}

We now introduce a crucial $2$-marking on $\Cart^n$.

\begin{notation}
Let $n\geq 0$ be an integer. We define a $2$-marking $\cF=(\cF_1,\cF_2)$ on $\Cart^n$ as follows. For $i=1,2$, we let $\bar\cF_i$ denote the set of edges of $\epsilon^2_i \Delta^{n,n}$, so that $\delta^*_{2+}\Delta^{n,n}\simeq(\Delta^n\times\Delta^n,\bar\cF_1,\bar \cF_2)$. We define
$\cF_i=(\pi^n)^{-1}(\bar\cF_i)$ for $i=1,2$. Graphically, $\cF_1$ (resp.\ $\cF_2$) consists of edges whose image under $\pi^n$ are vertical (resp.\ horizontal). Recall that $\cF$ induces a $2$-tiling $\cF^\cart$ defined by $\cF^\cart_{12} =\cF_1*_{\Cart^n}^\cart \cF_2$.
\end{notation}

For an order-preserving map $d\colon [m]\to [n]$, the map $\Cart(d)$ induces a map $(\Cart^m,\cF)\to(\Cart^n,\cF)$ of $2$-marked $\infty$-categories, and a map $(\Cart^m,\cF^\cart)\to(\Cart^n,\cF^\cart)$ of $2$-tiled $\infty$-categories.

\begin{construction}\label{a5cs:YZ}
Consider a $(\{1,2\}\coprod K)$-tiled $\infty$-category $(\cC,\cT)$ and a subset $L\subseteq K$. For brevity, we write $I$ for $\{1,2\}\coprod K$. We consider the following two simplicial sets
\begin{align*}
Y^n(\cT)&=\epsilon^I_1\Map(\delta_*^{2+}(\Cart^n,\cF)\boxtimes
\Delta^{n_k\res k\in K}_L,\delta_*^{I\square} (\cC,\cT)),\\
Z^n(\cT)&=\epsilon^I_1\Map(\delta_*^{2\square}(\Cart^n,\cF^\cart)\boxtimes
\Delta^{n_k\res k\in K}_L,\delta_*^{I\square} (\cC,\cT)).
\end{align*}
We have a natural commutative diagram
\[
\xymatrix{
& \Fun(\delta^*_2\delta_*^2\Cart^n\times\Delta^{[n_k]_{k\in K}}_L,\cC) \ar[d]
& \Fun(\Cart^n\times\Delta^{[n_k]_{k\in K}}_L,\cC)\ar[l]_-{f^n}\ar@/^1pc/[ld]_-{g^n}\ar@/^2pc/[ldd]_-{h^n} \\
Y^n(\cT) \ar@{^(->}[r]\ar[d] & \Fun(\delta^*_2\delta_*^{2+}(\Cart^n,\cF)\times\Delta^{[n_k]_{k\in K}}_L,\cC) \ar[d] \\
Z^n(\cT) \ar@{^(->}[r] & \Fun(\delta^*_2\delta_*^{2\square}(\Cart^n,\cF^\cart)\times\Delta^{[n_k]_{k\in K}}_L,\cC),
}
\]
where
\begin{itemize}
  \item The vertical arrows are induced by the inclusions
      \[
      \delta_*^{2\square}(\Cart^n,\cF^\cart)\subseteq\delta_*^{2+}(\Cart^n,\cF)\subseteq\delta_*^2\Cart^n;
      \]

  \item $f^n$ is induced by the adjunction $\delta^*_2\delta_*^2\Cart^n\to \Cart^n$;

  \item $g^n$ and $h^n$ are compositions of $f^n$ and the vertical arrows;

  \item In the inclusion on the second row we have used the isomorphism
      \[
      \epsilon^I_1\Map(\delta_*^{2+}(\Cart^n,\cF)\boxtimes\Delta^{n_k\res k\in K}_L ,\delta_*^I\cC)\simeq
      \Fun(\delta^*_2\delta_*^{2+}(\Cart^n,\cF)\times \Delta^{[n_k]_{k\in K}}_L,\cC)
      \]
      in Remark \ref{a3re:adjunction} and similarly for the inclusion on the third row.
\end{itemize}
Moreover, we have a commutative diagram
\[
\xymatrix{
Y^n(\cT) \ar[r]^-{y^n(\cT)}\ar[d] & \Map(\delta^*_2\delta_*^{2+}(\Cart^n,\cF)\times\Delta^{[n_k]_{k\in K}},
\delta^*_{I,L}\delta_*^{I\square}(\cC,\cT)) \ar[d] \\
Z^n(\cT) \ar[r]^-{z^n(\cT)} & \Map(\delta^*_2\delta_*^{2\square}(\Cart^n,\cF^\cart)\times\Delta^{[n_k]_{k\in K}},
\delta^*_{I,L}\delta_*^{I\square} (\cC,\cT)),
}
\]
where the vertical arrows are induced by the inclusion
\[
\delta_*^{2\square}(\Cart^n,\cF^\cart)\subseteq\delta_*^{2+}(\Cart^n,\cF)
\]
and the horizontal arrows are induced by $\delta^*_I$ in Remark \ref{a3re:adjunction}. In the above notation, we have kept the datum $\cT$ as we will now let it vary.
\end{construction}

\begin{lem}\label{a5le:restriction}
Assume that we are in the situation of Theorem \ref{a5th:cartesian_gluing}. Let $\cE^i\subseteq \cE$ be the subset of $i$-truncated edges, and let $\cT^i$ be the $(\{1,2\}\coprod K)$-tiling between $\cT$ and $\cT'$ determined by $\cT^i_{12}=\cT_1*_{\cC}^{\cE^i}\cT_2$. Then, for every map $\tau\colon\Delta^{n,n,n_k\res k\in K}_L\to \delta_*^{(\{1,2\}\amalg K)\square}(\cC,\cT^i)$, the simplicial set $\Kart(\tau)$ is a contractible Kan complex and the restriction of the map $g^n$ (resp.\ $h^n$) to $\Kart(\tau)\subseteq\Fun(\Cart^n\times\Delta^{[n_k]_{k\in K}}_L,\cC)$ has image contained in $Y^n(\cT^i)$ (resp.\ $Z^n(\cT^{i-1})$ for $i\ge -1$). Here $\Kart(\tau)$ is the simplicial set of Cartesianizations of $\tau$ (where $\tau$ is regarded as a functor $\Delta^n\times\Delta^n\times\Delta^{[n_k]_{k\in K}}_L\to\cC$) in Definition
\ref{a5de:cartesian}.
\end{lem}

In particular, we have induced maps
\begin{align*}
g(\tau)&\coloneqq y^n(\cT^i)\circ g^n  \colon\Kart(\tau) \\
& \to\Map(\delta^*_2\delta_*^{2+}(\Cart^n,\cF)\times\Delta^{[n_k]_{k\in K}},
\delta^*_{\{1,2\}\amalg K,L}\delta_*^{(\{1,2\}\amalg K)\square} (\cC,\cT^i)),
\end{align*}
\begin{align*}
h(\tau)&\coloneqq z^n(\cT^{i-1})\circ h^n \colon\Kart(\tau) \\
& \to\Map(\delta^*_2\delta_*^{2\square}(\Cart^n,\cF^\cart)\times\Delta^{[n_k]_{k\in K}},
\delta^*_{\{1,2\}\amalg K,L}\delta_*^{(\{1,2\}\amalg K)\square} (\cC,\cT^{i-1})).
\end{align*}

\begin{proof}
Consider an equivalence $e$ in the $\infty$-category
\begin{equation}\label{a5eq:restr}
\Fun(\delta^*_2\delta_*^{2+}(\Cart^n,\cF)\times\Delta^{[n_k]_{k\in K}}_L,\cC)
\end{equation}
with one vertex in $Y^n(\cT^i)$. By Remark \ref{a5re:explicit} (1), we know that the other vertex is also in $Y^n(\cT^i)$. Moreover, we have
$\cE^{-2}*^\cart_\cC\cT_\alpha\subseteq \cT^i_{1\alpha}$ for $\alpha\in\{2\}\coprod K$. It follows that $e$ is in $Y^n(\cT^i)$. Thus, for any
connected Kan complex $S$ contained in \eqref{a5eq:restr}, either $Y^n(\cT^i)\cap S=\emptyset$ or $S\subseteq Y^n(\cT^i)$. The same holds for
$Z^n(\cT^{i-1})$. As $\Kart(\tau)$ is either empty or a contractible Kan complex, its images in $\infty$-categories are contained in connected Kan complexes. Therefore, it suffices to find one vertex $F$ of $\Kart(\tau)$ satisfying $g^n(F)\in Y^n(\cT^i)$ and $h^n(F)\in Z^n(\cT^{i-1})$.

For clarity, let $G\colon\Delta^{[n,n,n_k]_{k\in K}}\to\cC$ be the functor corresponding to $\tau$ (we have till now denoted $G$ by $\tau$). Note that $G$ underlies a map of $(\{1,2\}\coprod K)$-tiled simplicial sets $\delta^*_{I\square}\Delta^{n,n,n_k\res k\in K}_L\to (\cC,\cT^i)$. We let $\bar\cG_\alpha$ denote the set of edges of $\epsilon^I_\alpha\Delta^{n,n,n_k\res k\in K}_L$. Then there are isomorphisms
\begin{align*}
\delta^*_{I\square}\Delta^{n,n,n_k\res k\in K}_L &\simeq \sfW\delta^*_{I+}\Delta^{n,n,n_k\res k\in K}_L,\\
\delta^*_{I+}\Delta^{n,n,n_k\res k\in K}_L &\simeq (\Delta^{[n,n,n_k]_{k\in K}}_L,\{\bar\cG_\alpha\}_{\alpha\in I}).
\end{align*}
We define an $I$-marked simplicial set $(\Cart^n\times \Delta^{[n_k]_{k\in K}}_L,\{\cG_\alpha\}_{\alpha\in I})$ by $\cG_\alpha=(\varsigma^n\times\id)^{-1}\bar \cG_\alpha$. The goal is to show that $G$ admits a right Kan extension $F\colon\Cart^n\times\Delta^{[n_k]_{k\in K}}_L\to\cC$ along $\varsigma^n\times \id$ such that $F$ sends squares in $\cG_\alpha*\cG_\beta$ to squares in $\cT^i_{\alpha\beta}$ for $\alpha,\beta\in I$, $\alpha\neq \beta$, and, for $i\ge -1$, $F$ sends squares in $\cG_1*^\cart \cG_2$ to squares in $\cT^{i-1}_{12}$.

Let us first show that there exists a right Kan extension $F$ of $G$ along $\varsigma^n\times \id$ such that for each vertex $(x,u)$ of $\Cart^n\times\Delta^{[n_k]_{k\in K}}_L$, the morphism $G(\pi^n(x),u)=F(\varsigma^n(\pi^n(x)),u)\to F(x,u)$ is in $\cE^i$. We construct the restriction of $F$ to $\Cart^n_{\varsigma^n(p,0)/}\times\Delta^{[n_k]_{k\in K}}_L$ by descending induction on $p$. In the case $p=n$, $\Crt^n_{\varsigma^n(n,0)/}$ is contained in the image of $\varsigma^n$ and there is nothing to prove. For $0\le p\le n-1$, and $x\in \Crt^n$ satisfying $\pi^n(x)=(p,q)$, consider the commutative diagram
\begin{equation}\label{a5eq:squares}
\xymatrix{
\varsigma^n(p,q)\ar[r]\ar[d] & x\ar[r]\ar[d]& \varsigma^n(p,q')\ar[d]\\
\varsigma^n(p+1,q)\ar[r] & x\vee \varsigma^n(p+1,q)\ar[r] & \varsigma^n(p+1,q'),
}
\end{equation}
where $q'=\min \{q_0\mid (p,q_0)\in x\}$. The right square is exact. The vertical (resp.\ horizontal) arrows are in $\cF_1$ (resp.\ $\cF_2$). The horizontal arrows in the left square are in $\cF_1\cap \cF_2$. By induction hypothesis, the morphism $G(p+1,q,u)\to F(x\vee\varsigma^n(p+1,q),u)$ is in $\cE^i$, so that $G(p,q,u)\to F(x\vee \varsigma^n(p+1,q),u)$ is in $\cT_1$, since $\cT_1$ is stable under composition. Thus, by the assumption $\cT_1*_\cC\cT_2=\cT_1*_\cC^{\cC_1}\cT_2$, the pullback $F(x\vee\varsigma^n(p+1,q),u)\times_{G(p+1,q',u)}G(p,q',u)$ exists in $\cC$, which provides $F(x,u)$ by the proof of Lemma \ref{a5le:U7}. The morphism $G(p,q,u)\to F(x,u)$ is the composition
\[
G(p,q,u)\to G(p+1,q,u)\times_{G(p+1,q',u)} G(p,q',u)\to F(x,u),
\]
where the first arrow is in $\cE^i$ by the assumption that $G$ carries $\bar\cG_1*\bar\cG_2$ into $\cT^i_{12}=\cT_1*_\cC^{\cE^i}\cT_2$, and the
second arrow is in $\cE^i$ by the assumption that $\cE^i$ is stable under pullback by $\cT_1$.

We claim that $F$ sends $\cG_1$ to $\cT_1$, $\cG_2$ to $\cT_2$, and $\cG_1\cap\cG_2$ to $\cE^i$. Let $e\colon (x,u)\to (y,u)$ be an edge in
$\cG_1\cup \cG_2$, where $x\to y$ in $\cF_1\cup \cF_2$ and $u$ is a vertex of $\Delta^{[n_k]_{k\in K}}$. We show by induction on $\# x$ that $F(e)\in\cT_1$ for $e\in \cG_1$, $F(e)\in \cT_2$ for $e\in \cG_2$, and $F(e)\in\cE^i$ for $\in \cG_1\cap \cG_2$. By Lemma \ref{a5le:exact}, any morphism $x\to y$ in $\Crt^n$ is a composition of a finite sequences of morphisms of the following classes:
\begin{enumerate}
  \item An exact pullback  of $\omega^n(p,n)\colon \varsigma^n(p,n)\to\varsigma^n(p+1,n)$ by $c\in \cF_2$;

  \item An exact pullback of $\omega^n(n,q)\colon \varsigma^n(n,q)\to\varsigma^n(n,q+1)$ by $c\in \cF_1$;

  \item An exact pullback of $\omega^n(p,q)\colon \varsigma^n(p,q)\to\varsigma^n(p,q)-\{(p,q)\}$ by $c$,
\end{enumerate}
where we have $(p,q)\in [n-1]\times [n-1]$, and $c\colon x'\to y'$ satisfies $\# x'<\# x$. If $e\in \cG_1$ (resp.\ $e\in\cG_2$), then class (2) (resp.\ (1)) does not appear. Since $\cT_1$, $\cT_2$ and $\cE^i$ are stable under composition, we may assume that $x\to y$ is in one of the three classes. In class (1), $\omega^n(p,n)$ is in $\varsigma^n(\bar\cF_1)$, and we conclude by Lemma \ref{a5le:U7} and the assumption that $\cT_1$ is stable under pullback by $\cT_2$. In class (2), $\omega^n(n,q)$ is in $\varsigma^n(\bar\cF_2)$, and we conclude by Lemma \ref{a5le:U7} and the assumption that $\cT_2$ is stable under pullback by $\cT_1$. In class (3), $c$ is a composition of an edge in $\cF_1$ and an edge in $\cF_2$ both satisfying the induction hypothesis, and we have a diagram
\[
\xymatrix{
\varsigma^n(p,q)\ar[rd]^{\omega^n(p,q)}\\
&\varsigma^n(p,q)-\{(p,q)\} \ar[r]\ar[d] & \varsigma^n(p,q+1) \ar[d] \\
&\varsigma^n(p+1,q) \ar[r] & \varsigma^n(p+1,q+1)
}
\]
with exact square in $\Crt^n$. By Lemma \ref{a5le:U7}, the morphism $F(\omega^n(p,q)\times \id_u)$ can be identified with the induced morphism
\[
G(p,q,u)\to G(p+1,q,u)\times_{G(p+1,q+1,u)} G(p,q+1,u),
\]
which belongs to $\cE^i$ since $G$ carries $\bar\cG_1*\bar \cG_2$ into $\cT^i_{12}=\cT_1*_\cC^{\cE^i}\cT_2$. We conclude by the assumption that
$\cE^i$ is stable under pullback by $\cT_1\cup \cT_2$. This finishes the proof of the claim.

Similarly, applying Condition (3), we see that $F$ carries $\cG_1*\cG_k$ into $\cT_{1k}$ and $\cG_2*\cG_k$ into $\cT_{2k}$ for all $k\in K$.

Next we show that $F$ carries squares in $\cG_k*\cG_l$ into $\cT_{kl}$ for all $k,l\in K$, $k\neq l$. Consider such a square and let $x\in\Crt^n$ be its projection. For $x$ in the image of $\varsigma^n$, this follows from the assumption that $G$ carries $\bar\cG_k*\bar\cG_l$ to $\cT_{kl}$. For the general case, we proceed by descending induction on $\pi_1(x)$. If $\pi(x)=(n,q)$, then $x=\varsigma^n(n,q)$ is in the image of $\varsigma^n$. For $\pi(x)=(p,q)$ with $p<n$, we consider the right square of \eqref{a5eq:squares}. We conclude by Condition (4) and the induction hypothesis applied to $x\vee \varsigma^n(p+1,q)$.

Finally we show that $F$ carries $\cG_1*\cG_2$ into $\cT^i_{12}$ and carries $\cG_1*^\cart \cG_2$ into $\cT^{i-1}_{12}$. Every square in
$\cF_1*_{\Cart^n} \cF_2$ of the form \eqref{a5eq:square} has a canonical decomposition
\[
\xymatrix{
w\ar[rd] \\ &y\wedge z\ar[r]\ar[d] & y\ar[d]\\&z \ar[r] & y\vee z\ar[rd] \\ &&& x,
}
\]
where the vertical (resp.\ horizontal) arrows are in $\cF_1$ (resp.\ $\cF_2$), and oblique arrows are in $\cF_1\cap \cF_2$. Note that $\cF^\cart_{12}$ is the set of squares such that $w=y\wedge z$. Multiplying by $\id_u$ and applying $F$, we obtain a similar diagram where the inner square is a pullback by Lemma \ref{a5le:U7} and the oblique arrows are in $\cE^i$ by the previous claim. Since we have already proved that $F$ carries $\cG_1*\cG_2$ into $\cT_1*_\cC\cT_2$, all we need to show is that the induced morphism $F(y\wedge z,u)\to F(y,u)\times_{F(x,u)}F(z,u)$ belongs to $\cE^{i-1}$. However, by Lemma \ref{a5le:pullback}, this morphism can be identified with the left vertical arrow of the pullback square
\[
\xymatrix{
F(y,u)\times_{F(x',u)}F(z,u) \ar[r]\ar[d]& F(x',u) \ar[d] \\
F(y,u)\times_{F(x,u)}F(z,u) \ar[r]& F(x',u)\times_{F(x,u)}F(x',u),
}
\]
where for brevity we have written $x'$ for $y\vee z$, and the right vertical arrow is the diagonal of $F(x',u)\to F(x,u)$ and hence belongs to
$\cE^{i-1}$. The lower horizontal arrow is a composition of a pullback of a morphism in $\cT_1$ by a morphism in $\cT_2$ and a pullback of a morphism in $\cT_2$ by a morphisms in $\cT_1$. Since $\cE^{i-1}$ is stable under pullback by $\cT_1\cup \cT_2$, the left vertical arrow belongs to $\cE^{i-1}$ as well.
\end{proof}

The functor $g^n$ in Construction \ref{a5cs:YZ} is induced by the map
\begin{equation}\label{a5eq:gn}
\delta_2^*\delta^{2+}_* (\Cart^n,\cF)\to \Cart^n,
\end{equation}
which carries a square in $\cF_1*_{\Cart^n}\cF_2$ to its diagonal. We now construct a family of sections of this map.

\begin{construction}\label{a5co:section}
Let $n\geq 0$ be an integer.
\begin{enumerate}
  \item For $x\le y$ in $\Crt^n$ and $(p,q)$ in $[n]\times [n]$, we define two elements of $\Crt^n_{x//y}$:
      \[
      \lambda^n_p(x,y)=(\varsigma^n(\pi^n_1(y)\vee p,0)\vee x)\wedge y,\qquad
      \mu^n_q(x,y)=(\varsigma^n(0,\pi^n_2(y)\vee q)\vee x)\wedge y.
      \]
      These formulas are increasing in $p$, $q$, $x$ and $y$. Moreover, we have the properties
      \begin{align}\label{a5eq:lambdamu1}
      \lambda^n_p(x,x)=\mu^n_q(x,x)=x,
      \end{align}
      and
      \begin{align}\label{a5eq:lambdamu2}
      \pi^n(\lambda^n_p(x,y))=(\pi^n_1(y),\pi^n_2(x)),\quad \pi^n(\mu^n_q(x,y))=(\pi^n_1(x),\pi^n_2(y)).
      \end{align}

  \item We construct a map
      \[
      \alpha^n\colon A^n\coloneqq (\Delta^n\times \Delta^n)^\sharp \times (\Cart^n)^\flat\to(\delta^*_2\delta_*^{2+}(\Cart^n,\cF))^\flat
      \]
      as follows. For an $m$-simplex $\tau=(\tau_1,\tau_2,\tau_3)\colon\Delta^m\to \Delta^n\times \Delta^n\times \Cart^n$, we define $\alpha^n(\tau)$ to be the map $\Delta^m\times \Delta^m\to \Cart^n$ carrying $(a,b)$ to $\lambda^n_{\tau_1(b)}(\tau_3(b),\tau_3(a))$ for $a\ge b$, and to $\mu^n_{\tau_2(a)}(\tau_3(a),\tau_3(b))$ for $a\le b$. By \eqref{a5eq:lambdamu1}, the two definitions coincide for
      $a=b$. By \eqref{a5eq:lambdamu2}, $\alpha^n(\tau)$ is an $m$-simplex of $\delta^*_2\delta_*^{2+}(\Cart^n,\cF)$. In particular, $\alpha^n$ carries an edge $(p,q,x)\to (p',q',y)$ of $\Delta^n\times\Delta^n\times \Cart^n$ to the square
      \begin{equation}\label{a5eq:lambdamusquare}
      \xymatrix{x\ar[r]\ar[d] & \mu^n_q(x,y)\ar[d]\\ \lambda^n_p(x,y)\ar[r] &y}
      \end{equation}
      in $\cF_1*_{\Cart^n}\cF_2$. By \eqref{a5eq:lambdamu1}, $\alpha^n$ carries marked edges of $A^n$ to degenerate edges. The composition
      \[
      A^n\xrightarrow{\alpha^n}(\delta^*_2\delta_*^{2+}(\Cart^n,\cF))^\flat \to (\Cart^n)^\flat,
      \]
      where the second map is \eqref{a5eq:gn}, is the projection.
\end{enumerate}
\end{construction}

\begin{remark}
For an order-preserving map $d\colon[m]\to[n]$, we have identities
\begin{align*}
\Crt(d)(\lambda^m_p(x,y))&=\lambda^n_{d(p)}(\Crt(d)(x),\Crt(d)(y)), \\
\qquad\Crt(d)(\mu^m_q(x,y))&=\mu^n_{d(q)}(\Crt(d)(x),\Crt(d)(y)).
\end{align*}
Thus the maps $\alpha^n$ for different $n$ are compatible with $\Cart(d)$ in the obvious sense.
\end{remark}

Next we define a restriction of $\alpha^n$, taking values in $\delta_2^*\delta^{2\square}_*(\Cart^n,\cF^\cart)$.

\begin{construction}\label{a5cs:boxplus}
Let $n\geq 0$ be an integer.
\begin{enumerate}
  \item We define order-preserving maps
      \[
      \xi^n,\eta^n\colon [n]\times[n]\to\Crt^n
      \]
      by
      \[
      \xi^n(p,q)=\varsigma^n(p,0)\wedge \varsigma^n(0,q),\quad \eta^n(p,q)=\varsigma^n(p,n)\wedge\varsigma^n(n,q).
      \]
      We have $\xi^n(p,q)\le \varsigma^n(p,q)\le \eta^n(p,q)$. We define a sublattice of $\Crt^n$ by
      \[
      \Crt^n_{p,q}\coloneqq\Crt^n_{\xi^n(p,q)//\eta^n(p,q)}
      \]
      and we put $\boxplus^n_{p,q}\coloneqq\rN(\Crt^n_{p,q})$. We put
      \[
      \boxplus^n\coloneqq\bigcup_{0\le p, q\le n}\boxplus^n_{p,q}\subseteq \Cart^n.
      \]
      Note that $\eta^n$ induces an isomorphism of lattices $[n]\times [n]\simeq\Crt^n_{n,n}=\Crt^n_{\xi^n(n,n)/}$ via which $\pi^n\colon\Crt^n\to[n]\times[n]$ can be identified with the morphism of lattices $-\vee \xi^n(n,n)\colon\Crt^n\to\Crt^n_{n,n}$.

  \item We define a marked simplicial subset $B^n$ of $A^n$ by
      \[
      B^n=\bigcup_{x\le y} \rN(I_{x,y})^\sharp \times (\Cart^n_{x//y})^\flat
      \subseteq (\Delta^n\times \Delta^n)^\sharp\times (\boxplus^n)^\flat\subseteq A^n.
      \]
      Here $x$ and $y$ run over elements of $\Crt^n$ and $I_{x,y}\subseteq[n]\times[n]$ denote the full subcategory spanned by pairs $(p,q)$ satisfying
      \begin{equation}\label{a5eq:Bn}
      \xi^n(p,q)\le x\le y\le \eta^n(p,q),
      \end{equation}
      or, equivalently, satisfying $\Cart^n_{x//y}\subseteq\boxplus^n_{p,q}$. We note that $\eta^n$ is a right adjoint of $\pi^n$: $y\le\eta^n(p,q)$ if and only if $\pi^n(y)\le (p,q)$.
\end{enumerate}
\end{construction}

We refer the reader to \eqref{a5eq:Hasse} for graphic depictions of $\Crt^n_{p,q}$ for some small values of $n$, $p$, $q$.

\begin{remark}
Let $d\colon [m]\to [n]$ be an order-preserving map. For $0\leq p,q\leq m$, we have
\begin{align*}
\Crt(d)(\xi^m(p,q))&=\varsigma^n(d(p),d(0))\wedge\varsigma^n(d(0),d(q))\ge \xi^n(d(p),d(q)),\\
\Crt(d)(\eta^m(p,q))&=\varsigma^n(d(p),d(m))\wedge\varsigma^n(d(m),d(q))\le \eta^n(d(p),d(q)).
\end{align*}
Thus $\Crt(d)$ induces morphisms of lattices $\Crt^m_{p,q}\to\Crt^n_{d(p),d(q)}$ and hence maps $\boxplus^m_{p,q}\to\boxplus^n_{d(p),d(q)}$ and $B^m\to B^n$.
\end{remark}

\begin{lem}
The map $\alpha^n$ induces a map
\[
\beta^n\colon B^n\to(\delta^*_2\delta_*^{2\square}(\Cart^n,\cF^\cart))^\flat.
\]
\end{lem}

\begin{proof}
It suffices to show that for every $m$-simplex $\tau$ of the underlying simplicial set of $B^n$, the diagram $\alpha^n(\tau)\colon\Delta^m\times\Delta^m\to \Cart^n$ carries the square spanned by the vertices $(a,b)$, $(a+1,b)$, $(a,b+1)$, $(a+1,b+1)$ to a pullback. For $a=b$, the assertion amounts to saying that for every edge $(p,q,x)\le (p',q',y)$ of $B^n$, the square \eqref{a5eq:lambdamusquare} is a pullback. We have
\[
\lambda^n_p(x,y)\wedge\mu^n_q(x,y)=(\xi^n(\pi^n(y)\vee (p,q))\vee x)\wedge y,
\]
which equals $x$ by the assumption $\xi^n(p,q)\le x$. For $a>b$, the assertion amounts to saying that for every $3$-simplex $(p,q,x)\le
(p',q',y)\le (p'',q'',z)\le (p''',q''',w)$ of $B^n$, the square
\[
\xymatrix{
\lambda^n_p(x,z)\ar[r]\ar[d] & \lambda^n_{p'}(y,z)\ar[d]\\
\lambda^n_p(x,w)\ar[r] &\lambda^n_{p'}(y,w)
}
\]
is a pullback. This is clear since $\lambda^n_p(x,z)=(\varsigma^n(p,0)\vee x)\wedge z$ by the assumption $\pi_1(z)\le p$ and similarly for the other vertices of the squares. The case $a<b$ is similar, with $\lambda^n_p$ replaced by $\mu^n_q$.
\end{proof}

\begin{remark}
The proof shows in fact that $\alpha^n$ carries the the simplicial subset $S\subseteq \Delta^n\times \Delta^n\times\Cart^n$ spanned by those edges $(p,q,x)\to (p',q',y)$ satisfying \eqref{a5eq:Bn} (with no restrictions on $(p',q')$) into $\delta^*_2\delta_*^{2\square}(\Cart^n,\cF^\cart)$. Note that $S$ is bigger than the underlying simplicial set of $B^n$ for $n\ge 1$.
\end{remark}

\begin{lem}\label{a5le:trivial_co}
The inclusion $B^n\subseteq A^n$ is a trivial cofibration in the category $\Mset$ for the Cartesian model structure.
\end{lem}

\begin{proof}
Choose an exhaustion of $\boxplus^n$ by a sequence of simplicial subsets
\[
\emptyset=K^0\subseteq K^1\subseteq \dots \subseteq K^N=\boxplus^n
\]
such that each $K^i$, $1\le i\le N$ is obtained from $K^{i-1}$ by adjoining a single nondegenerate simplex $\sigma^i\colon \Delta^{l_i}\to K^i$. This induces inclusions
\[
B^n=L^0\subseteq L^1\subseteq \dots\subseteq L^N=(\Delta^n\times \Delta^n)^\sharp\times(\boxplus^n)^\flat,
\]
where $L^i=B^n\cup ((\Delta^n\times \Delta^n)^\sharp\times(K^i)^\flat)$. By Lemma \ref{a6le:cart_inner}, $(\boxplus^n)^\flat\subseteq(\Cart^n)^\flat$ is a trivial cofibration in $\Mset$, so that $(\Delta^n\times\Delta^n)^\sharp\times(\boxplus^n)^\flat\subseteq A^n$ is a trivial cofibration in $\Mset$ by \cite{HTT}*{Corollary 3.1.4.3}.
Therefore, it suffices to show that the inclusion $L^{i-1}\subseteq L^i$ is a trivial cofibration in $\Mset$ for all $1\le i\le N$. However, this inclusion is a pushout of the map
\[
(\Delta^n\times \Delta^n)^\sharp \times (\partial \Delta^{l_i})^\flat\coprod_{\rN(I_{x,y})^\sharp\times(\partial \Delta^{l_i})^\flat}
\rN(I_{x,y})^\sharp\times(\Delta^{l_i})^\flat\to (\Delta^n\times \Delta^n)^\sharp\times (\Delta^{l_i})^\flat,
\]
where $x=\sigma^i(0)$, $y=\sigma^i(l_i)$. By the assumption that $\sigma^i$ is a simplex of $\boxplus^n$, the partially ordered set $I_{x,y}$ is nonempty, and admits an initial object $\pi^n(y)$. Thus the inclusion $\rN(I_{x,y})\subseteq \Delta^n\times \Delta^n$ is anodyne. It follows that the inclusion $\rN(I_{x,y})^\sharp\subseteq (\Delta^n\times \Delta^n)^\sharp$ is a trivial cofibration in $\Mset$ (by Remark \ref{a3re:Quillen}), and so is its smash product with $(\partial\Delta^{l_i})^\flat\subseteq(\Delta^{l_i})^\flat$ by \cite{HTT}*{Corollary 3.1.4.3}.
\end{proof}

\begin{proof}[Proof of Theorem \ref{a5th:cartesian_gluing}]
We adopt the notation of Lemma \ref{a5le:restriction}. By the first part of Condition (2), we have $\cE=\bigcup_{i\ge -2} \cE^i$, $\cT'=\bigcup_{i\ge-2}\cT^i$, and
\[
W_\infty\coloneqq \delta^*_{\{1,2\}\amalg K,L}\delta_*^{(\{1,2\}\amalg K)\square} (\cC,\cT')=\bigcup_{i\ge -2} W_i,
\]
where $W_i=\delta^*_{\{1,2\}\amalg K,L}\delta_*^{(\{1,2\}\amalg K)\square}(\cC,\cT^i)$. Since $\cE^{-2}$ is the set of equivalences of $\cC$, we have $\cT^{-2}=\cT$. Thus, the map $\iota$ in question is the transfinite composition of inclusions
\[
W_{-2}\to W_{-1}\to \dots \to W_i\to \dots \to W_\infty.
\]
Since the Joyal model structure on $\Sset$ is combinatorial, the trivial cofibrations form a weakly saturated class \cite{HTT}*{Definition A.1.2.2}. Thus it suffices to show that each inclusion $W_{-2}\to W_i$ is a categorical equivalence for every integer $i\geq -1$. By Lemma
\ref{a1le:categorical_equivalence} and induction, it suffices to show that for every $i \ge -1$ and every commutative diagram
\[
\xymatrix{
W_{-2}\ar[r]^{f'} &W_{i-1}\ar[r]^-v\ar[d]_f & \Fun(\Delta^l,\cD)\ar[d]^p\\
&W_i\ar[r]^-w \ar@{..>}[ur]^-{u} & \Fun(\partial \Delta^l,\cD)
}
\]
where $f$ and $f'$ are inclusions and $p$ is induced by the inclusion $\partial \Delta^l\subseteq \Delta^l$, there exists a map $u\colon W_i\to\Fun(\Delta^l,\cD)$ satisfying $p\circ u=w$ such that $u\circ f\circ f'$ and $v\circ f'$ are homotopic over $\Fun(\partial \Delta^l,\cD)$. The proof is mostly parallel to the proof of Theorem \ref{a4th:multisimplicial_descent}.

Let $\sigma$ be an $n$-simplex of $W_i$, corresponding to a map
\[
\tau\colon\Delta^{n,n,n_k\res k\in K}_L\to \delta^{(\{1,2\}\amalg K)\square}_*(\cC,\cT^i),
\]
where $n_k=n$. We consider the maps
\begin{align*}
w_* g(\tau)&\colon \Kart(\tau)\to \Fun(\delta_2^*\delta^{2+}_*(\Cart^n,\cF)\times \Delta^{[n_k]_{k\in K}},\Fun(\partial \Delta^l,\cD)),\\
v_* h(\tau)&\colon \Kart(\tau)\to \Fun(\delta_2^*\delta^{2\square}_*(\Cart^n,\cF^\cart)\times \Delta^{[n_k]_{k\in K}},\Fun(\Delta^l,\cD)),
\end{align*}
compositions of the maps $g(\tau)$ and $h(\tau)$ defined after the statement of Lemma \ref{a5le:restriction} and the maps induced by $w$ and $v$, respectively. Since $\Kart(\tau)$ is a contractible Kan complex, the maps $w_* g(\tau)$ and $v_*h(\tau)$ factorize through
\begin{align*}
w_* g(\tau)&\colon \Kart(\tau)\to \Map^\sharp((\partial \Delta^l)^\flat\times
(\delta_2^*\delta^{2+}_*(\Cart^n,\cF))^\flat\times (\Delta^{[n_k]_{k\in K}})^\flat,\cD^\natural),\\
v_* h(\tau)&\colon \Kart(\tau)\to \Map^\sharp((\Delta^l)^\flat\times
(\delta_2^*\delta^{2\square}_*(\Cart^n,\cF^\cart))^\flat\times(\Delta^{[n_k]_{k\in K}})^\flat,\cD^\natural),
\end{align*}
respectively. Composing with $\beta^n$ and $\alpha^n$, respectively, we obtain maps
\begin{align*}
\psi(\tau)&\colon \Kart(\tau)\to\Map^\sharp((\partial \Delta^l)^\flat\times A^n\times (\Delta^{[n_k]_{k\in
K}})^\flat,\cD^\natural),\\
\phi(\tau)&\colon \Kart(\tau)\to\Map^\sharp((\Delta^l)^\flat\times B^n\times
(\Delta^{[n_k]_{k\in K}})^\flat,\cD^\natural).
\end{align*}

Consider the commutative diagram
\begin{align}\label{a5eq:descent}
\xymatrix{\cN(\sigma)\ar[d]\ar[r]&\Kart(\tau)\ar[d]^-{h} \\
\Map^\sharp((\Delta^l)^\flat\times A^n\times(\Delta^{[n_k]_{k\in K}})^\flat,\cD^\natural)
\ar[r]^-{\rres_1}\ar[dd]^-{\rres_2} &\Map^\sharp(H\times(\Delta^{[n_k]_{k\in K}})^\flat,\cD^\natural)\ar[d]\\
&\Map^\sharp((\partial \Delta^l)^\flat\times A^n\times (\Delta^{[n_k]_{k\in K}})^\flat,\cD^\natural)\ar[d]^-{\rres_2} \\
\Map^\sharp((\Delta^l)^\flat\times (\Delta^n)^\flat,\cD^\natural)\ar[r]
&\Map^\sharp((\partial\Delta^l)^\flat\times(\Delta^n)^\flat,\cD^\natural).}
\end{align}
In the above diagram,
\begin{itemize}
  \item $\rres_1$ is induce by
     \[
     j\colon H=(\Delta^l)^\flat\times B^n
     \coprod_{(\partial\Delta^l)^\flat \times B^n } (\partial \Delta^l)^\flat\times A^n \hookrightarrow (\Delta^l)^\flat \times A^n;
     \]

  \item $h$ is the amalgamation of $\phi(\tau)$ and $\psi(\tau)$;

  \item $\cN(\sigma)$ is defined so that the upper square is a pullback square;

  \item the two maps $\rres_2$ are both induced by the composite embedding
      \begin{align*}
      \Delta^n &\xrightarrow{\r{diag}} \Delta^n\times \Delta^n\times \Delta^{n}\times \Delta^n\times
      \Delta^{[n_k]_{k\in K}} \\
      &\xrightarrow{\id_{\Delta^n\times \Delta^n}\times \varsigma^n\times \id_{\Delta^{[n_k]_{k\in K}}}}
      \Delta^n\times \Delta^n\times \Cart^n\times \Delta^{[n_k]_{k\in K}};
      \end{align*}

  \item the unmarked arrows in the lower square are obvious restrictions.
\end{itemize}

By Lemma \ref{a5le:trivial_co} and \cite{HTT}*{Corollary 3.1.4.3}, the map $j\times\id_{(\Delta^{[n_k]_{k\in K}})^\flat}$ is a trivial cofibration in $\Mset$ and consequently $\rres_1$ is a trivial Kan fibration. Thus $\cN(\sigma)$ is a contractible Kan complex.

We let $\Phi(\sigma)\colon \cN(\sigma)\to \Map^\sharp((\Delta^n)^\flat,\Fun(\Delta^l,\cD)^\natural)$ denote the composition of the vertical arrows in the first column of \eqref{a5eq:descent}. This construction is functorial in $\sigma$, giving rise to a morphism $\Phi\colon\cN\to\Map[W_i,\Fun(\Delta^l,\cD)]$ in $(\Sset)^{(\del_{/W_i})^{op}}$. The composition of the vertical arrows in the second column of \eqref{a5eq:descent} is constant of value $w(\sigma)$. Thus $\Map[W_i,p]\circ\Phi$ factors through the morphism $\Delta^0_{(\del_{/W_i})^{op}}$
corresponding to $w$ via Remark \ref{a2re:functors}.

Let $\sigma'$ be an $n$-simplex of $W_{-2}$ corresponding to a map
\[
\tau'\colon\Delta^{n,n,n_k\res k\in K}_L\to \delta^{(\{1,2\}\amalg K)\square}_*(\cC,\cT^{-2}).
\]
The composition
\[
\Cart^n\times\Delta^{[n_k]_{k\in K}}_L\xrightarrow{\pi^n\times\id}\Delta^n\times\Delta^n\times\Delta^{[n_k]_{k\in K}}_L\xrightarrow{\tau'}\cC
\]
is a vertex of $\Kart(\tau')$ by Lemma \ref{a5le:kan} and the equality $\cT^{-2}_{12}=\cT_1*_\cC^\cart \cT_2$. This vertex, together with the
composition
\begin{align*}
\Delta^n\times \Delta^n\times \Cart^n\times \Delta^{[n_k]_{k\in K}} &\to \Cart^n\times\Delta^{[n_k]_{k\in K}}  \\ &\xrightarrow{\pi^n\times \id}
\Delta^n\times \Delta^n\times\Delta^{[n_k]_{k\in K}}\to \Fun(\Delta^l,\cD),
\end{align*}
where the first map is the projection and the last map corresponds to the composition
\[
\Delta^{n,n,n_k\res k\in K}\xrightarrow{\tau'} \op^{\{1,2\}\amalg
K}_L\delta^{(\{1,2\}\amalg
K)\square}_*(\cC,\cT^{-2})\xrightarrow{v\circ f'}\delta_*^{\{1,2\}\amalg
K}\Fun(\Delta^l,\cD),
\]
provides a vertex of $\cN(f(f'(\sigma')))$, whose image under $\Phi(f(f'(\sigma')))$ is $v(f'(\sigma'))$. This construction is functorial in $\sigma'$, giving rise to $\nu\in \Gamma((f\circ f')^* \cN)_0$ satisfying $(f\circ f')^*\Phi\circ \nu=v\circ f'$.  Applying Proposition
\ref{a1pr:extension} to $\Phi$, the map $f\circ f'$, and the global section $\nu$, we obtain a map $u\colon W_i\to \Fun(\Delta^l,\cD)$ satisfying $p\circ u=w$ such that $u\circ f\circ f'$ and $v\circ f'$ are homotopic over $\Fun(\partial \Delta^l,\cD)$, as desired.
\end{proof}

\begin{remark}
As a special case of Theorem \ref{a5th:cartesian_gluing}, the inclusion
\[
\delta^*_2\delta_*^{2\square}(\Cart^n,\cF^\cart)\subseteq \delta^*_2\delta_*^{2+}(\Cart^n,\cF)
\]
is a categorical equivalence. If we have a direct proof of this special case, Construction \ref{a5co:section} through Lemma \ref{a5le:trivial_co} are not necessary and the proof of Theorem \ref{a5th:cartesian_gluing} can be achieved with $\psi(\tau)$ and $\phi(\tau)$ replaced by $g(\tau)$ and $h(\tau)$, respectively.
\end{remark}

\subsection{Some trivial cofibrations}
\label{a6ss}

In this section, we prove that certain inclusions of simplicial sets defined in combinatorial manners are inner anodyne or categorical equivalences. In particular, they are trivial cofibrations in $\Sset$ for the Joyal model structure \cite{HTT}*{Theorem 2.2.5.1}. Only Lemma \ref{a6le:cpt_inner} and Lemma \ref{a6le:cart_inner} are used in previous sections, namely, in \ref{a4ss} and \ref{a5ss}, respectively.

We let $\star$ denote \emph{joins} of categories and simplicial sets \cite{HTT}*{{\Sec}1.2.8}.

\begin{lem}\label{a6le:join}
Let $A_0\subseteq A$, $B_0\subseteq B$, $C_0\subseteq C$ be inclusions of simplicial sets. If $A_0\subseteq A$ is right anodyne and $C_0\subseteq C$ is left anodyne \cite{HTT}*{Definition 2.0.0.3}, then the induced inclusion
\[
A\star B_0 \star C\coprod_{A_0\star B_0\star C_0} A_0\star B\star C_0 \subseteq A\star B\star C
\]
is inner anodyne.
\end{lem}

\begin{proof}
Consider the commutative diagram of inclusions with pushout squares
\[
\resizebox{\hsize}{!}{
\xymatrix{
A_0\star B_0\star C_0\ar[r]\ar[d] & A_0\star B\star C_0\ar[rd]\ar[d]\\
A\star B_0\star C_0\ar[r]\ar[d] &A\star B_0\star C_0 \coprod_{A_0\star B_0 \star C_0} A_0\star B\star C_0\ar[r]^-f\ar[d]
& A\star B\star C_0\ar[d]\ar[rd]\\
A\star B_0\star C\ar[r] & A\star B_0 \star C\coprod_{A_0\star B_0\star C_0} A_0\star B\star C_0\ar[r]^-{f'}
& A\star B_0\star C\coprod_{A\star B_0\star C_0} A\star B\star C_0\ar[r]^-g
& A\star B\star C.
}
}
\]
By \cite{HTT}*{Lemma 2.1.2.3}, $f$ is inner anodyne since $A_0\subseteq A$ is right anodyne; $g$ is inner anodyne since $C_0\subseteq C$ is left anodyne. It follows that $g\circ f'$ is inner anodyne.
\end{proof}

\begin{lem}\label{a6le:cut}
Let $S$ be a partially ordered set and let $Q=[2] \subseteq S$, $R=S-\{1\}\subseteq S$ be full inclusions. Assume that $0$ is a final object of $R_{/1}$ and $2$ is an initial object of $R_{1/}$. Then the inclusion $\rN(Q)\cup\rN(R)\subseteq \rN(S)$ is inner anodyne.
\end{lem}

\begin{proof}
Consider the commutative diagram of inclusions
\[
\xymatrix{
\rN(Q\cap R)\ar[r]\ar[d] & \rN(Q)\ar[d]\ar[rd]\\
\rN(R_{/1}\star R_{1/})\ar[r]\ar[d]&\rN(R_{/1}\star R_{1/})\coprod_{\rN(\{0\}\star
\{2\})}
\rN([2])\ar[r]^-f\ar[d]&
\rN(R_{/1}\star \{1\}\star R_{1/})\ar[d]\\
\rN(R)\ar[r]&\rN(R)\cup \rN(Q)\ar[r]^g& \rN(S)
}
\]
in which the square on the left are clearly pushouts. Note that for any simplex $\sigma$ of $N(S)$, if $\sigma$ is not a simplex of $\rN(R)$, then $1$ is a vertex of $\sigma$, so that $\sigma$ is a simplex of $\rN(R_{/1}\star \{1\}\star R_{1/})$. Thus the lower outer square is a pushout. It follows that $g$ is a pushout of $f$. By assumption and \cite{HTT}*{Lemma 4.2.3.6}, $\rN(\{0\})\subseteq \rN(R_{/1})$ is right anodyne and $\rN(\{2\})\subseteq \rN(R_{1/})$ is left anodyne. It follows that $f$ is inner anodyne by Lemma \ref{a6le:join}. Therefore, $g$ is inner anodyne.
\end{proof}

\begin{remark}\label{a6re:order}
Let $P\subseteq Q$ and $P\subseteq R$ be full inclusions of partially ordered sets. The pushout $S=Q\coprod_P R$ in the category of partially
ordered sets admits the following description. The underlying set of $S$ is the set-theoretic pushout. The partial order on $S$ is uniquely
characterized by the following properties:
\begin{enumerate}
  \item $Q\subseteq S$ and $R\subseteq S$ are full inclusions; and

  \item for $q\in Q$, $r\in R$, we have $q\le r$ (resp.\ $q\ge r$) if and only if there exists $p\in P$ satisfying $q\le p\le r$ (resp.\
      $q\ge p\ge r$).
\end{enumerate}
\end{remark}

\begin{lem}\label{a6le:union}
Let $P\subseteq Q$ and $P\subseteq R$ be full inclusions of partially ordered sets and $S=Q\coprod_P R$ the pushout in the category of partially ordered sets. Suppose that the following conditions are satisfied:
\begin{enumerate}
  \item $Q$ admits pushouts and pushouts are preserved by the inclusion $Q\subseteq S$.

  \item $Q-P$ is finite.

  \item $P$ is an up-set of $Q$, that is, a subset such that $p\in P$ and $q\geq p$ with $q\in Q$ imply $q\in P$.
\end{enumerate}
Then the inclusion $\rN(Q)\cup \rN(R)\subseteq \rN(S)$ is inner anodyne.
\end{lem}

\begin{proof}
We proceed by induction on $n=\# (Q-P)$. For $n=0$, we have $R=S$ and the assertion is trivial. For $n=1$, put $Q-P\coloneqq\{q\}$. Then Condition (3) means that $q$ is a minimal element of $Q$, hence of $S$. Note that $\rN(R)\cup\rN(S_{q/})=\rN(S)$. Indeed, for any simplex $\sigma$ of $\rN(S)$, if $\sigma$ is a simplex of $\rN(R)$, then $q$ is a vertex of $\sigma$, so that $\sigma$ is a simplex of $\rN(S_{q/})$. Thus the inclusion $\rN(Q)\cup\rN(R)\subseteq \rN(S)$ is a pushout of the inclusion $\rN(Q_{q/})\cup\rN(R_{q/})\subseteq \rN(S_{q/})$. The latter is isomorphic to the inclusion
\begin{equation}\label{a7eq:preinner}
\rN(P_{q/})^{\triangleleft}\coprod_{\rN(P_{q/})}\rN(R_{q/})\subseteq\rN(R_{q/})^{\triangleleft}.
\end{equation}
By Condition (1), for every $r\in R_{q/}$, the partially ordered set $P_{q//r}$ is filtered. Indeed, for $p,p'\in P_{q//r}$, the pushout $p\vee_q p'$ is a common upper bound in $P_{q//r}$. Thus $\rN(P_{q//r})$ is weakly contractible by \cite{HTT}*{Theorem 5.3.1.13, Lemma 5.3.1.18}. It follows that $\rN(P_{q/})^{op}\subseteq \rN(R_{q/})^{op}$ is cofinal by \cite{HTT}*{Theorem 4.1.3.1}, thus right anodyne by \cite{HTT}*{Proposition 4.1.1.3(4)}. Therefore, \eqref{a7eq:preinner} is inner anodyne by \cite{HTT}*{Lemma 2.1.2.3}.

For $n\ge 2$, we choose a minimal element $q$ of $Q-P$. Then Condition (3) implies that $q$ is a minimal element of $Q$, hence of $S$. Put $S'\coloneqq S-\{q\}\supseteq R$ and $Q'\coloneqq Q-\{q\}=S'\cap Q$. Consider the diagram of inclusions with pushout square
\[
\xymatrix{\rN(Q')\cup \rN(R)\ar[r]^-f\ar[d] & \rN(S')\ar[d]\\
\rN(Q)\cup \rN(R)\ar[r]& \rN(Q)\cup \rN(S')\ar[r]^-g & \rN(S).
}
\]

By the induction hypothesis applied to the inclusions $P\subseteq Q'$ and $P\subseteq R$, we know that $f$ is inner anodyne. Indeed, we have $P=Q'\cap R$ and $S'$ is the pushout $Q'\coprod_P R$ in the category of partially ordered set, by the description in Remark \ref{a6re:order}. Condition (1) holds since $q$ is a minimal element of $Q$, the partially ordered set $Q'$ admits pushouts and pushouts are stable under the inclusion $Q'\subseteq Q$, hence under the inclusions $Q'\subseteq S$ and $Q'\subseteq S'$; for Condition (2), we have $\#(Q'-P)=n-1$; and for Condition (3), $P$ is an up-set of $Q$, hence of $Q'$.

By the induction hypothesis applied to the inclusions $Q'\subseteq Q$ and $Q'\subseteq S'$, we know that $g$ is inner anodyne as well. Indeed, we have $Q'=Q\cap S'$ and $S$ is the pushout $Q\coprod_{Q'} S$ in the category of partially ordered sets; Condition (1) remains unchanged; for Condition (2), we have $\#(Q-Q')=1$; and for Condition (3), $Q'$ is an up-set of $Q$ since $q$ is minimal.

Therefore, the inclusion $\rN(Q)\cup \rN(R)\subseteq \rN(S)$ is inner anodyne.
\end{proof}

\begin{lem}\label{a6le:lattice}
Let $P$ be a finite partially ordered set admitting pushouts and let $p_0\le\dots\le p_s$; $q_0\le\dots\le q_s$ be elements of $P$ such that
$p_i\le q_{i-1}$ for $1\le i\le s$. Then the inclusion
\begin{align*}
\bigcup_{i=0}^s\rN(P_{p_i//q_i})\subseteq\rN\(\bigcup_{i=0}^{s}P_{p_i//q_i}\)
\end{align*}
is inner anodyne.
\end{lem}

\begin{proof}
Put $P_i\coloneqq P_{p_i//q_i}$. The inclusion can be decomposed as $Q_0\subseteq\dots \subseteq Q_n$, where
\[
Q_j=\rN\(\bigcup_{i=0}^jP_i\)\cup\bigcup_{i=j+1}^{n}\rN(P_i).
\]
For $1\le j\le n$, the inclusion $Q_{j-1}\subseteq Q_j$ is a pushout of
\begin{equation}\label{a6eq:lattice}
\rN\(\bigcup_{i=0}^{j-1}P_i\)\cup \rN(P_j)\subseteq \rN\(\bigcup_{i=0}^jP_i\).
\end{equation}
Indeed, for $k>j$, we have $P_k\cap \(\bigcup_{i=0}^j P_i\)\subseteq P_j$. It then suffices to check that \eqref{a6eq:lattice} satisfies the assumptions of Lemma \ref{a6le:union}. We denote coproducts in $P_{p_0/}$ by $\vee$. Take $x\in A=\bigcup_{i=0}^{j-1}P_i$ and $y\in P_j$. If $x\ge y$, then $x,y\in P_{p_j//q_{j-1}}=P_{j-1}\cap P_j$. If $x\le y$, then $x\le x\vee p_j\le y$, where $x\vee p_j\in P_{j-1}\cap P_j$ by the assumption that $p_j\le q_{j-1}$. Thus $\bigcup_{i=0}^jP_i$ is the pushout $A\coprod_{A\cap P_j} P_j$ in the category of partially ordered sets, by Remark \ref{a6re:order}. Condition (1) of Lemma \ref{a6le:union} follows from the fact that for $x\in P_i$, $y\in P_{i'}$, we have $x\vee y\in P_{\max\{i,i'\}}$. Condition (2) is clear. For Condition (3), it suffices to note that $A\cap P_j=A_{p_j/}$.
\end{proof}

By an \emph{interval sublattice} of a finite lattice $P$, we mean a subset of the form $P_{p//q}$, where $p\le q$ are in $P$.

\begin{lem}\label{a6le:equivalence}
Let $P$ be a finite lattice and let $p_0\le\dots\le p_s\le q_0\le\dots\le q_s$ be elements of $P$ satisfying $\bigcup_{i=0}^s P_{p_i//q_i}=P$. Let $Q_1,\dots,Q_t$ be interval sublattices of $P$. Then the inclusion
\[
\bigcup_{i=0}^s\rN(P_{p_i//q_i})\cup \bigcup_{j=1}^t \rN(Q_j)\subseteq \rN(P)
\]
is a categorical equivalence.
\end{lem}

Note that the assumptions imply that $p_0$ is the minimum of $P$ and $q_s$ is the maximum of $P$.

\begin{proof}
We proceed by induction on $t$. Put $P_i\coloneqq P_{p_i//q_i}$ and $R_j\coloneqq\bigcup_{i=0}^s\rN(P_i)\cup\bigcup_{k=1}^j\rN(Q_k)$. We need to show that $R_t\subseteq N(P)$ is a categorical equivalence. By Lemma \ref{a6le:lattice}, the inclusion $R_0=\bigcup_{i=0}^s\rN(P_i)\subseteq\rN(P)$ is inner anodyne, thus a categorical equivalence \cite{HTT}*{Lemma 2.2.5.2}. Thus for $t=0$ we are done. For $t\ge 1$, it suffices to show that the inclusions $R_0\subseteq \dots \subseteq R_t$ are categorical equivalences. For $1\le j\le t$, the inclusion $R_{j-1}\subseteq R_j$ is a pushout of
\begin{equation}\label{a8eq:cateq}
\bigcup_{i=0}^s\rN(P_i\cap Q_j)\cup \bigcup_{k=1}^{j-1}\rN(Q_k\cap Q_j)\subseteq \rN(Q_j)
\end{equation}
by an inclusion. By \cite{HTT}*{Lemma A.2.4.3}, it suffices to show that \eqref{a8eq:cateq} is a categorical equivalence, which follows from the induction hypothesis. In fact, if we write $Q_j=P_{p//q}$, then $P_i\cap Q_j=P_{p_i\vee p//q_i\wedge q}$, and for $0\le i,i'\le s$ such that $P_i\cap Q_j\neq \emptyset$, $P_{i'}\cap Q_j\neq \emptyset$, we have $p_i\vee p\le q_{i'}\wedge q$.
\end{proof}

\begin{lem}\label{a6le:cpt_inner}
The inclusion $\Box^n\subseteq\CCpt^n$ is inner anodyne.
\end{lem}

\begin{proof}
We apply Lemma \ref{a6le:lattice} to the lattice $P=[n]\times [n]$, with $s=n$, $p_i=(0,i)$ and $q_i=(i,n)$. We have $p_0\le \dots \le p_n=q_0\le \dots \le q_n$. Thus, the inclusion
\begin{align*}
\Box^n&=\bigcup_{i=0}^n\rN\(\RCpt^n_{(0,i)//(i,n)}\)\subseteq\rN\(\bigcup_{i=0}^n\RCpt^n_{(0,i)//(i,n)}\)
=\rN(\RCpt^n)=\CCpt^n
\end{align*}
is inner anodyne.
\end{proof}

\begin{lem}\label{a6le:cart_inner}
The inclusion $\bigcup_{0\le p\le n} \boxplus^n_{p,n}\subseteq \Cart^n$ is inner anodyne and the inclusion $\boxplus^n \subseteq \Cart^n$ is a
categorical equivalence.
\end{lem}

\begin{proof}
We apply Lemma \ref{a6le:lattice} and Lemma \ref{a6le:equivalence} to the lattice $P=\Crt^n$, with $s=n$, $p_i=\xi^n(i,n)$, $q_i=\eta^n(i,n)$, and the $Q_j$'s given by $\Crt^n_{p,q}$ with $0\le p\le n$ and $0\le q<n$. We have $\xi^n(0,n)\le \dots \le\xi^n(n,n)\le \eta^n(0,n)\le\dots\le \eta^n(n,n)$. It remains to show $\Crt^n=\bigcup_{p=0}^n \Crt^n_{p,n}$. For $Q\in \Crt^n$, we let $p$ denote $\pi^n_1(Q)=\min_{(p',q')\in Q} p'$. Then we have
\[
\xi^n(p,n)\le \varsigma^n(p,0)\le Q\le \varsigma^n(p,n)=\eta^n(p,n),
\]
so that $Q\in \Crt^n_{p,n}$.
\end{proof}

\section{More preliminaries on $\infty$-categories}
\label{b1ss}

This chapter is a further collection of preliminaries on $\infty$-categories. In \Sec\ref{b1ss:preliminary_lemmas}, we record some basic lemmas. In \Sec\ref{b1ss:partial_adjoints}, we develop a method of taking partial adjoints, namely, taking adjoint functors along given directions. This will be used to construct the initial enhanced operation map for schemes. In \Sec\ref{b1ss:symmetric_monoidal}, we collect some general facts and constructions related to symmetric monoidal $\infty$-categories.

\subsection{Elementary lemmas}
\label{b1ss:preliminary_lemmas}

Let us start with the following lemma, which appears as \cite{DAG5}*{Lemma 2.4.6}. We include a proof for the convenience of the reader.

\begin{lem}\label{b1le:trivial_contractible}
Let $\cC$ be a nonempty $\infty$-category that admits product of two objects. Then the geometric realization $\lvert \cC\rvert$ is contractible.
\end{lem}

\begin{proof}
Fix an object $X$ of $\cC$ and a functor $\cC\to \cC$ sending $Y$ to $X\times Y$. The projections $X\times Y\to X$ and $X\times Y\to Y$ define functors $h,h'\colon \Delta^1 \times \cC\to \cC$ such that
\begin{itemize}
  \item $h\res\Delta^{\{0\}}\times \cC=h'\res\Delta^{\{0\}}\times \cC$;

  \item $h\res\Delta^{\{1\}}\times \cC$ is the constant functor of value $X$;

  \item $h'\res\Delta^{\{1\}}\times \cC=\id_\cC$.
\end{itemize}
Then $\lvert h\rvert$ and $\lvert h'\rvert$ provide a homotopy between $\id_{\lvert \cC\rvert}$ and the constant map of value $X$.
\end{proof}

The following is a variant of the Adjoint Functor Theorem \cite{HTT}*{5.5.2.9}.

\begin{lem}\label{b1le:adjoint_variant}
Let $F\colon \cC\to \cD$ be a functor between presentable $\infty$-categories. Let $\rh F\colon \rh \cC\to \rh \cD$ be the functor of
(unenriched) homotopy categories.
\begin{enumerate}
  \item The functor $F$ has a right adjoint if and only if it preserves pushouts and $\rh F$ has a right adjoint.

  \item The functor $F$ has a left adjoint if and only if it is accessible and preserves pullbacks and $\rh F$ has a left adjoint.
\end{enumerate}
\end{lem}

\begin{proof}
The necessity follows from \cite{HTT}*{5.2.2.9}. The sufficiency in (1) follows from the fact that small colimits can be constructed out of pushouts and small coproducts \cite{HTT}*{4.4.2.7} and preservation of small coproducts can be tested on $\rh F$. The sufficiency in (2) follows from dual statements.
\end{proof}

We will apply the above lemma in the following form.

\begin{lem}\label{b1le:adjoint_infinity}
Let $F\colon \cC\to \cD$ be a functor between presentable stable $\infty$-categories. Let $\rh F\colon \rh \cC\to \rh \cD$ be the functor of
(unenriched) homotopy categories. Then
\begin{enumerate}
  \item The functor $F$ admits a right adjoint if and only if $\rh F$ is a triangulated functor and admits a right adjoint.

  \item The functor $F$ admits a left adjoint if $F$ admits a right adjoint and $\rh F$ admits a left adjoint.
\end{enumerate}
\end{lem}

\begin{proof}
By \cite{HA}*{Lemma 1.2.4.14}, a functor $G$ between stable $\infty$-categories is exact if and only if $\rh G$ is triangulated. The lemma then follows from Lemma \ref{b1le:adjoint_variant} and \cite{HA}*{Proposition 1.1.4.1}.
\end{proof}

\begin{lem}\label{b1le:adjoint_R}
Let $F\colon \cA\to \cB$ be a left exact functor between Grothendieck Abelian categories that commutes with small coproducts. Assume that $F$ has finite cohomological dimension. Then the right derived functor $\rR F\colon\cD(\cA)\to \cD(\cB)$ admits a right adjoint.
\end{lem}

\begin{proof}
By the previous lemma, it suffices to show that $\rh(\rR F)$ commutes with small coproducts. This is standard. See \cite{KS}*{Proposition 14.3.4(ii)}.
\end{proof}

\subsection{Partial adjoints}
\label{b1ss:partial_adjoints}

We first recall the notion of adjoints of squares.

\begin{definition}
Consider diagrams of $\infty$-categories
\[
\xymatrix{
\cC\ar[d]_U\ar@{}[rd]|\sigma&\cD\ar[l]_{F}\ar[d]^V & \cC\ar[r]^G\ar[d]_U\ar@{}[rd]|\tau &\cD\ar[d]^V\\
\cC'&\cD'\ar[l]_{F'}& \cC'\ar[r]^-{G'} & \cD'}
\]
that commute up to specified equivalences $\alpha\colon F'\circ V\to U\circ F$ and $\beta\colon V\circ G\to G'\circ U$. We say that $\sigma$ is a \emph{left adjoint} to $\tau$ and $\tau$ is a \emph{right adjoint} to $\sigma$, if $F$ is a left adjoint of $G$, $F'$ is a left adjoint of $G'$, and $\alpha$ is equivalent to the composite transformation
\[
F'\circ V\to F'\circ V\circ G\circ F \xrightarrow{\beta} F'\circ G'\circ U\circ F\to U\circ F.
\]
\end{definition}

\begin{remark}
The diagram $\tau$ has a left adjoint if and only if $\tau$ is left 
adjointable in the sense of \cite{HTT}*{7.3.1.2} and \cite{HA}*{Definition 
4.7.4.13}. If $G$ and $G'$ are equivalences, then $\tau$ is left adjointable. 
We have analogous notions for ordinary categories. A square $\tau$ of 
$\infty$-categories is left adjointable if and only if $G$ and $G'$ admit 
left adjoints and the square $\rh\tau$ of homotopy categories is left 
adjointable. When visualizing a square $\Delta^1\times\Delta^1\to\cC$, we 
adopt the convention that the first factor of $\Delta^1\times\Delta^1$ is 
vertical and the second factor is horizontal. 
\end{remark}

\begin{lem}\label{b1le:adjunction}
Consider a diagram of right Quillen functors
\[
\xymatrix{
\bA\ar[r]^G\ar[d]_U &\bB\ar[d]^V\\
\bA'\ar[r]^-{G'} & \bB'}
\]
of model categories, that commutes up to a natural equivalence $\beta\colon V\circ G\to G'\circ U$ and is endowed with Quillen equivalences $(F,G)$ and $(F',G')$. Assume that $U$ preserves weak equivalences and all objects of $\bB'$ are cofibrant. Let $\alpha$ be the composite transformation
\[
F'\circ V\to F'\circ V\circ G\circ F \xrightarrow{\beta} F'\circ G'\circ U\circ F\to U\circ F.
\]
Then for every fibrant-cofibrant object $Y$ of $\bB$, the morphism $\alpha(Y)\colon (F'\circ V)(Y)\to (U\circ F)(Y)$ is a weak equivalence.
\end{lem}

\begin{proof}
The square $R\beta$
\[
\xymatrix{
\rh\bA\ar[r]^-{RG}\ar[d]_-{RU} & \rh\bB\ar[d]^-{RV}\\
\rh\bA'\ar[r]^-{RG'} & \rh\bB'}
\]
of homotopy categories is left adjointable. Let $\sigma\colon LF'\circ RV \to RU\circ LF$ be its left adjoint. For fibrant-cofibrant $Y$, $\alpha(Y)$ computes $\sigma(Y)$.
\end{proof}

We apply Lemma \ref{b1le:adjunction} to the straightening functor \cite{HTT}*{{\Sec}3.2.1}. Let $p\colon S'\to S$ be a map of simplicial sets, and $\pi\colon \cC'\to \cC$ a functor of simplicial categories fitting into a diagram
\[
\xymatrix{
\fC[S']\ar[r]^-{\phi'}\ar[d]_-{\fC[p]}& \cC'^{op}\ar[d]^-{\pi^{op}}\\
\fC[S]\ar[r]^-{\phi} & \cC^{op}}
\]
which is commutative up to a simplicial natural equivalence. By \cite{HTT}*{Proposition 3.2.1.4}, we have a diagram
\[
\xymatrix{
(\Mset)^\cC\ar[r]^-{\Un^+_\phi}\ar[d]_-{\pi^*} & (\Mset)_{/S}\ar[d]^-{p^*}\\
(\Mset)^{\cC'}\ar[r]^-{\Un^+_{\phi'}} & (\Mset)_{/S'},}
\]
which satisfies the assumptions of Lemma \ref{b1le:adjunction} if $\phi$ and $\phi'$ are equivalences of simplicial categories. In this case, for every fibrant object $f\colon X\to S$ of $(\Mset)_{/S}$, endowed with the Cartesian model structure, the morphism
\[
(\St^+_{\phi'} \circ p^*) X\to (\pi^*\circ \St^+_\phi) X
\]
is a pointwise Cartesian equivalence.

Similarly, if $g\colon \cC\to \cD$ is a functor of ($\cV$-small) categories, then \cite{HTT}*{Remark 3.2.5.14} provides a diagram
\[
\xymatrix{
(\Mset)^\cD\ar[r]^-{\rN^+_\bullet(\cD)}\ar[d]_-{g^*} & (\Mset)_{/\rN(\cD)}\ar[d]^-{\rN(g)^*}\\
(\Mset)^\cC\ar[r]^-{\rN^+_\bullet(\cC)} & (\Mset)_{/\rN(\cC)}}
\]
satisfying the assumptions of Lemma \ref{b1le:adjunction}. Thus, for every fibrant object $Y$ of $(\Mset)_{/\rN(\cD)}$, endowed with the coCartesian model structure, the morphism
\[
\fF^+_{\rN(g)^* Y}(\cC)\to g^*\fF^+_Y(\cD)
\]
is a pointwise coCartesian equivalence.

\begin{proposition}[partial adjoint]\label{b1pr:partial_adjoint}
Consider quadruples $(I,J,R,f)$ where $I$ is a set, $J\subseteq I$, $R$ is an $I$-simplicial set and $f\colon \delta_I^* R\to\Cat$ is a functor, satisfying the following conditions:
\begin{enumerate}
  \item For every $j \in J$ and every edge $e$ of $\epsilon^I_j R$, the functor $f(e)$ has a left adjoint.

  \item For every $i\in J^c\colonequals I\backslash J$, every $j\in J$ and every $\tau\in (\epsilon_{i,j}^I R)_{1,1}$, the square
      $f(\tau)\colon \Delta^1\times \Delta^1\to \Cat$ is left adjointable.
\end{enumerate}
There exists a way to associate, to every such quadruple, a functor $f_J\colon \delta_{I,J}^* R\to \Cat$, satisfying the following conclusions:
\begin{enumerate}
  \item $f_J\res \delta_{J^c}^*(\del^\iota)_*R=f\res\delta_{J^c}^*(\del^\iota)_*R$, where $\iota\colon J^c\to I$ is the inclusion.

  \item For every $j\in J$ and every edge $e$ of $\epsilon^I_jR$, the functor $f_J(e)$ is a left adjoint of $f(e)$.

  \item For every $i\in J^c$, every $j\in J$ and every $\tau\in(\epsilon_{i,j}^I R)_{1,1}$, the square $f_J(\tau)$ is a left adjoint of $f(\tau)$.

  \item For two quadruples $(I,J,R,f)$, $(I',J',R',f')$ and maps $\mu\colon I'\to I$, $u\colon (\del^\mu)^*R'\to R$ such that $J'=\mu^{-1}(J)$ and $f'=f\circ\delta^*_Iu$, the functor $f'_{J'}$ is equivalent to $f_J\circ\delta_{I,J}^*u$.
\end{enumerate}
\end{proposition}

Note that in conclusion (1), $\delta_{J^c}^*(\del^\iota)_*R$ is naturally a simplicial subset of both $\delta_I^* R$ and $\delta_{I,J}^* R$. When visualizing $(1,1)$-simplices of $\epsilon_{i,j}^I R$, we adopt the convention that direction $i$ is vertical and direction $j$ is horizontal. If $J^c$ is nonempty, then condition (2) implies condition (1), and conclusion (3) implies conclusion (2).

\begin{proof}
Recall that we have fixed a fibrant replacement functor $\Fibr\colon \Mset\to\Mset$.

Let $\sigma\in (\delta^*_{I,J}R)_n$ be an object of $\del_{/\delta^*_{I,J}R}$, corresponding to $\Delta^{n_i\res i\in I}_J\to R$, where $n_i=n$. It induces a functor $f(\sigma)\colon\rN(D)\simeq\Delta^{[n_i]_{i\in I}}_J\to \Cat$, where $D$ is the partially ordered set $S\times T^{op}$ with $S=[n]^{J^c}$ and $T=[n]^J$. This corresponds to a projectively fibrant simplicial functor $\cF\colon\fC[\rN(D)]\to \Mset$. Let $\phi_D\colon \fC[\rN(D)]\to D$ be the canonical equivalence of simplicial categories and put
\[
\cF'\coloneqq(\Fibr^D \circ\St^+_{\phi_D^{op}} \circ \Un^+_{\rN(D)^{op}})\cF\colon D\to \Mset.
\]
We have weak equivalences
\begin{align*}
\cF\gets (\St^+_{\rN(D)^{op}}\circ \Un^+_{\rN(D)^{op}})\cF
&\to (\phi_D^*\circ\phi_{D!}\circ \St^+_{\rN(D)^{op}}\circ \Un^+_{\rN(D)^{op}})\cF \\
&\simeq (\phi_D^* \circ\St^+_{\phi_D^{op}}\circ \Un^+_{\rN(D)^{op}})\cF \to\phi_D^*(\cF').
\end{align*}
Thus, for every $\tau\in (\epsilon_{i,j}^I \rN(D))_{1,1}$, the square $\cF'(\tau)$ is equivalent to $f(\tau)$, both taking values in $\Cat$.

Let $\cF''$ be the composition
\[
S\to (\Mset)^{T^{op}} \xrightarrow{\Un^+_{\phi_T}} (\Mset)_{/\rN(T)},
\]
where the first functor is induced by $\cF'$. For every $s\in S$, the value $\cF''(s)\colon X(s)\to \rN(T)$ is a fibrant object of $(\Mset)_{/\rN(T)}$ with respect to the Cartesian model structure. In other words, there exists a Cartesian fibration $p(s)\colon Y(s)\to \rN(T)$ and an isomorphism $X(s)\simeq Y(s)^\natural$. By condition (1), for every morphism $t\to t'$ of $T$, the induced functor $Y(s)_{t'}\to Y(s)_t$ has a left adjoint. By \cite{HTT}*{Corollary 5.2.2.5}, $p(s)$ is also a coCartesian fibration. We consider the object $(p(s),\cE(s))$ of $(\Mset)_{/\rN(T)}$, where $\cE(s)$ is the set of $p$-coCartesian edges of $Y(s)$. By condition (2), this construction is functorial in $s$, giving rise to a functor $\cG'\colon S\to(\Mset)_{/\rN(T)}$.

The composition
\[
S\xrightarrow{\cG'}(\Mset)_{/\rN(T)}\xrightarrow{\fF^+_\bullet(T)}(\Mset)^{T}\xrightarrow{\Fibr^T}(\Mset)^T
\]
induces a projectively fibrant diagram
\[
\cG \colon S\times T\to \Mset.
\]
We denote by $\cG_\sigma\colon [n]\to \Mset$ the composition
\[
[n]\to S\times T\to \Mset,
\]
where the first functor is the diagonal functor. The construction of $\cG_\sigma$ is not functorial in $\sigma$ because the straightening functors do not commute with pullbacks, even up to natural equivalences. Nevertheless, for every morphism $d\colon \sigma\to \tilde\sigma$ in
$\del_{/\delta^*_{I,J}R}$, we have a canonical morphism $\cG_\sigma\to d^*\cG_{\tilde\sigma}$ in $(\Mset)^{[n]}$, which is a weak equivalence by Lemma \ref{b1le:adjunction}. The functor
\[
(\del_{/\delta^*_{I,J}R})_{\sigma/}\to (\Mset)^{[n]}
\]
sending $d\colon \sigma\to\tilde\sigma$ to $d^*\cG_{\tilde\sigma}$ induces a map
\[
\cN(\sigma)\colonequals\rN((\del_{/\delta^*_{I,J}R})_{\sigma/})\to \Map^\sharp((\Delta^n)^\flat,(\Cat)^\natural),
\]
which we denote by $\Phi(\sigma)$. Since the category $(\del_{/\delta^*_{I,J}R})_{\sigma/}$ has an initial object, the simplicial set $\cN(\sigma)$ is weakly contractible. This construction is functorial in $\sigma$ so that $\Phi\colon \cN\to \Map[\delta^*_{I,J}R,\Cat]$ is a morphism of $(\Sset)^{(\del_{/\delta^*_{I,J}R})^{op}}$. Applying Corollary \ref{a2co:key}(1), we obtain a functor $\widetilde{f_J}\colon \delta^*_{I,J} R\to \Cat$ satisfying conclusions (2) and (3) up to homotopy.

Under the situation of conclusion (4), $\delta^*_{I,J}u\colon\delta^*_{I',J'}R'\to \delta^*_{I,J}R$ induces $\varphi\colon \cN'\to(\delta^*_{I,J}u)^*\cN$. By construction, there exists a homotopy between $\Phi'$ and $((\delta^*_{I,J}u)^*\Phi)\circ \varphi$. By Corollary \ref{a2co:key}(2), this implies that $\widetilde{f'_{J'}}$ and $\widetilde{f_J}\circ\delta^*_{I,J}u$ are homotopic.

By construction, there exists a homotopy between $r^*\Phi$ and the composite map $r^*\cN\to \Delta^0_{Q}\xrightarrow{f\res Q} \Map[Q,\Cat]$, where $Q=\delta_{J^c}^*(\del^\iota)_*R$ and $r\colon Q\to \delta^*_{I,J}R$ is the inclusion. By Corollary \ref{a2co:key}(2), this implies that $\widetilde{f_J}\res Q$ and $f\res Q$ are homotopic. Since the inclusion
\[
Q^\natural\times(\Delta^1)^\sharp \coprod_{Q^\natural\times (\Delta^{\{0\}})^\sharp}(\delta_{I,J}^*R)^\natural\times(\Delta^{\{0\}})^\sharp
\to(\delta_{I,J}^*R)^\natural\times(\Delta^1)^\sharp
\]
is marked anodyne, there exists $f_J\colon \delta_{I,J}^*R\to \Cat$ homotopic to $\widetilde{f_J}$ such that $f_J\res Q=f\res Q$.
\end{proof}

\begin{remark}\label{b1re:partial_adjoint}
We have the following remarks concerning Proposition \ref{b1pr:partial_adjoint}.
\begin{enumerate}
  \item There is an obvious dual version of Proposition \ref{b1pr:partial_adjoint} for right adjoints.

  \item Proposition \ref{b1pr:partial_adjoint} holds without the (implicit) convention that $R$ is $\cV$-small. To see this, it suffices to apply the proposition to the composite map $\delta^*_IR\xrightarrow {f}\Cat\to \Cat^\cW$, where $\cW\supseteq \cV$ is a universe containing $R$ and $\Cat^\cW$ is the $\infty$-category of $\infty$-categories in $\cW$.

  \item Consider the $2$-tiled $\infty$-category $(\Cat,\cT)$ where $\cT_1=(\Cat)_1$, $\cT_2$ consists of all functors that admit a left adjoint, and $\cT_{12}$ consists of all squares that are left adjointable. Let
      \[
      \phi\colon\delta^{2\square}_*(\Cat,\cT)\hookrightarrow\delta_2^*\delta_{2*}\Cat\to \Cat
      \]
      be the natural functor induced by the counit map. Applying Proposition \ref{b1pr:partial_adjoint} (and Remark \ref{b1re:partial_adjoint}(2)) to the quadruple $(\{1,2\},\{2\},\delta^{2\square}_*(\Cat,\cT),\phi)$, we get a functor
      \[
      \phi_{\{2\}}\colon\delta_{2,\{2\}}^*\delta^{2\square}_*(\Cat,\cT)\to\Cat.
      \]
      This functor is \emph{universal} in the sense that for any quadruple $(I,J,R,f)$ satisfying the conditions in Proposition \ref{b1pr:partial_adjoint}, if we denote by $\mu\colon I\to \{1,2\}$ the map given by $\mu^{-1}({2})=J$, then $f\colon\delta_2^*(\del^\mu)^* R\to \Cat$ uniquely determines a map $u\colon (\del^\mu)^*R\to \delta_{2*}\Cat$ by adjunction which
      factorizes through $\delta^{2\square}_*(\Cat,\cT)$ and $f_J$ can be taken to be the composite functor
      \[
      \delta^*_{I,J}R\simeq \delta_{2,\{2\}}^*(\del^\mu)^*R
      \xrightarrow{\delta_{2,\{2\}}^*u}\delta_{2,\{2\}}^*\delta^{2\square}_*(\Cat,\cT)\xrightarrow{\phi_{\{2\}}}\Cat.
      \]

  \item For the quadruple $(\{1\},\{1\},\PRR,\phi)$ where $\phi\colon\PRR\to\Cat$ is the natural inclusion, the functor $\phi_{\{1\}}$
      constructed in Proposition \ref{b1pr:partial_adjoint} induces an equivalence $\phi_{\PR}\colon (\PRR)^{op}\to\PRL$. This gives
      another proof of the second assertion of \cite{HTT}*{Corollary 5.5.3.4}. By restriction, this equivalence induces an equivalence
      $\phi_{\PS}\colon \PSL\to(\PSR)^{op}$ of $\infty$-categories.

  \item For the quadruple $(\{1,2\},\{1\},S^{op}\boxtimes\Fun^{\r{LAd}}(S^{op},\Cat),f)$ where
      \[
      f\colon S^{op}\times\Fun^{\r{LAd}}(S^{op},\Cat)\to\Cat
      \]
      is the evaluation map, the functor
      \[
      f_{\{1\}}\colon S\times\Fun^{\r{LAd}}(S^{op},\Cat)\to\Cat
      \]
      constructed in Proposition \ref{b1pr:partial_adjoint} induces an 
      equivalence $\Fun^{\r{LAd}}(S^{op},\Cat)\to\Fun^{\r{RAd}}(S,\Cat)$. 
      This gives an alternative proof of \cite{HA}*{Corollary 4.7.4.18(3)}. 
\end{enumerate}
\end{remark}

\subsection{Symmetric monoidal $\infty$-categories}
\label{b1ss:symmetric_monoidal}

Let $\Fin$ be the category of pointed finite sets defined in \cite{HA}*{Notation 2.0.0.2}. It is (equivalent to) the category whose objects are
sets $\langle n\rangle=\langle n\rangle^\circ\cup\{\ast\}$, where $\langle n\rangle^\circ=\{1,\dots,n\}$ ($\langle 0\rangle^\circ=\emptyset$) for $n\geq0$, and morphisms are maps of sets that map $\ast$ to $\ast$.

Let $\cC$ be an $\infty$-category that admits finite products. By \cite{HA}*{Proposition 2.4.1.5}, we have a symmetric monoidal $\infty$-category \cite{HA}*{Definition 2.0.0.7} $\cC^\times\to\rN(\Fin)$, known as the \emph{Cartesian symmetric monoidal $\infty$-category associated to $\cC$}. We put $\CAlg(\cC)\coloneqq\CAlg(\cC^\times)$ \cite{HA}*{Definition 2.1.3.1} as the $\infty$-category of commutative algebra objects in $\cC$. We have the functor
\begin{align}\label{b1eq:G}
\rG\colon\CAlg(\cC)\to\cC
\end{align}
by evaluating at $\langle1\rangle$.

\begin{remark}
In the above construction, if we put $\cC\coloneqq\Cat$, then $\CAlg(\Cat)$ 
is canonically equivalent to $\Cat^\otimes$, the $\infty$-category of 
symmetric monoidal $\infty$-categories \cite{HA}*{Variant 2.1.4.13}. The 
functor $\rG$ restricts to a functor $\Cat^\otimes\to\Cat$ sending 
$\cC^\otimes$ to its underlying $\infty$-category $\cC$. 
\end{remark}

Recall that a symmetric monoidal $\infty$-category $\cC^\otimes$ is {\em 
closed} \cite{HA}*{Definition 4.1.1.15} if the functor $-\otimes 
-\colon\cC\times\cC\to\cC$, written as $\cC\to\Fun(\cC,\cC)$, factorizes 
through $\FunL(\cC,\cC)$. 

\begin{definition}\label{b1de:presentable_monoidal}
We define a subcategory $\CAlg(\Cat)_\pr^\rL$ (resp.\ $\CAlg(\Cat)_{\r{pr,st}}^\rL$) of $\CAlg(\Cat)$ as follows:
\begin{itemize}
  \item An object that belongs to this subcategory is a symmetric monoidal $\infty$-categories $\cC^\otimes$ such that $\cC=\rG(\cC^\otimes)$ is presentable (resp.\ and stable).

  \item A morphism that belongs to this subcategory is a symmetric monoidal functor $F^\otimes\colon\cC^\otimes\to\cD^\otimes$ such that the underlying functor $F=\rG(F^\otimes)$ is a left adjoint functor.
\end{itemize}
In particular, we have functors
\[
\rG\colon\CAlg(\Cat)_\pr^\rL\to\PRL,\quad\rG\colon\CAlg(\Cat)_{\r{pr,st}}^\rL\to\PSL.
\]
Moreover, we define $\CAlg(\Cat)_{\r{cl}}\subseteq\CAlg(\Cat)$, $\CAlg(\Cat)_{\r{pr,cl}}^\rL\subseteq\CAlg(\Cat)_\pr^\rL$ and
$\PSLM\subseteq\CAlg(\Cat)_{\r{pr,st}}^\rL$ to be the full subcategories spanned by closed symmetric monoidal $\infty$-categories.
\end{definition}

\begin{remark}\label{b1re:monoidal}
The $\infty$-categories $\CAlg(\Cat)_{\r{pr,cl}}^\rL$ and $\PSLM$ admit small limits and such limits are preserved under the inclusions
\[
\PSLM\subseteq\CAlg(\Cat)_{\r{pr,cl}}^\rL\subseteq\CAlg(\Cat).
\]
In fact, we only have to show that for a small simplicial set $S$ and a diagram $p^\otimes\colon S\to\CAlg(\PR^\rL)$ such that $p^\otimes(s)=\cC_s^\otimes$ is closed for every vertex $s$ of $S$, the limit $\lim(p^\otimes)$ is closed. Let $p\colon S\to\CAlg(\PR^\rL)\to\PRL$ (resp.\ $p'\colon S\to\CAlg(\PR^\rL)\to\Fun(\Delta^1,\Cat)$) be the diagram induced by evaluating at the object $\langle 1\rangle$ (resp.\ unique active map $\langle 2\rangle\to\langle 1\rangle$) of $\rN(\Fin)$. For every object $c$ of $\cC=\lim(p)$, the diagram $p'$ induces a diagram $p'_c\colon S\to\Fun(\Delta^1,\PRL)$ such that $p'_c(s)$ is the functor $f_s^*c\otimes-\colon \cC_s\to\cC_s$ that admits right adjoints, where $f_s^*\colon \cC\to\cC_s$ is the obvious functor. Since $\PRL\subseteq\Cat$ is stable under small limits, the limit $\lim(p'_c)$ is an object of $\FunL(\cC,\cC)$, which shows that the limit $\lim(p^\otimes)$ is closed.

A diagram $p\colon S^\triangleleft\to\PSLM$ is a limit diagram if and only if
\[
\rG\circ p\colon S^\triangleleft\to\PSLM\xrightarrow{\rG}\Cat
\]
is a limit diagram, by the dual version of \cite{HTT}*{Corollary 5.1.2.3}.
\end{remark}

Let $\cC$ be an $\infty$-category. Recall that by \cite{HA}*{Construction 2.4.3.1, Proposition 2.4.3.3}, we have an $\infty$-operad $p\colon\cC^\amalg\to\rN(\Fin)$. Suppose that $\cC$ is a fibrant simplicial category. We define $\cC^\amalg$ to be the fibrant simplicial category such that an object of $\cC^\amalg$ consists of an object $\langle n\rangle\in\Fin$ together with a sequence of objects $(Y_1,\dots,Y_n)$ in $\cC$, and
\[
\Map_{\cC^\amalg}((X_1,\dots,X_m),(Y_1,\dots,Y_n))=\coprod_\alpha\prod_{i\in\alpha^{-1}\langle n\rangle^\circ}\Map_\cC(X_i,Y_{\alpha(i)}),
\]
where $\alpha$ runs through all maps of pointed sets from $\langle m\rangle$ to $\langle n\rangle$. By construction, we have a forgetful functor $\cC^\amalg\to\Fin$, and its simplicial nerve $\rN(\cC^\amalg)\to\rN(\Fin)$ is canonically isomorphic to $\rN(\cC)^\amalg\to\rN(\Fin)$.

\begin{definition}\label{b1de:static}
Let $p\colon\cC\to\rN(\Fin)$ be a functor of $\infty$-categories. We say that a diagram in $\cC$ is \emph{$p$-static} (or simply \emph{static} if $p$ is clear) if its composition with $p$ is constant.
\end{definition}

\begin{lem}\label{b1le:static_colimit}
Let $\cC$ be an $\infty$-category that admits finite colimits. Then a square
\begin{align*}
\xymatrix{
(X_1,\dots,X_m) \ar[r]\ar[d] & (Y_1,\dots,Y_n) \ar[d] \\
(X'_1,\dots,X'_m) \ar[r] & (Y'_1,\dots,Y'_n) }
\end{align*}
in $\cC^\amalg$ with static vertical morphisms is a pushout square if and only if for every $1\leq j\leq n$, the induced square
\begin{align*}
\xymatrix{
\coprod_{\alpha(i)=j}X_i \ar[r]\ar[d] & Y_j \ar[d] \\
\coprod_{\alpha(i)=j}X'_i \ar[r] & Y'_j
}
\end{align*}
in $\cC$ is a pushout square.
\end{lem}

\begin{proof}
It follows from the fact that for every pair of objects $\{X_i\}_{1\leq i\leq m}$, $\{Y_j\}_{1\leq j\leq m}$ of $\cC^\amalg$, the mapping space
$\Map_{\cC^\amalg}(\{X_i\}_{1\leq i\leq m},\{Y_j\}_{1\leq j\leq m})$ is naturally equivalent to
\[
\coprod_{\alpha\in\Hom_{\Fin}(\langle m\rangle,\langle n\rangle)}\prod_{i\in\alpha^{-1}\langle n\rangle^\circ}\Map_\cC(X_i,Y_{\alpha(i)}),
\]
and the discussion in \cite{HTT}*{{\Sec}4.4.2}.
\end{proof}

\begin{remark}\label{b1re:cartesian}
Let $\bT\colon \cC^\amalg\to\Cat$ be a functor that is a \emph{lax Cartesian 
structure} \cite{HA}*{Definition 2.4.1.1}. Then we have an induced 
$\infty$-operad map $\bT^\otimes\colon\cC^\amalg\to\Cat^\times$ 
\cite{HA}*{Proposition 2.4.1.7}, which is an object of 
$\Alg_{\cC^\amalg}(\Cat^\times)$. The choice of such $\bT^\otimes$ is 
parameterized by a trivial Kan complex. Since the obvious map 
$\Alg_{\cC^\amalg}(\Cat^\times)\to\Fun(\cC,\CAlg(\Cat))$ is a trivial Kan 
fibration \cite{HA}*{Theorem 2.4.3.18}, in what follows, we will regard 
$\bT^\otimes$ as a functor $\cC\to\CAlg(\Cat)$. 
\end{remark}

\section{Enhanced operations for ringed topoi and schemes}
\label{b2ss}

In this chapter, we construct the enhanced operation maps for the category of ringed topoi and for the category of coproducts of quasi-compact and separated schemes, and establish several properties of the maps.

The construction is based on the flat model structure. This marks a major 
difference with the study of quasi-coherent sheaves. For the latter one can 
simply start with the projective model structure constructed in 
\cite{HA}*{Remark 7.1.2.9}, because the category of quasi-coherent sheaves on 
affine schemes have enough projectives. The flat model structure for a ringed 
topological space has been constructed by Gillespie in \cite{Gil1} and 
\cite{Gil2}. In \Sec\ref{b2ss:flat_model_structure}, we adapt the 
construction to every topos with enough points.

In \Sec\ref{b2ss:enhanced_operations}, we construct a functor $\bT$ \eqref{b2eq:enhanced_premonoidal} and its induced functor $\bT^\otimes$
\eqref{b2eq:enhanced_monoidal} that enhance the derived $*$-pullback and derived tensor product for ringed topoi. It also encodes the symmetric
monoidal structures in a homotopy-coherent way. This serves as a starting point for the construction of the enhanced operation map.

In \Sec\ref{b3ss:abstract}, we introduce an abstract notion of (universal) descent and collect some basic properties. In \Sec\ref{b3ss:enhanced_operation}, we construct the enhanced operation maps \eqref{b3eq:premonoidal} and \eqref{b3eq:operation} based on the ones constructed for ringed topoi. In \Sec\ref{b3ss:poincare}, we establish some properties of the maps constructed in the previous section, including an enhanced version of (co)homological descent for smooth coverings. This property is crucial for the extension of the enhanced operation map to algebraic spaces and stacks in Chapter~\ref{b5ss}.

\subsection{The flat model structure}
\label{b2ss:flat_model_structure}

Let $(X,\cO_X)$ be a ringed topos. In other words, $X$ is a (Grothendieck) topos and $\cO_X$ is a sheaf of rings in $X$. An $\cO_X$-module $C$ is called \emph{cotorsion} if $\Ext^1(F,C)=0$ for every flat $\cO_X$-module $F$. The following definition is a special case of \cite{Gil2}*{Definition 2.1}.

\begin{definition}
Let $K$ be a cochain complex of $\cO_X$-modules.
\begin{itemize}
  \item $K$ is called a \emph{flat complex} if it is exact and $Z^n K$ is flat for all $n$.

  \item $K$ is called a \emph{cotorsion complex} if it is exact and $Z^n K$ is cotorsion for all $n$.

  \item $K$ is called a \emph{dg-flat complex} if $K^n$ is flat for every $n$, and every cochain map $K\to C$, where $C$ is a cotorsion complex, is homotopic to zero.

  \item $K$ is called a \emph{dg-cotorsion complex} if $K^n$ is cotorsion for every $n$, and every cochain map $F\to K$, where $F$ is a flat
      complex, is homotopic to zero.
\end{itemize}
\end{definition}

\begin{lem}
Let $(f,\gamma)\colon (Y,\cO_Y)\to (X,\cO_X)$ be a morphism of ringed topoi. Then
\begin{enumerate}
  \item $(f,\gamma)^*$ preserves flat modules, flat complexes, and dg-flat complexes;

  \item $(f,\gamma)_*$ preserves cotorsion modules, cotorsion complexes, and dg-cotorsion complexes.
\end{enumerate}
\end{lem}

Recall that the functor $(f,\gamma)^*=\cO_Y\otimes_{f^*\cO_X} f^*-\colon\Mod(X,\cO_X)\to \Mod(Y,\cO_Y)$ is a left adjoint of the functor
$(f,\gamma)_*\colon \Mod(Y,\cO_Y)\to\Mod(X,\cO_X)$.

\begin{proof}
Let $F\in \Mod(X,\cO_X)$ be flat, and $C\in \Mod(Y,\cO_Y)$ cotorsion. We have a monomorphism $\Ext^1(F,(f,\gamma)_*C)\to\Ext^1((f,\gamma)^*F,C)=0$. Thus, $(f,\gamma)_*C$ is cotorsion. Moreover, since short exact sequences of
cotorsion $\cO_Y$-modules are exact as sequences of presheaves, $(f,\gamma)_*$ preserves short exact sequences of cotorsion modules, hence it
preserves cotorsion complexes. It follows that $(f,\gamma)^*$ preserves dg-flat complexes.

It is well known that $(f,\gamma)^*$ preserves flat modules and short exact sequences of flat modules. It follows that $(f,\gamma)^*$ preserves flat complexes and hence $(f,\gamma)_*$ preserves dg-cotorsion complexes.
\end{proof}

The model structure in the following generalization of \cite{Gil2}*{Corollary 7.8} is called the \emph{flat model structure}.

\begin{proposition}\label{b2pr:flat_model}
Assume that $X$ has enough points. Then there exists a combinatorial model structure on $\Ch(\Mod(X,\cO_X))$ such that
\begin{itemize}
  \item The cofibrations are the monomorphisms with dg-flat cokernels.

  \item The fibrations are the epimorphisms with dg-cotorsion kernels.

  \item The weak equivalences are quasi-isomorphisms.
\end{itemize}
Furthermore, this model structure is monoidal with respect to the usual tensor product of chain complexes.
\end{proposition}

For a morphism $(f,\gamma)\colon (Y,\cO_Y)\to (X,\cO_X)$ of ringed topoi with enough points, the pair of functors $((f,\gamma)^*,(f,\gamma)_*)$ is a Quillen adjunction between the categories $\Ch(\Mod(Y,\cO_Y))$ and $\Ch(\Mod(X,\cO_X))$ endowed with the flat model structures.

\begin{remark}\label{b2re:flat_model}
We have the following remarks about different model structures.
\begin{enumerate}
  \item The functor $\id\colon \Ch(\Mod(X,\cO_X))^{\r{flat}}\to\Ch(\Mod(X,\cO_X))^{\r{inj}}$ is a right Quillen equivalence. Here
      $\Ch(\Mod(X,\cO_X))^{\r{flat}}$ (resp.\ $\Ch(\Mod(X,\cO_X))^{\r{inj}}$) is the model category $\Ch(\Mod(X,\cO_X))$ endowed with the flat model structure (resp.\ the injective model structure \cite{HA}*{Proposition 1.3.5.3}).

  \item If $X=*$ and $\cO_X=R$ is a (commutative) ring, then $\id\colon\Ch(\Mod(*,R))^{\r{proj}} \to \Ch(\Mod(*,R))^{\r{flat}}$ is a
      \emph{symmetric monoidal} left Quillen equivalence between symmetric monoidal model categories. Here $\Ch(\Mod(*,R))^{\r{proj}}$ is the model category $\Ch(\Mod(*,R))$ endowed with the (symmetric monoidal) projective model structure \cite{HA}*{Proposition 7.1.2.11}.
\end{enumerate}
\end{remark}

To prove Proposition \ref{b2pr:flat_model}, we adapt the proof of \cite{Gil2}*{Corollary 7.8}. Let $S$ be a site, and $G$ a small topologically
generating family \cite{SGA4}*{Expos\'{e} ii, D\'{e}finition 3.0.1} of $S$. For a presheaf $F$ on $S$, we put $\lvert F\rvert_G\coloneqq\sup_{U\in G}\card(F(U))$.

\begin{lem}
Let $\beta\ge \card(G)$ be an infinite cardinal such that $\beta\ge\card(\Hom(U,V))$ for all $U$ and $V$ in $G$, and $\kappa$ a cardinal $\ge
2^\beta$. Let $F$ be a presheaf on $S$ such that $\lvert F\rvert_G\le\kappa$, and $F^+$ the sheaf associated to $F$. Then $\lvert F^+\rvert_G\le\kappa$.
\end{lem}

\begin{proof}
By the construction in \cite{SGA4}*{Expos\'{e} ii, D\'{e}finition 3.5}, we have $F^+=LLF$, where
\[
(LF)(U)=\colim_{R\in J(U)} \Hom_{\hat{S}}(R,F)
\]
for $U\in S$ in which $J(U)$ is the set of sieves covering $U$ and $\hat{S}$ is the category of presheaves on $S$. By \cite{SGA4}*{Expos\'{e} ii, Proposition 3.0.4} and its proof, $\lvert LF\rvert_G\le\beta^2 \kappa^{\beta^2}=\kappa$.
\end{proof}

Let $\cO_S$ be a sheaf of rings on $S$. For an element $U\in S$, we denote by $j_{U!}$ the left adjoint of the restriction functor $\Mod(S,\cO_S)\to\Mod(U,\cO_U)$. Using the fact that $(j_{U!}\cO_U)_{U\in G}$ is a family of flat generators of $\Mod(S,\cO_S)$, we have the following analogue of \cite{Gil2}*{Lemma 7.7} with essentially the same proof.

\begin{lem}\label{b2le:cardinal}
Let $\beta\ge \card(G)$ be an infinite cardinal such that $\beta\ge\card(\Hom(U,V))$ for all $U$ and $V$ in $G$. Let $\kappa\ge\max\{2^\beta,\lvert \cO_S\rvert_G\}$ be a cardinal such that $j_{U!}\cO_U$ is $\kappa$-generated for every $U$ in $G$. Then the following conditions are equivalent for an $\cO_S$-module $F$:
\begin{enumerate}
  \item $\lvert F\rvert_G\le \kappa$;

  \item $F$ is $\kappa$-generated;

  \item $F$ is $\kappa$-presentable.
\end{enumerate}
\end{lem}

Let $F$ be an $\cO_S$-premodule. We say that an $\cO_S$-subpremodule $E\subseteq F$ is \emph{$G$-pure} if $E(U)\subseteq F(U)$ is pure for every $U$ in $G$. This implies that $E^+\subseteq F^+$ is pure. As in \cite{EO}*{Proposition 2.4}, one proves the following.

\begin{lem}
Let $\beta\ge \card(G)$ be an infinite cardinal such that $\beta\ge\card(\Hom(U,V))$ for all $U$ and $V$ in $G$. Let $\kappa\ge\max\{2^\beta,\lvert \cO_S\rvert_G\}$ be a cardinal, and let $E\subseteq F$ be $\cO_S$-premodules such that $\lvert E\rvert_G\le \kappa$. Then there exists a $G$-pure $\cO_S$-subpremodule $E'$ of $F$ containing $E$ such that $\lvert E'\rvert_G\le \kappa$.
\end{lem}

\begin{proof}[Proof of Proposition \ref{b2pr:flat_model}]
We choose a site $S$ of $X$, and a small topologically generating family $G$, and a cardinal $\kappa$ satisfying the assumptions of Lemma
\ref{b2le:cardinal}. Using the previous lemmas, one shows as in the proof of \cite{Gil2}*{Corollary 7.8} that the conditions of \cite{Gil2}*{Theorem~4.12 \& Theorem~5.1} are satisfied for $\kappa$, which finishes the proof.
\end{proof}

\begin{remark}\label{b2re:dg_flat}
Using the sheaves $i_*(\dQ/\dZ)$, where $i$ runs through points $P\to X$ of $X$, one can show as in \cite{Gil1}*{Proposition 5.6} that a complex $K$ of $\cO_X$-modules is dg-flat if and only if $K^n$ is flat for each $n$ and $K\otimes_{\cO_X} L$ is exact for each exact sequence $L$ of $\cO_X$-modules.
\end{remark}

\subsection{Enhanced operations for ringed topoi}
\label{b2ss:enhanced_operations}

Let us start by recalling the category of ringed topoi.

\begin{definition}\label{b2de:ringedptopos}
Let $\RingedPTopos$ be the $(2,1)$-category of ringed $\cU$-topoi in $\cV$ with enough points:
\begin{itemize}
  \item An object of $\RingedPTopos$ is a ringed topos $(X,\cO_X)$ such that $X$ has enough points.

  \item A morphism $(X,\cO_X)\to (X',\cO_{X'})$ in $\RingedPTopos$ is a morphism of ringed topoi in the sense of \cite{SGA4}*{Expos\'{e} iv, D\'{e}finition 13.3}, namely a pair $(f,\gamma)$, where $f\colon X\to X'$ is a morphism of topoi and $\gamma\colon f^*\cO_{X'}\to\cO_X$.

  \item A 2-morphism $(f_1,\gamma_1)\to(f_2,\gamma_2)$ in $\RingedPTopos$ is an equivalence $\epsilon\colon f_1\to f_2$ such that $\gamma_2$ equals the composition $f_2^*\cO_{X'}\xrightarrow{\epsilon^*} f_1^*\cO_{X'}\xrightarrow{\gamma_1}\cO_X$.

  \item Composition of morphisms and 2-morphisms are defined in the obvious way.
\end{itemize}
We sometimes simply write $X$ for an object of $\RingedPTopos$ if the structure sheaf is insensitive.
\end{definition}

Our goal in this section is to construct a functor
\begin{align}\label{b2eq:enhanced_premonoidal}
\bT\colon\rN(\RingedPTopos^{op})^\amalg\to\Cat
\end{align}
that is a lax Cartesian structure such that the induced functor $\bT^\otimes$ 
(see Remark \ref{b1re:cartesian}) factorizes through 
$\PSLM\subseteq\CAlg(\Cat)$. In other words, we have the induced functor 
\begin{align}\label{b2eq:enhanced_monoidal}
\bT^\otimes\colon\rN(\RingedPTopos^{op})\to\PSLM,
\end{align}
where $\PSLM$ is defined in Definition \ref{b1de:presentable_monoidal}.

Let $\cat^+$ be the $(2,1)$-category of \emph{marked categories}, namely pairs $(\cC,\cE)$ consisting of an (ordinary) category $\cC$ and a set of arrows $\cE$ containing all identity arrows. We have a simplicial functor $\cat^+\to \Mset$ sending $(\cC,\cE)$ to $(\rN(\cC),\cE)$. We start by constructing a pseudofunctor
\[
\rT\colon(\RingedPTopos^{op})^\amalg\to\cat^+.
\]
Recall that to every object $X\in\RingedPTopos$, we can associate a marked simplicial set
\[
(\rN(\Ch(\Mod(X))_{\dgflat}),W(X)),
\]
where $\Ch(\Mod(X))_{\dgflat}\subseteq\Ch(\Mod(X))$ is the full subcategory spanned by the dg-flat complexes, and $W(X)$ is the set of
quasi-isomorphisms. We define the image of an object $(X_1,\dots,X_m)$ under $\rT$ to be
\[
\prod_{i=1}^m(\Ch(\Mod(X_i))_{\dgflat},W(X_i)).
\]
By definition, a (1-)morphism $f\colon(X_1,\dots,X_m)\to(Y_1,\dots,Y_n)$ in $(\RingedPTopos^{op})^\amalg$ consists of a map $\alpha\colon\langle m\rangle\to\langle n\rangle$ and a morphism $f_i\colon Y_{\alpha(i)}\to X_i$ in $\RingedPTopos$ for every $i\in \alpha^{-1}\langle n\rangle^\circ$. Now we define the image of $f$ under $\rT$ to be the functor
\begin{align*}
\prod_{i=1}^m(\Ch(\Mod(X_i))_{\dgflat},W(X_i))&\to\prod_{j=1}^n(\Ch(\Mod(Y_j))_{\dgflat},W(Y_j)) \\
\{K_i\}_{1\leq i\leq m} &\mapsto\left\{\bigotimes_{\alpha(i)=j}f_i^*K_i\right\}_{1\leq j\leq n},
\end{align*}
where we take the unit object as the tensor product over an empty set. The image of 2-morphisms are defined in the obvious way. Composing with the simplicial functor $\cat^+\to\Sset^+\xrightarrow{\Fibr}(\Sset^+)^\circ$ and taking nerves, we obtain the desired functor $\bT$
\eqref{b2eq:enhanced_premonoidal}.

\begin{lem}\label{b2le:closed_cartesian}
We have that
\begin{enumerate}
  \item the functor $\bT$ is a lax Cartesian structure 
      \cite{HA}*{Definition 2.4.1.1}; 

  \item the functor $\bT^\otimes$ factorizes through $\PSLM$; and

  \item the functor $\bT^\otimes$ sends small coproducts to products.
\end{enumerate}
\end{lem}

\begin{proof}
Part (1) is clear from the construction.

For (2), we note that for an object $X$ of $\RingedPTopos$, its image under $\bT$, denoted by $\cD(X)$, is the fibrant replacement of
$(\rN(\Ch(\Mod(X))_{\dgflat}),W(X))$. In particular, by Remark \ref{b2re:flat_model}(1) and \cite{HA}*{Remark 1.3.4.16, Proposition 1.3.5.15}, $\cD(X)$ is equivalent to the derived $\infty$-category of $\Mod(X)$ defined in \cite{HA}*{Definition 1.3.5.8}. It is a presentable stable $\infty$-category by \cite{HA}*{Propositions 1.3.5.9, 1.3.5.21(1)}. Combining this with Lemma \ref{b1le:adjoint_infinity}, we deduce that the image of $\bT^\otimes$ is actually contained in $\PSLM$. This proves part (2).

Part (3) follows from the construction and Remark \ref{b1re:monoidal}.
\end{proof}

\begin{notation}\label{b2no:monoidal}
For an object $X$ of $\RingedPTopos$, we denote the image of $X$ under $\bT^\otimes$ by $\cD(X)^\otimes$, which is a symmetric monoidal
$\infty$-category, whose underlying $\infty$-category is denoted by $\cD(X)$ as in the proof of the previous lemma.
\end{notation}

\begin{remark}\label{b2re:topos}
We have the following remarks.
\begin{enumerate}
  \item The $\infty$-category $\bT((X_1,\dots,X_m))$ is equivalent to $\prod_{i=1}^m\cD(X_i)$.


  \item By Remark \ref{b2re:flat_model}(2) and \cite{HA}*{Remark 4.1.7.5}, 
      for every (commutative) ring $R$, $\cD(*,R)^\otimes$ is equivalent to 
      the symmetric monoidal $\infty$-category $\cD(\Ch(R))^\otimes$ 
      defined in \cite{HA}*{Remark 7.1.2.12}. 

  \item Let $f\colon X\to X'$ be a morphism of $\RingedPTopos$. It follows from Remark \ref{b2re:dg_flat} and \cite{KS}*{Lemma 14.4.1, Theorem 18.6.4} that the functors $f^*\colon \rD(X')\to \rD(X)$ and $-\otimes_X-\colon \rD(X)\times \rD(X) \to\rD(X)$ induced by $\bT^\otimes$ are equivalent to the respective functors constructed in \cite{KS}*{{\Sec}18.6}, where $\rD(X)=\rh\cD(X)$ and $\rD(X')=\rh\cD(X')$.
\end{enumerate}
\end{remark}

Let $\Ring$ be the category of (small commutative) rings. To deal with torsion and adic coefficients simultaneously. We introduce the category
$\Rind$ of ringed diagrams as follows.

\begin{definition}[Ringed diagram]\label{b2de:ringed_diagram}
We define a category $\Rind$ as follows:
\begin{itemize}
  \item An object of $\Rind$ is a pair $(\Xi,\Lambda)$, called a \emph{ringed diagram}, where $\Xi$ is a small partially ordered set and $\Lambda\colon \Xi^{op}\to \Ring$ is a functor. We identify $(\Xi,\Lambda)$ with the topos of presheaves on $\Xi$, ringed by $\Lambda$. A typical example is $(\mathbb{N},n\mapsto\dZ/\ell^{n+1}\dZ)$ with transition maps given by projections.

  \item A morphism of ringed diagrams $(\Xi',\Lambda')\to (\Xi,\Lambda)$ is a pair $(\Gamma,\gamma)$ where $\Gamma\colon \Xi'\to \Xi$ is a
      functor (that is, an order-preserving map) and $\gamma\colon\Gamma^*\Lambda\colonequals \Lambda\circ \Gamma^{op} \to \Lambda'$ is a morphism of $\Ring^{\Xi'^{op}}$.
\end{itemize}
For an object $(\Xi,\Lambda)$ of $\Rind$ and an object $\xi$ of $\Xi$, we define the \emph{over ringed diagram} $(\Xi,\Lambda)_{/\xi}$ to be the ringed diagram whose underlying category is $\Xi_{/\xi}$ and the corresponding functor is $\Lambda_{/\xi}\colonequals\Lambda\res\Xi_{/\xi}$.
\end{definition}

For a topos $X$ and a small partially ordered set $\Xi$, we denote by $X^\Xi$ the topos $\Fun(\Xi^{op},X)$. If $(\Xi,\Lambda)$ is a ringed diagram, then $\Lambda$ defines a sheaf of rings on $X^\Xi$, which we still denote by $\Lambda$. We thus obtain a pseudofunctor
\begin{align}\label{b2eq:topoi}
\PTopos\times \Rind\to \RingedPTopos
\end{align}
carrying $(X,(\Xi,\Lambda))$ to $(X^\Xi,\Lambda)$, where $\PTopos$ is the $(2,1)$-category of ringed topoi with enough points. Composing the nerve of \eqref{b2eq:topoi} with $\bT$ \eqref{b2eq:enhanced_premonoidal}, we obtain a functor
\begin{align}\label{b2eq:topos_operation}
\EO{}{\PTopos}{\rI}{}\colon(\rN(\PTopos)^{op}\times\rN(\Rind)^{op})^\amalg\to\Cat
\end{align}
that is a lax Cartesian structure.

\begin{notation}\label{b2no:image}
By abuse of notation, we denote by $\cD(X,\lambda)^\otimes$ the image of an object $(X,\lambda)$ of $\PTopos\times\Rind$ under the induced functor
\[
\EO{}{\PTopos}\otimes{}\coloneqq(\EO{}{\PTopos}{\rI}{})^\otimes\colon\rN(\PTopos)^{op}\times\rN(\Rind)^{op}\to\PSLM,
\]
whose underlying $\infty$-category is denoted by $\cD(X,\lambda)$ which is (equivalent to) the image of $(X,\lambda,\langle1\rangle,\{1\})$
under the functor $\EO{}{\PTopos}{\rI}{}$.
\end{notation}

\begin{definition}\label{b2de:perfect}
A morphism $(\Gamma,\gamma)\colon (\Xi',\Lambda')\to (\Xi,\Lambda)$ of $\Rind$ is said to be \emph{perfect} if for every $\xi\in \Xi'$, $\Lambda'(\xi)$ is a perfect complex in the derived category of $\Lambda(\Gamma(\xi))$-modules.
\end{definition}

\begin{lem}\label{b2le:upperstar_pi}
Let $f\colon Y\to X$ be a morphism of $\PTopos$, and $\pi\colon\lambda'\to\lambda$ a perfect morphism of $\Rind$. Then the square
\begin{align}\label{b2eq:upperstar_pi}
\xymatrix{\cD(Y,\lambda') & \cD(X,\lambda')\ar[l]_{f^*}\\
\cD(Y,\lambda)\ar[u]^{\pi^*} & \cD(X,\lambda)\ar[u]_{\pi^*}\ar[l]_{f^*}}
\end{align}
is right adjointable and its transpose is left adjointable.
\end{lem}

\begin{proof}
Write $\lambda=(\Xi,\Lambda)$ and $\lambda'=(\Xi',\Lambda')$. For $\xi\in 
\Xi'$, we denote by $e_\xi$ the natural morphism 
$(\{\xi\},\Lambda'(\xi))\to(\Xi',\Lambda')$. We show that 
\eqref{b2eq:upperstar_pi} is right adjointable and $\pi^*$ preserves small 
limits. As the family of functors $(e_\xi^*)_{\xi\in \Xi'}$ is conservative, 
it suffices to show these assertions with $\pi$ replaced by $e_\xi$ and by 
$\pi\circ e_\xi$. In other words, we may assume $\Xi'=\{\ast\}$. We decompose 
$\pi$ as 
\[
(\{\ast\},\Lambda')\xrightarrow{t}(\{\zeta\},\Lambda(\zeta)) \xrightarrow{s}(\Xi,\Lambda)_{/\zeta}\xrightarrow{i}(\Xi,\Lambda).
\]
We show that the assertions hold with $\pi^*$ replaced by $i^*$, by $s^*$, and by $t^*$. The assertions for $i^*$ follow from Lemma
\ref{b2le:topoi_adjointability} below. The assertions for $s^*$ are trivial as $s^*\simeq p_*$, where $p\colon(\Xi,\Lambda)_{/\zeta}\to(\{\zeta\},\Lambda(\zeta))$. As $t_*$ is conservative, the assertions for $t^*$ follow from the assertions for $t_*$ and $t_*t^*-\simeq\HOM_{\Lambda(\zeta)}({\Lambda'}\spcheck,-)$, which are trivial. Here we used the fact that for any perfect complex $M$ in the derived category of $\Lambda(\zeta)$-modules, the natural transformation $M\otimes_{\Lambda(\zeta)}-\to\HOM_{\Lambda(\zeta)}(M\spcheck,-)$ is a natural equivalence, where ${M}\spcheck=\Hom_{\Lambda(\zeta)}(M,{\Lambda(\zeta)})$. This applies to $M=\Lambda'$ by the assumption that $\pi$ is perfect.
\end{proof}

\begin{lem}\label{b2le:topoi_adjointability}
Let $f\colon (X',\Lambda')\to (X,\Lambda)$ be a morphism of ringed topoi, and $j\colon V\to U$ a morphism of $X$. Put $j'\coloneqq f^{-1}(j)\colon V'=f^{-1}(V)\to f^{-1}(U)=U'$. Then the square
\[
\xymatrix{
\cD(X_{/U},\Lambda\times U) \ar[r]^{j^*} \ar[d]_-{f_{/U}^*} & \cD(X_{/V},\Lambda\times V) \ar[d]^-{f_{/V}^*} \\
\cD(X'_{/U'},\Lambda'\times U') \ar[r]^{j'^*}  & \cD(X'_{/V'},\Lambda'\times V')
}
\]
is left adjointable and its transpose is right adjointable.
\end{lem}

\begin{proof}
The functor $j_!\colon \Mod(X_{/V},\Lambda\times V)\to\Mod(X_{/U},\Lambda\times U)$ is exact and induces a functor $\cD(X_{/V},\Lambda\times V)\to \cD(X_{/U},\Lambda\times U)$, left adjoint of $j^*$. The same holds for $j'_!$. The first assertion of the lemma follows from the existence of these left adjoints and the second assertion. The second assertion follows from the fact that $j'^*$ preserves fibrant objects
in $\Ch(\Mod(-))^{\r{inj}}$.
\end{proof}

\begin{remark}\label{3re:Funequiv}
Let $\Xi$ be a poset and let $\Lambda$ be a ring. Let $\Lambda_\Xi\colon 
\Xi^\op \to \Lambda$ be the constant functor of value $\Lambda$ and let 
$\rho\colon (\Xi,\Lambda_\Xi)\to (*,\Lambda)$ be the obvious morphism of 
ringed diagrams. By Remark \ref{b2re:flat_model}(1) and \cite[Proposition 
1.3.4.25]{HA}, for any topos $X$ with enough points, we have an equivalence 
of $\infty$-categories $\cD(X,(\Xi,\Lambda_\Xi))\to 
\Fun(\rN(\Xi^\op),\cD(X,\Lambda))$, via which $\rho^*$ can be identified with 
the diagonal embedding $\cD(X,\Lambda)\to \Fun(\rN(\Xi^\op),\cD(X,\Lambda))$. 
\end{remark}

\subsection{Abstract descent properties}
\label{b3ss:abstract}

We start from the definition of morphisms with descent properties.

\begin{definition}[$F$-descent]\label{b3de:descent}
Let $\cC$ be an $\infty$-category admitting pullbacks, $F\colon \cC^{op}\to\cD$ a functor of $\infty$-categories, and $f\colon X^+_0\to X^+_{-1}$ a morphism of $\cC$. We say that $f$ is \emph{of $F$-descent} if $F\circ(X^+_\bullet)^{op}\colon \rN(\del_+)\to \cD$ is a limit diagram in $\cD$, where $X^+_\bullet\colon \rN(\del_+)^{op}\to \cC$ is a \v{C}ech nerve of $f$ (see the definition after \cite{HTT}*{Proposition 6.1.2.11}). We say that $f$ is \emph{of universal $F$-descent} if every pullback of $f$ in $\cC$ is of $F$-descent. Dually, for a functor $G\colon \cC\to\cD$, we say that $f$ is \emph{of $G$-codescent} (resp.\ \emph{of universal $G$-codescent}) if it is of $G^{op}$-descent (resp.\ of universal $G^{op}$-descent).
\end{definition}

We say that a morphism $f$ of an $\infty$-category $\cC$ is a \emph{retraction} if it is a retraction in the homotopy category $\rh\cC$. Equivalently, $f$ is a retraction if it can be completed into \emph{a weak retraction diagram} \cite{HTT}*{Definition 4.4.5.4} $\Ret\to\cC$ of $\cC$, corresponding to a 2-cell of $\cC$ of the form
\[
\xymatrix{
& Y \ar[dr]^-{f} \\
X \ar[ur]^-{s} \ar[rr]^-{\id_X} && X. }
\]
The following is an $\infty$-categorical version of \cite{Giraud}*{Propositions 10.10, 10.11} (for ordinary descent) and \cite{SGA4}*{Expos\'{e} vbis, Proposition 3.3.1} (for cohomological descent).

\begin{lem}\label{b3le:descent}
Let $\cC$ be an $\infty$-category admitting pullbacks, and $F\colon\cC^{op}\to \cD$ a functor of $\infty$-categories. Then
\begin{enumerate}
  \item Every retraction $f$ in $\cC$ is of universal $F$-descent.

  \item Let
      \begin{align}\label{b3eq:descent}
      \xymatrix{
      W\ar[r]^-{g}\ar[d]_-{q} & Z\ar[d]^-{p}\\ Y\ar[r]^-{f}& X}
      \end{align}
      be a pullback diagram in $\cC$ such that the base change of $f$ to $(Z/X)^i$ is of $F$-descent for $i\ge 0$ and the base change of $p$ to $(Y/X)^j$ is of $F$-descent for $j\ge 1$. Then $p$ is of $F$-descent.

  \item Let
      \[
      \xymatrix{&Y\ar[rd]^-{f}\\
      Z\ar[rr]^-{h}\ar[ur]^-{g} && X}
      \]
      be a $2$-cell of $\cC$ such that $h$ is of universal $F$-descent. Then $f$ is of universal $F$-descent.

  \item Let
      \[
      \xymatrix{&Y\ar[rd]^-{f}\\
      Z\ar[rr]^-{h}\ar[ur]^-{g} && X}
      \]
      be a $2$-cell of $\cC$ such that $f$ is of $F$-descent and $g$ is of universal $F$-descent. Then $h$ is of $F$-descent.
\end{enumerate}
\end{lem}

The assumptions on $f$ and $p$ in (2) are satisfied if $f$ is of $F$-descent and $g$ and $q$ are of universal $F$-descent.

\begin{proof}
For (1), it suffices to show that $f$ is of $F$-descent. Consider the map $\rN(\del_+)^{op}\times\Ret\to\cC$, right Kan extension along the inclusion
\[
K=\{[-1]\}\times\Ret\coprod_{\{[-1]\}\times\{\emptyset\}}\rN(\del_+^{\le0})^{op}\times\{\emptyset\}
\subseteq\rN(\del_+)^{op}\times\Ret
\]
of the map $K\to \cC$ corresponding to the diagram
\[
\xymatrix{
Y\ar[rr]^-{\id_Y}\ar[dd]_-f&& Y \ar[dd]^-{f} \\
& Y \ar[dr]^-{f} \\
X \ar[ur]^-{s} \ar[rr]^-{\id_X} && X.}
\]
Then by \cite{HA}*{Corollary 4.7.2.9}, the \v{C}ech nerve of $f$ is split. 
Therefore, the assertion follows from the dual version of \cite{HTT}*{Lemma 
6.1.3.16}. 

For (2), let $X^+_{\bullet\bullet}\colon\rN(\del_+)^{op}\times\rN(\del_+)^{op}\to \cC$ be an augmented bisimplicial object of $\cC$ such that $X^+_{\bullet\bullet}$ is a right Kan extension of \eqref{b3eq:descent}, considered as a diagram $\rN(\del_+^{\le 0})^{op}\times\rN(\del_+^{\le 0})^{op}\to\cC$. By assumption, $F\circ(X_{i\bullet}^+)^{op}$ is a limit diagram in $\cD$ for $i\ge -1$ and $F\circ (X_{\bullet j}^+)^{op}$ is a limit diagram in $\cD$ for $j\ge 0$. By the dual version of \cite{HTT}*{Lemma 5.5.2.3}, $F\circ(X_{\bullet -1}^+)^{op}$ is a limit diagram in $\cD$, which proves (2) since $X_{\bullet -1}^+$ is a \v{C}ech nerve of $p$.

For (3), it suffices to show that $f$ is of $F$-descent. Consider the diagram
\begin{align}\label{b3eq:descent_diagram}
\xymatrix{
Z\ar@/_1pc/[rdd]_-{g}\ar[rd]\ar@/^1pc/[rrd]^-{\id_Z}\\
&Y\times_X Z\ar[r]_-{\pr_Z}\ar[d]^-{\pr_Y} & Z\ar[d]^-{h}\\
&Y\ar[r]^-{f}& X}
\end{align}
in $\cC$. Since $\pr_Z$ is a retraction, it is of universal $F$-descent by (1). It then suffices to apply (2).

For (4), consider the diagram \eqref{b3eq:descent_diagram}. By (3), $\pr_Y$ is of universal $F$-descent. It then suffices to apply (2).
\end{proof}

Next, we prove a descent lemma for general topoi. Let $X$ be a topos that has enough points, with a fixed final object $e$. Let $u_0\colon U_0\to e$ be a covering, which induces a hypercovering $u_\bullet\colon U_\bullet\to e$ by taking the \v{C}ech nerve. Let $\Lambda$ be a sheaf of rings in $X$, and put $\Lambda_n\coloneqq\Lambda\times U_n$. In particular, we obtain an augmented simplicial ringed topoi $(X_{/U_\bullet},\Lambda_\bullet)$, where $U_{-1}=e$ and $\Lambda_{-1}=\Lambda$. Suppose that for every $n\geq -1$, we are given a strictly full subcategory $\cC_n$ ($\cC=\cC_{-1}$) of $\Mod(X_{/U_n},\Lambda_n)$ such that for every morphism $\alpha\colon [m]\to[n]$ of $\del_+$, $u_\alpha^*\colon\Mod(X_{/U_n},\Lambda_n)\to\Mod(X_{/U_m},\Lambda_m)$ sends $\cC_n$ to $\cC_m$. Then, applying the functor $\rG\circ\bT^\otimes$ \eqref{b2eq:enhanced_monoidal}, we obtain an augmented cosimplicial $\infty$-category $\cD_{\cC_\bullet}(X_{/U_\bullet},\Lambda_\bullet)$, where
$\cD_{\cC_n}(X_{/U_n},\Lambda_n)$ is the full subcategory of $\cD(X_{/U_n},\Lambda_n)$ spanned by complexes whose cohomology sheaves belong to $\cC_n$.

\begin{lem}\label{b3le:topos_descent}
Assume that for every object $\sF$ of $\Mod(X,\Lambda)$ such that $u^*_{d^0_0}\sF$ belongs to $\cC_0$, we have $\sF\in\cC$. Then the natural
map
\[
\cD_\cC(X,\Lambda)\to\lim_{n\in\del}\cD_{\cC_n}(X_{/U_n},\Lambda_n)
\]
is an equivalence of $\infty$-categories.
\end{lem}

\begin{proof}
We first consider the case where $\cC_n=\Mod(X_{/U_n},\Lambda_n)$ for 
$n\geq-1$. We apply \cite{HA}*{Corollary 4.7.5.3}: Assumption (1) follows 
from the fact that $u_{d^0_0}^*\colon 
\cD(X,\Lambda)\to\cD(X_{/U_0},\Lambda_0)$ is a morphism of $\PSL$; and the 
functor $u^*_{d^0_0}$ is conservative since $u_0$ is a covering. Therefore, 
it remains to check Assumption (2) of \cite{HA}*{Corollary 4.7.5.3}, that is, 
the left adjointability of the diagram 
\begin{align*}
\xymatrix{
\cD(X_{/U_m},\Lambda_m) \ar[r]^-{u_{d^{m+1}_0}^*} \ar[d]_-{u_\alpha^*}
& \cD(X_{/U_{m+1}},\Lambda_{m+1}) \ar[d]^-{u_{\alpha'}^*} \\
\cD(X_{/U_n},\Lambda_n) \ar[r]^-{u_{d^{n+1}_0}^*} & \cD(X_{/U_{n+1}},\Lambda_{n+1})
}
\end{align*}
for every morphism $\alpha\colon [m]\to[n]$ of $\del_+$, where $\alpha'\colon[m+1]\to[n+1]$ is the induced morphism. This is a special case of Lemma \ref{b2le:topoi_adjointability}.

Now the general case follows from Lemma \ref{b3le:limit_restriction} below and the fact that $u^*_{d^0_0}$ is exact.
\end{proof}

\begin{lem}\label{b3le:limit_restriction}
Let $p\colon K^\triangleleft\to\Cat$ be a limit diagram. Suppose that for each vertex $k$ of $K^\triangleleft$, we are given a strictly full subcategory $\cD_k\subseteq\cC_k=p(k)$ such that
\begin{enumerate}
  \item For every morphism $f\colon k\to k'$, the induced functor $p(f)$ sends $\cD_k$ to $\cD_{k'}$.

  \item An object $c$ of $\cC_\infty$ belongs to $\cD_\infty$ if and only if for every vertex $k$ of $K$, $p(f_k)(c)$ belongs to $\cD_k$, where $\infty$ denotes the cone point of $K^\triangleleft$, $f_k\colon \infty\to k$ is the unique edge.
\end{enumerate}
Then the induced diagram $q\colon K^\triangleleft\to\Cat$ sending $k$ to $\cD_k$ is also a limit diagram.
\end{lem}

\begin{proof}
Let $\tilde{p}\colon X\to (K^{op})^\triangleright$ be a Cartesian fibration classified by $p$ \cite{HTT}*{Definition 3.3.2.2}. Let $Y\subseteq X$ be the simplicial subset spanned by vertices in each fiber $X_k$ that are in the essential image of $\cD_k$ for all vertices $k$ of $K^\triangleleft$. The map $\tilde{q}=\tilde{p}\res Y\colon Y\to(K^{op})^\triangleright$ has the property that if $f\colon x\to y$ is $\tilde{p}$-Cartesian and $y$ belongs to $Y$, then $x$ also belongs to $Y$ by assumption (1), and $f$ is $\tilde{q}$-Cartesian by the dual version of \cite{HTT}*{Proposition 2.4.1.8}. It follows that $\tilde{q}$ is a Cartesian fibration, which is in fact classified by $q$. By assumption (2) and \cite{HTT}*{Corollary 3.3.3.2}, $q$ is a limit diagram.
\end{proof}

\subsection{Enhanced operations for quasi-compact and separated schemes}
\label{b3ss:enhanced_operation}

\begin{notation}
For a property (P) in the category $\Ring$, we say that a ringed diagram $(\Gamma,\Lambda)$ (Definition \ref{b2de:ringed_diagram}) has the property (P) if for every object $\xi$ of $\Xi$, the ring $\Lambda(\xi)$ has the property (P). We denote by $\Rind_\tor$ the full subcategory of $\Rind$ consisting of torsion ringed diagrams.
\end{notation}

Let $\Schqcs\subseteq\Sch$ be the full subcategory spanned by (small) coproducts of quasi-compact and separated schemes. For each object $X$ of
$\Sch$ (resp.\ $\Schqcs$), we denote by $\Et(X)\subseteq \Sch_{/X}$ (resp.\ $\Et^{\r{qc.sep}}(X)\subseteq\Schqcs_{/X}$) the full subcategory spanned by the \'etale morphisms, which is naturally a site.  We denote by $X_{\et}$ (resp.\ $X_{\r{qc.sep}.\et}$) the associated topos, namely the category of sheaves on $\Et(X)$ (resp.\ $\Et^{\r{qc.sep}}(X)$). In \cite{SGA4}*{Expos\'{e} vii, {\Sec}1.2}, $\Et(X)$ is called the \'{e}tale site of $X$ and $X_{\et}$ is called the \'etale topos of $X$. The inclusion $\Et^{\r{qc.sep}}(X)\subseteq \Et(X)$ induces an equivalence of topoi $X_{\et}\to X_{\r{qc.sep}.\et}$. In this chapter, we will not distinguish between $X_{\et}$ and $X_{\r{qc.sep}.\et}$.

\begin{definition}\label{b3de:statically}
In what follows, we will often deal with $\infty$-categories of the form
\[
(\cC^{op}\times\cD^{op})^{\amalg,op}\coloneqq((\cC^{op}\times\cD^{op})^\amalg)^{op}
\]
where $\cC$ is an $\infty$-category and $\cD$ is a subcategory of $\rN(\Rind)$. Suppose that $\cE$ is a subset of edges of $\cC$ that contains every isomorphism.

We say that an edge $f\colon (\{(X'_i,Y'_i)\}_{1\leq i\leq m})\to(\{(X_i,Y_i)\}_{1\leq i\leq m})$ of $(\cC^{op}\times\cD^{op})^{\amalg,op}$ \emph{statically belongs to $\cE$} if $f^{op}$ is static (Definition \ref{b1de:static}) and the
corresponding edge $X'_i\to X_i$ (resp.\ $Y'_i\to Y_i$) of $\cC$ (resp.\ $\cD$) belongs to $\cE$ (resp.\ is an isomorphism). By abuse of notation, we will denote again by $\cE$ the subset of edges of $(\cC^{op}\times\cD^{op})^{\amalg,op}$ that statically belong to $\cE$. Moreover, if sometimes $\cE$ is defined by a property $P$, then edges that statically belong to $\cE$ are said to \emph{statically have the property $P$}. We also denote by ``$\all$'' the set of all edges of $(\cC^{op}\times\cD^{op})^{\amalg,op}$.
\end{definition}

For $\cC=\rN(\Schqcs)$, we denote by
\begin{itemize}
  \item $F$ the set of morphisms of $\cC$ locally of finite type;

  \item $P\subseteq F$ the subset consisting of proper morphisms;

  \item $I\subseteq F$ the subset consisting of local isomorphisms.
\end{itemize}

\begin{lem}\label{b3le:categorical_equivalence}
Let $\cD$ be a subcategory of $\rN(\Rind)$. The natural map
\[
\delta_{3,\{3\}}^*((\rN(\Schqcs)^{op}\times\cD^{op})^{\amalg,op})^\cart_{P,I,\all}\to
\delta_{2,\{2\}}^*((\rN(\Schqcs)^{op}\times\cD^{op})^{\amalg,op})^\cart_{F,\all}
\]
is a categorical equivalence.
\end{lem}

\begin{proof}
The proof is similar to Corollary \ref{0co:scheme2}. Let $F_{\r{ft}}\subseteq F$ be the set consisting of morphisms of finite type, and put $I_{\r{ft}}\coloneqq I\cap F_{\r{ft}}$. Consider the following commutative diagram
\[
\resizebox{\hsize}{!}{
\xymatrix{
\delta_{4,\{4\}}^*((\rN(\Schqcs)^{op}\times\cD^{op})^{\amalg,op})^\cart_{P,I_{\r{ft}},I,\all}\ar[d]\ar[r]
& \delta_{3,\{3\}}^*((\rN(\Schqcs)^{op}\times\cD^{op})^{\amalg,op})^\cart_{F_{\r{ft}},I,\all}\ar[d]\\
\delta_{3,\{3\}}^*((\rN(\Schqcs)^{op}\times\cD^{op})^{\amalg,op})^\cart_{P,I,\all}\ar[r]
&
\delta_{2,\{2\}}^*((\rN(\Schqcs)^{op}\times\cD^{op})^{\amalg,op})^\cart_{F,\all}.
}
}
\]
To show that the lower horizontal map is a categorical equivalence, it suffices to show that the other three maps are categorical equivalences.

In Theorem \ref{0th:main}, we set $k=4$, $\cC=(\rN(\Schqcs)^{op}\times\cD^{op})^{\amalg,op}$, $\cE_0=F_{\r{ft}}$, $\cE_1=P$, $\cE_2=I_{\r{ft}}$, $\cE_3=I$, and $\cE_4=\all$. Note that we have a canonical isomorphism
\[
(\rN(\Schqcs)^{op}\times\cD^{op})^\amalg\simeq(\rN(\Schqcs)^{op})^\amalg\times_{\rN(\Fin)}(\cD^{op})^\amalg.
\]
By Nagata compactification theorem \cite{Conrad}*{Theorem 4.1}, condition (2) of Theorem \ref{0th:main} is satisfied. The other conditions are also satisfied by Lemma \ref{b1le:static_colimit}. It follows that the map in the upper horizontal arrow is a categorical equivalence. Similarly, using Theorem \ref{0th:main}, one proves that the vertical arrows are also categorical equivalences.
\end{proof}

\begin{remark}
The same proof shows that the lemma also holds with $\Schqcs$ replaced by the category of disjoint unions of quasi-compact \emph{quasi}-separated schemes and $F$ replaced by the set of \emph{separated} morphisms locally of finite type.
\end{remark}

Our goal is to construct a map \eqref{b3eq:operation} which encodes $f^*$, $f_!$ and the monoidal structure given by tensor product.

We start by encoding $f^*$ and the monoidal structure. Composing the nerve of the pseudofunctor $\Schqcs\to \PTopos$ carrying $X$ to $X_{\et}$ with $\EO{}{\PTopos}{\rI}{}$ \eqref{b2eq:topos_operation}, we obtain a functor
\begin{align}\label{b3eq:premonoidal}
\EO{}{\Schqcs}{\rI}{}\colon(\rN(\Schqcs)^{op}\times\rN(\Rind)^{op})^\amalg\to\Cat
\end{align}
that is a lax Cartesian structure, which induces a functor (Notation 
\ref{b2no:monoidal}) 
\begin{align}\label{b3eq:monoidal}
\EO{}{\Schqcs}\otimes{}\coloneqq(\EO{}{\Schqcs}{\rI}{})^\otimes\colon\rN(\Schqcs)^{op}\times\rN(\Rind)^{op}\to\PSLM
\end{align}
by Lemma \ref{b2le:closed_cartesian}.

To encode $f_!$, we resort to the technique of taking partial adjoints. Consider the composite map
\begin{multline}\label{b3eq:monoidal1}
\delta^*_{3,\{1,2,3\}}((\rN(\Schqcs)^{op}\times\rN(\Rind)^{op})^{\amalg,op})^\cart_{P,I,\all}\\
\to(\rN(\Schqcs)^{op}\times\rN(\Rind)^{op})^\amalg\xrightarrow{\EO{}{\Schqcs}{\rI}{}\eqref{b3eq:premonoidal}}\Cat.
\end{multline}

First, we apply the dual version of Proposition \ref{b1pr:partial_adjoint} to
\eqref{b3eq:monoidal1} for direction 1 to construct the partial right adjoint
\begin{align}\label{b3eq:monoidal2}
\delta^*_{3,\{2,3\}}((\rN(\Schqcs)^{op}\times\rN(\Rind_\tor)^{op})^{\amalg,op})^\cart_{P,I,\all}\to\Cat.
\end{align}
The adjointability condition for direction $(1,2)$ is a special case of that
for direction $(1,3)$. We check the latter as follows.

\begin{lem}
Let $\alpha\colon\langle m\rangle\to\langle n\rangle$ be a morphism of $\Fin$. Let $f_i \colon X_i'\to X_i$ be proper morphisms of schemes in $\Schqcs$ and take $\lambda_i\in\Rind_\tor$ for $1\leq i\leq m$. For pullback squares
\[
\xymatrix{
Y_j'\ar[r]\ar[d] & Y_j\ar[d]\\
\prod_{\alpha(i)=j} X'_i\ar[r]^-{\prod f_i} & \prod_{\alpha(i)=j}X_i}
\]
of schemes in $\Schqcs$ and morphisms $\mu_j\to \prod_{\alpha(i)=j}\lambda_i$ in $\Rind_\tor$ for $1\leq j\leq n$, the square
\[
\xymatrix{
\prod_{j\in T}\cD(Y'_j,\mu_j) & \prod_{j\in T}\cD(Y_j,\mu_j)\ar[l]\\
\prod_{i\in S}\cD(X'_i,\lambda_i)\ar[u] & \prod_{i\in S}\cD(X_i,\lambda_i)\ar[u]\ar[l]}
\]
given by pullback and tensor product is right adjointable.
\end{lem}

Note that the right adjoints of the horizontal arrows admit right adjoints.
Indeed, for the lower arrow we may assume $X_i$ quasi-compact and apply
Lemma \ref{b1le:adjoint_R}.

\begin{proof}
Decomposing the product categories with respect to $\langle n\rangle$, we are reduced to two cases: (a) $n=0$; (b) $n=1$ and $\alpha(\langle m\rangle^\circ)\subseteq\{1\}$. Case (a) is trivial. For case (b), writing $(f_i)_{1\leq i\leq m}$ as a composition, we may further assume that at most one $f_i$ is not the identity. Changing notation, we are reduced to showing that for every pullback square
\[
\xymatrix{
Y'\ar[r]^-{f'}\ar[d]_{g'} & Y\ar[d]^{g}\\
X'\ar[r]^-{f} & X}
\]
of schemes in $\Schqcs$ with $f$ proper and every morphism $\pi\colon \mu\to\lambda$ in $\Rind_\tor$, the diagram
\[
\xymatrix{
\cD(Y',\mu) & \cD(Y,\mu)\ar[l]_-{f'^*}\\
\cD(X',\lambda)\ar[u]^{(g',\pi)^*-\otimes f'^*\sfK} & \cD(X,\lambda)\ar[l]_-{f^*}\ar[u]_{(g,\pi)^*-\otimes \sfK}}
\]
is right adjointable for every $\sfK\in\cD(Y,\mu)$. As in the proof of Lemma \ref{b2le:upperstar_pi}, we easily reduce to the case with $\lambda=(\{\ast\},\Lambda)$ and $\mu=(\{\ast\},M)$. This case is the combination of proper base change and projection formula. See \cite{SGA4}*{Expos\'{e} xvii, Th\'{e}or\`{e}me 4.3.1} for a proof in $\rD^-$. Finally, the right completeness of unbounded derived categories \cite{HA}*{Proposition 1.3.5.21} implies that every object $\sfL$ of $\cD(X,\lambda)$ is the sequential colimit of $\tau^{\le n}\sfL$. The unbounded case follow since the vertical arrows and the right adjoints of the horizontal arrows preserve sequential colimits.
\end{proof}

Second, we apply Proposition \ref{b1pr:partial_adjoint} to \eqref{b3eq:monoidal2} for direction 2 to construct a map
\begin{align}\label{b3eq:monoidal3}
\delta^*_{3,\{3\}}((\rN(\Schqcs)^{op}\times\rN(\Rind_\tor)^{op})^{\amalg,op})^\cart_{P,I,\all}\to\Cat.
\end{align}
The adjointability condition for direction (2,1) follows from the fact that, for every separated \'etale morphism $f$ of finite type between
quasi-separated and quasi-compact schemes, the functor $f_!$ constructed in \cite{SGA4}*{Expos\'{e} xvii, Th\'{e}or\`{e}me 5.1.8} is a left adjoint of $f^*$ \cite{SGA4}*{Expos\'{e} xvii, Proposition 6.2.11}. The adjointability condition for direction (2,3) follows from \'etale base change and a trivial projection formula \cite{KS}*{Proposition 18.2.5}.

Third, we compose \eqref{b3eq:monoidal3} with (a quasi-inverse) of the categorical equivalence in Lemma \ref{b3le:categorical_equivalence} to
construct a map
\begin{align}\label{b3eq:operation}
\EO{}{\Schqcs}{\r{II}}{}\colon \delta^*_{2,\{2\}}((\rN(\Schqcs)^{op}\times\rN(\Rind_\tor)^{op})^{\amalg,op})^\cart_{F,\all}\to\Cat.
\end{align}

Now we explain how to encode $f_*$ and $f^!$ via adjunction. Note that we have a natural map from $\delta^*_{2,\{2\}}((\rN(\Schqcs)^{op}\times\rN(\Rind_\tor)^{op})^{\amalg,op})^\cart_{F,\all}$ to $\rN(\Fin)$, whose fiber over $\langle 1\rangle$ is isomorphic to $\delta^*_{2,\{2\}}\rN(\Schqcs)^\cart_{F,\all}\times\rN(\Rind_\tor)^{op}$. Denote by $\EO{}{\Schqcs}{*}{!}$ the restriction of $\EO{}{\Schqcs}{\r{II}}{}$ to the above fiber. By construction, we see that the image of $\EO{}{\Schqcs}{*}{!}$ actually factorizes through the subcategory $\PSL\subseteq\Cat$. In other words, \eqref{b3eq:operation} induces a map
\begin{align}\label{b3eq:operation1}
\EO{}{\Schqcs}{*}{!}\colon \delta^*_{2,\{2\}}\rN(\Schqcs)^\cart_{F,\all}\times\rN(\Rind_\tor)^{op}\to\PSL.
\end{align}
Evaluating \eqref{b3eq:monoidal} at the object $\langle 1\rangle\in\Fin$, we obtain the map
\begin{align}\label{b3eq:upperstar}
\EO{}{\Schqcs}{*}{}\colon \rN(\Schqcs)^{op}\times\rN(\Rind)^{op}\to\PSL.
\end{align}
Note that this is equivalent to the map obtained by restricting 
\eqref{b3eq:operation1} to the second direction, on 
$\rN(\Schqcs)^{op}\times\rN(\Rind_\tor)^{op}$. Composing the equivalence 
$\phi_{\PS}$ in Remark \ref{b1re:partial_adjoint} with $\EO{}{\Schqcs}{*}{}$, 
we obtain the map 
\[
\EO{}{\Schqcs}{}{*}\colon \rN(\Schqcs)\times\rN(\Rind)\to\PSR.
\]
Restricting \eqref{b3eq:operation1} to the first direction, we obtain the map
\begin{align}\label{b3eq:lowersh}
\EO{}{\Schqcs}{}{!}\colon \rN(\Schqcs)_F\times\rN(\Rind_\tor)^{op}\to\PSL.
\end{align}
Composing the equivalence $\phi_{\PS}$ in Remark \ref{b1re:partial_adjoint} with $\EO{}{\Schqcs}{}{!}$, we obtain the map
\begin{align}\label{b3eq:uppersh}
\EO{}{\Schqcs}{!}{}\colon \rN(\Schqcs)^{op}_F\times\rN(\Rind_\tor)\to\PSR.
\end{align}

\begin{variant}\label{b3va:quasifinite}
Let $Q(\subseteq F)\subseteq\Ar(\Schqcs)$ be the set of locally quasi-finite morphisms \cite{SP}*{01TD}. Recall that base change for an integral morphism \cite{SGA4}*{Expos\'{e} viii, Corollaire 5.6} holds for all Abelian sheaves. Replacing proper base change by finite base change in the construction of \eqref{b3eq:operation}, we obtain
\[
\EO{\r{lqf}}{\Schqcs}{\r{II}}{}\colon
\delta^*_{2,\{2\}}((\rN(\Schqcs)^{op}\times\rN(\Rind)^{op})^{\amalg,op})^\cart_{Q,\all}\to\Cat.
\]
When restricted to their common domain of definition, this map is equivalent to $\EO{}{\Schqcs}{\r{II}}{}$ \eqref{b3eq:operation}.
\end{variant}

\begin{notation}\label{b3no:image}
We introduce the following notation.
\begin{enumerate}
  \item For an object $(X,\lambda)$ of $\Schqcs\times\Rind$, we denote its image under $\EO{}{\Schqcs}\otimes{}$ by $\cD(X,\lambda)^\otimes$, with the underlying $\infty$-category $\cD(X,\lambda)$. In other words, we have $\cD(X,\lambda)^\otimes=\cD(X_{\et},\lambda)^\otimes$ and $\cD(X,\lambda)=\cD(X_{\et},\lambda)$. By construction and Remark \ref{b2re:topos}(2), $\cD(X,\lambda)$ is equivalent to the derived $\infty$-category of $\Mod(X_{\et}^\Xi,\Lambda)$ if $\lambda=(\Xi,\Lambda)$, and the monoidal structure on $\cD(X,\lambda)^\otimes$ is an $\infty$-categorical enhancement of the usual (derived) tensor product in the classical derived category.

  \item For a morphism $f\colon(X',\lambda')\to(X,\lambda)$ of $\Schqcs\times\Rind$, we denote its image under $\EO{}{\Schqcs}\otimes{}$ by
      \[
      f^{*\otimes}\colon\cD(X,\lambda)^\otimes\to\cD(X',\lambda')^\otimes,
      \]
      with the underlying functor $f^*\colon\cD(X,\lambda)\to\cD(X',\lambda')$. Note that $f^*$ is an $\infty$-categorical enhancement of the usual (derived) pullback functor in the classical derived category, which is monoidal. If $\lambda'\to\lambda$ is the identity, we denote the image of $f$ under $\EO{}{\Schqcs}{}{*}$ by
      \[
      f_*\colon\cD(Y,\lambda)\to\cD(X,\lambda),
      \]
      which is an $\infty$-categorical enhancement of the usual (derived) pushforward functor.

  \item For a morphism $f\colon Y\to X$ locally of finite type of $\Schqcs$ and an object $\lambda$ of $\Rind_\tor$, we denote its image under $\EO{}{\Schqcs}{}{!}$ and $\EO{}{\Schqcs}{!}{}$ by
      \begin{align*}
      f_!\colon\cD(Y,\lambda)\to\cD(X,\lambda),\quad f^!\colon\cD(X,\lambda)\to\cD(Y,\lambda)
      \end{align*}
      which are $\infty$-categorical enhancement of the usual $f_!$ and $f^!$ in the classical derived category, respectively.
\end{enumerate}
\end{notation}

\begin{remark}\label{b3re:kunneth}
In the previous discussion, we have constructed two maps
\[
\EO{}{\Schqcs}{\rI}{},\quad\EO{}{\Schqcs}{\r{II}}{}
\]
from which we deduce the other six maps
\[
\EO{}{\Schqcs}\otimes{},\;\EO{}{\Schqcs}{*}{!},\;\EO{}{\Schqcs}{*}{},\;
\EO{}{\Schqcs}{}{*},\;\EO{}{\Schqcs}{}{!},\;\EO{}{\Schqcs}{!}{}.
\]
Moreover, maps $\EO{}{\Schqcs}{\rI}{}$ and $\EO{}{\Schqcs}{\r{II}}{}$ are equivalent on their common part of domain, which is
$(\rN(\Schqcs)^{op}\times\rN(\Rind_\tor)^{op})^\amalg$.

Now we explain how K\"{u}nneth Formula is encoded in the map $\EO{}{\Schqcs}{\r{II}}{}$. In particular, as special cases, Base Change and Projection Formula are also encoded. Suppose that we have a diagram
\[
\xymatrix{Y_1\ar[d]_-{f_1} & Y\ar[l]_-{q_1}\ar[d]^-{f}\ar[r]^-{q_2} & Y_2\ar[d]^{f_2}\\ X_1 & X\ar[l]_{p_1}\ar[r]^-{p_2} & X_2,}
\]
which exhibits $Y$ as the limit $Y_1\times_{X_1}\times X\times_{X_2}Y_2$ and such that $f_1$ and $f_2$ (hence $f$) are locally of finite type. Fix an object $\lambda$ of $\Rind_\tor$. They together induce an edge
\[
\xymatrix{
((Y_1,\lambda),(Y_2,\lambda)) \ar[r]\ar[d] & (Y,\lambda)\ar[d] \\
((X_1,\lambda),(X_2,\lambda)) \ar[r] & (X,\lambda)
}
\]
of $\delta^*_{2,\{2\}}((\rN(\Schqcs)^{op}\times\rN(\Rind_\tor)^{op})^{\amalg,op})^\cart_{F,\all}$ above the unique active map $\langle 2\rangle\to\langle1\rangle$ of $\Fin$. Applying the map $\EO{}{\Schqcs}{\r{II}}{}$ and by adjunction, we obtain the following square
\[
\xymatrix{
\cD(Y_1,\lambda)\times\cD(Y_2,\lambda)
\ar[rr]^-{q_1^*-\otimes_Y q_2^*-}\ar[d]_-{f_{1!}\times f_{2!}} && \cD(Y,\lambda)\ar[d]^-{f_!} \\
\cD(X_1,\lambda)\times\cD(X_2,\lambda) \ar[rr]^-{p_1^*-\otimes_X
p_2^*-} && \cD(X,\lambda) }
\]
in $\Cat$. At the level of homotopy categories, this recovers the classical K\"{u}nneth Formula.
\end{remark}

We end this section by the following adjointability result.

\begin{lem}\label{b3le:lowersh_pi}
Let $f\colon Y\to X$ be a morphism locally of finite type of $\Schqcs$, and $\pi\colon \lambda'\to\lambda$ a perfect morphism of $\Rind_\tor$ (Definition \ref{b2de:perfect}). Then the square
\[
\xymatrix{\cD(Y,\lambda')\ar[r]^{f_!} & \cD(X,\lambda')\\
\cD(Y,\lambda)\ar[u]^{\pi^*}\ar[r]^{f_!} & \cD(X,\lambda)\ar[u]_{\pi^*},}
\]
is right adjointable and its transpose is left adjointable.
\end{lem}

\begin{proof}
The assertion being trivial for $f$ in $I$, we may assume $f$ in $P$. As in the proof of Lemma \ref{b2le:upperstar_pi}, we are reduced to the case where $\pi^*$ is replaced $e_\zeta^*$ and $t_*\circ t^*$, respectively. Here, we have maps $(\{\ast\},\Lambda')\xrightarrow{t}(\{\zeta\},\Lambda(\zeta))\xrightarrow{e_\zeta}(\Xi,\Lambda)$.

The assertion for $t_*\circ t^*$ is trivial, since a left adjoint of $t_*\circ t^*$ is $-\otimes_{\Lambda(\zeta)}{\Lambda'}\spcheck\simeq
\HOM_{\Lambda(\zeta)}(\Lambda',-)$, where ${\Lambda'}\spcheck=\HOM_{\Lambda(\zeta)}(\Lambda',{\Lambda(\zeta)})$. We denote by $e_{\zeta!}$ a left adjoint of $e_\zeta^*$. For $\xi\in \Xi$, since $e_\xi^*$ commutes with $f_*$ by Lemma \ref{b2le:upperstar_pi}, it suffices to check that
$e_\xi^*\circ e_{\zeta!}$ commutes with $f_*$. Here $e_\xi\colon(\{\xi\},\Lambda(\xi))\to (\Xi,\Lambda)$ is the obvious morphism. For $\xi\le
\zeta$, we have $e_\xi^*\circ e_{\zeta!}\simeq-\otimes_{\Lambda(\zeta)}\Lambda(\xi)$ and the assertion follows from projection formula. For other $\xi\in \Xi$, the map $e_\xi^* \circ e_{\zeta!}$ is zero.
\end{proof}

\subsection{Poincar\'{e} duality and (co)homological descent}
\label{b3ss:poincare}

For an object $X$ of $\Schqcs$ and an object $\lambda=(\Xi,\Lambda)$ of $\Rind$, we have a t-structure $(\cD^{\le 0}(X,\lambda),\cD^{\ge
0}(X,\lambda))$ on $\cD(X,\lambda)$,\footnote{We use a \emph{cohomological} indexing convention, which is different from \cite{HA}*{Definition 1.2.1.4}.} which induces the usual t-structure on its homotopy category $\rD(X^\Xi_{\et},\Lambda)$. We denote by $\tau^{\le 0}$ and $\tau^{\ge 0}$
the corresponding truncation functors. The heart
\[
\cD^\heartsuit(X,\lambda)\coloneqq\cD^{\le 0}(X,\lambda)\cap\cD^{\ge 0}(X,\lambda)\subseteq\cD(X,\lambda)
\]
is canonically equivalent to (the nerve of) the Abelian category
\[
\Mod(X,\lambda)\coloneqq\Mod(X^\Xi_{\et},\Lambda).
\]
The constant sheaf $\lambda_X$ on $X^\Xi$ of value $\Lambda$ is an object of $\cD^\heartsuit(X,\lambda)$.

We fix a nonempty set $\Box$ of rational primes. Recall that a ring $R$ is a \emph{$\Box$-torsion} ring if each element is killed by an integer that is a product of primes in $\Box$. In particular, a $\Box$-torsion ring is a torsion ring. We denote by $\Rind_{\ltor}\subseteq\Rind_\tor$ the full subcategory spanned by $\Box$-torsion ringed diagrams. Recall that a scheme $X$ is \emph{$\Box$-coprime} if $\Box$ does not contain any residue characteristic of $X$. Let $\Schqcs_\Box$ be the full subcategory of $\Schqcs$ spanned by $\Box$-coprime schemes. In particular,
$\Spec\dZ[\Box^{-1}]$ is a final object of $\Schqcs_\Box$. By abuse of notation, we still use $A$ and $F$ to denote $A\cap\Ar(\Schqcs_\Box)$ and $F\cap\Ar(\Schqcs_\Box)$, respectively. Moreover, let $L\subseteq F$ be the set of smooth morphisms.

\begin{definition}[Tate twist]\label{b3de:twist}
We define a functor
\[
\r{tw}\colon(\rN(\Rind_{\ltor})^{op})^\triangleleft\to\Cat
\]
such that
\begin{enumerate}
  \item the restriction of $\r{tw}$ to $\rN(\Rind_{\ltor})^{op}$ coincides with the restriction of the functor $\EO{}{\Schqcs}{*}{}$
      \eqref{b3eq:upperstar} to $\{\Spec\dZ[\Box^{-1}]\}\times\rN(\Rind_{\ltor})^{op}$;

  \item $\r{tw}(-\infty)$ equals $\Delta^0$;

  \item for every object $\lambda$ of $\Rind_{\ltor}$, the image of $0$ under the functor $\r{tw}(-\infty\to\lambda)$ is the Tate twisted sheaf, denoted by $\lambda_\Box(1)$, is dualizable in the symmetric monoidal $\infty$-category $\cD(\Spec\dZ[\Box^{-1}],\lambda)^\otimes$.
\end{enumerate}
Let $(X,\lambda)$ be an object of $\Schqcs_\Box\times\Rind_{\ltor}$. We define the following functor
\[
-\langle 1\rangle\coloneqq(-\otimes s_X^*\lambda_\Box(1))[2]\colon\cD(X,\lambda)\to\cD(X,\lambda),
\]
where $s_X\colon X\to\Spec\dZ[\Box^{-1}]$ is the structure morphism. We know that $-\langle 1\rangle$ is an auto-equivalence since $\lambda_\Box(1)$ is dualizable and $s_X^*$ is monoidal. In general, for $d\in\dZ$, we define $-\langle d\rangle$ to be the (inverse of the, if $d<0$) $|d|$-th iteration of $-\langle 1\rangle$.
\end{definition}

We adapt the classical theory of trace maps and the Poincar\'{e} duality to the $\infty$-categorical setting, as follows. Let $f\colon Y\to X$ be a flat morphism in $\Schqcs_\Box$, locally of finite presentation, and such that every geometric fiber has dimension $\leq d$. Let $\lambda$ be an object of $\Rind_{\ltor}$. In \cite{SGA4}*{Expos\'{e} xviii, Th\'{e}or\`{e}me 2.9}, Deligne constructed the trace map
\begin{align}\label{b3eq:trace}
\Tr_f=\Tr_{f,\lambda}\colon \tau^{\ge 0}f_!\lambda_Y\langle d\rangle\to\lambda_X,
\end{align}
which turns out to be a morphism of $\cD^\heartsuit(X,\lambda)$. The construction satisfies the following functorial properties.

\begin{lem}[Functoriality of trace maps]\label{b3le:trace}
The trace maps $\Tr_f$ for all such $f$ and $\lambda$ are functorial in the following sense:
\begin{enumerate}
  \item For every morphism $\lambda\to\lambda'$ of $\Rind_{\ltor}$, the diagram
      \[
      \xymatrix{&\tau^{\ge 0}f_!\lambda_Y\langle d\rangle\ar[rd]^-{\Tr_{f,\lambda}} \\
      \tau^{\ge 0}((\tau^{\ge 0}f_!\lambda'_Y\langle d\rangle)\otimes_{\lambda'_X} \lambda_X)
      \ar[ur]^-\sim\ar[rr]^-{\tau^{\ge 0}(\Tr_{f,\lambda'}\otimes_{\lambda'_X} \lambda_X)}
      && \lambda_X}
      \]
      commutes.

  \item For every Cartesian diagram
        \[
        \xymatrix{Y'\ar[r]^-{f'}\ar[d]_-v & X'\ar[d]^-u\\
        Y \ar[r]^f & X}
        \]
        of $\Schqcs_\Box$, the diagram
        \[
        \xymatrix{u^*\tau^{\ge 0}f_!\lambda_Y\langle d\rangle\ar[rr]^-{u^*\Tr_f}\ar[d]_-{\simeq} && u^*\lambda_X\ar[d]^-{\simeq}\\
        \tau^{\ge 0}f'_!\lambda_{Y'}\langle d\rangle\ar[rr]^-{\Tr_{f'}}
        && \lambda_{X'}}
        \]
        commutes.

  \item Consider a $2$-cell
      \[
      \xymatrix{
      Z \ar[rd]_-{g} \ar[rr]^-{h}  && X\\
      & Y \ar[ru]_-{f} }
      \]
      of $\rN(\Schqcs_\Box)$ with $f$ (resp.\ $g$) flat, locally of finite presentation, and such that every geometric fiber has dimension $\leq d$ (resp.\ $\leq e$). Then $h$ is flat, locally of finite presentation, and such that every geometric fiber has dimension $\leq d+e$, and the diagram
      \[
      \xymatrix{\tau^{\ge 0}f_!(\tau^{\ge 0}g_!\lambda_Z\langle e\rangle)\langle d\rangle
      \ar[rr]^-{\tau^{\ge 0}f_!\Tr_g\langle d\rangle}\ar[d]_-{\simeq}
      && \tau^{\ge 0}f_!\lambda_Y\langle d\rangle\ar[d]^-{\Tr_f}\\
      \tau^{\ge 0}h_!\lambda_Z\langle d+e\rangle\ar[rr]^-{\Tr_h} && \lambda_X}
      \]
      commutes.
\end{enumerate}
\end{lem}

\begin{proof}
This is \cite{SGA4}*{Expos\'{e} xviii, {\Sec}2}.
\end{proof}

Let $f\colon Y\to X$ be as above. We have the following $2$-cell
\[
\xymatrix@R=0.5cm{
&         \cD(Y,\lambda) \ar[dd]^-{f_!}     \\
\cD(X,\lambda) \ar[ur]^-{f^*} \ar[dr]_-{f_!\lambda_Y\otimes-} \\
&         \cD(X,\lambda)                 }
\]
of $\Cat$. If we abuse of notation by writing $f^*\langle d\rangle$ for $-\langle d\rangle \circ f^*$, then the composition
\begin{align}\label{b3eq:trace_1}
u_f\colon f_!\circ f^*\langle d\rangle\xrightarrow{\sim} f_!\lambda_Y\langle d\rangle\otimes-
\to\tau^{\ge 0}f_!\lambda_Y\langle d\rangle\otimes-\xrightarrow{\Tr_f\otimes-}\lambda_X\otimes- \xrightarrow{\sim}\id_X
\end{align}
is a natural transformation, where $\id_X$ is the identity functor of $\cD(X,\lambda)$.

\begin{lem}\label{b3le:poincare}
If $f\colon Y\to X$ is smooth and of pure relative dimension $d$, then $u_f$ is a counit transformation. In particular, the functors $f^*\langle d\rangle$ and $f^!$ are equivalent.
\end{lem}

\begin{proof}
This follows from \cite{SGA4}*{Expos\'{e} xviii, Th\'{e}or\`{e}me 3.2.5} and the fact that $f^!$ is right adjoint to $f_!$.
\end{proof}

\begin{remark}\label{b3re:trace_quasifinite}
Let $f\colon Y\to X$ be a morphism in $\Schqcs$ that is flat, locally quasi-finite, and locally of finite presentation. Let $\lambda$ be an object of $\Rind$ (see Variant \ref{b3va:quasifinite} for the definition of the enhanced operation map in this setting). In \cite{SGA4}*{Expos\'{e} xvii, Th\'{e}or\`{e}me 6.2.3}, Deligne constructed the trace map
\[
\Tr_f\colon \tau^{\geq0}f_!\lambda_Y\to\lambda_X,
\]
which is a morphism of $\cD^\heartsuit(X,\lambda)$. It coincides with the trace map \eqref{b3eq:trace} when both are defined, and satisfies similar functorial properties. Moreover, by \cite{SGA4}*{Expos\'{e} xvii, Proposition 6.2.11}, the map $u_f\colon f_!\circ f^*\to \id_X$ constructed similarly as \eqref{b3eq:trace_1} is a counit transform when $f$ is \'etale. Thus, the functors $f^!$ and $f^*$ are equivalent in this case.
\end{remark}

The following proposition will be used in the construction of the enhanced operation map for quasi-separated schemes.

\begin{proposition}[(Co)homological descent]\label{b3pr:descent}
Let $f\colon X^+_0\to X^+_{-1}$ be a smooth and surjective morphism of $\Schqcs$. Then
\begin{enumerate}
  \item $(f,\id_\lambda)$ is of universal $\EO{}{\Schqcs}\otimes{}$-descent \eqref{b3eq:monoidal}, where $\lambda$ is an arbitrary object of $\Rind$;

  \item $(f,\id_\lambda)$ is of universal $\EO{}{\Schqcs}{}{!}$-codescent \eqref{b3eq:lowersh}, where $\lambda$ is an arbitrary object of
      $\Rind_\tor$.
\end{enumerate}
See Definition \ref{b3de:descent} for the definition of universal (co)descent.
\end{proposition}

\begin{proof}
We first prove the case where $f$ is \'{e}tale. For (1), let $X^+_\bullet$ be 
a \v{C}ech nerve of $f$, and put 
$(\cD^{\otimes*})^\bullet_+\coloneqq\EO{}{\Schqcs}\otimes{}\circ((X^+_\bullet)^{op}\times 
\{\lambda\})$. By Remark \ref{b1re:monoidal}, we only need to check that 
$(\cD^*)^\bullet_+=\rG\circ(\cD^{\otimes*})^\bullet_+$ is a limit diagram, 
where $\rG$ is the functor \eqref{b1eq:G}. This is a special case of Lemma 
\ref{b3le:topos_descent} by letting $U_\bullet$ be the sheaf represented by 
$X^+_\bullet$, and $\cC_\bullet$ be the whole category. For (2), we only need 
to prove that 
$(\cD^!)^\bullet_+\coloneqq\phi\circ\EO{}{\Schqcs}{!}{}\circ((X^+_\bullet)^{op}\times 
\{\lambda\})$ is a limit diagram, where $\phi\colon\PSR\to\Cat$ is the 
natural inclusion, and the functor $\EO{}{\Schqcs}{!}{}$ is the one in 
\eqref{b3eq:uppersh}. By Poincar\'{e} duality for \'{e}tale morphisms 
recalled in Remark \ref{b3re:trace_quasifinite}, $(\cD^!)^\bullet_+$ is 
equivalent to $(\cD^*)^\bullet_+$, which is a limit diagram as we have 
already seen. 

The general case where $u$ is smooth follows from the above case by Lemma 
\ref{b3le:descent}(3) (and its dual version), and the fact that there exists 
an \'etale surjective morphism $g\colon Y\to X$ of $\Schqcs$ that factorizes 
through $f$ \cite{EGAIV}*{Corollaire 17.16.3(ii)}. 
\end{proof}

\section{The program DESCENT}
\label{b4ss}

From Remark \ref{b3re:kunneth}, we know that all useful information of six operations for $\Schqcs$ is encoded in the maps $\EO{}{\Schqcs}{\rI}{}$ \eqref{b3eq:premonoidal} and $\EO{}{\Schqcs}{\r{II}}{}$ \eqref{b3eq:operation} constructed in \Sec\ref{b3ss:enhanced_operation}. In this chapter, we develop a program called DESCENT, which is an abstract categorical procedure to
extend the above two maps to larger categories. The extended maps satisfy similar properties as the original ones. This program will be run in the next chapter to extend our theory successively to quasi-separated schemes, to algebraic spaces, to Artin stacks, and eventually to higher Deligne--Mumford and higher Artin stacks.

In \Sec\ref{b4ss:description}, we describe the program by formalizing the data for $\Schqcs$. In \Sec\ref{b4ss:construction}, we construct the extension of the maps. In \Sec\ref{b4ss:properties}, we prove the required properties of the extended maps.

\subsection{Description}
\label{b4ss:description}

In \Sec\ref{b3ss:enhanced_operation}, we constructed two maps $\EO{}{\Schqcs}{\rI}{}$ \eqref{b3eq:premonoidal} and $\EO{}{\Schqcs}{\r{II}}{}$
\eqref{b3eq:operation}. They satisfy certain properties such as descent for smooth morphisms (Proposition \ref{b3pr:descent}). We would like to extend these maps to maps defined on the $\infty$-category of higher Deligne--Mumford or higher Artin stacks, satisfying similar properties. We
will achieve this in many steps, by first extending the maps to quasi-separated schemes, and then to algebraic spaces, and then to Artin stacks, and so on. All the steps are similar to each other. The output of one step provides the input for the next step. We will think of this as
recursively running a program, which we name DESCENT. In this section, we axiomatize the input and output of this program in an abstract setting.

Let us start with a toy model.

\begin{proposition}\label{b4pr:toy}
Let $(\tilde\cC,\tilde\cE)$ be a marked $\infty$-category such that $\tilde\cC$ admits pullbacks and $\tilde\cE$ is stable under composition and pullback. Let $\cC\subseteq \tilde\cC$ be a full subcategory stable under pullback such that for every object $X$ of $\tilde\cC$, there exists a morphism $Y\to X$ in $\tilde\cE$ representable in $\cC$ with $Y$ in $\cC$. Let $\cD$ be an $\infty$-category such that $\cD^{op}$ admits geometric realizations. Let $\Fun^\cE(\cC^{op},\cD)\subseteq \Fun(\cC^{op},\cD)$ (resp.\ $\Fun^{\tilde\cE}(\tilde\cC^{op},\cD)\subseteq\Fun(\tilde\cC^{op},\cD)$) be the full subcategory spanned by functors $F$ such that every edge in $\cE=\tilde\cE\cap \cC_1$ (resp.\ in $\tilde\cE$) is of $F$-descent. Then the restriction map
\[
\Fun^{\tilde\cE}(\tilde\cC^{op},\cD)\to \Fun^\cE(\cC^{op},\cD)
\]
is a trivial fibration.
\end{proposition}

The proof will be given at the end of \Sec\ref{b4ss:construction}.

\begin{example}\label{b4ex:toy}
Let $\Schqs\subseteq\Sch$ be the full subcategory spanned by quasi-separated schemes. It contains $\Schqcs$ as a full subcategory. By Proposition \ref{b3pr:descent}(1), we may apply Proposition \ref{b4pr:toy} to
\begin{itemize}
  \item $\tilde\cC=(\rN(\Schqs)^{op}\times\rN(\Rind)^{op})^{\amalg,op}$,

  \item $\cC=(\rN(\Schqcs)^{op}\times\rN(\Rind)^{op})^{\amalg,op}$,

  \item $\cD=\Cat$,

  \item and the set $\tilde\cE$ consists of edges $f$ that are statically smooth surjective (Definition \ref{b3de:statically}).
\end{itemize}
Then we obtain an extension of the map $\EO{}{\Schqcs}{\rI}{}$ with larger source $(\rN(\Schqs)^{op}\times\rN(\Rind)^{op})^\amalg$.
\end{example}

Now we describe the program in full. We begin by summarizing the categorical properties we need on the geometric side into the following definition.

\begin{definition}\label{b4de:geometric}
An $\infty$-category $\cC$ is \emph{geometric} if it admits small coproducts and pullbacks such that
\begin{enumerate}
  \item \emph{Coproducts are disjoint:} every coCartesian diagram
      \[
      \xymatrix{
      \emptyset \ar[r] \ar[d] & X \ar[d]\\
      Y \ar[r] & X\coprod Y}
      \]
      is also Cartesian, where $\emptyset$ denotes an initial object of $\cC$.

  \item \emph{Coproducts are universal:} For a small collection of Cartesian diagrams
      \[
      \xymatrix{
      Y_i \ar[r] \ar[d] & Y \ar[d]\\
      X_i \ar[r] & X,}
      \]
      $i\in I$, the diagram
      \[
      \xymatrix{
      \coprod_{i\in I}Y_i \ar[r] \ar[d] & Y \ar[d]\\
      \coprod_{i\in I}X_i \ar[r] & X,}
      \]
      is also Cartesian.
\end{enumerate}
\end{definition}

\begin{remark}
We have the following remarks about geometric categories.
\begin{enumerate}
  \item Let $\cC$ be geometric. Then a small coproduct of Cartesian diagrams of $\cC$ is again Cartesian.

  \item The $\infty$-categories $\rN(\Schqcs)$, $\rN(\Schqs)$, $\rN(\Esp)$, $\rN(\Chp)$, $\Chpar{k}$ and $\Chpdm{k}$ ($k\geq0$) appearing in this article are all geometric.
\end{enumerate}
\end{remark}

We now describe the input and the output of the program. The input has three parts: 0, I, and II. The output has two parts: I and II. We refer the reader to Example \ref{b4ex:scheme} for a typical example.\\

\paragraph*{\textbf{Input 0}} We are given
\begin{itemize}
  \item A $5$-marked $\infty$-category $(\tilde\cC,\tcEs,\tilde\cE',\tilde\cE'',\tcEt,\tilde\cF)$, a full subcategory $\cC\subseteq\tilde\cC$, and a morphism $\sf{s}''\to\sf{s}'$ of $(-1)$-truncated objects of $\cC$ \cite{HTT}*{Definition 5.5.6.1}.

  \item For each $d\in\dZ\cup\{-\infty\}$, a subset $\tilde\cE''_d$ of $\tilde\cE''$.

  \item A sequence of inclusions of $\infty$-categories $\cL''\subseteq\cL'\subseteq\cL$.

  \item A function $\dim^+\colon \tilde\cF\to\dZ\cup\{-\infty,+\infty\}$.
\end{itemize}
Put $\cEs\coloneqq\tcEs\cap\cC_1$, $\cE'\coloneqq\tilde\cE'\cap\cC_1$, $\cE''\coloneqq\tilde\cE''\cap\cC_1$, $\cE''_d\coloneqq\tilde\cE''_d\cap\cC_1$ ($d\in\dZ\cup\{-\infty\}$), $\cEt\coloneqq\tcEt\cap\cC_1$, and $\cF\coloneqq\tilde\cF\cap\cC_1$. Let $\cC'$ (resp.\ $\tilde\cC'$, $\cC''$, and $\tilde\cC''$) be the full subcategory of $\cC$ (resp.\ $\tilde\cC$, $\cC$, and $\tilde\cC$) spanned by those objects that admit morphisms to $\sf{s}'$ (resp.\ $\sf{s}'$, $\sf{s}''$, and $\sf{s}''$). Put $\cF'\coloneqq\cF\cap\cC'_1$ and $\tilde\cF'\coloneqq\tilde\cF\cap\tilde\cC'_1$. They satisfy
\begin{enumerate}
  \item $\tilde\cC$ is geometric, and the inclusion $\cC\subseteq\tilde\cC$ is stable under finite limits. Moreover, for every small coproduct $X=\coprod_{i\in I}X_i$ in $\tilde\cC$, $X$ belongs to $\cC$ if and only if $X_i$ belongs to $\cC$ for all $i\in I$.

  \item $\cL''\subseteq\cL'$ and $\cL'\subseteq\cL$ are full subcategories.

  \item $\tcEs,\tilde\cE',\tilde\cE'',\tcEt,\tilde\cF$ are stable under composition, pullback and small coproducts; and
      $\tilde\cE'\subseteq\tilde\cE''\subseteq\tcEt\subseteq\tilde\cF$.

  \item For every object $X$ of $\tilde\cC$, there exists an edge $f\colon Y\to X$ in $\tcEs\cap\tilde\cE'$ representable in $\cC$ with $Y$ in $\cC$. Such an edge $f$ is called an \emph{atlas} for $X$.


  \item For every edge $f\colon Y\to X$ in $\tilde\cE''$, there exist $2$-simplices
      \begin{align}\label{b4eq:2celld}
      \xymatrix{&Y\ar[rd]^-{f}\\
      Y_d\ar[ru]^-{i_d}\ar[rr]^-{f_d} &&X}
      \end{align}
      of $\tilde\cC$ with $f_d$ in $\tilde\cE''_d$ for $d\in\dZ$, such that the edges $i_d$ exhibit $Y$ as the coproduct $\coprod_{d\in\dZ}Y_d$.

  \item For every $d\in\dZ\cup\{-\infty\}$, we have $\tilde\cE''_d\subseteq \tilde\cE''$, that $\tilde\cE''_d$ is stable under pullback and small coproducts, and that $\tilde\cE''_{-\infty}$ is the set of edges whose source is an initial object. For distinct integers $d$ and $e$, we have $\tilde\cE''_d\cap\tilde\cE''_e=\tilde\cE''_{-\infty}$.

  \item For every small set $I$ and every pair of objects $X$ and $Y$ of $\tilde\cC$, the morphisms $X\to X\coprod Y$ and $\coprod_I X \to
      X$ are in $\tilde\cE''_0$. For every $2$-cell
      \begin{align}\label{b4eq:2cell}
      \xymatrix{&Y\ar[rd]^-{f}\\
      Z\ar[rr]^-{h}\ar[ru]^-{g} && X}
      \end{align}
      of $\tilde\cC$ with $f$ in $\tilde\cE''_d$ and $g$ in $\tilde\cE''_e$, where $d$ and $e$ are integers, $h$ is in $\tilde\cE''_{d+e}$.

  \item The function $\dim^+$ satisfies the following conditions.
    \begin{enumerate}
      \item $\dim^+(f)=-\infty$ if and only if $f$ is in $\tilde\cE''_{-\infty}$.

      \item The restriction of $\dim^+$ to $\tilde\cE''_d-\tilde\cE''_{-\infty}$ is of constant value $d$.

      \item For every $2$-cell \eqref{b4eq:2cell} in $\tilde\cC$ with edges in $\tilde\cF$, we have $\dim^+(h)\leq\dim^+(f)+\dim^+(g)$, and that the equality holds when $g$ belongs to $\tcEs\cap \tilde\cE''$.

      \item For every Cartesian diagram
          \[
          \xymatrix{
          W \ar[r]^-{g} \ar[d]_-{q} & Z \ar[d]^-{p}\\
          Y \ar[r]^-{f} & X}
          \]
          in $\tilde\cC$ with $f$ (and hence $g$) in $\tilde\cF$, we have $\dim^+(g)\leq\dim^+(f)$, and equality holds when $p$ belongs to $\tcEs$.

      \item For every edge $f\colon Y\to X$ in $\tilde\cF$ and every small collection
          \[
          \xymatrix{&Y\ar[rd]^-{f}\\
          Z_i\ar[rr]^-{h_i}\ar[ru]^-{g_i} && X}
          \]
          of $2$-simplices with $g_i$ in $\tilde\cE''_{d_i}$ such that the morphism $\coprod_{i\in I} Z_i\to Y$ is in $\tcEs$, we have $\dim^+(f)=\sup_{i\in I}\{\dim^+(h_i)-d_i\}$.
    \end{enumerate}

  \item We have $\tilde\cE'=\tilde\cE''_0$.
\end{enumerate}

\begin{remark}
In Input 0, by (7) and (8c,d,e), for every small collection $\{Y_i\xrightarrow{f_i}X_i\}_{i\in I}$ of edges in $\tilde\cF$, we have
$\dim^+(\coprod_{i\in I}f_i)=\sup_{i\in I}\{\dim^+(f_i)\}$.
\end{remark}

\paragraph*{\textbf{Input I}} Input I consists of two maps as follows.
\begin{itemize}
  \item The \emph{first abstract operation map}:
      \[
      \EO{}{\cC}{\rI}{}\colon(\cC^{op}\times\cL^{op})^\amalg\to\Cat.
      \]

  \item The \emph{second abstract operation map}:
      \[
      \EO{}{\cC'}{\r{II}}{}\colon\delta^*_{2,\{2\}}((\cC'^{op}\times\cL'^{op})^{\amalg,op})^\cart_{\cF',\all}\to\Cat.
      \]
\end{itemize}

Input I is subject to the following properties:

\begin{description}
  \item[P0] \emph{Monoidal symmetry.} The functor $\EO{}{\cC}{\rI}{}$ is a 
      lax Cartesian structure, and the induced functor 
      $\EO{}{\cC}\otimes{}\coloneqq(\EO{}{\cC}{\rI}{})^\otimes$ factorizes 
      through $\PSLM$ (see Remark \ref{b1re:cartesian}). 

  \item[P1] \emph{Disjointness.} The map $\EO{}{\cC}\otimes{}$ sends small coproducts to products.

  \item[P2] \emph{Compatibility.} The restrictions of $\EO{}{\cC}{\rI}{}$ and $\EO{}{\cC'}{\r{II}}{}$ to $(\cC'^{op}\times\cL'^{op})^\amalg$ are equivalent functors.
\end{description}

Before stating the remaining properties, we have to fix some notation. Similar to the construction of \eqref{b3eq:operation1}, we obtain a map
\[
\EO{}{\cC'}{*}{!}\colon \delta^*_{2,\{2\}}\cC'^\cart_{\cF',\cC'_1}\times\cL'^{op}\to\PSL.
\]
from $\EO{}{\cC'}{\r{II}}{}$. Similar to the construction of \eqref{b3eq:upperstar} and \eqref{b3eq:lowersh}, we obtain maps
\[
\EO{}{\cC}{*}{}\colon\cC^{op}\times\cL^{op}\to\PSL,\quad \EO{}{\cC'}{}{!}\colon\cC'_{\cF'}\times\cL'^{op}\to\PSL.
\]
Moreover, we will use similar notation as in Notation \ref{b3no:image} for the image of 0 and 1-cells under above maps, after replacing $\Schqcs$ (resp.\ $\Rind$) by $\cC$ (resp.\ $\cL$). Now we are ready to state the remaining properties.

\begin{description}
  \item[P3] \emph{Conservativeness.} If $f\colon Y\to X$ belongs to $\cEs$, then $f^*\colon\cD(X,\lambda)\to\cD(Y,\lambda)$ is conservative for every object $\lambda$ of $\cL$.

  \item[P4] \emph{Descent.} Let $f$ be a morphism of $\cC$ (resp.\ $\cC'$) and $\lambda$ an object of $\cL$ (resp.\ $\cL'$). If $f$
      belongs to $\cEs\cap\cE''$ (resp.\ $\cEs\cap\cE''\cap\cC'_1$), then $(f,\id_\lambda)$ is of universal $\EO{}{\cC}\otimes{}$-descent
      (resp.\ $\EO{}{\cC'}{}{!}$-codescent).

  \item[P5] \emph{Adjointability for $\cE'$.} Let
      \[
      \xymatrix{
      W \ar[r]^-{g} \ar[d]_-{q} & Z \ar[d]^-{p} \\
      Y \ar[r]^-{f} & X }
      \]
      be a Cartesian diagram of $\cC'$ with $f$ in $\cE'$, and $\lambda$ an object of $\cL'$. Then
      \begin{enumerate}
        \item The square
            \[
            \xymatrix{
            \cD(Z,\lambda)  \ar[d]_-{g^*} & \cD(X,\lambda) \ar[l]_-{p^*} \ar[d]^-{f^*}\\
            \cD(W,\lambda) & \cD(Y,\lambda) \ar[l]_-{q^*}
            }\]
            has a right adjoint which is a square of $\PSR$.

        \item If $p$ is also in $\cE'$, then the square
            \[
            \xymatrix{
            \cD(X,\lambda)  \ar[d]_-{p^*} & \cD(Y,\lambda) \ar[l]_-{f_!} \ar[d]^-{q^*}\\
            \cD(Z,\lambda) & \cD(W,\lambda) \ar[l]_-{g_!} }
            \]
            is right adjointable.
      \end{enumerate}

  \item[P5\bis] \emph{Adjointability for $\cE''$.} We have the same statement as in (P5) after replacing $\cC'$ by $\cC''$, $\cE'$ by $\cE''$, and $\cL'$ by $\cL''$.
\end{description}

\paragraph*{\textbf{Input II}} Input II consists of the following data.
\begin{itemize}
  \item A functor $\r{tw}\colon(\cL''^{op})^\triangleleft\to\Cat$ satisfying that
      \begin{itemize}
        \item the restriction of $\r{tw}$ to $\cL''^{op}$ coincides with the restriction of $\EO{}{\cC}{*}{}$ to $\{\sf{s}''\}\times\cL''^{op}$;

        \item $\r{tw}(-\infty)$ equals $\Delta^0$;

        \item for every object $\lambda$ of $\Rind_{\ltor}$, if we denote the image of $0$ under the functor $\r{tw}(-\infty\to\lambda)\colon\Delta^0\to\cD(\sf{s}'',\lambda)$ by $\lambda(1)$, then it is dualizable in the symmetric monoidal $\infty$-category $\cD(\sf{s}'',\lambda)^\otimes$.
      \end{itemize}

  \item A t-structure on $\cD(X,\lambda)$ for every object $X$ of $\cC$ and every object $\lambda$ of $\cL$.

  \item (\emph{Trace map} for $\cEt$) A map $\Tr_f\colon \tau^{\ge 0}f_!\lambda_Y\langle d\rangle\to \lambda_X$ for every edge $f\colon Y\to X$ in $\cEt\cap\cC''_1$, every integer $d\geq\dim^+(f)$, and every object $\lambda$ of $\cL''$. Here, $\lambda_X$ is a unit object of the monoidal $\infty$-category $\cD(X,\lambda)$ and similarly for $\lambda_Y$; $-\langle d\rangle$ is defined in the same way as in Definition \ref{b3de:twist}.

  \item (\emph{Trace map} for $\cE'$) A map $\Tr_f\colon \tau^{\ge 0}f_!\lambda_Y\to \lambda_X$ for every edge $f\colon Y\to X$ in $\cE'\cap\cC'_1$ and every object $\lambda$ of $\cL'$, which coincides with the one above for $f\in\cE'\cap\cC''_1$.
\end{itemize}

Input II is subject to the following properties.

\begin{description}
  \item[P6] \emph{t-structure.} Let $\lambda$ be an arbitrary object of $\cL$. We have
      \begin{enumerate}
        \item For every object $X$ of $\cC$, we have $\lambda_X\in\cD^\heartsuit(X,\lambda)$.

        \item If $\lambda$ belongs to $\cL''$ and $X$ is an object of $\cC''$, then the auto-equivalence $-\otimes s_X^*\lambda(1)$ of $\cD(X,\lambda)$ is t-exact.

        \item For every object $X$ of $\cC$, the t-structure on $\cD(X,\lambda)$ is accessible, right complete, and $\cD^{\le -\infty}(X,\lambda)\colonequals\bigcap_n\cD^{\le -n}(X,\lambda)$ consists of zero objects.

        \item For every morphism $f\colon Y\to X$ of $\cC$, the functor $f^*\colon \cD(X,\lambda)\to\cD(Y,\lambda)$ is t-exact.
      \end{enumerate}

  \item[P7] \emph{Poincar\'{e} duality for $\cE''$.} We have
      \begin{enumerate}
         \item For every $f$ in $\cEt\cap\cC''_1$, every integer $d\geq\dim^+(f)$, and every object $\lambda$ of $\cL''$, the source of the trace map $\Tr_f$ belongs to the heart $\cD^\heartsuit(X,\lambda)$. Moreover, $\Tr_f$ is functorial in the same way as in Lemma \ref{b3le:trace}. See Remark \ref{b4re:trace} below for more details.

         \item For every $f$ in $\cE''_d\cap\cC''_1$, and every object $\lambda$ of $\cL''$, the map $u_f\colon f_!\circ f^*\langle d\rangle\to\id_X$, induced by the trace map $\Tr_f\colon \tau^{\ge 0}f_!\lambda_Y\langle d\rangle\to\lambda_X$ similarly as \eqref{b3eq:trace_1}, is a counit transformation. Here $\id_X$ is the identity functor of $\cD(X,\lambda)$.
      \end{enumerate}

  \item [P7\bis] \emph{Poincar\'{e} duality for $\cE'$.}  We have the same statement as in (P7) after letting $d=0$, and replacing $\cC''$ by $\cC'$, $\cEt$ by $\cE'$, and $\cL''$ by $\cL'$.
\end{description}

\begin{remark}\label{b4re:trace}
In (P7)(1) above, the trace maps $\Tr_f$ for all such $f$ and $\lambda$ are functorial in the following sense:
\begin{enumerate}
  \item For every morphism $\lambda\to\lambda'$ of $\cL''$, the diagram
      \[
      \xymatrix{&\tau^{\ge 0}f_!\lambda_Y\langle d\rangle\ar[rd]^-{\Tr_{f,\lambda}} \\
      \tau^{\ge 0}((\tau^{\ge 0}f_!\lambda'_Y\langle d\rangle)\otimes_{\lambda'_X} \lambda_X)
      \ar[ur]^-\sim\ar[rr]^-{\tau^{\ge 0}(\Tr_{f,\lambda'}\otimes_{\lambda'_X} \lambda_X)}
      && \lambda_X}
      \]
      commutes.

  \item For every Cartesian diagram
      \[
      \xymatrix{Y'\ar[r]^-{f'}\ar[d]_-v & X'\ar[d]^-u\\
      Y \ar[r]^f & X}
      \]
      of $\cC''$, the diagram
      \[
      \xymatrix{u^*\tau^{\ge 0}f_!\lambda_Y\langle d\rangle\ar[rr]^-{u^*\Tr_f}\ar[d]_-{\simeq} && u^*\lambda_X\ar[d]^-{\simeq}\\
      \tau^{\ge 0}f'_!\lambda_{Y'}\langle d\rangle\ar[rr]^-{\Tr_{f'}}
      && \lambda_{X'}}
      \]
      commutes.

  \item Consider a $2$-cell
      \[\xymatrix{
      Z \ar[rd]_-{g} \ar[rr]^-{h}  && X\\
      & Y \ar[ru]_-{f} }
      \]
      of $\cC''$ with $f,g\in\cEt\cap\cC''_1$ such that $\dim^+(f)\leq d$ and $\dim^+(g)\leq e$. In particular, we have $h\in\cEt\cap\cC''_1$ and $\dim^+(h)\leq d+e$. Then the diagram
      \[
      \xymatrix{\tau^{\ge 0}f_!(\tau^{\ge 0}g_!\lambda_Z\langle e\rangle)\langle d\rangle
      \ar[rr]^-{\tau^{\ge 0}f_!\Tr_g\langle d\rangle}\ar[d]_-{\simeq}
      && \tau^{\ge 0}f_!\lambda_Y\langle d\rangle\ar[d]^-{\Tr_f}\\
      \tau^{\ge 0}h_!\lambda_Z\langle d+e\rangle\ar[rr]^-{\Tr_h} && \lambda_X}
      \]
      commutes.
\end{enumerate}
\end{remark}

\begin{remark}\label{b4re:input}
We have the following remarks concerning input.
\begin{enumerate}
  \item (P0) and (P4) imply the following: If $f$ is an edge of $(\cC^{op}\times\cL^{op})^{\amalg,op}$ that statically belongs to $\cEs\cap\cE''$, then it is of universal $\EO{}{\cC}{\rI}{}$-descent.

  \item (P4) implies that (P3) holds for $f\in \cEs\cap\cE''$.

  \item If $d>\dim^+(f)$, then the trace map $\Tr_f$ is not interesting because its source $\tau^{\ge 0} f_!\lambda_Y\langle d\rangle$ is a
      zero object. We have included such maps in the data in order to state the functoriality as in Remark \ref{b4re:trace} more conveniently.

  \item We extend the trace map to morphisms $f\colon Y\to X$ in $\cEt\cap \cC''_1$ endowed with $2$-simplices \eqref{b4eq:2celld} satisfying $\dim^+(f_d)\le d$ and such that the morphisms $i_d$ exhibit $Y$ as $\coprod_{d\in\dZ} Y_d$. For every object $\lambda$ of $\cL''$, the map
      \[
      \cD(Y,\lambda)\to\prod_{d\in\dZ}\cD(Y_d,\lambda),
      \]
      induced by $i_d$ is an equivalence by (P1). We write $-\langle\dim^+ \rangle\colon \cD(Y,\lambda)\to \cD(Y,\lambda)$ for the product of $(-\langle d\rangle\colon \cD(Y_d,\lambda)\to\cD(Y_d,\lambda))_{d\in\dZ}$. Since $\lambda_Y\simeq\bigoplus_{d\in\dZ}i_{d!}\lambda_{Y_d}$, the maps $\Tr_{f_d}$ induce a map $\Tr_f\colon\tau^{\ge 0}f_!\lambda_Y\langle
      \dim^+\rangle \to \lambda_X$. Moreover,the trace map is functorial in the sense that an analogue of Remark \ref{b4re:trace} holds.

  \item (P7)(2) still holds for morphisms $f\colon Y\to X$ in $\cE''\cap\cC''_1$. For such morphisms, the $2$-simplices in Input 0(5) are unique up to equivalence by Input 0(6). We write $-\langle \dim f\rangle\colon \cD(Y,\lambda)\to \cD(Y,\lambda)$ for the product of $(-\langle d\rangle\colon\cD(Y_d,\lambda)\to\cD(Y_d,\lambda))_{d\in\dZ}$. Then, (P7)(2) for the morphisms $f_d$ implies that the map $u_f\colon f_!\circ f^*\langle \dim f\rangle \to \id_X$ induced by the trace map $\Tr_f\colon \tau^{\ge 0}f_!\lambda_Y\langle d\rangle\to \lambda_X$ is a counit transformation.
\end{enumerate}
\end{remark}

The output has two parts: I \& II.

\paragraph*{\textbf{Output I}} Output I consists of two maps as follows.
\begin{itemize}
  \item The \emph{first abstract operation map}:
      \[
      \EO{}{\tilde\cC}{\rI}{}\colon(\tilde\cC^{op}\times\cL^{op})^\amalg\to\Cat
      \]
      extending $\EO{}{\cC}{\rI}{}$.

  \item The \emph{second abstract operation map}:
      \[
      \EO{}{\tilde\cC'}{\r{II}}{}\colon
      \delta^*_{2,\{2\}}((\tilde\cC'^{op}\times\cL'^{op})^{\amalg,op})^\cart_{\tilde\cF',\all}\to\Cat
      \]
      extending $\EO{}{\cC'}{\r{II}}{}$.
\end{itemize}

\paragraph*{\textbf{Output II}} Output II consists of the following data, all extending the existed data in Input II.
\begin{itemize}
  \item A functor $\r{tw}\colon(\cL''^{op})^\triangleleft\to\Cat$ same as in Input II.

  \item A t-structure on $\cD(X,\lambda)$ for every object $X$ of $\tilde\cC$ and every object $\lambda$ of $\cL$.

  \item (\emph{Trace map} for $\tcEt$) A map $\Tr_f\colon \tau^{\ge 0} f_!\lambda_Y\langle d\rangle\to \lambda_X$ for every edge $f\colon
      Y\to X$ in $\tcEt\cap\tilde\cC''_1$, every integer $d\geq\dim^+(f)$, and every object $\lambda$ of $\cL''$.

  \item (\emph{Trace map} for $\tilde\cE'$) A map $\Tr_f\colon \tau^{\ge 0} f_!\lambda_Y\to \lambda_X$ for every edge $f\colon Y\to X$ in
      $\tilde\cE'\cap\tilde\cC'_1$ and every object $\lambda$ of $\cL'$, which coincides with the one above for $f\in\tilde\cE'\cap\tilde\cC''_1$.
\end{itemize}

We introduce properties (P0) through (P7\bis) for Output I and II by replacing $\cC'$, $\cC''$ and $(\cC,\cEs,\cE',\cE'',\cEt,\cF)$ by
$\tilde\cC'$, $\tilde\cC''$ and $(\tilde\cC,\tcEs,\tilde\cE',\tilde\cE'',\tcEt,\tilde\cF)$, respectively. The following theorem shows how our program works.

\begin{theorem}\label{b4th:program}
Fix an Input 0. Then
\begin{enumerate}
  \item Every Input I satisfying (P0) through (P5\bis) can be extended to an Output I satisfying (P0) through (P5\bis).

  \item For given Input I, II satisfying (P0) through (P7\bis) and given Output I extending Input I and satisfying (P0) through (P5\bis),
      there exists an Output II extending Input II and satisfying (P6), (P7), (P7\bis).
\end{enumerate}
\end{theorem}

Output I will be accomplished in \Sec\ref{b4ss:construction}. Output II and the proof of properties (P1) through (P7\bis) will be accomplished in \Sec\ref{b4ss:properties}.

\begin{variant}\label{b4va:program}
Let us introduce a variant of DESCENT. In Input 0, we let $\tilde\cE'=\tilde\cE''$, $\sf{s}'\to\sf{s}''$ be a degenerate edge, $\cL'=\cL''$, and \emph{ignore} (9). In Input II (resp.\ Output II), we also \emph{ignore} the trace map for $\cE'$ (resp.\ $\tilde\cE'$) and property
(P7\bis). In particular, (P5) and (P5\bis) coincide. Theorem \ref{b4th:program} for this variant still holds and will be applied to (higher) Artin stacks.
\end{variant}

\begin{remark}\label{b4re:program}
We have the following remarks concerning Theorem \ref{b4th:program}.
\begin{enumerate}
  \item If the only goal is to extend the first and second operation maps, the statement of Theorem \ref{b4th:program}(1) can be made more compact: every Input I satisfying properties (P0), (P2), (P4), and (P5) can be extended to an Output I satisfying (P0), (P2), (P4), and (P5). This will follow from our proof of Theorem \ref{b4th:program} in this chapter.

  \item The Output I in Theorem \ref{b4th:program}(1) is unique up to equivalence. More precisely, we can define a simplicial set $K$
      classifying those Input I that satisfy (P2) and (P4). The vertices of $K$ are triples $(\EO{}{\cC}{\rI}{}, \EO{}{\cC}{\r{II}}{}, h)$,
      where $h$ is the equivalence in (P2). Similarly, let $\tilde{K}$ be the simplicial set classifying those Output II that satisfy (P2)
      and (P4). Then the restriction map $\tilde{K}\to K$ satisfies the right lifting property with respect to $\partial\Delta^n\subseteq\Delta^n$ for all $n\geq 1$. One can show this by adapting our proof of Theorem \ref{b4th:program}. Moreover, in all the above, $h$ can be taken to be the identity without loss of generality.

  \item The Output II in Theorem \ref{b4th:program}(2) is also unique up to equivalence. More precisely, let us fix an Output I extending
      Input I and satisfying (P2) and (P4). Note that the functor $\r{tw}$ remains the same. Fix an assignment of t-structures for the Input satisfying (P6). Then there exists a unique extension to the Output satisfying (P6). Moreover, for every assignment of traces for the Input satisfying (P7) (resp.\ (P7\bis)), there exists a unique extension to the Output satisfying (P7) (resp.\ (P7\bis)). Note that the trace map is defined in the heart, so that no homotopy issue arises.
\end{enumerate}
\end{remark}

\begin{definition}\label{b4de:dimension}
For a morphism $f\colon Y\to X$ locally of finite type between algebraic spaces, we define the \emph{upper relative dimension} of $f$ to be
\[
\sup\{\dim (Y\times_X\Spec\Omega)\}\in \dZ\cup\{-\infty,+\infty\}
\]
\cite{SP}*{04N6}, where the supremum is taken over all geometric points $\Spec\Omega\to X$. We adopt the convention that the empty scheme has dimension $-\infty$.
\end{definition}

\begin{example}\label{b4ex:scheme}
The initial input for DESCENT is the following:
\begin{itemize}
  \item $\tilde\cC=\rN(\Schqs)$, where $\Schqs\subseteq\Sch$ is the full subcategory spanned by quasi-separated schemes as in Example \ref{b4ex:toy}. It is geometric and admits $\Spec\dZ$ as a final object.

  \item $\cC=\rN(\Schqcs)$, and $\sf{s}''\to\sf{s}'$ is the unique morphism $\Spec\dZ[\Box^{-1}]\to\Spec\dZ$. In particular, $\cC'=\cC$ and $\tilde\cC'=\tilde\cC$.

  \item $\tcEs$ is the set of \emph{surjective} morphisms.

  \item $\tilde\cE'$ is the set of \emph{\'{e}tale} morphisms.

  \item $\tilde\cE''$ is the set of \emph{smooth} morphisms.

  \item $\tilde\cE''_d$ is the set of \emph{smooth} morphisms of pure relative dimension $d$.

  \item $\tcEt$ is the set of morphisms that are \emph{flat and locally of finite presentation}.

  \item $\tilde\cF$ is the set of morphisms \emph{locally of finite type}.

  \item $\cL=\rN(\Rind)^{op}$, $\cL'=\rN(\Rind_\tor)^{op}$, and $\cL''=\rN(\Rind_{\ltor})^{op}$.

  \item $\dim^+$ is the (function of) upper relative dimension (Definition \ref{b4de:dimension}).

  \item $\EO{}{\cC}{\rI}{}$ is \eqref{b3eq:premonoidal}, and $\EO{}{\cC'}{\r{II}}{}$ is \eqref{b3eq:operation}.

  \item $\r{tw}$ is defined in Definition \ref{b3de:twist}.

  \item $\cD(X,\lambda)$ is endowed with its usual t-structure recalled at the beginning of \Sec\ref{b3ss:poincare}.

  \item The trace maps are the classical ones  \eqref{b3eq:trace}; see also Remark \ref{b3re:trace_quasifinite}.
\end{itemize}

Properties (P0) through (P7\bis) are satisfied as follows:
\begin{itemize}
  \item[(P0)] This is Lemma \ref{b2le:closed_cartesian}(1,2).

  \item[(P1)] This is Lemma \ref{b2le:closed_cartesian}(3).

  \item[(P2)] This follows from our construction. In fact, the two maps are equal in this case.

  \item[(P3)] This is obvious.

  \item[(P4)] This is Proposition \ref{b3pr:descent}.

  \item[(P5)] This follows from Lemma \ref{b4le:outputii} below. Part (1) of (P5), namely the \'etale base change, is trivial.

  \item[(P5\bis)] This follows from Lemma \ref{b4le:outputii} below. Part (1) of (P5\bis) is the smooth base change.

  \item[(P6)] Part (3) follows from \cite{HA}*{Proposition 1.3.5.21}. The rest follows from construction.

  \item[(P7)] This has been recalled in Lemma \ref{b3le:trace} and Lemma \ref{b3le:poincare}.

  \item[(P7\bis)] This has been recalled in Remark \ref{b3re:trace_quasifinite}.
\end{itemize}
\end{example}

\begin{lem}\label{b4le:outputii}
Assume (P7). Then (P5) holds. In fact, we have the stronger result that part (2) of (P5) holds without the assumption that $p$ is also in $\cE'$. The similar statements hold concerning (P7\bis) and (P5\bis).
\end{lem}

\begin{proof}
We denote by $p_*$ (resp.\ $q_*$) a right adjoint of $p^*$ (resp.\ $q^*$) and by $f^!$ (resp.\ $g^!$) a right adjoint of $f_!$ (resp.\ $g_!$).

By (P7) or (P7\bis), $f^*$ and $g^*$ have left adjoints. Moreover, the diagram
\begin{align}\label{b4eq:outputii}
\resizebox{\hsize}{!}{
\xymatrix{
f^*p_*\langle\dim f\rangle \ar[r]\ar[d]\ar@/_4pc/[dd]_{\simeq} & q_* g^*\langle\dim f\rangle\ar@{=}[rr]\ar[d]
&&q_* g^*\langle\dim f\rangle\ar[d]\ar@/^4pc/[dd]^{\simeq}\\
f^!f_!f^*p_* \ar[r]\ar[d]^-{\Tr_f}\langle\dim f\rangle\ar[r] & f^!f_!q_*g^*\langle\dim f\rangle\ar[r]
& f^!p_*g_!g^*\langle\dim f\rangle\ar[r]^\sim\ar[d]_-{\Tr_g} &  q_*g^!g_!g^*\langle\dim f\rangle\ar[d]_-{\Tr_g}\\
f^!p_*\ar@{=}[rr]&&f^!p_*\ar[r]^{\sim} & q_*g^!
}
}
\end{align}
is commutative up to homotopy. It follows that the top horizontal arrow is an equivalence.

Since the diagram
\begin{align*}
\resizebox{\hsize}{!}{
\xymatrix{
q^*f^*\langle\dim f\rangle\ar@/_4pc/[dd]_{\simeq} \ar@{=}[rr]\ar[d]
&&q^*f^*\langle\dim f\rangle\ar[r]^\sim & g^*p^*\langle\dim f\rangle\ar@/^4pc/[dd]^\simeq \ar[d]\\
q^*f^!f_!f^*\langle\dim f\rangle\ar[r]\ar[d]_-{\Tr_f} & g^!p^*f_!f^*\langle\dim f\rangle\ar[r]^\sim\ar[d]^-{\Tr_f}
& g^!g_!q^*f^* \langle\dim f\rangle\ar[r]^\sim & g^!g_!g^*p^*\langle\dim f\rangle\ar[d]_-{\Tr_g}\\
q^*f^!\ar[r] & g^!p^*\ar@{=}[rr] && g^!p^*
}
}
\end{align*}
is commutative up to homotopy, the bottom horizontal arrow is an equivalence.
\end{proof}

\subsection{Construction}
\label{b4ss:construction}

The goal of this subsection is to construct the maps $\EO{}{\tilde\cC}{\rI}{}$ and $\EO{}{\tilde\cC'}{\r{II}}{}$ in Output I in \Sec\ref{b4ss:description}. We will construct Output II and check the properties (P0) -- (P7\bis) in the next section.

Let us start from the construction of second abstract operation map $\EO{}{\tilde\cC'}{\r{II}}{}$. The first one $\EO{}{\tilde\cC}{\rI}{}$ will be constructed at the end of this section, after the proof of Proposition \ref{b4pr:toy}.

Let $\cR\subseteq\tilde\cF'$ be the subset of
morphisms that are representable in $\cC'$. We have successive inclusions
\begin{align*}
\delta^*_{2,\{2\}}((\cC'^{op}\times\cL'^{op})^{\amalg,op})^\cart_{\cF',\all}
&\subseteq\delta^*_{2,\{2\}}((\tilde\cC'^{op}\times\cL'^{op})^{\amalg,op})^\cart_{\cR,\all} \\
&\subseteq\delta^*_{2,\{2\}}((\tilde\cC'^{op}\times\cL'^{op})^{\amalg,op})^\cart_{\tilde\cF',\all}.
\end{align*}
We proceed in two steps.\\

\paragraph*{\textbf{Step 1}} We first extend $\EO{}{\cC'}{\r{II}}{}$ to the map $\EO{\cR}{\cC'}{\r{II}}{}$ with the new source
\[
\delta^*_{2,\{2\}}((\tilde\cC'^{op}\times\cL'^{op})^{\amalg,op})^\cart_{\cR,\all}.
\]
An $n$-cell of the above source is given by a functor
\[
\sigma\colon\Delta^n\times(\Delta^n)^{op}\to(\tilde\cC'^{op}\times\cL'^{op})^{\amalg,op}
\]
We define $\Cov(\sigma)$ to be the full subcategory of
\[
\Fun(\Delta^n\times(\Delta^n)^{op}\times\rN(\del_+)^{op},(\tilde\cC'^{op}\times\cL'^{op})^{\amalg,op})
\times_{\Fun(\Delta^n\times(\Delta^n)^{op}\times\{[-1]\},(\tilde\cC'^{op}\times\cL'^{op})^{\amalg,op})}\{\sigma\}
\]
spanned by functors $\sigma^0\colon \Delta^n\times(\Delta^n)^{op}\times\rN(\del_+)^{op}\to(\tilde\cC'^{op}\times\cL'^{op})^{\amalg,op}$
such that
\begin{itemize}
  \item for every $0\leq j\leq n$, the restriction $\sigma^0\res\Delta^{(n,j)}\times\rN(\del_+^{\leq 0})^{op}$, regarded as an edge of
      $(\tilde\cC'^{op}\times\cL'^{op})^{\amalg,op}$, is statically an atlas (see Definition \ref{b3de:statically} and Input 0(4));

  \item $\sigma^0$ is a right Kan extension of $\sigma^0\res\Delta^{\{n\}}\times(\Delta^n)^{op}\times\rN(\del_+^{\leq0})^{op}\cup\Delta^n\times(\Delta^n)^{op}\times\{[-1]\}$ along the obvious inclusion.
\end{itemize}
In particular, for every object $(i,j)$ of $\Delta^n\times(\Delta^n)^{op}$, the restriction $\sigma^0\res\Delta^{(i,j)}\times\rN(\del_+)^{op}$ is a \v{C}ech nerve of the restriction $\sigma^0\res\Delta^{(i,j)}\times\rN(\del_+^{\leq0})^{op}$.

The $\infty$-category $\Cov(\sigma)$ is nonempty by Input 0(4), and admits product of two objects. Indeed, for every pair of objects $\sigma_1^0$ and $\sigma_2^0$ of $\Cov(\sigma)$, the assignment
\[
(i,j,[k])\mapsto\sigma_1^0(i,j,[k])\times_{\sigma(i,j)}\sigma_2^0(i,j,[k])
\]
induces a product of $\sigma_1^0$ and $\sigma_2^0$ by Lemma \ref{b1le:static_colimit}. Therefore, by Lemma \ref{b1le:trivial_contractible}, $\Cov(\sigma)$ is a weakly contractible Kan complex.

Since atlases are representable in $\cC$ by Input 0(4), by restriction, $\Cov(\sigma)$ induces a functor
\[
\Cov(\sigma)\to\Fun(\rN(\del)^{op}\times\Delta^n\times(\Delta^n)^{op},(\cC'^{op}\times\cL'^{op})^{\amalg,op}),
\]
which induces a map
\[
\Cov(\sigma)^{op}\to\Fun(\rN(\del),\Fun(\Delta^n,\delta^*_{2,\{2\}}((\cC'^{op}\times\cL'^{op})^{\amalg,op})^\cart_{\cF',\all})).
\]
Composing with the map $\EO{}{\cC'}{\r{II}}{}$, we obtain a functor
\[
\phi(\sigma)\colon\Cov(\sigma)^{op}\to\Fun(\rN(\del),\Fun(\Delta^n,\Cat)).
\]
Let $\cK\subseteq\Fun(\rN(\del_+),\Fun(\Delta^n,\Cat))$ be the full subcategory spanned by those functors $F\colon\rN(\del_+)\to\Fun(\Delta^n,\Cat)$ that are right Kan extensions of $F\res\rN(\del)$. Consider the following diagram
\[
\xymatrix{
& \cN(\sigma) \ar[r]\ar[d]^{\rres_1^*\phi(\sigma)} & \Cov(\sigma)^{op} \ar[d]^-{\phi(\sigma)} \\
\Fun(\Delta^n,\Cat) & \cK \ar[l]_-{\rres_2}\ar[r]^-{\rres_1} &
\Fun(\rN(\del),\Fun(\Delta^n,\Cat)) }
\]
in which the right square is Cartesian, and $\rres_2$ is the restriction to $\{[-1]\}$. Put
\[
\Phi(\sigma)\coloneqq\rres_2\circ\rres_1^*\phi(\sigma)\colon\cN(\sigma)\to\Fun(\Delta^n,\Cat).
\]
It is easy to see that the above process is functorial so that the collection of $\Phi(\sigma)$ defines a morphism $\Phi$ of the category
\[
(\Sset)^{(\del_{\delta^*_{2,\{2\}}((\tilde\cC'^{op}\times\cL'^{op})^{\amalg,op})^\cart_{\cR,\all}})^{op}}.
\]

\begin{lem}\label{b4le:sharp}
The map $\Phi(\sigma)$ takes values in $\Map^\sharp((\Delta^n)^\flat,\Cat^\natural)$.
\end{lem}

\begin{proof}
Let $X_{-1}$ be an object of $(\tilde\cC'^{op}\times\cL'^{op})^{\amalg,op}$, and $\Cov(X_{-1})$ the full subcategory of
\[
\Fun(\rN(\del_+)^{op},(\tilde\cC'^{op}\times\cL'^{op})^{\amalg,op})
\times_{\Fun(\{[-1]\},(\tilde\cC'^{op}\times\cL'^{op})^{\amalg,op})}\{X_{-1}\}
\]
spanned by functors $X_\bullet$ such that the edge $X_0\to X_{-1}$ is statically an atlas and $X_\bullet$ is a \v{C}ech nerve of $X_0\to X_{-1}$. By (P2), it suffices to show that for every morphism $f$ of $\Cov(X_{-1})$, considered as a functor $f\colon\Delta^1\times\rN(\del_+)^{op}\to(\tilde\cC'^{op}\times\cL'^{op})^{\amalg,op}$, and every right Kan extension $F$ of
$\EO{}{\cC}{\rI}{}\circ(f\res\Delta^1\times \rN(\del)^{op})^{op}$, the morphism $F\res(\Delta^1\times\{[-1]\})^{op}$ is an equivalence in $\Cat$.

In fact, let $f\colon X_\bullet^0 \to X_\bullet^1$ be a morphism of $\Cov(X_{-1})$. Let $X_\bullet^2$ be an object of $\Cov(X_{-1})$. Then we
have a diagram
\[
\xymatrix{
& X_\bullet^0\times X_\bullet^2\ar@/_1pc/[ld]_-{\pr} \ar[r]^-{\pr} \ar[d]^-{f\times X_\bullet^2} & X_\bullet^0 \ar[d]^-{f} \\
X_\bullet^2&  X_\bullet^1\times X_\bullet^2\ar[l]_-{\pr}\ar[r]^-{\pr} & X_\bullet^1. }
\]
Here products are taken in $\Cov(X_{-1})$. Thus, it suffices to show the assertion for the projection $X_\bullet\times X'_\bullet\to X'_\bullet$, where $X_\bullet$ and $X'_\bullet$ are objects of $\Cov(X_{-1})$.

Let $Y_{\bullet\bullet}\colon\rN(\del_+)^{op}\times\rN(\del_+)^{op}\to\tilde\cC'$ be an augmented bisimplicial object of $\tilde\cC'$ such that
\begin{itemize}
  \item $Y_{-1\bullet}=X'_\bullet$, $Y_{\bullet-1}=X_\bullet$.

  \item $Y_{\bullet\bullet}$ is a right Kan extension of $Y_{-1\bullet}\cup Y_{\bullet-1}$.
\end{itemize}
Let $\delta\colon [1]\times\del_+^{op}\to\del_+^{op}\times\del_+^{op}$ be the functor sending $(0,[n])$ to $([n],[n])$ and $(1,[n])$ to $([-1],[n])$. It suffices to show the assertion for $Y_{\bullet\bullet}\circ\rN(\delta)$, regarded as a morphism of $\Cov(X_{-1})$. This follows from Lemma \ref{b4le:plimit} below by taking $p$ to be $\Cat\to*$ and $c^{\bullet\bullet}$ to be a right Kan extension of
$\EO{}{\cC}{\rI}{}\circ (Y_{\bullet\bullet}\res\rN(\del_{++})^{op})^{op}$. Here, $\del_{++}\subseteq\del_+\times\del_+$ is the full subcategory spanned by all objects except the initial one. Assumptions (2) and (3) of Lemma \ref{b4le:plimit} are satisfied thanks to (P0) and (P4); see Remark \ref{b4re:input}(1).
\end{proof}

\begin{lem}\label{b4le:plimit}
Let $p\colon \cC\to\cD$ be a categorical fibration of $\infty$-categories. Let $c^{\bullet\bullet}\colon \rN(\del_+)\times\rN(\del_+)\to\cC$ be an augmented bicosimplicial object of $\cC$. For $n\geq -1$, put $c^{n\bullet}\coloneqq c^{\bullet\bullet}\res\{[n]\}\times\rN(\del_+)$ and
$c^{\bullet n}\coloneqq c^{\bullet\bullet}\res\rN(\del_+)\times\{[n]\}$, respectively. Assume that
\begin{enumerate}[(a)]
  \item $c^{\bullet\bullet}$ is a $p$-limit \cite{HTT}*{Definition 4.3.1.1} of $c^{\bullet\bullet}\res\rN(\del_{++})$, where $\del_{++}\subseteq\del_+\times\del_+$ is the full subcategory spanned by all objects except the initial one.

  \item For every $n\geq0$, $c^{n\bullet}$ is a $p$-limit of $c^{n\bullet}\res\rN(\del)$.

  \item For every $n\geq0$, $c^{\bullet n}$ is a $p$-limit of $c^{\bullet n}\res\rN(\del)$.
\end{enumerate}
Then
\begin{enumerate}
  \item $c^{-1\bullet}$ is a $p$-limit of $c^{-1\bullet}\res\{[-1]\}\times\rN(\del)$.

  \item $c^{\bullet-1}$ is a $p$-limit of $c^{\bullet-1}\res\rN(\del)\times\{[-1]\}$.

  \item $c^{\bullet\bullet}\res\rN(\del_+)_\diag$ is a $p$-limit of $c^{\bullet\bullet}\res\rN(\del)_\diag$, where
      $\rN(\del_+)_\diag\subseteq\rN(\del_+)\times\rN(\del_+)$ is the image of the diagonal inclusion $\diag\colon\rN(\del_+)\to\rN(\del_+)\times\rN(\del_+)$ and $\rN(\del)_\diag$ is defined similarly.
\end{enumerate}
\end{lem}

\begin{proof}
For (1), we apply (the dual version of) \cite{HTT}*{Proposition 4.3.2.8} to $p$ and $\rN(\del_+\times\del)\subseteq\rN(\del_{++})\subseteq\rN(\del_+\times\del_+)$. By (the dual version of) \cite{HTT}*{Proposition 4.3.2.9} and assumption (b), the restriction $c^{\bullet\bullet}\res\rN(\del\times \del_+)$ is a $p$-right Kan extension of the restriction $c^{\bullet\bullet}\res\rN(\del\times\del)$ \cite{HTT}*{Definition 4.3.2.2}. It follows that $c^{\bullet\bullet}\res\rN(\del_{++})$ is a $p$-right Kan extension of $c^{\bullet\bullet}\res\rN(\del_+\times\del)$. By assumption (a), $c^{\bullet\bullet}$ is a $p$-right Kan extension of $c^{\bullet\bullet}\res\rN(\del_{++})$. Therefore, $c^{\bullet\bullet}$ is a $p$-right Kan extension of $c^{\bullet\bullet}\res\rN(\del_+\times\del)$. By \cite{HTT}*{Proposition 4.3.2.9} again, $c^{-1\bullet}$ is a $p$-limit of $c^{-1\bullet}\res\{[-1]\}\times\rN(\del)$.

For (2), it follows from conclusion (1) by symmetry.

For (3), we view $(\del\times\del)^\triangleleft$ as a full subcategory of $\del_+\times\del_+$ by sending the cone point to the initial object. By \cite{HTT}*{Lemma 4.3.2.7}, we find that $c^{\bullet\bullet}\res(\del\times\del)^\triangleleft$ is a $p$-limit diagram. By \cite{HTT}*{Lemma 5.5.8.4}, the simplicial set $\rN(\del)^{op}$ is \emph{sifted} \cite{HTT}*{Definition 5.5.8.1}, that is, the diagonal map $\rN(\del)^{op}\to\rN(\del)^{op}\times\rN(\del)^{op}$ is cofinal. Therefore, $c^{\bullet\bullet}\res\rN(\del_+)_\diag$ is a $p$-limit of
$c^{\bullet\bullet}\res\rN(\del)_\diag$.
\end{proof}

Since $\rres_1$ is a trivial fibration by \cite{HTT}*{Proposition 4.3.2.15}, the simplicial set $\cN(\sigma)$ is weakly contractible. By Lemma \ref{b4le:sharp}, we can apply Proposition \ref{a1pr:extension} to
\[
K=\delta^*_{2,\{2\}}((\tilde\cC'^{op}\times\cL'^{op})^{\amalg,op})^\cart_{\cR,\all},\quad
K'=\delta^*_{2,\{2\}}((\cC'^{op}\times\cL'^{op})^{\amalg,op})^\cart_{\cF',\all},\quad
g\colon K'\hookrightarrow K,
\]
and the section $\nu$ given by $\EO{}{\cC'}{\r{II}}{}$. This extends $\EO{}{\cC'}{\r{II}}{}$ to a map
\[
\EO{\cR}{\tilde\cC'}{\r{II}}{}\colon \delta^*_{2,\{2\}}((\tilde\cC'^{op}\times\cL'^{op})^{\amalg,op})^\cart_{\cR,\all}\to\Cat.\\
\]

\paragraph*{\textbf{Step 2}} Now we are going to extend $\EO{\cR}{\tilde\cC'}{\r{II}}{}$ to the map $\EO{}{\tilde\cC'}{\r{II}}{}$ with the new source
\[
\delta^*_{2,\{2\}}((\tilde\cC'^{op}\times\cL'^{op})^{\amalg,op})^\cart_{\tilde\cF',\all}.
\]
An $n$-cell of the above source is given by a functor
\[
\varsigma\colon\Delta^n\times(\Delta^n)^{op}\to(\tilde\cC'^{op}\times\cL'^{op})^{\amalg,op}
\]
We define $\Kov(\varsigma)$ to be the full subcategory of
\[
\Fun(\Delta^n\times(\Delta^n)^{op}\times\rN(\del_+)^{op},(\tilde\cC'^{op}\times\cL'^{op})^{\amalg,op})
\times_{\Fun(\Delta^n\times(\Delta^n)^{op}\times\{[-1]\},(\tilde\cC'^{op}\times\cL'^{op})^{\amalg,op})}\{\varsigma\}
\]
spanned by functors $\varsigma^0\colon\Delta^n\times(\Delta^n)^{op}\times\rN(\del_+)^{op}\to(\tilde\cC'^{op}\times\cL'^{op})^{\amalg,op}$ such that
\begin{itemize}
  \item for every $0\leq i\leq n$, the restriction $\varsigma^0\res\Delta^{(i,0)}\times\rN(\del_+^{\leq 0})^{op}$, regarded as an edge of
      $(\tilde\cC'^{op}\times\cL'^{op})^{\amalg,op}$, statically belongs to $\tcEs\cap\tilde\cE'\cap\cR$;

  \item $\varsigma^0$ is a right Kan extension of $\varsigma^0\res\Delta^n\times(\Delta^{\{0\}})^{op}\times\rN(\del_+^{\leq0})^{op}\cup\Delta^n\times(\Delta^n)^{op}\times\{[-1]\}$ along the obvious inclusion;

  \item the restriction $\varsigma^0\res\Delta^n\times(\Delta^{\{0\}})^{op}\times\{[0]\}$ corresponds to an $n$-cell of
      $(\tilde\cC'^{op}\times\cL'^{op})^{\amalg,op})_\cR$.
\end{itemize}
In particular, for every object $(i,j)$ of $\Delta^n\times(\Delta^n)^{op}$, the restriction $\varsigma^0\res\Delta^{(i,j)}\times\rN(\del_+)^{op}$ is a \v{C}ech nerve of the restriction
$\varsigma^0\res\Delta^{(i,j)}\times\rN(\del_+^{\leq0})^{op}$. Moreover, the restriction $\varsigma^0\res\Delta^n\times(\Delta^n)^{op}\times\{[0]\}$ corresponds to an $n$-cell of
$\delta^*_{2,\{2\}}((\tilde\cC'^{op}\times\cL'^{op})^{\amalg,op})^\cart_{\cR,\all}$.

Similar to $\Cov(\sigma)$, the $\infty$-category $\Kov(\varsigma)$ is nonempty and admits product of two objects. Therefore, by Lemma
\ref{b1le:trivial_contractible}, $\Kov(\varsigma)$ is a weakly contractible Kan complex.

The restriction functor
\[
\Kov(\varsigma)\to\Fun(\rN(\del)^{op}\times\Delta^n\times(\Delta^n)^{op},(\cC'^{op}\times\cL'^{op})^{\amalg,op})
\]
induces a map
\[
\Kov(\varsigma)\to\Fun(\rN(\del)^{op},\Fun(\Delta^n,\delta^*_{2,\{2\}}((\cC'^{op}\times\cL'^{op})^{\amalg,op})^\cart_{\cF',\all})).
\]
Composing with the map $\EO{\cR}{\tilde\cC'}{\r{II}}{}$, we obtain a functor
\[
\phi(\varsigma)\colon\Kov(\sigma)\to\Fun(\rN(\del)^{op},\Fun(\Delta^n,\Cat)).
\]
Let $\cK'\subseteq\Fun(\rN(\del_+)^{op},\Fun(\Delta^n,\Cat))$ be the full subcategory spanned by those functors
$F\colon\rN(\del_+)^{op}\to\Fun(\Delta^n,\Cat)$ that are left Kan extensions of $F\res\rN(\del)^{op}$. Consider the following diagram
\[
\xymatrix{
& \cN(\varsigma) \ar[r]\ar[d]^{\rres_1^*\phi(\varsigma)} & \Kov(\varsigma) \ar[d]^-{\phi(\varsigma)} \\
\Fun(\Delta^n,\Cat) & \cK' \ar[l]_-{\rres_2}\ar[r]^-{\rres_1} &
\Fun(\rN(\del)^{op},\Fun(\Delta^n,\Cat)) }
\]
in which the right square is Cartesian, and $\rres_2$ is the restriction to $\{[-1]\}$. Put
\[
\Phi(\varsigma)\coloneqq\rres_2\circ\rres_1^*\phi(\varsigma)\colon\cN(\varsigma)\to\Fun(\Delta^n,\Cat).
\]
It is easy to see that the above process is functorial so that the collection of $\Phi(\varsigma)$ defines a morphism $\Phi$ of the category
\[
(\Sset)^{(\del_{\delta^*_{2,\{2\}}((\tilde\cC'^{op}\times\cL'^{op})^{\amalg,op})^\cart_{\tilde\cF',\all}})^{op}}.
\]

\begin{lem}\label{b4le:sharp2}
The map $\Phi(\varsigma)$ takes values in $\Map^\sharp((\Delta^n)^\flat,\Cat^\natural)$.
\end{lem}

\begin{proof}
Let $X_\bullet\colon\rN(\del_+)^{op}\to(\tilde\cC'^{op}\times\cL'^{op})^{\amalg,op}$ be an augmented simplicial object that is a \v{C}ech nerve of $f\colon X_0\to X_{-1}$ such that $f$ statically belongs to $\tcEs\cap\tilde\cE'\cap\cR$. By the construction of $\Phi(\varsigma)$, it suffices to show that $R\circ X_\bullet$ is a left Kan extension of $R\circ X_\bullet\res\rN(\del)^{op}$, where
$R=\EO{\cR}{\tilde\cC'}{\r{II}}{}\res((\tilde\cC'^{op}\times\cL'^{op})^{\amalg,op})_\cR$ is the restriction along direction 1.

Choose an object $X'_\bullet$ of $\Cov(X_{-1})$ and form a bisimplicial 
object 
$Y_{\bullet\bullet}\colon\rN(\del_+)^{op}\times\rN(\del_+)^{op}\to(\tilde\cC'^{op}\times\cL'^{op})^{\amalg,op}$ 
as in the proof of Lemma \ref{b4le:sharp}, which is static. Applying 
$\EO{\cR}{\tilde\cC'}{\r{II}}{}$ to $Y_{\bullet\bullet}$ and by adjunction, 
we obtain a diagram $\chi^\bullet_\bullet\colon \rN(\del_+)^{op}\times 
\rN(\del_+)\to \Cat$. By the construction of 
$\EO{\cR}{\tilde\cC'}{\r{II}}{}$, we have that $\chi^\bullet_n$ is a limit 
diagram for $n\geq -1$. By (P4), $\chi^n_\bullet$ is a colimit diagram for 
$n\ge 0$. Therefore, by (P5)(2) and \cite{HA}*{Proposition 4.7.4.19} applied 
to the restriction $\chi_\bullet^\bullet\res 
\rN(\del_{s,+})^{op}\times\rN(\del_{s,+})$, we have that $R\circ 
X_\bullet=\chi_\bullet^{-1}$ is a colimit diagram. In the last sentence, we 
used \cite{HTT}*{Lemma 6.5.3.7} twice. 
\end{proof}

Since $\rres_1$ is a trivial fibration by \cite{HTT}*{Proposition 4.3.2.15}, the simplicial set $\cN(\varsigma)$ is weakly contractible. By Lemma \ref{b4le:sharp2}, we can apply Proposition \ref{a1pr:extension} to
\[
K=\delta^*_{2,\{2\}}((\tilde\cC'^{op}\times\cL'^{op})^{\amalg,op})^\cart_{\tilde\cF',\all},\quad
K'=\delta^*_{2,\{2\}}((\tilde\cC'^{op}\times\cL'^{op})^{\amalg,op})^\cart_{\cR,\all},\quad
g\colon K'\hookrightarrow K,
\]
and the section $\nu$ given by $\EO{\cR}{\tilde\cC'}{\r{II}}{}$. This extends $\EO{\cR}{\tilde\cC'}{\r{II}}{}$ to a map
\[
\EO{}{\tilde\cC'}{\r{II}}{}\colon \delta^*_{2,\{2\}}((\tilde\cC'^{op}\times\cL'^{op})^{\amalg,op})^\cart_{\tilde\cF',\all}\to\Cat,
\]
as demanded.

Now we prove Proposition \ref{b4pr:toy}, which will be applied to construct the first abstract operation map $\EO{}{\tilde\cC}{\rI}{}$ in Output I.

\begin{proof}[Proof of Proposition \ref{b4pr:toy}]
The proof is similar to Step 1 above. Consider the diagram
\[
\xymatrix{
\partial \Delta^n\ar@{^(->}[d]\ar[r]^-G & \Fun^{\tilde\cE}(\tilde\cC^{op},\cD)\ar[d]\\
\Delta^n\ar[r]_-{F}\ar@{-->}[ru] & \Fun^\cE(\cC^{op},\cD).}
\]
Let $\sigma\colon (\Delta^m)^{op}\to \cC$ be an $m$-cell of ${\tilde\cC}^{op}$. We denote by $\Cov(\sigma)$ the full subcategory of
\[
\Fun((\Delta^m)^{op}\times \rN(\del_+)^{op},\tilde\cC)\times_{\Fun((\Delta^m)^{op}\times \{[-1]\},\tilde\cC)}\{\sigma\}
\]
spanned by \v{C}ech nerves $\sigma^0\colon (\Delta^m)^{op}\times\rN(\del_+)^{op}\to \tilde\cC$ such that $\sigma^0\res(\Delta^m)^{op}\times\rN(\del)^{op}$ factorizes through $\cC$, and that $\sigma^0\res\Delta^{\{j\}} \times \rN(\del_+^{\le 0})^{op}$ belongs to $\tilde\cE$ and is representable in $\cC$ for all $0\le j\le m$. Since $\Cov(\sigma)$ admits product of two objects, it is a contractible Kan complex by Lemma \ref{b1le:trivial_contractible}.

Let $\cK\subseteq \Fun(\rN(\del_+),\Fun(\Delta^m,\cD))$ be the full subcategories spanned by augmented cosimplicial objects $X_\bullet^+$ that
are right Kan extensions of $X_\bullet^+\res \rN(\del)$. By \cite{HTT}*{Proposition 4.3.2.15}, the restriction map $\cK\to\Fun(\rN(\del),\Fun(\Delta^m,\cD))$ is a trivial fibration. We have a diagram
\[
\xymatrix{\Cov(\sigma)^{op}\ar[rd]^-{\phi} \ar@/^1.2pc/[rrd]^-\alpha\ar@/_1pc/[rdd]_-{\beta} \\
&\cK'\ar[r]\ar[d] & \Fun(\Delta^n,\Fun(\rN(\del)\times \Delta^m,\cD))\ar[d]\\
& \Fun(\partial\Delta^n,\cK)\ar[r] & \Fun(\partial
\Delta^n,\Fun(\rN(\del)\times \Delta^m,\cD)) }
\]
where the square is Cartesian, $\alpha$ is induced by $F$, and $\beta$ is induced by $G$. Consider the diagram
\[
\xymatrix{ & \cN(\sigma)\ar[r]\ar[d]_-{\rres_1^*\phi} & \Cov(\sigma)^{op}\ar[d]^-{\phi}\\
\Fun(\Delta^n,\Fun(\Delta^m,\cD)) & \Fun(\Delta^n,\cK)\ar[r]^-{\rres_1}\ar[l]_-{\rres_2} & \cK', }
\]
where the square is Cartesian and $\rres_2$ is the restriction to $\{[-1]\}$. Since $\rres_1$ is a trivial fibration, $\cN(\sigma)$ is a
contractible Kan complex.

Put $\Phi(\sigma)\coloneqq\rres_2\circ \rres_1^*\phi$. The construction is functorial in $\sigma$ in the sense that it defines a morphism $\Phi$ of the category $(\Sset)^{(\del_{\tilde\cC^{op}})^{op}}$. Moreover, $\Phi(\sigma)$ takes values in $\Map^\sharp((\Delta^m)^\flat,\Fun(\Delta^n,\cD)^\natural)$. In fact, this is trivial for $n>0$ and the proof of Lemma \ref{b4le:sharp} can be easily adapted to treat the case $n=0$. Applying Corollary \ref{a2co:key} to $\Phi$ and $a=G$, we obtain a lifting $\tilde{F}\colon \Delta^n\to
\Fun(\tilde\cC^{op}, \cD)$ of $F$ extending $G$.

It remains to show that $\tilde{F}$ factorizes through $\Fun^{\tilde\cE}(\tilde\cC^{op},\cD)$. This is trivial for $n>0$. For $n=0$, we need to show that every morphism $f\colon Y \to X$ in $\tilde\cE$ is of $\tilde{F}$-descent, where we regard $\tilde{F}$ as a functor $\cC^{op}\to\cD$. Let $u\colon X'\to X$ be a morphism in $\tilde\cE$ with $X'$ in $\cC$, and $v$ the composite morphism $Y'\xrightarrow{w} Y\times_X X'\to Y$ of the pullback of $u$ and a morphism $w$ in $\cE$ with $Y'$ in $\cC$. This provides a diagram
\[
\xymatrix{
Y'\ar[r]^-{f'}\ar[d]_-{v} & X'\ar[d]^-{u}\\
Y\ar[r]^-{f}& X}
\]
where $u$ and $v$ are in $\tilde\cE$ and $f'$ belongs to $\cE$. Then $f'$ and $u$ are of $\tilde{F}$-descent by construction. It follows that $f$ is of $F$-descent by Lemma \ref{b3le:descent}(3,4).
\end{proof}

Thanks to (P0) and (P4) (see Remark \ref{b4re:input}(1)), we may apply Proposition \ref{b4pr:toy} to
\begin{itemize}
  \item $\tilde\cC=(\tilde\cC^{op}\times\cL^{op})^{\amalg,op}$,

  \item $\cC=(\cC^{op}\times\cL^{op})^{\amalg,op}$,

  \item $\cD=\Cat$,

  \item and the set $\tilde\cE$ consists of edges $f$ that statically belong to $\tcEs\cap \tilde\cE''$,
\end{itemize}
and obtain an extension of the functor $\EO{}{\cC}{\rI}{}$ to a functor
\[
\EO{}{\tilde\cC}{\rI}{}\colon(\tilde\cC^{op}\times\cL^{op})^\amalg\to\Cat
\]
as demanded.

\subsection{Properties}
\label{b4ss:properties}

We construct Output II and prove that Output I and Output II satisfy all required properties.

\begin{lem}[P0]
The functor $\EO{}{\tilde\cC}{\rI}{}$ is a lax Cartesian structure, and the 
induced functor 
$\EO{}{\tilde\cC}\otimes{}\coloneqq(\EO{}{\tilde\cC}{\rI}{})^\otimes$ 
factorizes through $\PSLM$. 
\end{lem}

\begin{proof}
This follows from the construction of $\EO{}{\tilde\cC}{\rI}{}$ as the properties in (P0) are preserved under limits.
\end{proof}

\begin{lem}[P1]
The map $\EO{}{\tilde\cC}\otimes{}$ sends small coproducts to products.
\end{lem}

\begin{proof}
Since $\tilde\cC$ is geometric (Definition \ref{b4de:geometric}), small coproducts commute with pullbacks. Therefore, forming \v{C}ech nerves
commutes with the such coproducts. Then the lemma follows from the construction of $\EO{}{\tilde\cC}\otimes{}$ and the property (P1) for
$\EO{}{\cC}\otimes{}$.
\end{proof}

\begin{lem}[P2]\label{b4le:p2}
The restrictions of $\EO{}{\tilde\cC}{\rI}{}$ and $\EO{}{\tilde\cC'}{\r{II}}{}$ to the subcategory $(\tilde\cC'^{op}\times\cL'^{op})^\amalg$ are equivalent functors.
\end{lem}

\begin{proof}
By Proposition \ref{b4pr:toy} and the original (P2), it suffices to show that the restriction
$F\coloneqq\EO{}{\tilde\cC'}{\r{II}}{}\res(\tilde\cC'^{op}\times\cL'^{op})^\amalg$ belongs to $\Fun^{\tilde\cE}((\tilde\cC'^{op}\times\cL'^{op})^\amalg,\Cat)$ where set $\tilde\cE$ consists of edges $f$ of that statically belong to $\tcEs\cap\tilde\cE''\cap\tilde\cC'_1$. In other words, it suffices to show that $f$ is of $F$-descent.

By construction, the assertions are true if $f$ is statically an atlas. Moreover, by the original (P4), the assertions are also true if $f$ is a morphism of $\cC'$. In the general case, consider a diagram
\[
\xymatrix{Y'\ar[r]^-{f'}\ar[d]_-{v} & X'\ar[d]^-{u}\\
Y\ar[r]^-{f}& X}
\]
where $u$ is an atlas and $f'$ belongs to $\cEs\cap\cE''$. For example, we can take $v$ to be an atlas of $Y\times_XX'$. The proposition then follows from Lemma \ref{b3le:descent}(3,4).
\end{proof}

\begin{lem}[P3]\label{b4le:p3}
If $f\colon Y\to X$ belongs to $\tcEs$, then $f^*\colon\cD(X,\lambda)\to\cD(Y,\lambda)$ is conservative for every object $\lambda$ of $\cL$.
\end{lem}

\begin{proof}
We may put $f$ into the following diagram
\[
\xymatrix{
 Y' \ar[r]^-{f'}  \ar[d]_{v}  & X' \ar[d]^-{u} \\
 Y \ar[r]^-{f} & X }
\]
where $u$ is an atlas, $Y$ belongs to $\cC$ and $f'$ belongs to $\cEs$. Then 
we only need to show that $v^*\circ f^*$, which is equivalent to $f'^*\circ 
u^*$, is conservative. By \cite{HA}*{Theorem 4.7.5.2(3)}, $u^*$ is 
conservative, and $f'^*$ is also conservative by the original (P3). 
Therefore, $f^*$ is conservative. 
\end{proof}

\begin{proposition}[P4]\label{b4pr:p4}
Let $f$ be a morphism of $\tilde\cC^{op}$ (resp.\ $\tilde\cC'$).
\begin{enumerate}
  \item If $f$ belongs to $\tcEs\cap\tilde\cE''$, then $(f,\id_\lambda)$ is of universal $\EO{}{\tilde\cC}\otimes{}$-descent for every object $\lambda$ of $\cL$.

  \item If $f$ belongs to $\tcEs\cap\tilde\cE''\cap\tilde\cC'_1$, then $(f,\id_\lambda)$ is of universal $\EO{}{\tilde\cC'}{}{!}$-codescent  for every object $\lambda$ of $\cL'$.
\end{enumerate}
\end{proposition}

\begin{proof}
Part (1) follows from the construction of $\EO{}{\tilde\cC}{\rI}{}$. Part (2) follows from the same argument as in Lemma \ref{b4le:p2}.
\end{proof}

We will only check (P5), and (P5\bis) follows in the same way.

\begin{proposition}[P5]\label{b4pr:p5}
Let
\[
\xymatrix{
W \ar[r]^-{g} \ar[d]_-{q} & Z \ar[d]^-{p} \\
Y \ar[r]^-{f} & X }
\]
be a Cartesian diagram of $\tilde\cC'$ with $f$ in $\tilde\cE'$, and $\lambda$ an object of $\cL'$. Then
\begin{enumerate}
  \item The square
      \begin{align}\label{b4eq:adjoint1}
      \xymatrix{
      \cD(Z,\lambda)  \ar[d]_-{g^*} & \cD(X,\lambda) \ar[l]_-{p^*} \ar[d]^-{f^*}\\
      \cD(W,\lambda) & \cD(Y,\lambda) \ar[l]_-{q^*}
      }
      \end{align}
      has a right adjoint which is a square of $\PSR$.

  \item If $p$ is also in $\tilde\cE'$, the square
      \begin{align}\label{b4eq:adjoint2}
      \xymatrix{
      \cD(X,\lambda)  \ar[d]_-{p^*} & \cD(Y,\lambda) \ar[l]_-{f_!} \ar[d]^-{q^*}\\
      \cD(Z,\lambda) & \cD(W,\lambda) \ar[l]_-{g_!}
      }
      \end{align}
      is right adjointable.
\end{enumerate}
\end{proposition}

We first prove a technical lemma.

\begin{lem}\label{b4le:adjoint_limit}
Let $K$ be a simplicial set, and $p\colon K\to\Fun(\Delta^1\times\Delta^1,\Cat)$ a diagram of squares of $\infty$-categories. We view $p$ as a functor $K\times\Delta^1\times\Delta^1\to\Cat$. If for every edge $\sigma\colon\Delta^1\to K\times\Delta^1$, the induced square
$p\circ(\sigma\times\id_{\Delta^1})\colon \Delta^1\times\Delta^1\to\Cat$ is right adjointable (resp.\ left adjointable), then the limit square $\lim(p)$ is right adjointable (resp.\ left adjointable).
\end{lem}

Recall from the remark following Proposition \ref{b1pr:partial_adjoint} that when visualizing squares, we adopt the convention that direction $1$ is vertical and direction $2$ is horizontal.

\begin{proof}
Let us prove the right adjointable case, the proof of the other case being essentially the same. The assumption allows us to view $p$ as a functor
\[
p'\colon K\to\Fun(\Delta^1,\Fun^{\r{RAd}}(\Delta^1,\Cat))
\]
\cite{HA}*{Definition 4.7.4.16}. By \cite{HA}*{Corollary 4.7.4.18} and (the 
dual version of) \cite{HTT}*{Corollary 5.1.2.3}, the $\infty$-category 
$\Fun(\Delta^1,\Fun^{\r{RAd}}(\Delta^1,\Cat))$ admits all limits and these 
limits are preserved by the inclusion 
\[
\Fun(\Delta^1,\Fun^{\r{RAd}}(\Delta^1,\Cat))\subseteq\Fun(\Delta^1,\Fun(\Delta^1,\Cat)).
\]
Therefore, the limit square $\lim(p)$ is equivalent to $\lim(p')$ which is right adjointable.
\end{proof}

\begin{proof}[Proof of Proposition \ref{b4pr:p5}]
For (1), it is clear from the construction and the original (P5)(1) that both $f^*$ and $g^*$ admit left adjoints. Therefore, we only need to show that \eqref{b4eq:adjoint1} is right adjointable. By Lemma \ref{b4le:adjoint_limit}, we may assume that $f$ belongs to $\cE'$. Then it reduces to show that the transpose of \eqref{b4eq:adjoint1} is left adjointable, which allows us to assume that $p$ is a morphism of $\cC'$, again by Lemma \ref{b4le:adjoint_limit}. Then it follows from the original (P5)(1).

For (2), by Lemma \ref{b4le:adjoint_limit}, we may assume that $p$ belongs to $\cE'$. Then $p^*$ and $q^*$ admit left adjoints. Therefore, we only need to prove that the transpose of \eqref{b4eq:adjoint2} is left adjointable, which allows us to assume that $f$ is also in $\cE'$, again by Lemma \ref{b4le:adjoint_limit}. Then it follows from the original (P5)(2).
\end{proof}

Next we define the t-structure. Let $X$ be an object of $\tilde\cC$ and let $\lambda$ be an object of $\cL$. For an atlas $f\colon X_0\to X$, we denote by $\cD^{\leq0}_f(X,\lambda)\subseteq\cD(X,\lambda)$ (resp.\ $\cD^{\geq0}_f(X,\lambda)\subseteq\cD(X,\lambda)$) the full subcategory
spanned by complexes $\sfK$ such that $f^*\sfK$ belongs to $\cD^{\leq0}(X_0,\lambda)$ (resp.\ $\cD^{\geq0}(X_0,\lambda)$).

\begin{lem}\label{b4le:t_structure}
We have
\begin{enumerate}
  \item The pair of subcategories $(\cD^{\leq0}_f(X,\lambda),\cD^{\geq0}_f(X,\lambda))$ determine a t-structure on $\cD(X,\lambda)$.

  \item The pair of subcategories $(\cD^{\leq0}_f(X,\lambda),\cD^{\geq0}_f(X,\lambda))$ do not depend on the choice of $f$.
\end{enumerate}
\end{lem}

In what follows, we will write $(\cD^{\leq0}(X,\lambda),\cD^{\geq0}(X,\lambda))$ for $(\cD^{\leq0}_f(X,\lambda),\cD^{\geq0}_f(X,\lambda))$ for an arbitrary atlas $f$. Moreover, if $X$ is an object of $\cC$, then the new t-structure coincides with the old one since $\id_X\colon X\to X$ is an atlas.

\begin{proof}
For (1), let $f_\bullet\colon X_\bullet\to X$ be a \v{C}ech nerve of $f_0=f$. We need to check the axioms of \cite{HA}*{Definition 1.2.1.1}. To check axiom (1), let $\sfK$ be an object of $\cD^{\le 0}_f(X,\lambda)$ and $\sfL$ an object of $\cD^{\ge 1}_f(X,\lambda)$. By (P6) for the input and Proposition \ref{b4pr:p4}(1), $\Map(\sfK,\sfL)$ is a homotopy limit of $\Map(f_n^*\sfK,f_n^*\sfL)$ by \cite{HTT}*{Theorem 4.2.4.1, Corollary A.3.2.28} and is thus a weakly contractible Kan complex. Axiom (2) is trivial. By (P6) for the input, we have a cosimplicial diagram $p\colon
\rN(\del)\to\Fun(\Delta^1,\Cat)$ sending $[n]$ to the functor $\cD(X_n,\lambda)\to\Fun(\Delta^1\times\Delta^1,\cD(X_n,\lambda))$ that
corresponds to the following Cartesian diagram of functors:
\[
\xymatrix{
\tau_n^{\leq0} \ar[r] \ar[d] & \id_{X_n} \ar[d] \\
0 \ar[r] & \tau_n^{\geq1}, }
\]
where $\tau_n^{\leq0}$ and $\tau_n^{\geq1}$ (resp.\ $\id_{X_n}$) are the truncation functors (resp.\ is the identity functor) of $\cD(X_n,\lambda)$. Axiom (3) follows from the fact that $\lim(p)$ provides a similar Cartesian diagram of endofunctors of $\cD(X,\lambda)$.

For (2), by (1) it suffices to show that for every other atlas $f'\colon X'_0\to X$, we have $\cD^{\le 0}_f(X,\lambda)= \cD_{f'}^{\le 0}(X,\lambda)$. Let $\sfK$ be an object of $\cD_f^{\leq0}(X,\lambda)$ and form a Cartesian diagram
\[
\xymatrix{
Y \ar[r]^-{g} \ar[d]_-{g'} & X'_0 \ar[d]^-{f'} \\
X_0 \ar[r]^-{f} & X. }
\]
By (P6) for the input, the functors $g^*$ and $g'^*$ are t-exact, so that
\[
g^* \tau^{\ge 1}f'^*\sfK\simeq \tau^{\ge 1}g^*f'^*\sfK\simeq\tau^{\ge 1} g'^*f^*\sfK\simeq g'^*\tau^{\ge 1}f^*\sfK=0.
\]
As $g^*$ is conservative by (P3) for the input, we have $\tau^{\ge 1}f'^*\sfK=0$. In other words, $f'^*\sfK$ belongs to $\cD^{\leq0}(X'_0,\lambda)$. Therefore, we have $\cD^{\le0}_f(X,\lambda)\subseteq \cD_{f'}^{\le 0}(X,\lambda)$. By symmetry, we have
$\cD^{\le 0}_f(X,\lambda)\supseteq \cD_{f'}^{\le 0}(X,\lambda)$. It follows that $\cD^{\le 0}_f(X,\lambda)=\cD_{f'}^{\le 0}(X,\lambda)$.
\end{proof}

\begin{lem}[P6]\label{b4le:p6}
Let $\lambda$ be an arbitrary object of $\cL$. We have
\begin{enumerate}
  \item For every object $X$ of $\tilde\cC$, we have $\lambda_X\in\cD^\heartsuit(X,\lambda)$.

  \item If $\lambda$ belongs to $\cL''$ and $X$ is an object of $\tilde\cC''$, then the auto-equivalence $-\otimes s_X^*\lambda(1)$ of $\cD(X,\lambda)$ is t-exact.

  \item For every object $X$ of $\tilde\cC$, the t-structure on $\cD(X,\lambda)$ is accessible, right complete, and $\cD^{\le-\infty}(X,\lambda)\colonequals\bigcap_n \cD^{\le -n}(X,\lambda)$ consists of zero objects.

  \item For every morphism $f\colon Y\to X$ of $\tilde\cC$, the functor $f^*\colon \cD(X,\lambda)\to\cD(Y,\lambda)$ is t-exact.
\end{enumerate}
\end{lem}

\begin{proof}
We choose an atlas $f\colon X_0\to X$. Then (1) and (2) follows from (4), the definition of the t-structure, and that $f^*\lambda_X\simeq\lambda_{X_0}$. Moreover, (3) follows from the construction, the conservativeness of $f^*$, and the corresponding properties for $X_0$. Therefore, it remains to show (4).

However, we may put $f\colon Y\to X$ into a diagram
\[
\xymatrix{
Y' \ar[r]^-{f'}  \ar[d]_{v}  & X' \ar[d]^-{u} \\
Y \ar[r]^-{f} & X }
\]
where $u$ and $v$ are both atlases. Then the assertion follows from the definition of the t-structure and the fact that $f'^*$ is t-exact.
\end{proof}

Finally we construct the trace maps. We will construct the trace maps for $\tcEt$ and check (P7). Construction of the trace maps for $\tilde\cE'$ and verification of (P7\bis) are similar and in fact easier.

Same as before, we have two steps. We first construct the trace maps for $\cR\cap\tcEt$.

\begin{lem}\label{b4le:poincare1}
There exists a unique way to define the trace map
\[
\Tr_f\colon \tau^{\ge 0}f_!\lambda_Y\langle d\rangle\to \lambda_X,
\]
for morphisms $f\colon Y \to X$ in $\cR\cap\tcEt\cap\tilde\cC''_1$ and integers $d\geq\dim^+(f)$, satisfying (P7)(1) and extending the input. In particular, for such a morphism $f$, we have $f_!\lambda_Y\langle d\rangle\in\cD^{\leq 0}(X,\lambda)$.
\end{lem}

\begin{proof}
Let
\begin{align}\label{b4eq:atlas}
\xymatrix{
   Y_0 \ar[d]_-{y_0} \ar[r]^-{f_0} & X_0 \ar[d]^-{x_0} \\
   Y \ar[r]^-{f} & X   }
\end{align}
be a Cartesian diagram in $\tilde\cC''$, where $x_0$ and hence $y_0$ are atlases. Let $\rN(\del_+)^{op}\times\Delta^1\to\tilde\cC''$ be a \v{C}ech nerve, as shown in the following diagram
\begin{align}\label{b4eq:poincare}
\xymatrix{
   Y_\bullet \ar[d]_-{y_\bullet} \ar[r]^-{f_\bullet} & X_\bullet \ar[d]^-{x_\bullet} \\
   Y \ar[r]^-{f} & X.   }
\end{align}
We call such a diagram a \emph{simplicial Cartesian atlas} of $f$. We have $\dim^+(f_n)=\dim^+(f)$ for every $n\geq0$. By Base Change which is encoded in $\EO{}{\tilde\cC'}{\r{II}}{}$ and the definition of $-\langle d\rangle$, we have
\[
x_0^*f_!\lambda_Y\langle d\rangle\simeq f_{0!}y_0^*\lambda_Y\langle d\rangle\simeq f_{0!}\lambda_{Y_0}\langle d\rangle\in \cD^{\le0}(X_0,\lambda),
\]
which implies that $f_!\lambda_Y\langle d\rangle$ belongs to $\cD^{\leq0}(X,\lambda)$ by the definition of the t-structure. The uniqueness of the trace map follows from condition (2) of Remark \ref{b4re:trace} applied to the diagram \eqref{b4eq:atlas} and (P3) applied to $x_0$.

For $n\geq0$, we have trace maps $\Tr_{f_n}\colon\tau^{\geq0}f_{n!}\lambda_{Y_n}\langle d\rangle\to\lambda_{X_n}$. By condition (2) of Remark \ref{b4re:trace} applied to the squares induced by $f_\bullet$, we know that $\tau^{\le 0}x_{\bullet *}\Tr_{f_\bullet}$ is a morphism of cosimplicial objects of $\cD^\heartsuit(X,\lambda)$. Taking limit, we obtain a map
\[
\lim_{n\in \del}\tau^{\le 0}x_{n*}\Tr_{f_n}\colon
\lim_{n\in \del}\tau^{\le 0}x_{n*}\tau^{\geq0}f_{n!}\lambda_{Y_n}\langle d\rangle\to\lim_{n\in \del}\tau^{\le 0}x_{n*}\lambda_{X_n}\simeq\lambda_X.
\]
However, the left-hand side is isomorphic to
\begin{align*}
\lim_{n\in \del}\tau^{\le 0}x_{n*}\tau^{\geq0}f_{n!}y_n^*\lambda_Y\langle d\rangle
&\simeq\lim_{n\in \del }\tau^{\le0}x_{n*}\tau^{\geq0}x_n^*f_!\lambda_Y\langle d\rangle\\
&\simeq\lim_{n\in\del}\tau^{\le 0}x_{n*}x_n^*\tau^{\geq0}f_!\lambda_Y\langle d\rangle\simeq\tau^{\geq0}f_!\lambda_Y\langle d\rangle.
\end{align*}
Therefore, we obtain a map $\Tr_{f_\bullet}\colon \tau^{\ge0}f_!\lambda_Y\langle d\rangle\to \lambda_X$.

This extends the trace map of the input. In fact, for $f$ in $\cC''_1$, by condition (2) of Remark \ref{b4re:trace} applied to \eqref{b4eq:poincare}, $\Tr_{f_\bullet}$ can be identified with $\varprojlim_{n\in \del}x_{n*}x_n^*\Tr_f$. Moreover, condition (2) of Remark \ref{b4re:trace} holds in general if one interprets $\Tr_{f}$ as $\Tr_{f_\bullet}$ and $\Tr_{f'}$ as $\Tr_{f'_\bullet}$, where $f'_\bullet$ is a simplicial Cartesian atlas of $f'$, compatible with $f_\bullet$. In fact, by condition (2) of Remark \ref{b4re:trace} for the input, the bottom square of the diagram
\begin{align*}
\resizebox{\hsize}{!}{
\xymatrix{
u^*\tau^{\ge 0}f_!\lambda_Y\langle d \rangle \ar[rr]^-{u^*\Tr_{f_\bullet}}\ar[rd]^-{\simeq}\ar[dd]_\simeq && u^*\lambda_X\ar[rd]^\simeq\ar[dd]|\hole^(.7)\simeq\\
&\tau^{\ge 0}f'_!\lambda_{Y'}\langle d\rangle\ar[rr]^(.3){\Tr_{f'_\bullet}}\ar[dd]^(.3)\simeq &&
\lambda_{X'}\ar[dd]^-\simeq\\
\lim \tau^{\le 0}x'_{n*}u_n^*\tau^{\ge 0}f_{n!} \lambda_{Y_n}\langle d\rangle
\ar[rr]|-\hole^(.4){\lim\tau^{\le 0} x'_{n*}u_n^*\Tr_{f_{n}}}\ar[dr]^\simeq
&& \lim \tau^{\le 0} x'_{n*}u_n^* \lambda_{X_n}\ar[rd]^\simeq\\
&\lim \tau^{\le 0}x'_{n*} \tau^{\ge 0}f'_{n!}\lambda_{Y'_n}\ar[rr]^-{\lim
\tau^{\le 0}x'_{n*}\Tr_{f'_n}} && \lim \tau^{\le 0}x'_{n*}\lambda_{X'_n}
}
}
\end{align*}
is commutative, where all the limits are taken over $n\in\del$. Since the vertical squares are commutative, it follows that the top square is commutative as well. The case of condition (2) of Remark \ref{b4re:trace} where $u$ is an atlas then implies that $\Tr_{f_\bullet}$ does not depend on the choice of $f_\bullet$. We may therefore denote it by $\Tr_f$.

It remains to check conditions (1) and (3) of Remark \ref{b4re:trace}. Similarly to the situation of condition (2), these follow from the input by taking limits.
\end{proof}

\begin{lem}\label{b4le:poincare2}
If $f\colon Y\to X$ belongs to $\cR\cap\tilde\cE''_d\cap\tilde\cC''_1$, then the induced natural transformation
\[
f^*\langle d\rangle=\id_Y\circ f^*\langle d\rangle\to f^!\circ f_!\circ f^*\langle d\rangle\xrightarrow{f^!\circ u_f} f^!
\]
is an equivalence, where the first arrow is given by the unit transformation and $u_f$ is defined similarly as \eqref{b3eq:trace_1}.
\end{lem}

\begin{proof}
Consider diagram \eqref{b4eq:poincare}. We need to show that for every object $\sfK$ of $\cD(X,\lambda)$, the natural map $f^*\sfK\langle d\rangle\to f^!\sfK$ is an equivalence. By Proposition \ref{b4pr:p4}(1), the map $\sfK\to\lim_{n\in\del}u_{n*}u_n^*\sfK$ is an equivalence. Moreover, $f^!$ preserves small limits, and, by (P5\bis)(1), so does  $f^*$, since $f$ belongs to $\tilde\cE''$. Therefore, we may assume $\sfK=x_{n*}\sfL$, where $\sfL\in\cD(X_n,\lambda)$. Similarly to \eqref{b4eq:outputii}, the diagram
\[
\xymatrix{f^*x_{n*}\sfL\langle d\rangle\ar[r]\ar[d] & y_{n*}f_n^*\sfL\langle d\rangle\ar[d]\\
f^!x_{n*}\sfL \ar[r] & y_{n*}f_n^!\sfL}
\]
is commutative up to homotopy. The upper horizontal arrow is an equivalence by (P5\bis)(1), the lower horizontal arrow is an equivalence by $\EO{}{\tilde\cC'}{*}{!}$, and the right vertical arrow is an equivalence by (P6) for the input. It follows that the left vertical arrow is an
equivalence.
\end{proof}

\begin{proposition}[P7(1)]
There exists a unique way to define the trace map
\[
\Tr_f\colon \tau^{\ge 0}f_!\lambda_Y\langle d\rangle\to \lambda_X,
\]
for morphisms $f\colon Y \to X$ in $\tcEt\cap\tilde\cC''_1$ and integers $d\geq\dim^+(f)$, satisfying (P7)(1) and extending the input. In particular, for such a morphism $f$, we have $f_!\lambda_Y\langle d\rangle\in\cD^{\leq0}(X,\lambda)$.
\end{proposition}

\begin{proof}
Let $Y_\bullet\colon \rN(\del_+)^{op}\to\tilde\cC'$ be a \v{C}ech nerve of an atlas $y_0\colon Y_0\to Y$, and form a triangle
\begin{align}\label{b4eq:triangle}
\xymatrix{ & Y\ar[rd]^-{f}\\
Y_\bullet\ar[rr]^-{f_\bullet} \ar[ur]^-{y_\bullet}&&X.}
\end{align}
For $n\ge 0$, we have $f_n\in\cR\cap \tcEt\cap\tilde\cC''_1$. By Proposition \ref{b4pr:p4}(2), we have equivalences
\[
\colim_{n\in\del^{op}}f_{n!}y_n^!\lambda_Y \simeq \colim_{n\in\del^{op}}f_!y_{n!}y_n^!\lambda_Y
\xrightarrow{\sim}f_!\colim_{n\in\del^{op}}y_{n!}y_n^!\lambda_Y\xrightarrow{\sim} f_!\lambda_Y.
\]
Since $y_n$ belongs to $\cR\cap\tilde\cE''\cap\tilde\cC''_1$, by Lemma \ref{b4le:poincare2} and Remark \ref{b4re:input}(5), we have equivalences
\[
\colim_{n\in\del^{op}}f_{n!}\lambda_{Y_n}\langle d+\dim y_n\rangle\simeq \colim_{n\in\del^{op}}f_{n!}y_n^*\lambda_Y\langle d+\dim y_n\rangle
\xrightarrow{\sim}\colim_{n\in\del^{op}}f_{n!}y_n^!\lambda_Y\langle d\rangle.
\]
Combining the above ones, we obtain an equivalence
\[
\colim_{n\in\del^{op}}f_{n!}\lambda_{Y_n}\langle d+\dim y_n\rangle\xrightarrow{\sim} f_!\lambda_Y\langle d\rangle.
\]
By Lemma \ref{b4le:poincare1}, $f_{n!}\lambda_{Y_n}\langle d+\dim y_n\rangle$ belongs to $\cD^{\leq 0}(X,\lambda)$ for every $n\geq 0$. It follows that the colimit is as well by \cite{HA}*{Corollary 1.2.1.6}. Moreover, the composite map
\begin{align*}
\tau^{\geq0}f_{n!}\lambda_{Y_n}\langle d+\dim y_n\rangle
&\to \colim_{n\in\del^{op}}\tau^{\geq0}f_{n!}\lambda_{Y_n}\langle d+\dim y_n\rangle \\
&\xrightarrow{\sim}\tau^{\geq0}\colim_{n\in\del^{op}}f_{n!}\lambda_{Y_n}\langle d+\dim y_n\rangle \xrightarrow{\sim} \tau^{\ge 0}f_!\lambda_Y\langle d\rangle
\end{align*}
is induced by $\Tr_{f_n}$. The uniqueness of $\Tr_f$ then follows from condition (3) of Remark \ref{b4re:trace} applied to the triangle
\eqref{b4eq:triangle}.

Condition (3) of Remark \ref{b4re:trace} applied to the triangles induced by $f_\bullet$ implies the compatibility of
\[
\Tr_{f_n}\colon\tau^{\geq0}f_{n!}\lambda_{Y_n}\langle d+\dim y_n\rangle\to\lambda_X
\]
with the transition maps, so that we obtain a map $\Tr_{f_\bullet}\colon\tau^{\ge 0}f_!\lambda_Y\langle d\rangle\to \lambda_X$. This extends the trace map of Lemma \ref{b4le:poincare1}, by condition (3) of Remark \ref{b4re:trace} applied to \eqref{b4eq:triangle} for
$f\in\cR\cap\tcEt\cap\tilde\cC''_1$. Moreover, condition (3) of Remark \ref{b4re:trace} holds for $g\in\cR\cap\tcEt\cap\tilde\cC''_1$, if we
interpret $\Tr_f$ as $\Tr_{f_\bullet}$ and $\Tr_h$ as $\Tr_{h_\bullet}$, where $h_\bullet\colon Y_\bullet\times_Y Z\to X$. In fact, by condition (3) of Remark \ref{b4re:trace} for morphisms in $\cR\cap\tcEt\cap\tilde\cC''_1$, the diagram
\begin{align*}
\resizebox{\hsize}{!}{
\xymatrix{
\colim \tau^{\ge 0}f_{n!}(\tau^{\ge 0}g_!\lambda_Z\langle e\rangle)\langle d+\dim
 y_n\rangle\ar[rr]^-{\colim \tau^{\ge 0}f_{n!}\Tr_g\langle d+\dim y_n\rangle}\ar[rd]^\simeq\ar[d] &&
\colim \tau^{\ge 0}f_!\lambda_Y\langle d\rangle\ar[rd]^\simeq\\
\colim \tau^{\ge 0}h_{n!}\lambda_Z\langle d+e+\dim y_n\rangle\ar[rd]^\simeq
&\tau^{\ge 0}f_!(\tau^{\ge 0}g_!\lambda_Z\langle e\rangle)\langle d\rangle
\ar[rr]^-{\tau^{\ge 0}f_!\Tr_g\langle d\rangle}\ar[d]_-{\simeq}
&& \tau^{\ge 0}f_!\lambda_Y\langle d\rangle\ar[d]^-{\Tr_{f_\bullet}}\\
&\tau^{\ge 0}h_!\lambda_Z\langle d+e\rangle\ar[rr]^-{\Tr_{h_\bullet}} &&
\lambda_X
}
}
\end{align*}
commutes, where all the colimits are taken over $n\in \del^{op}$. It follows that $\Tr_{f_\bullet}$ does not depend on the choice of $f_\bullet$. We may therefore denote it by $\Tr_f$.

It remains to check the functoriality of the trace map. Similarly to the above special case of condition (2) of Remark \ref{b4re:trace}, this follows from the functoriality of the trace map for morphisms in $\cR\cap\tcEt\cap\tilde\cC''_1$ by taking colimits.
\end{proof}

\begin{proposition}[P7(2)]\label{b4pr:p7}
If $f\colon Y\to X$ belongs to $\tilde\cE''_d\cap\tilde\cC''_1$, the induced natural transformation
\[
f^*\langle d\rangle=\id_Y\circ f^*\langle d\rangle\to f^!\circ f_!\circ f^*\langle d\rangle\xrightarrow{f^!\circ u_f} f^!
\]
is an equivalence, where the first arrow is given by the unit transformation and $u_f$ is defined similarly as \eqref{b3eq:trace_1}.
\end{proposition}

\begin{proof}
We need to show that $f^*\sfK\langle d\rangle\to f^!\sfK$ is an equivalence of every object $\sfK$ of $\cD(X,\lambda)$. Let $y_0\colon Y_0\to Y$ be an atlas. Since $v_0^*$ is conservative by Lemma \ref{b4le:p3}, we only need to show that the composite map
\[
y_0^*\sfK\langle\dim f_0\rangle\xrightarrow{\sim}y_0^*f^*\sfK\langle d+\dim y_0\rangle\to y_0^*f^!\sfK\langle\dim y_0\rangle\xrightarrow{\sim}y_0^!f^!\sfK\xrightarrow{\sim}f_0^!\sfK
\]
is an equivalence, where $f_0\colon Y_0\to X$ is a composite of $f$ and $y_0$. However, this follows from Lemma \ref{b4le:poincare2} applied to $f_0$.
\end{proof}

\section{Running DESCENT}
\label{b5ss}

In this chapter, we run the program DESCENT recursively to construct the theory of six operations of quasi-separated schemes in \Sec\ref{b5ss:schemes}, algebraic spaces in \Sec\ref{b5ss:spaces}, (classical) Artin stacks in \Sec\ref{b5ss:stacks}, and eventually higher Artin stacks in \Sec\ref{b5ss:artin}. Moreover, we start from algebraic spaces to construct the theory for higher Deligne--Mumford (DM) stacks as well in \Sec\ref{b5ss:dm}. We would like to point out that although higher DM stacks are special cases of higher Artin stacks, we have less restrictions on the coefficient rings for the former.

Throughout this chapter, we fix a nonempty set $\Box$ of rational primes. See Remark \ref{b5re:prime} for the relevance on $\Box$.

\subsection{Quasi-separated schemes}
\label{b5ss:schemes}

Recall from Example \ref{b4ex:scheme} that $\Schqs$ is the full subcategory of $\Sch$ spanned by quasi-separated schemes, which contains $\Schqcs$ as a full subcategory. We run the program DESCENT with the input data in Example \ref{b4ex:scheme}. Then the output consists of the following two maps: a functor
\begin{align}\label{b5eq:premonoidal_scheme}
\EO{}{\Schqs}{\rI}{}\colon(\rN(\Schqs)^{op}\times\rN(\Rind)^{op})^\amalg\to\Cat
\end{align}
that is a lax Cartesian structure, and a map
\begin{align}\label{b5eq:operation_scheme}
\EO{}{\Schqs}{\r{II}}{}\colon \delta^*_{2,\{2\}}((\rN(\Schqs)^{op}\times\rN(\Rind_\tor)^{op})^{\amalg,op})^\cart_{F,\all}\to\Cat,
\end{align}
and Output II. Here we recall that $F$ denotes the set of morphisms locally of finite type of quasi-separated schemes.\\

For each object $X$ of $\Schqs$, we denote by $\Et^{\r{qs}}(X)$ the quasi-separated \'{e}tale site of $X$. Its underlying category is the full
subcategory of $\Schqs_{/X}$ spanned by \'{e}tale morphisms. We denote by $X_{\qset}$ the associated topos, namely the category of sheaves on
$\Et^{\r{qs}}(X)$. For every object $X$ of $\Schqcs$, the inclusions $\Et^{\r{qc.sep}}(X)\subseteq \Et^{\r{qs}}(X)\subseteq \Et(X)$ induce
equivalences of topoi $X_{\r{qc.sep}.\et}\to X_{\qset}\to X_{\et}$.

The pseudofunctor $\Schqs\times\Rind\to\RingedPTopos$ sending $(X,(\Xi,\Lambda))$ to $(X_{\qset}^\Xi,\Lambda)$ induces a map
$\rN(\Schqs)\times\rN(\Rind)\to\rN(\RingedPTopos)$. Composing with $\bT$ \eqref{b2eq:enhanced_premonoidal}, we obtain a functor
\begin{align}\label{b5eq:premonoidal_scheme_etale}
\EO{\qset}{\Schqs}{\rI}{}\colon(\rN(\Schqs)^{op}\times\rN(\Rind)^{op})^\amalg\to\Cat
\end{align}
that is a lax Cartesian structure. It is clear that the restriction of 
$\EO{\qset}{\Schqs}{\rI}{}$ to 
$(\rN(\Schqcs)^{op}\times\rN(\Rind)^{op})^\amalg$ is equivalent to 
$\EO{}{\Schqcs}{\rI}{}$. 

\begin{proposition}[Cohomological descent for \'{e}tale topoi]
Let $f$ be an edge of $(\rN(\Schqs)^{op}\times\rN(\Rind)^{op})^\amalg$ that is statically a smooth surjective morphism of quasi-separated schemes. Then $f$ is of universal $\EO{\qset}{\Schqs}{\rI}{}$-descent.
\end{proposition}

\begin{proof}
This follows from the same proof of Proposition \ref{b3pr:descent}(1).
\end{proof}

\begin{proposition}
The two functors $\EO{}{\Schqs}{\rI}{}$ \eqref{b5eq:premonoidal_scheme} and $\EO{\qset}{\Schqs}{\rI}{}$ \eqref{b5eq:premonoidal_scheme_etale} are equivalent.
\end{proposition}

\begin{proof}
This follows from Proposition \ref{b4pr:toy} and the previous proposition.
\end{proof}

\begin{remark}
Let $X$ be object of $\Schqs$, and $\lambda=(\Xi,\Lambda)$ an object of $\Rind$. Then it is easy to see that the usual t-structure on
$\cD(X_{\qset}^\Xi,\Lambda)$ coincides with the one on $\cD(X,\lambda)$ obtained in Output II of the program DESCENT.
\end{remark}

\subsection{Algebraic spaces}
\label{b5ss:spaces}

Let $\Esp$ be the category of algebraic spaces (\Sec\ref{0ss:results_torsion}). It contains $\Schqs$ as a full subcategory. We run the program DESCENT with the following input:
\begin{itemize}
  \item $\tilde\cC=\rN(\Esp)$. It is geometric.

  \item $\cC=\rN(\Schqs)$, and $\sf{s}''\to\sf{s}'$ is the unique morphism $\Spec\dZ[\Box^{-1}]\to\Spec\dZ$. In particular,
      $\cC'=\cC$ and $\tilde\cC'=\tilde\cC$.

  \item $\tcEs$ is the set of \emph{surjective} morphisms of algebraic spaces.

  \item $\tilde\cE'$ is the set of \emph{\'{e}tale} morphisms of algebraic spaces.

  \item $\tilde\cE''$ is the set of \emph{smooth} morphisms of algebraic spaces.

  \item $\tilde\cE''_d$ is the set of \emph{smooth} morphisms of algebraic spaces of pure relative dimension $d$. In particular,
      $\tilde\cE'=\tilde\cE''_0$.

  \item $\tcEt$ is the set of \emph{flat} morphisms \emph{locally of finite presentation} of algebraic spaces.

  \item $\tilde\cF=F$ is the set of morphisms locally of finite type of algebraic spaces.

  \item $\cL=\rN(\Rind)^{op}$, $\cL'=\rN(\Rind_\tor)^{op}$, and $\cL''=\rN(\Rind_{\ltor})^{op}$.

  \item $\dim^+$ is the upper relative dimension (Definition \ref{b4de:dimension}).

  \item Input I and II are the output of \Sec\ref{b5ss:schemes}. In particular, $\EO{}{\cC}{\rI}{}$ is \eqref{b5eq:premonoidal_scheme}, and $\EO{}{\cC'}{\r{II}}{}$ is \eqref{b5eq:operation_scheme}.
\end{itemize}
Then the output consists of the following two maps: a functor
\begin{align}\label{b5eq:premonoidal_space}
\EO{}{\Esp}{\rI}{}\colon(\rN(\Esp)^{op}\times\rN(\Rind)^{op})^\amalg\to\Cat
\end{align}
that is a lax Cartesian structure, and a map
\begin{align}\label{b5eq:operation_space}
\EO{}{\Esp}{\r{II}}{}\colon \delta^*_{2,\{2\}}((\rN(\Esp)^{op}\times\rN(\Rind_\tor)^{op})^{\amalg,op})^\cart_{F,\all}
\to\Cat,
\end{align}
and Output II.\\

For each object $X$ of $\Esp$, we denote by $\Et^{\r{esp}}(X)$ the spatial \'{e}tale site of $X$. Its underlying category is the full subcategory of $\Esp_{/X}$ spanned by \'{e}tale morphisms. We denote by $X_{\espet}$ the associated topos, namely the category of sheaves on $\Et^{\r{esp}}(X)$. For every object $X$ of $\Schqs$, the inclusion of the original \'{e}tale site $\Et^{\r{qs}}(X)$ of $X$ into $\Et^{\r{esp}}(X)$ induces an equivalence of topoi $X_{\espet}\to X_{\qset}$.

As in \Sec\ref{b5ss:schemes}, we obtain a functor
\begin{align}\label{b5eq:premonoidal_space_etale}
\EO{\espet}{\Esp}{\rI}{}\colon(\rN(\Esp)^{op}\times\rN(\Rind)^{op})^\amalg\to\Cat
\end{align}
that is a lax Cartesian structure. It is clear that the restriction 
$\EO{\espet}{\Esp}{\rI}{}\res(\rN(\Schqs)^{op}\times\rN(\Rind)^{op})^\amalg$ 
is equivalent to $\EO{}{\Schqs}{\rI}{}$.

\begin{proposition}[Cohomological descent for \'{e}tale topoi]
Let $f$ be an edge of $(\rN(\Esp)^{op}\times\rN(\Rind)^{op})^\amalg$ that is statically a smooth surjective morphism of algebraic spaces. Then $f$ is of universal $\EO{\espet}{\Esp}{\rI}{}$-descent.
\end{proposition}

\begin{proof}
This follows from the same proof of Proposition \ref{b3pr:descent}(1).
\end{proof}

\begin{proposition}
The two functors $\EO{}{\Esp}{\rI}{}$ \eqref{b5eq:premonoidal_space} and $\EO{\espet}{\Esp}{\rI}{}$ \eqref{b5eq:premonoidal_space_etale} are equivalent.
\end{proposition}

\begin{proof}
This follows from Proposition \ref{b4pr:toy} and the previous proposition.
\end{proof}

\begin{remark}
Let $X$ be object of $\Esp$, and $\lambda=(\Xi,\Lambda)$ an object of $\Rind$. Then it is easy to see that the usual t-structure on
$\cD(X_{\qset}^\Xi,\Lambda)$ coincides with the one on $\cD(X,\lambda)$ obtained in Output II of the program DESCENT.
\end{remark}

\begin{remark}\label{b5re:proper}
In our construction of the map \eqref{b3eq:operation} in \Sec\ref{b3ss:enhanced_operation}, the essential facts we used from algebraic geometry are Nagata's compactification and proper base change. Nagata's compactification has been extended to separated morphisms of finite type between quasi-compact and quasi-separated algebraic spaces \cite{CLO}*{Theorem 1.2.1}. Proper base change for algebraic spaces follows from the case of schemes by cohomological descent and Chow's lemma for algebraic spaces \cite{RG}*{Premi\`{e}re partie, Corollaire 5.7.13} or the existence theorem of a finite cover by a scheme. The latter is a special case of \cite{Rydh}*{Theorem B} and also follows from the Noetherian case \cite{LMB}*{Th\'{e}or\`{e}me 16.6} by Noetherian approximation of algebraic spaces \cite{CLO}*{Theorem 1.2.2}.

Therefore, if we denote by $\Espqcs$ the full subcategory of $\Esp$ spanned by (small) coproducts of quasi-compact and separated algebraic spaces (hence contains $\Schqcs$ as a full subcategory), and repeat the process in \Sec\ref{b3ss:enhanced_operation}, then we obtain a map
\[
\EO{\r{var}}{\Espqcs}{\r{II}}{}\colon\delta^*_{2,\{2\}}((\rN(\Espqcs)^{op}\times\rN(\Rind_\tor)^{op})^{\amalg,op})^\cart_{F,\all}\to\Cat,
\]
whose restriction to $\delta^*_{2,\{2\}}((\rN(\Schqcs)^{op}\times\rN(\Rind_\tor)^{op})^{\amalg,op})^\cart_{F,\all}$ is equivalent to the map $\EO{}{\Schqcs}{\r{II}}{}$.

Moreover, the restriction $\EO{}{\Esp}{\r{II}}{}\res\delta^*_{2,\{2\}}((\rN(\Espqcs)^{op}\times\rN(\Rind_\tor)^{op})^{\amalg,op})^\cart_{F,\all}$ is equivalent to the map $\EO{\r{var}}{\Espqcs}{\r{II}}{}$. In fact, by Remark \ref{b4re:program}(2), it suffices to prove that $\EO{\r{var}}{\Espqcs}{\r{II}}{}$ satisfies (P4). For this, we can repeat the proof of Proposition \ref{b3pr:descent}. The analogue of Remark \ref{b3re:trace_quasifinite} holds for algebraic spaces because the definition of trace maps is local for the \'etale topology on the target.
\end{remark}

\subsection{Artin stacks}
\label{b5ss:stacks}

Let $\Chp$ be the $(2,1)$-category of Artin stacks (\Sec\ref{0ss:results_torsion}). It contains $\Esp$ as a full subcategory. We run the \emph{simplified} DESCENT (see Variant \ref{b4va:program}) with the following input:
\begin{itemize}
  \item $\tilde\cC=\rN(\Chp)$. It is geometric.

  \item $\cC=\rN(\Esp)$, and $\sf{s}''\to\sf{s}'$ is the identity morphism of $\Spec\dZ[\Box^{-1}]$. In particular, $\cC'=\cC''=\rN(\Esp_\Box)$ (resp.\ $\tilde\cC'=\tilde\cC''=\rN(\Chp_\Box)$), where $\Esp_\Box$ (resp.\ $\Chp_\Box$) is the category of $\Box$-coprime algebraic spaces (resp.\ Artin stacks).

  \item $\tcEs$ is the set of \emph{surjective} morphisms of Artin stacks.

  \item $\tilde\cE'=\tilde\cE''$ is the set of \emph{smooth} morphisms of Artin stacks.

  \item $\tilde\cE''_d$ is the set of \emph{smooth} morphisms of Artin stacks of pure relative dimension $d$.

  \item $\tcEt$ is the set of \emph{flat} morphisms \emph{locally of finite presentation} of Artin stacks.

  \item $\tilde\cF=F$ is the set of morphisms \emph{locally of finite type} of Artin stacks.

  \item $\cL=\rN(\Rind)^{op}$, and $\cL'=\cL''=\rN(\Rind_{\ltor})^{op}$.

  \item $\dim^+$ is upper relative dimension, which is defined as a special case in Definition \ref{b5de:dimension} later.

  \item Input I and II are given by the output of \Sec\ref{b5ss:spaces}. In particular, $\EO{}{\cC}{\rI}{}$ is \eqref{b5eq:operation_space}, and $\EO{}{\cC}{\r{II}}{}=\EO{}{\Esp_\Box}{\r{II}}{}$ is defined as the restriction of $\EO{}{\Esp}{\rI}{}$ \eqref{b5eq:premonoidal_space} to
      \[
      \delta^*_{2,\{2\}}((\rN(\Esp_\Box)^{op}\times\rN(\Rind_{\ltor})^{op})^{\amalg,op})^\cart_{F,\all}.
      \]
\end{itemize}
Then the output consists of the following two maps: a functor
\begin{align}\label{b5eq:premonoidal_stack}
\EO{}{\Chp}{\rI}{}\colon(\rN(\Chp)^{op}\times\rN(\Rind)^{op})^\amalg\to\Cat
\end{align}
that is a lax Cartesian structure, and a map
\[
\EO{}{\Chp_\Box}{\r{II}}{}\colon \delta^*_{2,\{2\}}((\rN(\Chp_\Box)^{op}\times\rN(\Rind_{\ltor})^{op})^{\amalg,op})^\cart_{F,\all}\to\Cat,
\]
and Output II.\\

Now we study the values of objects under the above two maps. Let us recall the lisse-\'{e}tale site $\Liset(X)$ of an Artin stack $X$. Its underlying category, the full subcategory (which is in fact an ordinary category) of $\Chp_{/X}$ spanned by smooth morphisms whose sources are algebraic spaces, is equivalent to a $\cU$-small category. In particular, $\Liset(X)$ endowed with the \'etale topology is a $\cU$-site. We denote by $X_{\liset}$ the associated topos. Let $M\subseteq\Ar(\Chp)$ be the set of smooth representable morphisms of Artin stacks. The lisse-\'etale topos has enough points by \cite{LMB}*{Remarque 12.2.2}, and is functorial with respect to $M$, so that we obtain a functor $\Chp_M\times \Rind \to \RingedPTopos$. Composing
with $\bT$ \eqref{b2eq:enhanced_premonoidal}, we obtain a functor
\begin{align}\label{b5eq:premonoidal_stack_lisse}
(\rN(\Chp)_M^{op}\times\rN(\Rind)^{op})^\amalg\to\Cat
\end{align}
that is a lax Cartesian structure.

To simplify the notation, for an algebraic space $U$, we will write $U_{\et}$ instead of $U_{\espet}$ in what follows. Let $\lambda=(\Xi,\Lambda)$ be an object of $\Rind$. We denote by
\[
\cD_\cart(X_{\liset},\lambda)\subseteq\cD(X_{\liset},\lambda)
\]
(Notation \ref{b2no:image}) the full subcategory consisting of complexes whose cohomology sheaves are all Cartesian (\Sec\ref{0ss:results_torsion}), or,
equivalently, complexes $\sfK$ such that for every morphism $f\colon Y'\to Y$ of $\Liset(X)$, the map $f^*(\sfK\res Y_{\et}) \to (\sfK\res Y'_{\et})$ is an equivalence. This full subcategory is \emph{functorial under $\bT$} in the sense that \eqref{b5eq:premonoidal_stack_lisse} restricts to a new functor
\begin{align}\label{b5eq:premonoidal_stack_lisse2}
\EO{\liset}{\Chp}{\rI}{}\colon(\rN(\Chp)_M^{op}\times\rN(\Rind)^{op})^\amalg\to\Cat
\end{align}
that is a lax Cartesian structure, whose value at $(X,\lambda)$ is 
$\cD_\cart(X_{\liset},\lambda)$. It is clear that the restrictions of 
$\EO{\liset}{\Chp}{\rI}{}$ and $\EO{}{\Esp}{\rI}{}$ 
\eqref{b5eq:premonoidal_space} to 
$(\rN(\Esp)_{M'}^{op}\times\rN(\Rind)^{op})^\amalg$ are equivalent, where 
$M'=M\cap \Ar(\Esp)$. In order to compare $\EO{\liset}{\Chp}{\rI}{}$ and 
$\EO{}{\Chp}{\rI}{}$ more generally, we start from the following lemma, which 
is a variant of Proposition \ref{b4pr:toy}. 

\begin{lem}\label{b4le:toy}
Let $(\tilde\cC,\tilde\cE,\tilde\cF)$ be a $2$-marked $\infty$-category such that $\tilde\cC$ admits pullbacks and $\tilde\cE\subseteq\tilde\cF$ are stable under composition and pullback. Let $\cC\subseteq \tilde\cC$ be a full subcategory stable under pullback such that every edge in $\tilde\cF$ is representable in $\cC$ and for every object $X$ of $\tilde\cC$, there exists a morphism $Y\to X$ in $\tilde\cE$ with $Y$ in $\cC$. Let $\cD$ be an $\infty$-category such that $\cD^{op}$ admits geometric realizations. Put $\cE\coloneqq\tilde\cE\cap \cC_1$ and $\cF\coloneqq\tilde\cF\cap \cC_1$. Let $\Fun^\cE(\cC^{op}_\cF,\cD)\subseteq \Fun(\cC^{op}_\cF,\cD)$ (resp.\ $\Fun^{\tilde\cE}(\tilde\cC^{op}_{\tilde\cF},\cD)\subseteq
\Fun(\tilde\cC^{op}_{\tilde\cF},\cD)$) be the full subcategory spanned by functors $F$ such that for every edge $f\colon X_0^+\to X_{-1}^+$ in $\cE$ (resp.\ in $\tilde\cE$), $F\circ(X_\bullet^{s,+})^{op}\colon\rN(\del_{s,+})\to \cD$ is a limit diagram, where $X_\bullet^{s,+}$ is a semisimplicial \v{C}ech nerve of $f$ in $\cC$ (resp.\ $\tilde\cC$) \cite{HTT}*{Notation 6.5.3.6}. Then the restriction map
\[
\Fun^{\tilde\cE}(\tilde\cC^{op}_{\tilde\cF},\cD)\to \Fun^\cE(\cC^{op}_\cF,\cD)
\]
is a trivial fibration.
\end{lem}

\begin{proof}
The proof is similar to Proposition \ref{b4pr:toy}, whose details we leave to the reader.
\end{proof}

For an object $V\to X$ of $\Liset(X)$, we denote by $\widetilde{V}$ the sheaf in $X_{\liset}$ represented by $V$. The overcategory
$(X_{\liset})_{/\widetilde V}$ is equivalent to the topos defined by the site $\Liset(X)_{/V}$ endowed with the \'etale topology \cite{SGA4}*{Expos\'{e} iii, Proposition 5.4}. A morphism $f\colon U\to U'$ of $\Liset(X)_{/V}$ induces a 2-commutative diagram
\[
\xymatrix{
&(X_{\liset})_{/\widetilde{U}} \ar@/_1pc/[ld]_{u_*}\ar[d]_-{f_*} \ar[r]^-{\epsilon_{U*}} & U_{\et} \ar[d]^-{f_{\et*}}\\
(X_{\liset})_{/\widetilde{V}} &(X_{\liset})_{/\widetilde{U'}} \ar[l]_{u'_*} \ar[r]^-{\epsilon_{U'*}} & U'_{\et} }
\]
of topoi \cite{SGA4}*{Expos\'{e} iv, {\Sec}5.5}.

For an object $\lambda=(\Xi,\Lambda)$ of $\Rind$, let $\cD_\cart((X_{\liset})_{/\widetilde{V}},\lambda)^\otimes\subseteq\cD((X_{\liset})_{/\widetilde{V}},\lambda)^\otimes$ be the full (monoidal)
subcategory spanned by complexes on which the natural transformation $f^*\circ \epsilon_{U'*}\circ u'^*\to\epsilon_{U*}\circ u^*$ is an isomorphism for all $f$. We have a functor
\[
[1]\times\Liset(X)\times\Rind\to\RingedPTopos
\]
sending $[1]\times\{f\colon U\to V\}\times\{\lambda\}$ to the square
\[
\xymatrix{
((X_{\liset})_{/\widetilde{U}}^\Xi,\Lambda) \ar[d]_-{f_*} \ar[r]^-{\epsilon_{U*}} & (U_{\et}^\Xi,\Lambda) \ar[d]^-{f_{\et*}}\\
((X_{\liset})_{/\widetilde{V}}^\Xi,\Lambda)  \ar[r]^-{\epsilon_{V*}} &
(V_{\et}^\Xi,\Lambda). }
\]
Composing with the functor $\bT^\otimes$ \eqref{b2eq:enhanced_monoidal}, we obtain a functor
\[
F\colon(\Delta^1)^{op}\times\rN(\Liset(X))^{op}\times\rN(\Rind)^{op}\to\PSLM.
\]
By construction, $F(0,V,\lambda)=\cD((X_{\liset})_{/\widetilde{V}},\lambda)^\otimes$. Replacing $F(0,V,\lambda)$ by the full subcategory
$\cD_\cart((X_{\liset})_{/\widetilde{V}},\lambda)^\otimes$, we obtain a new functor
\[
F'\colon(\Delta^1)^{op}\times\rN(\Liset(X))^{op}\times\rN(\Rind)^{op}\to\CAlg(\Cat)
\]
sending $(\Delta^1)^{op}\times\{f\colon U\to V\}\times\{\lambda\}$ to the square
\[
\xymatrix{
\cD_\cart((X_{\liset})_{/\widetilde{U}},\lambda)^\otimes & \cD(U_{\et},\lambda)^\otimes \ar[l]_-{\epsilon_U^*} \\
\cD_\cart((X_{\liset})_{/\widetilde{V}},\lambda)^\otimes \ar[u]^-{f^*}
& \cD(V_{\et},\lambda)^\otimes. \ar[u]_-{f_{\et}^*} \ar[l]_-{\epsilon_V^*}
}\]
We have the following two lemmas.

\begin{lem}\label{b4le:Dcart}
The functor $F'$, viewed as an edge of
\[
\Fun(\rN(\Liset(X))^{op}\times\rN(\Rind)^{op},\CAlg(\Cat)),
\]
is an equivalence. In particular, the functor $F'$ factorizes through $\PSLM$.
\end{lem}

\begin{proof}
We only need to prove that for every object $V$ of $\Liset(X)$, the functor
\[
\epsilon_V^*\colon \cD(V_{\et},\lambda)\to\cD_\cart((X_{\liset})_{/\widetilde{V}},\lambda)
\]
is an equivalence. This follows from the fact that
\[ \epsilon_V^*\colon
\Mod(V_{\et},\lambda)\to\Mod_\cart((X_{\liset})_{/\widetilde{V}},\lambda)
\]
is an equivalence of categories and that the functor
\[\epsilon_{V*}\colon
\Mod((X_{\liset})_{/\widetilde{V}},\lambda)\to (V_{\et},\lambda)
\]
is exact, by the following lemma.
\end{proof}

\begin{lem}
Let $F\colon \cA\to\cB$ be an exact fully faithful functor between Grothendieck Abelian categories that admit an exact right adjoint $G$. Then
$F$ induces an equivalence of $\infty$-categories $\cD(\cA)\to\cD_{\cA}(\cB)$, where $\cD_{\cA}(\cB)$ denotes the full subcategory of $\cD(\cB)$ spanned by complexes with cohomology in the essential image of $\epsilon$.
\end{lem}

\begin{proof}
This is standard. The pair $(F,G)$ induce a pair of t-exact adjoint between $\cD(\cA)$ and $\cD_{\cA}(\cB)$. To check that the unit and counit
are natural equivalences, we may reduce to objects in the Abelian categories, for which the assertion follows from the assumptions.
\end{proof}

\begin{lem}\label{b5le:cartesian}
Let $v\colon V\to X$ be an object of $\Liset(X)$, viewed as a morphism of $\Chp$. Assume that $v$ is surjective. Then a complex $\sfK\in
\cD(X_{\liset},\lambda)$ belongs to $\cD_\cart(X_{\liset},\lambda)$ if and only if $v^*\sfK$ belongs to $\cD_\cart((X_{\liset})_{/\widetilde{V}},\lambda)$.
\end{lem}

\begin{proof}
The necessity is trivial. Assume that $v^*\sfK$ belongs to $\cD_\cart((X_{\liset})_{/\widetilde{V}},\lambda)$. We need to show that for every morphism $f\colon Y'\to Y$ of $\Liset(X)$, the map $f^*(\sfK\res Y_{\et}) \to (\sfK\res Y'_{\et})$ is an equivalence. The problem is local for the \'etale topology on $Y$. However, locally for the \'etale topology on $Y$, the morphism $Y\to X$ factorizes through $v$ \cite{EGAIV}*{Corollaire 17.16.3 (ii)}. The assertions thus follows from the assumption.
\end{proof}

Now let $V_\bullet\colon \rN(\del_+)^{op}\to\rN(\Chp)$ be a \v{C}ech nerve of $v$ where $v\colon V\to X$ be an object of $\Liset(X)$, which can be viewed as a simplicial object of $\Liset(X)$. By Lemma \ref{b5le:cartesian}, we can apply Lemma \ref{b3le:topos_descent} to $U_\bullet=\widetilde{V_\bullet}^\Xi$ and $\cC_\bullet=\Mod_\cart((X_{\liset})_{\widetilde{V_\bullet}}^\Xi,\Lambda)$. We obtain a natural equivalence of symmetric monoidal $\infty$-categories
\begin{align}\label{b5eq:descent}
\cD_\cart(X_{\liset},\lambda)^\otimes\xrightarrow{\sim}\lim_{n\in\del}\cD_\cart((X_{\liset})_{/\widetilde{V_n}},\lambda)^\otimes.
\end{align}

\begin{proposition}[Cohomological descent for lisse-\'{e}tale topoi]\label{b5pr:descent_lisse}
Let $X$ be an Artin stack, $V$ an algebraic space, and $v\colon V\to X$ a surjective smooth morphism. Then there is an equivalence in
$\Fun(\rN(\Rind)^{op},\PSLM)$ sending $\lambda$ to the equivalence
\[
\cD_\cart(X_{\liset},\lambda)^\otimes\xrightarrow{\sim}\lim_{n\in\del}\cD(V_{n,\et},\lambda)^\otimes,
\]
where $V_\bullet$ is a \v{C}ech nerve of $v$.
\end{proposition}

\begin{proof}
This follows from \eqref{b5eq:descent} and a quasi-inverse of the equivalence in Lemma \ref{b4le:Dcart}.
\end{proof}

The previous proposition has the following four corollaries.

\begin{corollary}
Let $f\colon Y\to X$ be a smooth surjective representable morphism of Artin stacks, $\lambda$ an object of $\Rind$, and $Y_\bullet$ a \v{C}ech nerve of $f$. Then the functor
\[
\cD_\cart(X_{\liset},\lambda)^\otimes\xrightarrow{\sim}\lim_{n\in\del_s}\cD_\cart(Y_{n,\liset},\lambda)^\otimes
\]
is an equivalence.
\end{corollary}

\begin{corollary}
The functor $\EO{\liset}{\Chp}{\rI}{}$ \eqref{b5eq:premonoidal_stack_lisse2} belongs to $\Fun^{\tilde\cE}(\tilde\cC^{op}_{\tilde\cF},\Cat)$ with the notation in Lemma \ref{b4le:toy}, where
\begin{itemize}
  \item $\tilde\cC=(\rN(\Chp)^{op}\times\rN(\Rind)^{op})^{\amalg,op}$;

  \item $\tilde\cF$ consists of edges of that statically belong to $M$; and

  \item $\tilde\cE\subseteq\tilde\cF$ consists of edges that are also statically surjective.
\end{itemize}
\end{corollary}

\begin{corollary}\label{b5co:compatibility}
The functor $\EO{\liset}{\Chp}{\rI}{}$ \eqref{b5eq:premonoidal_stack_lisse2} is equivalent to the restriction of the functor $\EO{}{\Chp}{\rI}{}$ \eqref{b5eq:premonoidal_stack} to $(\rN(\Chp)_M^{op}\times\rN(\Rind)^{op})^\amalg$. In particular, for every Artin stack $X$ and every object $\lambda$ of $\Rind$, we have an equivalence
\[
\cD_\cart(X_{\liset},\lambda)^\otimes\simeq\cD(X,\lambda)^\otimes
\]
of symmetric monoidal $\infty$-categories. Consequently, $\cD_\cart(X_{\liset},\lambda)^\otimes$ is a closed presentable stable symmetric monoidal $\infty$-category. Here we recall that $\cD(X,\lambda)^\otimes$ is the value of $(X,\lambda,\langle1\rangle,\{1\})$ under the functor $\EO{}{\Chp}{\rI}{}$.
\end{corollary}

\begin{corollary}
Let $X$ be an Artin stack, and $\lambda$ an object of $\Rind$. Under the equivalence in Corollary \ref{b5co:compatibility}, the usual t-structure on $\cD_\cart(X_{\liset},\lambda)$ coincides with the t-structure on $\cD(X,\lambda)$ obtained in Output II. In particular, the heart of $\cD(X,\lambda)$ is equivalent to (the nerve of) $\Mod_\cart(X_{\liset}^\Xi,\Lambda)$, the Abelian category of Cartesian
$(X^\Xi_{\liset},\Lambda)$-modules.
\end{corollary}

\begin{remark}[de~Jong]\label{b5re:dejong}
The $*$-pullback encoded by $\EO{}{\Chp}{\rI}{}$ can be described more directly using big \'etale topoi of Artin stacks. For any Artin stack $X$, we consider the full subcategories $\Esp_{\r{lfp}/X}\subseteq\Chp_{\r{rep.lfp}/X}$ of $\Chp_{/X}$ spanned by morphisms locally of finite
presentation whose sources are algebraic spaces and by representable morphisms locally of finite presentation,\footnote{We impose the ``locally of finite presentation'' condition here to avoid set-theoretic issues.} respectively. They are ordinary categories and we endow them with the \'etale topology. The corresponding topoi are equivalent, and we denote them by $X_{\r{big.\acute et}}$. The construction of $X_{\r{big.\acute et}}$ is functorial in $X$, so that we obtain a functor $\Chp\times \Rind\to\RingedPTopos$. Composing with $\bT$, we obtain a functor
\[
(\rN(\Chp)^{op}\times \rN(\Rind)^{op})^\amalg\to\Cat
\]
that is a lax Cartesian structure, sending $(X,\lambda)$ to 
$\cD(X_{\r{big.\acute et}},\lambda)$. Replacing the latter by the full 
subcategory $\cD_\cart(X_{\r{big.\acute et}},\lambda)$ consisting of 
complexes $\sfK$ such that $f^*(\sfK\res Y'_{\et}) \to (\sfK\res Y_{\et})$ is 
an equivalence for every morphism $f\colon Y\to Y'$ of $\Esp_{/X}$, we obtain 
a new functor 
\[
\EO{\r{big}}{\Chp}{\rI}{}\colon (\rN(\Chp)^{op}\times \rN(\Rind)^{op})^\amalg\to\Cat
\]
that is a lax Cartesian structure. Using similar arguments as in this 
section, with Lemma \ref{b4le:toy} replaced by Proposition \ref{b4pr:toy}, 
one shows that $\EO{\r{big}}{\Chp}{\rI}{}$ and $\EO{}{\Chp}{\rI}{}$ are 
equivalent. 
\end{remark}

\subsection{Higher Artin stacks}
\label{b5ss:artin}

We begin by recalling the definition of higher Artin stacks. We will use the fppf topology instead of the \'etale topology adopted in \cite{Toen}. The two definitions are equivalent \cite{Toen2}.  Let $\Schaff\subseteq\Sch$ be the full subcategory spanned by affine schemes. Recall that $\cS_\cW$ is the $\infty$-category of spaces in $\cW\in\{\cU,\cV\}$.\footnote{We refer to \Sec\ref{0ss:conventions} for conventions on set-theoretical issues.}

\begin{definition}[Prestack and stack]\label{b5de:prestack}
We defined the $\infty$-category of \emph{($\cV$-)prestacks} to be $\Chppre\coloneqq\Fun(\rN(\Schaff)^{op},\cS_\cV)$. We endow $\rN(\Schaff)$
with the fppf topology. We define the $\infty$-category of (small) stacks $\Chpfppf$ to be the essential image of the following inclusion
\[
\r{Shv}(\rN(\Schaff)_{\r{fppf}})\cap\Fun(\rN(\Schaff)^{op},\cS_\cU)\subseteq\Chppre,
\]
where $\r{Shv}(\rN(\Schaff)_{\r{fppf}})\subseteq\Fun(\rN(\Schaff)^{op},\cS_\cV)$ is the full subcategory spanned by fppf sheaves \cite{HTT}*{Definition 6.2.2.6}. A prestack $F$ is \emph{$k$-truncated} \cite{HTT}*{Definition 5.5.6.1} for an integer $k\geq-1$, if $\pi_i(F(A))=0$ for every object $A$ of $\Schaff$ and every integer $i>k$.
\end{definition}

The Yoneda embedding $\rN(\Schaff)\to\Chppre$ extends to a fully faithful functor $\rN(\Esp)\to \Chppre$ sending $X$ to the discrete Kan complex $\Hom_{\Esp}(\Spec A,X)$. The image of this functor is contained in $\Chpfppf$. We will generally not distinguish between $\rN(\Esp)$ and its
essential image in $\Chpfppf$. A stack $X$ belongs to (the essential image of) $\rN(\Esp)$ if and only if it satisfies the following conditions.
\begin{itemize}
  \item It is $0$-truncated.

  \item The diagonal morphism $X\to X\times X$ is schematic, that is, for every morphism $Z\to X\times X$ with $Z$ a scheme, the fiber product $X\times_{X\times X} Z$ is a scheme.

  \item There exists a scheme $Y$ and an (automatically schematic) morphism $f\colon Y\to X$ that is smooth (resp.\ \'etale) and surjective.
      In other words, for every morphism $Z\to X$ with $Z$ a scheme, the induced morphism $Y\times_XZ\to Z$ is smooth (resp.\ \'etale) and
      surjective. The morphism $f$ is called an \emph{atlas} (resp.\ \emph{\'etale atlas}) for $X$.
\end{itemize}

\begin{definition}[Higher Artin stack; see \cite{Toen} and \cite{Gai3}]\label{b5de:artin_stack}
We define $k$-Artin stacks inductively for $k\ge 0$.
\begin{itemize}
  \item A stack $X$ is a \emph{$0$-Artin stack} if it belongs to (the essential image of) $\rN(\Esp)$.
\end{itemize}

For $k\ge 0$, assume that we have defined $k$-Artin stacks. We define:
\begin{itemize}
  \item A morphism $F'\to F$ of prestacks is \emph{$k$-Artin} if for every morphism $Z\to F$ where $Z$ is a $k$-Artin stack, the fiber product $F'\times_F Z$ is a $k$-Artin stack.

  \item A $k$-Artin morphism $F'\to F$ is flat (resp.\ locally of finite type, resp.\ locally of finite presentation, resp.\ smooth, resp.\
      surjective) if for every morphism $Z\to F$ and every atlas $f\colon Y\to F'\times_F Z$ where $Y$ and $Z$ are schemes, the composite
      morphism $Y\to F'\times_FZ\to Z$ is a flat (resp.\ locally of finite type, resp.\ locally of finite presentation, resp.\ smooth, resp.\ surjective) morphism of schemes.

  \item A stack $X$ is a $(k+1)$-Artin stack if the diagonal morphism $X\to X\times X$ is $k$-Artin, and there exists a scheme $Y$
      together with an (automatically $k$-Artin) morphism $f\colon Y\to X$ that is smooth and surjective. The morphism $f$ is called an
      \emph{atlas} for $X$.
\end{itemize}
We denote by $\Chpar{k}\subseteq\Chpfppf$ the full subcategory spanned by $k$-Artin stacks. We define \emph{higher Artin stacks} to be objects of $\Chpar{}\coloneqq\bigcup_{k\geq0}\Chpar{k}$. A morphism $F'\to F$ of prestacks is \emph{higher Artin} if for every morphism $Z\to F$ where $Z$ is a higher Artin stack, the fiber product $F'\times_F Z$ is a higher Artin stack.

To simplify the notation, we put $\Chpar{(-1)}\coloneqq\rN(\Schqs)$ and $\Chpar{(-2)}\coloneqq\rN(\Schqcs)$, and we call their objects $(-1)$-Artin stacks and $(-2)$-Artin stacks, respectively.
\end{definition}

By definition, $\Chpar{0}$ and $\Chpar{1}$ are equivalent to $\rN(\Esp)$ and $\rN(\Chp)$, respectively. For $k\ge 0$, $k$-Artin stacks are $k$-truncated prestacks. Higher Artin stacks are \emph{hypercomplete} sheaves \cite{HTT}*{Lemma 6.5.2.9}. Every flat surjective morphism locally of finite presentation of higher Artin stacks is an \emph{effective epimorphism} in the $\infty$-topos $\r{Shv}(\rN(\Schaff)_{\r{fppf}})$ in the sense after \cite{HTT}*{Corollary 6.2.3.5}. A higher Artin morphism of prestacks is $k$-Artin for some $k\ge 0$.

\begin{definition}\label{b5de:finite_presentation}
We have the following notion of quasi-compactness.
\begin{itemize}
  \item A higher Artin stack $X$ is \emph{quasi-compact} if there exists an atlas $f\colon Y\to X$ such that $Y$ is a quasi-compact scheme.

  \item A higher Artin morphism $F'\to F$ of prestacks is \emph{quasi-compact} if for every morphism $Z\to F$ where $Z$ is a quasi-compact scheme, the fiber product $F'\times_F Z$ is a quasi-compact higher Artin stack.
\end{itemize}
We define quasi-separated higher Artin morphisms of prestacks by induction as follows.
\begin{itemize}
  \item A $0$-Artin morphism of prestacks $F'\to F$ is \emph{quasi-separated} if the diagonal morphism $F'\to F'\times_F F'$, which is automatically schematic, is quasi-compact.

  \item For $k\ge 0$, a $(k+1)$-Artin morphism of prestacks $F'\to F$ is \emph{quasi-separated} if the diagonal morphism $F'\to F'\times_F
      F'$, which is automatically $k$-Artin, is quasi-separated and quasi-compact.
\end{itemize}
We say that a morphism of higher Artin stacks is \emph{of finite
presentation} if it is quasi-compact, quasi-separated, and locally of finite
presentation.

We say that a higher Artin stack $X$ is $\Box$-coprime if there exists a morphism $X\to \Spec \dZ[\Box^{-1}]$. This is equivalent to the existence of a $\Box$-coprime atlas. We denote by $\Chpar{}_\Box\subseteq \Chpar{}$ the full subcategory spanned by $\Box$-coprime higher Artin stacks. We put $\Chpar{k}_\Box\coloneqq\Chpar{k}\cap \Chpar{}_\Box$.
\end{definition}

\begin{definition}[Relative dimension]\label{b5de:dimension}
We define by induction the class of smooth morphisms \emph{of pure relative dimension $d$} of $k$-Artin stacks for $d\in\dZ\cup\{-\infty\}$ and the \emph{upper relative dimension} $\dim^+(f)$ for every morphism $f$ locally of finite type of $k$-Artin stacks. If in Input 0 of \Sec\ref{b4ss:description}, we let $\tilde\cF$ (resp.\ $\tilde\cE''$, $\tilde\cE''_d$) be the set of morphisms locally of finite type (resp.\ smooth morphisms, smooth morphisms of pure relative dimension $d$) of $k$-Artin stacks, then such definitions should satisfy Input 0(5--8).

When $k=0$, we use the usual definitions for algebraic spaces, with the upper relative dimension given in Definition \ref{b4de:dimension}. For $k\ge 0$, assuming that these notions are defined for $k$-Artin stacks. We first extend these definitions to $k$-representable morphisms locally of finite type of $(k+1)$-Artin stacks. Let $f\colon Y\to X$ be such a morphism, and $X_0\xrightarrow{u}X$ an atlas of $X$. Let $f_0\colon Y_0\to X_0$ be the base change of $f$ by $u$. Then $f_0$ is a morphism locally of finite type of $k$-Artin stacks. We define $\dim^+(f)=\dim^+(f_0)$. It is easy to see that this is independent of the atlas we choose by Input 0(8d). We say that $f$ is smooth of pure relative dimension $d$ if $f_0$ is -- this is independent of the atlas we choose by Input 0(6). We need to check Input 0(5--8). Input 0(6--8) are easy, and (5) can be argued as follows. Since $f_0$ is a smooth morphism of $k$-Artin stacks,
there is a decomposition $f_0\colon Y_0\simeq\coprod_{d\in\dZ}Y_{0,d}\xrightarrow{(f_{0,d})}X_0$. Let $X_\bullet\to X$ be a \v{C}ech nerve of $u$, and put $Y_{\bullet, d}\coloneqq Y_{0,d}\times_{X_0} X_\bullet$. Then $\coprod_{d\in\dZ}Y_{\bullet,d}\to Y$ is a \v{C}ech nerve of $v\colon Y_0\to Y$. Put $Y_d\coloneqq\colim_{n\in\del^{op}}Y_{n,d}$. Then $Y\simeq\coprod_{d\in\dZ}Y_d$ is the desired decomposition.

Next we extend these definitions to all  morphisms locally of finite type of $(k+1)$-Artin stacks. Let $f\colon Y\to X$ be such a morphism, and $v_0\colon Y_0=\coprod_{d\in\dZ}Y_{0,d}\xrightarrow{(v_{0,d})}Y$ an atlas of $Y$ such that $v_{0,d}$ is smooth of pure relative dimension $d$. We define
\[
\dim^+(f)=\sup_{d\in\dZ}\{\dim^+(f\circ v_{0,d})-d\}.
\]
We say that $f$ is smooth of pure relative dimension $d$ if for every $e\in\dZ$, the morphism $f\circ v_{0,e}$ is smooth of pure relative dimension $d+e$. We leave it to the reader to check that these definitions are independent of the atlas we choose, and satisfy Input 0(6--8). We sketch the proof for Input 0(5). Since $f\circ v_{0,e}$ is smooth and $k$-representable, it can be decomposed as $Y_{0,e}\simeq\coprod_{e'\in\dZ}Y_{0,e,e'}\xrightarrow{(f_{e,e'})} X$ such that $f_{e,e'}$ is of pure relative dimension $e'$. We let $Y_d$ be the colimit of the underlying groupoid object of the \v{C}ech nerve of $\coprod_{e'-e=d}Y_{0,e,e'}\to X$. Then $Y\simeq\coprod_{d\in\dZ}Y_d\to X$ is the desired decomposition.
\end{definition}

Let $F$ be the set of morphisms locally of finite type of higher Artin stacks. For every $k$, we are going to construct a functor
\[
\EO{}{\Chpar{k}}{\rI}{}\colon((\Chpar{k})^{op}\times\rN(\Rind)^{op})^\amalg\to\Cat
\]
that is a lax Cartesian structure, and a map
\[
\EO{}{\Chpar{k}_\Box}{\r{II}}{}\colon \delta^*_{2,\{2\}}(((\Chpar{k}_\Box)^{op}\times\rN(\Rind_{\ltor})^{op})^{\amalg,op})^\cart_{F,\all}\to\Cat,
\]
such that their restrictions to $(k-1)$-Artin stacks coincide with those for the latter.

We construct by induction. When $k=-2,-1,0,1$, they have been constructed in \Sec\ref{b3ss:enhanced_operation}, \Sec\ref{b5ss:schemes},
\Sec\ref{b5ss:spaces}, and \Sec\ref{b5ss:stacks}, respectively. Assume that they have been extended to $k$-Artin stacks. We run the version of DESCENT in Variant \ref{b4va:program} with the following input:
\begin{itemize}
  \item $\tilde\cC=\Chpar{(k+1)}$. It is geometric.

  \item $\cC=\Chpar{k}$, $\sf{s}''\to\sf{s}'$ is the identity morphism of $\Spec\dZ[\Box^{-1}]$. In particular, $\cC'=\cC''=\Chpar{k}_\Box$, and $\tilde\cC'=\tilde\cC''=\Chpar{(k+1)}_\Box$.

  \item $\tcEs$ is the set of \emph{surjective} morphisms of $(k+1)$-Artin stacks.

  \item $\tilde\cE'=\tilde\cE''$ is the set of {\em smooth} morphisms of $(k+1)$-Artin stacks.

  \item $\tilde\cE''_d$ is the set of \emph{smooth} morphisms of $(k+1)$-Artin stacks of pure relative dimension $d$.

  \item $\tcEt$ is the set of \emph{flat} morphisms \emph{locally of finite presentation} of $(k+1)$-Artin stacks.

  \item $\tilde\cF=F$ is the set of morphisms \emph{locally of finite type} of $(k+1)$-Artin stacks.

  \item $\cL=\rN(\Rind)^{op}$, and $\cL'=\cL''=\rN(\Rind_{\ltor})^{op}$.

  \item $\dim^+$ is the upper relative dimension in Definition \ref{b5de:dimension}.

  \item Input I and II is given by induction hypothesis. In particular, we take
      \[
      \EO{}{\cC}{\rI}{}=\EO{}{\Chpar{k}}{\rI}{},\quad\EO{}{\cC'}{\r{II}}{}=\EO{}{\Chpar{k}_\Box}{\r{II}}{}.
      \]
\end{itemize}
Then the output consists of desired two maps $\EO{}{\Chpar{k+1}}{\rI}{}$, $\EO{}{\Chpar{k+1}_\Box}{\r{II}}{}$ and Output II, satisfying (P0) -- (P7\bis). Taking union of all $k\geq0$, we obtain the following two maps: a functor
\begin{align}\label{b5eq:premonoidal_artin}
\EO{}{\Chpar{}}{\rI}{}\colon((\Chpar{})^{op}\times\rN(\Rind)^{op})^\amalg\to\Cat
\end{align}
that is a lax Cartesian structure, and a map
\begin{align}\label{b5eq:operation_artin}
\EO{}{\Chpar{}_\Box}{\r{II}}{}\colon \delta^*_{2,\{2\}}(((\Chpar{}_\Box)^{op}\times\rN(\Rind_{\ltor})^{op})^{\amalg,op})^\cart_{F,\all}\to\Cat.
\end{align}

\subsection{Higher Deligne--Mumford stacks}
\label{b5ss:dm}

The definition of higher Deligne--Mumford (DM) stacks is similar to that of higher Artin stacks (Definition \ref{b5de:artin_stack}).

\begin{definition}[Higher DM stack]\leavevmode
\begin{itemize}
  \item A stack $X$ is a \emph{$0$-DM stack} if it belongs to (the essential image of) $\rN(\Esp)$.
\end{itemize}
For $k\ge 0$, assume that we have defined $k$-DM stacks. We define:
\begin{itemize}
  \item A morphism $F'\to F$ of prestacks is \emph{$k$-DM} if for every morphism $Z\to F$ where $Z$ is a $k$-DM stack, the fiber product
      $F'\times_F Z$ is a $k$-DM stack.

  \item A $k$-DM morphism $F'\to F$ of prestacks is \emph{\'etale} (resp.\ \emph{locally quasi-finite}) if for every morphism $Z\to F$ and every \'etale atlas $f\colon Y\to F'\times_F Z$ where $Y$ and $Z$ are schemes, the composite morphism $Y\to F'\times_FZ\to Z$ is an \'etale (resp.\ locally quasi-finite) morphism of schemes.

  \item A stack $X$ is a $(k+1)$-DM stack if the diagonal morphism $X\to X\times X$ is $k$-DM, and there exists a scheme $Y$ together with an (automatically $k$-DM) morphism $f\colon Y\to X$ that is \'etale and surjective. The morphism $f$ is called an \emph{\'etale atlas} for $X$.
\end{itemize}
We denote by $\Chpdm{k}\subseteq\Chpfppf$ the full subcategory spanned by $k$-DM stacks. We define \emph{higher DM stacks} to be objects of
$\Chpdm{}\coloneqq\bigcup_{k\geq0}\Chpdm{k}$. We put $\Chpdm{}_\Box\coloneqq\Chpdm{}\cap \Chpar{}_\Box$, and $\Chpdm{k}_\Box\coloneqq\Chpdm{k}\cap\Chpdm{}_\Box$.

A morphism of higher DM stacks is \'{e}tale if and only if it is smooth of pure relative dimension $0$.
\end{definition}

Let $F$ be the set of morphisms locally of finite type of higher DM stacks. For every $k$, we are going to construct a functor
\[
\EO{}{\Chpdm{k}}{\rI}{}\colon((\Chpdm{k})^{op}\times\rN(\Rind)^{op})^\amalg\to\Cat
\]
that is a lax Cartesian structure, and a map
\[
\EO{}{\Chpdm{k}}{\r{II}}{}\colon \delta^*_{2,\{2\}}(((\Chpdm{k})^{op}\times\rN(\Rind_\tor)^{op})^{\amalg,op})^\cart_{F,\all}\to\Cat,
\]
such that their restrictions to $(k-1)$-DM stacks coincide with those for the latter. Note that the first functor has already been constructed in \Sec\ref{b5ss:artin}, after restriction. However for induction, we construct it again, which in fact coincides with the previous one.

We construct by induction. When $k=0$, they have been constructed in \Sec\ref{b5ss:spaces}. Assuming that they have been extended to $k$-DM stacks. We run the program DESCENT with the following input:
\begin{itemize}
  \item $\tilde\cC=\Chpdm{(k+1)}$. It is geometric.

  \item $\cC=\Chpdm{k}$, $\sf{s}''\to\sf{s}'$ is the morphism $\Spec\dZ[\Box^{-1}]\to\Spec\dZ$.

  \item $\tcEs$ is the set of \emph{surjective} morphisms of $(k+1)$-DM stacks.

  \item $\tilde\cE'$ is the set of \emph{\'{e}tale} morphisms of $(k+1)$-DM stacks.

  \item $\tilde\cE''$ is the set of \emph{smooth} morphisms of $(k+1)$-DM stacks.

  \item $\tilde\cE''_d$ is the set of \emph{smooth} morphisms of $(k+1)$-DM stacks of pure relative dimension $d$.

  \item $\tcEt$ is the set of \emph{flat} morphisms \emph{locally of finite presentation} of $(k+1)$-DM stacks.

  \item $\tilde\cF=F$ is the set of morphisms \emph{locally of finite type} of $(k+1)$-DM stacks.

  \item $\cL=\rN(\Rind)^{op}$, $\cL'=\rN(\Rind_\tor)^{op}$, and $\cL''=\rN(\Rind_{\ltor})^{op}$.

  \item $\dim^+$ is the upper relative dimension.

  \item Input I and II is given by induction hypothesis. In particular, we take
      \[
      \EO{}{\cC}{\rI}{}=\EO{}{\Chpdm{k}}{\rI}{},\quad\EO{}{\cC'}{\r{II}}{}=\EO{}{\Chpdm{k}}{\r{II}}{}.
      \]
\end{itemize}
Then the output consists of desired two maps $\EO{}{\Chpdm{k+1}}{\rI}{}$, $\EO{}{\Chpdm{k+1}}{\r{II}}{}$ and Output II, satisfying (P0) -- (P7\bis). Taking union of all $k\geq0$, we obtain a functor
\begin{align}\label{b5eq:premonoidal_dm}
\EO{}{\Chpdm{}}{\rI}{}\colon((\Chpdm{})^{op}\times\rN(\Rind)^{op})^\amalg\to\Cat
\end{align}
that is a lax Cartesian structure, and a map
\begin{align}\label{b5eq:operation_dm}
\EO{}{\Chpdm{}}{\r{II}}{}\colon \delta^*_{2,\{2\}}(((\Chpdm{})^{op}\times\rN(\Rind_\tor)^{op})^{\amalg,op})^\cart_{F,\all}\to\Cat.
\end{align}

\begin{remark}
We have the following compatibility properties:
\begin{itemize}
  \item The restriction of $\EO{}{\Chpar{}}{\rI}{}$ to $((\Chpdm{})^{op}\times\rN(\Rind)^{op})^\amalg$ is equivalent to $\EO{}{\Chpdm{}}\otimes{}$.

  \item The restrictions of $\EO{}{\Chpdm{}}{\r{II}}{}$ and $\EO{}{\Chpar{}_\Box{}}{\r{II}}{}$ to the common domain
      \[
      \delta^*_{2,\{2\}}(((\Chpdm{}_\Box)^{op}\times\rN(\Rind_{\ltor})^{op})^{\amalg,op})^\cart_{F,\all}
      \]
      are equivalent.
\end{itemize}
\end{remark}

\begin{variant}
We denote by $Q\subseteq F$ the set of locally quasi-finite morphisms. Applying DESCENT to the map $\EO{\r{lqf}}{\Schqcs}{\r{II}}{}$ constructed in
Variant \ref{b3va:quasifinite} (and $\EO{}{\Schqcs}{\rI}{}$), we obtain a map
\begin{align}\label{b5eq:operation_dm_quasifinite}
\EO{\r{lqf}}{\Chpdm{}}{\r{II}}{}\colon\delta^*_{2,\{2\}}(((\Chpdm{})^{op}\times\rN(\Rind)^{op})^{\amalg,op})^\cart_{Q,\all}\to\Cat.
\end{align}
This map and $\EO{}{\Chpdm{}}{\r{II}}{}$ are equivalent when restricted to their common domain.
\end{variant}

\begin{remark}
The $\infty$-category $\Chpdm{}$ can be identified with a full subcategory of the $\infty$-category $\mathrm{Sch}(\cG_{\et}(\dZ))$ of
$\mathcal{G}_{\et}(\dZ)$-schemes in the sense of \cite{DAG5}*{Definition 2.3.9, Remark 2.6.11}. The constructions of this section can be extended to $\mathrm{Sch}(\cG_{\et}(\dZ))$ by hyperdescent. We will provide more details in Remark \ref{c4re:etale}.
\end{remark}

\begin{remark}\label{b5re:prime}
Note that in this chapter, we have fixed a non-empty set $\Box$ of rational primes. In fact, our constructions are compatible for different $\Box$ in the obvious sense. For example, if we are given $\Box_1\subseteq\Box_2$, then the maps $\EO{}{\Chp^{\r{Ar}}_{\Box_1}}{\r{II}}{}$ and
$\EO{}{\Chp^{\r{Ar}}_{\Box_2}}{\r{II}}{}$ are equivalent when restricted to their common domain, which is
\[
\delta^*_{2,\{2\}}(((\Chp^{\r{Ar}}_{\Box_2})^{op}\times\rN(\Rind_{\Box_1\text{-}\tor})^{op})^{\amalg,op})^\cart_{F,\all}.
\]
We also have obvious compatibility properties for Output II under different $\Box$.
\end{remark}

\section{Summary and complements for torsion coefficients}
\label{b6ss}

In this chapter we summarize the construction in the previous chapter and presents several complements. In \Sec\ref{b6ss:correspondences}, we study the relation of our construction with category of correspondences. In \Sec\ref{b6ss:operations}, we write down the resulting six operations for the most general situations and summarize their properties. In \Sec\ref{b6ss:adjointness}, we prove some additional adjointness properties in the finite-dimensional Noetherian case. In \Sec\ref{b6ss:constructible}, we develop a theory of constructible complexes, based on finiteness results of Deligne \cite{SGA4d}*{Th.\ finitude} and Gabber \cite{TG}*{Expos\'{e} XIIIp}. In \Sec\ref{b6ss:compatibility}, we show that our results for constructible complexes are compatible with those of Laszlo--Olsson \cite{LO1}.

We remark that \Sec\ref{b6ss:correspondences} is independent to the later sections, so readers may skip the first section is they are not interested in the relation with category of correspondences.

Once again, we fix a nonempty set $\Box$ of rational primes.

\subsection{Symmetric monoidal category of correspondences}
\label{b6ss:correspondences}

The $\infty$-category of correspondences was introduced by Gaitsgory 
\cite{Gai1}. We start by recalling the construction of the simplicial set of 
correspondences from Example \ref{a4ex:corr}. 

For $n\geq0$, we define $\sfC(\Delta^n)$ to be the full subcategory of $\Delta^n\times(\Delta^n)^{op}$ spanned by $(i,j)$ with $i\leq j$. An edge of $\sfC(\Delta^n)$ is \emph{vertical} (resp.\ \emph{horizontal}) if its projection to the second (resp.\ first) factor is degenerate. A square of $\sfC(\Delta^n)$ is \emph{exact} if it is both a pushout square and a pullback square. We extend the above construction to a colimit preserving functor $\sfC\colon\Sset\to\Sset$. Then $\sfC$ also preserves finite products. The right adjoint functor is denoted by $\Corr$. In particular, we have $\Corr(K)_n=\Hom(\sfC(\Delta^n),K)$ for a simplicial set $K$.

\begin{definition}
Let $(\cC,\cE_1,\cE_2)$ be a $2$-marked $\infty$-category. We define a simplicial subset $\cC_{\corr\colon\cE_1,\cE_2}$ of $\Corr(\cC)$, called the \emph{simplicial set of correspondences}, such that its $n$-cells are given by maps $\sfC(\Delta^n)\to\cC$ that send vertical (resp.\ horizontal) edges into $\cE_1$ (resp.\ $\cE_2$), and exact squares to pullback squares.
\end{definition}

By construction, there is an obvious map
\[
\delta_{2,\{2\}}^*\cC^\cart_{\cE_1,\cE_2}\to\cC_{\corr\colon\cE_1,\cE_2},
\]
which is a categorical equivalence by Example \ref{a4ex:corr}.

The following lemma shows that under certain mild conditions, $\cC_{\corr\colon\cE_1,\cE_2}$ is an $\infty$-category.

\begin{lem}\label{b6le:correspondence}
Let $(\cC,\cE_1,\cE_2)$ be a $2$-marked $\infty$-category such that
\begin{enumerate}
  \item both $\cE_1$ and $\cE_2$ are stable under composition;

  \item pullbacks of $\cE_1$ by $\cE_2$ exist and remain in $\cE_1$;

  \item pullbacks of $\cE_2$ by $\cE_1$ exist and remain in $\cE_2$.
\end{enumerate}
Then $\cC_{\corr\colon\cE_1,\cE_2}$ is an $\infty$-category.
\end{lem}

\begin{proof}
We check that $\cC_{\corr\colon\cE_1,\cE_2}\to *$ has the right lifting property with respect to the collection $A_2$ in \cite{HTT}*{Proposition 2.3.2.1}. Since $\sfC$ preserves colimits and finite products, to give a map
\[
f\colon(\Delta^m\times\Lambda^2_1)\coprod_{\partial\Delta^m\times\Lambda^2_1}(\partial\Delta^m\times\Delta^2)\to\Corr(\cC)
\]
is equivalent to give a map
\[
f^\sharp\colon\(\sfC(\Delta^m)\times\sfC(\Lambda^2_1)\)\coprod_{\sfC(\partial\Delta^m)
\times\sfC(\Lambda^2_1)}\(\sfC(\partial\Delta^m)\times\sfC(\Delta^2)\)\to\cC.
\]
Let $\cK$ and $\cK'$ be defined as in the dual version of \cite{HTT}*{Proposition 4.3.2.15} with $\cC=\sfC(\Delta^2)$, $\cC^0=\sfC(\Lambda^2_1)$, and $\cD=\cC$ (in our setup). If $f$ factorizes through $\cC_{\corr\colon\cE_1,\cE_2}$, then $f^\sharp$ induces a commutative square
\[
\xymatrix{
\sfC(\partial\Delta^m)  \ar@{^(->}[d]\ar[r] & \cK \ar[d] \\
\sfC(\Delta^m) \ar[r] \ar@{-->}[ru] & \cK' }
\]
by assumption (2) or (3). Since the restriction map $\cK\to\cK'$ is a trivial fibration by the dual of \cite{HTT}*{Proposition 4.3.2.15}, there exists a dotted arrow $g^\sharp\colon\sfC(\Delta^m)\to\cK$ as indicated above. We regard $g^\sharp$ as a map $\sfC(\Delta^m\times\Delta^2)\simeq\sfC(\Delta^m)\times\sfC(\Delta^2)\to\cC$, thus induces a map $g\colon\Delta^m\times\Delta^2\to\Corr(\cC)$. Since all exact squares of $\sfC(\Delta^m\times\Delta^2)$ can be obtained by composition from exact squares either contained in the source of $f^\sharp$ or being constant under the projection to $\sfC(\Delta^m)$, the three assumptions ensure that if $f$ factorizes through
$\cC_{\corr\colon\cE_1,\cE_2}$, then so does $g$.
\end{proof}

Now we study a certain natural symmetric monoidal structure on the $\infty$-category $\cC_{\corr\colon\cE_1,\cE_2}$. Let $(\cC,\cE)$ be a marked $\infty$-category. We construct a $2$-marked $\infty$-categories $((\cC^{op})^{\amalg,op},\cE^-,\cE^+)$ as follows:
We write an edge $f$ of $(\cC^{op})^{\amalg,op}$ in the form $\{Y_j\}_{1\leq j\leq n}\to\{X_i\}_{1\leq i\leq m}$ lying over an edge $\alpha\colon\langle m\rangle\to\langle n\rangle$ of $\rN(\Fin)$. Then $\cE^+$ consists of $f$ such that the induced edge $Y_{\alpha(i)}\to X_i$ belongs to $\cE$ for every $i\in\alpha^{-1}\langle n\rangle^\circ$. Define $\cE^-$ to be the subset of $\cE^+$ such that the edge $\alpha$ is degenerate.

\begin{proposition}\label{b6pr:correspondence}
Let $(\cC,\cE_1,\cE_2)$ be a $2$-marked $\infty$-category satisfying the assumptions in Lemma \ref{b6le:correspondence} and such that $\cC_{\cE_2}$ admits finite products. Then
\begin{align}\label{b6eq:correspondence}
p\colon((\cC^{op})^{\amalg,op})_{\corr\colon\cE_1^-,\cE_2^+}\to\rN(\Fin)
\end{align}
is a symmetric monoidal $\infty$-category, whose underlying $\infty$-category is $\cC_{\corr\colon\cE_1,\cE_2}$.
\end{proposition}

\begin{proof}
Put $\cO^\otimes\coloneqq((\cC^{op})^{\amalg,op})_{\corr\colon\cE_1^-,\cE_2^+}$ for simplicity. If $(\cC,\cE_1,\cE_2)$ satisfies the assumptions in Lemma \ref{b6le:correspondence}, then so does $((\cC^{op})^{\amalg,op},\cE_1^-,\cE_2^+)$. Therefore, by Lemma \ref{b6le:correspondence}, $\cO^\otimes$ is an $\infty$-category hence \eqref{b6eq:correspondence} is an inner fibration by \cite{HTT}*{Proposition 2.3.1.5}. By Lemma \ref{b6le:cocartesian} below, we know that $p$ is a coCartesian fibration since $\cC_{\cE_2}$ admits finite products. Moreover, we have the obvious isomorphism $\cO^\otimes_{\langle n\rangle}\simeq\prod_{1\leq i\leq n}\cO^\otimes_{\langle 1\rangle}$ induced by
$\rho^i_!\colon\cO^\otimes_{\langle n\rangle}\to\cO^\otimes_{\langle 1\rangle}$. By \cite{HA}*{Definition 2.0.0.7}, \eqref{b6eq:correspondence} is a symmetric monoidal $\infty$-category.
\end{proof}

\begin{lem}\label{b6le:cocartesian}
Suppose that $(\cC,\cE_1,\cE_2)$ satisfies the assumptions in Lemma \ref{b6le:correspondence}. If we write an edge $f$ of
$((\cC^{op})^{\amalg,op})_{\corr\colon\cE_1^-,\cE_2^+}$ in the form
\[
\xymatrix{
\{Z_j\}_{1\leq j \leq n} \ar[r]\ar[d]  &  \{X_i\}_{1\leq i\leq m} \\
\{Y_j\}_{1\leq j\leq n} }
\]
lying over an edge $\alpha\colon\langle m\rangle\to\langle n\rangle$ of $\rN(\Fin)$ under \eqref{b6eq:correspondence}, then $f$ is $p$-coCartesian \cite{HTT}*{Definition 2.4.2.1} if and only if
\begin{enumerate}
  \item for every $1\leq j\leq n$, the induced morphism $Z_j\to Y_j$ is an isomorphism; and

  \item for every $1\leq j\leq n$, the induced morphisms $Z_j\to X_i$ with $\alpha(i)=j$ exhibit $Z_j$ as the product of $\{X_i\}_{\alpha(i)=j}$ in $\cC_{\cE_2}$.
\end{enumerate}
\end{lem}

\begin{proof}
The \emph{only if} part: Suppose that $f$ is a $p$-coCartesian edge.

We first show (1). Without lost of generality, we may assume that $\alpha$ is the degenerate edge at $\langle 1\rangle$. In particular, the edge $f$ we consider has the form
\[
\xymatrix{
z  \ar[r]\ar[d]  &  x. \\
y }
\]
Assume that $f$ is $p$-coCartesian. In terms of the dual version of \cite{HTT}*{Remark 2.4.1.4}, we are going to construct a diagram of the form
\begin{align}\label{b6eq:cocartesian}
\xymatrix{
\Delta^{\{0,1\}} \ar@{^(->}[d]\ar[rd]^-{f} \\
\Lambda^n_0  \ar@{^(->}[d]\ar[r]^-{g}  &  ((\cC^{op})^{\amalg,op})_{\corr\colon\cE_1^-,\cE_2^+} \ar[d]^-{p} \\
\Delta^n  \ar[r]\ar@{-->}[ur]  & \rN(\Fin) }
\end{align}
in which $n=3$ and the bottom map is constant with value $\langle1\rangle$. We may construct a map $g$ in \eqref{b6eq:cocartesian} such that its image of $\sfC(\Delta^{\{0,1,2\}})$, $\sfC(\Delta^{\{0,1,3\}})$, $\sfC(\Delta^{\{0,2,3\}})$ are
\[
\xymatrix{
z \ar@{=}@/_1pc/[dd]\ar[r]\ar[d] & z \ar[r]\ar[d] & x, \\
y' \ar[r]\ar[d] & y \\
z \\
}\quad
\xymatrix{
z \ar@{=}[r]\ar[d] & z \ar[r]\ar[d] & x, \\
y \ar@{=}[r]\ar@{=}[d] & y \\
y \\
}\quad
\xymatrix{
z \ar@{=}[r]\ar@{=}[d] & z \ar[r]\ar@{=}[d] & x, \\
z \ar@{=}[r]\ar[d] & z \\
y \\
}
\]
respectively, in which
\begin{itemize}
  \item all squares are Cartesian diagrams;

  \item all edges $z\to x$ are same as the one in the presentation of $f$;

  \item all vertical edges $z\to y$ are same as the one in the presentation of $f$;

  \item in the second and third diagrams, all $2$-cells are degenerate.
\end{itemize}
Note that the existence of the first diagram is due to the lifting property for $n=2$. Now we lift $g$ to a dotted arrow as in \eqref{b6eq:cocartesian}. The image of the unique nondegenerate exact square in $\sfC(\Delta^{\{1,2,3\}})$ provides a pullback square
\[
\xymatrix{
y \ar[r]\ar[d] & y' \ar[d] \\
z \ar@{=}[r] & z. }
\]
Therefore, the edge $y\to y'$ is an isomorphism, and it is easy to check that the left vertical edge $y\to z$ is an inverse of the edge $z\to y$ in the presentation of $f$.

Next we show (2). Without lost of generality, we may assume that $\alpha$ is the unique active map from $\langle m\rangle$ to $\langle 1\rangle$ \cite{HA}*{Definition 2.1.2.1}; and the edge $f$ has the form
\[
\xymatrix{
y  \ar[r]\ar@{=}[d]  &  \{x_i\}_{1\leq i\leq m}. \\
y }
\]
We construct a diagram \eqref{b6eq:cocartesian} as follows. The bottom map $\Delta^n\to\rN(\Fin)$ is given by the sequence of morphisms
\[
\langle m\rangle \xrightarrow{\alpha} \langle 1\rangle\xrightarrow{\id}\cdots\xrightarrow{\id}\langle 1\rangle.
\]
Note that we have a projection map $\pi\colon\sfC(\Delta^n)\to(\Delta^n)^{op}$ to the second factor. Denote by $\sfC(\Delta^n)_0$ the preimage of $(\Delta^{\{1,\dots,n\}})^{op}$ under $\pi$, and $\sfC(\Delta^n)_{00}$ the preimage of $(\partial\Delta^{\{1,\dots,n\}})^{op}$ under $\pi$. It is clear that $\sfC(\Lambda^n_0)\cap\sfC(\Delta^n)_0\subseteq\sfC(\Delta^n)_{00}$. Suppose that we are given a map
\[
\alpha\colon(\partial\Delta^{\{1,\dots,n\}})^{op}\to(\cC_{\cE_2})_{/\{x_i\}_{1\leq i\leq m}}
\]
such that $\alpha\res\Delta^{\{0\}}$ is represented by $y\to\{x_i\}_{1\leq i\leq m}$ as in the edge $f$. We regard $\alpha$ as a map
$\alpha'\colon(\partial\Delta^{\{1,\dots,n\}})^{op}\star\langle m\rangle^\circ\to\cC_{\cE_2}$. Note that $\pi$ induces a projection map
\[
\pi'\colon(\sfC(\Lambda^n_0)\cap\sfC(\Delta^n)_0)\star\langle m\rangle^\circ\to(\partial\Delta^{\{1,\dots,n\}})^{op}\star\langle m\rangle^\circ.
\]
We then have a map $g_\alpha\coloneqq\alpha'\circ\pi'\colon(\sfC(\Lambda^n_0)\cap\sfC(\Delta^n)_0)\star\langle m\rangle^\circ\to\cC_{\cE_2},$ which induces a map $g$ as in \eqref{b6eq:cocartesian}. The existence of the dotted arrow in \eqref{b6eq:cocartesian} will provide a filling of $\alpha$ to $(\Delta^{\{1,\dots,n\}})^{op}$. This implies that $y\to\{x_i\}_{1\leq i\leq m}$ is a final object of $(\cC_{\cE_2})_{/\{x_i\}_{1\leq i\leq m}}$.

The \emph{if} part: Let $f$ be an edge satisfying (1) and (2). To show that $f$ is $p$-coCartesian, we again consider the diagram
\eqref{b6eq:cocartesian}. Define $\sfC(\Delta^n)'$ to be the $\infty$-category by adding one more object $(0,0)'$ emitting from $(0,0)$ in $\sfC(\Delta^n)$, which can be depicted as in the following diagram
\[
\xymatrix{
\cdots & (0,2) \ar[r]\ar[d] & (0,1) \ar[r]\ar[d] & (0,0) \ar[r] & (0,0)'. \\
\cdots & (1,2) \ar[r]\ar[d] & (1,1) \\
\cdots & (2,2) \\
\cdots }
\]
We have maps $\sfC(\Delta^n)\xrightarrow{\iota}\sfC(\Delta^n)'\xrightarrow{\gamma}\sfC(\Delta^n)$, in which $\iota$ is the obvious inclusion, and $\gamma$ collapse the edge $(0,1)\to (0,0)$ to the single object $(0,1)$ and sends $(0,0)'$ to $(0,0)$. Let $K\subseteq\sfC(\Delta^n)$ be the simplicial subset that is the union of $\sfC(\Lambda^n_0)$ and the top row of $\sfC(\Delta^n)$. Define $K'$ to be the inverse image of $K$ under $\gamma$. Then $\iota$ sends $\sfC(\Lambda^n_0)$ into $K'$. We have one more inclusion $\iota'\colon\sfC(\Delta^n)\to\sfC(\Delta^n)'$ that sends $(0,0)$ to $(0,0)'$ and keeps the other objects.

A map $g$ as in \eqref{b6eq:cocartesian} gives rise to a map $g^\sharp\colon\sfC(\Lambda^n_0)\to(\cC^{op})^{\amalg,op}$. By (2) and \cite{HA}*{Remark 2.4.3.4}, we may extend $g^\sharp$ to $K$. Consider the new map $g^\sharp\circ\gamma\circ\iota\colon\sfC(\Lambda^n_0)\to(\cC^{op})^{\amalg,op}$, which gives rise to a map $g'$ as in \eqref{b6eq:cocartesian} however with the restriction $g'\res\Delta^{\{0,1\}}$ being an equivalence in the $\infty$-category $((\cC^{op})^{\amalg,op})_{\corr\colon\cE_1^-,\cE_2^+}$ by (1). Therefore, we may lift $g'$ to an edge $\tilde{g}'$ as the dotted arrow
in \eqref{b6eq:cocartesian} by \cite{HTT}*{Proposition 2.4.1.5}. Now $\tilde{g}'$ induces a map $\tilde{g}'^\sharp\colon\sfC(\Delta^n)\to(\cC^{op})^{\amalg,op}$. To find a lifting of $g$ as the dotted arrow in \eqref{b6eq:cocartesian}, it
suffices to extend $\tilde{g}'^\sharp$ to $\sfC(\Delta^n)'$ under the inclusion $\iota$ such that its restriction to $\sfC(\Lambda^n_0)$ with
respect to the other inclusion $\iota'$ coincides with $g^\sharp$. However, this lifting problem only involves the top row of $\sfC(\Delta^n)'$, which can be solved because of (2).
\end{proof}

\begin{definition}[symmetric monoidal $\infty$-category of correspondences]\label{b6de:correspondence}
Given a $2$-marked $\infty$-category $(\cC,\cE_1,\cE_2)$ satisfying the assumptions in Proposition \ref{b6pr:correspondence}(3), we call
\eqref{b6eq:correspondence} the \emph{symmetric monoidal $\infty$-category of correspondences} associated to $(\cC,\cE_1,\cE_2)$, denoted by
$p\colon\cC^\otimes_{\corr\colon\cE_1,\cE_2}\to\rN(\Fin)$ or simply $\cC^\otimes_{\corr\colon\cE_1,\cE_2}$. It is a reasonable abuse of notation since its underlying $\infty$-category is $\cC_{\corr\colon\cE_1,\cE_2}$.
\end{definition}

We apply the above construction to the source of the map $\EO{}{\Chpar{}_\Box}{\r{II}}{}$ \eqref{b5eq:operation_artin}. Take
$\cC=\Chpar{}_\Box\times\rN(\Rind_{\ltor})$, $\cE_1\coloneqq\cE_F$ to be the set of edges of the form $(f,g)$ where $f$ belongs to $F$ and $g$ is an isomorphism, and $\cE_2\coloneqq\all$ to be the set of all edges. Note that $(\cC,\cE_1,\cE_2)$ satisfies the assumptions in Proposition
\ref{b6pr:correspondence}(2) hence defines a symmetric monoidal $\infty$-category $\cC^\otimes_{\corr\colon\cE_F,\all}$.

By definition, we have the identity
\[
\delta^*_{2,\{2\}}(((\Chpar{}_\Box)^{op}\times\rN(\Rind_{\ltor})^{op})^{\amalg,op})^\cart_{F,\all}=
\delta^*_{2,\{2\}}((\cC^{op})^{\amalg,op})^\cart_{\cE_1^-,\cE_2^+}.
\]
Since the map
\[
\delta^*_{2,\{2\}}((\cC^{op})^{\amalg,op})^\cart_{\cE_1^-,\cE_2^+}\to
((\cC^{op})^{\amalg,op})_{\corr\colon\cE_1^-,\cE_2^+}=\cC^\otimes_{\corr\colon\cE_F,\all}
\]
is a categorical equivalence, by Proposition \ref{b6pr:correspondence}(1), the map \eqref{b5eq:operation_artin} induces a map
\begin{align}\label{b6eq:operadic}
\cC^\otimes_{\corr\colon\cE_F,\all}\to\Cat.
\end{align}

\begin{lem}
The functor \eqref{b6eq:operadic} is a lax Cartesian structure.
\end{lem}

\begin{proof}
It follows from the fact that \eqref{b5eq:premonoidal_artin} is a lax 
Cartesian structure, the construction of \eqref{b5eq:operation_artin}, and 
Lemma \ref{b6le:cocartesian}. 
\end{proof}

From the above lemma, we know that \eqref{b6eq:operadic} induces an $\infty$-operad map
\begin{align}\label{b6eq:correspondence_artin}
\EO{}{\Chpar{}_\Box}{}{\corr}\colon(\Chpar{}_\Box\times\rN(\Rind_{\ltor}))^\otimes_{\corr\colon\cE_F,\all}\to\Cat^\times
\end{align}
between symmetric monoidal $\infty$-categories. Similarly, we have two more $\infty$-operad maps
\begin{align}\label{b6eq:correspondence_dm}
\EO{}{\Chpdm{}}{}{\corr}\colon(\Chpdm{}\times\rN(\Rind_\tor))^\otimes_{\corr\colon\cE_F,\all}\to\Cat^\times,
\end{align}
and
\begin{align}\label{b6eq:correspondence_dm_quasifinite}
\EO{\r{lqf}}{\Chpdm{}}{}{\corr}\colon(\Chpdm{}\times\rN(\Rind))^\otimes_{\corr\colon\cE_Q,\all}\to\Cat^\times,
\end{align}
induced from \eqref{b5eq:operation_dm} and \eqref{b5eq:operation_dm_quasifinite}, respectively.

\begin{remark}
By all the constructions and (P2) of DESCENT, we obtain the following square
\[
\xymatrix{
((\Chpar{}_\Box\times\rN(\Rind_{\ltor}))^{op})^\amalg \ar@{^(->}[r]\ar@{^(->}[d]
& ((\Chpar{}\times\rN(\Rind))^{op})^\amalg \ar[d]  \\
(\Chpar{}_\Box\times\rN(\Rind_{\ltor}))^\otimes_{\corr\colon\cE_F,\all}
\ar[r]^-{\EO{}{\Chpar{}_\Box}{}{\corr}\eqref{b6eq:correspondence_artin}} & \Cat^\times }
\]
in the $\infty$-category of symmetric monoidal $\infty$-categories with $\infty$-operad maps, where the right vertical map is induced from $\EO{}{\Chpar{}}{\rI}{}$ \eqref{b5eq:premonoidal_artin}.

The new functor $\EO{}{\Chpar{}_\Box}{}{\corr}$ loses no information from the original one $\EO{}{\Chpar{}_\Box}{\r{II}}{}$. However, the new one has the advantage that its source is an $\infty$-category as well.

The above remarks can be applied to the other two cases as well.
\end{remark}

\subsection{The six operations}
\label{b6ss:operations}

Now we can summarize our construction of Grothendieck's six operations. Let $f\colon \cY\to\cX$ be a morphism of $\Chpar{}$ (resp.\ $\Chpdm{}$, resp.\ $\Chpdm{}$), and $\lambda$ an object of $\Rind$. From $\EO{}{\Chpar{}}{\rI}{}$ \eqref{b5eq:premonoidal_artin} (resp.\
$\EO{}{\Chpdm{}}{\rI}{}$ \eqref{b5eq:premonoidal_dm}, resp.\ $\EO{}{\Chpdm{}}{\rI}{}$) and $\EO{}{\Chpar{}_\Box}{}{\corr}$
\eqref{b6eq:correspondence_artin} (resp.\ $\EO{}{\Chpdm{}}{}{\corr}$ \eqref{b6eq:correspondence_dm}, resp.\ $\EO{\r{lqf}}{\Chpdm{}}{}{\corr}$
\eqref{b6eq:correspondence_dm_quasifinite}), we directly obtain three operations:
\begin{description}
  \item[1L] $f^*\colon \cD(\cX,\lambda)\to\cD(\cY,\lambda)$, which underlies a monoidal functor
      \[
      f^{*\otimes}\colon\cD(\cX,\lambda)^\otimes\to\cD(\cY,\lambda)^\otimes;
      \]

  \item[2L] $f_!\colon \cD(\cY,\lambda)\to\cD(\cX,\lambda)$ if $f$ is locally of finite type, $\lambda$ belongs to $\Rind_{\ltor}$ and $\cX$ is $\Box$-coprime (resp.\ $f$ is locally of finite type and $\lambda$ belongs to $\Rind_\tor$, resp.\ $f$ is locally quasi-finite and $\lambda$ is arbitrary);

  \item[3L] $-\otimes-=-\otimes_\cX-\colon\cD(\cX,\lambda)\times\cD(\cX,\lambda)\to\cD(\cX,\lambda)$.
\end{description}
If $\cX$ is a $1$-Artin stack (resp.\ $1$-DM stack), then $\cD(\cX,\lambda)^\otimes$ is equivalent to $\cD_\cart(\cX_{\liset},\lambda)^\otimes$ (resp.\ $\cD(\cX_{\et},\lambda)^\otimes$) as symmetric monoidal $\infty$-categories.

Taking right adjoints for (1L) and (2L), respectively, we obtain:
\begin{description}
  \item[1R] $f_*\colon \cD(\cY,\lambda)\to\cD(\cX,\lambda)$;

  \item[2R] $f^!\colon \cD(\cX,\lambda)\to\cD(\cY,\lambda)$ under the same condition as (2L).
\end{description}
For (3L), moving the first factor of the source $\cD(\cX,\lambda)\times\cD(\cX,\lambda)$ to the target side, we can write the functor $-\otimes-$ in the form $\cD(\cX,\lambda)\to\FunL(\cD(\cX,\lambda),\cD(\cX,\lambda))$, since the tensor product on $\cD(\cX,\lambda)$ is closed. Taking opposites and applying \cite{HTT}*{Proposition 5.2.6.2}, we obtain a functor $\cD(\cX,\lambda)^{op}\to\FunR(\cD(\cX,\lambda),\cD(\cX,\lambda))$, which can be written as
\begin{description}
  \item[3R] $\HOM(-,-)=\HOM_\cX(-,-)\colon\cD(\cX,\lambda)^{op}\times\cD(\cX,\lambda) \to\cD(\cX,\lambda)$.
\end{description}

Besides these six operations, for every morphism $\pi\colon\lambda'\to\lambda$ of $\Rind$, we have the following functor of \emph{extension of scalars}:
\begin{description}
  \item[4L] $\pi^*\colon \cD(\cX,\lambda)\to\cD(\cX,\lambda')$, which underlies a monoidal functor
      \[
      \pi^{*\otimes}\colon\cD(\cX,\lambda)^\otimes\to\cD(\cX,\lambda')^\otimes.
      \]
\end{description}
The right adjoint of the functor $\pi^*$ is the functor of \emph{restriction of scalars}:
\begin{description}
  \item[4R] $\pi_*\colon \cD(\cX,\lambda')\to\cD(\cX,\lambda)$.
\end{description}

\begin{theorem}[K\"{u}nneth Formula]\label{b6th:kunneth}
Let $f_i\colon \cY_i\to\cX_i$ ($i=1,\dots,n$) be finitely many morphisms of $\Chpar{}_\Box{}$ (resp.\ $\Chpdm{}$, resp.\ $\Chpdm{}$) that are locally of finite type (resp.\ locally of finite type, resp.\ locally quasi-finite). Given a pullbacks square
\[
\xymatrix{
\cY \ar[rr]^-{(q_1,\dots,q_n)}\ar[d]_-{f}   &&  \cY_1\times\cdots\times \cY_n \ar[d]^-{f_1\times\cdots\times f_n} \\
\cX \ar[rr]^-{(p_1,\dots,p_n)} && \cX_1\times\cdots\times \cX_n }
\]
of $\Chpar{}_\Box{}$ (resp.\ $\Chpdm{}$, resp.\ $\Chpdm{}$), then for every object $\lambda$ of $\Rind_{\ltor}$ (resp.\ $\Rind_\tor$, resp.\ $\Rind$), the following square
\[
\xymatrix{
\cD(\cY_1,\lambda)\times\cdots\times\cD(\cY_n,\lambda)  \ar[rrr]^-{q_1^*-\otimes_\cY\cdots\otimes_\cY q_n^*-}
\ar[d]_-{f_{1!}\times\cdots\times f_{n!}} &&& \cD(\cY,\lambda)  \ar[d]^-{f_!} \\
\cD(\cX_1,\lambda)\times\cdots\times\cD(\cX_n,\lambda)  \ar[rrr]^-{p_1^*-\otimes_\cX\cdots\otimes_\cX p_n^*-}
 &&& \cD(\cX,\lambda)
}\]
is commutative up to equivalence.
\end{theorem}

\begin{proof}
It is a consequence of existence of the map $\EO{}{\Chpar{}_\Box{}}{}{\corr}$ \eqref{b6eq:correspondence_artin} (resp.\
$\EO{}{\Chpdm{}}{}{\corr}$ \eqref{b6eq:correspondence_dm}, resp.\ $\EO{\r{lqf}}{\Chpdm{}}{}{\corr}$ \eqref{b6eq:correspondence_dm_quasifinite}).
\end{proof}

The previous theorem has the following two corollaries.

\begin{corollary}[Base Change]\label{b6co:base_change}
Let
\[
\xymatrix{
\cW \ar[d]_-{q} \ar[r]^-{g} & \cZ \ar[d]^-{p} \\
\cY \ar[r]^-{f} & \cX   }
\]
be a Cartesian diagram in $\Chpar{}_\Box$ (resp.\ $\Chpdm{}$, resp.\ $\Chpdm{}$) where $p$ is locally of finite type (resp.\ locally of finite
type, resp.\ locally quasi-finite). Then for every object $\lambda$ of $\Rind_{\ltor}$ (resp.\ $\Rind_\tor$, resp.\ $\Rind$), the following square
\[
\xymatrix{
\cD(\cW,\lambda) \ar[d]_-{q_!}  & \cD(\cZ,\lambda) \ar[l]_-{g^*} \ar[d]^-{p_!} \\
\cD(\cY,\lambda) & \cD(\cX,\lambda) \ar[l]_-{f^*}   }
\]
is commutative up to equivalence.
\end{corollary}

\begin{corollary}[Projection Formula]\label{b6co:projection}
Let $f\colon \cY\to\cX$ be a morphism of $\Chpar{}_\Box{}$ (resp.\ $\Chpdm{}$, resp.\ $\Chpdm{}$) that is locally of finite type (resp.\ locally of finite type, resp.\ locally quasi-finite). Then the following square
\[
\xymatrix{
\cD(\cY,\lambda)\times\cD(\cX,\lambda) \ar[d]_-{f_!\times\id} \ar[rr]^-{-\otimes_\cY f^*-}
&& \cD(\cY,\lambda) \ar[d]^-{f_!} \\
\cD(\cX,\lambda)\times\cD(\cX,\lambda) \ar[rr]^-{-\otimes_\cX-} && \cD(\cX,\lambda)   }
\]
is commutative up to equivalence.
\end{corollary}

\begin{proposition}\label{b6pr:hom}
Let $f\colon \cY\to\cX$ be a morphism of $\Chpar{}$, and $\lambda$ an object of $\Rind$. Then
\begin{enumerate}
  \item The functors $f^*(-\otimes_\cX-)$ and $(f^*-)\otimes_\cY(f^*-)$ are equivalent.

  \item The functors $\HOM_\cX(-,f_*-)$ and $f_*\HOM_\cY(f^*-,-)$ are equivalent.

  \item If $f$ is a morphism of $\Chpar{}_\Box$ (resp.\ $\Chpdm{}$, resp.\ $\Chpdm{}$) that is locally of finite type (resp.\ locally of finite type, resp.\ locally quasi-finite), and $\lambda$ belongs to $\Rind_{\ltor}$ (resp.\ $\Rind_\tor$, resp.\ $\Rind$), then the functors $f^!\HOM_\cX(-,-)$ and $\HOM_\cY(f^*-,f^!-)$ are equivalent.

  \item Under the same assumptions as in (3), the functors $f_*\HOM_\cY(-,f^!-)$ and $\HOM_\cX(f_!-,-)$ are equivalent.
\end{enumerate}
\end{proposition}

\begin{proof}
For (1), it follows from the fact that $f^*$ is a symmetric monoidal functor.

For (2), the functor $\HOM(-,f_*-)\colon\cD(\cX,\lambda)^{op}\times\cD(\cY,\lambda)\to\cD(\cX,\lambda)$ induces a functor $\cD(\cX,\lambda)^{op}\to\FunR(\cD(\cY,\lambda),\cD(\cX,\lambda))$. Taking opposite, we obtain a functor $\cD(\cX,\lambda)\to\FunL(\cD(\cX,\lambda),\cD(\cY,\lambda))$, which induces a functor $\cD(\cX,\lambda)\times\cD(\cX,\lambda)\to\cD(\cY,\lambda)$. By construction, the latter is equivalent to the functor $f^*(-\otimes_\cX-)$.
Repeating the same process for $f_*\HOM(f^*-,-)$, we obtain $(f^*-)\otimes_\cY(f^*-)$. Therefore, by (1), the functors $\HOM(-,f_*-)$ and
$f_*\HOM(f^*-,-)$ are equivalent.

For (3), the functor $f^!\HOM(-,-)\colon\cD(\cX,\lambda)^{op}\times\cD(\cX,\lambda)\to\cD(\cY,\lambda)$ induces a functor $\cD(\cX,\lambda)^{op}\to\FunR(\cD(\cX,\lambda),\cD(\cY,\lambda))$. Taking opposite, we obtain a functor
$\cD(\cX,\lambda)\to\FunL(\cD(\cY,\lambda),\cD(\cX,\lambda))$, which induces a functor $\cD(\cX,\lambda)\times\cD(\cY,\lambda)\to\cD(\cX,\lambda)$. By construction, the latter is equivalent to the functor $-\otimes_\cX(f_!-)$.
Repeating the same process for $\HOM(f^*-,f^!-)$, we obtain $f_!((f^*-)\otimes_\cY-)$. Therefore, by Corollary \ref{b6co:projection}, the functors $f^!\HOM(-,-)$ and $\HOM(f^*-,f^!-)$ are equivalent.

For (4), the functor $f_*\HOM(-,f^!-)\colon\cD(\cY,\lambda)^{op}\times\cD(\cX,\lambda)\to\cD(\cX,\lambda)$ induces a functor $\cD(\cY,\lambda)^{op}\to\FunR(\cD(\cX,\lambda),\cD(\cX,\lambda))$. Taking opposite, we obtain a functor
$\cD(\cY,\lambda)\to\FunL(\cD(\cX,\lambda),\cD(\cX,\lambda))$, which induces a functor $\cD(\cY,\lambda)\times\cD(\cX,\lambda)\to\cD(\cX,\lambda)$. By construction, the latter is equivalent to the functor $f_!(-\otimes_\cY(f^*-))$. Repeating the same process for $\HOM(f_!-,-)$, we obtain $(f_!-)\otimes_\cX-$. Therefore, by Corollary \ref{b6co:projection}, the functors $f_*\HOM(-,f^!-)$ and $\HOM(f_!-,-)$ are equivalent.
\end{proof}

\begin{proposition}\label{b6pr:hom_pi}
Let $\cX$ be an object of $\Chpar{}$, and $\pi\colon \lambda'\to \lambda$ a morphism of $\Rind$. Then
\begin{enumerate}
  \item The functors $\pi^*(-\otimes_\lambda -)$ and $(\pi^*-)\otimes_{\lambda'}(\pi^*-)$ are equivalent.

  \item The functors $\HOM_\lambda(-,\pi_*-)$ and $\pi_*\HOM_{\lambda'}(\pi^*-,-)$ are equivalent.
\end{enumerate}
\end{proposition}

\begin{proof}
The proof is similar to Proposition \ref{b6pr:hom}.
\end{proof}

\begin{proposition}\label{b6pr:upperstar_pi}
Let $f\colon \cY\to \cX$ be a morphism of $\Chpar{}$, and $\pi\colon\lambda'\to \lambda$ a perfect morphism of $\Rind$. Then the square
\begin{align}\label{b6eq:upperstar_pi}
\xymatrix{\cD(\cY,\lambda') & \cD(\cX,\lambda')\ar[l]_{f^*}\\
\cD(\cY,\lambda)\ar[u]^{\pi^*} & \cD(\cX,\lambda)\ar[u]_{\pi^*}\ar[l]_{f^*}}
\end{align}
is right adjointable and its transpose is left adjointable.
\end{proposition}

In particular, if $\cX$ is an object of $\Chpar{}$ and $\pi\colon \lambda'\to\lambda$ is a perfect morphism of $\Rind$, then $\pi^*$ admits a left adjoint
\[
\pi_!\colon \cD(\cX,\lambda')\to \cD(\cX,\lambda).
\]

\begin{proof}
The first assertion follows from the second one. To show the second assertion, by Lemma \ref{b4le:adjoint_limit}, we may assume that $f$ is a
morphism of $\Schqcs$. In this case the proposition reduces to Lemma \ref{b2le:upperstar_pi}.
\end{proof}

\begin{proposition}\label{b6pr:lowersh_pi}
Let $f\colon \cY\to\cX$ be a morphism of $\Chpar{}_\Box{}$ (resp.\ $\Chpdm{}$, resp.\ $\Chpdm{}$) that is locally of finite type (resp.\ locally of finite type, resp.\ locally quasi-finite), and $\pi\colon \lambda'\to\lambda$ a perfect morphism of $\Rind_{\ltor}$ (resp.\ $\Rind_\tor$, resp.\ $\Rind$). Then the square
\begin{align}\label{b6eq:lowersh_pi}
\xymatrix{\cD(\cY,\lambda')\ar[r]^{f_!} & \cD(\cX,\lambda')\\
\cD(\cY,\lambda)\ar[r]^{f_!}\ar[u]^{\pi^*} & \cD(\cX,\lambda)\ar[u]_{\pi^*}}
\end{align}
is right adjointable and its transpose is left adjointable.
\end{proposition}

\begin{proof}
This follows from Lemma \ref{b4le:adjoint_limit} and Lemma \ref{b3le:lowersh_pi}.
\end{proof}

\begin{proposition}\label{b6pr:Hom_pi2}
Let $\cX$ be an object of $\Chpar{}$, $\lambda=(\Xi,\Lambda)$ an object of $\Rind$, and $\xi$ an object of $\Xi$. Consider the obvious morphism
$\pi\colon \lambda'\coloneqq(\Xi_{/\xi}, \Lambda \res \Xi_{/\xi})\to\lambda$. Then
\begin{enumerate}
  \item The natural transformation $\pi_!(-\otimes_{\lambda'}\pi^* -)\to(\pi_!-)\otimes_\lambda -$ is a natural equivalence.

  \item The natural transformation $\pi^*\HOM_\lambda(-,-)\to\HOM_{\lambda'}(\pi^*-,\pi^*-)$ is a natural equivalence.

  \item The natural transformation $\HOM_\lambda(\pi_!-,-)\to\pi_*\HOM_{\lambda'}(-,\pi^*-)$ is a natural equivalence.
\end{enumerate}
\end{proposition}

\begin{proof}
Similarly to the proof of Proposition \ref{b6pr:hom}(3,4), one shows that the three assertions are equivalent (for every given $\cX$). For assertion (1), we may assume that $\cX$ is an object of $\Schqcs$. In this case, assertion (2) follows from the fact that $\pi^*$ preserves fibrant objects in $\Ch(\Mod(-))^{\r{inj}}$.
\end{proof}

Let $\cX$ be an object of $\Chpar{}$, and $\lambda=(\Xi,\Lambda)$ and object of $\Rind$. There is a t-structure on $\cD(\cX,\lambda)$, such that if $\cX$ is a $1$-Artin stack (resp.\ $1$-DM stack), then it induces the usual t-structure on its homotopy category $\rD_\cart(\cX_{\liset}^\Xi,\Lambda)$ (resp.\ $\rD(\cX_{\et}^\Xi,\Lambda)$). For an object $s_\cX\colon\cX\to\Spec\dZ$ of $\Chpar{}$, we put $\lambda_\cX\coloneqq s_\cX^*\lambda_{\Spec\dZ}$, which is a monoidal unit of $\cD(\cX,\lambda)^\otimes$ and also an object of
$\cD^\heartsuit(\cX,\lambda)$.

\begin{theorem}[Poincar\'{e} duality]\label{b6th:poincare}
Let $f\colon \cY\to\cX$ be a morphism of $\Chpar{}_\Box$ (resp.\ $\Chpdm{}$) that is flat (resp.\ flat and locally quasi-finite) and locally of finite presentation. Let $\lambda$ be an object of $\Rind_{\ltor}$ (resp.\ $\Rind$). Then
\begin{enumerate}
  \item There is a trace map
      \[
      \Tr_f\colon\tau^{\geq0}f_!\lambda_\cY\langle d\rangle=\tau^{\geq0}f_!(f^*\lambda_\cX)\langle d\rangle\to\lambda_\cX
      \]
      for every integer $d\geq\dim^+(f)$, which is functorial in the sense of Remark \ref{b4re:trace}.

  \item If $f$ is moreover smooth, the induced natural transformation
      \[
      u_f\colon f_!\circ f^*\langle\dim f\rangle\to\id_\cX
      \]
      is a counit transformation, so that the induced map
      \[
      f^*\langle\dim f\rangle\to f^!\colon \cD(\cX,\lambda)\to\cD(\cY,\lambda)
      \]
      is a natural equivalence of functors.
\end{enumerate}
\end{theorem}

\begin{proof}
This is simply (P7) of DESCENT.
\end{proof}

\begin{corollary}[Smooth (resp.\ \'Etale) Base Change]\label{b6co:smooth_base_change}
Let
\[
\xymatrix{
\cW \ar[d]_-{q} \ar[r]^-{g} & \cZ \ar[d]^-{p} \\
\cY \ar[r]^-{f} & \cX   }
\]
be a Cartesian diagram in $\Chpar{}_\Box$ (resp.\ $\Chpdm{}$) where $p$ is smooth (resp.\ \'etale). Then for every object $\lambda$ of $\Rind_{\ltor}$ (resp.\ $\Rind$), the following square
\[
\xymatrix{
\cD(\cW,\lambda)   & \cD(\cZ,\lambda) \ar[l]_-{g^*}  \\
\cD(\cY,\lambda) \ar[u]^-{q^*} & \cD(\cX,\lambda) \ar[l]_-{f^*}\ar[u]_-{p^*}   }
\]
is right adjointable.
\end{corollary}

\begin{proof}
This is part (1) of (P5\bis). It also follows from Corollary \ref{b6co:base_change} and Theorem \ref{b6th:poincare}(2) as in Lemma \ref{b4le:outputii}.
\end{proof}

\begin{proposition}\label{b6pr:support}
Let $f\colon \cY\to\cX$ be a morphism of $\Chpar{}_\Box$ (resp.\ $\Chpdm{}$), and $\lambda$ an object of $\Rind_{\ltor}$ (resp.\ $\Rind_\tor$). Assume that for every morphism $X\to\cX$ from an algebraic space $X$, the base change $\cY\times_\cX X\to X$ is a proper morphism of algebraic spaces; in particular, $f$ is locally of finite type. Then
\[
f_*,f_!\colon\cD(\cY,\lambda)\to\cD(\cX,\lambda)
\]
are equivalent functors.
\end{proposition}

\begin{proof}
We only prove the proposition for $\Chpar{}_\Box$ and leave the other case to readers. For simplicity, we call such morphism $f$ in the proposition as \emph{proper}. For every integer $k\geq 0$, denote by $\cC^k$ the subcategory of $\Fun(\Delta^1,\Chpar{k}_\Box)$ spanned by objects of the form $f\colon\cY\to\cX$ that is proper and edges of the form
\begin{align}\label{b6eq:support}
\xymatrix{
\cY' \ar[r]^-{f'}\ar[d]_-{q} & \cX' \ar[d]^-{p} \\
\cY \ar[r]^-{f} & \cX
}
\end{align}
that is a Cartesian diagram in which $p$ hence $q$ are smooth. In addition, we let $\cC^{-1}$ be the subcategory of $\cC^0$ spanned by $f\colon\cY\to\cX$ such that $\cX$ hence $\cY$ are quasi-compact separated algebraic spaces. For $k\geq-1$, denote by $\cE^k$ the subset of $(\cC^k)_1$ consists of \eqref{b6eq:support} in which $p$ hence $q$ are moreover surjective. We have $\cE^k\cap(\cC^{k-1})_1=\cE^{k-1}$ for $k\geq 0$.

By Corollary \ref{b6co:smooth_base_change} and the map $\EO{}{\Chpar{k}_\Box}{*}{!}$ (obtained from $\EO{}{\Chpar{k}_\Box}{\r{II}}{}$ as in \eqref{b3eq:operation1}), for every $k\geq -1$, we have two functors
\[
F^k_*,F^k_!\colon(\cC^k)^{op}\to\Fun(\Delta^1,\Cat)
\]
in which the first (resp.\ second) one sends $f\colon\cY\to\cX$ to $f_*\colon\cD(\cY,\lambda)\to\cD(\cX,\lambda)$ (resp.\ $f_!\colon\cD(\cY,\lambda)\to\cD(\cX,\lambda)$), and an edge \eqref{b6eq:support} to
\[
\xymatrix{
\cD(\cY',\lambda) \ar[rr]^-{f'_*\text{ (resp.\ $f'_!$)}} && \cD(\cX',\lambda) \\
\cD(\cY,\lambda) \ar[rr]^-{f_*\text{ (resp.\ $f_!$)}}\ar[u]^-{q^*} &&  \cD(\cX,\lambda).\ar[u]_-{p^*}
}
\]
By Remark \ref{b5re:proper}, $F^{-1}_*$ and $F^{-1}_!$ are equivalence. Applying Proposition \ref{b4pr:toy} successively to marked $\infty$-categories $(\cC^k,\cE^k)$, we conclude that $F^k_*$ and $F^k_!$ are equivalence for every $k\geq 0$. The proposition follows.
\end{proof}


\begin{remark}\label{b6re:support}
Let $f\colon \cY\to\cX$ be a morphism of $\Chpar{}_\Box$ (resp.\ $\Chpdm{}$) that is locally of finite type and representable by DM stacks, and $\lambda$ an object of $\Rind_{\ltor}$ (resp.\ $\Rind_\tor$). We can always construct a natural transformation
\[
f_!\to f_*\colon\cD(\cY,\lambda)\to\cD(\cX,\lambda)
\]
of functors, which specializes to the equivalence in Proposition \ref{b6pr:support} if $f$ satisfies the property there.
\end{remark}

\begin{theorem}[(Co)homological descent]\label{b6th:descent}
Let $f\colon X^+_0\to X^+_{-1}$ be a smooth surjective morphism of $\Chpar{}$ (resp.\ $\Chpdm{}$), and $X^+_\bullet$ a \v{C}ech nerve of $f$.
\begin{enumerate}
  \item For every object $\lambda$ of $\Rind$, the functor
      \[
      \cD(X^+_{-1},\lambda)\to \lim_{n\in \del} \cD(X^+_n,\lambda)
      \]
      is an equivalence, where the transition maps in the limit are provided by $*$-pullback.

  \item Suppose that $f$ is a morphism of $\Chpar{}_\Box$ (resp.\ $\Chpdm{}$). For every object $\lambda$ of $\Rind_{\ltor}$ (resp.\
      $\Rind_\tor$), the functor
      \[
      \cD(X^+_{-1},\lambda)\to\lim_{n\in \del}\cD(X^+_n,\lambda)
      \]
      is an equivalence, where the transition maps in the limit are provided by $!$-pullback.
\end{enumerate}
\end{theorem}

\begin{proof}
This follows from (P4) of DESCENT.
\end{proof}

\begin{corollary}\label{b6co:descent}
Let $f\colon Y\to X$ be a morphism of $\Chpar{}$ (resp.\ $\Chpdm{}$) and let $y\colon Y^+_0\to Y$ be a smooth surjective morphism of $\Chpar{}$ (resp.\ $\Chpdm{}$). Denote $Y^+_\bullet$ the \v{C}ech nerve of $y$ with the morphism $y_n\colon Y^+_n\to Y^+_{-1}=Y$. Put $f_n\coloneqq f\circ y_n\colon Y^+_n\to X$.
\begin{enumerate}
  \item For every object $\lambda$ of $\Rind$ and every object $\sfK\in\cD^{\geq0}(Y,\Lambda)$, we have a convergent spectral sequence
      \[
      \rE_1^{p,q}=\rH^q(f_{p*}y_p^*\sfK)\Rightarrow \rH^{p+q} f_* \sfK.
      \]

  \item Suppose that $f$ is a morphism of $\Chpar{}_\Box$ (resp.\ $\Chpdm{}$). For every object $\lambda$ of $\Rind_{\ltor}$ (resp.\ $\Rind_\tor$) and every object $\sfK\in\cD^{\leq0}(Y,\Lambda)$, we have a convergent spectral sequence
      \[
      \tilde\rE_1^{p,q}=\rH^q(f_{-p!}y_{-p}^!\sfK)\Rightarrow\rH^{p+q}f_!\sfK.
      \]
\end{enumerate}
\end{corollary}

\begin{proof}
This essentially follows from Theorem \ref{b6th:descent} and 
\cite{HA}*{Proposition 1.2.4.5, Variant 1.2.4.9}. 

For (1), we obtain a cosimplicial object $\rN(\del)\to\cD^{\geq0}(Y,\Lambda)$ whose value at $[n]$ is $y_{n*}y_n^*\sfK$, such that $\sfK$ is its limit by Theorem \ref{b6th:descent}(1); in other words, we have $\sfK\xrightarrow{\sim}\lim_{n\in \del}y_{n*}y_n^*\sfK$. Applying the functor $f_*$, we obtain another cosimplicial object $\rN(\del)\to\cD^{\geq0}(X,\Lambda)$ whose value at $[n]$ is $f_{n*}y_n^*\sfK$, such that $f_*\sfK$ is its limit. Put $\cC\coloneqq\cD(X,\Lambda)^{op}$ and let $\cC_{\geq 0}\coloneqq\cD^{\geq0}(X,\Lambda)^{op}$, $\cC_{\leq 0}\coloneqq\cD^{\leq0}(X,\Lambda)^{op}$ be the induced (homological) t-structure. Then we obtain a simplicial object $\rN(\del)^{op}\to\cC_{\geq 0}$ whose value at $[n]$ is $f_{n*}y_n^*\sfK$, with $f_*\sfK$ its geometric realization. By \cite{HA}*{Proposition 1.2.4.5, Variant 1.2.4.9}, we obtain a spectral sequence $\{\rE_r^{p,q}\}_{r\geq 1}$ abutting to $\rH^{p+q} f_* \sfK$, with $\rE_1^{p,q}=\rH^q(f_{p*}y_p^*\sfK)$.

For (2), by Theorem \ref{b6th:descent}(2), the functor $\cD(Y,\lambda)^{op}\to\lim_{n\in \del}\cD(Y^+_n,\lambda)^{op}$ is an equivalence, where the transition maps in the limit are provided by $!$-pullback. Similar to (1), we obtain a cosimplicial object $\rN(\del)\to\cD^{\leq0}(Y,\Lambda)^{op}$ whose value at $[n]$ is $y_{n!}y_n^!\sfK$, such that $\sfK$ is its limit. Applying the functor $f_!$, we obtain another cosimplicial object $\rN(\del)\to\cD^{\leq0}(Y,\Lambda)^{op}$ whose value at $[n]$ is $f_{n!}y_n^!\sfK$, such that $f_!\sfK$ is its limit. Put $\cC\coloneqq\cD(X,\Lambda)$ and let $\cC_{\geq 0}\coloneqq\cD^{\leq0}(X,\Lambda)$, $\cC_{\leq 0}\coloneqq\cD^{\geq0}(X,\Lambda)$ be the induced (homological) t-structure. Then we obtain a simplicial object $\rN(\del)^{op}\to\cC_{\geq 0}$ whose value at $[n]$ is $f_{n!}y_n^!\sfK$, with $f_!\sfK$ its geometric realization. By \cite{HA}*{Proposition 1.2.4.5, Variant 1.2.4.9}, we obtain a spectral sequence $\{\tilde\rE_r^{p,q}\}_{r\geq 1}$ abutting to $\rH^{p+q} f_!\sfK$, with $\tilde\rE_1^{p,q}=\rH^q(f_{-p!}y_{-p}^!\sfK)$.
\end{proof}

The following lemma will be used in \Sec\ref{b6ss:constructible}.

\begin{lem}\label{b6le:dimension}
Let $f\colon Y\to X$ be a morphism locally of finite type of $\Chpar{}_\Box$ 
(resp.\ $\Chpdm{}$), and $\lambda$ an object of $\Rind_{\ltor}$ (resp.\ 
$\Rind_\tor$). Then $f_!$ restricts to a functor 
$\cD^{\le0}(Y,\lambda)\to\cD^{\le 2d}(X,\lambda)$, where $d=\dim^+(f)$. 
Moreover, if $f$ is smooth (resp.\ \'etale), then $f_!\circ f^!$ restricts to 
a functor $\cD^{\le0}(X,\lambda)\to\cD^{\le 0}(X,\lambda)$. 
\end{lem}

\begin{proof}
We may assume that $X$ is the spectrum of a separably closed field.

We prove the first assertion by induction on $k$ when $Y$ is a $k$-Artin stack. Take an object $\sfK\in \cD^{\le 0}(Y,\lambda)$. For $k=-2$, $Y$ is the coproduct of a family $(Y_i)_{i\in I}$ of morphisms of schemes separated and of finite type over $X$, so that
\[
f_!\sfK=\bigoplus_{i\in I}f_{i!}(\sfK\res Y_i)\in \cD^{\le 2d}(X,\lambda),
\]
where $f_i$ is the composite morphism $Y_i\to Y\xrightarrow{f}X$. Assume the assertion proved for some $k\ge -2$, and let $Y$ be a $(k+1)$-Artin stack. Let $Y_\bullet$ be a \v{C}ech nerve of an atlas (resp.\ \'etale atlas) $y_0 \colon Y_0\to Y$ and form a triangle
\[
\xymatrix{ & Y\ar[rd]^-{f}\\
Y_\bullet\ar[rr]^-{f_\bullet} \ar[ur]^-{y_\bullet}&&X.}
\]
Then, by Theorem \ref{b6th:descent}(2), we have $f_!\sfK\simeq\colim_{n\in\del^{op}}f_{n!}y_n^!\sfK$. Thus, it suffices to show that for every smooth (resp.\ \'etale) morphism $g\colon Z\to X$ where $Z$ is a $k$-Artin stack, $(f\circ g)_{!}g^! \sfK$ belongs to $\cD^{\le 2d}(X,\lambda)$. For this, we may assume that $g$ is of pure dimension $e$ (resp.\ $0$). The assertion then follows from Theorem \ref{b6th:poincare} and induction hypothesis.

For the second assertion, we may assume that $f$ is of pure dimension $d$ (resp.\ $0$). It then follows from Theorem \ref{b6th:poincare}(2) and the first assertion.
\end{proof}

\begin{remark}\label{b6re:comparison}
Let $f\colon \cY\to \cX$ be a smooth morphism of (1-)Artin stacks, and $\pi\colon \Lambda'\to \Lambda$ a ring homomorphism. Standard functors for the lisse-\'etale topoi induce
\begin{align*}
\rL f_{\liset}^*&\colon \rD_\cart(\cX_\liset,\Lambda)\to\rD_\cart(\cY_\liset,\Lambda),\\
-\overset{\rL}{\otimes}_\cX-&\colon
\rD_\cart(\cX_\liset,\Lambda)\times\rD_\cart(\cX_\liset,\Lambda)\to\rD_\cart(\cX_\liset,\Lambda),\\
\rL\pi^*&\colon
\rD_\cart(\cX_\liset,\Lambda)\to\rD_\cart(\cX_\liset,\Lambda').
\end{align*}
By Corollary \ref{b5co:compatibility}, we have an equivalence of categories
\begin{align}\label{b6eq:equivalence}
\rh\cD(\cX,\Lambda)\simeq\rD_\cart(\cX_\liset,\Lambda),
\end{align}
and isomorphisms of functors
\[
\rh f^*\simeq\rL f_{\liset}^*,\quad\rh(-\otimes_\cX-)\simeq(-\overset{\rL}{\otimes}_\cX-),\quad \rh \pi^*\simeq\rL\pi^*,
\]
compatible with \eqref{b6eq:equivalence}.

Let $f\colon \cY\to \cX$ be a morphism of Artin stacks. Using the methods of \cite{Olsson}*{(9.16.2)}, one can define a functor
\[
\rL^+f^*\colon \rD_\cart^+(\cX_\liset,\Lambda)\to\rD_\cart^+(\cY_\liset,\Lambda).
\]
Similarly to Proposition \ref{b6pr:compatibility} in \Sec\ref{b6ss:compatibility}, there is an isomorphism between $\rh f^{*+}\simeq\rL^+f_{\liset}^*$, compatible with \eqref{b6eq:equivalence}, where $f^{*+}$ denotes the obvious restriction of $f^*$.

Assume that there exists a nonempty set $\Box$ of rational primes such that $\Lambda$ is $\Box$-torsion and $\cX$ is $\Box$-coprime. Then the functors $\rR^+f_{\liset*}$ and $\Rhom_\cX$ for the lisse-\'etale topoi induce
\begin{align*}
\rR^+f_{\liset*}&\colon \rD_\cart^{+}(\cY_\liset,\Lambda)\to\rD_\cart^{+}(\cX_\liset,\Lambda),\\
\Rhom_\cX&\colon
\rD_\cart(\cX_\liset,\Lambda)^{op}\times\rD_\cart(\cX_\liset,\Lambda)\to\rD_\cart(\cX_\liset,\Lambda).
\end{align*}
Indeed, the statement for $\rR f_{\liset*}$, similar to \cite{Olsson}*{Proposition 9.9}, follows from smooth base change; and the statement for $\Rhom_\cX$, similar to \cite{LO1}*{Corollary 4.2.2}, follows from the fact that the map $g^*\Rhom_X(-,-)\to\Rhom_Y(g^*-,g^*-)$ is an equivalence for every smooth morphism $g\colon Y\to X$ of $\Box$-coprime schemes, which in turn follows from the Poincar\'e duality. By adjunction, we obtain isomorphisms of functors $\rh\HOM_\cX\simeq\Rhom_\cX$ and $\rh f^+_*\simeq\rR^+f_{\liset*}$, compatible with \eqref{b6eq:equivalence}.
\end{remark}

\subsection{More adjointness in the finite-dimensional Noetherian case}
\label{b6ss:adjointness}

Recall the following result of Gabber: for every morphism $f\colon Y\to X$ of finite type between finite-dimensional Noetherian schemes, and every prime number $\ell$ invertible on $X$, the $\ell$-cohomological dimension of $f_*$ is finite \cite{TG}*{Expos\'{e} XVIII-A, Corollary 1.4}. This extends easily to morphisms representable by algebraic spaces as follows.

\begin{lem}\label{6le:cohdim}
Let $f\colon Y\to X$ be a morphism of finite presentation between $\square$-coprime finite-dimensional Noetherian higher Artin stacks. Let $\lambda$ be a $\square$-torsion ringed diagram. Then $f^!\colon\cD(X,\lambda)\to \cD(Y,\lambda)$ and, if $f$ is $0$-Artin, $f_*\colon\cD(Y,\lambda)\to \cD(X,\lambda)$ have bounded cohomological amplitude.
\end{lem}

\begin{proof}
For the first assertion, we reduce by Poincar\'e duality first to the case of a morphism between affine schemes, and then to the case of a closed immersion. In this case, the assertion follows from Gabber's theorem for the complementary open immersion. For the second assertion, we reduce to the case where $X$ is a scheme. Then $Y$ is an algebraic space. By Noetherian induction, it suffices to show that for every open immersion $j\colon V\to Y$ with $V$ a scheme, the $\ell$-cohomological dimensions of $j_*$ and $(fj)_*$ is finite. Thus we may assume $f$ is representable by schemes. This case follows readily from the case of schemes.
\end{proof}

We say that a higher Artin stack $X$ is locally Noetherian (resp.\ locally finite-dimensional) if $X$ admitting an atlas $Y\to X$ where $Y$ is a coproduct of Noetherian (resp.\ finite-dimensional) schemes.

\begin{proposition}\label{b6pr:more_adj}
Let $f\colon Y\to X$ be a morphism locally of finite type of $\Chpar{}_\Box$, and $\pi\colon \lambda'\to\lambda$ an arbitrary morphism of $\Rind_{\ltor}$. Assume that $X$ is locally Noetherian and locally finite-dimensional. Then $f^!\colon\cD(X,\lambda)\to \cD(Y,\lambda)$ admits a right adjoint; the squares \eqref{b6eq:upperstar_pi} and \eqref{b6eq:lowersh_pi} are right adjointable. Moreover, if $f$ is 0-Artin, quasi-compact and quasi-separated, then $f_*\colon \cD(Y,\lambda)\to\cD(X,\lambda)$ also admits a right adjoint.
\end{proposition}

\begin{proof}
Let $g\colon \coprod Z_i=Z\to Y$ be an atlas of $Y$. By the Poincar\'e duality, $g^!$ is conservative, and $h_i^!$ exhibits $\cD(Z,\lambda)$ as the product of $\cD(Z_i,\lambda)$, where $h_i\colon Z_i\to Z$. Therefore, to show that $f^!$ preserves small colimits, it suffices to show that, for every $i$, $(f\circ g_i)^!$ preserves small colimits, where $g_i\colon Z_i\to Y$. We may thus assume that $X$ and $Y$ are both affine schemes. Let $i$ be a closed embedding of $Y$ into an affine space over $X$. It then suffices to show that $i^!$ preserves small colimits, which follows from the finiteness of cohomological dimension of $j_*$, where $j$ is the complementary open immersion.

To show that \eqref{b6eq:upperstar_pi} and \eqref{b6eq:lowersh_pi} are right adjointable, we reduce by Lemma \ref{b4le:adjoint_limit} to the case of affine schemes. By the factorization above and the Poincar\'e duality, the assertion for $f^!$ reduces to the assertion for $f_*$. We may further assume that $\Xi'=\Xi=\{\ast\}$ where $\lambda=(\Xi,\Lambda)$ and $\lambda'=(\Xi',\Lambda')$. In this case, it suffices to take a resolution of $\Lambda'$ by free $\Lambda$-modules.

For the last assertion, by smooth base change, we may assume that $X$ is an affine Noetherian scheme. In this case, by Lemma \ref{6le:cohdim}, $f_*\colon \cD(Y,\lambda)\to \cD(X,\lambda)$ commutes with small colimits and hence admits a right adjoint.
\end{proof}

\subsection{Constructible complexes}
\label{b6ss:constructible}

We study constructible complexes on higher Artin stacks and their behavior under the six operations. Let $\lambda=(\Xi,\Lambda)$ be a \emph{Noetherian} ringed diagram. For every object $\xi$ of $\Xi$, we denote by $e_\xi$ the morphism $(\{\xi\},\Lambda(\xi))\to(\Xi,\Lambda)$.

We start from the case of schemes. Let $X$ be a scheme. Recall from \cite{SGA4}*{Expos\'{e} ix, D\'{e}finition 2.3} that for a Noetherian ring $R$, a sheaf $\cF$ of $R$-modules on $X$ is said to be \emph{constructible} if the stalks of $\cF$ are finitely-generated $R$-modules and every affine open subset of $X$ is the disjoint union of finitely many constructible subschemes $U_i$ such that the restriction of $\cF$ to each $U_i$ is locally constant.

\begin{definition}\label{b6de:constructible}
We say that an object $\sfK$ of $\cD(X,\lambda)$ is \emph{a constructible complex} or simply \emph{constructible} if for every object $\xi$ of $\Xi$ and every $q\in \dZ$, the sheaf $\rH^q e_\xi^*\sfK\in\Mod(X,\Lambda(\xi))$ is constructible. We say that an object $\sfK$ of $\cD(X,\lambda)$ is \emph{locally bounded from below} (resp.\ \emph{locally bounded from above}) if for every object $\xi$ of $\Xi$ and every quasi-compact open subscheme $U$ of $X$, $e_\xi^*\sfK\res U$ is bounded from below (resp.\ bounded from above).
\end{definition}

Note that we do not require constructible complexes to be bounded in either direction. Note that $\sfK\in\cD(X,\lambda)$ is locally bounded from below (resp.\ from above) if and only if there exists a Zariski open covering $(U_i)_{i\in I}$ of $X$ such that $\sfK\res U_i$ is bounded from below (resp.\ from above).

\begin{lem}\label{b6le:cons}
Let $f\colon Y\to X$ be a morphism of schemes. Let $\sfK$ be an object of $\cD(X,\lambda)$. If $\sfK$ is constructible (resp.\ locally bounded from below, resp.\ locally bounded from above), then $f^*\sfK$ satisfies the same property. The converse holds when $f$ is surjective and locally of finite presentation.
\end{lem}

\begin{proof}
The constructible case follows from \cite{SGA4}*{Expos\'{e} ix, Propositions 2.4(iii), 2.8}. For the locally bounded case we use the characterization by open coverings. The first assertion is then clear. For the second assertion, by \cite{SGA4}*{Expos\'{e} ix, Lemme 2.8.1} we may assume $f$ flat, hence open. In this case the image of an open covering of $Y$ is an open covering of $X$.
\end{proof}

The lemma implies that Definition \ref{b6de:constructible} is compatible with the following.

\begin{definition}[Constructible complex]\label{b6de:construcible_stack}
Let $X$ be a higher Artin stack. We say that an object $\sfK$ of $\cD(X,\lambda)$ is \emph{a constructible complex} or simply
\emph{constructible} (resp.\ \emph{locally bounded from below}, resp.\ \emph{locally bounded from above}) if there exists an atlas $f\colon Y\to X$ with $Y$ a scheme, $f^*\sfK$ is constructible (resp.\ locally bounded from below, resp.\ locally bounded from above).
\end{definition}

We denote by $\cD_\cons(X,\lambda)$ (resp.\ $\cD^{(+)}(X,\lambda)$, $\cD^{(-)}(X,\lambda)$ or $\cD^{(\rb)}(X,\lambda)$) the full subcategory of $\cD(X,\lambda)$ spanned by objects that are constructible (resp.\ locally bounded from below, locally bounded from above, or locally bounded from both sides). Moreover, we put
\begin{align*}
\cD_\cons^{(+)}(X,\lambda)&\coloneqq\cD_\cons(X,\lambda)\cap\cD^{(+)}(X,\lambda),\\
\cD_\cons^{(-)}(X,\lambda)&\coloneqq\cD_\cons(X,\lambda)\cap\cD^{(-)}(X,\lambda),\\
\cD_\cons^{(\rb)}(X,\lambda)&\coloneqq\cD_\cons(X,\lambda)\cap\cD^{(\rb)}(X,\lambda).
\end{align*}

\begin{proposition}\label{b6pr:constructible}
Let $f\colon Y\to X$ be a morphism of higher Artin stacks.
\begin{enumerate}
  \item Let $\sfK$ be an object of $\cD(X,\lambda)$. If $\sfK$ is constructible (resp.\ locally bounded from below, resp.\ locally bounded from above), then $f^*\sfK$ satisfies the same property. The converse holds when $f$ is surjective and locally of finite presentation. In particular, $f^*$ restricts to a functor
      \begin{description}
      \item[1L'] $f^*\colon\cD_\cons(X,\lambda)\to\cD_\cons(Y,\lambda)$.
      \end{description}

  \item Suppose that $X$ and $Y$ are $\Box$-coprime higher Artin stacks (resp.\ higher DM stacks), and $f$ is of finite presentation
      (Definition \ref{b5de:finite_presentation}). Let $\lambda$ be a $\Box$-torsion (resp.\ torsion) Noetherian ringed diagram. Then $f_!$ restricts to
        \begin{description}
        \item[2L'] $f_!\colon\cD^{(-)}_\cons(Y,\lambda)\to\cD^{(-)}_\cons(X,\lambda)$, and if $f$ is $0$-Artin (resp.\ $0$-DM), $f_!\colon\cD_\cons(Y,\lambda)\to\cD_\cons(X,\lambda)$.
        \end{description}

  \item The functor $-\otimes_X -$ restricts to a functor
      \begin{description}
        \item[3L'] $-\otimes_X-\colon\cD^{(-)}_\cons(X,\lambda)\times\cD^{(-)}_\cons(X,\lambda)\to\cD^{(-)}_\cons(X,\lambda)$.
      \end{description}
      In particular, $\cD_\cons^{(-)}(X,\lambda)^\otimes$ is a symmetric monoidal subcategory.
\end{enumerate}
\end{proposition}

\begin{proof}
For (1), we reduce by taking atlases to the case of schemes, which is Lemma \ref{b6le:cons}. The reduction for the second assertion is clear. The reduction for the first assertion uses the second assertion.

For (2), we may assume $\Xi=\{\ast\}$. We prove by induction on $k$ that the assertion holds when $f$ is a morphism of $k$-Artin (resp.\ $k$-DM) stacks. The case $k=-2$ is \cite{SGA4}*{Expos\'{e} xvii, Th\'{e}or\`{e}me 5.3.6}. Now assume that the assertions hold for some $k\ge -2$ and let $f$ be a morphism of $(k+1)$-Artin (resp.\ $(k+1)$-DM) stacks. By smooth base change (Corollary \ref{b6co:smooth_base_change}), we may assume that $X$ is an affine scheme. Then $Y$ is a $(k+1)$-Artin (resp.\ $(k+1)$-DM) stack, of finite presentation over $X$. It suffices to show that for every object $\sfK$ of $\cD^{\leq0}_\cons(Y,\lambda)$, $f_!\sfK$ belongs to $\cD^{\le 2d}_\cons(Y,\lambda)$, where $d=\dim^+(f)$. Let $Y_\bullet$ be a \v{C}ech nerve of an atlas $y_0\colon Y_0\to Y$, where $Y_0$ is an affine scheme, and form a triangle
\[
\xymatrix{ & Y\ar[rd]^-{f}\\
Y_\bullet\ar[rr]^-{f_\bullet} \ar[ur]^-{y_\bullet}&&X.}
\]
Then for $n\geq0$, $f_n$ is a quasi-compact and quasi-separated morphism of $k$-Artin (resp.\ $k$-DM) stacks. By Theorem \ref{b6th:descent} and the dual version of \cite{HA}*{Variant 1.2.4.9}, we have a convergent spectral sequence
\[
\rE_1^{p,q}=\rH^q (f_{-p!}y_{-p}^! \sfK)\Rightarrow \rH^{p+q} f_! \sfK.
\]
By induction hypothesis and the Poincar\'e duality (Theorem \ref{b6th:poincare}(2)), $\rE_1^{p,q}$ is constructible for all $p$ and $q$. Moreover, $\rE_1^{p,q}$ vanishes for $p>0$ or $q>2d$ by Lemma \ref{b6le:dimension}. Therefore, $f_!\sfK$ belongs to $\cD^{\le 2d}_\cons(X,\lambda)$.

For (3), we may assume $X$ is an affine scheme. The assertion is then trivial.
\end{proof}

To state the results for the other operations, we work in a relative setting. Let $\dS$ be a $\Box$-coprime higher Artin stack. Assume that there exists an atlas $S\to \dS$, where $S$ is a coproduct of Noetherian quasi-excellent schemes\footnote{Recall from \cite{TG}*{Expos\'{e} I, D\'{e}finition 2.10} that a ring is \emph{quasi-excellent} if it is Noetherian and satisfies conditions (2), (3) of \cite{EGAIV}*{D\'{e}finition 7.8.2}. A Noetherian scheme is \emph{quasi-excellent} if it admits a Zariski open cover by spectra of quasi-excellent rings.} and regular schemes of dimension $\le 1$. We denote by $\Chplft{\dS}\subseteq\Chpar{}_{/\dS}$ the full subcategory spanned by morphisms $X\to \dS$ locally of finite type.

\begin{proposition}\label{b6pr:constructible2}
Let $f\colon Y\to X$ be a morphism of $\Chplft{\dS}$, and $\lambda$ a $\Box$-torsion Noetherian ringed diagram. Then the operations introduced in \Sec\ref{b6ss:operations} restrict to the following
\begin{description}
  \item[1R']
      $f_*\colon\cD^{(+)}_\cons(Y,\lambda)\to\cD^{(+)}_\cons(X,\lambda)$
      if $f$ is quasi-compact and quasi-separated;

  \item[2R']
      $f^!\colon\cD^{(+)}_\cons(X,\lambda)\to\cD^{(+)}_\cons(Y,\lambda)$;

  \item[3R']
      $\HOM_X(-,-)\colon\cD^{(-)}_\cons(X,\lambda)^{op}\times\cD^{(+)}_\cons(X,\lambda)\to\cD^{(+)}_\cons(X,\lambda)$
      if $\Xi_{/\xi}$ is finite for all $\xi\in \Xi$.
\end{description}
\end{proposition}

\begin{proof}
Suppose that $\lambda=(\Xi,\Lambda)$. We first reduce to the case $\Xi=\{\ast\}$. The reduction follows from Proposition \ref{b6pr:upperstar_pi} and Proposition \ref{b6pr:lowersh_pi} for (1R') and (2R'), respectively. For (3R'), by Proposition \ref{b6pr:Hom_pi2}(2) and the assumption on $\Xi_{/\xi}$, we may assume $\Xi$ finite. In this case, by Proposition \ref{b6pr:hom_pi}(2), it suffices to prove that every $\sfK\in \cD^{(+)}_\cons(X,\lambda)$ is a successive extension of $e_{\xi*}\sfL_\xi$, where $\sfL_\xi\in\cD^{(+)}_\cons(X_\xi,\Lambda(\xi))$ for every object $\xi\in \Xi$. This being trivial for $\Xi=\emptyset$, we proceed by induction on the cardinality of $\Xi$. Let $\Xi'\subseteq \Xi$ be the partially ordered subset spanned by the minimal elements of $\Xi$, and let $\Xi''$ be the complement of $\Xi'$. Then we have a fibre sequence $i_*\sfL\to \sfK \to \prod_{\xi \in \Xi'}e_{\xi*}e_\xi^*\sfK$, where $i\colon (\Xi'',\Lambda\res \Xi'')\to \lambda$ and $\sfL\in \cD^{(+)}_\cons(\Xi'',\Lambda\res \Xi'')$. Since $\Xi'$ is nonempty, it then suffices to apply the induction hypothesis to $\sfL$.

We then prove by induction on $k$ that the assertions for $\Xi=\{\ast\}$ hold when $f$ is a morphism of $k$-Artin stacks. The case $k=-2$ is due to Deligne \cite{SGA4d}*{Th.\ Finitude, Corollaires 1.5, 1.6} if $S$ is regular of dimension $\le 1$ and to Gabber \cite{TG}*{Expos\'{e} XIII} if $S$ is quasi-excellent. In fact, in the latter case, by arguments similar to \cite{SGA4d}*{Th.\ Finitude, {\Sec}2.2}, we may assume $\lambda=(*,\dZ/n\dZ)$. In the finite-dimensional case we also need the finiteness of cohomological dimension recalled at the beginning of \Sec\ref{b6ss:adjointness}. Now assume that the assertions hold for some $k\ge-2$ and let $f$ be a morphism of $(k+1)$-Artin stacks. Then (2R') follows from induction hypothesis, Theorem \ref{b6th:poincare}(2) and (1L'); (3R') follows from induction hypothesis, Proposition \ref{b6pr:hom}(3), Theorem
\ref{b6th:poincare}(2) and (1L'), (2R'). The proof of (1R') is similar to the proof of Proposition \ref{b6pr:constructible}. Indeed, to show that for every object $\sfK$ of $\cD^{\geq0}_\cons(Y,\lambda)$, $f_*\sfK$ belongs to $\cD^{\geq0}_\cons(X,\lambda)$, it suffices to apply the convergent spectral sequence
\[
\rE_1^{p,q}=\rH^q(f_{p*}y_p^*\sfK)\Rightarrow \rH^{p+q} f_* \sfK
\]
and induction hypothesis.
\end{proof}

\begin{proposition}\label{b6pr:constructible3}
Let $f\colon Y\to X$ be a morphism of $\Chplft{\dS}$, and $\lambda$ a $\Box$-torsion Noetherian ringed diagram. Assume that $\dS$ is locally finite-dimensional. Then the operations introduced in \Sec\ref{b6ss:operations} restrict to the following
\begin{description}
  \item[1R'] $f_*\colon \cD_\cons(Y,\lambda)\to\cD_\cons(X,\lambda)$ if $f$ is quasi-compact, quasi-separated, and $0$-Artin;

  \item[2R'] $f^!\colon \cD_\cons(X,\lambda)\to\cD_\cons(Y,\lambda)$;

  \item[3R'] $\HOM_X(-,-)\colon\cD^{(\r{ft})}_\cons(X,\lambda)^{op}\times\cD_\cons(X,\lambda)\to\cD_\cons(X,\lambda)$ if $\Xi_{/\xi}$ is finite for all $\xi\in \Xi$.
\end{description}
\end{proposition}

Here $\cD^{(\r{ft})}_{\cons}(X,\lambda)\subseteq \cD_{\cons}(X,\lambda)$ denotes the full subcategory spanned by objects $\sfK$ such that for every $\xi\in \Xi$, $e_\xi^*\sfK$ is locally of finite tor-dimension.

\begin{proof}
This follows from Proposition \ref{b6pr:constructible2} and Lemmas \ref{6le:cohdim} and \ref{6le:HOMfin}.
\end{proof}

\begin{lem}\label{6le:HOMfin}
Let $X$ be a $\square$-coprime finite-dimensional Noetherian higher Artin stack. Let $\lambda=(\Xi,\Lambda)$ be a $\square$-torsion Noetherian ringed diagram with $\Xi$ finite and let $\sfK\in \cD_\cons(X,\lambda)$ such that for every $\xi\in \Xi$, $e_\xi^*\sfK$ is locally of finite tor-dimension. Then $\HOM_X(\sfK,-)$ has finite cohomological amplitude.
\end{lem}

\begin{proof}
As in the proof of Proposition \ref{b6pr:constructible2}, any $\sfL\in\cD^{\le 0}(X,\lambda)$ is a successive extension of $e_{\xi*}\sfL_\xi$ with $\sfL_\xi\in \cD^{\le l}(X,\Lambda(\xi))$, where $l$ denotes the greatest length of chains in $\Xi$. We are thus reduced to the case $\Xi=\{\ast\}$. We then reduce to the case where $X$ is a scheme and $\sfK=j_!\sfK'$, where
$j\colon U\to X$ is an immersion and $\sfK'\in \cD_\cons(U,\Lambda)$ is a perfect complex. In this case
\[
\HOM_X(\sfK,\sfL)\simeq j_*\HOM_X(\sfK',j^*\sfL)\simeq j_*(j^*\sfL\otimes \HOM_X(\sfK',\Lambda))
\]
and we conclude by the fact that $j_*$ has finite cohomological amplitude.
\end{proof}

\subsection{Compatibility with Laszlo--Olsson (torsion coefficients)}
\label{b6ss:compatibility}

In this section we establish the compatibility between our theory and the work of Laszlo and Olsson \cite{LO1}, under the (more restrictive) assumptions of the latter.

We fix $\Box=\{\ell\}$ and a Gorenstein local ring $\Lambda$ of dimension $0$ and residual characteristic $\ell$. We will suppress $\Lambda$ from the notation when no confusion arises. Let $\dS$ be a $\Box$-coprime scheme, endowed with a global dimension function, satisfying the following conditions.
\begin{enumerate}
  \item $\dS$ is affine excellent and finite-dimensional;

  \item For every $\dS$-scheme $X$ of finite type, there exists an \'{e}tale cover $X'\to X$ such that, for every scheme $Y$ \'{e}tale and of finite type over $X'$, $\r{cd}_{\ell}(Y)<\infty$;
\end{enumerate}

\begin{remark}\label{b6re:dimension}
In \cite{LO1}, the authors did not explicitly include the existence of a global dimension function in their assumptions. However, their method relies on pinned dualizing complexes (see below), which makes use of the dimension function. Note that assumption (2) above is slightly weaker than the assumption on cohomological dimension in \cite{LO1}; for example, (2) allows the case $\dS=\Spec\dR$ and $\ell=2$ while the assumption in \cite{LO1} does not. Nevertheless, assumption (2) implies that the right derived functor of the countable product functor on $\Mod(X_{\et},\Lambda)$ has finite cohomological dimension, which is in fact sufficient for the construction in \cite{LO1}.
\end{remark}

Let $\Chplmb{\dS}$ be the full subcategory of $\Chplft{\dS}$ spanned by ($1$-)Artin stacks locally of finite type over $S$, with quasi-compact and separated diagonal. Stacks with such diagonal are called algebraic stacks in \cite{LMB} and \cite{LO1}. We adopt the notation
$\rD_\cons(\cX_{\liset})\subseteq \rD_\cart(\cX_{\liset})$ from \Sec\ref{0ss:results_torsion}. For a morphism $f\colon \cY\to\cX$ of finite type (of
$\Chplmb{\dS}$), Laszlo--Olsson defined functors
\begin{align*}
\rR f_*&\colon \rD_\cons^{(+)}(\cY_\liset)\to\rD_\cons^{(+)}(\cX_{\liset}),\\
\rR f_!&\colon \rD_\cons^{(-)}(\cY_\liset)\to\rD_\cons^{(-)}(\cX_\liset),\\
\rL f^*&\colon \rD_\cons(\cX_{\liset})\to\rD_\cons(\cY_\liset),\\
\rR f^!&\colon \rD_\cons(\cX)\to\rD_\cons(\cY_\liset),\\
\Rhom_\cX&\colon \rD_\cons^{(-)}(\cX_\liset)^{op}\times\rD_\cons^{(+)}(\cX_\liset)\to\rD_\cons^{(+)}(\cX_\liset),\\
-\overset{\rL}{\otimes}_\cX -&\colon
\rD_\cons^{(-)}(\cX_\liset)\times\rD_\cons^{(-)}(\cX_\liset)\to\rD_\cons^{(-)}(\cX_\liset).
\end{align*}
Three of the six functors, $\rR f_*$, $\Rhom_\cX$, and $-\overset{\rL}{\otimes}_\cX -$, are standard functors for the lisse-\'etale topoi and can be extended to $\rD_\cart$ (see Remarks \ref{b6re:comparison} and \ref{b5re:dejong}):
\begin{align*}
\rR f_*&\colon \rD_\cart^{(+)}(\cY_\liset)\to\rD_\cart^{(+)}(\cX_\liset),\\
\Rhom_\cX&\colon \rD_\cart(\cX_\liset)^{op}\times\rD_\cart(\cX_\liset)\to\rD_\cart(\cX_\liset),\\
-\overset{\rL}{\otimes}_\cX -&\colon
\rD_\cart(\cX_\liset)\times\rD_\cart(\cX_\liset)\to\rD_\cart(\cX_\liset).
\end{align*}
Moreover, the construction of $\rL f^*$ in \cite{LO1}*{{\Sec}4.3} can also be extended to $\rD_\cart$:
\[
\rL f^*\colon \rD_\cart(\cX_\liset)\to\rD_\cart(\cY_\liset).
\]
In fact, it suffices to apply \cite{LO1}*{Theorem 2.2.3} to $\rD_\cart$. The six operations satisfy all the usual adjointness properties (cf.\
\cite{LO1}*{Propositions 4.3.1, 4.4.2}). On the other hand, restricting our constructions in the two previous sections, we have
\begin{align*}
f_*&\colon \cD^{(+)}(\cY)\to\cD^{(+)}(\cX),\\
f_!&\colon \cD_\cons^{(-)}(\cY)\to\cD_\cons^{(-)}(\cX),\\
f^*&\colon \cD(\cX)\to\cD(\cY),\\
f^!&\colon \cD_\cons(\cX)\to\cD_\cons(\cY),\\
\HOM_\cX &\colon \cD(\cX)^{op}\times\cD(\cX)\to\cD(\cX),\\
-\otimes_\cX-&\colon \cD(\cX)\times\cD(\cX)\to\cD(\cX).
\end{align*}
The equivalence of categories $\rh\cD(\cX)\simeq\rD_\cart(\cX_\liset)$ \eqref{b6eq:equivalence} restricts to an equivalence
$\rh\cD_\cons(\cX)\simeq\rD_\cons(\cX_\liset)$. The main result of this section is the following.

\begin{proposition}\label{b6pr:compatibility}
We have equivalences of functors
\begin{align*}
\rh f_*\simeq\rR f_*,\quad\rh f_!\simeq\rR f_!,\quad \rh f^*\simeq\rL f^*,\quad \rh f^!\simeq\rR f^!,\\
\rh\HOM_\cX\simeq\Rhom_\cX,\quad \rh(-\otimes_\cX-)\simeq(-\overset{\rL}{\otimes}_\cX-),
\end{align*}
compatible with \eqref{b6eq:equivalence}.
\end{proposition}

\begin{proof}
The assertions for $-\otimes_\cX -$ and $\HOM_\cX$ are special cases of Remark \ref{b6re:comparison}. Moreover, by adjunction, the assertion for $f_*$ (resp.\ $f_!$) will follow from the one for $f^*$ (resp.\ $f^!$).

Let us first prove that $\rh f^*\simeq\rL f^*\colon\rD_\cart(\cX_\liset)\to\rD_\cart(\cY_\liset)$. We choose a commutative diagram
\[
\xymatrix{
Y \ar[r] \ar[d] & X \ar[d] \\
\cY \ar[r]& \cX}
\]
where the vertical morphisms are atlases. It induces a 2-commutative diagram
\[
\xymatrix{
Y_\bullet \ar[r]^-{f_\bullet} \ar[d]_-{\eta_Y} & X_\bullet \ar[d]^-{\eta_X}  \\
\cY \ar[r]^-{f} & \cX.}
\]
Using arguments similar to \Sec\ref{b5ss:artin}, we get the following diagram
\[
\resizebox{\hsize}{!}{
\xymatrix{
\cD_\cart(\Mod(Y_{\bullet,\et})) \ar[rd] && \cD_\cart(\Mod(X_{\bullet,\et})) \ar[ll]_-{f_{\bullet,\et}^*} \ar[rd]  \\
& \lim_{n\in\del}\cD(Y_{n,\et}) && \lim_{n\in\del}\cD(X_{n,\et}) \ar[ll]_(.7){\lim_{n\in\del}f_{n,\et}^*}  \\
\cD_\cart(\cY_{\liset}) \ar[uu]^-{\eta_{Y,\cart}^*} \ar[ur]^-{\sim} && \cD_\cart(\cX_{\liset}).
\ar@{-->}[ll]_-{f^*} \ar[uu]_(.3){\eta_{X,\cart}^*}|\hole \ar[ur]^-{\sim}
}
}
\]
By \cite{LO1}*{Theorem 2.2.3}, $\eta_{X,\cart}^*$ and $\eta_{Y,\cart}^*$ are equivalences. By the construction of $\rL f^*$, $\rL f^*$ fits into a homotopy version of the rectangle in the above diagram. Therefore, we have an equivalence $\rh f^*\simeq \rL f^*$.

Let $\Omega_\dS\in\cD(\dS)$ be a potential dualizing complex (with respect to the fixed dimension function) in the sense of \cite{TG}*{Expos\'{e} XVII, D\'{e}finition 2.1.2}, which is unique up to isomorphism by \cite{TG}*{Expos\'{e} XVII, Th\'{e}or\`{e}me 5.1.1} (see Remark \ref{b6re:riou}). For every object $\cX$ of $\Chplmb{\dS}$, with structure morphism $a\colon\cX\to\dS$, we put $\Omega_\cX\coloneqq a^!\Omega_\dS$. Let $u\colon U\to\cX$ be an object of $\Liset(\cX)$. Then $u^*\Omega_\cX\simeq \Omega_U\langle-d\rangle$ by the Poincar\'e duality (Theorem \ref{b6th:poincare}(2)), where $d=\dim u$. Consider the morphism of topoi $(\epsilon_*,\epsilon^*)\colon(\cX_\liset)_{/\widetilde U} \to U_{\et}$. Applying Lemma \ref{b4le:Dcart}, we get an equivalence $\Omega_\cX\res (\cX_\liset)_{/\widetilde U}\simeq\epsilon^*\Omega_U\langle -d\rangle$, where we regard $\Omega_\cX$ as an object of $\cD_\cart(\cX_\liset)$ and $\Omega_U$ as an object of $\cD(U_{\et})$. The equivalence is compatible with restriction by morphisms of $\Liset(\cX)$, so that $\Omega_\cX$ is a dualizing complex of $\cX$ in the sense of \cite{LO1}*{Definition 3.4.5}, which is unique up to isomorphism by
\cite{LO1}*{Proposition 3.4.3, Lemma 3.4.4}. Put $\cD_\cX\coloneqq\HOM_\cX(-,\Omega_\cX)$ and $\rD_\cX\coloneqq\Rhom_\cX(-,\Omega_\cX)\simeq \rh\cD_\cX$. By \cite{LO1}*{Corollary 3.5.7}, the biduality functor $\id\to\rD_\cX\circ\rD_\cX$ is a natural isomorphism of endofunctors of $\rD_\cons(\cX_\liset)$. Therefore, the natural
transformation $\rh f^!\to \rh f^!\circ \rD_\cX\circ \rD_\cX$ is a natural equivalence when restricted to $ \rD_\cons(\cX_\liset)$. By Proposition \ref{b6pr:hom}(3), we have
\begin{align*}
f^!\circ \cD_\cX\circ \cD_\cX &\simeq f^!\HOM_\cX(\cD_\cX -,\Omega_\cX)\simeq\HOM_\cY(f^*\cD_\cX-,f^!\Omega_\cX)\\
&\simeq\HOM_\cY(f^*\cD_\cX-,\Omega_\cY)=\cD_\cY\circ f^*\circ \cD_\cX.
\end{align*}
Since $\rh f^*\simeq \rL f^*$, this shows
\[
\rh f^!\simeq\rD_\cY\circ\rL f^*\circ\rD_\cX=\rR f^!,
\]
where the last identity is the definition of $\rR f^!$ in \cite{LO1}*{Definition 4.4.1}.
\end{proof}

\begin{remark}\label{b6re:riou}
As Jo\"el Riou observed (private communication), although the definition, existence and uniqueness of potential dualizing complexes are only stated for the coefficient ring $R=\dZ/n\dZ$ in \cite{TG}*{Expos\'{e} XVII, D\'{e}finition 2.1.2, Th\'{e}or\`{e}me 5.1.1}, they can be extended to any Noetherian ring $R'$ over $R$. In fact, if $\delta$ is a dimension function of an excellent $\dZ[1/n]$-scheme $X$ and $K_R$ is a potential dualizing complex for $(X,\delta)$ relative to $R$, then $K_{R'}=K_R\overset{\rL}{\otimes}_R R'$ is a potential dualizing complex for $(X,\delta)$ relative to $R'$ by the projection formula $\rR\Gamma_x(K_R)\overset{\rL}{\otimes}_RR'\simeq\rR\Gamma_x(K_R\overset{\rL}{\otimes}_RR')$, where $x$ is a geometric point of $X$. The formula follows from the fact that the punctured strict localization of $X$ at $x$ has finite cohomological dimension \cite{TG}*{Expos\'{e} XVIII-A, Corollary 1.4}. Moreover, by the theorem of local biduality \cite{TG}*{Expos\'{e} XVII, Th\'{e}or\`{e}mes 6.1.1, 7.1.2}, $K_{R'}$ is a dualizing complex for $\rD^\rb_\cons(X_{\et},R')$ in the sense of \cite{TG}*{Expos\'{e} XVII, D\'{e}finition 7.1.1} as long as $R'$ is Gorenstein of dimension 0.
\end{remark}

\section{Adic formalism}
\label{c1}

In this chapter, we provide the adic formalism for Grothendieck's six 
operations. In \Sec\ref{c1ss:limit}, we provide our adic formalism by 
constructing two enhanced operation maps via the limit construction. In 
\Sec\ref{c1ss:adic_properties}, we study several properties of the enhanced 
operation maps we constructed previously. In \Sec\ref{c1ss:relation}, we 
study the relation between the limit construction and so-called adic 
complexes. In \Sec\ref{c1ss:constructible} and \Sec\ref{c1ss:adic_dualizing}, 
we study  constructible adic complexes and construct adic dualizing 
complexes. In \Sec\ref{c2ss:madic}, we study a special kind of ringed 
diagrams for which the adic formalism is the most satisfactory. This includes 
the most common application, namely, the $\ell$-adic one. The last section 
\Sec\ref{c2ss:compatibility} is dedicated to proving the compatibility 
between our theory and Laszlo--Olsson (\cite{LO2} and \cite{LO3}) under their 
restrictions.

\subsection{The limit construction}
\label{c1ss:limit}

Recall from \Sec\ref{b5ss:artin} that for higher Artin stacks, we construct the \emph{first enhanced operation map}
\begin{align*}
\EO{}{\Chpar{}}{\r{I}}{}\colon((\Chpar{})^{op}\times\rN(\Rind)^{op})^\amalg\to\Cat,
\end{align*}
and the \emph{second enhanced operation map}
\begin{align*}
\EO{}{\Chpar{}_\Box}{\r{II}}{}\colon \delta^*_{2,\{2\}}(((\Chpar{}_\Box)^{op}\times\rN(\Rind_{\ltor})^{op})^{\amalg,op})^\cart_{F,\all}\to\Cat.
\end{align*}
Their restrictions to the common domain $((\Chpar{}_\Box)^{op}\times\rN(\Rind_{\ltor})^{op})^\amalg$ are equivalent. In particular, for every object $X$ of $\Chpar{}$ and every object $\lambda=(\Xi,\Lambda)$ of $\Rind$, we obtain a diagram $\Xi^{op}\to\PSL$ given by $\xi\mapsto\cD(X,\Lambda(\xi))$ with the transition map given by extension of scalars.

\begin{definition}
We define the \emph{adic derived $\infty$-category} of $\lambda$-modules on $X$ to be
\[
\cD(X,\lambda)_\ra\coloneqq\varprojlim_{\rN(\Xi)^{op}}\cD(X,\Lambda(\xi)).
\]
\end{definition}

The goal of this section is to make the above definition functorial in a homotopy coherent way. Namely, we will construct the \emph{first enhanced adic operation map}
\begin{align}\label{c1eq:premonoidal_artin_adic}
\EO{\ra}{\Chpar{}}{\r{I}}{}\colon((\Chpar{})^{op}\times\rN(\Rind)^{op})^\amalg\to\Cat,
\end{align}
and the \emph{second enhanced adic operation map}
\begin{align}\label{c1eq:operation_artin_adic}
\EO{\ra}{\Chpar{}_\Box}{\r{II}}{}\colon \delta^*_{2,\{2\}}(((\Chpar{}_\Box)^{op}\times\rN(\Rind_{\ltor})^{op})^{\amalg,op})^\cart_{F,\all}\to\Cat,
\end{align}
such that their values on $(X,\lambda)$ are both (equivalent to) $\cD(X,\lambda)_\ra$.

By definition, there is a tautological functor $\Rind\to\cat$ sending $(\Xi,\Lambda)$ to $\Xi$. Applying Grothendieck's construction, we obtain an op-fibration $\pi\colon\Rind^\univ\to\Rind$. More precisely, $\Rind^\univ$ is an ordinary category whose objects are pairs $((\Xi,\Lambda),\xi)$ where $(\Xi,\Lambda)$ is an object of $\Rind$ and $\xi$ is an object of $\Xi$, and a morphism from $((\Xi,\Lambda),\xi)$ to $((\Xi',\Lambda'),\xi')$ is a morphism $(\Gamma,\gamma)\colon(\Xi,\Lambda)\to(\Xi',\Lambda')$ of $\Rind$ such that $\Gamma(\xi)$ admits an arrow to $\xi'$. We have another functor $\sigma\colon\Rind^\univ\to\Rind$ sending $((\Xi,\Lambda),\xi)$ to $(\ast,\Lambda(\xi))$. We have two natural inclusion
\begin{align*}
j_0&\colon\rN(\Rind)^{op}\to\rN(\Rind)^{op}\diamond_{\rN(\Rind)^{op}}\rN(\Rind^\univ)^{op},\\
j_1&\colon\rN(\Rind^\univ)^{op}\to\rN(\Rind)^{op}\diamond_{\rN(\Rind)^{op}}\rN(\Rind^\univ)^{op}
\end{align*}
of simplicial sets.

To construct \eqref{c1eq:premonoidal_artin_adic}, we start from the map
\[
\EO{\sigma}{\Chpar{}}{\r{I}}{}\colon((\Chpar{})^{op}\times\rN(\Rind^\univ)^{op})^\amalg\to\Cat
\]
as the composition of
\[
(\id_{(\Chpar{})^{op}}\times\rN(\sigma)^{op})^\amalg\colon
((\Chpar{})^{op}\times\rN(\Rind^\univ)^{op})^\amalg\to((\Chpar{})^{op}\times\rN(\Rind)^{op})^\amalg
\]
and $\EO{}{\Chpar{}}{\r{I}}{}$. Taking the right Kan extension of $\EO{\sigma}{\Chpar{}}{\r{I}}{}$ along the inclusion
\[
((\Chpar{})^{op}\times\rN(\Rind^\univ)^{op})^\amalg\hookrightarrow
((\Chpar{})^{op}\times\rN(\Rind)^{op}\diamond_{\rN(\Rind)^{op}}\rN(\Rind^\univ)^{op})^\amalg
\]
induced by $j_1$, and restricting to $((\Chpar{})^{op}\times\rN(\Rind)^{op})^\amalg$ via $j_0$, we obtain the desired map $\EO{\ra}{\Chpar{}}{\r{I}}{}$ \eqref{c1eq:premonoidal_artin_adic}.

The construction of \eqref{c1eq:operation_artin_adic} is similar. We have the map
\[
\EO{\sigma}{\Chpar{}_\Box}{\r{II}}{}\colon
\delta^*_{2,\{2\}}(((\Chpar{}_\Box)^{op}\times\rN(\Rind_{\ltor}^\univ)^{op})^{\amalg,op})^\cart_{F,\all}\to\Cat,
\]
where $\Rind_{\ltor}^\univ=\Rind^\univ\times_{\Rind}\Rind_{\ltor}$ in which the first functor in the fiber product is $\pi$. Taking the right Kan extension of $\EO{\sigma}{\Chpar{}_\Box}{\r{II}}{}$ along the inclusion
\begin{multline*}
\delta^*_{2,\{2\}}(((\Chpar{}_\Box)^{op}\times\rN(\Rind_{\ltor}^\univ)^{op})^{\amalg,op})^\cart_{F,\all}\\
\hookrightarrow\delta^*_{2,\{2\}}(((\Chpar{}_\Box)^{op}\times
\rN(\Rind_{\ltor})^{op}\diamond_{\rN(\Rind_{\ltor})^{op}}\rN(\Rind_{\ltor}^\univ)^{op})^{\amalg,op})^\cart_{F,\all}
\end{multline*}
induced by $j_1$, and restricting to $\delta^*_{2,\{2\}}(((\Chpar{}_\Box)^{op}\times\rN(\Rind_{\ltor})^{op})^{\amalg,op})^\cart_{F,\all}$ via $j_0$, we obtain the desired map $\EO{\ra}{\Chpar{}_\Box}{\r{II}}{}$ \eqref{c1eq:operation_artin_adic}.

By the similar process, we obtain enhanced adic operation maps for higher Deligne--Mumford stacks:
\begin{align*}
\EO{\ra}{\Chpdm{}}{\r{I}}{}\colon((\Chpdm{})^{op}\times\rN(\Rind)^{op})^\amalg\to\Cat,
\end{align*}
and a map
\begin{align*}
\EO{\ra}{\Chpdm{}}{\r{II}}{}\colon \delta^*_{2,\{2\}}(((\Chpdm{})^{op}\times\rN(\Rind_\tor)^{op})^{\amalg,op})^\cart_{F,\all}\to\Cat,
\end{align*}
satisfying the obvious compatibility properties with higher Artin stacks.

\subsection{Properties of enhanced adic operations}
\label{c1ss:adic_properties}

In this section, we study properties of the two enhanced adic operation maps constructed previously, in a way parallel to the non-adic ones in \Sec\ref{b5ss:artin}.

To simplify notation, we will only discuss properties for higher Artin stacks, that is, the two maps \eqref{c1eq:premonoidal_artin_adic} and \eqref{c1eq:operation_artin_adic}. We will leave the analogous discussion for higher DM stacks to readers.

\begin{proposition}\label{c1pr:monoidal}
We have
\begin{description}
  \item[(P0)] (Monoidal symmetry) The functor $\EO{\ra}{\Chpar{}}{\r{I}}{}$ 
      is a lax Cartesian structure (Remark \ref{b1re:cartesian}), and the 
      induced functor 
      $\EO{\ra}{\Chpar{}}\otimes{}\coloneqq(\EO{\ra}{\Chpar{}}{\r{I}}{})^\otimes$ 
      factorizes through $\PSLM$. 

  \item[(P1)] (Disjointness) The map $\EO{\ra}{\Chpar{}}\otimes{}$ sends small coproducts to products.

  \item[(P2)] (Compatibility) The restrictions of $\EO{\ra}{\Chpar{}}{\r{I}}{}$ and $\EO{\ra}{\Chpar{}_\Box}{\r{II}}{}$ to the common domain $((\Chpar{}_\Box)^{op}\times\rN(\Rind_{\ltor})^{op})^\amalg$ are equivalent functors.
\end{description}
\end{proposition}

\begin{proof}
By construction, the value of $\EO{\ra}{\Chpar{}}{\r{I}}{}$ on an object $((X_1,\lambda_1),\dots,(X_n,\lambda_n))$ in the target is an $\infty$-category equivalent to
\[
\prod_{i=1}^n\cD(X_i,\lambda_i)_\ra=\prod_{i=1}^n\varprojlim_{\Xi_i^{op}}\cD(X_i,\Lambda_i(\xi))
\]
if $\lambda_i=(\Xi_i,\Lambda_i)$. We also note that the inclusion functor
\[
\PSLM\to\CAlg(\Cat)
\]
preserves small limits. Therefore, (P0) and (P1) follow immediately. (P2) is clear from the construction.
\end{proof}

Before discussing the other properties, we introduce more notation. Similar to the non-adic case, we have the map
\begin{align}\label{c1eq:operation1}
\EO{\ra}{\Chpar{}_\Box}{*}{!}\colon \delta^*_{2,\{2\}}\rN(\Chpar{}_\Box)^\cart_{F,\all}\times\rN(\Rind_{\ltor})^{op}\to\PSL
\end{align}
induced from \eqref{c1eq:operation_artin_adic}.

Evaluating \eqref{c1eq:premonoidal_artin_adic} at the object $\langle 1\rangle\in\Fin$, we obtain the map
\begin{align}\label{c1eq:upperstar_adic}
\EO{\ra}{\Chpar{}}{*}{}\colon \rN(\Chpar{})^{op}\times\rN(\Rind)^{op}\to\PSL.
\end{align}
Note that this is equivalent to the map by restricting \eqref{c1eq:operation1} to the second direction, on $\rN(\Chpar{})^{op}\times\rN(\Rind_{\ltor})^{op}$. Taking right adjoints, we obtain the map
\begin{align}\label{c1eq:lowerstar_adic}
\EO{\ra}{\Chpar{}}{}{*}\colon \rN(\Chpar{})\times\rN(\Rind)\to\PSR.
\end{align}
Restricting \eqref{c1eq:operation1} to the first direction, we obtain the map
\begin{align}\label{c1eq:lowersh_adic}
\EO{\ra}{\Chpar{}_\Box}{}{!}\colon \rN(\Chpar{}_\Box)_F\times\rN(\Rind_{\ltor})^{op}\to\PSL.
\end{align}
Again by taking right adjoints, we obtain the map
\begin{align}\label{c1eq:uppersh_adic}
\EO{\ra}{\Chpar{}_\Box}{!}{}\colon \rN(\Chpar{}_\Box)^{op}_F\times\rN(\Rind_{\ltor})\to\PSR.
\end{align}

More concretely, we have the following enhance adic operations:
\begin{description}
  \item[1L] $f^{*\ra}\colon\cD(X,\lambda)_\ra\to\cD(Y,\lambda)_\ra$, obtained by applying \eqref{c1eq:upperstar_adic} to a morphism $f\colon Y\to X$ in $\Chpar{}$ and an object $\lambda=(\Xi,\lambda)\in\Rind$. It coincides with the limit of functors $f_\xi^*\colon\cD(X,\Lambda(\xi))\to\cD(Y,\Lambda(\xi))$ over $\Xi^{op}$, and underlies a monoidal functor $f^{*\otimes\ra}\colon\cD(X,\lambda)^\otimes_\ra\to\cD(Y,\lambda)^\otimes_\ra$ obtained from $\EO{\ra}{\Chpar{}}\otimes{}$.

  \item[1R] $f_{*\ra}\colon\cD(Y,\lambda)_\ra\to\cD(X,\lambda)_\ra$, obtained by applying \eqref{c1eq:lowerstar_adic} to a morphism $f\colon Y\to X$ in $\Chpar{}$ and an object $\lambda\in\Rind$. It is right adjoint to $f^{*\ra}$.

  \item[2L] $f_{!\ra}\colon\cD(Y,\lambda)_\ra\to\cD(X,\lambda)_\ra$, obtained by applying \eqref{c1eq:lowersh_adic} to a morphism $f\colon Y\to X$ in $\Chpar{}_\Box$ and an object $\lambda=(\Xi,\Lambda)\in\Rind_{\ltor}$. It coincides with the limit of functors $f_{\xi!}\colon\cD(Y,\Lambda(\xi))\to\cD(X,\Lambda(\xi))$ over $\Xi^{op}$.

  \item[2R] $f^{!\ra}\colon\cD(X,\lambda)_\ra\to\cD(Y,\lambda)_\ra$, obtained by applying \eqref{c1eq:uppersh_adic} to a morphism $f\colon Y\to X$ in $\Chpar{}_\Box$ and an object $\lambda\in\Rind_{\ltor}$. It is right adjoint to $f_{!\ra}$.

  \item[3L] $-\atimes_X-\colon\cD(X,\lambda)_\ra\times\cD(X,\lambda)_\ra\to\cD(X,\lambda)_\ra$, the symmetric tensor product obtained from Proposition \ref{c1pr:monoidal}(P0) for every object $(X,\lambda)$ of $\Chpar{}\times\rN(\Rind)$.

  \item[3R] 
      $\HOM^\ra_X\colon\cD(X,\lambda)_\ra^{op}\times\cD(X,\lambda)_\ra\to\cD(X,\lambda)_\ra$, 
      induced from $-\atimes_X-$ in the same way as $\HOM_X$ was induced 
      from $-\otimes_X-$ in \Sec\ref{b6ss:operations}. In particular, for 
      every object $\sfK\in\cD(X,\lambda)_\ra$, we have a pair of adjoint 
      functors $(-\atimes_X\sfK,\HOM^\ra_X(\sfK,-))$. 

  \item[4L] $\pi^{*\ra}\colon\cD(X,\lambda)_\ra\to\cD(X,\lambda')_\ra$, obtained by applying \eqref{c1eq:upperstar_adic} to an object $X\in\Chpar{}$ and a morphism $\pi\colon\lambda'\to\lambda$ of $\Rind$. It is symmetric monoidal.

  \item[4R] $\pi_{*\ra}\colon\cD(X,\lambda')_\ra\to\cD(X,\lambda)_\ra$, which is a right adjoint of $\pi^*$.
\end{description}

\begin{proposition}\label{c1pr:p34}
Let $f\colon Y\to X$ be a morphism of $\Chpar{}$ and $\lambda$ an object of $\Rind$.
\begin{description}
  \item[(P3)] (Conservativeness) If $f$ is surjective, then $f^{*\ra}\colon\cD(X,\lambda)_\ra\to\cD(Y,\lambda)_\ra$ is conservative.

  \item[(P4)] (Descent) Suppose that $f$ is smooth surjective. Then $(f,\id_\lambda)$ is of universal $\EO{\ra}{\Chpar{}}\otimes{}$-descent. If $X$ belongs to $\Chpar{}_\Box$ and and $\lambda$ belongs to $\Rind_{\ltor}$, then $(f,\id_\lambda)$ is of universal $\EO{\ra}{\Chpar{}}{}{!}$-codescent. See Definition \ref{b3de:descent} for the definition of (co)descent.
\end{description}
\end{proposition}

\begin{proof}
(P3) follows from the construction and the fact that
\[
\varprojlim_{\rN(\Xi)^{op}}\cD(X,\Lambda(\xi))\to\varprojlim_{\rN(\Xi)^{op}}\cD(Y,\Lambda(\xi))
\]
is conservative if each functor $\cD(X,\Lambda(\xi))\to\cD(Y,\Lambda(\xi))$ is, where $\lambda=(\Xi,\Lambda)$. The latter is true as $f$ is surjective.

Now we consider (P4). The universal descent property for 
$\EO{\ra}{\Chpar{}}\otimes{}$ follows from the construction, the same 
property in the non-adic case, and (the dual version of) 
\cite{HTT}*{Proposition 4.3.2.9}. The universal codescent property for 
$\EO{\ra}{\Chpar{}}{}{!}$ follows from the construction, the same property in 
the non-adic case, and \cite{HA}*{Proposition 4.7.4.19}. Note that condition 
(c) in \cite{HA}*{Proposition 4.7.4.19} is fulfilled by the Poincar\'{e} 
duality, namely, Theorem \ref{b6th:poincare}. 
\end{proof}

\begin{proposition}[\textbf{(P5)} Smooth Base Change]
Let
\[
\xymatrix{
W \ar[d]_-{q} \ar[r]^-{g} & Z \ar[d]^-{p} \\
Y \ar[r]^-{f} & X   }
\]
be a Cartesian diagram in $\Chpar{}_\Box$ where $p$ is smooth. Then for every object $\lambda$ of $\Rind_{\ltor}$, the following square
\[
\xymatrix{
\cD(W,\lambda)_\ra   & \cD(Z,\lambda)_\ra \ar[l]_-{g^{*\ra}}  \\
\cD(Y,\lambda)_\ra \ar[u]^-{q^{*\ra}} & \cD(X,\lambda)_\ra \ar[l]_-{f^{*\ra}}\ar[u]_-{p^{*\ra}}   }
\]
is right adjointable.
\end{proposition}

\begin{proof}
This follows from the construction, the same property in the non-adic case, and Lemma \ref{b4le:adjoint_limit}.
\end{proof}

Now we consider the usual t-structure on $\cD(X,\lambda)$ for an object 
$(X,\lambda)\in\Chpar{}\times\rN(\Rind)$. Recall from \cite{HA}*{Definition 
1.4.4.12} that, for a presentable stable $\infty$-category $\cD$, a 
t-structure\footnote{As before, we use a \emph{cohomological} indexing 
convention, which is different from \cite{HA}*{Definition 1.2.1.4}.} is 
\emph{accessible} if the full subcategory $\cD^{\leq0}$ is presentable. For a 
scheme $X\in\Schqcs$, the usual t-structure on $\cD(X,\lambda)$ is accessible 
by \cite{HA}*{Proposition 1.3.5.21}. For a higher Artin stack $X$, the usual 
t-structure on $\cD(X,\lambda)$ is accessible by construction Lemma 
\ref{b4le:p6}(3). 

Suppose that $\lambda=(\Xi,\Lambda)$. For $n\in\dZ$, we let $\cD^{\leq n}(X,\lambda)_\ra$ be the full subcategory of $\cD(X,\lambda)_\ra$ spanned by objects $\sfK=(\sfK_\xi)_{\xi\in\Xi}$ with $\sfK_\xi\in\cD^{\leq n}(X,\Lambda(\xi))$. Put
\[
\cD^{\geq n}(X,\lambda)_\ra\coloneqq\cD^{\leq n-1}(X,\lambda)^{\perp}_\ra
\]
as a full subcategory of $\cD(X,\lambda)_\ra$. By Lemma \ref{b3le:limit_restriction}, we have an equivalence
\[
\cD^{\leq n}(X,\lambda)_\ra\simeq\varprojlim_{\rN(\Xi)^{op}}\cD^{\leq n}(Y,\Lambda(\xi)).
\]
Here, we have used the fact that transition functors, which are (derived) extension of scalars, are left exact. In particular, $\cD^{\leq n}(X,\lambda)_\ra$ is presentable; the inclusion $\cD^{\leq n}(X,\lambda)_\ra\subseteq\cD(X,\lambda)$ preserves all small colimits; and $\cD^{\leq n}(X,\lambda)_\ra$ is closed under extension. By \cite{HA}*{Proposition 1.4.4.11(1)}, the pair $(\cD^{\leq n}(X,\lambda)_\ra,\cD^{\geq n}(X,\lambda)_\ra)$ define an accessible t-structure, called the \emph{usual t-structure}, on $\cD(X,\lambda)_\ra$. We have truncation functors
\[
\tau_\ra^{\leq n}\colon\cD(X,\lambda)_\ra\to\cD^{\leq n}(X,\lambda)_\ra,\quad
\tau_\ra^{\geq n}\colon\cD(X,\lambda)_\ra\to\cD^{\geq n}(X,\lambda)_\ra
\]
for every $n\in\dZ$.

\begin{remark}[P6]\label{c1re:p6}
We have the following remarks concerning the above t-structure.
\begin{enumerate}
  \item The constant sheaf $\lambda_X\coloneqq(\Lambda(\xi)_X)_{\xi\in\Xi}\in\cD(X,\lambda)_\ra$ belongs to the heart
      \[
      \cD^\heartsuit(X,\lambda)_\ra\coloneqq\cD^{\leq 0}(X,\lambda)_\ra\cap\cD^{\geq 0}(X,\lambda)_\ra
      \]
      by Lemma \ref{c1le:p6} below.

  \item For an object $\lambda$ of $\Rind_{\ltor}$ and an object $X$ of $\Chpar{}_\Box$, the auto-equivalence $-\otimes\lambda_X(1)$ of $\cD(X,\lambda)_\ra$ is t-exact.

  \item  The usual t-structure on $\cD(X,\lambda)_\ra$ is accessible. 
      Moreover, the intersection 
      $\cD^{\leq-\infty}(X,\lambda)_\ra=\bigcap_{n}\cD^{\leq-n}(X,\lambda)_\ra$ 
      consists of zero objects.\footnote{Unlike the non-adic case, 
      $\cD(X,\lambda)_\ra$ is not right complete in general. See Example 
      \ref{7ex:notrightcomplete} below. See also Corollary \ref{7co:adict} 
      below for a positive result.} 

  \item The functors $f^{*\ra}$, $-\atimes_X-$, $\pi^{*\ra}$ are all left t-exact (that is, preserve $\cD^{\leq n}$). The functors $f_{*\ra}$, $\HOM^\ra_X$, $\pi_{*\ra}$ are all right t-exact (that is, preserve $\cD^{\geq n}$).

  \item It follows from Lemma \ref{b6le:dimension} that $f_{!\ra}[2d]$ is 
      left t-exact, hence $f^{!\ra}[-2d]$ is right t-exact. In particular, 
      if $f$ is a smooth morphism in $\Chpar{}_\Box$ and $\lambda$ is in 
      $\Rind_{\ltor}$, then $f^{*\ra}\simeq f^{!\ra}[-2d]$ is t-exact.
\end{enumerate}
\end{remark}

\begin{lem}\label{c1le:p6}
Let $n\in\dZ$ and let $\cD(X,\lambda)_\ra^{\geq n}$ be the full subcategory  
of $\cD(X,\lambda)_\ra$ spanned by objects $\sfK=(\sfK_\xi)_{\xi\in\Xi}$ with 
$\sfK_\xi\in\cD^{\geq n}(X,\Lambda(\xi))$. Then $\cD(X,\lambda)_\ra^{\geq 
n}\subseteq \cD^{\geq n}(X,\lambda)_\ra$. 
\end{lem}

\begin{proof}
Let $\sfK'\in \cD^{\leq n-1}(X,\lambda)_\ra$ and $\sfK\in\cD(X,\lambda)_\ra^{\geq n}$. Then $\sfK\simeq \varprojlim_{\rN(\Xi)^{op}}r_\xi \sfK_\xi$, where $r_\xi\colon \cD(X,\Lambda(\xi))\to\cD(X,\lambda)_\ra$ is a right adjoint to the projection $\cD(X,\lambda)_\ra\to\cD(X,\Lambda(\xi))$. We have
\[
\Hom_{\rh\cD(X,\lambda)_\ra}(\sfK',r_\xi\sfK_\xi)\simeq \Hom_{\rh\cD(X,\lambda(\xi))}(\sfK'_\xi,\sfK_\xi)=0.
\]
It follows that $\Hom_{\rh\cD(X,\lambda)_\ra}(\sfK',\sfK)=0$.
\end{proof}

The functor $-\langle d\rangle\colon\cD(X,\lambda)\to\cD(X,\lambda)$ from \Sec\ref{b4ss:description} Input II restricts to a functor
\[
-\langle d\rangle\colon\cD(X,\lambda)_\ra\to\cD(X,\lambda)_\ra
\]
for every integer $d$. The proof of the theorem below will be given in the next section, after we introduce adic complexes.

\begin{theorem}[\textbf{(P7)} Poincar\'{e} duality]\label{c1th:poincare}
Let $f\colon Y\to X$ be a morphism of $\Chpar{}_\Box$ that is flat and locally of finite presentation. Let $\lambda$ be an object of $\Rind_{\ltor}$. Then
\begin{enumerate}
  \item There is a functorial (in the sense of Remark \ref{b4re:trace}) trace map
      \[
      \Tr_f\colon\tau_\ra^{\geq0}f_{!\ra}\lambda_Y\langle d\rangle\to\lambda_X
      \]
      in the heart $\cD^\heartsuit(X,\lambda)_\ra$ for every integer $d\geq\dim^+(f)$.

  \item If $f$ is moreover smooth, the induced natural transformation
      \[
      u_f\colon f_{!\ra}\circ f^{*\ra}\langle\dim f\rangle\to\id_X
      \]
      is a counit transformation, so that the induced map
      \[
      f^{*\ra}\langle\dim f\rangle\to f^{!\ra}\colon \cD(X,\lambda)_\ra\to\cD(Y,\lambda)_\ra
      \]
      is a natural equivalence of functors.
\end{enumerate}
\end{theorem}

We summarize some other properties of enhanced adic operations in the following theorem.

\begin{theorem}\label{c1th:properties}
We have
\begin{enumerate}
  \item The \emph{K\"{u}nneth Formula}, namely Theorem \ref{b6th:kunneth}, holds in the adic case.

  \item The \emph{Base Change}, namely Corollary \ref{b6co:base_change}, holds in the adic case.

  \item The \emph{Projection Formula}, namely Corollary \ref{b6co:projection}, holds in the adic case.

  \item The following statements hold in the adic case as well: Proposition \ref{b6pr:hom}, Proposition \ref{b6pr:hom_pi}, Proposition \ref{b6pr:Hom_pi2}, Proposition \ref{b6pr:support}, Theorem \ref{b6th:descent}, Corollary \ref{b6co:descent}, and Lemma \ref{b6le:dimension}.
\end{enumerate}
\end{theorem}

\begin{proof}
The properties follow by the same proofs in their non-adic counterparts.
\end{proof}

\subsection{Relation with adic complexes}
\label{c1ss:relation}

In this section, we define a natural full subcategory $\cD(X,\lambda)'_\ra$ of $\cD(X,\lambda)$ consisting of \emph{adic complexes} and show that there is a canonical equivalence $\cD(X,\lambda)'_\ra\simeq\cD(X,\lambda)_\ra$ of $\infty$-categories.

Let $X$ be an object of $\Chpar{}$, and $\lambda=(\Xi,\Lambda)$ an object of $\Rind$. For every morphism $\varphi\colon \xi\to\xi'$ in $\Xi$, there is a commutative diagram in $\Rind$ of the form
\[
\xymatrix{
(\Xi,\Lambda) \ar@{=}[d] & (\Xi_{/\xi},\Lambda_{/\xi})
\ar[d]^-{i_{\varphi}} \ar[l]_-{i_{\xi}} \ar[r]^-{p_{\xi}}
& (\{\xi\},\Lambda(\xi)) \ar[d]^-{\tilde{\varphi}} \\
(\Xi,\Lambda) & (\Xi_{/\xi'},\Lambda_{/\xi'}) \ar[l]_-{i_{\xi'}} \ar[r]^-{p_{\xi'}} & (\{\xi'\},\Lambda(\xi')),   }
\]
which induces the following diagram in $\PRL$:
\begin{align}\label{c1eq:adic}
\xymatrix{
\cD(X,\lambda) \ar[r]^-{i_{\xi}^*} & \cD(X,\lambda_{/\xi})
& \cD(X,\Lambda(\xi)) \ar[l]_-{p_{\xi}^*} \\
\cD(X,\lambda) \ar[r]^-{i_{\xi'}^*} \ar@{=}[u] & \cD(X,\lambda_{/\xi'})
\ar[u]_-{i_{\varphi}^*}   & \cD(X,\Lambda(\xi')) \ar[l]_-{p_{\xi'}^*} \ar[u]_-{\tilde{\varphi}^*},  }
\end{align}
where $\lambda_{/\xi}\coloneqq(\Xi_{/\xi},\Lambda_{/\xi})$. Let $p_{\xi*}$ (resp.\ $p_{\xi'*}$) be a right adjoint of $p_{\xi}^*$ (resp.\ $p_{\xi'}^*$) and let $\alpha_{\varphi}\colon\tilde{\varphi}^*p_{\xi'*}\to p_{\xi*}i_{\varphi}^*$ be the natural transformation.

\begin{definition}[Adic complex]\label{c1de:adic_object}
We say that an element $\sfK\in\cD(X,\lambda)$ is an \emph{adic complex} if the natural morphism
\[
\alpha_{\varphi}(i_{\xi'}^*\sfK)\colon\tilde{\varphi}^*p_{\xi'*}i_{\xi'}^*\sfK\to p_{\xi*}i_{\varphi}^*i_{\xi'}^*\sfK
\]
is an equivalence for every morphism $\varphi\colon \xi\to\xi'$ in $\Xi$. The target of $\alpha_{\varphi}(i_{\xi'}^*\sfK)$ is equivalent to
$p_{\xi*}i_{\xi}^*\sfK$. It is clear that adic complexes are stable under equivalence. Denote by
\[
\cD(X,\lambda)'_\ra\subseteq\cD(X,\lambda)
\]
the full subcategory spanned by adic complexes.
\end{definition}

\begin{lem}\label{c1le:evaluation}
Let $f\colon Y\to X$ be a morphism in $\Chpar{}$. If $\sfK$ is an adic complex in $\cD(X,\lambda)$, then $f^*\sfK$ is also an adic complex in $\cD(Y,\lambda)$. If $f$ is surjective, then the converse holds as well.
\end{lem}

\begin{proof}
The first statement follows if we can show that the following diagram
\begin{align}\label{c1eq:evaluation}
\xymatrix{
\cD(X,\lambda_{/\xi}) \ar[d]_-{f^*} & \cD(X,\Lambda(\xi)) \ar[l]_-{p_{\xi}^*}\ar[d]^-{f^*} \\
\cD(Y,\lambda_{/\xi})  & \cD(Y,\Lambda(\xi)) \ar[l]_-{p_{\xi}^*} \\
}
\end{align}
is right adjointable. By the construction of $\EO{}{\Chpar{}}{\r{I}}{}$ and Lemma \ref{b4le:adjoint_limit}, we may assume that $f$ is a morphism in $\Schqcs$. Then the following diagram
\[
\xymatrix{
\Mod(X^{\Xi_{/\xi}}_{\et},\Lambda_{/\xi})  \ar[d]_-{f^*} & \Mod(X_{\et},\Lambda(\xi)) \ar[l]_-{p_{\xi}^*}\ar[d]^-{f^*}  \\
\Mod(Y^{\Xi_{/\xi}}_{\et},\Lambda_{/\xi})  & \Mod(Y_{\et},\Lambda(\xi)) \ar[l]_-{p_{\xi}^*}
}
\]
has a right adjoint, which is
\[
\xymatrix{
\Mod(X^{\Xi_{/\xi}}_{\et},\Lambda_{/\xi})  \ar[r]^-{s_{\xi}^*}\ar[d]_-{f^*} & \Mod(X_{\et},\Lambda(\xi))\ar[d]^-{f^*} \\
\Mod(Y^{\Xi_{/\xi}}_{\et},\Lambda_{/\xi})  \ar[r]^-{s_{\xi}^*} & \Mod(Y_{\et},\Lambda(\xi))
}
\]
where $s_\xi\colon\{\xi\}\to\Xi_{/\xi}$ is the inclusion map. Thus, \eqref{c1eq:evaluation} is right adjointable.

The second statement follows from the first one and property (P3) for $\EO{}{\Chpar{}}{\r{I}}{}$.
\end{proof}

In general, if $\lambda=(\Xi,\Lambda)$ is an object of $\Rind$ and $\xi\in\Xi$, then we have successive inclusions
\[
e_{\xi}\colon(\{\xi\},\Lambda(\xi))\xrightarrow{s_{\xi}}(\Xi_{/\xi},\Lambda_{/\xi})\xrightarrow{i_{\xi}}(\Xi,\Lambda)
\]
which induce the \emph{evaluation functor (at $\xi$)}
\[
e_{\xi}^*\colon\cD(X,\lambda)\to\cD(X,\Lambda(\xi))
\]
for a higher Artin stack $X$. As $s^*_{\xi}$ is equivalent to $p_{\xi*}$, 
$e_{\xi}^*$ and $p_{\xi*}\circ i_{\xi}^*$ are equivalent. For brevity, we 
sometimes also write $\sfK_\xi$ for $e_{\xi}^*\sfK$ for an object 
$\sfK\in\cD(X,\lambda)$. 

The functor
\[
\prod_{\xi\in\Xi}e_{\xi}^*\colon\cD(X,\lambda)\to\prod_{\xi\in\Xi}\cD(X,\Lambda(\xi))
\]
is conservative. This is obvious when $X$ is in $\Schqcs$. The general case follows, because simplicial limits of conservative functors are
conservative.

\begin{lem}\label{c1le:adic}
Suppose that $\Xi$ admits a final object $\xi$. Then the functor 
$p_{\xi}^*\colon\cD(X,\Lambda(\xi))\to\cD(X,\lambda)$ is fully faithful with 
essential image $\cD(X,\lambda)'_\ra$.
\end{lem}

\begin{proof}
The fact that the image of the functor $p_{\xi}^*$ is contained in 
$\cD(X,\lambda)'_\ra$ follows from Definition \ref{c1de:adic_object} and the 
natural isomorphism between $p_{\xi'*}$ and $s_{\xi'}^*$ as in 
\eqref{c1eq:adic} for an arbitrary object $\xi'$ of $\Xi$. 

To conclude, we only need to show that for every adic complex 
$\sfK\in\cD(X,\lambda)'_\ra$, the adjunction map 
$p_{\xi}^*p_{\xi*}\sfK\to\sfK$ is an equivalence. Since the functor 
$\prod_{\xi'\in\Xi}e_{\xi'}^*$ is conservative, this is equivalent to showing 
that the map $\beta\colon e_{\xi'}^*p_{\xi}^*p_{\xi*}\sfK\to e_{\xi'}^*\sfK$ 
is an equivalence for every object $\xi'\in\Xi$. Let $\varphi$ be the map 
$\xi'\to\xi$. Since $\sfK$ is an adic complex, the composite 
\[
{\tilde\varphi}^*p_{\xi*}\sfK\xrightarrow{\alpha}p_{\xi'*}p_{\xi'}^*{\tilde\varphi}^*p_{\xi*}\sfK\simeq p_{\xi'*}i_\varphi^*p_\xi^*p_{\xi*}\sfK\xrightarrow{\beta}p_{\xi'*}i_\varphi^*\sfK
\]
is an equivalence, where we adopt the notation in \eqref{c1eq:adic}. Moreover, we have shown that $\alpha$ is an equivalence as $p_{\xi'*}\simeq s_{\xi}^*$. Therefore, $\beta$ is an equivalence.
\end{proof}

\begin{proposition}\label{c1pr:adic_inclusion}
The inclusion $\cD(X,\lambda)'_\ra\to\cD(X,\lambda)$ is a morphism in $\PRL$.
\end{proposition}

\begin{proof}
By definition, the inclusion $\cD(X,\lambda)'_\ra\subseteq\cD(X,\lambda)$ fits into the following diagram
\[
\xymatrix{
\cD(X,\lambda)'_\ra \ar[rr]\ar[d] && \prod_{\xi\in\Xi}\cD(X,\lambda_{/\xi})'_\ra \ar[d] \\
\cD(X,\lambda) \ar[rr]^-{\prod_{\xi\in\Xi}i_{\xi}^*} &&
\prod_{\xi\in\Xi}\cD(X,\lambda_{/\xi}), }
\]
which is a pullback diagram in $\Cat$ by Lemma \ref{c1le:pullback} below. By 
Lemma \ref{c1le:adic},  the inclusion $\cD(X,\lambda_{/\xi})'_\ra\to 
\cD(X,\lambda_{/\xi})$ is equivalent to $p^*_{\xi}$, which is a morphism of 
$\PRL$. Therefore, the right vertical arrow is a morphism in $\PRL$ as $\Xi$ 
is small. Moreover, the functor $\prod_{\xi\in\Xi}i_{\xi}^*$ preserves small 
colimits since each $i_{\xi}^*$ does and $\Xi$ is small. Therefore, the 
inclusion $\cD(X,\lambda)'_\ra\to\cD(X,\lambda)$ is a morphism in $\PRL$, 
because the inclusion $\PRL\subseteq\Cat$ preserves small limits. 
\end{proof}

\begin{lem}\label{c1le:pullback}
Let $\cD$ be a full subcategory of an $\infty$-category $\cC$ and $f\colon\cD\to\cC$ be the inclusion. Then the pullback of $f$ in the category $\Sset$ by any functor $g\colon \cC'\to \cC$ with source in $\Cat$ is a pullback in $\Cat$.
\end{lem}

\begin{proof}
This follows immediately from Lemma \ref{b3le:limit_restriction} applied to the pullback of $\id_\cC$ by $g$.
\end{proof}

Next, we will construct a natural functor
\begin{align}\label{c1eq:limit}
\cD(X,\lambda)'_\ra\to\cD(X,\lambda)_\ra=\varprojlim_{\rN(\Xi)^{op}}\cD(X,\Lambda(\xi))
\end{align}
and show that it is an equivalence. We have a diagram $(\Xi^{op})^\triangleleft\to\Rind$ sending $\xi\in\Xi$ to $\lambda_{/\xi}$ and the left vertex to $\lambda$, which gives rise to a functor
\[
\cD(X,\lambda)\to\varprojlim_{\rN(\Xi)^{op}}\cD(X,\lambda_{/\xi}).
\]
It is clear that for every object $\xi\in\Xi$, $i_\xi^*$ sends $\cD(X,\lambda)'_\ra$ to $\cD(X,\lambda_{/\xi})'_\ra$; and for every morphism $\varphi\colon\xi\to\xi'$ in $\Xi$, $i_\varphi^*$ sends $\cD(X,\lambda_{/\xi'})'_\ra$ to $\cD(X,\lambda_{/\xi})'_\ra$. Therefore, by restricting to full subcategories, we obtain a functor
\[
\cD(X,\lambda)'_\ra\to\varprojlim_{\rN(\Xi)^{op}}\cD(X,\lambda_{/\xi})'_\ra.
\]
By Lemma \ref{c1le:adic}, the right-hand side is equivalent to
\[
\varprojlim_{\rN(\Xi)^{op}}\cD(X,\Lambda(\xi))=\cD(X,\lambda)_\ra.
\]
Thus, we obtain the desired functor \eqref{c1eq:limit}.

\begin{theorem}\label{c1th:limit}
For objects $X$ of $\Chpar{}$ and $\lambda=(\Xi,\Lambda)$ of $\Rind$, the 
functor 
\[
\cD(X,\lambda)'_\ra\to\cD(X,\lambda)_\ra=\varprojlim_{\rN(\Xi)^{op}}\cD(X,\Lambda(\xi))
\]
\eqref{c1eq:limit} is an equivalence of $\infty$-categories.
\end{theorem}

We need some preparation before the proof. Let $X$ be an object of $\Schqcs$. For simplicity, we will write $X$ for $X_{\et}$ as well. By definition, $\cD(X,\lambda)'_\ra$ is a full subcategory of $\cD(X,\lambda)=\cD(X^\Xi,\Lambda)=\cD(\Mod(X^\Xi,\Lambda))$. For every object $\xi$ of $\Xi$, we have an \emph{evaluation functor}
\[
e_\xi^*\colon\Mod(X^\Xi,\Lambda)\to\Mod(X,\Lambda(\xi))
\]
at $\xi$ on the level of Abelian categories. It is exact and admits a (right exact) left adjoint functor
\begin{align}\label{c1eq:limit1}
e_{\xi!}\colon\Mod(X,\Lambda(\xi))\to\Mod(X^\Xi,\Lambda).
\end{align}
Moreover, we define a truncation functor
\begin{align}\label{c1eq:limit2}
t_{\leq\xi}\colon\Mod(X^\Xi,\Lambda)\to\Mod(X^\Xi,\Lambda)
\end{align}
such that for a $\Lambda$-module $F_{\bullet}\in\Mod(X^\Xi,\Lambda)$, we have
\[
(t_{\leq\xi}F_{\bullet})_{\xi'}=
\begin{dcases}
F_{\xi'} & \text{if $\xi'\leq\xi$,} \\
0 & \text{otherwise.}
\end{dcases}
\]
It is exact and admits a right adjoint.

\begin{proof}
By Lemma \ref{c1le:evaluation}, Lemma \ref{b3le:limit_restriction}, property (P4) for $\EO{}{\Chpar{}}{\r{I}}{}$, and Proposition \ref{c1pr:p34}, we may assume $X\in\Schqcs$.

We first study the functor
\[
\alpha\colon\cD(X,\lambda)'_\ra\to\varprojlim_{\rN(\Xi)^{op}}\cD(X,\Lambda(\xi))
\]
from the point of view of coCartesian fibrations. First, we have a functor $\Delta^1\times\rN(\Xi)\to\Cat$ sending
$\Delta^1\times(\varphi\colon\xi\to\xi')$ to the square
\[
\xymatrix{
\cD(X^{\Xi_{/\xi}},\Lambda_{/\xi}) \ar[r]^-{p_{\xi*}} \ar[d]_-{i_{\varphi*}} & \cD(X,\Lambda(\xi)) \ar[d]^-{\tilde{\varphi}_*} \\
\cD(X^{\Xi_{/\xi'}},\Lambda_{/\xi'}) \ar[r]^-{p_{\xi'*}}  & \cD(X,\Lambda(\xi')). }
\]
This corresponds to a projectively fibrant simplicial functor $\cF\colon\fC[\rN(D)]\to\Mset$, where $D=[1]\times\Xi$. Let $\phi_D\colon\fC[\rN(D)]\to D$ be the canonical equivalence of simplicial categories and put
\[
\cF'=(\Fibr^D \circ \St^+_{\phi_D^{op}}\circ\Un^+_{\rN(D)^{op}})\cF\colon D\to \Mset.
\]
We write $\cF'$ in the form $\cF'\colon[1]\to(\Mset)^\Xi$. Applying the marked unstraightening functor $\Un^+_{\phi}$ for the weak equivalence of simplicial categories $\phi\colon\fC[\rN(\Xi)^{op}]\to\Xi^{op}$, we obtain a morphism $\tilde\alpha\colon F_1\to F_2$ of Cartesian fibrations in the category $(\Mset)_{/\rN(\Xi)^{op}}$. Moreover, by \cite{HTT}*{Corollary 5.2.2.5}, both $F_1$ and $F_2$ are \emph{coCartesian} fibrations as well, but $\tilde\alpha$ does \emph{not} send coCartesian edges to coCartesian ones in general. By a similar argument, we have a map
\[
\cD(X^\Xi,\Lambda)\to\Map^{\r{coCart}}_{\rN(\Xi)^{op}}(\rN(\Xi)^{op},F_1)
\coloneqq\Map^\flat_{\rN(\Xi)^{op}}((\rN(\Xi)^{op})^\sharp,(F_1,\cE)),
\]
where $\cE$ is the set of coCartesian edges of $F_1$. Composing with the obvious inclusion $\Map^{\r{coCart}}_{\rN(\Xi)^{op}}(\rN(\Xi)^{op},F_1)\subseteq\Map_{\rN(\Xi)^{op}}(\rN(\Xi)^{op},F_1)$ and $\Map_{\rN(\Xi)^{op}}(\rN(\Xi)^{op},\tilde\alpha)$, we obtain a map
\[
\alpha'\colon\cD(X^\Xi,\Lambda)\to\Map_{\rN(\Xi)^{op}}(\rN(\Xi)^{op},F_2).
\]
We have the equivalence
\[
\Map^{\r{coCart}}_{\rN(\Xi)^{op}}(\rN(\Xi)^{op},F_2)\simeq\lim_{\Xi^{op}}\cD(X,\Lambda(\xi))
\]
by \cite{HTT}*{Corollary 3.3.3.2}, and the following pullback diagram
\[
\xymatrix{
\cD(X^\Xi,\Lambda)'_\ra \ar[r]^-{\alpha} \ar[d] & \Map^{\r{coCart}}_{\rN(\Xi)^{op}}(\rN(\Xi)^{op},F_2) \ar[d] \\
\cD(X^\Xi,\Lambda) \ar[r]^-{\alpha'} & \Map_{\rN(\Xi)^{op}}(\rN(\Xi)^{op},F_2) }
\]
by the definition of adic complexes, where vertical arrows are inclusions. We also note that $\alpha'$ commutes with small colimits by \cite{HTT}*{Proposition 5.1.2.2}. Thus, the goal is to show that $\alpha$ is an equivalence.

To construct an inverse $\beta$ of $\alpha$, we use $\del_{/\Xi}$: the category of simplices of $\Xi$. Then all $n$-cells
of $\rN(\del_{/\Xi})$ are degenerate for $n\geq2$. Define a functor
\[
\beta'\colon\rN(\del_{/\Xi}^{op})\to\Fun(\Map_{\rN(\Xi)^{op}}(\rN(\Xi)^{op},F_2),\cD(X^\Xi,\Lambda))
\]
sending a typical subcategory $\xi\to(\xi\to\xi')\leftarrow\xi'$ of $\del_{/\Xi}$ to
\[
\xymatrix{
\rL e_{\xi!}\circ \epsilon_\xi & t_{\leq\xi}\circ\rL e_{\xi'!}\circ\epsilon_{\xi'} \ar[l]\ar[r]
&  \rL e_{\xi'!}\circ \epsilon_{\xi'}, }
\]
where $\epsilon_\xi\colon\Map_{\rN(\Xi)^{op}}(\rN(\Xi)^{op},F_2)\to\cD(X,\Lambda(\xi))$ is the restriction functor to the fiber at $\xi$. The functor $\Fun(\alpha',\cD(X^\Xi,\Lambda))\circ\beta'$ extends to a functor
$\rN(\del_{/\Xi}^{op})^{\triangleright}\to\Fun(\cD(X^\Xi,\Lambda),\cD(X^\Xi,\Lambda))$ carrying $(\xi\to(\xi\to\xi')\leftarrow\xi')^\triangleleft$ to
\[
\xymatrix{
\rL e_{\xi!}\circ \epsilon_{\xi}\circ \alpha' \ar[rd] & t_{\leq\xi}\circ\rL e_{\xi'!}\circ\epsilon_{\xi'}\circ \alpha'
\ar[l]\ar[r]\ar[d] &  \rL e_{\xi'!}\circ \epsilon_{\xi'}\circ\alpha'\ar[ld] \\ & \r{id} }
\]
which induces a natural transformation
\[
(\colim\beta')\circ\alpha'\simeq\colim(\Fun(\alpha',\cD(X^\Xi,\Lambda))\circ\beta')\to\r{id}.
\]
Now we put
\[
\beta\coloneqq\colim\beta'\res\Map_{\rN(\Xi)^{op}}^{\r{coCart}}(\rN(\Xi)^{op},F_2).
\]
It is easy to check  that $\beta$ takes values in $\cD(X^\Xi,\Lambda)_\ra$.

We show that the induced natural transformation $\beta\circ\alpha\to\r{id}$ 
is an equivalence. Pick an object $\sfK$ of $\cD(X^\Xi,\Lambda)'_\ra$. We 
need to show that the diagram 
\[
\beta^{\triangleright}_{\sfK}\colon\rN(\del_{/\Xi}^{op})^{\triangleright}\to\cD(X^{\Xi},\Lambda),
\]
depicted as
\[
\xymatrix{
\rL e_{\xi!}\sfK_\xi \ar[dr] & t_{\leq\xi}\rL e_{\xi'!} \sfK_{\xi'} \ar[l]\ar[r]\ar[d] &
\rL e_{\xi'!} \sfK_{\xi'}\ar[dl] \\ & \sfK }
\]
is a colimit diagram. We only need to check this after applying $e_{\xi_0}^*$ for every $\xi_0\in\Xi$, since $e_{\xi_0}^*$ commutes with colimits. The composite functor $e_{\xi_0}^*\circ\beta^{\triangleright}_{\sfK}$ has value (equivalent to) $\sfK_{\xi_0}$ (resp.\ $0$) on the cone point, vertices $\{\xi\}$ and $(\xi\to\xi')$ of $\del_{/\Xi}$ for $\xi\geq \xi_0$ (resp.\ otherwise), with all morphisms being either identities on $\sfK_{\xi_0}$ or $0$, or the zero morphism $0\to\sfK_{\xi_0}$. It is clear that $e_{\xi_0}^*\circ\beta^{\triangleright}_{\sfK}$ induces an equivalence $\colim(e_{\xi_0}^*\circ\beta^{\triangleright}_{\sfK}\res\rN(\del_{/\Xi}^{op}))\simeq\sfK_{\xi_0}$ in $\cD(X,\Lambda(\xi_0))$.

For the other direction, that is, a natural equivalence $\alpha\circ\beta\to\r{id}$, we note that the functor $\Fun_{\rN(\Xi)^{op}}(\rN(\Xi)^{op},F_2),\alpha')\circ\beta'$ also extends to a functor
\[
\rN(\del_{/\Xi}^{op})^{\triangleright}\to\Fun(\Map_{\rN(\Xi)^{op}}(\rN(\Xi)^{op},F_2),\Map_{\rN(\Xi)^{op}}(\rN(\Xi)^{op},F_2))
\]
carrying $(\xi\to(\xi\to\xi')\leftarrow\xi')^\triangleleft$ to
\[
\xymatrix{
\alpha'\circ\rL e_{\xi!}\circ \epsilon_\xi \ar[rd] & \alpha'\circ t_{\leq\xi}\circ\rL e_{\xi'!}\circ\epsilon_{\xi'}
\ar[l]\ar[r]\ar[d] & \alpha'\circ \rL e_{\xi'!}\circ \epsilon_{\xi'}\ar[ld] \\ & \r{id} }
\]
which induces a natural transformation
\[
\alpha'\circ(\colim\beta')\simeq\colim(\Fun_{\rN(\Xi)^{op}}(\rN(\Xi)^{op},F_2),\alpha')\circ\beta')\to\id,
\]
where the equivalence of two functors is due to the fact that $\alpha'$ commutes with colimits. Restricting to $\Map^{\r{coCart}}_{\rN(\Xi)^{op}}(\rN(\Xi)^{op},F_2)$, one obtains a natural transformation $\alpha\circ\beta\to\id$ which is an equivalence by an argument similar to the previous one. Therefore, $\alpha$ is an equivalence and the proposition follows.
\end{proof}

By Theorem \ref{c1th:limit}, in what follows, we will identify $\cD(X,\lambda)'_\ra$ with $\cD(X,\lambda)_\ra$. In particular, we will regard $\cD(X,\lambda)_\ra$ as a full subcategory of $\cD(X,\lambda)$.

\begin{remark}\label{c1re:limit}
We have the following remarks.
\begin{enumerate}
  \item When we regard $\cD(X,\lambda)_\ra$ as a full subcategory of $\cD(X,\lambda)$, $\lambda_X$ coincides with the constant sheaf in $\cD(X,\lambda)$.

  \item By Proposition \ref{c1pr:adic_inclusion}, the inclusion functor $\cD(X,\lambda)_\ra\to\cD(X,\lambda)$ admits a right adjoint, which we denote by $\fR_X\colon\cD(X,\lambda)\to\cD(X,\lambda)_\ra$. It is a colocalization functor \cite{HTT}*{{\Sec}5.2.7}.

  \item Let $f\colon Y\to X$ be a morphism of $\Chpar{}$. The functor $f^*\colon\cD(X,\lambda)\to\cD(Y,\lambda)$ preserves adic complexes, and the induced functor $f^*\colon\cD(X,\lambda)_\ra\to\cD(Y,\lambda)_\ra$ coincides with $f^{*\ra}$ up to equivalence. The functor $f_{*\ra}$ is equivalent to the composition of the inclusion $\cD(Y,\lambda)_\ra\to\cD(Y,\lambda)$, $f_*\colon\cD(Y,\lambda)\to\cD(X,\lambda)$ and the functor $\fR_X$.

  \item Let $f\colon Y\to X$ be a locally of finite type morphism of $\Chpar{}_\Box$, and suppose that $\lambda\in\Rind_{\ltor}$. The functor $f_!\colon\cD(Y,\lambda)\to\cD(X,\lambda)$ preserves adic complexes, and the induced functor $f_!\colon\cD(Y,\lambda)_\ra\to\cD(X,\lambda)_\ra$ coincides with $f_{!\ra}$ up to equivalence. The functor $f^{!\ra}$ is equivalent to the composition of the inclusion $\cD(X,\lambda)_\ra\to\cD(X,\lambda)$, $f^!\colon\cD(X,\lambda)\to\cD(Y,\lambda)$ and the functor $\fR_Y$.

  \item The functor $-\otimes_X-\colon\cD(X,\lambda)\times\cD(X,\lambda)\to\cD(X,\lambda)$ preserves adic complexes, and the induced functor
      $-\otimes_X-\colon\cD(X,\lambda)_\ra\times\cD(X,\lambda)_\ra\to\cD(X,\lambda)_\ra$ coincides with $-\atimes_X-$ up to equivalence. The functor $\HOM^\ra_X$ is equivalent to the composition of the inclusion $\cD(X,\lambda)_\ra^{op}\times\cD(X,\lambda)_\ra\to\cD(X,\lambda)^{op}\times\cD(X,\lambda)$, $\HOM_X$ and $\fR_X$.

  \item We have $\cD^{\leq n}(X,\lambda)_\ra=\cD^{\leq n}(X,\lambda)\cap\cD(X,\lambda)_\ra$ for every $n\in\dZ$.
  
  \item Theorem \ref{c1th:limit} also holds if $X$ is a topos with enough 
      points. 
\end{enumerate}
\end{remark}

\begin{proof}[Proof of Theorem \ref{c1th:poincare}]
For (1), we note that $f_{!\ra}\lambda_Y\langle d\rangle=f_!\lambda_Y\langle d\rangle\in\cD^{\leq 0}(X,\lambda)$ by part (1) of (P7) in \Sec\ref{b4ss:description}. Thus, by definition, $f_{!\ra}\lambda_Y\langle d\rangle\in\cD^{\leq 0}(X,\lambda)_\ra$. Note that we have a trace map $f_!\lambda_Y\langle d\rangle\to\lambda_X$ in the non-adic case. Applying $\tau^{\geq0}_\ra$, we obtain the desired trace map
\[
\Tr_f\colon\tau_\ra^{\geq0}f_{!\ra}\lambda_Y\langle d\rangle=\tau_\ra^{\geq0}f_!\lambda_Y\langle d\rangle\to\tau_\ra^{\geq0}\lambda_X=\lambda_X
\]
which is a map in $\cD^\heartsuit(X,\lambda)_\ra$. The functoriality is automatic.

For (2), by the Poincar\'{e} duality $f^*\langle\dim f\rangle\simeq f^!$ in the non-adic case, $f^!$ preserves adic complexes hence $f^{!\ra}=f^!\res\cD(X,\lambda)_\ra$. Then it follows from the corresponding argument in the non-adic case.

\end{proof}

The following is a variant of Proposition \ref{b6pr:Hom_pi2}.

\begin{proposition}\label{7pr:adicHOM}
Let $X$ be an object of $\Chpar{}$, $\pi\colon \lambda'\to \lambda$ a perfect 
morphism of $\Rind$, and $\sfK$ an object of $\cD(X,\lambda)_\ra$. Then 
\begin{enumerate}
  \item The natural transformation $\pi_!(-\otimes_{\lambda'}\pi^* 
      \sfK)\to(\pi_!-)\otimes_\lambda \sfK$ is a natural equivalence. 

  \item The natural transformation 
      $\pi^*\HOM_\lambda(\sfK,-)\to\HOM_{\lambda'}(\pi^*\sfK,\pi^*-)$ is a 
      natural equivalence. 
\end{enumerate}
\end{proposition}

\begin{proof}
As in Proposition \ref{b6pr:Hom_pi2}, the two assertions are equivalent and 
for (1) we may assume that $X$ is an object of $\Schqcs$. In this case, the 
proof of (2) is similar to that of Lemma \ref{b2le:upperstar_pi}. Write 
$\lambda=(\Xi,\Lambda)$ and $\lambda'=(\Xi',\Lambda')$. As the family of 
functors $(e_{\xi'}^*)_{\xi'\in \Xi'}$ is conservative, it suffices to show 
that (2) holds with $\pi$ replaced by $e_{\xi'}$ and by $\pi\circ e_{\xi'}$. 
In other words, we may assume $\Xi'=\{\ast\}$. We decompose $\pi$ as 
\[
(\{\ast\},\Lambda')\xrightarrow{t}(\{\xi\},\Lambda(\xi)) \xrightarrow{s_\xi}(\Xi,\Lambda)_{/\xi}\xrightarrow{i_\xi}(\Xi,\Lambda).
\]
We show that (2) holds with $\pi$ replaced by $i_\xi$, by $s_\xi$, and by 
$t$. The assertion for $i_\xi$ is Proposition \ref{b6pr:Hom_pi2}. The 
assertion for $s_\xi$ is trivial as $s_\xi^*\simeq p_{\xi*}$ and 
$p_\xi^*p_{\xi*}\sfK'\simeq \sfK'$ for every object $\sfK'$ of 
$\cD(\cX,(\Xi,\Lambda)_{/\xi})_{\ra}$ by Lemma \ref{c1le:adic}. It remains to 
prove (2) with $\pi$ replaced by $t$. Changing notation, it suffices to prove 
(2) under the additional assumption $\Xi=\Xi'=\{*\}$. Then $\pi_*$ applied to 
(2) is given by
\[\pi_*\pi^* \HOM_{\Lambda'}(\sfK,-)\to \HOM_{\Lambda'}(\sfK,\pi_*\pi^*-)\simeq \pi_*\HOM_{\Lambda'}(\pi^*\sfK,\pi^*-),\]
which is a natural equivalence since  
$\pi_*\pi^*-\simeq\HOM_{\Lambda}({\Lambda'}\spcheck,-)$. We conclude by the 
fact that $\pi_*$ is conservative in this case.
\end{proof}

\begin{example}\label{7ex:notrightcomplete}
We give an example for which $\cD^{\ge \infty}(X,\lambda)_\ra=\bigcap_n 
\cD^{\ge n}(X,\lambda)_\ra$ contains nonzero objects. Let $k$ be a ring and 
let $\Lambda=k[x_0,x_1,\dots]$ be the polynomial ring in indeterminates 
$x_0,x_1,\dots$. Consider the functor $\Lambda_\bullet\colon \dN^\op \to 
\Ring$ carrying $n$ to $k[x_0,\dots,x_{n-1}]\simeq 
\Lambda/(x_{n},x_{n+1},\dots)$. Consider the homomorphism of 
$\Lambda_n$-modules $\phi_n\colon \Lambda_n\to \Lambda_n [t]$ sending $1$ to 
$\sum_{i=0}^{n-1} x_it^{i}$. We define a complex $\sfK$ of 
$\Lambda_\bullet$-modules by taking $\sfK_n=\mathrm{Kos}^\bullet(\phi_n)$ to 
be the Koszul complex. The transition maps are given by the obvious 
projections. Note that $\sfK_n\in \cD^{\ge n}$. Clearly $\sfK$ is adic. We 
claim that $\sfK\in \cD^{\ge \infty}(X,(\dN,\Lambda_\bullet))_\ra$. Let 
$\sfK'\in \cD^{\le n-1}(X,(\dN,\Lambda_\bullet))_\ra$ for some $n\ge 0$. 
Consider the morphism $j_n\colon (\dN_{\ge n}, \Lambda_{\bullet,\ge n})\to 
(\dN, \Lambda_{\bullet})$. Since $\sfK'$ is adic, we have 
$j_{n!}j_n^*\sfK'\simeq \sfK'$. Thus 
\[\Hom_{\rh\cD(X,(\dN,\Lambda_\bullet))}(\sfK',\sfK)\simeq 
\Hom_{\rh\cD(X,(\dN,\Lambda_\bullet))}(j_{n!}j_n^*\sfK',\sfK)
\simeq \Hom_{\rh\cD(X,(\dN_{\ge n},\Lambda_{\bullet,\ge n}))}(j_n^*\sfK',j_n^*\sfK)=0.
\]
It follows that $\sfK\in \bigcap\cD^{\ge 
\infty}(X,(\dN,\Lambda_\bullet))_\ra$. For $X$ nonempty and $k$ nonzero, 
$\sfK$ is nonzero. In particular, the t-structure on 
$\cD(X,(\dN,\Lambda_\bullet))_\ra$ is not right complete. 
\end{example}

We end this section with more results on the preservation of adic complexes under Noetherian assumptions. Put
\[
\cD(X,\lambda)^{(*)}_\ra\coloneqq\cD(X,\lambda)_\ra\cap\cD^{(*)}(X,\lambda)
\]
for $*=+,-,\rb$.\footnote{For $*=+,\rb$, this intersection does not coincide 
in general with $\cD^{(*)}(X,\lambda)_\ra$, the one given by the usual 
t-structure on $\cD(X,\lambda)_\ra$ introduced in 
\Sec\ref{c1ss:adic_properties}, for example if $\Xi$ has a finite object 
$\xi$ and $\Lambda(\xi')$ is of infinite tor-dimension over $\Lambda(\xi)$ 
for some $\xi'\in \Xi$. However, see Remark \ref{c2re:bound} below.} 

\begin{proposition}\label{7pr:preserve_adic}
Let $\lambda=(\Xi,\Lambda)$ be an object of $\Rind_{\ltor}$. Let $f\colon Y\to X$ be a morphism of $\Chpar{}_\Box$ that is locally of finite type such that $X$ is locally Noetherian. Then
\begin{enumerate}
  \item $f_*$ restricts to $f_{*\ra}\colon \cD(Y,\lambda)^{(+)}_\ra\to\cD(X,\lambda)^{(+)}_\ra$ if $f$ is quasi-finite and quasi-separated and $f_{*\ra}\colon \cD(X,\lambda)_\ra\to \cD(X,\lambda)_\ra$ if in addition $f$ is $0$-Artin and $X$ is locally finite-dimensional.

  \item $f^!$ restricts to $f^{!\ra}\colon \cD(X,\lambda)^{(+)}_\ra\to\cD(Y,\lambda)^{(+)}_\ra$ and, if $X$ is locally finite-dimensional, $f^{!\ra}\colon D(X,\lambda)_\ra\to \cD(Y,\lambda)_\ra$.

  \item Assume that $\Lambda(\xi)$ is Noetherian for every $\xi\in \Xi$. 
      Then $\HOM_X$ restricts to $\HOM_X^{\ra}\colon 
      (\cD(X,\lambda)^{(\r{ft})}_{\ra,\rc})^\op 
      \times\cD(X,\lambda)^{(+)}_\ra\to \cD(X,\lambda)^{(+)}_\ra$ and, if 
      $X$ is locally finite-dimensional, 
      $\HOM_X^{\ra}\colon(\cD(X,\lambda)^{(\r{ft})}_{\ra,\rc})^\op\times\cD(X,\lambda)_\ra\to\cD(X,\lambda)_\ra$. 
      Here, 
      $\cD(X,\lambda)^{(\r{ft})}_{\ra,\rc}=\cD(X,\lambda)_{\ra}\cap\cD^{(\r{ft})}_\cons(X,\lambda)$. 
\end{enumerate}
\end{proposition}

\begin{proof}
The second assertion of (1) and the second assertion of (2) follow from Proposition \ref{b6pr:more_adj}. 

For the first assertion of (1), we reduce easily to the case of complexes bounded from below and where $X$ is a coherent scheme. By the usual descent spectral sequence, we then reduce to the case where $Y$ is also a scheme. In this case, the assertion is the projection formula in \cite{LZ}*{Lemma~1.20(d)}. 

For the first assertion of (2), we reduce easily to the case of affine schemes, and then to the case of a closed immersion, which follows from (1). 

For (3), we may assume that $X$ is a coherent scheme. By Proposition 
\ref{7pr:adicHOM} below, it suffices to show that for all $\xi'\le\xi$ in 
$\Xi$, $\sfK\in\cD_\cons(X,\Lambda(\xi))$ of finite tor-dimension, and 
$\sfL\in\cD^+(X,\Lambda(\xi))$ (or $\sfL\in \cD(X,\Lambda(\xi))$ in the case 
where $X$ is finite-dimensional), the canonical morphism 
\begin{align*}
\HOM_{(X,\Lambda(\xi))}(\sfK,\sfL)\otimes \Lambda(\xi')
&\to\HOM_{(X,\Lambda(\xi))}(\sfK,\sfL\otimes \Lambda(\xi')) \\
&\simeq \HOM_{(X,\Lambda(\xi'))}(\sfK\otimes_{\Lambda(\xi)}\Lambda(\xi'),\sfL\otimes_{\Lambda(\xi)}\Lambda(\xi'))
\end{align*}
is an equivalence. For this, we may assume that $\sfK=j_!\sfK'$ with $j\colon U\to X$ an immersion and $\sfK'\in \cD(U,\Lambda(\xi))$ is a perfect complex. Then
\[
\HOM_{X}(\sfK,-)\simeq\HOM_U(\sfK',\Lambda(\xi))\otimes j_*-.
\]
We conclude by (1).
\end{proof}

\subsection{Constructible adic complexes}
\label{c1ss:constructible}

In this section let $\lambda=(\Xi,\Lambda)\in \Rind$ such that $\Lambda(\xi)$ is Noetherian for every $\xi\in \Xi$. For a higher Artin stack $X\in\Chpar{}$, we put
\begin{align*}
\cD(X,\lambda)_{\ra,\rc}&\coloneqq\cD(X,\lambda)_\ra\cap \cD_\cons(X,\lambda),\\
\cD(X,\lambda)_{\ra,\rc}^{(*)}&\coloneqq\cD(X,\lambda)_\ra\cap\cD_\cons^{(*)}(X,\lambda),
\end{align*}
where $*=+,-,\rb$. Note that $\cD(X,\lambda)_{\ra,\rc}^{(-)}=\cD^{(-)}(X,\lambda)_\ra\cap\cD_\cons(X,\lambda)$ always holds.

\begin{proposition}\label{c2le:constructible_adic}
Let $f\colon Y\to X$ be a morphism of higher Artin stacks. Then $f^*$ and $-\otimes_X-$ restrict to the following functors:
\begin{description}
  \item[1L'] $f^{*\ra}\colon\cD(X,\lambda)_{\ra,\rc}\to\cD(Y,\lambda)_{\ra,\rc}$.

  \item[3L'] $-\atimes_X-\colon\cD(X,\lambda)^{(-)}_{\ra,\rc}\times\cD(X,\lambda)^{(-)}_{\ra,\rc}\to\cD(X,\lambda)^{(-)}_{\ra,\rc}$.
\end{description}
In particular, we have a symmetric monoidal subcategory $(\cD(X,\lambda)^{(-)}_{\ra,\rc})^\otimes$ of $\cD(X,\lambda)^\otimes_\ra$. Moreover, if $\Lambda(\xi)$ is Noetherian and $\square$-torsion for every $\xi\in \Xi$, $X$ is $\square$-coprime, $f$ is of finite presentation (Definition \ref{b5de:finite_presentation}), then $f_!$ restricts to the following functors:
\begin{description}
  \item[2L'] $f_{!\ra}\colon\cD(Y,\lambda)^{(-)}_{\ra,\rc}\to\cD(X,\lambda)^{(-)}_{\ra,\rc}$ and, if $f$ is $0$-Artin,
      $f_{!\ra}\colon\cD(Y,\lambda)_{\ra,\rc}\to\cD(X,\lambda)_{\ra,\rc}$.
\end{description}
\end{proposition}

\begin{proof}
This follows immediately from Proposition \ref{b6pr:constructible}.
\end{proof}

As in \Sec\ref{b6ss:constructible}, to state the results for the other operations, we work in a relative setting. Let $\dS$ be a $\Box$-coprime higher Artin stack. Assume that there exists an atlas $S\to\dS$, where $S$ is either a quasi-excellent scheme or a regular scheme of dimension $\le 1$.

\begin{proposition}\label{c2pr:constructible_adic}
Suppose that $\lambda\in\Rind_{\ltor}$. Let $f\colon Y\to X$ be a morphism of $\Chpars$. Then $f_*$, $f^!$, $\HOM_X$ restrict to the following functors:
\begin{description}
  \item[1R'] $f_{*\ra}\colon\cD(Y,\lambda)^{(+)}_{\ra,\rc}\to\cD(X,\lambda)^{(+)}_{\ra,\rc}$ if $f$ is quasi-compact and quasi-separated (Definition \ref{b5de:finite_presentation}) and $f_{*\ra}\colon\cD(Y,\lambda)_{\ra,\rc}\to\cD(X,\lambda)_{\ra,\rc}$ if in addition $f$ is $0$-Artin and $\dS$ is locally finite-dimensional.

  \item[2R'] $f^{!\ra}\colon\cD(X,\lambda)^{(+)}_{\ra,\rc}\to\cD(Y,\lambda)^{(+)}_{\ra,\rc}$ and, if $\dS$ is locally finite-dimensional, $f^{!\ra}\colon\cD(X,\lambda)_{\ra,\rc}\to\cD(Y,\lambda)_{\ra,\rc}$.

  \item[3R'] $\HOM^\ra_X\colon(\cD(X,\lambda)^{(\r{ft})}_{\ra,\rc})^{op}\times\cD(X,\lambda)^{(+)}_{\ra,\rc}
      \to\cD(X,\lambda)^{(+)}_{\ra,\rc}$ and, if $\dS$ is locally finite-dimensional,
      $\HOM^\ra_X\colon(\cD(X,\lambda)^{(\r{ft})}_{\ra,\rc})^{op}\times\cD(X,\lambda)_{\ra,\rc}\to\cD(X,\lambda)_{\ra,\rc}$.
\end{description}
\end{proposition}

Note that in 3R' above, we do not need the assumption in Propositions 
\ref{b6pr:constructible2} and \ref{b6pr:constructible3} that $\Xi_{/\xi}$ is 
finite. 

\begin{proof}
This follows from Propositions \ref{b6pr:constructible2}, 
\ref{b6pr:constructible3}, \ref{7pr:adicHOM}, and \ref{7pr:preserve_adic}. 
(For the assertions on $\HOM^\ra$, we use Propositions \ref{7pr:adicHOM} and 
\ref{7pr:preserve_adic} to reduce to the case where $\Xi=\{*\}$.) 
\end{proof}

Note that $\cD(X,\lambda)^{(\r{ft})}_{\ra,\rc}=\cD(X,\lambda)^{(\rb)}_{\ra,\rc}$ if for every $\xi\in \Xi$, $\Lambda(\xi)$ is a local ring and there exists a morphism $\xi\to \xi'$ in $\Xi$ such that $\Lambda(\xi)\to \Lambda(\xi')$ identifies $\Lambda(\xi')$ with the residue field of $\Lambda(\xi)$. This is the case if $\cO$ is a Noetherian local ring of maximal ideal $\fm$ and $\lambda=(\dN,\Lambda)$ with $\Lambda(n)=\cO/\fm^{n+1}$.

\subsection{Adic dualizing complexes}
\label{c1ss:adic_dualizing}

In this section, we construct adic dualizing complexes and study the biduality properties in the adic case.

Let $X$ be an object of $\Chpar{}$, and $\lambda=(\Xi,\Lambda)$ an object of $\Rind$. Let $\Omega$ be an object of $\cD(X,\lambda)$ (resp.\ $\cD(X,\lambda)_\ra$). By adjunction of the pair of functors $-\otimes\sfK\coloneqq-\otimes_X\sfK$ and $\HOM(\sfK,-)\coloneqq\HOM_X(\sfK,-)$ (resp.\ $-\atimes\sfK\coloneqq-\atimes_X\sfK$ and $\HOM^\ra(\sfK,-)\coloneqq\HOM^\ra_X(\sfK,-)$), we have a natural transformation
\begin{align}\label{c1eq:biduality1}
\delta_\Omega&\colon\id\to\rh\HOM(\rh\HOM(-,\Omega),\Omega)\\\label{c1eq:biduality2}
\text{resp.\ }\delta^\ra_\Omega&\colon\id\to\rh\HOM^\ra(\rh\HOM^\ra(-,\Omega),\Omega)
\end{align}
between endofunctors of $\rh\cD(X,\lambda)$ (resp.\ $\rh\cD(X,\lambda)_\ra$), which is called the \emph{biduality transformation}.\footnote{In fact, $\delta_\Omega$ can be enhanced to a natural transformation $\tilde\delta_\Omega\colon\id\to\HOM(\HOM(-,\Omega),\Omega)$ between endofunctors of $\cD(X,\lambda)$, that is, $\rh\tilde\delta_\Omega=\delta_{\Omega}$; and similar for the adic case. We omit the details here since we do not need such enhancement in what follows.}

In the remaining of this section, we fix a $\Box$-coprime base scheme $\dS$ that is a disjoint union of \emph{excellent} schemes,\footnote{A scheme is \emph{excellent} if it is quasi-compact and admits a Zariski open cover by spectra of excellent rings \cite{EGAIV}*{D\'{e}finition 7.8.2}.} \emph{endowed with a global dimension function}. Let $\Rind_{\Box\text{-}\r{dual}}$ be the full subcategory of $\Rind_{\ltor}$ spanned by ringed diagrams $\Lambda\colon\Xi^{op}\to\Ring$ such that $\Lambda(\xi)$ is a ($\Box$-torsion) Gorenstein ring of dimension 0 for every object $\xi$ of $\Xi$.

\begin{definition}[Potential dualizing complex]\label{c1de:dualizing_complex}
Let $\lambda=(\Xi,\Lambda)$ be an object of $\Rind_{\Box\text{-}\r{dual}}$. For an object $f\colon X\to\dS$ of $\Chpars$ with $X$ in $\Schqcs$, we say that an object $\Omega\in\cD(X,\lambda)$ is a \emph{pinned/potential dualizing complex} (on $X$) if
\begin{enumerate}
  \item $\Omega$ is an adic complex, and

  \item for every object $\xi$ of $\Xi$, $\Omega_\xi=e_\xi^*\Omega\in\cD(X,\Lambda(\xi))$ is a pinned/potential dualizing complex.
\end{enumerate}
For a general object $f\colon X\to\dS$ of $\Chpars$, we say that an object $\Omega\in\cD(X,\lambda)$ is a \emph{pinned/potential dualizing complex} if for every atlas $u\colon X_0\to X$ with $X_0$ in $\Schqcs$, $u^!\Omega$ is a pinned/potential dualizing complex on $X_0$.
\end{definition}

\begin{proposition}\label{c1pr:dualizing_complex}
Let $f\colon X\to\dS$ be an object of $\Chpars$ and $\lambda$ an object of $\Rind_{\Box\text{-}\r{dual}}$. The full subcategory of $\cD(X,\lambda)$ spanned by all pinned/potential dualizing complexes is equivalent to the nerve of an ordinary category
consisting of only one object $\Omega$ with
\[
\Hom(\Omega,\Omega)=\(\lim_{\xi\in\Xi}\Lambda(\xi)\)^{\pi_0(X)}.
\]
Moreover, pinned/potential dualizing complexes are constructible and compatible under extension of scalars.
\end{proposition}

In the proof, we will use the following observation which is essentially \cite{HTT}*{Proposition A.3.2.27}. Let $\cC\colon K^{\triangleleft}\to\Cat$ be a functor that is a limit diagram. Let $X,Y$ be two objects in the limit $\infty$-category $\cC_{-\infty}$ and write $X_k,Y_k$ the natural images in $\cC_k$ for every vertex $k$ of $K$. Then $\Map_{\cC_{-\infty}}(X,Y)$ is naturally the homotopy limit (in the $\infty$-category $\cH$ of spaces) of a diagram $K\to\cH$ sending $k$ to $\Map_{\cC_k}(X_k,Y_k)$.

\begin{proof}
We first consider the case where $\Xi=*$ is a singleton.

In this case, if $X$ is in $\Schqcs$, then the proposition is proved in \cite{TG}*{Expos\'{e} XVIII-A} (see Remark \ref{b6re:riou}). We also note that if $\Omega_\dS$ is a pinned dualizing complex on $\dS$, then $f^!\Omega_\dS$ is a pinned dualizing complex on $X$. We prove by induction on $k$ that for an object $f\colon X\to\dS$ of $\Chpars$ with $X$ in $\Chpar{k}$,
\begin{enumerate}
  \item For any two pinned dualizing complexes $\Omega$ and $\Omega'$, $\Map_{\cD(X,\Lambda)}(\Omega,\Omega')$ is discrete;\footnote{More precisely, it means that $\Map_{\cD(X,\Lambda)}(\Omega,\Omega')$ is equivalent to a discrete set in $\cH$.}

  \item There is a unique distinguished equivalence $o\colon\Omega\to\Omega'$ such that for every atlas $u\colon X_0\to X$ with $X_0$ in $\Schqcs$, $u^!o$ is the one preserving pinning.
\end{enumerate}
It is clear that once the equivalence $o$ in (2) exists, it is compatible under $f^!$ for every smooth morphism $f$. Choose an atlas $u\colon Y\to X$ (with $Y$ in $\Chpar{(k-1)}$). Since $u$ is of universal $\EO{}{\Chpar{}_\Box}{!}{}$-descent, both (1) and (2) follow from the induction hypothesis, the above observation, and the fact that limit of $k$-truncated spaces is $k$-truncated (which follows from \cite{HTT}*{Proposition 5.5.6.5}).

Then we show that $\Map_{\cD(X,\Lambda)}(\Omega,\Omega)\simeq\pi_0\Map_{\cD(X,\Lambda)}(\Omega,\Omega)$ is isomorphic to $\Lambda^{\pi_0(X)}$. Without loss of generality, we assume that $X$ is connected. Choose an atlas $u=\coprod_Iu_i\colon \coprod_IY_i\to X$ with $Y_i$ in $\Schqcs$ that is connected. We have the following commutative diagram
\[
\xymatrix{
\Lambda \ar[r]^-{\alpha} \ar@{=}[d] & \pi_0\Map_{\cD(X,\Lambda)}(\Omega,\Omega) \ar[d]^-{\beta} \\
\Lambda \ar[r] &
\bigoplus_I\pi_0\Map_{\cD(Y_i,\Lambda)}(u_i^!\Omega,u_i^!\Omega). }
\]
Since $u^!$ is conservative, we know that the map $\beta$ is injective. Since the map $\Lambda\to\pi_0\Map_{\cD(Y_i,\Lambda)}(u_i^!\Omega,u_i^!\Omega)$ is an isomorphism for every $i\in I$, we know that $\alpha$ is injective. If we write elements of $\bigoplus_I\pi_0\Map_{\cD(Y_i,\Lambda)}(u_i^!\Omega,u_i^!\Omega)$ in the coordinate form $(\dots,\lambda_i,\dots)$ with respect to the basis consisting of distinguished equivalences, then the image of $u^!$ must belong to the diagonal since $X$ is connected. Therefore, $\alpha$ is an isomorphism. The fact that pinned dualizing complexes are constructible and compatible under extension of scalars follows from the case of schemes.

We then consider the case of general coefficient $\lambda=(\Xi,\Lambda)$. We start by constructing a pinned dualizing complex $\Omega_{\dS,\lambda}$ on the base scheme $\dS$. Recall that $\del_{/\Xi}$ is the category of simplices of $\Xi$, whose $n$-simplices are degenerate for $n\geq2$. For every object $\xi$ of $\Xi$, denote by $\Omega_{\dS,\xi}$ the pinned dualizing complex in $\cD(\dS,\Lambda(\xi))$. Recall the functors $e_{\xi!}$ \eqref{c1eq:limit1} and $t_{\leq\xi}$ \eqref{c1eq:limit2}. Define a functor $\delta\colon\rN(\del_{/\Xi})\to\cD(\dS,\lambda)$ sending a typical subcategory $\xi\leftarrow(\xi\to\xi')\rightarrow\xi'$ of $\del_{/\Xi}$ to
\[
\xymatrix{
\rL e_{\xi!}\Omega_{\dS,\xi} & \rL e_{\xi!}(\Omega_{\dS,\xi'}\overset{\rL}{\otimes}_{\Lambda(\xi')}\Lambda(\xi))
\simeq t_{\leq\xi}\rL e_{\xi'!}\Omega_{\dS,\xi'} \ar[l]\ar[r] &  \rL e_{\xi'!}\Omega_{\dS,\xi'} }
\]
in which the left arrow is given by the distinguished equivalence $\Omega_{\dS,\xi'}\overset{\rL}{\otimes}_{\Lambda(\xi')}\Lambda(\xi)\xrightarrow{\sim}\Omega_{\dS,\xi}$. It is easy to see that
$\Omega_{\dS,\lambda}\coloneqq\lim\delta$, viewed as an element in $\cD(\dS,\lambda)$, satisfies the two requirements in Definition
\ref{c1de:dualizing_complex}, hence is a pinned dualizing complex. For an object $f\colon X\to\dS$ of $\Chpars$, put $\Omega_{f,\lambda}=f^!\Omega_{\dS,\lambda}$. Then it is a pinned dualizing complex on $X$. The rest of the proposition follows from the fact that $\Omega_{f,\lambda}$ is adic, Theorem \ref{c1th:limit}, the observation before the proof, and the same assertion when $\Xi$ is a singleton.
\end{proof}

\begin{definition}\label{c1de:dualizing}
We introduce the following dualizing functors:
\begin{align*}
\cD=\cD_X\coloneqq\HOM_X(-,\Omega_{X,\lambda})&\colon\cD(X,\lambda)^{op}\to\cD(X,\lambda), \\
\cD^\ra=\cD^\ra_X\coloneqq\HOM^\ra_X(-,\Omega_{X,\lambda})&\colon\cD(X,\lambda)_\ra^{op}\to\cD(X,\lambda)_\ra.
\end{align*}
Put $\rD=\rh\cD$ and $\rD^\ra=\rh\cD^\ra$.
\end{definition}

\begin{proposition}
Let $(X,\lambda)$ be an object of $\Chpar{}\times\rN(\Rind)$. Let $\sfK\in\cD(X,\lambda)_\ra$ be an object such that $\delta_{\Omega_{X,\Lambda(\xi)}}(e_\xi^*\sfK)$ is an equivalence for every object $\xi$ of $\Xi$, where $\delta$ is the biduality transformation \eqref{c1eq:biduality1}. Then $\delta^\ra_{\Omega_{X,\lambda}}(\sfK)$ is an equivalence as well, where $\delta^\ra$ is the biduality transformation \eqref{c1eq:biduality2}.
\end{proposition}

\begin{proof}
We need to show that the natural morphism $\sfK\to\rD^\ra\rD^\ra\sfK$ is an isomorphism (in the homotopy category $\rh\cD(X,\lambda)_\ra$). By definition, we have
\begin{align*}
\rD^\ra\rD^\ra\sfK
&=\rh\HOM^\ra(\sfK,\rh\HOM^\ra(\sfK,\Omega_{X,\lambda}))\\
&\simeq\rh\fR_X\rh\HOM(\sfK,\rh\fR_X\rh\HOM(\sfK,\Omega_{X,\lambda})) \\
&\simeq\rh\fR_X\rh\HOM(\sfK,\rh\HOM(\sfK,\Omega_{X,\lambda})).
\end{align*}
It suffices to show that the map $\delta_{\Omega_{X,\lambda}}(\sfK)\colon\sfK\to\rh\HOM(\sfK,\rh\HOM(\sfK,\Omega_{X,\lambda}))$ is an equivalence. In fact, since $\sfK$ is adic, we have
\begin{align*}
e_\xi^*\rh\HOM(\sfK,\rh\HOM(\sfK,\Omega_{X,\lambda}))
&\simeq\rh\HOM(e_\xi^*\sfK,\rh\HOM(e_\xi^*\sfK,e_\xi^*\Omega_{X,\lambda}))\\
&\simeq\rh\HOM(e_\xi^*\sfK,\rh\HOM(e_\xi^*\sfK,\Omega_{X,\Lambda(\xi)}))
\end{align*}
for every object $\xi\in\Xi$ by Lemma \ref{c1le:adic_hom} below, which is equivalent to $e_\xi^*\sfK$ by the assumption.
\end{proof}

\begin{lem}\label{c1le:adic_hom}
Let $\lambda=(\Xi,\Lambda)$ be an object of $\Rind$, $\xi$ an object of $\Xi$, and $\sK$ an object of $\cD(X,\lambda)_\ra$. Then the following diagram
\[
\xymatrix{
\cD(X,\lambda) \ar[d]_-{e_\xi^*} && \cD(X,\lambda) \ar[d]^-{e_\xi^*}\ar[ll]_-{-\otimes_X\sfK} \\
\cD(X,\Lambda(\xi)) && \cD(X,\Lambda(\xi)) \ar[ll]_-{-\otimes_X e_\xi^*\sfK} }
\]
is right adjointable and its transpose is left adjointable. In other words, the natural maps $e_{\xi!}(\sfL\otimes_X e_\xi^*\sfK)\to(e_{\xi!}\sfL)\otimes_X\sfK$ and $e_\xi^*\HOM_X(\sfK,\sfL')\to\HOM(e_\xi^*\sfK,e_\xi^*\sfL')$ are equivalences for objects $\sfL$ of $\cD(X,\Lambda(\xi))$ and $\sfL'$ of $\cD(X,\lambda)$.
\end{lem}

\begin{proof}
By Proposition \ref{b6pr:lowersh_pi}, we may assume that $\xi$ is the final object of $\Xi$. In this case, $e_\xi^*$ can be identified with $\pi_*$, where $\pi\colon(\Xi,\Lambda)\to(\{\xi\},\Lambda(\xi))$ is the projection. Since $\sfK$ is adic, the morphism $\pi^*e_\xi^*\sfK\to\sfK$ is an equivalence. A left adjoint of the transpose of the above diagram is then given by the diagram
\[
\xymatrix{
\cD(X,\lambda) \ar[d]_-{-\otimes_X\sfK} && \cD(X,\Lambda(\xi)) \ar[ll]_-{\pi^*}\ar[d]^-{-\otimes_X e_\xi^*\sfK}  \\
\cD(X,\lambda) && \cD(X,\Lambda(\xi)) \ar[ll]_-{\pi^*}.
}
\]
The lemma follows by adjunction.
\end{proof}

\subsection{The $\fm$-adic formalism}
\label{c2ss:madic}

\begin{definition}\label{c2de:pring}
Define a category $\PRing$ as follows. The objects are pairs $(\Lambda,\fm)$, 
where $\Lambda$ is a (small) ring and $\fm\subseteq\Lambda$ is a principal 
ideal generated by an element that is not a zero divisor. A morphism from 
$(\Lambda',\fm')$ to $(\Lambda,\fm)$ is a ring homomorphism 
$\phi\colon\Lambda\to\Lambda'$ satisfying $\phi(\fm)\subseteq\fm'$. Let 
$\Lambda_n=\Lambda/\fm^n$. We denote by $\PRing_{\tor}$ (resp.\ 
$\PRing_{\ltor}$) the full subcategory of $\PRing$ spanned by $(\Lambda,\fm)$ 
such that $(\dN,\Lambda_\bullet)$ belongs to $\Rind_\tor$ (resp.\ 
$\Rind_{\ltor}$). 

We have a natural functor $\PRing\to\Fun([1],\Rind)$ sending $(\Lambda,\fm)$ to $(\dN,\Lambda_\bullet)\xrightarrow{\pi}(\ast,\Lambda)$. In what follows, we simply write $\Lambda_\bullet$ for the ringed diagram $(\dN,\Lambda_\bullet)$.
\end{definition}

Let $(\Lambda,\fm)$ be an object of $\PRing$. In this section, we will show 
that adic complexes for $\Lambda_\bullet$ enjoy very nice properties. In 
particular, they are preserved by the six operations. We start by stating a 
new characterization of adic complexes. Let $X\in\Chpar{}$ be a higher Artin 
stack. Recall that $\pi$ is perfect in the sense of Lemma 
\ref{b2le:upperstar_pi} and the functor 
\[
\pi^*\colon\cD(X,\Lambda)\to\cD(X,\Lambda_\bullet)
\]
admits a left adjoint $\pi_!$ by Proposition \ref{b6pr:upperstar_pi} and a 
right adjoint $\pi_*$.

\begin{theorem}\label{7th:madic}
For every $\sfK\in \cD(X,\Lambda_\bullet)$, the following conditions are 
equivalent: 
\begin{enumerate}
  \item $\sfK\in \cD(X,\Lambda_\bullet)_\ra$;
  \item $\sfK$ is in the essential image of $\pi^*$;
  \item The adjunction map $\pi^*\pi_*\sfK\to \sfK$ is an equivalence;
  \item The adjunction map $\sfK\to \pi^*\pi_!\sfK$ is an equivalence.
\end{enumerate}
\end{theorem}

To prove the theorem, we need some preliminaries on $\pi^*$ and its adjoints. 
We decompose $\pi$ into 
\[(\dN,\Lambda_\bullet)\xrightarrow{(\id,\gamma)} (\dN,\Lambda_{\dN})\xrightarrow{\rho} (*,\Lambda),\]
where $\Lambda_{\dN}\colon \dN^\op\to \Ring$ is the constant functor of value 
$\Lambda$. Let $\varpi$ be a generator of $\fm$. The following is a standard 
fact about derived completion. See \cite{SP} for variants. 

\begin{proposition}\label{7pr:RpiLpi}
We have fiber sequences
\begin{align*}
\pi_!\pi^*\sfK&\to \sfK\to \Lambda[1/\varpi]\otimes_\Lambda \sfK,\\
\HOM(\Lambda[1/\varpi],\sfK) &\to \sfK \to \pi_*\pi^* \sfK,
\end{align*} 
functorial in $\sfK\in \cD(X,\Lambda)$. 
\end{proposition}

\begin{proof}
The two fiber sequences being adjoint to each other, it suffices to prove the 
second one. 

We have a short exact sequence 
\[0\to Z_\bullet \to \Lambda_{\dN} \to \Lambda_\bullet\to 0,\]
where $Z_\bullet =(\dotsm \to \Lambda \xrightarrow{\times \varpi} \Lambda 
\xrightarrow{\times \varpi} \Lambda)$. Applying $-\otimes_{\Lambda_\dN} 
\rho^*\sfK$, we obtain a fiber sequence 
\begin{equation}\label{7eq:RpiLpifiber}
Z_\bullet \otimes_{\Lambda_\dN} \rho^*\sfK \to \rho^*\sfK \to \pi^* \sfK.
\end{equation}

Let $f_0\colon X_0\to X$ be a smooth atlas and let $X_\bullet$ be a \v{C}ech 
nerve of $f_0$. Then we have $\sfK\simeq\lim f_{n*}f_n^*\sfK$ and 
$\pi^*\sfK\simeq\lim f_{n*}f_n^*\pi^*\sfK$, where $f_n\colon X_n\to X$ is the 
induced morphism. Since $\HOM(\Lambda[1/\varpi],-)$ commutes with $f_{n*}$ up 
to equivalence, we are reduced to proving the second fiber sequence for each 
$X_n$. By induction, we may assume that $X$ is a scheme. In this case, by 
Remark \ref{3re:Funequiv}, 
\[\rho_*(Z_\bullet \otimes_{\Lambda_\dN} \rho^*\sfK)\simeq \HOM(\colim (\Lambda \xrightarrow{\times \varpi} \Lambda 
\xrightarrow{\times \varpi} \Lambda\to \dotsm),\sfK)\simeq 
\HOM(\Lambda[1/\varpi],\sfK).
\] 
Moreover, we have $\rho_!\simeq e_0^*$ and $\rho_!\rho^*\simeq \id$. By 
adjunction, it follows that $\id\simeq \rho_*\rho^*$. We conclude by applying 
$\rho_*$ to \eqref{7eq:RpiLpifiber}. 
\end{proof}

\begin{corollary}\label{7co:Lpi0}
For $\sfK\in \cD(X,\Lambda)$, $\pi^*\sfK=0$ if and only if multiplication by 
$\varpi$ is an equivalence on $\sfK$.
\end{corollary}

\begin{proof}
If $\times \varpi$ is an equivalence on $\sfK$, then $e_n^*\pi^*\sfK=0$ for 
all $n$. Conversely,  if $\pi^*\sfK=0$, then, by Proposition 
\ref{7pr:RpiLpi}, $\sfK\simeq \HOM(\Lambda[1/\varpi],\sfK)$ and it suffices 
to remark that $\times \varpi$ is an isomorphism on $\Lambda[1/\varpi]$. 
\end{proof}

\begin{corollary}\label{7co:RpiLpi}
The natural transformations 
\begin{gather*}
  \pi^*\to \pi^*\circ \pi_*\circ \pi^*,\quad \pi^*\circ 
    \pi_*\circ \pi^*\to \pi^*,\\
  \pi_*\to \pi_*\circ \pi^*\circ \pi_*,\quad  \pi_*\circ \pi^*\circ \pi_*
  \to \pi_*,
\end{gather*} 
induced by the unit $\epsilon\colon \id\to \pi_*\pi^*$ and the counit $\eta\colon 
\pi^*\pi_*\to \id$ of the adjunction $\pi^*\dashv \pi_*$ are natural 
equivalences. 
\end{corollary}

\begin{proof}
The composition of the two natural transformations on the first line (resp.\ 
second) is equivalent to the identity. Thus it suffices to show that the two 
natural transformations induced by $\epsilon$ are natural equivalences. For 
every $\sfK\in \cD(X,\Lambda)$, $\pi^*\HOM(\Lambda[1/\varpi],\sfK)=0$ and 
$\pi^*\epsilon_\sfK$ is an equivalence by Proposition \ref{7pr:RpiLpi}. For 
$\sfL\in \cD(X,\Lambda_\bullet)$, $\HOM(\Lambda[1/\varpi],\pi_*\sfL)\simeq 
\pi_*\HOM(\pi^*\Lambda[1/\varpi],\sfL)=0$ and $\epsilon_{\pi_*\sfL}$ is an 
equivalence by Proposition \ref{7pr:RpiLpi}. 
\end{proof}

\begin{remark}\label{7re:perfect}
Let $\Lambda\to \Lambda'$ be a ring homomorphism and let $t\colon 
(*,\Lambda')\to (*,\Lambda)$ be the corresponding morphism in $\Rind$. Then 
$t^*=\Lambda' \otimes_\Lambda -$ and $t_*$ is restriction of scalars. If $t$ 
is perfect, then $t^*$ admits a left adjoint, $t_!$. In this case, we have an 
equivalence $t^*t_!-\simeq (t^*t_!\Lambda') \otimes_{\Lambda'} -$. More 
precisely, the map $t^*t_! \sfK\to (t^*t_!\Lambda') \otimes_{\Lambda'} \sfK$, 
adjoint to the map $\sfK\to (t^*t_*t^*t_!\Lambda') \otimes_{\Lambda'} 
\sfK\simeq t^*t_*((t^*t_!\Lambda') \otimes_{\Lambda'} \sfK)$ is an 
equivalence.  Indeed, $t^*t_*-\simeq (\Lambda'\otimes_\Lambda 
\Lambda')\otimes_{\Lambda'} -$ and its left adjoint $t^*t_!$ is equivalent to 
$(\Lambda'\otimes_\Lambda \Lambda')\spcheck\otimes_{\Lambda'}-$ (which is 
also a right adjoint of $t^*t_*$). 
\end{remark}

\begin{proof}[Proof of Theorem \ref{7th:madic}]
(4)$\implies$(2). Obvious.

(2)$\implies$(1). Since $\cD(X,\Lambda)=\cD(X,\Lambda)_\ra$, the image of 
$\pi^*$ is contained in $\cD(X,\Lambda_\bullet)_\ra$.

(1)$\implies$(4).  We denote by $\epsilon\colon \id\to \pi^*\pi_!$ the 
adjunction map. We will show that $\epsilon_\sfK$ is an equivalence for every 
$\sfK\in \cD(X,\Lambda_\bullet)_\ra$. Consider the inclusions 
\[(\{n\},\Lambda_n)\xrightarrow{s_n} 
(\dN_{\le n},\Lambda_{\bullet,\le n})\xrightarrow{i_n} (\dN,\Lambda_\bullet).
\] 
We have $\sfK\simeq \colim_{n\in \dN} i_{n!}i_n^*\sfK\simeq \colim_{n\in \dN} 
e_{n!}\sfK_n$. Here in the second equivalence we used the equivalence 
$i_n^*\sfK\simeq s_{n!}\sfK_n$, which follows from the assumption that $\sfK$ 
is adic. We 
have a diagram 
\[\xymatrix{i_n^*\sfK\ar[rrr]^{i_n^*\epsilon_{e_{n!}\sfK_n}}\ar[d] &&& i_n^*\pi^*\pi_!e_{n!}\sfK_n\ar[d]\\
i_n^*(e_{n!}\Lambda_n \otimes_{\Lambda_\bullet} \sfK) \ar[rrr]^{i_n^*(\epsilon_{e_{n!}\Lambda_n}\otimes_{\Lambda_\bullet}\sfK)} 
&&& i_n^*(\pi^*\pi_!e_{n!}\Lambda_n\otimes_{\Lambda_\bullet} \sfK)}
\]
where the vertical arrows are equivalences. The vertical arrow on the right 
is given by the fact that the source and target are both adic and the 
equivalence $e_n^*\pi^*\pi_!e_{n!}\sfK_n\simeq t_n^*t_{n!}\sfK_n\simeq 
t_n^*t_{n!}\Lambda_n \otimes_{\Lambda_n} \sfK_n$ in Remark \ref{7re:perfect}, where $t_n\coloneqq\pi \circ e_n\colon (\{n\},\Lambda_n)\to (*,\Lambda)$. 
Restricting the diagram to $(\{m\},\Lambda_m)$ and taking colimit for $n\in 
\dN_{\ge m}$, we see that $e_m^*\epsilon_{\sfK}$ is equivalent to 
$e_m^*(\epsilon_{\Lambda_\bullet}\otimes_{\Lambda_\bullet}\sfK)$. Thus, it 
remains to show that $\epsilon_{\Lambda_\bullet}$ is an equivalence. By 
Corollary \ref{7co:RpiLpi}, the adjunction map $\pi^* \to \pi^* \circ \pi_! 
\circ \pi^*$ is an equivalence. In particular, 
$\epsilon_{\Lambda_\bullet}=\epsilon_{\pi^*\Lambda}$ is an equivalence. 

(3)$\implies$(2). Obvious.

(2)$\implies$(3). This follows immediately from Corollary \ref{7co:RpiLpi}.
\end{proof}

\begin{corollary}\label{7co:normlocal}    
The inclusion functor $\cD(X,\Lambda_\bullet)_\ra\to \cD(X,\Lambda_\bullet)$ 
admits a left adjoint given by $\pi^*\circ\pi_!$ and a right adjoint given by 
$\pi^*\circ \pi_*$. 
\end{corollary}

\begin{corollary}\label{7co:adictor}
For any $\sfK\in \cD(X,\Lambda)$, the following conditions are equivalent:
\begin{enumerate}
\item $\sfK$ is in the essential image of $\pi_!$;
\item The adjunction map $\pi_!\pi^* \sfK\to \sfK$ is an equivalence;
\item $\Lambda[1/\varpi]\otimes_\Lambda \sfK=0$.
\end{enumerate}
We let $\cD(X,\Lambda)_{\tor}\subseteq \cD(X,\Lambda)$ denote the full 
subcategory spanned by $\sfK$ satisfying the above conditions. Then $\pi^*$ 
and $\pi_!$ induce equivalences between $\cD(X,\Lambda)_{\tor}$ and 
$\cD(X,\Lambda_\bullet)_\ra$. 
\end{corollary}

Objects of $\cD(X,\Lambda)_\tor$ are said to be \emph{$\fm^\infty$-torsion} 
objects. By (3), $\sfK\in \cD(X,\Lambda)_\tor$ if and only if $\rH^i\sfK\in 
\cD(X,\Lambda)_\tor$ for all $i\in \dZ$. 

\begin{proof}
We have (1)$\iff$(2) by Corollary \ref{7co:RpiLpi} and (2)$\iff$(3) by 
Proposition \ref{7pr:RpiLpi}. The last assertion follows from Theorem 
\ref{7th:madic}. 
\end{proof}

\begin{remark}
Dually, we let 
    $\cD(X,\Lambda)_\compl\subseteq \cD(X,\Lambda)$ denote the essential 
    image of the localization functor $\pi_*\circ \pi^*\colon \cD(X,\Lambda)\to \cD(X,\Lambda)$, 
    which is also the essential image of 
    $\pi_*$. The functors $\pi^*$  and $\pi_*$ induce equivalences between 
    $\cD(X,\Lambda)_\compl$ and $\cD(X,\Lambda_\bullet)_\ra$.
\end{remark}

We have seen that $f^*$, $f_!$, and $-\otimes_{\Lambda_\bullet} -$ preserve 
adic complexes in Remark \ref{c1re:limit}. We can now prove that the other 
three operations preserve $\fm$-adic complexes, extending Proposition 
\ref{7pr:preserve_adic}. 

\begin{proposition}\label{c2pr:preserve_adic2}
Let $f\colon Y\to X$ be a morphism of higher Artin stacks and let 
$(\Lambda,\fm)$ be an object of $\PRing$. Then 
\begin{enumerate}
  \item $f_*$ restricts to $f_{*\ra}\colon 
      \cD(Y,\Lambda_\bullet)_\ra\to\cD(X,\Lambda_\bullet)_\ra$. 

  \item $\HOM_X$ restricts to $\HOM_X^{\ra}\colon 
      (\cD(X,\Lambda_\bullet)_{\ra})^\op \times\cD(X,\Lambda_\bullet)_\ra\to 
      \cD(X,\Lambda_\bullet)_\ra$. 

  \item $f^!$ restricts to $f^{!\ra}\colon 
      \cD(X,\Lambda_\bullet)_\ra\to\cD(Y,\Lambda_\bullet)_\ra$ if $f$ is morphism of 
      $\Chpar{}_\Box$ that is locally of finite type and $(\Lambda,\fm)$ is 
      an object of $\PRing_{\ltor}$. 
\end{enumerate}
\end{proposition}

\begin{proof}
The assertions for $f_*$ and $\HOM_X$ follow from  the commutation of these 
two functors with $\pi^*$ (Propositions \ref{b6pr:upperstar_pi} and 
\ref{7pr:adicHOM}). By Poincar\'e duality, the assertion for $f^!$ reduces to 
the case where $f$ is a closed immersion of schemes. This case follows again 
from the commutation of $f^!$ with $\pi^*$ (Proposition 
\ref{b6pr:lowersh_pi}). 
\end{proof}

Next we discuss a new t-structure on $\cD(X,\Lambda_\bullet)_\ra$. Recall that we already have the usual t-structure $(\cD^{\leq n}(X,\Lambda_\bullet)_\ra,\cD^{\geq n}(X,\Lambda_\bullet)_\ra)$ on $\cD(X,\Lambda_\bullet)_\ra$ from \Sec\ref{c1ss:adic_properties}.

Put $\cD^{\le n}(X,\Lambda)_\tor\coloneqq\cD(X,\Lambda)_\tor\cap \cD^{\le 
n}(X,\Lambda)$ and $\cD^{\ge n}(X,\Lambda)_\tor\coloneqq\cD(X,\Lambda)_\tor\cap 
\cD^{\ge n}(X,\Lambda)$. Since truncation functors on $\cD(X,\Lambda)$ 
preserve $\cD(X,\Lambda)_\tor$, $(\cD^{\le 1}(X,\Lambda)_\tor,\cD^{\ge 
1}(X,\Lambda)_\tor)$ is a t-structure on $\cD(X,\Lambda)_\tor$. Via the 
equivalence of $\infty$-categories in Corollary \ref{7co:adictor}, we obtain 
a t-structure $(\pi^*\cD^{\le 1}(X,\Lambda)_\tor,\pi^*\cD^{\ge 
1}(X,\Lambda)_\tor)$ on $\cD(X,\Lambda_\bullet)_\ra$, with truncation 
functors given by $\pi^*\tau^{\le 1}\pi_!$ and $\pi^*\tau^{\ge 1}\pi_!$. We 
denote the above t-structure by $(\cD_!^{\le 
0}(X,\Lambda_\bullet)_\ra, \cD_!^{\ge 0}(X,\Lambda_\bullet)_\ra)$. 

\begin{proposition}\label{7pr:adictcompare}
Let $X$ be a higher Artin stacks and let $(\Lambda,\fm)$ be an object of $\PRing$. 
\begin{enumerate}
\item The t-structure $(\cD_!^{\le 0}(X,\Lambda_\bullet)_\ra, \cD_!^{\ge 
    0}(X,\Lambda_\bullet)_\ra)$ is right complete. 
\item We have 
    \begin{align*}
    \cD^{\le 0}(X,\Lambda_\bullet)_\ra &\subseteq 
    \cD_!^{\le 0}(X,\Lambda_\bullet)_\ra \subseteq \cD^{\le 
    1}(X,\Lambda_\bullet)_\ra, \\
    \cD^{\ge 
    1}(X,\Lambda_\bullet)_\ra &\subseteq \cD_!^{\ge 
    0}(X,\Lambda_\bullet)_\ra \subseteq \cD(X,\Lambda_\bullet)_\ra^{\ge 
    0}\subseteq \cD^{\ge 0}(X,\Lambda_\bullet)_\ra.
    \end{align*}
    Here, $\cD(X,\Lambda_\bullet)_\ra^{\ge*}$ is introduced in Lemma \ref{c1le:p6}.
\end{enumerate} 
\end{proposition}

\begin{proof}
(1) The right completeness follows from the criterion \cite[Proposition 
1.2.1.19]{HA}: $\pi^*\cD^{\ge 1}(X,\Lambda)_\tor$ is stable under countable 
coproducts and $\bigcap_n \pi^*\cD^{\ge n}(X,\Lambda)_\tor$ consists of zero 
objects. 

(2) Since $\pi^*\colon \cD(X,\Lambda)\to \cD(X,\Lambda_\bullet)$ has 
t-amplitude contained in $[-1,0]$ for the usual t-structures, we have 
$\cD_!^{\le 
    0}(X,\Lambda_\bullet)_\ra \subseteq \cD^{\le 1}(X,\Lambda_\bullet)_\ra$ and $\cD_!^{\ge 
    0}(X,\Lambda_\bullet)_\ra \subseteq \cD(X,\Lambda_\bullet)_\ra^{\ge 0}$. 
The inclusion $\cD(X,\Lambda_\bullet)_\ra^{\ge 0}\subseteq \cD^{\ge 
    0}(X,\Lambda_\bullet)_\ra$ is Lemma \ref{c1le:p6}. The other inclusions 
    follow by orthogonality.
\end{proof}

\begin{remark}\label{c2re:bound}
It follows from Proposition \ref{7pr:adictcompare}(2) that $\cD^{(+)}(X,\Lambda_\bullet)_\ra=\cD(X,\Lambda_\bullet)^{(+)}_\ra$.
\end{remark}

\begin{corollary}\label{7co:adict}
The usual t-structure on $\cD(X,\Lambda_\bullet)_\ra$ is right complete.
\end{corollary}

\begin{proof}
This follows immediately from Proposition \ref{7pr:adictcompare}. 
\end{proof}

The t-structure on $\cD(X,\Lambda)_\tor$ corresponding to the usual 
t-structure on $\cD(X,\Lambda)_\ra$ can be described as follows. We say that 
an object $\cF$ of $\cD^\heartsuit(X,\Lambda)_\tor$ is $\fm$-divisible if 
$\cF\xrightarrow{\times \varpi} \cF$ is a surjection or, equivalently, if 
$\rH^0\pi^*\cF=0$. 

\begin{corollary}\label{7co:adictdiv}
Let $\sfK\in \cD(X,\Lambda)_\tor$.
\begin{enumerate}
\item $\pi^*\sfK\in \cD^{\le 0}(X,\Lambda_\bullet)_\ra$ if and only if 
    $\sfK\in \cD^{\le 1}(X,\Lambda)_\tor$ and $\rH^1\sfK$ is 
    $\fm$-divisible. 
\item $\pi^*\sfK\in \cD^{\ge 0}(X,\Lambda_\bullet)_\ra$ if and only if 
    $\sfK\in \cD^{\ge 0}(X,\Lambda)_\tor$ and $\rH^0\sfK$ contains no 
    nonzero $\fm$-divisible sub-object. 
\end{enumerate}
\end{corollary}

In other words, the usual t-structure on $\cD(X,\Lambda)_\ra$ corresponds to 
the tilt (in the sense of \cite{HRS}, \cite{Polishchuk}) of the usual 
t-structure on $\cD(X,\Lambda)_\tor$ with respect to the torsion pair 
$(\cT,\cT^\perp)$, where $\cT$ is the class of $\fm$-divisible objects in 
$\cD^\heartsuit(X,\Lambda)_\tor$. See also \cite{BBD}*{\Sec3.3}.

\begin{proof}
(1) Under the condition $\sfK\in \cD^{\le 1}(X,\Lambda)_\tor$, 
$\rH^1\pi^*\sfK\simeq \rH^0\pi^*\rH^1\sfK$ is zero if and only if 
$\pi^*\sfK\in \cD^{\le 0}(X,\Lambda_\bullet)_\ra$. Thus it suffices to show 
that $\cD^{\le 0}(X,\Lambda_\bullet)_\ra=\cD_!^{\le 
0}(X,\Lambda_\bullet)_\ra\cap \cD^{\le 0}(X,\Lambda_\bullet)_\ra$, which 
follows from Proposition \ref{7pr:adictcompare}(2). 

(2) By Proposition \ref{7pr:adictcompare}(2), $\pi^*\sfK\in \cD^{\ge 
0}(X,\Lambda_\bullet)_\ra$ implies $\sfK\in \cD^{\ge 0}(X,\Lambda)_\tor$. 
Thus we may assume  $\sfK\in \cD^{\ge 0}(X,\Lambda)_\tor$. Then $\pi^*\sfK\in 
\cD^{\ge 0}(X,\Lambda_\bullet)_\ra$ if and only if for every $\sfL\in 
\cD(X,\Lambda_\bullet)_\ra$ satisfying $\pi^*\sfL\in \cD^{\le 
-1}(X,\Lambda_\bullet)_\ra$, we have 
$\Hom_{\rh\cD(X,\Lambda)_\tor}(\sfL,\sfK)\simeq 
\Hom_{\rh\cD(X,\Lambda_\bullet)_\ra}(\pi^*\sfK,\pi^*\sfL)=0$. By (1), 
$\pi^*\sfL\in \cD^{\le -1}(X,\Lambda_\bullet)_\ra$ if and only if $\sfL\in 
\cD^{\le 0}(X,\Lambda)_\tor$ and $\rH^0\sfL$ is $\fm$-divisible. In this 
case, 
$\Hom_{\rh\cD(X,\Lambda)_\tor}(\sfL,\sfK)\simeq\Hom_{\rh\cD^\heartsuit(X,\Lambda)_\tor}(\rH^0\sfL,\rH^0\sfK)$. 
Assertion (2) follows. 
\end{proof}

The following result is obvious.

\begin{proposition}
For every morphism $f\colon Y\to X$ of higher Artin stacks, $f^{*\ra}\colon 
\cD(X,\Lambda_\bullet)_\ra\to \cD(Y,\Lambda_\bullet)_\ra$ is t-exact for the 
t-structures $(\cD_!^{\le 0}, \cD_!^{\ge 0})$.  
\end{proposition}

\begin{proposition}
Let $f\colon Y\to X$ be a morphism higher Artin stacks. Assume that one of 
the following conditions hold:
\begin{enumerate}
\item $f=i$ is a closed immersion; or
\item $f$ is locally of finite type and $f$ is in $\Chpar{}_\Box$  and 
    $(\Lambda,\fm)$ is in $\PRing_{\ltor}$.
\end{enumerate}
Then $f^{*\ra}\colon \cD(X,\Lambda_\bullet)_\ra\to 
\cD(Y,\Lambda_\bullet)_\ra$ is t-exact for the usual t-structures.  
\end{proposition}

\begin{proof}
(1) By Corollary \ref{7co:adictdiv}, it suffices to show that for $\cF\in 
\cD^\heartsuit(X,\Lambda)_\tor$ that contains no nonzero $\fm$-divisible 
sub-object, $i^*\cF$ satisfies the same property. If $\cG$ is an 
$\fm$-divisible sub-object of $i^*\cF$, then we have monomorphisms $i_*\cG\to 
i_*i^*\cG\to \cF$, which implies that $\cG=0$. 

(2) The case $f$ smooth being already known, we reduce easily to (1). 
\end{proof}

\begin{example}
Without the assumption that $f\colon Y\to X$ is locally of finite type, 
$f^{*\ra}\colon \cD(X,\Lambda_\bullet)_\ra\to \cD(Y,\Lambda_\bullet)_\ra$ is 
in general not t-exact for the usual t-structures. Take $X$ to be the 
spectrum of an absolutely integrally closed valuation ring with valuation 
group $\bigoplus_{n=-\infty}^0\dQ$, ordered lexicographically. Take $f$ to be 
the inclusion of the generic point $Y$ of $X$. The open subsets of $X$ form a 
chain $X=U_0\supsetneq U_{-1}\supsetneq U_{-2}\supsetneq\dotsm \supsetneq 
\emptyset$. Assume that $\fm\subsetneq \Lambda$. Consider the 
$\fm^\infty$-torsion sheaf $\cF\in \Mod(X_{\et},\Lambda)$ given by 
$\cF(U_n)=\varpi^{n}\Lambda/\Lambda$ for all $n\in \dN_{\le 0}$, with 
restriction maps given by the inclusions. The only $\fm$-divisible submodule 
of $\cF$ is zero. However, $f^*\cF=\Lambda[\varpi^{-1}]/\Lambda$ is 
$\fm$-divisible. Thus, by Corollary \ref{7co:adictdiv}, $\pi^*\cF\in 
\cD^\heartsuit(X,\Lambda_\bullet)_\ra$ and $f^{*\ra}\pi^*\cF\simeq 
\pi^*f^*\cF\in \cD^\heartsuit(Y,\Lambda_\bullet)_\ra[1]$. 
\end{example}

The results of this section also hold for $X$ a topos with enough points. 

\begin{remark}
Let $X$ be a replete topos \cite{BS}. Since $\sfK\in \cD(X,\Lambda)$ is derived 
complete if and only if each $\rH^q\sfK$ is, $(\cD^{\le 
0}(X,\Lambda)_\compl, \cD^{\ge 0}(X,\Lambda)_\compl)$ is a t-structure on 
$\cD(X,\Lambda)_\compl$, where $\cD^{\le 
0}(X,\Lambda)_\compl=\cD(X,\Lambda)_\compl\cap \cD^{\le 0}(X,\Lambda)$ and  
$\cD^{\ge 0}(X,\Lambda)_\compl=\cD(X,\Lambda)_\compl\cap \cD^{\ge 
0}(X,\Lambda)$. Moreover, for every $\sfL\in \cD^{\le 
0}(X,\Lambda_\bullet)_\ra$, $\rH^0\sfL$ is a surjective system. It follows 
that $\pi_*\colon \cD(X,\Lambda_\bullet)_\ra\to \cD(X,\Lambda)_\compl$ is 
t-exact. Thus  $\cD^{\le 0}(X,\Lambda)_\ra=\pi^*\cD^{\le 
0}(X,\Lambda)_\compl$  and $\cD^{\ge 0}(X,\Lambda)_\ra=\pi^*\cD^{\ge 
0}(X,\Lambda)_\compl$. 
\end{remark}

\subsection{Compatibility with Laszlo--Olsson (adic coefficients)}
\label{c2ss:compatibility}

We prove the compatibility between our adic formalism and Laszlo--Olsson's \cite{LO2}, under their assumptions.

Put $\Box=\{\ell\}$ where $\ell$ is a rational prime. Let $\dS$ be a $\Box$-coprime scheme satisfying that
\begin{enumerate}
  \item it is affine excellent and finite-dimensional;

  \item for every scheme $X$ of finite type over $\dS$, there exists an \'{e}tale cover $X'\to X$ such that $\r{cd}_{\ell}(Y)<\infty$ for every scheme $Y$ \'{e}tale and of finite type over $X'$;\footnote{According to our notation, $\r{cd}_{\ell}$ is nothing but $\r{cd}_{\d{F}_{\ell}}$.}

  \item it admits a global dimension function and we fix such a function (see Remark \ref{b6re:dimension}).
\end{enumerate}

Recall from \Sec\ref{b6ss:compatibility} that we denote $\Chplmb{\dS}$ the full subcategory of $\Chplft{\dS}$ spanned by ($1$-)Artin stacks locally of finite type over $S$, with quasi-compact and separated diagonal.

For the coefficient, we fix a complete discrete valuation ring $\Lambda$ with 
the maximal ideal $\fm$ and residue characteristic $\ell$ such that 
$\Lambda=\lim_n\Lambda_n$, where $\Lambda_n=\Lambda/\f{m}^{n+1}$, as in 
\cite{LO2}. In particular, $(\Lambda,\fm)$ is an object of $\PRing_{\ltor}$ 
in our notation. 

From the definition of $\cD(\cX,\Lambda_\bullet)_{\ra,\rc}$, which is the full subcategory of $\cD(\cX,\Lambda_\bullet)$ spanned by constructible adic complexes, \cite{LO2}*{Proposition 3.0.10, Theorem 3.0.14, Proposition 3.0.18}, and Proposition \ref{b5pr:descent_lisse}, we have a canonical equivalence between categories
\begin{align}\label{c2eq:compatibility}
\rh\cD(\cX,\Lambda_\bullet)_{\ra,\rc}\simeq\bD_c(\cX,\Lambda),
\end{align}
where the latter one is defined in \cite{LO2}*{Definition 3.0.6}.

\begin{proposition}
For a morphism $f\colon\cY\to\cX$ of finite type in $\Chplmb{\dS}$, there are natural isomorphisms of functors:
\begin{align*}
\rh f^{*\ra}\simeq\rL f^*&\colon\bD_c(\cX,\Lambda)\to\bD_c(\cY,\Lambda),\\
\rh f_{*\ra}\simeq\rR f_*&\colon\bD_c^{(+)}(\cY,\Lambda)\to\bD_c^{(+)}(\cX,\Lambda),\\
\rh f_{!\ra}\simeq\rR f_!&\colon\bD_c^{(-)}(\cY,\Lambda)\to\bD_c^{(-)}(\cX,\Lambda),\\
\rh f^{!\ra}\simeq\rR f^!&\colon\bD_c(\cX,\Lambda)\to\bD_c(\cY,\Lambda),\\
\rh(-\atimes_\cX-)\simeq(-)\overset{\bL}{\otimes}(-)&\colon
\bD_c^{(-)}(\cX,\Lambda)\times\bD_c^{(-)}(\cX,\Lambda)\to\bD_c^{(-)}(\cX,\Lambda),\\
\rh\HOM^\ra_\cX\simeq\b{\s{R}hom}_\Lambda&\colon
\bD_c^{(-)}(\cX,\Lambda)^{\r{opp}}\times\bD_c^{(+)}(\cX,\Lambda)\to\bD_c^{(+)}(\cX,\Lambda)
\end{align*}
that are compatible with \eqref{c2eq:compatibility}. Here, on the right side of the equivalences, we adopt notation from \cite{LO2}*{\Sec1}.
\end{proposition}

By Lemma \ref{c2le:constructible_adic} and Proposition \ref{c2pr:constructible_adic}, the six operations on the left side in the above proposition do have the correct range.

\begin{proof}
The isomorphisms for tensor product, internal Hom and $f^*$ simply follow from the same definitions here and in \cite{LO2}*{\Sec4, \Sec6}. The isomorphism for $f_*$ follows from the adjunction and that for $f^*$ (Proposition \ref{b6pr:compatibility}). The isomorphism for $f_!$ will follows from the adjunction and that for $f^!$ which will be proved below.

By the compatibility of dualizing complexes and the isomorphisms for internal Hom, we have natural isomorphisms $\rD^\ra_\cX\simeq\rD_\cX$ and $\rD^\ra_\cY\simeq\rD_\cY$ (Definition \ref{c1de:dualizing}). Therefore, by \cite{LO2}*{Definition 9.1}, to show the isomorphism for $f^!$, we only need to show that our functors satisfy
\[
\rh f^{!\ra}\simeq\rD^\ra_\cY\circ\rh f^{*\ra}\circ\rD^\ra_\cX.
\]
Note that for every $\sfK\in\bD_c(\cX,\Lambda)$, the biduality map $\delta^\ra_{\Omega_\cX}(\sfK)\colon\sfK\to\rD^\ra_\cX(\rD^\ra_\cX(\sfK))$ is an isomorphism by \cite{LO2}*{Theorem 7.3.1}. Thus, we have
\begin{align*}
\rh f^{!\ra}\sfK&\simeq\rh f^{!\ra}(\rD^\ra_\cX(\rD^\ra_\cX(\sfK)))\\
&=\rh f^{!\ra}(\rh\HOM^\ra_\cX(\rh\HOM^\ra_\cX(\sfK,\Omega_\cX),\Omega_\cX)) \\
&\simeq \rh\fR_\cY(\rh f^!(\rh\HOM_\cX(\rh\HOM^\ra_\cX(\sfK,\Omega_\cX),\Omega_\cX)))  \\
&\simeq \rh\fR_\cY(\rh\HOM_\cY(\rh f^*(\rh\HOM^\ra_\cX(\sfK,\Omega_\cX),f^!\Omega_\cX)))  \\
&\simeq \rh\HOM^\ra_\cY(\rh f^{*\ra}(\rh\HOM^\ra_\cX(\sfK,\Omega_\cX),\Omega_\cY))\\
&=\rD^\ra_\cY(\rh f^{*\ra}(\rD^\ra_\cX(\sfK))).
\end{align*}
The proposition is proved.
\end{proof}

\begin{remark}
In view of the above compatibility, we have proven all the expected 
properties of the six operations, in particular the Base Change Theorem, in 
the adic case of Laszlo--Olsson \cite{LO2}. 
\end{remark}

\section{Perverse t-structures}
\label{c3}

In this chapters, we study perverse t-structures for stacks. In 
\Sec\ref{c3ss:perversity_evaluation}, we define the notion of perversity 
evaluations on stacks, to which we will associate t-structures. In 
\Sec\ref{c3ss:perverse_t}, we construct the perverse t-structure with respect 
to a perverse evaluation. In \Sec\ref{c3ss:adic_perverse_t}, we construct 
perverse t-structures in the adic case.

\subsection{Perversity evaluations}
\label{c3ss:perversity_evaluation}

We first recall various notion of perversity functions on schemes, introduced by Gabber.

\begin{definition}
Let $X$ be a scheme in $\Schqcs$. Denote by $|X|$ the underlying topological space of $X$.
\begin{enumerate}
  \item Following \cite{Gabber}*{{\Sec}1}, a \emph{weak perversity function} on $X$ is a function
      \[
      p\colon|X|\to\dZ\cup\{+\infty\}
      \]
      such that for every $n\in\dZ$, the set $\{x\in|X|\res p(x)\geq n\}$ is ind-constructible.

  \item An \emph{admissible perversity function} on $X$ is a weak perversity function $p$ such that for every $x\in|X|$, there is an open dense subset $U\subseteq\overline{\{x\}}$ satisfying the condition that for every $x'\in U$, $p(x')\leq p(x)+2\codim(x',x)$.

  \item A \emph{codimension perversity function} on $X$ is a function $p\colon|X|\to\dZ\cup\{+\infty\}$ such that for every immediate \'etale specialization $x'$ of $x$, $p(x')=p(x)+1$.
\end{enumerate}
\end{definition}

\begin{remark}
We have the following remarks concerning perversity functions.
\begin{enumerate}
  \item A weak perversity function on a locally Noetherian scheme is locally bounded from below.

  \item An admissible perversity function on a scheme that is locally Noetherian and of finite dimension is locally bounded from above.

  \item A codimension perversity function on a scheme is \emph{not} necessarily a weak perversity function.

  \item A codimension perversity function that is also a weak perversity function is an admissible perversity function. If $X$ is locally
      Noetherian, then a codimension perversity function is a weak perversity function and hence an admissible perversity function.

  \item A codimension perversity function is the opposite of a dimension function in the sense of \cite{TG}*{Expos\'{e} XIV, D\'{e}finition 2.1.8}. If $X$ is locally Noetherian and admits a dimension function, then $X$ is universally catenary by \cite{TG}*{Expos\'{e} XIV, Proposition 2.2.6}. In this case, immediate \'etale specializations coincide with immediate Zariski specializations \cite{TG}*{Expos\'{e} XIV, Proposition 2.1.4}.

  \item If $p$ is a weak (resp.\ admissible, resp.\ codimension) perversity function on $X$ and $d\colon|X|\to\dZ\cup\{+\infty\}$ is a locally constant function, then $p+d$ is a weak (resp.\ admissible, resp.\ codimension) perversity function on $X$.
\end{enumerate}
\end{remark}

\begin{definition}\label{c3de:moderate}
A function $q\colon\dN\to\dZ$ or $q\colon\dZ\to\dZ$ is called \emph{moderate} if $q$ and $\b{2}-q$ are both increasing. Here, $\b{2}$ is the function $\b{2}(x)=2x$ and similarly for $\b{0}$ and $\b{1}$, which will be used below.
\end{definition}

\begin{notation}\label{c3no:pullback}
Let $f\colon Y\to X$ be a morphism of schemes in $\Schqcs$. For a function 
$p\colon|X|\to\dZ\cup\{+\infty\}$, we define the pullback 
$f^*_{\b{0}}p\colon|Y|\to\dZ\cup\{+\infty\}$ by $f_{\b{0}}^*p=p\circ f$. If 
$f$ is locally of finite type and $q\colon\dN\to\dZ$ is a function, we define 
more generally the $q$-weighted pullback 
$f^*_qp\colon|Y|\to\dZ\cup\{+\infty\}$ by 
\[
(f^*_qp)(y)=p(f(y))-q(\trdeg[k(y):k(f(y))])
\]
for every point $y\in|Y|$.
\end{notation}

In the following two lemmas we list some stability properties of weighted 
pullbacks of perversity functions.

\begin{lem}\label{c3le:perverse_zero}
Let $f\colon Y\to X$ be a morphism (resp.\ \'etale morphism, resp.\ \'etale morphism) of schemes in $\Schqcs$. If $p$ is a weak (resp.\ admissible, resp.\ codimension) perversity function on $X$, then $f^*_{\b{0}}p$ is a weak (resp.\ admissible, resp.\ codimension) perversity function on $Y$.
\end{lem}

\begin{proof}
We have $f^*_{\b{0}}p=p\circ f$. If $p$ is a weak perversity function, then
\[
\{y\in|Y| \mid f^*_{\b{0}}p(y)\geq n\}=f^{-1}(\{x\in|X| \mid p(x) \geq n\})
\]
is ind-constructible by \cite{EGAIV}*{Proposition 1.9.5(vi)}. The other two cases follow from the trivial fact that $\codim(y',y)=\codim(f(y'),f(y))$ for every specialization $y'$ of $y$ on $Y$.
\end{proof}

\begin{lem}\label{c3le:perverse_pullback}
Let $f\colon Y\to X$ be a morphism of locally Noetherian schemes in $\Schqcs$, locally of finite type.
\begin{enumerate}
  \item Let $p$ be a weak perversity function on $X$ and $q\colon\dN\to\dZ$ an increasing function. Then $f^*_qp$ is a weak perversity function on $Y$.

  \item Let $p$ be an admissible perversity function on $X$ an $q\colon\dN\to\dZ$ a moderate function (Definition \ref{c3de:moderate}). Then $f^*_qp$ is an admissible perversity function on $Y$.

  \item Let $p$ be a codimension perversity function on $X$. Then $f^*_{\b{1}}p$ is a codimension perversity function on $Y$.
\end{enumerate}
\end{lem}

\begin{proof}
For a locally closed subset $Z$ of a scheme $X$, we endow it with the reduced induced subscheme structure. For every point $y\in|Y|$, let
$U_y\subset \overline{\{y\}}$ be a nonempty open subset such that the induced morphism $f_y\colon\overline{\{y\}}\to\overline{\{f(y)\}}$ is
flat. Such an open subset exists by \cite{EGAIV}*{Th\'{e}or\`{e}me 6.9.1}. For $y'\in U_y$, we have
\begin{align*}
\delta(y',y)&\coloneqq\trdeg[k(y):k(f(y))]-\trdeg[k(y'):k(f(y'))]\\
&=\codim(y',U_y\times_{\overline{f(y)}}\{f(y')\})\geq 0
\end{align*}
by \cite{EGAIV}*{Proposition 14.3.13} since $f_y$ is universally open \cite{EGAIV}*{Th\'{e}or\`{e}me 2.4.6}.

For (1), we know that for every $n\in\dZ$,
\[
\{y\in|Y|\res f^*_qp(y)\geq n\}=\bigcup_{y\in|Y|}f^{-1}\left\{x\in|X|\mid p(x)\geq n+q(\trdeg[k(y):k(f(y))])\right\}\cap U_y
\]
is a union of ind-constructible subsets, and hence is itself ind-constructible. In other words, $f_q^*p$ is a weak perversity function.

For (2), let $y\in|Y|$ be a point; put $x=f(y)$; and let $U_x\subset\overline{\{x\}}$ be a dense open subset such that $p(x')\leq p(x)+2\codim(x',x)$ for every $x'\in U_x$. We prove that for $y'\in U_y\cap f^{-1}(U_x)$,
\[
f^*_qp(y')\leq f^*_qp(y)+2\codim(y',y)
\]
holds. We may assume $p(x)\in\dZ$. Put $x'=f(y')$. We have
\[
f^*_qp(y)=p(x)-q(\trdeg[k(y):k(x)])
\]
and
\[
f^*_qp(y')=p(x')-q(\trdeg[k(y'):k(x')]).
\]
Moreover, by \cite{EGAIV}*{Corollaire 6.1.2}, we have
\[
\delta(y',y)=\codim(y',y)-\codim(x',x).
\]
Therefore, we have
\begin{align*}
f^*_qp(y')-f^*_qp(y)&=p(x')-p(x)+q(\trdeg[k(y):k(x)])-q(\trdeg[k(y'):k(x')])\\
&\leq2\codim(x',x)+2\delta(y',y)=2\codim(y',y)
\end{align*}
since $q$ is moderate. In other words, $f_q^*p$ is an admissible perversity function on $Y$.

For (3), it is essentially proved in \cite{TG}*{Expos\'{e} XIV, Corollaire 2.5.2}.
\end{proof}

Now we generalize the notion of perversity functions from schemes to stacks, by starting from the following definition.

\begin{definition}[Pointed schematic neighborhood]
Let $X$ be a higher Artin (resp.\ Deligne--Mumford) stack. A \emph{pointed smooth (resp.\ \'etale) schematic neighborhood} of $X$ is a triple
$(X_0,u_0,x_0)$ where $u_0\colon X_0\to X$ is a smooth (resp.\ an \'{e}tale) morphism with $X_0\in\Schqcs$ and $x_0\in|X_0|$ a scheme-theoretical point. A morphism $v\colon (X_1,u_1,x_1)\to(X_0,u_0,x_0)$ of pointed smooth (resp.\ \'etale) schematic neighborhoods is a smooth (resp.\ an \'{e}tale) morphism $v\colon X_1\to X_0$ such that there is a triangle
\begin{align}\label{c3eq:schematic_neighborhood}
\xymatrix{
X_1 \ar[rd]_-{u_1} \ar[rr]^-{v} & & X_0 \ar[ld]^-{u_0} \\
& X }
\end{align}
with $v(x_1)=x_0$. We say that $(X_1,u_1,x_1)$ \emph{dominates} $(X_0,u_0,x_0)$ if there is such a morphism. The category of pointed smooth (resp.\ \'etale) schematic neighborhoods of $X$ is denoted by $\Vosm(X)$ (resp.\ $\Voet(X)$).
\end{definition}

\begin{lem}\label{c3le:codimension}
Let $X$ be a higher Artin stack, and let $v\colon(X_1,u_1,x_1)\to(X_0,u_0,x_0)$ be a morphism of pointed smooth schematic neighborhoods of $X$. Then the codimension of $x_1$ in the base change scheme $X_{1,x_0}=X_1\times_{X_0}\{x_0\}$ depends only on the source and the target of $v$.
\end{lem}

\begin{proof}
Note that $\codim(x_1,X_{1,x_0})=\dim_{x_1}(v)-\trdeg[k(x_1):k(x_0)]$. It is clear that the term $\dim_{x_1}(v)=\dim_{x_1}(u_1)-\dim_{x_0}(u_0)$ does not depend on $v$. We will show that the other term $\trdeg[k(x_1):k(x_0)]$ does not depend on $v$ either.

Let $f\colon Y\to X$ be an atlas of $X$ with $Y$ a scheme in $\Schqcs$. Let
\[
\xymatrix{
Y_1 \ar[rr]^-{v'} \ar[dr]_-{u'_1} && Y_0 \ar[dl]^-{u'_0} \\ & Y }
\]
be the base change of \eqref{c3eq:schematic_neighborhood}, and $f_0\colon Y_0\to X_0$, $f_1\colon Y_1\to X_1$ the induced morphisms. Let $w_0\colon Y'_0\to Y_0$ be an atlas with $Y'_0$ a scheme in $\Schqcs$, and let
\[
\xymatrix{
Y'_1 \ar[r]^-{v''} \ar[d]_-{w_1} & Y'_0 \ar[d]^-{w_0} \\
Y_1 \ar[r]^-{v'} & Y_0 }
\]
be the base change. Then $v''$ is a smooth morphism of schemes in $\Schqcs$. Since $f_0\circ w_0\colon Y'_0\to X_0$ is smooth and surjective, the base change scheme $Y'_{0,x_0}=Y'_0\times_{X_0}\{x_0\}$ is nonempty and smooth over the residue field $k(x_0)$ of $x_0$. Similarly, we have a nonempty scheme $Y'_{1,x_1}$, smooth over $k(x_1)$. Choose a generic point $y'_1$ of $Y'_{1,x_1}$. Then its image $y'_0$ in $Y'_{0,x_0}$ is a generic point. Let $y$ be the image of $y'_0$ in $Y$. Then we have
\[
\trdeg[k(x_1):k(x_0)]=\r{tr.deg}[k(y'_1):k(y)]-\trdeg[k(y'_0):k(y)]
\]
which does \emph{not} depend on $v$. The lemma follows.
\end{proof}

\begin{notation}\label{c3no:codimension}
Let $X$ be a higher Artin stack, and let $v\colon(X_1,u_1,x_1)\to(X_0,u_0,x_0)$ be a morphism of pointed smooth schematic neighborhoods of $X$. We will denote by $\delta^{(X_1,u_1,x_1)}_{(X_0,u_0,x_0)}$ the codimension appeared in Lemma \ref{c3le:codimension}. It is clear that
\[
\delta^{(X_2,u_2,x_2)}_{(X_0,u_0,x_0)}=\delta^{(X_2,u_2,x_2)}_{(X_1,u_1,x_1)}+\delta^{(X_1,u_1,x_1)}_{(X_0,u_0,x_0)}
\]
if $(X_2,u_2,x_2)$ dominates $(X_1,u_1,x_1)$. Moreover, if $v$ is \'etale, then we have $\delta^{(X_1,u_1,x_1)}_{(X_0,u_0,x_0)}=0$.
\end{notation}

\begin{notation}
For a higher Artin (resp.\ Deligne--Mumford) stack $X$ and a function $\sfp\colon\Ob(\Vosm(X))\to\dZ\cup\{+\infty\}$ (resp.\ $\sfp\colon\Ob(\Voet(X))\to\dZ\cup\{+\infty\}$), we have, by restriction, the function $\sfp_{u_0}\colon|X_0|\to\dZ\cup\{+\infty\}$ for every smooth (resp.\ \'{e}tale) morphism $u_0\colon X_0\to X$ with $X_0$ in $\Schqcs$.

If $f\colon Y\to X$ is a smooth (resp.\ an \'{e}tale) morphism of higher Artin (resp.\ Deligne--Mumford) stacks, then composition with $f$ induces a functor $f\colon \Vosm(Y)\to\Vosm(X)$ (resp.\ $f\colon \Voet(Y)\to\Voet(X)$), and we put $f^*\sfp=\sfp\circ f$.
\end{notation}

\begin{definition}[(admissible/codimension) perversity evaluations]\label{c3de:perversity_evaluation}
Let $X$ be a higher Artin stack. A \emph{smooth evaluation} on $X$ is a function
\[
\sfp\colon\Ob(\Vosm(X))\to\dZ\cup\{+\infty\}
\]
such that for $(X_1,u_1,x_1)$ dominating $(X_0,u_0,x_0)$, we have
\[
\sfp(X_0,u_0,x_0) \le\sfp(X_1,u_1,x_1)\le\sfp(X_0,u_0,x_0)+2\delta^{(X_1,u_1,x_1)}_{(X_0,u_0,x_0)}.
\]

A \emph{perversity smooth evaluation} (resp.\ \emph{admissible perversity smooth evaluation}, \emph{codimension perversity smooth evaluation}) on $X$ is a smooth evaluation $\sfp$ such that for every $(X_0,u_0,x_0)\in\Ob(\Vosm(X))$, $\sfp_{u_0}$ is a weak perversity function (resp.\ admissible perversity function, codimension perversity function) on $X_0$.

Similarly, we define \'etale evaluations and (admissible/codimension) perversity \'etale evaluations on a higher Deligne--Mumford stack $X$ using $\Voet(X)$.

We say that a smooth (resp.\ \'etale) evaluation $\sfp$ is \emph{locally bounded} if for every smooth (resp.\ \'etale) morphism $u_0\colon X_0\to X$ with $X_0$ a quasi-compact separated scheme, $\sfp_{u_0}$ is bounded.
\end{definition}

\begin{remark}\label{8re:etaleev}
If $X$ is a scheme in $\Schqcs$, then the map from the set of \'etale 
evaluations on $X$ to the set of functions $|X|\to\dZ\cup\{+\infty\}$, 
carrying $\sfp$ to $\sfp_{\id_X}$, is bijective. Under this bijection, the 
notions of (weak) perversity, admissible perversity, and codimension 
perversity coincide.
\end{remark}

\begin{example}
We have the following examples of perversity smooth/\'etale evaluations.
\begin{enumerate}
  \item Let $X$ be a higher Artin (resp.\ Deligne--Mumford) stack. Then every constant smooth (resp.\ \'etale) evaluation is an admissible
      perversity smooth (resp.\ \'etale) evaluation.

  \item Let $f\colon Y\to X$ be a morphism of higher Deligne--Mumford 
      stacks. Let $\sfp$ be an \'etale evaluation on $X$. We define an 
      \'etale evaluation $f^*_{\b{0}}\sfp$ on $Y$ as follows. For any 
      object $(Y_0,v_0,y_0)$ of $\Voet(Y)$, there exists a morphism 
      $(Y_1,v_1,y_1)\to(Y_0,v_0,y_0)$ in $\Voet(Y)$ such that there exists 
      a diagram 
      \[
      \xymatrix{Y_1\ar[r]^{v_1}\ar[d]_{f_0} & Y\ar[d]^f\\
      X_0\ar[r]^{u_0} & X,}
      \]
      where $X_0$ is in $\Schqcs$ and $u_0$ is \'etale. We put
      \[
      f^*_{\b{0}}\sfp(Y_0,v_0,y_0)=\sfp(X_0,u_0,f_0(y_1)).
      \]
      This clearly does not depend on choices. If $\sfp$ is a perversity 
      \'etale evaluation, then so is $f^*_{\b{0}}\sfp$ by Lemma 
      \ref{c3le:perverse_zero}. If $f$ is \'etale, then 
      $f^*_{\b{0}}\sfp=f^*\sfp$. 
      
      If $f$ is locally of finite type and $q\colon\dN\to\dZ$ is a 
      function, we define more generally an \'etale evaluation $f^*_q\sfp$ 
      on $Y$ by
      \[ 
      f^*_q\sfp(Y_0,v_0,y_0)=\sfp(X_0,u_0,f_0(y_1))-q(\trdeg[k(y_1):k(f_0(y_1))]). 
      \]
      
      In the case where $X$ and $Y$ are schemes, the above notation is 
      compatible with Notation \ref{c3no:pullback} via the bijection in 
      Remark \ref{8re:etaleev}. 

  \item Let $f\colon Y\to X$ be a morphism of higher Artin stacks with $X$ 
      being a higher Deligne--Mumford stack. Let $\sfp$ be an \'etale 
      evaluation on $X$, and $q\colon\dZ\to\dZ$ a moderate function 
      (Definition \ref{c3de:moderate}). Assume that $f$ is locally of 
      finite type in the case $q\neq \b{0}$. We define a smooth evaluation 
      $f^*_q\sfp$ on $Y$ by the formula 
      \[
      (f^*_q\sfp)(Y_0,v_0,y_0)=((v_0\circ f)^*_{q'}\sfp)_{\id_{Y_0}}(y_0)
      \]
      for every object $(Y_0,v_0,y_0)$ of $\Vosm(Y)$, where 
      $q'\colon\dN\to\dZ$ is the function $q'(n)=q(n-\dim_{y_0}(v_0))$. If 
      $\sfp$ is a perversity \'etale evaluation, then $f^*_{\b{0}}\sfp$ is 
      a perversity smooth evaluation. If $X$ is locally Noetherian, $f$ is 
      locally of finite type, and $\sfp$ is a perversity (resp.\ admissible 
      perversity, resp.\ codimension perversity) \'etale evaluation, then 
      $f^*_q\sfp$ (resp.\ $f^*_q\sfp$, resp.\ $f^*_{\b{1}}\sfp$) is a 
      perversity (resp.\ admissible perversity, resp.\ codimension 
      perversity) smooth evaluation by Lemma \ref{c3le:perverse_pullback}. 
\end{enumerate}
\end{example}

\subsection{Perverse t-structures}
\label{c3ss:perverse_t}

In this section, we define t-structures associated to perversity evaluations.

\begin{definition}\label{c3de:complete}
Let $\cC$ be a stable $\infty$-category equipped with a t-structure. We say that $\cC$ is \emph{weakly left complete} (resp.\ \emph{weakly right complete}) if $\cC^{\le -\infty}\coloneqq\bigcap_n \cC^{\le -n}$ (resp.\ $\cC^{\ge\infty}\coloneqq\bigcap_n \cC^{\ge n}$) consists of zero objects.
\end{definition}

The family $(\rH^i)_{i\in \dZ}$ is conservative if and only if $\cC$ is both weakly left complete and weakly right complete (cf.\ \cite{BBD}*{Proposition 1.3.7}). The following lemma slightly extends \cite{HA}*{Proposition 1.2.1.19}.

\begin{lem}\label{c3le:wlc}
Let $\cC$ be a stable $\infty$-category equipped with a t-structure. Consider the following conditions
\begin{enumerate}
  \item The $\infty$-category $\cC$ is left complete.

  \item The $\infty$-category $\cC$ is weakly left complete.
\end{enumerate}
Then (1) implies (2). Moreover, if $\cC$ admits countable products and there exists an integer $a$ such that countable products of objects of $\cC^{\le0}$ belong to $\cC^{\le a}$, then (2) implies (1).
\end{lem}

\begin{proof}
The first assertion is obvious since the image of $\cC^{\le -\infty}$ under the functor $\cC\to\widehat\cC$ consists of zero objects, where $\widehat\cC$ is defined prior to \cite{HA}*{Proposition 1.2.1.17}.

To show the second assertion, it suffices to replace $f(n-1)$ by $f(n-a-1)$ in the proof of \cite{HA}*{Proposition 1.2.1.19}.
\end{proof}

Let $X$ be a scheme in $\Schqcs$, let $p\colon|X|\to\dZ\cup\{+\infty\}$ be a function, and let $\lambda=(\Xi,\Lambda)$ be an object of $\Rind$. Following Gabber \cite{Gabber}*{{\Sec}2}, we define full subcategories $\TS{p}{\cD}{\leq0}(X,\lambda),\TS{p}{\cD}{\geq0}(X,\lambda)\subseteq\cD(X,\lambda)$ as follows: For $\sfK\in\cD(X,\lambda)$,
\begin{itemize}
  \item $\sfK$ belongs to $\TS{p}{\cD}{\le 0}(X,\lambda)$ if and only if
      \[
      i^*_{\overline{x}}j^*_{\overline{x}}\sfK\in\cD^{\le p(x)}(\overline{x},\lambda)
      \]
      for every $x\in|X|$.

  \item $\sfK$ belongs to $\TS{p}{\cD}{\ge 0}(X,\lambda)$ if and only if $\sfK\in \cD^{(+)}(X,\lambda)$ and
      \[
      i^!_{\overline{x}}j^*_{\overline{x}}\sfK\in\cD^{\ge p(x)}(\overline{x},\lambda)
      \]
      for every $x\in|X|$.
\end{itemize}
Here $\overline{x}$ is a geometric point above $x$, and we have natural morphisms
\[
i_{\overline{x}}\colon\overline{x}\to X_{(\overline{x})},\quad j_{\overline{x}}\colon X_{(\overline{x})}\to X.
\]
We will omit $j_{\overline{x}}^*$ from the notation when no confusion arises.

\begin{lem}\label{c3le:gabber}
If $p$ is a weak perversity function, then $(\TS{p}{\cD}{\leq 0}(X,\lambda),\TS{p}{\cD}{\geq 0}(X,\lambda))$ is a t-structure on $\cD(X,\lambda)$. Moreover,
\begin{enumerate}
  \item this t-structure is accessible;

  \item this t-structure is weakly left complete if $p$ takes values in 
      $\dZ$; 

  \item this t-structure is right complete;

  \item this t-structure is left complete if $p$ is locally bounded and 
      every quasi-compact closed open subscheme of $X$ is 
      $\lambda$-cohomologically finite. Here, we say that a scheme $Y$ is 
      \emph{$\lambda$-cohomologically finite} if there exists an integer 
      $n$ such that, for every $\xi\in\Xi$, the 
      $\Lambda(\xi)$-cohomological dimension of the \'etale topos of $Y$ is 
      at most $n$. 
\end{enumerate}
\end{lem}

\begin{proof}
The fact that $(\TS{p}{\cD}{\leq 0}(X,\lambda),\TS{p}{\cD}{\geq 
0}(X,\lambda))$ is a t-structure is a theorem of Gabber \cite{Gabber} when 
$\Xi$ is a singleton. This generalizes easily to the case of general $\Xi$ as 
follows. By \cite{HA}*{Proposition 1.4.4.11}, there exists a t-structure 
$(\TS{p}{\cD}{\leq 0}(X,\lambda),\cD')$ on $\cD(X,\lambda)$. For 
$\sfK\in\TS{p}{\cD}{\leq 0}(X,\lambda)$ and $\sfL\in\TS{p}{\cD}{\geq 
0}(X,\lambda)$, we have $a_*\HOM(\sfK,\sfL[1])\in\cD^{\ge1}(\ast,\lambda)$, 
hence $\Hom(\sfK,\sfL[1])=\rH^0(\Xi,a_*\HOM(\sfK,\sfL[1]))=0$, where $a\colon 
X_{\et}\to\ast$ is the morphism of topoi. Thus, we have $\TS{p}{\cD}{\geq 
0}(X,\lambda)\subseteq \cD'$. For every $\xi\in\Xi$, the functor $\rL 
e_{\xi!}\colon\cD(X,\Lambda(\xi))\to\cD(X,\lambda)$ is left t-exact for the 
t-structures 
$(\TS{p}{\cD}{\leq0}(X,\Lambda(\xi)),\TS{p}{\cD}{\geq0}(X,\Lambda(\xi)))$ and 
$(\TS{p}{\cD}{\leq 0}(X,\lambda),\cD')$. It follows that $e_\xi^*$ is right 
t-exact for the same t-structures. Thus, we have 
$\cD'\subseteq\TS{p}{\cD}{\geq 0}(X,\lambda)$ as well. 

For the properties, (1) and (2) follow from the definition directly; (3) follows from \cite{Gabber}*{Lemma 3.1}; and (4) follows from Lemma \ref{c3le:wlc}.
\end{proof}

Now we define t-structures for stacks associated to perversity evaluations. Let $X$ be a $\Box$-coprime higher Artin (resp.\ a higher Deligne--Mumford) stack equipped with a perversity smooth (resp.\ \'{e}tale) evaluation $\sfp$ (Definition \ref{c3de:perversity_evaluation}), and let $\lambda$ be an object of $\Rind_{\ltor}$ (resp.\ $\Rind$). For an atlas (resp.\ \'{e}tale atlas) $u\colon X_0\to X$ with $X_0$ a scheme in $\Schqcs$, we denote by $\TS{\sfp}{\cD}{\leq0}_u(X,\lambda)\subseteq\cD(X,\lambda)$ (resp.\ $\TS{\sfp}{\cD}{\geq0}_u(X,\lambda)\subseteq\cD(X,\lambda)$) the full subcategory spanned by complexes $\sfK$ such that $u^*\sfK$ is in $\TS{\sfp_u}{\cD}{\leq0}(X_0,\lambda)$ (resp.\ $\TS{\sfp_u}{\cD}{\geq0}(X_0,\lambda)$).

\begin{proposition}\label{c3pr:perverse_t}
Let $X$ be a $\Box$-coprime higher Artin (resp.\ a higher Deligne--Mumford) stack equipped with a perversity smooth (resp.\ \'{e}tale)
evaluation $\sfp$, and let $\lambda$ be an object of $\Rind_{\ltor}$ (resp.\ $\Rind$). Then
\begin{enumerate}
  \item The pair of subcategories $(\TS{\sfp}{\cD}{\leq0}_u(X,\lambda),\TS{\sfp}{\cD}{\geq0}_u(X,\lambda))$ do not depend on the choice of $u$. We will denote them by $(\TS{\sfp}{\cD}{\leq0}(X,\lambda),\TS{\sfp}{\cD}{\geq0}(X,\lambda))$.

  \item The pair of subcategories 
      $(\TS{\sfp}{\cD}{\leq0}(X,\lambda),\TS{\sfp}{\cD}{\geq0}(X,\lambda))$ 
      determine a right complete accessible t-structure on 
      $\cD(X,\lambda)$, which is weakly left complete if $\sfp$ takes 
      values in $\dZ$. This t-structure is left complete if $\sfp$ is 
      locally bounded and if for every smooth (resp.\ \'etale) morphism 
      $X_0\to X$ with $X_0$ a quasi-compact separated scheme, $X_0$ is 
      $\lambda$-cohomologically finite. 

  \item If $f\colon Y\to X$ is a smooth (resp.\ \'etale) morphism, then $f^*\colon \cD(X,\lambda)\to\cD(Y,\lambda)$ is t-exact with respect to the t-structures associated to $\sfp$ and $f^*\sfp$.
\end{enumerate}
\end{proposition}

\begin{proof}
There exists $k\ge -2$ such that $X$ and $Y$ are in $\Chpar{k}$ (resp.\ 
$\Chpdm{k}$). We proceed by induction on $k$. The case $k=-2$ follows from 
Lemma \ref{c3le:gabber} and Lemma \ref{c3le:perverse_smooth_pullback} below. 
The induction step follows the same proof as in Lemma \ref{b4le:t_structure} 
and Lemma \ref{b4le:p6}. 
\end{proof}

\begin{lem}\label{c3le:perverse_smooth_pullback}
Let $f\colon Y\to X$ be a smooth morphism of schemes in $\Schqcs_\Box$, let 
$\lambda$ be an object of $\Rind_{\ltor}$, and let 
$p\colon|X|\to\dZ\cup\{+\infty\}$ be a function. Then $f^!$ carries 
$\TS{p}{\cD}{\ge 0}(X,\lambda)$ to $\TS{f^*_{\b{2}}p}{\cD}{\ge 
0}(Y,\lambda)$. Moreover, if $p$ is a weak perversity function on $X$ and $q$ 
is a weak perversity function on $Y$ satisfying $f^*_{\b{0}}p\le q \le 
f^*_{\b{2}}p+2\dim f$, then $f^*\colon\cD(X,\lambda)\to\cD(Y,\lambda)$ is 
t-exact with respect to the t-structures associated to $p$ and $q$. 
\end{lem}

\begin{proof}
The first assertion follows from Lemma \ref{c3le:perverse_smooth_pullback0} below. The second assertion follows from the first assertion and the Poincar\'e duality $f^!\simeq f^*\langle\dim f\rangle$.
\end{proof}

\begin{lem}\label{c3le:perverse_smooth_pullback0}
Let $f\colon Y\to X$ be a smooth morphism in $\Schqcs_\Box$, and $\lambda$ an object of $\Rind_{\ltor}$. Let $\overline{y}$ be a geometric point of $Y$ above $y$; put $\overline{x}=f(\overline{y})$ and $x=f(y)$. Then there is an equivalence of functors
\[
i_{\overline{y}}^!\circ f^!\simeq g^*\circ i_{\overline{x}}^!\langle d\rangle\colon\cD^{(+)}(X,\lambda)\to\cD^+(\overline{y},\lambda),
\]
where $g\colon\overline{y}\to\overline{x}$ is the induced morphism and $d=\trdeg[k(y):k(x)]$.
\end{lem}

\begin{proof}
Consider the diagram with Cartesian squares
\[
\xymatrix{
\overline{y}\ar[rrd]_g\ar[r]^{i_{\overline{y}}} & V\ar[r]^{j} & Y_{\overline x}\ar[d]_{f_{\overline
x}}\ar[r]^{i'_{\overline x}} & Y_{\overline{\{x\}}} \ar[r]^{i'}\ar[d]^{f_{\overline{\{x\}}}}
& Y\ar[d]^f \\
&&\overline x \ar[r]^{i_{\overline x}} & \overline{\{x\}}\ar[r]^i & X}
\]
where $V$ is a regular integral subscheme of $Y_{\overline{x}}$ such that the image of $\overline y$ in $V$ is a generic point. We have a sequence of equivalences of functors
\[
i_{\overline y}^!\circ f^! \simeq i_{\overline y}^*\circ j^! \circ i_{\overline x}^{\prime*} \circ {i'}^!\circ f^!
\simeq i_{\overline y}^*\circ j^! \circ i_{\overline x}^{\prime*} \circ f_{\overline{\{x\}}}^! \circ i^!
\simeq i_{\overline y}^*\circ j^! \circ f_{\overline x}^! \circ i_{\overline x}^* \circ i^!
\]
which, by the Poincar\'{e} duality, is equivalent to
\[
i_{\overline y}^*\circ (f_{\overline x}\circ j)^! \circ i_{\overline x}^!
\simeq i_{\overline y}^*\circ (f_{\overline x}\circ j)^* \circ i_{\overline x}^! \langle d\rangle
\simeq g^*\circ i_{\overline x}^!\langle d\rangle.
\]
The lemma follows.
\end{proof}

\begin{remark}\label{c3re:stalk}
We call the t-structure in Proposition \ref{c3pr:perverse_t} the \emph{perverse t-structure} with respect to $\sfp$ and denote by $\TS{\sfp}{\tau}{\leq0}$ and $\TS{\sfp}{\tau}{\geq0}$ the corresponding truncation functors, respectively.
\begin{enumerate}
  \item For every (\'{e}tale) atlas $u\colon X_0\to X$ with $X_0$ a scheme in $\Schqcs$, we have $u^*\circ\TS{\sfp}{\tau}{\leq0}\simeq \TS{\sfp_u}{\tau}{\leq0}\circ u$ and $u^*\circ\TS{\sfp}{\tau}{\geq0}\simeq \TS{\sfp_u}{\tau}{\geq0}\circ u$.

  \item If $\sfp=0$, then we recover the usual t-structure. If $X$ is a higher Deligne-Mumford stack and $\sfp$ is a perversity smooth evaluation, then the t-structure associated to $\sfp$ coincides with the t-structure associated to $\sfp\res\Voet(X)$. If $X$ is in $\Schqcs$, then the t-structure associated to $\sfp$ coincides with the t-structure defined by Gabber (as in Lemma \ref{c3le:gabber}) associated to the function $\sfp_{\id_X}$.

  \item Let $\sfK$ be a complex in $\cD(X,\lambda)$. Then by definition,
      \begin{itemize}
        \item $\sfK$ belongs to $\TS{\sfp}{\cD}{\leq n}(X,\lambda)$ if and only if for every pointed smooth (resp.\ \'etale) schematic neighborhood $(X_0,u_0,x_0)$ of $X$ and a geometric point $\overline{x_0}$ lying over $x_0$, we have $i_{\overline{x_0}}^*u_0^*\sfK\in\cD^{\leq\sfp(X_0,u_0,x_0)+n}(\overline{x_0},\lambda)$.

        \item $\sfK$ belongs to $\TS{\sfp}{\cD}{\geq n}(X,\lambda)$ if and only if $\sfK\in\cD^{(+)}(X,\lambda)$, and for every pointed smooth (resp.\ \'etale) schematic neighborhood $(X_0,u_0,x_0)$ of $X$ and a geometric point $\overline{x_0}$ lying over $x_0$, we have $i_{\overline{x_0}}^!u_0^*\sfK\in\cD^{\geq\sfp(X_0,u_0,x_0)+n}(\overline{x_0},\lambda)$.
      \end{itemize}
\end{enumerate}
\end{remark}

At the end of the section, we study the restriction of perverse t-structures 
constructed above to various subcategories of constructible complexes. We fix 
a $\Box$-coprime base scheme $\dS$ that is a disjoint union of excellent 
schemes, endowed with a global dimension function. 

\begin{proposition}
Let $\lambda=(\Xi,\Lambda)$ be an object of $\Rind_{\Box\text{-}\r{dual}}$. Let $f\colon X\to\dS$ be an object of $\Chpars$ equipped with an \emph{admissible} perversity smooth evaluation $\sfp$ (Definition \ref{c3de:perversity_evaluation}). Then the truncation functors $\TS{\sfp}{\tau}{\leq0}$, $\TS{\sfp}{\tau}{\geq0}$ preserve the full subcategory $\cD_{\cons}^{(\rb)}(X,\lambda)$. Moreover, if $\sfp$ is locally bounded, then $\TS{\sfp}{\tau}{\leq0}$, $\TS{\sfp}{\tau}{\geq0}$ preserve $\cD_{\cons}^{?}(X,\lambda)$ for $?=(+),(-)$ or empty.
\end{proposition}

\begin{proof}
We reduce easily to the case of a scheme. In this case, the result is essentially \cite{Gabber}*{Theorem 8.2}.
\end{proof}

\subsection{Adic perverse t-structures}
\label{c3ss:adic_perverse_t}

For perverse t-structures in the adic formalism, we define
\[
\TS{\sfp}{\cD}{\leq n}(X,\lambda)_\ra=\TS{\sfp}{\cD}{\leq n}(X,\lambda)\cap\cD(X,\lambda)_\ra,\quad
\TS{\sfp}{\cD}{\geq n}(X,\lambda)_\ra=\TS{\sfp}{\cD}{\leq n-1}(X,\lambda)_\ra^{\perp}
\]
both as full subcategories of $\cD(X,\lambda)_\ra$. Then the pair $(\TS{\sfp}{\cD}{\leq 0}(X,\lambda)_\ra,\TS{\sfp}{\cD}{\geq 0}(X,\lambda)_\ra)$ define a t-structure, called the \emph{adic perverse t-structure} with respect to $\sfp$, on $\cD(X,\lambda)_\ra$. Denote
$\TS{\sfp}{\tau}{\leq0}_\ra$ and $\TS{\sfp}{\tau}{\geq0}_\ra$ the corresponding truncation functors respectively. We have the
following results.

\begin{lem}\label{c3le:adic_perverse_t_indep}
Let $X$ be a $\Box$-coprime higher Artin stack (resp.\ a higher Deligne--Mumford stack) equipped with a perversity smooth (resp.\ \'etale)
evaluation $\sfp$, and $\lambda$ an object of $\Rind_{\ltor}$ (resp.\ $\Rind$). Let $\sfK\in\cD(X,\lambda)_\ra$ be an (adic) complex. Let
$u\colon X_0\to X$ be an atlas (resp.\ \'{e}tale atlas) with $X_0$ a scheme in $\Schqcs$. Then $\sfK$ belongs to $\TS{\sfp}{\cD}{\leq n}(X,\lambda)_\ra$ (resp.\ $\TS{\sfp}{\cD}{\geq n}(X,\lambda)_\ra$) if and only if $u^{*\ra}\sfK$ belongs to $\TS{\sfp_u}{\cD}{\leq
n}(X_0,\lambda)_\ra$ (resp.\ $\TS{\sfp_u}{\cD}{\geq n}(X_0,\lambda)_\ra$).
\end{lem}

\begin{proof}
We only need to show that $u^{*\ra}$ is t-exact. By definition, we obviously have $u^{*\ra}\TS{\sfp}{\cD}{\leq n}(X,\lambda)_\ra\subseteq\TS{\sfp_u}{\cD}{\leq n}(X_0,\lambda)_\ra$. For the other direction, assume $\sfK\in\TS{\sfp}{\cD}{>n}(X,\lambda)_\ra$, that is, $\Hom(\sfL,\sfK)=0$ for every $\sfL\in\cD(X,\lambda)_\ra\cap\TS{\sfp}{\cD}{\leq n}(X,\lambda)$. By the Poincar\'{e} duality, it suffices to show that for every $\sfL'\in\cD(X_0,\lambda)_\ra\cap\TS{\sfp_u}{\cD}{\leq n-2\dim
u}(X_0,\lambda)$, we have $\Hom(\sfL',u^{!\ra}\sfK)=0$, or equivalently, $\Hom(u_{!\ra}\sfL',\sfK)=0$. This follows from the fact that $u_!$ preserves adic complexes and we have $u_!\sfL'\in\TS{\sfp}{\cD}{\leq n}(X,\lambda)$.
\end{proof}

\begin{proposition}\label{c3pr:adic_perverse_stalk}
Let $X$ be a $\Box$-coprime higher Artin stack (resp.\ a higher Deligne--Mumford stack) equipped with a perversity smooth (resp.\ \'etale)
evaluation $\sfp$, and $\lambda$ an object of $\Rind_{\ltor}$ (resp.\ $\Rind$). Let $\sfK\in\cD(X,\lambda)_\ra$ be an (adic) complex.
\begin{enumerate}
  \item Then $\sfK$ belongs to $\TS{\sfp}{\cD}{\leq n}(X,\lambda)_\ra$ if and only if for every pointed smooth (resp.\ \'etale) schematic
      neighborhood $(X_0,u_0,x_0)$ of $X$ and a geometric point $\overline{x_0}$ lying over $x_0$, we have
      $i_{\overline{x_0}}^{*\ra}u_0^{*\ra}\sfK\in\cD^{\leq\sfp(X_0,u_0,x_0)+n}(\overline{x_0},\lambda)_\ra$.

  \item Assume that $\sfp$ is \emph{locally bounded}. Then $\sfK$ belongs to $\TS{\sfp}{\cD}{\geq n}(X,\lambda)_\ra$ if and only if
      $\sfK\in\cD^{(+)}(X,\lambda)_\ra$, and for every pointed smooth (resp.\ \'etale) schematic neighborhood $(X_0,u_0,x_0)$ of
      $X$ and a geometric point $\overline{x_0}$ lying over $x_0$, we have
      $i_{\overline{x_0}}^{!\ra}u_0^{*\ra}\sfK\in\cD^{\geq\sfp(X_0,u_0,x_0)+n}(\overline{x_0},\lambda)_\ra$.
\end{enumerate}
\end{proposition}

\begin{proof}
Part (1) is a consequence of the definition and Remark \ref{c3re:stalk}(3).

For (2), by Lemma \ref{c3le:adic_perverse_t_indep}, we may assume that $X\in\Schqcs$ is quasi-compact and $\sfp=p$ is a bounded weak perversity
function. Then $\sfK\in\TS{p}{\cD}{\geq n}(X,\lambda)_\ra$ is equivalent to that for every $\sfL\in\TS{p}{\cD}{<n}(X,\lambda)_\ra$, $\HOM(\sfL,\sfK)\in\cD^{>0}(X,\lambda)$, which is then equivalent to $\HOM(\sfL,\sfK)\in\cD^+(X,\lambda)$ and $i_{\overline{x}}^!\HOM(\sfL,\sfK)\in\cD^{>0}(\overline{x},\lambda)$ for every geometric point $\overline{x}$ of $X$. By Proposition \ref{c1th:properties}(4), we have isomorphisms
\[
i_{\overline{x}}^!\HOM(\sfL,\sfK)\simeq\HOM(i_{\overline{x}}^*\sfL,i_{\overline{x}}^!\sfK)
\simeq\HOM(i_{\overline{x}}^{*\ra}\sfL,i_{\overline{x}}^{!\ra}\sfK).
\]
Now we may assume $\alpha<p<\beta$ for some $\alpha,\beta\in\dZ$ since $p$ is bounded. Then $\TS{p}{\cD}{<n}(X,\lambda)_\ra$ contains $\cD^{<\alpha+n}(X,\lambda)_\ra$.

Now for $\sfK\in\TS{p}{\cD}{\geq n}(X,\lambda)_\ra$, we have $\sfK\in\cD^{\geq\alpha+n}(X,\lambda)_\ra\subseteq\cD^+(X,\lambda)_\ra$ and
$i_{\overline{x}}^{!\ra}\sfK\in\cD^{\geq p(x)+n}(\overline{x},\lambda)_\ra$ for every geometric point $\overline{x}$ of $X$ lying over $x$.

Conversely, assume $\sfK\in\cD^+(X,\lambda)_\ra$, say in $\cD^{\geq\gamma}(X,\lambda)_\ra$, and $i_{\overline{x}}^{!\ra}\sfK\in\cD^{\geq p(x)+n}(\overline{x},\lambda)_\ra$ for every geometric point $\overline{x}$ of $X$ lying over $x$. We have $\HOM(\sfL,\sfK)\in\cD^{\geq\gamma-\beta-n}(X,\lambda)\subseteq\cD^+(X,\lambda)$ and $\HOM(i_{\overline{x}}^{*\ra}\sfL,i_{\overline{x}}^{!\ra}\sfK)\in\cD^{>0}(\overline{x},\lambda)$. Thus, we have $\sfK\in\TS{p}{\cD}{\geq n}(X,\lambda)_\ra$.
\end{proof}

\begin{remark}\label{c3re:perverse_truncation}
Let $\sfp,\sfq$ be two perversity smooth (resp.\ \'etale) evaluations on a $\Box$-coprime higher Artin stack (resp.\ a higher Deligne--Mumford stack) $X$. Let $\lambda$ be an object of $\Rind_{\ltor}$ (resp.\ $\Rind$). Let the subscript $?$ be either ``$\ra$'' or empty.
\begin{enumerate}
  \item The intersection of the pair of subcategories $(\TS{\sfp}{\cD}{\leq0}(X,\lambda)_?,\TS{\sfp}{\cD}{\geq0}(X,\lambda)_?)$ with $\cD^{(+)}(X,\lambda)_?$ induces a t-structure on the latter stable $\infty$-category.

  \item If $\sfp\leq\sfq$, then
      \begin{enumerate}
        \item $\TS{\sfp}{\tau}{\leq0}_?$ preserves $\TS{\sfq}{\cD}{\leq0}(X,\lambda)_?$;

        \item $\TS{\sfq}{\tau}{\geq0}_?$ preserves $\TS{\sfp}{\cD}{\geq0}(X,\lambda)_?$;

        \item $\TS{\sfp}{\tau}{\geq0}_?$ is equivalent to the identity functor when restricted to $\TS{\sfq}{\cD}{\geq0}(X,\lambda)_?$;

        \item $\TS{\sfq}{\tau}{\leq0}_?$ is equivalent to the identity functor when restricted to $\TS{\sfp}{\cD}{\leq0}(X,\lambda)_?$;

        \item $\TS{\sfp}{\tau}{<0}_?$ is equivalent to the null functor when restricted to $\TS{\sfq}{\cD}{\geq0}(X,\lambda)_?$;

        \item $\TS{\sfq}{\tau}{>0}_?$ is equivalent to the null functor when restricted to $\TS{\sfp}{\cD}{\leq0}(X,\lambda)_?$.
      \end{enumerate}

  \item By (2a), if $\sfp$ is locally bounded, then the intersection of the pair of subcategories $(\TS{\sfp}{\cD}{\leq0}(X,\lambda)_?,\TS{\sfp}{\cD}{\geq0}(X,\lambda)_?)$ with $\cD^{(-)}(X,\lambda)_?$ or $\cD^{(\rb)}(X,\lambda)_?$ induces a t-structure on the latter stable $\infty$-category.

  \item By (2e) and (2f), if $X$ is quasi-compact and $\sfp$ is bounded, then there exist constant integers $\alpha<\beta$ such that
      $\TS{\sfp}{\rH}{0}_?=\TS{\sfp}{\rH}{0}_?\circ\tau^{[\alpha,\beta]}_?$, where $\TS{\sfp}{\rH}{0}_?=\TS{\sfp}{\tau}{\geq0}_?\circ\TS{\sfp}{\tau}{\leq0}_?$ is the cohomology functor.
\end{enumerate}
\end{remark}

\section{Hyperdescent properties}
\label{c4}

In this chapter, we study hyperdescent properties for certain operations on stacks. In \Sec\ref{c4ss:hyperdescent}, we study some general facts for hyperdescent. In \Sec\ref{c4ss:smooth}, \Sec\ref{c4ss:proper} and \Sec\ref{c4ss:flat}, we study smooth, proper and flat hyperdescent, respectively.

\subsection{Hyperdescent}
\label{c4ss:hyperdescent}

In this section, we study hyperdescent properties in the general setup.

\begin{definition}
Let $\cC$, $\cD$ be $\infty$-categories, let $F\colon\cC^{op}\to \cD$ be a functor, and let $X_\bullet^+\colon\rN(\del_+)^{op}\to\cC$ be an augmented simplicial object of $\cC$.
\begin{enumerate}
  \item We say that $X_\bullet^+$ is an \emph{augmentation of $F$-descent} if $F\circ(X_\bullet^+)^{op}$ is a limit diagram in $\cD$.

  \item Assume that $\cC$ admits pullbacks. We say that $X_\bullet^+$ is a \emph{hypercovering for universal $F$-descent} if $X^+_q\to(\cosk_{q-1}(X^+_{\bullet}/X^+_{-1}))_q$ is a morphism of universal $F$-descent for all $q\ge 0$.
\end{enumerate}
\end{definition}

By definition, a morphism of $\cC$ is of $F$-descent (Definition \ref{b3de:descent}) if and only if its \v{C}ech nerve is an augmentation of $F$-descent. We now give several criteria for $(2)\Rightarrow(1)$.

\begin{proposition}\label{c4pr:n_category}
Let $\cC$ be an $\infty$-category admitting pullbacks, let $\cD$ be an $n$-category admitting finite limits for an integer $n\ge 0$, and let $F\colon\cC^{op}\to\cD$ be a functor. Then every hypercovering $X_\bullet^+$ for universal $F$-descent is an augmentation of $F$-descent.
\end{proposition}

To prove Proposition \ref{c4pr:n_category}, we need a few lemmas.

\begin{lem}\label{c4le:induction}
Let $\cC$, $\cD$ be $\infty$-categories such that $\cC$ admits finite limits, let $F\colon\cC^{op}\to\cD$ be a functor, and let $e$ be a final object of $\cC$. Let $f_\bullet\colon U_\bullet\to V_\bullet$ be a morphism of simplicial objects of $\cC$ such that $V_\bullet \to e$ is an augmentation of $F$-descent and $f_q$ is a morphism of $F$-descent for all $q$. Assume that there exists an integer $n\ge 0$ such that $U_\bullet$ is $n$-coskeletal, $V_\bullet$ is $(n-1)$-coskeletal, and $f_q$ is an equivalence for $q<n$. Then $U_\bullet\to e$ is an augmentation of $F$-descent.
\end{lem}

\begin{proof}
With out lost of generality, we may assume that $F(e)$ is an initial object of $\cD$. Let $W_+\colon\rN(\del_+\times\del)^{op}\to\Fun(\Delta^1,\cC)$ be a \v{C}ech nerve of $f_\bullet$, and put $W\coloneqq W_+\res\rN(\del\times\del)^{op}$. For every $q\ge 0$, $W_+\res\rN(\del_+\times\{[q]\})^{op}$ is a \v{C}ech nerve of $f_q$, which is a morphism of $F$-descent by assumption. It follows that $F\circ W_+^{op}\res\rN(\del_+\times\{[q]\})$ is a limit diagram. Thus, we may identify the limit of $F\circ W^{op}$ with the limit $F\circ W_+^{op}\res\rN(\{[-1]\}\times\del_s)$. Since $W_+\res\rN(\{[-1]\}\times\del)^{op}$ can be identified with $V_\bullet$, the limit of $F\circ W^{op}$ can be identified with $F(e)$. Put $D_\bullet\coloneqq W\circ \delta$, where $\delta\colon\rN(\del)^{op}\to\rN(\del\times\del)^{op}$ is the diagonal map. Since $\rN(\del)^{op}$ is sifted \cite{HTT}*{Lemma 5.5.8.4}, the limit of $F\circ D_\bullet^{op}$ can be identified with $F(e)$. The proof of \cite{HTT}*{Lemma 6.5.3.9} exhibits $U_\bullet\res\rN(\del_s)^{op}$ as a retract of $D_\bullet\res\rN(\del_s)^{op}$. It follows that the limit of $F\circ U_\bullet^{op}$ is a retract of $F(e)$, hence is $F(e)$. The lemma follows.
\end{proof}

\begin{lem}\label{c4le:coskeletal}
Let $\cC$, $\cD$ be $\infty$-categories such that $\cC$ admits pullbacks, let $F\colon\cC^{op}\to\cD$ be a functor, and let $X_\bullet^+$ be an
$n$-coskeletal hypercovering for universal $F$-descent for an integer $n\ge-1$. Then $X_\bullet^+$ is an augmentation of $F$-descent.
\end{lem}

\begin{proof}
Since morphisms of universal $F$-descent are stable under pullbacks and compositions, the morphism $\cosk_m(X_\bullet^+/X_{-1}^+)\to\cosk_{m-1}(X_\bullet^+/X_{-1}^+)$ satisfies the assumptions of Lemma \ref{c4le:induction}. It follows by induction that $\cosk_n(X^+_\bullet/X^+_{-1})$ is an augmentation of $F$-descent.
\end{proof}

\begin{lem}\label{c4le:n_category}
Let $n\ge -1$ be an integer, let $\cD$ be an $n$-category admitting finite colimits, and let $f_\bullet \colon Y_\bullet \to X_\bullet$ be a morphism of semisimplicial (resp.\ simplicial) objects of $\cD$ such that $Y_q\to X_q$ is an equivalence for $q\le n$. Then the induced morphism between geometric realizations $|f_\bullet|\colon|Y_\bullet|\to|X_\bullet|$ is an equivalence in $\cD$.
\end{lem}

\begin{proof}
The existence of the geometric realizations is guaranteed by \cite{HA}*{Lemma 1.3.3.10}. The semisimplicial case follows from the simplicial case by taking left Kan extensions. The simplicial case follows from the proof of \cite{HA}*{Lemma 1.3.3.10}.
\end{proof}

\begin{proof}[Proof of Proposition \ref{c4pr:n_category}]
It suffices to apply the dual version of Lemma \ref{c4le:n_category} to the morphism $h\colon X^+_\bullet \to \cosk_n (X^+_\bullet/X_{-1}^+)$ and Lemma \ref{c4le:coskeletal}.
\end{proof}

The following proposition can be used to deduce Gabber's hyper base change theorem \cite{TG}*{Expos\'{e} XIII, Th\'eor\`eme 2.2.5} (see \cite{TG}*{Expos\'{e} XII, Remark 2.3}).

\begin{proposition}\label{c4pr:t_descent}
Let $\cC$ be an $\infty$-category admitting pullbacks, let $\cD$ be a stable $\infty$-category endowed with a weakly right complete t-structure that either admits countable limits or is right complete, let $F\colon\cC^{op}\to\cD$ be a functor, and let $X_\bullet^+\colon\rN(\del_+)^{op}\to\cC$ be a hypercovering for universal $F$-descent such that $F\circ(X_\bullet^+)^{op}$ factorizes through $\cD^{\ge 0}$. Then $X_\bullet^+$ is an augmentation of $F$-descent.
\end{proposition}

\begin{proof}
Let $n\ge 0$. By Lemma \ref{c4le:coskeletal}, $Y_\bullet^+=\cosk_n(X_\bullet^+/X_{-1}^+)$ is an augmentation of $F$-descent, so that it suffices to show that the morphism
\[
c\colon K\coloneqq\lim_{p\in\del}F(X_p)\to L\coloneqq\lim_{p\in\del}F(Y_p)
\]
induced by $h_\bullet\colon X_\bullet^+\to Y_\bullet^+$ is an isomorphism. By \cite{HA}*{Remark 1.2.4.4, Proposition 1.2.4.5}, we have a morphism of converging spectral sequences
\[
\xymatrix{
\rE_1^{p,q}=\rH^q F(X_p)\ar@{=>}[r]\ar[d]_-{c_1^{p,q}} & \rH^{p+q} K\ar[d]^-{\rH^{p+q} c}\\
\TS{\prime}{\rE}{p,q}_1=\rH^q F(Y_p)\ar@{=>}[r] & \rH^{p+q} L,}
\]
concentrated in the first quadrant. For $p\le n$, since $h_p$ is an equivalence, $c_1^{p,q}$ is an isomorphism for all $q$. It follows that
$c_r^{p,q}$ is an isomorphism for $p+q\le n-1$, and $\tau^{\le n-1}c$ is an equivalence. Since $n$ is arbitrary and $\cD$ is weakly right complete, $c$ is an equivalence.
\end{proof}

We denote by $\PSLt$ (resp.\ $\PSRt$) the $\infty$-category defined as follows:
\begin{itemize}
  \item Objects of $\PSLt$ (resp.\ $\PSRt$) are presentable stable $\infty$-categories equipped with a t-structure.

  \item Morphisms of $\PSLt$ (resp.\ $\PSRt$) are t-exact functors admitting right (resp.\ left) adjoints.
\end{itemize}
The $\infty$-categories $\PSLt$ (resp.\ $\PSRt$) admit small limits, and those limits are preserved by the forgetful functor $\PSLt\to\PSL$ (resp.\ $\PSRt\to \PSR$). For a diagram $K\to\PSLt$ or $K\to\PSRt$, $(\lim\cC_k)^{\le 0}$ (resp.\ $(\lim\cC_k)^{\ge 0}$) is the full subcategory of $\lim\cC_k$ spanned by objects whose image in $\cC_k$ is in $\cC_k^{\le 0}$ (resp.\ $\cC_k^{\ge 0}$). For an interval $I\subseteq\dZ$, we have an equivalence $(\lim\cC_k)^{\in I}\to \lim \cC_k^{\in I}$.

We denote by $\PSLt[,\r{wrc}]$ (resp.\ $\PSRt[,\r{rc},\r{wlc}]$) the full subcategory of $\PSLt$ (resp.\ $\PSRt$) spanned by those $\cC$ that are weakly right complete (resp.\ right complete and weakly left complete). This full subcategory is stable under small limits in $\PSLt$ (resp.\ $\PSRt$).

\begin{proposition}\label{c4pr:hyperdescent}
Consider a diagram
\[
\xymatrix{
\cD'^{op}\ar[d]_-{j^{op}}\ar[r]^-F & \PSRt[,\r{rc},\r{wlc}]\ar[d]^-{P}\\
\cD^{op}\ar[r]^-{G} & \Cat}
\]
of $\infty$-categories, in which $\cD$ admits pullbacks, $j$ is an inclusion satisfying the right lifting property with respect to $\partial\Delta^n\subseteq\Delta^n$ for $n\ge 2$, and $P$ is the forgetful functor. Assume that the arrows in $\cD'$ are stable under pullbacks in $\cD$ by arrows in $\cD'$. Let $X_\bullet^+\colon\rN(\del_+)^{op}\to\cD$ be a hypercovering for universal $G$-descent such that $X_\bullet^+\res\rN(\del_{s+})^{op}$ factorizes through $j$. Then $X_\bullet^+$ is an augmentation of $G$-descent.
\end{proposition}

\begin{proof}
By the right completeness of $F(X^+_p)$ for $p\ge -1$, it suffices to show that $(F\circ(X_\bullet^+)^{op}\res\rN(\del_{s+}))^{\le 0}$ is a limit diagram. Put $\cC=\lim(F\circ (X_\bullet^+)^{op}\res\rN(\del_s))$ for simplicity. We then have the induced t-exact functor $f^*\colon F(X^+_{-1})\to\cC$. Let $f_!\colon\cC\to F(X^+_{-1})$ be a left adjoint of $f^*$. The restrictions of these provide adjoint functors
\[
(f_!)^{\le 0}\colon\cC^{\le 0}\to F(X^+_{-1})^{\le 0},\quad (f^*)^{\le 0}\colon F(X^+_{-1})^{\le 0}\to\cC^{\le 0}.
\]
Let us first show that $a\colon f_! f^*K \to K$ is an equivalence for all $K\in F(X^+_{-1})^{\le 0}$, namely, that $(f^*)^{\le 0}$ is fully faithful. This is similar to Proposition \ref{c4pr:t_descent}. Take $n\ge 0$. The morphism $h_\bullet\colon X^+_\bullet\to\cosk_n(X^+_\bullet/X^+_{-1})=Y^+_\bullet$ induces a diagram
\[
\xymatrix{f_!f^* K \ar[rr]^-{c}\ar[rd]_-{a} && g_!g^* K\ar[dl]^-{b}\\
&K}
\]
where $g_!$ is a left adjoint of the t-exact functor $g^*\colon F(X^+_{-1})\to \lim F\circ (Y^+_\bullet)^{op}\res\rN(\del_s)$. By Lemma
\ref{c4le:coskeletal}, $Y^+_\bullet$ is an augmentation of $G$-descent, so that $b$ is an equivalence. Moreover, we have $c=\colim (f_{p!} f_p^* K\to g_{p!}g_p^*K)$, where $f_{p!}$ is a left adjoint of $f_p^*\colon F(X^+_{-1})\to F(X^+_{p})$, $g_{p!}$ is a left adjoint of $g_p^*\colon
F(Y^+_{-1})\to F(Y^+_{p})$, and $f_{p!}f_p^*\to g_{p!}g_p^*$ is induced by $h_p$. By \cite{HA}*{Remark 1.2.4.4, Proposition 1.2.4.5}, we have a morphism of converging spectral sequences
\[
\xymatrix{
\rE_1^{p,q}=\rH^q(f_{-p!}f_{-p}^* K)\ar@{=>}[r]\ar[d]_-{c_1^{p,q}} & \rH^{p+q} f_! f^* K\ar[d]^-{\rH^{p+q} c}\\
\TS{\prime}{\rE}{p,q}_1=\rH^q(g_{-p!}g_{-p}^* K)\ar@{=>}[r] & \rH^{p+q} g_! g^* K,}
\]
concentrated in the third quadrant. For $p\ge -n$, since $h_p$ is an equivalence, $c_1^{p,q}$ is an isomorphism for all $q$. It follows that
$c_r^{p,q}$ is an isomorphism for $p+q\ge 1-n$, and $\tau^{\ge 1-n}c$ is an equivalence. Therefore, $\tau^{\ge 1-n} a$ is an equivalence. Since $n$ is arbitrary and $F(X^+_{-1})$ is weakly left complete, $a$ is an equivalence.

It remains to show that $d\colon L\to f^*f_! L$ is an equivalence for every $L\in \cC^{\le 0}$. Since $\cC$ is weakly left complete, it suffices to show that $\tau^{\ge 1-n}d$ is an equivalence for every $n \ge 1$. For this, we may assume $L\in \cC^{[1-n,0]}$. We will show that $L$ is in the essential image of $(f^*)^{\le 0}$. Since $(f^*)^{\le 0}$ is fully faithful, this proves that $d$ is an equivalence. Let $H\colon\PSRt[,\r{rc},\r{wlc}]\to\c{C}\r{at}_n$ be the functor sending $\cF$ to $\cF^{[1-n,0]}$, where $\c{C}\r{at}_n$ is the $\infty$-category of $n$-categories. It suffices to show that $H\circ F\circ (X^+_\bullet)^{op}\res\rN(\del_{s+})$ is a limit diagram. Since $\c{C}\r{at}_n$ is an $(n+1)$-category, we may assume that $X^+_\bullet/X^+_{-1}$ is $(n+1)$-coskeletal by Lemma \ref{c4le:n_category} applied to $X^+_\bullet\to\cosk_{n+1}(X^+_\bullet/X^+_{-1})$. In this case, $F\circ (X^+_\bullet)^{op}\res\rN(\del_{s+})$ is a limit diagram by Lemma \ref{c4le:coskeletal}.
\end{proof}

The following variant of Proposition \ref{c4pr:hyperdescent} will be used to establish proper hyperdescent. To state it conveniently, we introduce a bit of terminology. Let $\cC$ be an $\infty$-category admitting pullbacks, and $F\colon\cC^{op}\to \Cat$ a functor. We say that a morphism $f$ of $\cC$ is \emph{$F$-conservative} if $F(f)$ is conservative. We say that $f$ is \emph{universally $F$-conservative} if every pullback of $f$ in $\cC$ is $F$-conservative. We say that an augmented simplicial object $X_\bullet^+$ of $\cC$ is a \emph{hypercovering for universal $F$-conservativeness} if $X^+_n\to (\cosk_{n-1}(X^+_\bullet/X^+_{-1}))_n$ is universally $F$-conservative for every $n\ge 0$.

\begin{proposition}\label{c4pr:hyperdescent2}
Let $\cC$ be an $\infty$-category admitting pullbacks, let $F\colon\cC^{op}\to\PSLt[,\r{wrc}]$ be a functor, and let $a$ be an integer.
\begin{enumerate}
  \item Let $G\colon\PSLt[,\r{wrc}]\to\Cat$ be the functor sending $\cC$ to $\cC^{\ge a}$. If $X_\bullet^+$ is a hypercovering for universal $(G\circ F)$-descent, then it is an augmentation of $(G\circ F)$-descent.

  \item Let $G\colon\PSLt[,\r{wrc}]\to\Cat$ be the functor sending $\cC$ to $\cC^+\coloneqq\bigcup_n \cC^{\ge n}$. If $X_\bullet^+$ is a hypercovering for universal $(G\circ F)$-descent and for universal $(P\circ F)$-conservativeness, where $P\colon \PSLt[,\r{wrc}]\to \Cat$ is the forgetful functor, then it is an augmentation of $(G\circ F)$-descent.
\end{enumerate}
\end{proposition}

\begin{proof}
The proof for (1) is similar to the proof of Proposition \ref{c4pr:hyperdescent}. For (2), the conservativeness implies that $G(\lim F\circ (X^+_\bullet)^{op})\to \lim G\circ F\circ(X_\bullet^+)^{op}$ is an equivalence. The rest of the proof is similar.
\end{proof}

\subsection{Smooth hyperdescent}
\label{c4ss:smooth}

The \'etale $\infty$-topos of an affine scheme is not hypercomplete (see \cite{HTT}*{{\Sec}6.5.2} for the definition) in general. By contrast, the stable $\infty$-categories we constructed satisfy smooth hyperdescent.

We regard the map
\begin{align*}
\EO{}{\Chpar{}}\otimes{}\coloneqq(\EO{}{\Chpar{}}{\r{I}}{})^\otimes\colon\rN(\Chpar{})^{op}\times\rN(\Rind)^{op}\to\PSLM
\end{align*}
and the map
\begin{align*}
\EO{}{\Chpar{}_\Box}{}{!}\colon \rN(\Chpar{}_\Box)_F\times\rN(\Rind_{\ltor})^{op}\to\PSL
\end{align*}
from \Sec\ref{b5ss:artin} as functors
\begin{align*}
\EO{}{\Chpar{}}\otimes{}&\colon\rN(\Chpar{})^{op}\to\Fun(\rN(\Rind)^{op},\PSLM),\\
\EO{}{\Chpar{}_\Box}{}{!}&\colon\rN(\Chpar{}_\Box)_F\to\Fun(\rN(\Rind_{\ltor})^{op},\PSL).
\end{align*}
In the adic case, we have similar functors
\begin{align*}
\EO{\ra}{\Chpar{}}\otimes{}&\colon\rN(\Chpar{})^{op}\to\Fun(\rN(\Rind)^{op},\PSLM),\\
\EO{\ra}{\Chpar{}_\Box}{}{!}&\colon\rN(\Chpar{}_\Box)_F\to\Fun(\rN(\Rind_{\ltor})^{op},\PSL).
\end{align*}
from Proposition \ref{c1pr:monoidal} and \eqref{c1eq:lowersh_adic}, respectively.

\begin{definition}\label{c4de:covering}
We say that an augmented simplicial object $X_\bullet^+$ in $\Chpar{}$ (or similar $\infty$-categories) is a \emph{(P) hypercovering} for a property (P) on morphisms if $X^+_q \to (\cosk_{q-1}(X^+_{\bullet}/X^+_{-1}))_q$ is \emph{surjective} and satisfies (P) for every $q\ge 0$.
\end{definition}

\begin{proposition}\label{c4pr:hyperdescent_stack}
Every smooth hypercovering in $\Chpar{}$ (resp.\ $\Chpar{}_\Box$) is an augmentation of both $\EO{}{\Chpar{}}\otimes{}$-descent (resp.\
$\EO{}{\Chpar{}_\Box}{op}{!}$-descent) and $\EO{\ra}{\Chpar{}}\otimes{}$-descent (resp.\ $\EO{\ra}{\Chpar{}_\Box}{op}{!}$-descent).
\end{proposition}

\begin{proof}
Let $X_\bullet^+$ be an augmented simplicial object of $\Chpar{}$ (resp.\ $\Chpar{}_\Box$). It suffices to apply Proposition \ref{c4pr:hyperdescent} to the full subcategory $\Chpar{}_{\r{sm}/X_{-1}}\subseteq \Chpar{}_{/X_{-1}}$ spanned by higher Artin stacks smooth over $X_{-1}$. In the notation of Proposition \ref{c4pr:hyperdescent}, $F$ associates the usual t-structure (resp.\ the usual t-structure shifted by twice the relative dimension over $X_{-1}$). This proof applies to both the non-adic case and the adic case. The adic case can also be deduced from the non-adic case by taking limits.
\end{proof}

\subsection{Proper hyperdescent}
\label{c4ss:proper}

In this section, we study hyperdescent properties for proper morphisms. We start from some lemmas for preparation.

\begin{lem}\label{c4le:real}
Let $\cC$ and $\cD$ be stable $\infty$-categories equipped with left complete t-structures. Let $F\colon \cC\to \cD$ be a t-exact functor. Then
$\cC^{\le 0}$ admits geometric realizations, and geometric realizations are preserved by $F$.
\end{lem}

\begin{proof}
By \cite{HA}*{Proposition 1.2.4.5}, for any simplicial object $X_\bullet$ of $\cC$, there exist a geometric realization $X=\lvert X_\bullet \rvert$ in $\cC$ and a geometric realization $Y=\lvert FX_\bullet\rvert$ in $\cD$, and $\rH^n(f)$ is an isomorphism for all $n$,  where $f$ is the morphism $Y\to FX$. It follows that $f$ is an equivalence.
\end{proof}

\begin{lem}\label{c4le:split}
Let $\cC$, $\cD$, $\cE$ be stable $\infty$-categories equipped with 
t-structures such that $\cC$ and $\cD$ are both left and right complete. Let 
$F\colon \cC\to \cD$ and $G\colon \cC\to \cE$ be t-exact functors. Assume $G$ 
conservative. Then $\cC$ admits $G$-split \cite{HA}*{Definition 4.7.2.2} 
geometric realizations, and those geometric realizations are preserved by 
$F$. 
\end{lem}

\begin{proof}
Let $X_\bullet$ be a $G$-split simplicial object of $\cC$, and $Y_\bullet\colon\rN(\del_+)^{op}\to\cD$ a split augmentation of $G\circ
X_\bullet$. Then the unnormalized cochain complex
\[
\dots \to \rH^q Y_2\to \rH^q Y_1 \to \rH^q Y_0\to \rH^q Y_{-1} \to 0
\]
is acyclic. Since $G$ is conservative, it follows that the unnormalized cochain complex
\[
\dots \to \rH^q X_2\to \rH^q X_1 \xrightarrow{\theta^q} \rH^q X_0
\]
is an acyclic resolution of the object $A^q=\r{coker}(\theta^q)$ in the heart of $\cC$ and the same holds after applying the functor $F$. By \cite{HA}*{Corollary 1.2.4.12}, $X_\bullet$ admits a geometric realization $X$, $FX_\bullet$ admits a geometric realization $Z$, and $\rH^n(f)$ is an isomorphism for all $n$, where $f$ is the morphism $Z\to FX$. It follows that $f$ is an equivalence.
\end{proof}

The functor $\EO{}{\Chpar{}}\otimes{}$ restricts to a functor
\[
\EO{\geq 0}{\Chpar{}_\Box}{*}{}\colon\rN(\Chpar{}_\Box)^{op}\to\Fun(\rN(\Rind_{\ltor})^{op},\Cat)
\]
sending $X$ to the assignment $\lambda\mapsto\cD^{\geq 0}(X,\lambda)$.

\begin{proposition}\label{c4pr:proper_descent}
Let $\dS$ be a $\Box$-coprime (resp.\ $\Box$-coprime \emph{locally Noetherian}, that is, there exists an atlas $S\to\dS$ where $S$ is a locally
Noetherian scheme) higher Artin stack.
\begin{enumerate}
  \item For every object $\lambda$ of $\Rind_{\ltor}$ and every Cartesian square
      \[
      \xymatrix{
      W \ar[r]^-{g} \ar[d]_-{q} & Z \ar[d]^-{p} \\
      Y \ar[r]^-{f} & X }
      \]
      in $\Chpar{}_\Box$ (resp.\ $\Chpars$) with $p$ proper of finite diagonal (resp.\ proper and $1$-Artin), the induced square
      \[
      \xymatrix{
      \cD^{\ge 0}(Z,\lambda)  \ar[d]_-{g^*} & \cD^{\ge 0}(X,\lambda) \ar[l]_-{p^*} \ar[d]^-{f^*}\\
      \cD^{\ge 0}(W,\lambda) & \cD^{\ge 0}(Y,\lambda) \ar[l]_-{q^*} }
      \]
      is right adjointable.

  \item Every proper finite-diagonal hypercovering in $\Chpar{}_\Box$ (resp.\ proper and $1$-Artin hypercovering in $\Chpars$) is an augmentation of $\EO{\geq 0}{\Chpar{}_\Box}{*}{}$-descent.
\end{enumerate}
\end{proposition}

\begin{proof}
Let us first show that (1) implies (2). By Proposition 
\ref{c4pr:hyperdescent2}, to show (2), it suffices to show that every 
surjective morphism proper of finite diagonal (resp.\ proper and $1$-Artin) 
is of $\EO{\geq 0}{\Chpar{}_\Box}{*}{}$-descent. For this, we apply 
\cite{HA}*{Corollary 4.7.5.3}: Assumption (1) follows from the dual of Lemma 
\ref{c4le:real}; Assumption (2) is simply part (1); and the conservativeness 
is clear. 

To show (1), applying Proposition \ref{b4pr:p5} and the smooth base change, we are reduced to the case where $X$ and $Y$ are in $\Schqcs$. In this case, there exists a finite \cite{Rydh}*{Theorem~B} (resp.\ proper \cite{OlChow}*{Theorem 1.1}) surjective morphism $r_0\colon Z_0\to Z$ with $Z_0$ a scheme. Since (1) is known in the case where $p$ is proper and $0$-Artin, $r_0$ is $\EO{\geq 0}{\Chpar{}_\Box}{*}{}$-descent by the above proof of (2). Thus, every object of $\cD^{\ge 0}(Z,\lambda)$ has the form $\lim_{n\in \del}r_{n*}r_n^*\sfK$, where $r_\bullet$ is a \v{C}ech nerve of $r_0$. By Lemma \ref{c4le:real}, the functors $f^*$ and $g^*$ preserve limits indexed by $\del$. Thus, it suffices to check that the natural transformation $f^*\circ p_*\circ r_{n*}\to q_*\circ g^*\circ r_{n*}$ is a natural equivalence. This follows from the known cases of (1) with $p$ replaced by the proper $0$-Artin morphisms $r_n$ and $p\circ r_n$.
\end{proof}

The above result can be extended to $\cD(X,\lambda)^\otimes$ under cohomological finiteness conditions. We fix an object $\lambda$ of $\Rind_{\ltor}$. The functors $\EO{}{\Chpar{}}\otimes{}$ and $\EO{\ra}{\Chpar{}}\otimes{}$ restrict to functors
\begin{align*}
\EO{}{\Chpar{}_\Box}{\otimes}{\lambda}&\colon\rN(\Chpar{}_\Box)^{op}\to\PSLM,\\
\EO{\ra}{\Chpar{}_\Box}{\otimes}{\lambda}&\colon\rN(\Chpar{}_\Box)^{op}\to\PSLM
\end{align*}
sending $X$ to $\cD(X,\lambda)^\otimes$ and $\cD(X,\lambda)_\ra^\otimes$, respectively.

\begin{proposition}\label{c4pr:proper_descent2}
Let $\dS$ be a $\Box$-coprime (resp.\ $\Box$-coprime \emph{locally Noetherian}) higher Artin stack. Let $\lambda$ be an object of $\Rind_{\ltor}$.
\begin{enumerate}
  \item Consider a Cartesian square
      \[
      \xymatrix{
      W \ar[r]^-{g} \ar[d]_-{q} & Z \ar[d]^-{p} \\
      Y \ar[r]^-{f} & X
      }
      \]
      in $\Chpar{}_\Box$ (resp.\ $\Chpars$) with $p$ proper of finite diagonal (resp.\ proper and $1$-Artin). Assume that for every morphism $U\to X$ locally of finite type with $U$ an affine scheme, $X_0$ is $\lambda$-cohomologically finite. Then the induced square
      \[
      \xymatrix{
      \cD(Z,\lambda)  \ar[d]_-{g^*} & \cD(X,\lambda) \ar[l]_-{p^*} \ar[d]^-{f^*}\\
      \cD(W,\lambda) & \cD(Y,\lambda) \ar[l]_-{q^*} }
      \]
      is right adjointable.

  \item Let $X^+_\bullet$ be a proper finite-diagonal hypercovering in $\Chpar{}_\Box$ (resp.\ proper and $1$-Artin hypercovering in $\Chpars$). Assume that for every morphism $U\to X^+_{-1}$ locally of finite type with $U$ an affine scheme, $X_0$ is $\lambda$-cohomologically finite. Then $X^+_\bullet$ is an augmentation of both $\EO{}{\Chpar{}_\Box}{\otimes}{\lambda}$-descent and
      $\EO{\ra}{\Chpar{}_\Box}{\otimes}{\lambda}$-descent.
\end{enumerate}
\end{proposition}

\begin{proof}
We first show that (1) implies (2) for $\EO{}{\Chpar{}_\Box}{\otimes}{\lambda}$-descent. One only needs to repeat the proof of Proposition \ref{c4pr:proper_descent} with Proposition \ref{c4pr:hyperdescent2} replaced by Proposition \ref{c4pr:hyperdescent} and Lemma \ref{c4le:real} replaced by Lemma \ref{c4le:split}. Note that the case for $\EO{}{\Chpar{}_\Box}{\otimes}{\lambda}$-descent implies the case for $\EO{\ra}{\Chpar{}_\Box}{\otimes}{\lambda}$-descent by Lemma \ref{b3le:limit_restriction}.

The proof for (1) is similar to Proposition \ref{c4pr:proper_descent} since $r_0$ is of $\EO{}{\Chpar{}_\Box}{\otimes}{}$-descent as well.
\end{proof}

\subsection{Flat hyperdescent}
\label{c4ss:flat}

The following proposition is an analogue of flat cohomological descent \cite{SGA4}*{Expos\'{e} vbis, Proposition 4.3.3(c)}.

\begin{proposition}\label{c4pr:flat}
Every flat and locally finitely presented hypercovering of higher Artin stacks is an augmentation of $\EO{\geq 0}{\Chpar{}_\Box}{*}{}$-descent.
\end{proposition}

\begin{proof}
By Proposition \ref{c4pr:hyperdescent2}, we are reduced to show that every 
surjective flat and locally finitely presented morphism $f\colon Y\to X$ in 
$\Chpar{}_\Box$ is of $\EO{\geq 0}{\Chpar{}_\Box}{*}{}$-descent. By Lemma 
\ref{b3le:descent} and the smooth descent, we are reduced to the case of 
schemes. Let $X'$ be a disjoint union of strict localizations of $X$, such 
that the morphism  is surjective. By \cite{EGAIV}*{Corollaire 17.16.2, 
Th\'{e}or\`{e}me 18.5.11}, there exists a surjective \'etale morphism of 
schemes $g\colon X'\to X$ and a finite surjective morphism of schemes 
$g'\colon Z\to X'$ in $\Schqcs$ such that the composite morphism $Z\to X$ 
factorizes through $f$. By Lemma \ref{b3le:descent} and \'etale descent, it 
suffices to show that $g'$ is of universal $\EO{\geq 
0}{\Schqcs}{*}{}$-descent. For this, we apply \cite{HA}*{Corollary 4.7.5.3}: 
Assumption (1) follows from the dual of Lemma \ref{c4le:real}; Assumption (2) 
follows from finite base change; and the conservativeness is clear. 
\end{proof}

The above proposition can be extended to $\cD(X,\lambda)^\otimes$ under cohomological finiteness conditions, similar to the case of proper hyperdescent. We leave details to the reader.

\begin{remark}\label{c4re:etale}
We define the $\infty$-category of \emph{$\infty$-DM stacks} $\Chpdm{\infty}$ to be the $\infty$-category $\r{Sch}(\cG_{\et}(\dZ))$ of $\cG_{\et}(\dZ)$-schemes in the sense of \cite{DAG5}*{Definition 2.3.9, Remark 2.6.11}. Using Proposition \ref{c4pr:hyperdescent}, we can adapt the DESCENT program in Chapter \ref{b4ss} to define the first and the second enhanced operation maps for $\infty$-DM stacks, namely, a functor
\[
\EO{}{\Chpdm{\infty}}{\r{I}}{}\colon((\Chpdm{\infty})^{op}\times\rN(\Rind)^{op})^\amalg\to\Cat
\]
that is a lax Cartesian structure, and a map
\[
\EO{}{\Chpdm{\infty}}{\r{II}}{}\colon
\delta^*_{2,\{2\}}(((\Chpdm{\infty})^{op}\times\rN(\Rind_\tor)^{op})^{\amalg,op})^\cart_{F,\all}\to\Cat.
\]
Applying the construction in \Sec\ref{c1ss:limit}, we obtain the first and the second enhanced adic operation maps for $\infty$-DM stacks, namely, a functor
\[
\EO{\ra}{\Chpdm{\infty}}{\r{I}}{}\colon((\Chpdm{\infty})^{op}\times\rN(\Rind)^{op})^\amalg\to\Cat
\]
that is a lax Cartesian structure, and a map
\[
\EO{\ra}{\Chpdm{\infty}}{\r{II}}{}\colon
\delta^*_{2,\{2\}}(((\Chpdm{\infty})^{op}\times\rN(\Rind_\tor)^{op})^{\amalg,op})^\cart_{F,\all}\to\Cat.
\]
By restriction, we have similar functors $\EO{}{\Chpdm{\infty}}{}{!}$ and $\EO{\ra}{\Chpdm{\infty}}{}{!}$. Parallel to Proposition \ref{c4pr:hyperdescent_stack}, we have that every smooth hypercovering in $\Chpdm{\infty}$ is an augmentation of both $\EO{}{\Chpdm{\infty}}\otimes{}$-descent (resp.\ $\EO{}{\Chpdm{\infty}}{op}{!}$-descent) and $\EO{\ra}{\Chpdm{\infty}}\otimes{}$-descent (resp.\ $\EO{\ra}{\Chpdm{\infty}}{op}{!}$-descent). We have similar results for proper and flat hyperdescent.
\end{remark}

\end{document}